\documentclass[11pt,reqno]{preprint}
\usepackage[full]{textcomp}
\usepackage[osf]{newtxtext}
\usepackage{comment}

\usepackage{amssymb}
\usepackage{mathtools}
\usepackage{hyperref}
\usepackage{breakurl}
\usepackage{mhenvs}
\usepackage{mhequ}
\usepackage{mhsymb}
\usepackage{booktabs}
\usepackage{tikz}
\usepackage{mathrsfs}
\usepackage{longtable}
\usepackage{halloweenmath}
\usepackage{scalerel}
\usepackage{cprotect}
\usepackage{multirow}
\usepackage{colortbl}
\usepackage{makecell}
\usepackage{resizegather}

\usepackage{comment}
\usepackage{wasysym}
\usepackage{centernot}
\usepackage{enumitem}
\usepackage{stackrel}

\usetikzlibrary{shapes.misc}
\usetikzlibrary{shapes.symbols}
\usetikzlibrary{shapes.geometric}
\usetikzlibrary{decorations}
\usetikzlibrary{decorations.markings}
\usetikzlibrary{patterns} 
\usetikzlibrary{snakes} 

\usepackage{subfig}
\usetikzlibrary{plotmarks}
\usetikzlibrary{decorations.pathreplacing}
\usetikzlibrary{decorations.pathmorphing}
\usetikzlibrary{calc}
\usetikzlibrary{external}
\usetikzlibrary{arrows.meta}

\let\oldskull\skull
\def\skull{\mathord{\oldskull}}

\DeclareMathAlphabet{\mathbbm}{U}{bbm}{m}{n}

\DeclareFontFamily{U}{BOONDOX-calo}{\skewchar\font=45 }
\DeclareFontShape{U}{BOONDOX-calo}{m}{n}{
  <-> s*[1.05] BOONDOX-r-calo}{}
\DeclareFontShape{U}{BOONDOX-calo}{b}{n}{
  <-> s*[1.05] BOONDOX-b-calo}{}
\DeclareMathAlphabet{\mcb}{U}{BOONDOX-calo}{m}{n}
\SetMathAlphabet{\mcb}{bold}{U}{BOONDOX-calo}{b}{n}
\setlist{noitemsep,topsep=4pt}

\makeatletter
\def\DeclareSymbol#1#2#3{%
	\expandafter\gdef\csname MH@symb@#1\endcsname{\tikzsetnextfilename{symbol#1}%
	\tikz[baseline=#2,scale=0.15,draw=symbols,line join=round,line cap=round]{#3}}%
	\expandafter\gdef\csname MH@symb@#1s\endcsname{\scalebox{0.75}{\tikzsetnextfilename{symbol#1}%
	\tikz[baseline=#2,scale=0.15,draw=symbols,line join=round,line cap=round]{#3}}}%
	\expandafter\gdef\csname MH@symb@#1ss\endcsname{\scalebox{0.65}{\tikzsetnextfilename{symbol#1}%
	\tikz[baseline=#2,scale=0.15,draw=symbols,line join=round,line cap=round]{#3}}}%
	}
\def\<#1>{\ifthenelse{\boolean{mmode}}{\mathchoice{\csname MH@symb@#1\endcsname}{\csname MH@symb@#1\endcsname}{\csname MH@symb@#1s\endcsname}{\csname MH@symb@#1ss\endcsname}}{\csname MH@symb@#1\endcsname}}
\makeatother

\DeclareSymbol{PsiXi1}{0}{\draw (0,0) node[xi1] {} -- (1.2,1.2) node[xi] {};}

\DeclareSymbol{I(PsiXi1)Xi1}{-2}{\draw  (1.2,1.2) node[xi] {} -- (0,0) node[xi1] {} -- (1.2,-1.2) node[xi1] {}; }

\DeclareSymbol{I(PsiXi1)Xi1less}{-5}{\draw  (0,0) node[xi1] {} -- (-1.2,-1.2) node[xi1] {}; }

\DeclareSymbol{supercherry1}{0}{\draw (-1,1) node[xi] {} -- (0,0);
\draw  (0,2) node[xi] {}-- (1,1);
\draw  (1,1)[kernels2] node[xi1] {} -- (0,0) {};}

\DeclareSymbol{supercherry1less}{0}{
\draw  (0,2) node[xi] {}-- (1,1);
\draw  (1,1)[kernels2] node[xi1] {} -- (0,0) {};}

\DeclareSymbol{supercherry2}{0}{
\draw  (0,2) node[xi] {}-- (-1,1);
\draw (-1,1) node[xi1] {} -- (0,0);
\draw  (1,1)[kernels2] node[xi] {} -- (0,0) {};}

\DeclareSymbol{cherry231}{2}{		%using arbitrary names...
\draw  (0,2) node[xi] {}-- (1,1);
\draw[kernels2] (1,1) node[xi1] {} -- (0,0);
\draw[kernels2]  (-1,1) node[xi1] {} -- (0,0) {};}

\DeclareSymbol{cherry232}{-1}{		%using arbitrary names...
\draw[kernels2] (1,1) node[xi1] {} -- (0,0);
\draw[kernels2]  (-1,1) node[xi1] {} -- (0,0) {};}

\DeclareSymbol{I'Xibar}{-1}{
\draw (0,1)[kernels2] node[xi1] {} -- (0,-.4);}

\DeclareSymbol{I'Xitilde}{-1}{
\draw (0,1)[kernels2] node[xi2] {} -- (0,-.4);}

\DeclareSymbol{IXiI'Xibar}{-1}{\draw (-1,1) node[xi] {} -- (0,0);
\draw (1,1)[kernels2] node[xi1] {} -- (0,0);}

\DeclareSymbol{I'XiIXibar}{-1}{\draw (-1,1) node[xi1] {} -- (0,0);
\draw (1,1)[kernels2] node[xi] {} -- (0,0);}

\DeclareSymbol{cherryY}{-5}{\draw (-1,1) node[xi] {} -- (0,0);
\draw (1,1)[kernels2] node[xi] {} -- (0,0);
\draw (0,-1.5) node[xi1] {} -- (0,0);
}

\DeclareSymbol{cherryYless}{-5}{
\draw (1,1)[kernels2] node[xi] {} -- (0,0);
\draw (0,-1.5) node[xi1] {} -- (0,0);
}

\DeclareSymbol{R2-1}{-1}{\draw (-1,1)[kernels2] node[xi] {} -- (0,0);
\draw (1,1)[kernels2] node[xi1] {} -- (0,0);}

\DeclareSymbol{R2-1new}{-1}{\draw (-1,1)[kernels2] node[xi2] {} -- (0,0);
\draw (1,1)[kernels2] node[xi2] {} -- (0,0);}

\DeclareSymbol{I(R2-1new)}{-1}{\draw (-1,2)[kernels2] node[xi2] {} -- (0,1);
\draw (1,2)[kernels2] node[xi2] {} -- (0,1);
\draw (0,1) {} -- (0,-0.4);}

\DeclareSymbol{I'XiI'[IXiI'Xi]}{1}{\draw[kernels2] (2,2) node[xi] {} -- (1,1);
\draw (0,2) node[xi] {} -- (1,1) node {};
\draw[kernels2] (1,1) -- (0,0) node {};
\draw[kernels2] (-1,1) node[xi1] {} -- (0,0);}

\DeclareSymbol{I'XiI'[IXiI'Xi]less}{1}{\draw[kernels2] (0,2) node[xi] {} -- (1,1);
\draw[kernels2] (1,1) -- (0,0) node {};
\draw[kernels2] (-1,1) node[xi1] {} -- (0,0);}

\DeclareSymbol{I'[IXiI'Xi]double}{-1}{\draw[kernels2] (2,2) node[xi] {} -- (1,1);
\draw (.6,2) node[xi] {} -- (1,1) node {};
\draw[kernels2] (1,1) -- (0,0) node {};
\draw[kernels2] (-2,2) node[xi] {} -- (-1,1);
\draw (-.6,2) node[xi] {} -- (-1,1) node {};
\draw[kernels2] (-1,1) -- (0,0) node {};
}

\DeclareSymbol{I'[IXiI'Xi]double-less1}{-1}{\draw[kernels2] (0.6,2) node[xi] {} -- (1,1);
\draw[kernels2] (1,1) -- (0,0) node {};
\draw[kernels2] (-2,2) node[xi] {} -- (-1,1);
\draw (-.6,2) node[xi] {} -- (-1,1) node {};
\draw[kernels2] (-1,1) -- (0,0) node {};
}

\DeclareSymbol{I'[IXiI'Xi]double-less2}{-1}{\draw[kernels2] (.6,2) node[xi] {} -- (1,1);
\draw[kernels2] (1,1) -- (0,0) node {};
\draw[kernels2] (-.6,2) node[xi] {} -- (-1,1);
\draw[kernels2] (-1,1) -- (0,0) node {};
}

\DeclareSymbol{I'[IXiI'Xi]single-less}{-1}{\draw[kernels2] (0,2) node[xi] {} -- (1,1);
\draw[kernels2] (1,1) -- (0,0) node {};
}

\DeclareSymbol{I'[IXiI'Xi]single}{-5}{\draw (-1,1) node[xi] {} -- (0,0);
\draw (1,1)[kernels2] node[xi] {} -- (0,0);
\draw[kernels2] (0,-1.5) {} -- (0,0);
}

\DeclareSymbol{I'(R2-1)}{-1}{
\draw[kernels2] (0,0) -- (0,-1);
\draw (-1,1)[kernels2] node[xi] {} -- (0,0);
\draw (1,1)[kernels2] node[xi1] {} -- (0,0);
}

\DeclareSymbol{I'(R2-1)new}{-1}{
\draw[kernels2] (0,0) -- (0,-1);
\draw (-1,1)[kernels2] node[xi2] {} -- (0,0);
\draw (1,1)[kernels2] node[xi2] {} -- (0,0);
}

\DeclareSymbol{I(R2-1)Xi1new}{-1}{
\draw (0,0) -- (0,-1.5) node[xi1] {};
\draw (-1,1)[kernels2] node[xi2] {} -- (0,0);
\draw (1,1)[kernels2] node[xi2] {} -- (0,0);
}

\DeclareSymbol{IXiI'(R2-1)}{-1}{
\draw (-2,0) node[xi] {} -- (-1,-1);
\draw[kernels2] (0,0) -- (-1,-1);
\draw (-1,1)[kernels2] node[xi] {} -- (0,0);
\draw (1,1)[kernels2] node[xi1] {} -- (0,0);
}

\DeclareSymbol{IXiI'(R2-1)new}{-1}{
\draw (-2,0) node[xi] {} -- (-1,-1);
\draw[kernels2] (0,0) -- (-1,-1);
\draw (-1,1)[kernels2] node[xi2] {} -- (0,0);
\draw (1,1)[kernels2] node[xi2] {} -- (0,0);
}

\DeclareSymbol{IXiI(R2-1)new}{-1}{
\draw (-2,0) node[xi] {} -- (-1,-1);
\draw (0,0) -- (-1,-1);
\draw (-1,1)[kernels2] node[xi2] {} -- (0,0);
\draw (1,1)[kernels2] node[xi2] {} -- (0,0);
}

\DeclareSymbol{I'XiI(R2-1)}{-1}{
\draw (0,0)  -- (1,-1);
\draw (2,0)[kernels2] node[xi] {} -- (1,-1);
\draw (-1,1)[kernels2] node[xi1] {} -- (0,0);
\draw (1,1)[kernels2] node[xi] {} -- (0,0);
}

\DeclareSymbol{I'XiI(R2-1)new}{-1}{
\draw (0,0)  -- (1,-1);
\draw (2,0)[kernels2] node[xi] {} -- (1,-1);
\draw (-1,1)[kernels2] node[xi2] {} -- (0,0);
\draw (1,1)[kernels2] node[xi2] {} -- (0,0);
}

\DeclareSymbol{I'Xi1I'(R2-1)new}{-1}{
\draw[kernels2] (0,0)  -- (1,-1);
\draw (2,0)[kernels2] node[xi1] {} -- (1,-1);
\draw (-1,1)[kernels2] node[xi2] {} -- (0,0);
\draw (1,1)[kernels2] node[xi2] {} -- (0,0);
}

\DeclareSymbol{IXibar}{-1}{
\draw (0,1) node[xi1] {} -- (0,-.4);}

\DeclareSymbol{PsiPsibarXibar}{-1}{\draw (-1,1.3) node[xi] {} -- (0,0);
\draw (1,1.3) node[xi1] {} -- (0,0) node[xi1] {};}

\DeclareSymbol{PsiPsibar}{-1}{\draw (-1,1.3) node[xi] {} -- (0,0);
\draw (1,1.3) node[xi1] {} -- (0,0) {};}

\DeclareSymbol{IPsiXi1}{0}{\draw (0,-0.4) {} -- (0,1) node[xi1] {} -- (1.2,2.2) node[xi] {};}

\DeclareSymbol{r_z_large}{2}{
\draw (-1,1.3) node[xi] {} -- (0,0) {};
\draw (0,2.6) node[xi] {} -- (1,1.3) node[xi1] {} -- (0,0) node[xi1] {};}

\DeclareSymbol{r_z_large_pos}{2}{
\draw (-1,1.3) node[xi] {} -- (0,0) {};
\draw (0,2.6) node[xi] {} -- (1,1.3) node[xi1] {} -- (0,0) {};}

\DeclareSymbol{r_z_avoid1}{2}{
\draw (-1,1.3) node[xi] {} -- (0,0) {};
\draw (0,2.6) node[xi] {} -- (1,1.3) {} -- (0,0) node[xi1] {};
\draw (2,2.6)[kernels2] node[xi] {} -- (1,1.3) {};}

\DeclareSymbol{r_z_avoid2}{2}{
\draw (-1,1.3) node[xi] {} -- (0,0) {};
\draw (0,2.6)[kernels2] node[xi] {} -- (1,1.3) {};
\draw (1,1.3) {} -- (0,0) node[xi1] {};}

\DeclareSymbol{r_z_avoid3}{2}{
\draw (-1,1.3) node[xi] {} -- (0,0) {};
\draw (0,2.6)[kernels2] node[xi2] {} -- (1,1.3);
\draw (2,2.6)[kernels2] node[xi2] {} -- (1,1.3);
\draw (1,1.3) {} -- (0,0) node[xi1] {};}

\DeclareSymbol{1}{0}{\draw[white] (-.5,0) -- (.5,0); \draw (0,0)  -- (0,1.5) node[sdot] {};}
\DeclareSymbol{11}{0}{\draw (-0.5,1.2) node[sdot] {} -- (0,0) -- (0.5,1.2) node[sdot] {};}
\DeclareSymbol{111}{0}{\draw (0,0) -- (0,1.2) node[sdot] {}; \draw (-.7,1) node[sdot] {} -- (0,0) -- (.7,1) node[sdot] {};}
\DeclareSymbol{131}{-3}{\draw (0,0) -- (0,-1) -- (1,0) node[sdot] {}; \draw (0,0) -- (0,-1) -- (-1,0) node[sdot] {}; \draw (0,0) -- (0,1.2) node[sdot] {}; \draw (-.7,1) node[sdot] {} -- (0,0) -- (.7,1) node[sdot] {};}
\DeclareSymbol{11}{0}{\draw (-0.5,1.2) node[sdot] {} -- (0,0) -- (0.5,1.2) node[sdot] {};}
\DeclareSymbol{12}{-3}{\draw (-.8,1) node[sdot] {} -- (0,0) -- (0,-1); \draw (1,0) node[sdot] {} -- (0,-1) -- (-1,0) node[sdot] {};}
\DeclareSymbol{30}{-3}{\draw (0,0) -- (0,-1); \draw (0,0) -- (0,1.2) node[sdot] {}; \draw (-.7,1) node[sdot] {} -- (0,0) -- (.7,1) node[sdot] {};}
\DeclareSymbol{10}{-3}{\draw (-.8,1) node[sdot] {} -- (0,0) -- (0,-1);}
\DeclareSymbol{22}{-3}{\draw (0,0.3) -- (0,-1) -- (1,0) node[sdot] {}; \draw (0,0.3) -- (0,-1) -- (-1,0) node[sdot] {};\draw (-.7,1) node[sdot] {} -- (0,0.3) -- (.7,1) node[sdot] {};}
\DeclareSymbol{211}{0}{\draw (-0.5,2.4) node[sdot] {} -- (-1,1.6);\draw (-1.5,2.4) node[sdot] {} -- (-0.5,0.8); \draw (0,1.6) node[sdot] {} -- (-0.5,0.8) -- (0,0) -- (0.5,0.8) node[sdot] {};}

\DeclareSymbol{Xi22}{0.5}{\draw (0,0) node[xi] {} -- (-1,1) node[not] {} -- (0,2) node[xi] {};}

\DeclareSymbol{Xi2}{-2}{\draw (0,-0.25) node[xi] {} -- (-1,1) node[xi] {};}
\DeclareSymbol{Xi3}{0}{\draw (0,0) node[xi] {} -- (-1,1) node[xi] {} -- (0,2) node[xi] {};}
\DeclareSymbol{Xi4}{2}{\draw (0,0) node[not] {} -- (-1,1) node[xi] {} -- (0,2) node[xi] {} -- (-1,3) node[xi] {};}
\DeclareSymbol{Xi2X}{-2}{\draw (0,-0.25) node[xi] {} -- (-1,1) node[xix] {};}
\DeclareSymbol{XXi2}{-2}{\draw (0,-0.25) node[xix] {} -- (-1,1) node[xi] {};}

\DeclareSymbol{IXi^2asym}{-1}{\draw (-1,1) node[xired] {} -- (0,-.4);
\draw (1,1) node[xigreen] {} -- (0,-.4);}

\DeclareSymbol{IXi^2asym2}{-1}{\draw (-1,1) node[xigreen] {} -- (0,-.4);
\draw (1,1) node[xired] {} -- (0,-.4);}

\DeclareSymbol{IXi^2asym3}{-1}{\draw (-1,1) node[xired] {} -- (0,-.4);
\draw (1,1) node[xigreen] {} -- (0,-.4);}

\DeclareSymbol{IXi^2}{-1}{\draw (-1,1) node[xi] {} -- (0,-.4);
\draw (1,1) node[xi] {} -- (0,-.4);}

\DeclareSymbol{rIXi^2}{-1}{\draw (-1,1) node[xired] {} -- (0,-.4);
\draw (1,1) node[xired] {} -- (0,-.4);}

\DeclareSymbol{IXi^2green}{-1}{\draw (-1,1) node[xigreen] {} -- (0,-.4);
\draw (1,1) node[xigreen] {} -- (0,-.4);}

\DeclareSymbol{IXi^2_typed}{-1}{\draw (-1,1) node[xiorange] {} -- (0,-.4);
\draw (1,1) node[xigreen] {} -- (0,-.4);}

\DeclareSymbol{IXi^2green*}{-1}{\draw (-1,1) node[xigreen1] {} -- (0,-.4);
\draw (1,1) node[xigreen1] {} -- (0,-.4);}
\DeclareSymbol{IXiI'Xi_notriangle}{-1}{\draw (-1,1) node[xi] {} -- (0,0);
\draw (1,1)[kernels2] node[xi] {} -- (0,0) {};}

\DeclareSymbol{I(cherry)}{-1}{\draw (-1,2) node[xi] {} -- (0,1);
\draw (1,2)[kernels2] node[xi] {} -- (0,1) {};
\draw (0,1) {} -- (0,-0.4) {};}

\DeclareSymbol{IDPsi}{2}{\draw (0,2)[kernels2] node[xi] {} -- (1,1);
\draw (1,1) {} -- (0,0) {};}

\DeclareSymbol{I'XiIXi_notriangle}{-1}{\draw (-1,1)[kernels2] node[xi] {} -- (0,0);
\draw (1,1) node[xi] {} -- (0,0) {};}

\DeclareSymbol{IXiI'Xi}{-1}{\draw (-1,1) node[xi] {} -- (0,0);
\draw (1,1)[kernels2] node[xi] {} -- (0,0) node[not] {};}

\DeclareSymbol{IXiI'Xi_typed1}{-1}{\draw (-1,1) node[xigreen] {} -- (0,0);
\draw (1,1)[kernels2,greennode] node[xired] {} -- (0,0) node[not] {};}

\DeclareSymbol{IXiI'Xi_typed2}{-1}{\draw (-1,1) node[xigreen] {} -- (0,0);
\draw (1,1)[kernels2,rednode] node[xigreen] {} -- (0,0) node[not] {};}

\DeclareSymbol{IXiI'Xi_typed3}{-1}{\draw (-1,1) node[xiorange] {} -- (0,0);
\draw (1,1)[kernels2,orangenode] node[xigreen] {} -- (0,0) node[notgreen] {};}

\DeclareSymbol{IXiI'Xi_typed4}{-1}{\draw (-1,1) node[xiorange] {} -- (0,0);
\draw (1,1)[kernels2,greennode] node[xiorange] {} -- (0,0) node[notgreen] {};}

\DeclareSymbol{IXiI'Xi_typed2_notriangle}{-1}{\draw (-1,1) node[xigreen] {} -- (0,0);
\draw (1,1)[kernels2,rednode] node[xigreen] {} -- (0,0) {};}

\DeclareSymbol{IXi^3}{-1}{\draw (-1.3,1) node[xi] {} -- (0,0);
\draw (1.3,1) node[xi] {} -- (0,0);
\draw (0,1.5) node[xi] {} -- (0,0) node[not] {};}

\DeclareSymbol{IXi^3_notriangle}{0}{\draw (-1.3,1) node[xi] {} -- (0,0);
\draw (1.3,1) node[xi] {} -- (0,0);
\draw (0,1.5) node[xi] {} -- (0,0) {};}

\DeclareSymbol{IXi^3_typed}{-1}{\draw (-1.3,1) node[xigreen] {} -- (0,0);
\draw (1.3,1) node[xired] {} -- (0,0);
\draw (0,1) node[xigreen] {} -- (0,0) node[not] {};}

\DeclareSymbol{IXi^3-typed112}{-1}{\draw (-1.3,1) node[xi] {} -- (0,0);
\draw (1.3,1) node[xigreen] {} -- (0,0);
\draw (0,1) node[xi] {} -- (0,0) node[not] {};}
\DeclareSymbol{I'Xi}{-1}{
\draw[kernels2] (1,1) node[xi] {} -- (0,0) node[not] {};}

\DeclareSymbol{I'Xi_notriangle}{-1}{
\draw[kernels2] (0,1) node[xi] {} -- (0,-0.4) {};}

\DeclareSymbol{I'Xigreen-red}{-1}{
\draw[kernels2,rednode] (0,1) node[xi,fill=greennode] {} -- (0,-.4);}

\DeclareSymbol{IXi}{-1}{
\draw (0,1) node[xi] {} -- (0,-.4);}

\DeclareSymbol{IXigreen}{-1}{
\draw (0,1) node[xigreen] {} -- (0,-.4);}

\DeclareSymbol{green}{-1}{
\draw[greennode,very thick,line cap=round,scale=1.3] (0.05,0.4) -- (0.2,0.55) -- (0,0) -- (0.6,0.6) -- (0.4,0.05) -- (0.55,0.2);}
\DeclareSymbol{red}{-1}{
\draw[rednode,very thick,line cap=round,scale=1.3] (0.05,0.4) -- (0.2,0.55) -- (0,0) -- (0.6,0.6) -- (0.4,0.05) -- (0.55,0.2);}
\DeclareSymbol{lblue}{-1}{
\draw[lbluenode,very thick,line cap=round,scale=1.3] (0.05,0.4) -- (0.2,0.55) -- (0,0) -- (0.6,0.6) -- (0.4,0.05) -- (0.55,0.2);}
\DeclareSymbol{dblue}{-1}{
\draw[dbluenode,very thick,line cap=round,scale=1.3] (0.05,0.4) -- (0.2,0.55) -- (0,0) -- (0.6,0.6) -- (0.4,0.05) -- (0.55,0.2);}
\DeclareSymbol{orange}{-1}{
\draw[orangenode,very thick,line cap=round,scale=1.3] (0.05,0.4) -- (0.2,0.55) -- (0,0) -- (0.6,0.6) -- (0.4,0.05) -- (0.55,0.2);}

\DeclareSymbol{I'XXi}{-1}{
\draw[kernels2] (1,1) node[xi] {} -- (0,0) node[not] {};
\draw (1,1) node[xix] {};}

\DeclareSymbol{I'XXi_notriangle}{-1}{
\draw[kernels2] (1,1) node[xi] {} -- (0,0) {};
\draw (1,1) node[xix] {};}

\DeclareSymbol{IXiI'[IXiI'Xi]}{-1}{\draw[kernels2] (2,2) node[xi] {} -- (1,1);
\draw (0,2) node[xi] {} -- (1,1) node[not] {};
\draw[kernels2] (1,1) -- (0,0) node[not] {};
\draw (-1,1) node[xi] {} -- (0,0);}

\DeclareSymbol{IXiI'[IXiI'Xi]_notriangle}{1}{\draw[kernels2] (2,2) node[xi] {} -- (1,1);
\draw (0,2) node[xi] {} -- (1,1) {};
\draw[kernels2] (1,1) -- (0,0) {};
\draw (-1,1) node[xi] {} -- (0,0);}

\DeclareSymbol{I'[IXiI'Xi]}{-1}{\draw[kernels2] (2,2) node[xi] {} -- (1,1);
\draw (0,2) node[xi] {} -- (1,1) node[not] {};
\draw[kernels2] (1,1) -- (0,0) node[not] {};}

\DeclareSymbol{I'[IXiI'Xi]_notriangle}{1}{\draw[kernels2] (2,2) node[xi] {} -- (1,1);
\draw (0,2) node[xi] {} -- (1,1) {};
\draw[kernels2] (1,1) -- (0,0) {};}

\DeclareSymbol{I[IXi^2]IXi-typed112}{-1}{\draw (-2,2) node[xi] {} -- (-1,1);
\draw (0,2) node[xi] {} -- (-1,1) node[not] {};
\draw (-1,1) -- (0,0) node[not] {};
\draw (1,1) node[xigreen] {} -- (0,0);}

\DeclareSymbol{I[IXi^2]IXi-typed211}{-1}{\draw (-2,2) node[xigreen] {} -- (-1,1);
\draw (0,2) node[xi] {} -- (-1,1) node[not] {};
\draw (-1,1) -- (0,0) node[not] {};
\draw (1,1) node[xi] {} -- (0,0);}

\DeclareSymbol{I[IXiI'Xi]I'Xi}{-1}{\draw (-2,2) node[xi] {} -- (-1,1);
\draw[kernels2] (0,2) node[xi] {} -- (-1,1) node[not] {};
\draw (-1,1) -- (0,0) node[not] {};
\draw[kernels2] (1,1) node[xi] {} -- (0,0);}

\DeclareSymbol{I[IXiI'Xi]I'Xi_notriangle}{1}{\draw (-2,2) node[xi] {} -- (-1,1);
\draw[kernels2] (0,2) node[xi] {} -- (-1,1) {};
\draw (-1,1) -- (0,0) {};
\draw[kernels2] (1,1) node[xi] {} -- (0,0);}
\DeclareSymbol{I[I[Xi]]Xi_notriangle}{1}{
\draw (0,2) node[xi] {} -- (-1,1) {};
\draw[snake=snake , segment length=1pt, segment amplitude=1pt] (-1,1) -- (0,0) node[xi] {};}

\DeclareSymbol{I[I[Xi]]Xi_typed}{1}{
\draw (0,2) node[xigreen] {} -- (-1,1) {};
\draw[snake=snake , segment length=1pt, segment amplitude=1pt] (-1,1) -- (0,0) node[xigreen] {};}

\DeclareSymbol{I[I[Xi]]Xi_typed*}{1}{
\draw (0,2) node[xigreen1] {} -- (-1,1) {};
\draw[snake=snake , segment length=1pt, segment amplitude=1pt] (-1,1) -- (0,0) node[xigreen1] {};}
\DeclareSymbol{I[I'Xi]I'Xi}{1}{\draw[kernels2] (0,2) node[xi] {} -- (-1,1);
\draw (-1,1) node[fill=rednode,not] {} -- (0,0);
\draw[kernels2] (1,1) node[xi] {} -- (0,0);}

\DeclareSymbol{I[I'Xi]I'Xi_notriangle}{1}{\draw[kernels2] (0,2) node[xi] {} -- (-1,1);
\draw (-1,1) {} -- (0,0);
\draw[kernels2] (1,1) node[xi] {} -- (0,0);}

\DeclareSymbol{I[I'Xi]I'Xi_typed}{1}{\draw[kernels2,rednode] (0,2) node[xigreen] {} -- (-1,1);
\draw (-1,1) node[notorange] {} -- (0,0);
\draw[kernels2,rednode] (1,1) node[xigreen] {} -- (0,0);}

\DeclareSymbol{I[I'Xi]I'Xi_typed*}{1}{\draw[kernels2,rednode] (0,2) node[xigreen1] {} -- (-1,1);
\draw (-1,1) node[notorange] {} -- (0,0);
\draw[kernels2,rednode] (1,1) node[xigreen1] {} -- (0,0);}

\DeclareSymbol{I[I'Xi]I'Xi_typed-h}{1}{\draw[kernels2,greennode] (0,2) node[xigreen] {} -- (-1,1);
\draw[snake=snake , segment length=3pt, segment amplitude=1pt] (-1,1) node[notgreen] {} -- (0,0);
\draw[kernels2,greennode] (1,1) node[xigreen] {} -- (0,0);}

\DeclareSymbol{I[I'Xi]I'Xi_typed-h*}{1}{\draw[kernels2,greennode] (0,2) node[xigreen1] {} -- (-1,1);
\draw[snake=snake , segment length=3pt, segment amplitude=1pt] (-1,1) node[notgreen] {} -- (0,0);
\draw[kernels2,greennode] (1,1) node[xigreen1] {} -- (0,0);}

\DeclareSymbol{I[I'Xi]I'Xib}{1}{\draw[kernels2] (0,2) node[xi] {} -- (1,1);
\draw (1,1) node[not] {} -- (0,0);
\draw[kernels2] (-1,1) node[xi] {} -- (0,0);}

\DeclareSymbol{I[I'Xi]I'Xib_notriangle}{1}{\draw[kernels2] (0,2) node[xi] {} -- (1,1);
\draw (1,1) {} -- (0,0);
\draw[kernels2] (-1,1) node[xi] {} -- (0,0);}
\DeclareSymbol{IXiI'[I'Xi]}{1}{\draw[kernels2] (0,2) node[xi] {} -- (1,1);
\draw[kernels2] (1,1) node[not] {} -- (0,0);
\draw (-1,1) node[xi] {} -- (0,0);}

\DeclareSymbol{IXiI'[I'Xi]_notriangle}{1}{\draw[kernels2] (0,2) node[xi] {} -- (1,1);
\draw[kernels2] (1,1)  {} -- (0,0);
\draw (-1,1) node[xi] {} -- (0,0);}

\DeclareSymbol{IXiI'[I'Xi]_typed}{1}{\draw[kernels2,rednode] (0,2) node[xigreen] {} -- (1,1);
\draw[kernels2,rednode] (1,1) node[notorange] {} -- (0,0);
\draw (-1,1) node[xigreen] {} -- (0,0);}

\DeclareSymbol{IXiI'[I'Xi]_typed-h}{1}{\draw[kernels2,greennode] (0,2) node[xigreen] {} -- (1,1);
\draw[kernels2,greennode,snake=snake , segment length=3pt, segment amplitude=1pt] (1,1) node[notgreen] {} -- (0,0);
\draw (-1,1) node[xigreen] {} -- (0,0);}

\DeclareSymbol{IXiI'[I'Xi]_typed*}{1}{\draw[kernels2,rednode] (0,2) node[xigreen1] {} -- (1,1);
\draw[kernels2,rednode] (1,1) node[notorange] {} -- (0,0);
\draw (-1,1) node[xigreen1] {} -- (0,0);}

\DeclareSymbol{IXiI'[I'Xi]_typed-h*}{1}{\draw[kernels2,greennode] (0,2) node[xigreen1] {} -- (1,1);
\draw[kernels2,greennode,snake=snake , segment length=3pt, segment amplitude=1pt] (1,1) node[notgreen] {} -- (0,0);
\draw (-1,1) node[xigreen1] {} -- (0,0);}
\DeclareSymbol{I[I'Xi]}{-1}{\draw[kernels2] (0,2) node[xi] {} -- (1,1);
\draw (1,1) node[not] {} -- (0,0) node[not] {};}

\DeclareSymbol{I[I'Xi]_typed}{1}{\draw[kernels2,rednode] (0,2) node[xigreen] {} -- (-1,1);
\draw (-1,1) node[notorange] {} -- (0,0);}

\DeclareSymbol{I'[I'Xi]}{-1}{\draw[kernels2] (0,2) node[xi] {} -- (1,1);
\draw[kernels2] (1,1) node[not] {} -- (0,0) node[not] {};}

\DeclareSymbol{I'[I'Xi]_notriangle}{1}{\draw[kernels2] (0,2) node[xi] {} -- (1,1);
\draw[kernels2] (1,1) {} -- (0,0) {};}

\DeclareSymbol{I'[I'Xi]_typed}{1}{\draw[kernels2,rednode] (0,2) node[xigreen] {} -- (1,1);
\draw[kernels2,rednode] (1,1) node[notorange] {} -- (0,0)  {};}
\DeclareSymbol{IXiI'XXi}{-1}{\draw (-1,1) node[xi] {} -- (0,0);
\draw[kernels2] (1,1) node[xi] {} -- (0,0) node[not] {};
\draw (1,1) node[xix] {};}

\DeclareSymbol{IXiI'XXi_notriangle}{-1}{\draw (-1,1) node[xi] {} -- (0,0);
\draw[kernels2] (1,1) node[xi] {} -- (0,0) {};
\draw (1,1) node[xix] {};}

\DeclareSymbol{IXiI'XXi_typed}{-1}{\draw (-1,1) node[xigreen] {} -- (0,-0.3);
\draw[kernels2,rednode] (1,1)  -- (0,-0.3) {};
\draw (1,1) node[xix-green-red] {};
}

\DeclareSymbol{IXiI'XXi_typed*}{-1}{\draw (-1,1) node[xigreen1] {} -- (0,-0.3);
\draw[kernels2,rednode] (1,1)  -- (0,-0.3) {};
\draw (1,1) node[xix-green-red1] {};
}

\DeclareSymbol{IXXiI'Xi}{-1}{\draw (-1,1) node[xix] {} -- (0,0);
\draw[kernels2] (1,1) node[xi] {} -- (0,0) node[not] {};}

\DeclareSymbol{IXXiI'Xi_notriangle}{-1}{\draw (-1,1) node[xix] {} -- (0,0);
\draw[kernels2] (1,1) node[xi] {} -- (0,0) {};}

\DeclareSymbol{IXXiI'Xi_typed}{-1}{\draw (-1,1) node[xix-green-red] {} -- (0,-0.3);
\draw[kernels2,rednode] (1,1) node[xigreen] {} -- (0,-0.3) {};}

\DeclareSymbol{IXXiI'Xi_typed*}{-1}{\draw (-1,1) node[xix-green-red1] {} -- (0,-0.3);
\draw[kernels2,rednode] (1,1) node[xigreen1] {} -- (0,-0.3) {};}
\DeclareSymbol{XiX}{-2.8}{\node[xibx] {};}
\DeclareSymbol{XXi}{-2.8}{\node[xibx] {};}
\DeclareSymbol{tauX}{-2.8}{ \draw[kernels2] (0,0) node[xibx] {};}
\DeclareSymbol{Xi}{-2.8}{\node[xib] {};}
\DeclareSymbol{Xi1}{-2.8}{\node[xib1] {};}
\DeclareSymbol{Xigreen}{-2.8}{\node[xigreen] {};}
\DeclareSymbol{Xired}{-2.8}{\node[xired] {};}
\DeclareSymbol{not}{-2.8}{\node[not,minimum size=1.6mm] {};}
\DeclareSymbol{XiY}{-2.8}{\node[xie] {};}
\DeclareSymbol{XiZ}{-2.8}{\node[xid] {};}
\DeclareSymbol{IXiX}{0}{\draw (0,-0.25) node[not] {} -- (0,1.5) node[xix] {};}

\makeatletter % Stolen from the internet to make a fat \cdot which isn't as fat as a \bullet
\newcommand*{\bigcdot}{}% Check if undefined
\DeclareRobustCommand*{\bigcdot}{%
  \mathbin{\mathpalette\bigcdot@{}}%
}
\newcommand*{\bigcdot@scalefactor}{.5}
\newcommand*{\bigcdot@widthfactor}{1.15}
\newcommand*{\bigcdot@}[2]{%
  \sbox0{$#1\vcenter{}$}% math axis
  \sbox2{$#1\cdot\m@th$}%
  \hbox to \bigcdot@widthfactor\wd2{%
    \hfil
    \raise\ht0\hbox{%
      \scalebox{\bigcdot@scalefactor}{%
        \lower\ht0\hbox{$#1\bullet\m@th$}%
      }%
    }%
    \hfil
  }%
}
\makeatother

\def\act{\bigcdot}

\newcommand{\cT}{\mcb{T}}
\newcommand{\cbT}{\mcb{T}}

\newcommand{\cbI}{\mcb{I}}

\newcommand{\cut}{\mathfrak{C}}

\newcommand{\mrd}{\mathop{}\!\mathrm{d}}
\newcommand{\dt}{\mrd_t}

\newcommand{\mcE}{\mathcal{E}}
\newcommand{\mcH}{\mathcal{H}}

\newcommand{\mcR}{\mathcal{R}}
\newcommand{\mcC}{\mathcal{C}}

\newcommand{\mcS}{\mathcal{S}}

\newcommand{\mcL}{\mathcal{L}}

\newcommand{\mcI}{\mathcal{I}}

\newcommand{\mcP}{\mathcal{P}}

\newcommand{\mbA}{\mathbf{A}}

\newcommand{\mbY}{\mathbf{Y}}
\newcommand{\mbX}{\mathbf{X}}

\newcommand{\T}{\mathbf{T}}

\newcommand{\bonds}{{\mathbf{B}}}
\newcommand{\obonds}{{\overline{\mathbf{B}}}}
\newcommand{\Grid}{\mathbf{G}}
\newcommand{\olines}{\overline{\mathbf{L}}}
\newcommand{\lines}{{\mathbf{L}}}

\newcommand{\rect}{{\mathbf{R}}}
\newcommand{\plaq}{{\mathbb{P}}}

\def\Diff{\mathrm{Diff}}

\newcommand{\U}{\mathrm{U}}
\newcommand{\SO}{\mathrm{SO}}

\newcommand{\SU}{\mathrm{SU}}

\def\boldR{\boldsymbol{R}}

\def\boldH{\boldsymbol{H}}

\def\${|\!|\!|}

\def\id{\mathrm{id}}

\def\var#1{#1\textnormal{-var}}
\def\Hol#1{#1\textnormal{-H{\"o}l}}
\def\gr#1{#1\textnormal{-gr}}

\def\v#1{#1\textnormal{-vee}}

\newcommand{\fix}{\textnormal{fix}}
\newcommand{\nfix}{\textnormal{nfix}}

\DeclareMathOperator{\hol}{\mathrm{hol}}

\def\conf{Q}

\def\scal#1{{\langle#1\rangle}}

\def\CS{\mathcal{S}}
\def\CR{\mathcal{R}}

\def\CT{\mcb{T}}
\def\CG{\mcb{G}}

\def\bone{\mathbf{1}}
\def\ST{\boldsymbol{\mcb T}}

\def\sol{{\mathop{\mathrm{sol}}}}

\newcommand{\mfu}{\mathfrak{u}}

\newcommand{\mfT}{\mathfrak{T}}

\newcommand{\mfL}{\mathfrak{L}}

\newcommand{\mfa}{\mathfrak{a}}

\newcommand{\mfR}{\mathfrak{R}}

\newcommand{\mft}{\mathfrak{t}}

\newcommand{\mfp}{\mathfrak{p}}

\newcommand{\mfl}{\mathfrak{l}}
\newcommand{\mfh}{\mathfrak{h}}
\newcommand{\mfq}{\mathfrak{q}}

\newcommand{\mfG}{\mathfrak{G}}
\newcommand{\mfO}{\mathfrak{O}}

\newcommand{\mfg}{\mathfrak{g}}
\newcommand{\mfk}{\mathfrak{k}}
\newcommand{\mfj}{\mathfrak{j}}

\def\cC{\mathscr{C}}

\def\cS{\mathscr{S}}
\def\cD{\mathscr{D}}
\def\cE{\mathscr{E}}

\def\sym{\textnormal{\scriptsize \textsc{sym}}}

\def\YM{\textnormal{\small \textsc{ym}}}

\def\northeast{\textnormal{\scriptsize \textsc{ne}}}
\def\northwest{\textnormal{\scriptsize \textsc{nw}}}
\def\southeast{\textnormal{\scriptsize \textsc{se}}}
\def\southwest{\textnormal{\scriptsize \textsc{sw}}}
\def\Southeast{\textnormal{\textsc{se}}}
\def\Southwest{\textnormal{\textsc{sw}}}
\def\Northwest{\textnormal{\textsc{nw}}}
\newcommand{\ad}{\mathrm{ad}}
\newcommand{\Ad}{\mathrm{Ad}}

\newcommand{\Aut}{\mathrm{Aut}}

\DeclareMathOperator{\Trace}{Tr}

\newcommand{\SYM}{\mathrm{SYM}}
\DeclareMathOperator{\Log}{Log}

\def\combplus[#1,#2,#3,#4]{\binom{#1\ {\scriptstyle #4} }{#2\ #3}}

\def\singlescalegenvert[#1,#2]{\hat{H}^{#2}_{#1}}
\def\multiscalegenvert[#1,#2]{H^{#2}_{#1}}

\def\moll{\chi}

\def\nr[#1]{\tilde{N}[#1]} %nr stands for non-root
\def\inn[#1]{\mathring{N}[#1]}
\def\nrinn[#1]{\hat{N}_{#1}} %nrinn stands for non-root inner nodes
\def\nrmod[#1,#2]{\tilde{N}_{#1}(#2)}
\def\nrinnmod[#1,#2]{\hat{N}_{#1}(#2)}

\def\ident[#1]{\underline{#1}}

\def\mylink#1#2{\mathrel{\vbox{\offinterlineskip\ialign{%
    \hfil##\hfil\cr
    $\scriptscriptstyle#1$\cr
    \noalign{\kern0.1ex}
    $#2$\cr
}}}}
\def\mysublink[#1]#2#3{\mathrel{\vbox{\offinterlineskip\ialign{%
    \hfil##\hfil\cr
    $\scriptscriptstyle#2$\cr
    \noalign{\kern0.1ex}
    $#3$\cr
    \noalign{\kern-0.2ex}
    \smash{\raisebox{-\height}{\hbox{$\scriptscriptstyle #1$}}}\cr
    \noalign{\kern0.2ex}
}}}}

\def\fon[#1]{\cC_{#1}}

\def\mincompproj[#1]{\mfp_{#1}}

\def\Proj_#1{\mathop{\mathrm{Proj}_{#1}}}

\def\negrenorm[#1]{\mfR_{#1}}
\def\topnegrenorm[#1]{\overline{\mfR}_{#1}}

\def\quotedge[#1]{E^{q}_{#1}}

\def\posrenorm[#1]{\mcC_{#1}}
\def\topposrenorm[#1]{\overline{\mcC_{#1}}}
\def\cutsmod[#1]{\mathbb{C}_{+,#1}}

\def\fullcutsmod[#1]{\cut_{#1}}

\def\rem{{\mathop{\mathrm{rem}}}}
\def\ym{{\mathop{\mathrm{YM}}}}
\def\ren{{\mathop{\mathrm{ren}}}}

\def\emptyset{{\centernot\Circle}}

\colorlet{symbols}{blue!30!black!50}
\colorlet{testcolor}{green!60!black}
\colorlet{darkblue}{blue!60!black}
\colorlet{darkgreen}{green!60!black}
\definecolor{darkergreen}{rgb}{0.0, 0.5, 0.0}

\definecolor{purple}{rgb}{0.55,0.05,0.8}

\colorlet{redkernel}{red!80}

\def\symbol#1{{\mathbf{#1}}}

\def\1{\mathbf{\symbol{1}}}
\def\X{\symbol{X}}

\makeatletter
\pgfdeclareshape{crosscircle}
{
  \inheritsavedanchors[from=circle] % this is nearly a circle
  \inheritanchorborder[from=circle]
  \inheritanchor[from=circle]{north}
  \inheritanchor[from=circle]{north west}
  \inheritanchor[from=circle]{north east}
  \inheritanchor[from=circle]{center}
  \inheritanchor[from=circle]{west}
  \inheritanchor[from=circle]{east}
  \inheritanchor[from=circle]{mid}
  \inheritanchor[from=circle]{mid west}
  \inheritanchor[from=circle]{mid east}
  \inheritanchor[from=circle]{base}
  \inheritanchor[from=circle]{base west}
  \inheritanchor[from=circle]{base east}
  \inheritanchor[from=circle]{south}
  \inheritanchor[from=circle]{south west}
  \inheritanchor[from=circle]{south east}
  \inheritbackgroundpath[from=circle]
  \foregroundpath{
    \centerpoint%
    \pgf@xc=\pgf@x%
    \pgf@yc=\pgf@y%
    \pgfutil@tempdima=\radius%
    \pgfmathsetlength{\pgf@xb}{\pgfkeysvalueof{/pgf/outer xsep}}%  
    \pgfmathsetlength{\pgf@yb}{\pgfkeysvalueof{/pgf/outer ysep}}%  
    \ifdim\pgf@xb<\pgf@yb%
      \advance\pgfutil@tempdima by-\pgf@yb%
    \else%
      \advance\pgfutil@tempdima by-\pgf@xb%
    \fi%
    \pgfpathmoveto{\pgfpointadd{\pgfqpoint{\pgf@xc}{\pgf@yc}}{\pgfqpoint{-0.707107\pgfutil@tempdima}{-0.707107\pgfutil@tempdima}}}
    \pgfpathlineto{\pgfpointadd{\pgfqpoint{\pgf@xc}{\pgf@yc}}{\pgfqpoint{0.707107\pgfutil@tempdima}{0.707107\pgfutil@tempdima}}}
    \pgfpathmoveto{\pgfpointadd{\pgfqpoint{\pgf@xc}{\pgf@yc}}{\pgfqpoint{-0.707107\pgfutil@tempdima}{0.707107\pgfutil@tempdima}}}
    \pgfpathlineto{\pgfpointadd{\pgfqpoint{\pgf@xc}{\pgf@yc}}{\pgfqpoint{0.707107\pgfutil@tempdima}{-0.707107\pgfutil@tempdima}}}
  }
}
\makeatother

\colorlet{greennode}{green!50!black}
\colorlet{rednode}{red!50!black}
\colorlet{lbluenode}{blue!25}
\colorlet{dbluenode}{blue}
\colorlet{orangenode}{orange}

\definecolor{connection}{rgb}{0.7,0.1,0.1}

\tikzset{
dot/.style={circle,fill=black,inner sep=0pt, minimum size=1mm},
root/.style={circle,fill=black!50,inner sep=0pt, minimum size=3mm},
        var/.style={circle,fill=black!10,draw=black,inner sep=0pt, minimum size=1.6mm},
        delta/.style={densely dotted},
        var1/.style={rectangle,fill=black!10,draw=black,inner sep=0pt, minimum size=1.6mm},
        var2/.style={diamond,fill=black!10,draw=black,inner sep=0pt, minimum size=2mm},
        kernel/.style={semithick,shorten >=2pt,shorten <=2pt},
       kernel1/.style={postaction={decorate,decoration={markings,mark=at position 0.45 with {\draw[-] (0,-0.08) -- (0,0.08);}}}},
        kernels/.style={snake=snake,segment amplitude=1pt,segment length=4pt},
        rho/.style={densely dashed,semithick,shorten >=2pt,shorten <=2pt},
           testfcn/.style={dotted,semithick,shorten >=2pt,shorten <=2pt},
           tau/.style={circle,inner sep=1pt,draw=black,fill=white,text=black,thin},
        renorm/.style={shape=circle,fill=white,inner sep=1pt},
        labl/.style={shape=rectangle,fill=white,inner sep=1pt},
        xi/.style={very thin,circle,fill=lbluenode,draw=symbols,inner sep=0pt,minimum size=1.2mm},
        xi1/.style={very thin,rectangle,fill=lbluenode,draw=symbols,inner sep=0pt,minimum size=1.2mm},
        xi2/.style={very thin,diamond,fill=lbluenode,draw=symbols,inner sep=0pt,minimum size=1.6mm},
        xigreen/.style={very thin,circle,fill=greennode,draw=symbols,inner sep=0pt,minimum size=1.2mm},
        xigreen1/.style={very thin,rectangle,fill=greennode,draw=symbols,inner sep=0pt,minimum size=1.2mm},
        xired/.style={very thin,circle,fill=rednode,draw=symbols,inner sep=0pt,minimum size=1.2mm},
        xilblue/.style={very thin,circle,fill=lbluenode,draw=symbols,inner sep=0pt,minimum size=1.2mm},
        xidblue/.style={very thin,circle,fill=dbluenode,draw=symbols,inner sep=0pt,minimum size=1.2mm},
        xiorange/.style={very thin,circle,fill=orangenode,draw=symbols,inner sep=0pt,minimum size=1.2mm},
        xix/.style={crosscircle,fill=lbluenode,draw=symbols,inner sep=0pt,minimum size=1.2mm},
xix-green-red/.style={circle, fill=greennode!70!white,draw=rednode,inner sep=0pt,minimum size=1.6mm,append after command={node [inner sep=0pt,minimum size=0.8mm,thick, draw = rednode, cross out]{}}},
xix-green-red1/.style={rectangle, fill=greennode!70!white,draw=rednode,inner sep=0pt,minimum size=1.5mm,append after command={node [inner sep=0pt,minimum size=1mm,thick, draw = rednode, cross out]{}}},
	xib/.style={very thin,circle,fill=lbluenode,draw=symbols,inner sep=0pt,minimum size=1.6mm},
	xib1/.style={very thin,rectangle,fill=lbluenode,draw=symbols,inner sep=0pt,minimum size=1.6mm},
	xie/.style={very thin,circle,fill=greennode,draw=symbols,inner sep=0pt,minimum size=1.6mm},
	xid/.style={very thin,circle,fill=lbluenode,draw=symbols,inner sep=0pt,minimum size=1.6mm},
	xibx/.style={crosscircle,fill=lbluenode,draw=symbols,inner sep=0pt,minimum size=1.6mm},
	kernels2/.style={ultra thick,draw=symbols,segment length=12pt},
	not/.style={thin,regular polygon, regular polygon sides=3,draw=connection,fill=connection,inner sep=0pt,minimum size=1.2mm},
	notlblue/.style={thin,regular polygon, regular polygon sides=3,draw=lbluenode,fill=lbluenode,inner sep=0pt,minimum size=1.2mm},
	notorange/.style={thin,regular polygon, regular polygon sides=3,draw=orangenode,fill=orangenode,inner sep=0pt,minimum size=1.2mm},
	notgreen/.style={thin,regular polygon, regular polygon sides=3,draw=greennode,fill=greennode,inner sep=0pt,minimum size=1.2mm},
	>=stealth,
  }

\newtheorem{assumption}[lemma]{Assumption}

\colorlet{darkblue}{blue!90!black}
\colorlet{darkred}{red!90!black}
\colorlet{darkgreen}{green!70!black}
\let\comm\comment

\def\CI{\mcb{I}}
\def\A{\mathsf{A}}
\def\fancyC{\mathsf{C}}
\def\BM{\mathsf{W}}

\def\s{\mathfrak{s}}

\newcommand{\e}{\varepsilon}

\def\K{\mathfrak{K}}

\def\${|\!|\!|}
\def\Wick#1{\colon\!\! #1 \! \colon}

\DeclareMathOperator{\Ker}{Ker}

\def\?{{\color{red}?}}

\def\id{\mathrm{id}}

\def\restr{\mathbin{\upharpoonright}}

\newcommand{\VERT}{\vert\!\vert\!\vert}

\def\Hom{\mathord{\mathrm{Hom}}}

\def\Cas{\mathord{\mathrm{Cas}}}

\def\Func{\mathbf{F}}

\def\id{\mathrm{id}}
\newcommand{\TV}{\mathrm{TV}}

\def\bXi{\boldsymbol{\Xi}}
\def\bPsi{\boldsymbol{\Psi}}
\def\be{{\boldsymbol{\epsilon}}}
\def\bs{{\boldsymbol{s}}}
\def\bw{{\boldsymbol{w}}}

\def\dash{\leavevmode\unskip\kern0.18em--\penalty\exhyphenpenalty\kern0.18em}
\def\slash{\leavevmode\unskip\kern0.15em/\penalty\exhyphenpenalty\kern0.15em}

\newcommand{\floor}[1]{\lfloor #1 \rfloor}

\usepackage{stmaryrd}
\def\fancynorm#1{{\talloblong #1 \talloblong}}
\def\plaqnorm#1{{|\!|\!| #1 |\!|\!|}}

\newtheorem{example}[lemma]{Example}

\let\basepoint\logof
\def\logof{\mathord{{\basepoint}}} % Turns it into the behaviour of an ordinary math symbol

\title{Invariant measure and universality of the 2D Yang--Mills Langevin dynamic}
\author{Ilya Chevyrev$^1$, Hao Shen$^2$}

\institute{International School for Advanced Studies (SISSA), \email{ichevyrev@gmail.com} \and University of Wisconsin-Madison, \email{pkushenhao@gmail.com}}

\begin{document}
\maketitle
\begin{abstract}
We prove that the Yang--Mills (YM) measure for the trivial principal bundle over the two-dimensional torus, with any connected, compact structure group, is invariant for the associated renormalised Langevin dynamic. Our argument relies on a combination of regularity structures, lattice gauge-fixing, and Bourgain's method for invariant measures. Several corollaries are presented including a gauge-fixed decomposition of the YM measure into a Gaussian free field and an almost Lipschitz remainder, and a proof of universality for the YM measure that we derive from a universality for the Langevin dynamic for a wide class of discrete approximations. The latter includes standard lattice gauge theories associated to Wilson, Villain, and Manton actions. An important step in the argument, which is of independent interest, is a proof of uniqueness for the mass renormalisation of the gauge-covariant continuum Langevin dynamic, which allows us to identify the limit of discrete approximations. This latter result relies on Euler estimates for singular SPDEs and for Young ODEs arising from Wilson loops.
\\[.4em]
\end{abstract}
\setcounter{tocdepth}{2}

\tableofcontents

\section{Introduction}

We show that the Euclidean Yang--Mills (YM) measure associated to the trivial principal 
$G$-bundle over the two-dimensional (2D) torus $\T^2$ is invariant
for the corresponding renormalised Langevin dynamic, 
where $G$ is an arbitrary connected, compact Lie group.
The YM measure is formally given by the Gibbs-type probability measure
\begin{equ}\label{eq:YM_measure}
\mu(\mrd A) \propto e^{-S(A)}\mrd A
\end{equ}
where $S(A)$ is the YM action
\begin{equ}\label{eq:YM_action}
S(A) = \int_{\T^2} |F(A)|^2  \mrd x
\end{equ}
and $\mrd A$ is a formal Lebesgue measure on the space of connections on a principal $G$-bundle $P\to \T^2$.
Above, $F(A)$ is the curvature $2$-form of $A$.
In everything that follows, we consider $P$ as the trivial bundle $P=\T^2\times G$.
In this case, after choosing a global section, every connection can be written as a 
$\mfg$-valued $1$-form $A=(A_1,A_2) \colon \T^2\to \mfg^2$, where $\mfg$ is the Lie algebra of $G$.
In this case
\begin{equ}\label{eq:curvature}
F_{ij}(A)= \d_iA_j - \d_jA_i + [A_i,A_j]
\end{equ}
and the norm $|F(A)|$ in~\eqref{eq:YM_action} is understood as $|F_{12}(A)|^2$ where $|X|^2=\scal{X,X}$ and $\scal{\cdot,\cdot}$ is an $\Ad$-invariant inner product on $\mfg$, e.g.
if $G$ and $\mfg$ are subsets of unitary and skew-Hermitian matrices respectively, one can take $\scal{X,Y}=-\Trace(XY)$.

The 2D YM measure $\mu$ is one of the few quantum gauge theories that has been rigorously constructed.
This construction relies on an elegant exact solvability property, known as `invariance under subdivision',
initially observed in the physics literature by Migdal~\cite{Migdal75} (see also the work of Witten~\cite{Witten91}).
Mathematically, the YM measure on a compact surfaces has been constructed as a stochastic process indexed by Wilson loops in~\cite{Levy03,Sengupta97}, which built on earlier works~\cite{Driver89,GKS89} for the plane.
We refer to the surveys~\cite{Chatterjee18, Levy20,LevySengupta17} for further discussions and references.
The two constructions that are most relevant for us are that of L{\'e}vy~\cite{Levy06},
which identifies the YM measure associated to a given isomorphism class of principal $G$-bundles as the limit of discrete approximations,
and that of~\cite{Chevyrev19YM}, which constructs a gauge-fixed representation of 
the measure as a random distributional $1$-form (under the assumption that $G$ is simply connected).

A natural approach to construct and study Euclidean quantum field theories (EQFTs) is via stochastic quantisation.
The stochastic quantisation equation, or Langevin dynamic, for any measure of the form~\eqref{eq:YM_measure} is the stochastic partial differential equation (SPDE) 
\begin{equ}\label{eq:SQE}
\d_t A = -\frac12\nabla S(A) + \xi\;,
\end{equ}
where $\xi$ is a space-time white noise and $\nabla S$ is the gradient of $S$,
both taken with respect to some metric on the space of fields $A$.
There has been a lot of progress in recent years in using the Langevin dynamic~\eqref{eq:SQE}
to construct and study properties (e.g. Osterwalder--Schrader axioms) of scalar EQFTs, most notably for the $\Phi^4_3$ measure~\cite{AK20,GH21,Hairer_Steele_22,MW17Phi43,MoinatWeber20}
(see also~\cite{BG20,GM24} for related approaches).

For the YM measure, taking the natural $L^2$-metric on connections, the Langevin dynamic is
\begin{equ}\label{eq:SYM_no_coord}
\d_t A = -\mrd_A^* F(A) + \xi\;,
\end{equ}
where $\mrd_A^*$ is the adjoint covariant derivative.
This is (the 2D, finite volume version of)
the stochastic quantisation equation introduced by Parisi--Wu~\cite{ParisiWu}.
To indicate the type of nonlinearities in~\eqref{eq:SYM_no_coord}, we heuristically write the equation as
\begin{equ}\label{eq:SYM_no_DeTurck_heuristic}
\d_t A = \Delta A + \d^2 A + A\d A + A^3 + \xi\;,
\end{equ}
where $\Delta$ is the Laplacian and $\Delta + \d^2$ indicates a second order differential operator that is \textit{not} elliptic (i.e. has an infinite dimensional kernel).
The lack of ellipticity of $\Delta+\d^2$ is a well-known feature of YM theory
and stems from the invariance of the YM action under the infinite dimensional gauge group of bundle automorphisms,
i.e. the group of functions $g\colon \T^2\to G$ that act on connections by
\begin{equ}
A\mapsto A^g \eqdef gAg^{-1} - \mrd g g^{-1}\;.
\end{equ}
Gauge invariance is an important feature of YM theory which indicates that 
all physical quantities should depend on the gauge orbit $[A]\eqdef \{A^g\,:\,g\colon \T^2\to G\}$ instead of on $A$.
($A$ can be seen as a coordinate-based representation of $[A]$ coming from our initial choice of the global section of the bundle $P\to\T^2$.)

A celebrated method to bypass the lack of parabolicity of~\eqref{eq:SYM_no_coord} (see~\cite{Donaldson,DK90,Sadun,zwanziger81})
is to consider instead
the equation
\begin{equ}\label{eq:SYM_DeTurck_no_coord}
\d_t A = -\mrd_A^* F(A) - \mrd_A\mrd^* A + \xi\;,
\end{equ}
where $-\mrd_A\mrd^* A$ is called the \textit{DeTurck term}, after the work of DeTurck~\cite{deturck83} on the Ricci flow,
and has the heuristic form $\d^2A + A\d A$ that precisely cancels the $\d^2 A$ term in~\eqref{eq:SYM_no_DeTurck_heuristic}.
Furthermore, on the level of gauge orbits,~\eqref{eq:SYM_DeTurck_no_coord}
is equivalent to~\eqref{eq:SYM_no_coord} since any term $\mrd_A \omega$ for a $0$-form $\omega\colon \T^2\to \mfg$ is tangent to $[A]$ at $A$.

In coordinates,~\eqref{eq:SYM_DeTurck_no_coord} reads
\begin{equ}[eq:SYM_DeTurck_coord]
\d_t A_i = \Delta A_i + [A_j, 2\d_jA_i - \d_i A_j+[A_j,A_i]] + \xi_i \;,\quad i=1,2\;, 
\end{equ}
with implicit summation over $j=1,2$
and which we write in the sequel simply as
\begin{equ}\label{eq:SYM_heuristic}
\d_t A = \Delta A + A\d A + A^3 + \xi\;.
\end{equ}
We see at this stage that~\eqref{eq:SYM_heuristic} is \textit{singular} in the sense that we expect the solution $A$
to be a distribution of the same regularity as the linear stochastic heat equation (SHE),
which renders the products $A\d A$ and $A^3$ ill-defined.
It is therefore non-trivial to make rigorous sense of the SPDE~\eqref{eq:SYM_heuristic}.

A programme to solve and analyse~\eqref{eq:SYM_heuristic} was initiated in~\cite{CCHS_2D}.
It was shown therein that, given a space-time mollifier $\moll$ and any operator $C\in L(\mfg^2,\mfg^2)$,
the solution to the regularised and renormalised equation
\begin{equ}\label{eq:SYM_moll}
\d_t A = \Delta A + A\d A + A^3 + CA + \moll^\eps *\xi
\end{equ}
converges, as $\e\downarrow0$, locally in time to a Markov process with values in a suitable Banach space of distributional $1$-forms $\Omega^1_\alpha$
(which is similar to but stronger than the space introduced in~\cite{Chevyrev19YM}).
Here we write $\moll^\eps(t,x) = \eps^{-4}\moll(\e^{-2}t,\e^{-1}x)$ for the natural approximation of the Dirac delta built from $\moll$.
Furthermore, and more importantly, if $\moll$ is non-anticipative,
then there exists a choice for $C\in L_G(\mfg^2,\mfg^2)$, which depends explicitly on $\moll$ and takes the form $C(A_1,A_2)=(CA_1,CA_2)$,
such that the limiting dynamic $A$ projects to a Markov process $X\colon t\mapsto [A(t)]$ on the quotient space of gauge orbits $\mfO_\alpha = \Omega^1_\alpha/{\sim}$, where
$\sim$ is a canonical extension of gauge equivalence to $\Omega^1_\alpha$.
Here $L_G$ denotes the space of operators that commute with the adjoint action of $G$.
For a simple heuristic reason behind this `gauge-covariance' based on It\^o isometry, see~\cite[Sec.~2.2]{CCHS_2D} or the survey~\cite[Sec.~I.A]{Chevyrev22YM}.

It was conjectured in~\cite{CCHS_2D} that the YM measure for the trivial principal $G$-bundle is invariant for the Markov process $X$.
More precisely, the YM measure $\mu$ as constructed in~\cite{Levy06} 
can be seen as a probability measure on the space of functions $G^\mcL/G$, where $\mcL$ is a suitable class of loops $\ell\colon[0,1]\to\T^2$
(for concreteness, taken as all piecewise smooth loops based at $0\in \T^2$)
and $G^\mcL$ carries the diagonal adjoint action $G^\mcL\times G \to G^\mcL$, $(H,g)\mapsto \{\ell\mapsto gH(\ell)g^{-1}\}$.
On the other hand, every $A\in\Omega^1_\alpha$ canonically defines a `holonomy process' $H_A\in G^\mcL$
with the property that $A\sim B$ if and only if $gH_Ag^{-1}=H_B$ for some $g\in G$,
i.e. the orbit space $\Omega^1_\alpha/{\sim}$ from~\cite[Sec.~3.6]{CCHS_2D} can be identified with a subset of $G^\mcL/G$.
The precise conjecture is that the YM measure
$\mu$ gives full measure to  $\Omega^1_\alpha/{\sim}$
and is the unique invariant measure of the Markov process $X$.

The main result of this article answers this conjecture in the affirmative, thereby making the first rigorous link between the YM measure and its Langevin dynamic. 
Our argument is based on discrete regularity structures,
% in Sections~\ref{sec:DiscreteRS},~\ref{sec:Aeps},~\ref{sec:diagonal_argument})
lattice gauge fixing,
%  (for which we use ideas from~\cite{Chevyrev19YM} and rough paths theory),
% in Section~\ref{sec:moment_estimates} and Appendix~\ref{app:Gaussian_tails}),
and Bourgain's method for invariant measures.
We summarise the challenges of carrying out this approach and our methodology in Section~\ref{subsec:challeneges}.
% (Section~\ref{sec:invar_measure}).
%We further employ ergodicity properties of parabolic singular SPDEs~\cite{HM16,HS22}.

In our use of Bourgain's method~\cite{Bourgain94,Bourgain96}, we approximate the solution to~\eqref{eq:SYM_moll} through lattice dynamics, 
and one of the novelties in our approach is the way in which we identify the continuum limit of these discrete dynamics. 
Our argument for this is geometric, based on preservation of gauge symmetries,
and an important step is to show that there is a \textit{unique} mass renormalisation of the YM Langevin dynamic \eqref{eq:SYM_moll}
that preserves gauge covariance (see Theorem~\ref{thm:C_unique}).
% and Section~\ref{sec:gauge-covar}).
This step resolves another problem from~\cite{CCHS_2D},
where the question of uniqueness of the mass renormalisation constant $C\in L_G(\mfg^2,\mfg^2)$ was left open.\footnote{\cite[Thm.~2.9]{CCHS_2D} proved gauge covariance via a particular coupling argument, for which 
uniqueness of $C$ does hold. See also \cite[Theorem~1.11 and Remark~1.12]{Chevyrev22YM}.}
Our proof of this result, which we believe is of independent interest, relies on Euler-type estimates for singular SPDEs and for Young ODEs associated to Wilson loops,
combined with ideas from sub-Riemannian geometry. 
%(Appendix~\ref{app:reps}).

There are several consequences of our main result.
The first and simplest corollary is the long-time existence of the Markov process $X$ on gauge orbits associated to the YM Langevin dynamic as constructed in~\cite{CCHS_2D} (see Corollary~\ref{cor:long_time}). 

We further derive a number of consequences on the YM measure itself based on our main result.
Notably, we prove that
the scaling limit of a broad class of lattice gauge theories on $\T^2$ is the YM measure,
thereby establishing a general universality result for the YM measure.
More precisely, we provide conditions (Assumptions~\ref{assump:R} and~\ref{assum:P_N})
on the local action of a lattice gauge theory that ensures its scaling limit is the YM measure for the trivial principal bundle -
see Corollary~\ref{cor:universality}.
% for a precise statement.
%(see also Proposition~\ref{prop:criterion_S_N} for a method to verify Assumption~\ref{assum:P_N}).
%By `lattice gauge theory' we mean any probability measure of the form $\mcZ^{-1}\prod_p e^{S_N(U(\d p))} \mrd U$, where $p$ ranges over the plaquettes of a lattice
%and $S_N$ is a suitable class function on $G$.
This universality result covers the Wilson, Manton, and Villain actions (the result for Villain action was already known \cite{Levy06}).
A related result on $\R^2$ (or more simply $[0,1]^2$) for the Villain and Wilson actions, was shown in~\cite{Driver89}
by use of a central limit theorem (CLT) for group-valued random variables
(see also~\cite{BS83} for a type of universality related to $\tau$-continuum limits~\cite{Kogut79,KS75});
it appears difficult, however, to extend proofs of universality using the CLT
to surfaces with non-trivial topology, like $\T^2$ that we cover here.
This universality of the measure, which can be seen as a functional CLT for random connections, is derived from a corresponding universality for the dynamic (Theorem~\ref{thm:discrete_dynamics}).
To our knowledge, this is the first example where one can prove {\it universality} of the measure using uniqueness of the dynamic.

Another consequence is a gauge-fixed decomposition of the YM measure $\mu$ (Corollary~\ref{cor:decomposition}):
there exist, on the same probability space, random distributional $1$-forms $\Psi$ and $B$ such that $\Psi$ is a Gaussian free field (GFF) and $|B|_{\CC^{1-\kappa}}$ has moments of all orders for any $\kappa>0$, and such that
$[\Psi+B] \sim \mu$.
This decomposition rigorously establishes the YM measure as a perturbation of the GFF.
It furthermore implies strong regularity properties of the measure, improving the main result of~\cite{Chevyrev19YM} in two directions: (i)
it removes the restriction of axis-parallel lines in the norms of~\cite{Chevyrev19YM},
and (ii) it establishes moments bounds for these norms -- see Remark~\ref{rem:Che19_compare} for more details.

\begin{remark}
We use the construction of the YM measure from the Villain (heat kernel) action in~\cite{Levy03,Levy06}
(see also~\cite{Sengupta97}) as an important guide, but
we do not explicitly use this construction except to identify the invariant measure of $X$ with the well-known YM measure.
In particular, we show existence and uniqueness of the invariant measure using only SPDE techniques,
and therefore our results can be seen as providing a new construction and characterisation of the YM measure.
\end{remark}

\subsection{Challenges, novelties and methodologies}
\label{subsec:challeneges}

While our overall methodology towards the main results is described above, to carry on the proofs there are a number of essential obstacles  which  require novel ideas. We discuss these here.

1. As mentioned above, in order to use Bourgain’s invariant measure argument, we approximate by Langevin dynamics of lattice gauge theories.
These lattice theories have the crucial advantage of preserving the exact gauge symmetries but come with a cost of a complicated setup, in particular the fields need to take values in $G$ whereas in the limit \eqref{eq:SYM_DeTurck_no_coord} takes values in $\mfg$.
Heuristic justifications (c.f. \cite{Chatterjee18}) are usually done by an expansion using the Baker--Campbell--Hausdorff (BCH) formula.

Only using the BCH formula, however, turns out far from being enough for us to rigorously derive the approximate dynamic (for fixed $\eps>0$) on the level of the Lie algebra
since naively expanding the right-hand side of e.g. \eqref{eq:discrete_hat_U} using BCH and the logarithm map would make the structure of the dynamic obscure and difficult to write out.
We instead introduce a novel geometric derivation in which we endow the Lie algebra with a Riemannian manifold structure (not just a vector space) that is isometric with the Lie group around the identity, and pull back the action, noise, and other structures to the Lie algebra via this isometry.
This isometry allows us to perform BCH computations much more efficiently with the help of the derivative of the exponential map and avoid considering power series expansion of $\log$.
% Intuitively, we are substituting many computations based on the BCH formula and the logarithm map by computations using the derivative of the exponential map, which is a simpler object.
Furthermore, we streamline our computations by grouping terms that take values in a certain Lie ideal, which allows us to show that they contribute $O(1)$ remainder terms regardless of their precise forms.

This strategy turns out to be very robust and allows us to deal with a large class of lattice models (which appear to have quite distinct forms, see Section~\ref{subsec:assumptions}) simultaneously.
In fact, although we only prove convergence in 2D, we derive the approximate equations in {\it arbitrary} dimensions in Section~\ref{sec:lattice_dynamics}.
In summary, the approximate dynamic in the 2D Yang--Mills problem is significantly more involved than typical realisations of Bourgain’s argument
since in our case it is crucial to use gauge invariant lattice approximations while simple Fourier cutoffs or finite difference discretisations
break gauge symmetries.

2.  Our proof of convergence for the approximate dynamics relies on the discrete implementation of regularity structures. It turns out that this is not a direct application of the existing frameworks of \cite{EH19,HM18} (see also~\cite{EH21,Shen18}), and new aspects are necessary. We introduce shifting operators to the discrete framework to keep track of non-localities, and multiplication by $\eps$ on the abstract level to deal with various error terms. The latter is similar in spirit with \cite{KPZJeremy}, but in our ``semi-discrete’’ setting the multiplication by $\eps$ does not improve temporal regularity and additional effort is taken.  See Section~\ref{sec:DiscreteRS}.
% Moreover, in establishing the necessary stochastic estimates in Section \ref{subsec:moments_discrete_models}, we need to carefully examine the effect of the lattice discretisations, including asymmetry of discrete derivatives.
We believe that the discrete machineries we developed here are strong enough to prove limits of dynamics for more general models, such as non-abelian gauge theories incorporating a class of Higgs fields in various representations (along the line of \cite{SZZ24}) in 2D. 
%Remark that implementing the invariant measure argument of Bourgain with Higgs fields still has substantial challenge, see Section~\ref{subsec:related_works} below.
It is worth noting, however, that implementing Bourgain's invariant measure argument in the presence of Higgs fields remains a significant challenge (see Section~\ref{subsec:related_works}  below).

3.  One of the most intriguing and challenging questions is to {\it identify} the dynamic constructed in \cite{CCHS_2D} as the limit of our discrete dynamics. Theorem~\ref{thm:some_C} whose proof spans Sections~\ref{sec:DiscreteRS} - \ref{sec:diagonal_argument} only shows that  the limits of the discrete dynamics are solutions of \eqref{eq:SYM_moll} with {\it some} mass renormalisation $C$, but it does not tell us whether it coincides with the choice of  mass renormalisation in \cite{CCHS_2D}.
The challenge here is that this map $C$ arises as an aggregate of an enormous number of error terms which makes its numerical evaluation essentially impossible, and
we are not allowed to ``renormalise’’ the lattice gauge theories to tune this degree of freedom since doing that would break the gauge symmetry and make the problem meaningless. 
(In fact, we only show boundedness of these error terms and thus, a priori, $C$ may depend on a choice of subsequence $\eps\downarrow0$ of lattice spacings and on the underlying lattice gauge theory used to define the approximate dynamic.)
Note that this is an important contrast between our situation and the universality results of KPZ or $\Phi^4$ (e.g. \cite{KPZJeremy,MR4038037,MR3628883}).
In the latter cases, one has families of solutions parametrised by the limiting coefficients and finite shifts of renormalisation constants. For Yang--Mills, however, gauge invariance is much more rigid and we prove that there is only one single limit for all the models in the universality class.
Moreover, in the previous works on KPZ or $\Phi^4$, the renormalisation constants are essentially explicitly computable from the microscopic models, which is far from the case for us.

Surprisingly, overcoming the above difficulty of identifying the unique limit  requires certain ``intrinsic’’, namely topological and geometric arguments, which is the content of Section~\ref{sec:gauge-covar} and Appendix~\ref{app:reps}. The idea which allows us to conclude that the limits of lattice gauge theories cannot be anything else but the dynamic in \cite{CCHS_2D} is to use gauge invariant observable (Wilson loop) to detect the breaking of gauge covariance of any other possible limiting dynamic. In particular, we establish in Theorem~\ref{thm:A_tilde_A} a {\it lower} bound on the discrepancy of Wilson loop expectations  of two solutions at $t>0$ to such a dynamic starting from delicately chosen gauge equivalent initial conditions. 
The construction of such (gauge transformation of) initial conditions is achieved by tools from sub-Riemannian geometry, and the crucial lower bound is proved by perturbation argument of the SPDE and the holonomy ODE. To re-iterate, without this intrinsic step we would not be able to identify the limit and close the Bourgain’s argument (even for Villain case).
We believe that this result is of independent interest and the methodology possibly generalises to 3D.

4. Another key input to the Bourgain’s invariant measure argument is  that it requires sufficiently strong moment bounds on the invariant measure, for Markov inequality to apply. We obtain these moment bounds in Section~\ref{sec:moment_estimates}, Appendix~\ref{app:Gaussian_tails} and~\ref{app:RWs}.  This requires nontrivial extensions from \cite{Chevyrev19YM}, using rough paths theory, a quantitative rough version of Uhlenbeck compactness that uses a mix of Landau and axial gauges, and estimates on transition functions of random walks on groups inspired by \cite{Hebisch_Saloff_Coste_93}.

5. Another novelty is that we obtain universality of a large class of 2D Yang--Mills measures, by using the convergence of the {\it dynamics} combined with ergodicity results. The crucial step is our identification of the universal limiting measure as the unique invariant measure to the (unique, as discussed above) gauge covariant Yang--Mills Langevin dynamic.
Proving convergence of measures using convergence of the dynamics has a precursor, see the case of  Kac-Ising model \cite{MR3790155}, but we emphasise that we prove {\it universality} for a {\it wide} class of gauge theories and our limit is a ``singleton’’ rather than a family.

\subsection{Related works and open problems}
\label{subsec:related_works}

Some of the results of~\cite{CCHS_2D} were recently extended to 3D in~\cite{CCHS_3D},
including the construction of a state space of distributional $1$-forms and of a Markov process on gauge orbits corresponding to~\eqref{eq:SYM_no_coord}.
See~\cite{Chevyrev22YM} for a summary and comparison of the results of~\cite{CCHS_2D,CCHS_3D}.
See also~\cite{Sourav_flow,Sourav_state} for a different but related approach to a state space for 3D quantum YM.
In contrast to~\cite{CCHS_2D,Chevyrev19YM},
the state spaces in~\cite{Sourav_state,CCHS_3D} are based on the deterministic DeTurck--YM heat flow (see~\cite{CG13} for related ideas) and are nonlinear,
a property shown in~\cite{Chevyrev22_norm_inf} to be unavoidable.

The results of this article are restricted to the 2D pure YM theory, and identifying the invariant measure of~\eqref{eq:SYM_no_coord} in 3D remains an important open problem.
In fact, the construction of the 3D YM measure in the continuum is open, even in finite volume,
although progress was made in establishing
a form of ultraviolet stability~\cite{Balaban85IV,Balaban89,MRS93}.

The restriction to dimension 2 and to the \textit{pure} YM theory (i.e. no coupling with matter) arises from our use of Bourgain's invariant measure argument~\cite{Bourgain94} (see also~\cite[Cor.~1.3]{HM18}), which requires moment bounds on the invariant measure.
Improving the results of~\cite{Chevyrev19YM}, we derive in this article the necessary moment bounds
for lattice approximations of the 2D (pure) YM measure by exploiting its exact solvability together with ideas from rough paths.
It remains an open problem to extend these bounds to
higher dimensions or to other models, such as the 2D YM--Higgs model that has no known exact solvability property.
(The 2D and 3D Abelian Higgs models have been constructed in~\cite{BFS79,BFS80,BFS81,King86I,King86II},
but a construction in the non-Abelian case remains open even in 2D.)
However, see~\cite{ChandraChevyrev22} for progress on moment estimates for the Abelian Higgs model,
and~\cite{Shen18} where the stochastic quantisation equation of an Abelian gauge theory is solved using lattice approximations.
See also the recent progress \cite{BringmannCaoHiggs} on global existence of the Langevin dynamic for Abelian-Higgs on $\T^2$.
The discrete version of~\eqref{eq:SYM_no_coord}
was furthermore used in~\cite{SSZloop,SZZ22} to study lattice YM.

Another problem worth mentioning is the extension of the results of this article to non-trivial principal bundles, which would
require a solution theory for~\eqref{eq:SYM_moll} outside the `periodic' setting.
It would furthermore be of interest to extend the results from finite to infinite volume, i.e. from $\T^2$ to $\R^2$.

Finally, we mention recent progress on understanding invariance of EQFTs for corresponding Hamiltonian flows in the singular regime~\cite{Bringmann20,BDNY22, DNY19,OT20,OOT20,OOT21}, that parallels developments for the Langevin dynamic~\eqref{eq:SQE} (see~\cite{Bourgain94,Bourgain96} for earlier work). In particular \cite{GKOT22,ORT23} studies randomness from both initial data and stochastic forcing.
These results have largely been restricted to scalar theories (but see~\cite{BLS21,Bringmann_Rodnianski_23} for exceptions), and it would be of interest if similar results can be obtained for the YM measure;
Bringmann--Rodnianski~\cite{Bringmann_Rodnianski_23} recently exposed a potential obstruction to such a result.

\subsection*{Acknowledgements}

{\small
HS gratefully acknowledges support by NSF via DMS-1954091 and CAREER DMS-2044415.
IC gratefully acknowledges support by EPSRC via the New Investigator Award EP/X015688/1
and by the ERC via the Starting Grant SQGT 101116964.
We would like to thank Ajay Chandra and Martin Hairer for numerous discussions on closely related problems, and Rhys Steele for pointing out a potential technical issue on the norms in discrete regularity structures.
IC is also grateful to Massimiliano Gubinelli and Thierry L\'evy for discussions on gauge theories; and HS also would like to thank  Bjoern Bringmann and Sky Cao for some interesting conversations.
}

\section{Main results and preliminaries}

We state in this section our main results and corresponding notations.
We fix $d\geq 1$ %throughout this section 
(taken as $d=2$ further down).

\subsection{Notation}
\label{sec:Notation}

For linear spaces $E,F$, we let $L(E,F)$ denote the space of linear maps $E\to F$.

Let $\T^d=\R^d/\Z^d$ denote the $d$-dimensional torus which we identify as a set with $[0,1)^d$.
We equip $\T^d$ with the geodesic distance inherited from the Euclidean structure of $\R^d$, which we denote as $|x-y|$ for $x,y\in\T^d$.

For $N\geq 1$,
let $\Lambda_N\subset \T^d$\label{Lamdba_page_ref} denote the lattice with spacing $2^{-N}$,
which we identify with the set $\{k2^{-N}\,:\,k=0,\ldots, 2^N-1\}^d$.
Let $\bonds_N \eqdef \{(x,y)\in\Lambda_N^2\,:\, |x-y|=2^{-N}\}$,\label{bonds_page_ref}
denote the set of oriented bonds (also called edges) on the lattice.

When clear from the context, we will drop reference to $N$ in our notation and simply write $\Lambda$, $\bonds$, etc. and write $\eps = 2^{-N}$.

We denote $[d]=\{1,\ldots, d\}$ and write $\{e_j\}_{j\in [d]}$ for the
set of unit basis vectors in $\R^d$.
For $j\in [d]$ we will write 
$\eps_j$ as shorthand for $\eps e_j$.
%We write 
%$\CE_{\pm} \eqdef  \{\pm\e_j\}_{j\in [d]}$ and $\CE\eqdef \CE_+\cup \CE_-$.\label{page:CE_pm}

Let $\obonds_N\equiv \obonds \eqdef \{(x,x+\eps_j)\,:\, x\in\Lambda\,,\; j\in[d]\}\subset \bonds$,\label{obonds_page_ref}
denote the set of positively oriented bonds.
We also write 
$\obonds_j \eqdef \{(x,x+\eps_j)\,:\, x\in\Lambda\}$
so that $\obonds = \cup_{j\in[d]} \obonds_j$.
%(the clash of notation between $\obonds_N$ and $\obonds_j$ will not cause confusion).
For $e=(x,y)\in \bonds$ we write $e\pm \e_i \eqdef (x\pm\e_i,y\pm\e_i)$.

\begin{definition}
\label{def:plaquette}
An oriented plaquette, or sometimes just a plaquette,
is a tuple $p=(x,\e,\bar\e)$, where $x\in\Lambda$ is a lattice site and %$\e,\bar \e \in \CE$ 
$\e,\bar\e \in \{\pm\e_j\}_{j\in [d]}$
where $\e\neq \pm \bar \e$. 
We will also use the alternative notation 
\begin{equ}
p=(x_1,x_2,x_3,x_4)\eqdef (x,x+\e,x+\e+\bar\e,x+\bar\e)
\end{equ}
and
$p=(e^{(1)},e^{(2)},e^{(3)},e^{(4)})$ where $e^{(i)}=(x_i,x_{i+1})$ for $i=1,\ldots, 4$ with $x_5\eqdef x_1$,
since this plaquette $p$ can be also thought as the square enclosed by these four points or edges.
We will call $x$ the base point of $p$, and say that $\overleftarrow{p}\eqdef (x,\bar\e,\e)$ is the plaquette with the reversed orientation.
We denote by $\plaq\equiv \plaq_N$ the set of all plaquettes of $\Lambda$.
%For $p=(x,\e,\bar\e)\in\plaq$ and $i\in[d]$, we use the notation $p+\e_i \eqdef (x+\e_i,\e,\bar\e)\in\plaq$.
\end{definition}

We fix throughout the paper a connected compact Lie group $G$ and let $\mfg= T_\id G$ be its Lie algebra, which we identify with the space of left invariant vector fields on $G$.
Without loss of generality, we assume there exists $n\geq 1$ such that $G\subset \U(n)$,\label{page:G} the $n\times n$ unitary matrices,
and $\mfg\subset \mfu(n)$,\label{page:mfg} the $n\times n$ skew-Hermitian matrices.

Let $\conf\equiv \conf_N$ denote the set of all functions $U\colon\bonds\to G$ with the property that $U(x,y)=U(y,x)^{-1}$
for all $(x,y)\in\bonds$.
Remark that we can identify $\conf \simeq G^{\obonds}$ and view it as the configuration space of discrete gauge fields. 

Let further $\mfq\equiv \mfq_N$ denote the space of functions 
$A\colon \bonds\to \mfg$ with the property that $A(x,y)=- A(y,x)$ for all $(x,y)\in\bonds$.
Remark that $\mathfrak{q}$ can be identified with $\mfg^{\obonds}$
and thus with the Lie algebra $T_\id G^{\obonds}$ of $G^{\obonds}$.
We call elements of $\mfq$ discrete $\mfg$-valued $1$-forms, or simply $1$-forms.
We further write $\mfq_i = \mfg^{\obonds_i}$ so that $\mfq \simeq \oplus_{i=1}^d \mfq_i$.\footnote{The
clash of notation $\obonds_N$ vs. $\obonds_j$ and $\mfq_N$ vs. $\mfq_i$ will not cause confusion.}

Fix sufficiently small neighbourhoods $\mathring V\subset \mfg$\label{ring_V_page_ref} of $0$ and $\mathring W\subset G$\label{ring_W_page_ref} of $\id$
on which $\exp \colon \mathring V\to \mathring W$ is a diffeomorphism.
For $x\in G$, let $\Ad_x\colon\mfg\to\mfg$ denote the adjoint action.
We equip $\mfg$ with an $\Ad$-invariant inner product $\scal{\cdot,\cdot}_\mfg$
and equip $\mfq$ with the corresponding inner product\footnote{The factor $\e^{d-2}$ is such that the
right-hand side of \eqref{e:mfq-inner} is of order $1$, since $A,B$ here will be, in a sense, order $\eps$.}
\begin{equ}[e:mfq-inner]
\scal{A,B}_\mfq = \eps^{d-2}\sum\nolimits_{e\in\obonds} \scal{A(e),B(e)}_\mfg\;.
\end{equ}

%As in~\cite[Sec.~2.3]{Chevyrev19YM}, 
We fix a bounded, measurable map $\log\colon G\to \mfg$ such that (see~\cite[Sec.~2.3]{Chevyrev19YM} for existence of such a map $\log$)\label{page:log}
\begin{itemize}
\item $\log = \exp^{-1}$ when restricted to $\mathring W$,
\item $\exp(\log x)=x$ for all $x\in G$,
\item $\log(yxy^{-1})=\Ad_y(\log x)$ and $|\log x|_\mfg = |\log x^{-1}|_\mfg$ for all $x,y\in G$, and
\item $\log x=-\log (x^{-1})$ for all $x\in G$ outside a set with Haar measure zero.
\end{itemize}
\begin{remark}
We choose $\log$ to be defined on all of $G$ only for convenience; the more common choice of defining $\log$ only on $\mathring W$ would suffice for our purposes but causes some complications down the line.
\end{remark}
Define $W \eqdef \mathring W^{\obonds}\subset\conf$,\label{W_page_ref} the subset of functions $U$ on bonds taking values in $\mathring W$,
and similarly $V \eqdef \mathring V^{\obonds} \subset \mfq$.\label{V_page_ref}
Then $\exp \colon V\to W$
is still a diffeomorphism.

For a plaquette $p$, which we write as $p=(x_1,x_2,x_3,x_4)$ as in Definition~\ref{def:plaquette}, and $U\in\conf$, we define the {\it holonomy} of $U$ around $p$ as
\begin{equ}[e:U-pp]
U(\partial p) \eqdef U(x_1,x_2)U(x_2,x_3)U(x_3,x_4)U(x_4,x_1)\;.
\end{equ}

Let $\mfG\equiv \mfG_N$ denote the group of all functions $g\colon \Lambda\to G$.
Elements of $\mfG$ are called \emph{gauge transformation}.
There is a left group action of $\mfG$ on $\conf$ which we denote, for $g\in\mfG$ and $U\in\conf$, by
$(g,U) \mapsto g\act U \eqdef U^g$,
where
\begin{equ}\label{eq:gauge_transform}
U^g(x,y)= g(x)U(x,y)g(y)^{-1}\;.
\end{equ}
Remark that $U^g(\partial p)=g(x)U(\partial p)g(x)^{-1}$ for any plaquette $p=(x,\e,\bar\e)$.

For $U\in\conf$,
we write $[U]=\{U^g\,:\,g\in\mfG\}\subset \conf$, which is called the {\it gauge orbit} of $U$.
We say that $\bar U$ and $U$ are gauge equivalent and write $\bar U\sim U$, if $\bar U\in [U]$.
We write $\pi\colon Q\to Q/{\sim}$ for the  projection $\pi(U)=[U]$.\label{page:pi1}

\subsection{Approximations of the Yang--Mills measure and Langevin dynamic}
\label{sec:approx-YM}

In this subsection, we describe the types of lattice approximations of the YM measure that we consider together with their Langevin dynamics
that we show convergence to a continuum limit.
Fix again $d\geq 1$ in this subsection.

Let $G$ (resp. $Q$) be endowed with the bi-invariant Riemannian metric $\scal{\act,\act}$ coming from the $\Ad$-invariant inner product $\scal{\act,\act}_\mfg$ on $\mfg$ (resp. $\scal{\act,\act}_{\mfq}$) as fixed in Section~\ref{sec:Notation}.
(Note that this metric on $Q\simeq G^{\obonds}$ differs from the product metric by a factor $\e^{d-2}$.)
We also equip $G$ and $\conf\simeq G^{\obonds_N}$ with their Haar measures of total mass $1$.
We denote by $\int_G f(x)\mrd x$ the integral of a function $f$ on $G$ against the Haar measure, and similarly for $\conf$.

Let $S\equiv S_N \colon G \to \R\cup\{\infty\}$
be a measurable symmetric class function, i.e. $S(x)=S(x^{-1})$ and $S(yxy^{-1})=S(x)$ for all $x,y\in G$, that is twice differentiable on a fattening of $\mathring W$.
Define $\CS\equiv \CS_N\colon\conf\to \R\cup\{\infty\}$ by
\begin{equ}
\mcS(U)
=\frac18 \sum_{p\in\plaq} S[U(\partial p)]\;. 
\end{equ}
(The factor $1/8$ arises from the fact that $8$ plaquettes in $\plaq$ enclose the same square in $\T^d$.)
Consider the probability measure $\mu_N$ on $\conf$ defined by
\begin{equ}\label{eq:mu_N}
\mu_N(\mrd U) \eqdef Z^{-1} e^{-\mcS(U)} \mrd U\;,
\end{equ}
where $\mrd U$ is the Haar measure on $Q$ and
 $Z\equiv Z_N$ is the normalisation constant which makes $\mu_N$ a probability measure.
By the remark on $U^g$ below \eqref{eq:gauge_transform}, $\mcS$ and $\mu_N$ are invariant under the action 
by $\mfG$.
%(We allow $S,\CS$ to take values $\infty$ to allow for $\mu_N$ whose density $Z^{-1}e^{-\CS}$ takes values $0$ outside a neighbourhood of $1_G$.)
(We allow $S,\CS$ to take values $\infty$ so that the density $e^{-\CS}$ is allowed to be $0$ outside a neighbourhood of identity.)

\begin{remark}\label{rem:Up}
One should think of the plaquette variable $U(\partial p)$ as approximating (the exponential of) the curvature.
Writing $U=e^{\eps A}$,  the
Baker--Campbell--Hausdorff (BCH) formula (e.g. \cite[Sec.~3]{Chatterjee18}) yields $U(\partial p) \approx 
e^{\e^2 F_{ij}(A)(x)}$
%\[
%U(\partial p)=
%\exp \Big(\eps^2 \Big(\frac{\d A_k}{\d x_j} - \frac{\d A_j}{\d x_k} + [A_j(x), A_k(x)]\Big)+O(\e^3)\Big)
%\approx 
%e^{\e^2 F_{jk}(A)(x)}
%\]
for $p=(x,\e_i,\e_j)$ for $i\neq j\in [d]$
where $F(A)$ is as in~\eqref{eq:curvature}.
This is one of the motivations for viewing
$\e^{d-4} |\log x|_{\mfg}^2$ as the ``leading term'' in our class function $S_N$ (see \eqref{e:S=xR} below),
since the YM action in the continuum is $|F(A)|^2_{L^2}$.
\end{remark}
Recall that $\mfq \simeq \mfg^{\obonds}$ is identified with the left invariant vector fields on $\conf \simeq G^{\obonds}$.
For a differentiable function $f\colon Q\to \R$, the linear map $Df_U\colon T_U Q \to \R$ (i.e. ``differential'' of $f$ at $U$) is given  at $A\in \mfq \simeq T_U Q$ by
\begin{equ}
Df_U(A) = \frac{\mrd}{\mrd \alpha} \Big|_{\alpha=0} f(Ue^{\alpha A}) = \lim_{\alpha\to 0} \frac{f(Ue^{\alpha A}) - f(U)}{\alpha}\;,
\end{equ}
where the product $Ue^{\alpha A}$ in $\conf$ is understood via the identifications $\conf \simeq G^{\obonds}$ and $\mfq \simeq \mfg^{\obonds}$ (i.e. the exponential and product are taken over positively oriented bonds).

For any $U\in Q$, the gradient $\nabla \mcS(U)$ is the unique element of the tangent space $T_U Q$ such that $\scal{\nabla \mcS(U),X} = D \mcS_U(X)$ for all $X\in T_U Q$.
A natural diffusion to consider, which is a discrete version of~\eqref{eq:SQE}, is
\begin{equ}\label{e:ParisiWu}
\dt \tilde U_t = -\frac12\nabla \mcS(\tilde U_t) \mrd t + \tilde U_t (\circ \dt \tilde \BM_t)\;,
\end{equ}
where $\tilde \BM$ is a $\mfq$-valued Brownian motion
%$\E\int_0^T f_t \mrd_t\tilde \BM_t = \E\int_0^T|f_t|^2_\mfq\mrd t$
and we denote time and stochastic differentials by $\dt$ to distinguish from exterior and covariant derivatives $\mrd$ and $\mrd_A$.
By a $\mfq$-valued Brownian motion we mean a centred Gaussian process with covariance
$\E\scal{\tilde\BM_t,X}_\mfq\scal{\tilde\BM_s,Y}_\mfq = (s\wedge t)\scal{X,Y}_\mfq$ where we recall  \eqref{e:mfq-inner}.
%Indeed, the invariant measure of $\tilde U$ is, at least formally, $\mu_N$ defined by~\eqref{eq:mu_N}.

There are two issues with this definition of $\tilde U$:
\begin{enumerate}[label=(\roman*)]
\item the SDE~\eqref{e:ParisiWu} approximates the SPDE~\eqref{eq:SYM_no_DeTurck_heuristic}, which is not parabolic, and thus we have no way to establish its convergence, and
\item the function $S_N\colon G\to \R\cup\{\infty\}$ is assumed differentiable only on a neighbourhood of $\id$, so $\nabla \mcS(U)$ may not make sense for $U$ far from $\id$.
\end{enumerate}
To resolve these issues, we instead define a dynamic with an added DeTurck term and which is stopped upon exiting $W$.
Specifically, let $\Log \in \CC^\infty(G,\mfg)$ such that $\Log\restr_{\mathring{W}} = \log\restr_{\mathring{W}}$
and define $\mrd^* \in\CC^\infty(Q, \mfg^\Lambda)$ by
\begin{equ}\label{eq:disc_adjoint_der}
(\mrd^* U)(x)
\eqdef
\eps^{-1} \sum_{j=1}^d \big[ \Log U(x,x+\e_j) - \Log U(x-\e_j,x) \big] \;,
\end{equ}
which is a discrete approximation of the adjoint exterior derivative.
For $U \in \conf$ and $\omega \in \mfg^\Lambda$, define the infinitesimal transformation
\begin{equ}[e:d_U-omega]
(\mrd_U \omega)(x,y) \eqdef \eps^{-1}(\omega(x) U(x,y) - U(x,y)\omega(y))\;,
\end{equ}
which is a tangent vector at $U$ that
is tangent to the gauge orbit of $U$.

We further consider another class function $\check S\equiv \check S_N \colon G\to \R$ which is symmetric, twice differentiable
everywhere, agrees with $S_N$ on $\mathring W$,
and arbitrary otherwise.
We define $\check\mcS = \frac18 \sum_{p\in\plaq} \check S[U(\d p)]$
together with the dynamic $\check U\equiv \check U^{(N)}$ that solves
\begin{equ}\label{eq:discrete_hat_U}
\dt \check U_t = -\frac12\nabla \check \CS(\check U_t)\mrd t  - \mrd_{\check U_t} \mrd^* \check U_t\mrd t + \check U_t (\circ \dt \BM_t)\;,
\end{equ}
where $\BM$ is a $\mfq$-valued Brownian motion as in \eqref{e:ParisiWu}.
We then define the exit time
\begin{equ}\label{eq:varpi_def}
\varpi = \inf\{t\geq 0\,:\, \check U_t \notin W\}\;,
\end{equ}
and the process
\begin{equ}\label{eq:discrete_U}
U_t \equiv U^{(N)}_t \eqdef \check U_{t\wedge \varpi}\;.
\end{equ}
One of our main results (Theorem~\ref{thm:discrete_dynamics}) is that $\log U$ converges as $N\to\infty$ to the $\eps\downarrow0$ limit of~\eqref{eq:SYM_moll} identified in~\cite[Thm.~2.9]{CCHS_2D}.

The following proposition reveals the basic gauge-covariance properties of the above processes.
For $X\in\mfq$ and $g\in\mfG$, define $X^g\in \mfq$ by $
X^g_{xy} \eqdef g_y X_{xy} g_y^{-1}$ for all
$(x,y)\in\obonds$.
In particular, for a H\"older continuous process $\BM\colon[0,T]\to\mfq$ and a finite-variation process $g\colon[0,T]\to\mfG$, define $\BM^g\colon[0,T]\to\mfq$ by
$
\dt (\BM^g)_{xy} \eqdef g_y\dt \BM_{xy} g_y^{-1}$
for all $(x,y)\in\obonds$.
If $g$ is adapted and $\BM$ is a $\mfq$-valued Brownian motion, remark that $\BM^g$ is another $\mfq$-valued Brownian motion by It\^o isometry.

\begin{proposition}\label{prop:gauge_covar}
Suppose $S_N\colon G\to\R$ is twice differentiable on $G$ and $S_N=\check S_N$.

\begin{enumerate}[label=(\roman*)]
\item\label{pt:invar_measure} The diffusion~\eqref{e:ParisiWu} has invariant measure $\mu_N$ defined in \eqref{eq:mu_N}.
\item\label{pt:DeTurck_gauge_covar} Suppose $\tilde U$ satisfies~\eqref{e:ParisiWu}.
Define adapted finite-variation  $g\colon[0,T]\to\mfG$ by
\begin{equ}[e:disc-DT]
(\dt g) g^{-1} = -\eps^{-1} \mrd^*(\tilde U^g) \mrd t\;, \quad g_0=\id\;,
\end{equ}
(which admits a unique  solution since $\mrd^*\colon Q \to \mfg^\Lambda$ is smooth).
Then $\check U \eqdef \tilde U^g$ solves~\eqref{eq:discrete_hat_U} with $\BM\eqdef \tilde\BM^g$.

Conversely, given $\check U$ satisfying~\eqref{eq:discrete_hat_U}, there exists an adapted finite-variation $g\colon[0,T]\to\mfG$  such that $g_0=\id$ and $\tilde U \eqdef \check U^g$ solves~\eqref{e:ParisiWu} with $\tilde\BM\eqdef \BM^g$.

In particular, $[\check U_t] \eqlaw [\tilde U_t]$ for all $t\geq 0$ whenever $\check U_0=\tilde U_0$.

\item\label{pt:ic_gauge_covar} Let $\tilde U^{(u)}$ denote the solution to~\eqref{e:ParisiWu} with $\tilde U_0=u$.
If $v=u^h$ for some $h\in\mfG$, then $(\tilde U^{(u)})^h$ solves~\eqref{e:ParisiWu} with $\tilde\BM$ replaced by $\tilde\BM^h$. 

In particular, $[\tilde U^{(u)}_t] \eqlaw [\tilde U^{(v)}_t]$ for all $t\geq 0$ whenever $v\sim u$.

\item\label{pt:DeTurck_to_DeTurck}
Let $\check U^{(u)}$ denote the solution to~\eqref{eq:discrete_hat_U} with $\check U_0=u$.
If $v=u^h$ for some $h\in\mfG$,
then there exists an adapted finite-variation $g\colon[0,T]\to\mfG$ with $g_0=h$ such that $(\check U^{(u)})^g$ solves~\eqref{eq:discrete_hat_U} 
with $\BM$ replaced by $\BM^g$.  

In particular, $[\check U^{(u)}_t] \eqlaw [\check U^{(v)}_t]$ for all $t\geq 0$ whenever $v\sim u$.
\end{enumerate}
\end{proposition}

\begin{proof}
\ref{pt:invar_measure} is clear from general principles (see also~\cite[Lem.~3.3]{SSZloop})
while~\ref{pt:DeTurck_to_DeTurck} follows from~\ref{pt:DeTurck_gauge_covar} and~\ref{pt:ic_gauge_covar}
and the fact that $Q\times\mfG \ni (U,g) \mapsto U^g\in Q$ and $\mfq\times\mfG \ni (X,g) \mapsto X^g\in \mfq$ are left group actions.

\ref{pt:DeTurck_gauge_covar}
We distinguish for the moment between $\mfq$ and $T_UQ$ and write $UX\in T_UQ$ for the left-invariant vector field associated to $X\in\mfq$.
Note that $\nabla \mcS$ is gauge covariant in the sense that,
for any $g\in \mfG$,
\begin{equ}\label{eq:mcS_gauge_covar}
\nabla\mcS(U^g)=[\nabla\mcS(U)]^g \in T_{U^g}Q\;,
\end{equ}
by which we mean
$\nabla\mcS (U^g)(x,y) = g(x) \nabla \mcS(U)(x,y) g(y)^{-1}$ for every $(x,y)\in\obonds$.
Indeed, for any $X\in \mfq$ and $\alpha\in\R$, one has $\mcS(U^g e^{\alpha X})=\mcS((U^g e^{\alpha X})^{g^{-1}}) = \mcS(U \exp(\alpha X^{g^{-1}}))$ by gauge-invariance of $\mcS$.
Therefore
\begin{equs}
\scal{ \nabla\mcS (U^g), U^g X}
%= D\mcS_{U^g}(U^g X)
&=\frac{\mrd}{\mrd \alpha} \Big|_{\alpha=0} \mcS(U^g e^{\alpha X})
=\frac{\mrd}{\mrd \alpha} \Big|_{\alpha=0} \mcS(U \exp(\alpha X^{g^{-1}}))
\\
&= \scal{ \nabla\mcS (U),  U X^{g^{-1}} }
= \scal{(\nabla \mcS(U))^g, U^g X}
\end{equs}
where we used the bi-invariance of $\scal{\act,\act}$ in the final step and which proves~\eqref{eq:mcS_gauge_covar}.

Furthermore (recall \eqref{e:d_U-omega})
\begin{equ}\label{eq:mrd_U_g}
\dt (U^g) = (\dt U)^g + \eps\mrd_{U^g}(\dt g g^{-1})
\end{equ}
(which is valid for any semi-martingale $U$ and finite-variation $g$).
Therefore 
\begin{equs}
\dt \check U &= 
 (\dt \tilde U)^g + \eps \mrd_{\check U} (\dt g g^{-1})
\\
&= -\frac12\nabla \mcS( \check U)\mrd t  +  \check U (\circ \dt \BM)
- \mrd_{\check U} (\mrd^* \check U)\mrd t\;,
\end{equs}
where we used~\eqref{eq:mrd_U_g} in the first equality, and~\eqref{e:disc-DT}-\eqref{eq:mcS_gauge_covar}
in the second equality.
%and the identity for all $U\in Q,X\in\mfq,f\in\mfG$
%\begin{equ}\label{eq:UX_g}
%(U X)^f = U^f X^f\;.
%\end{equ}
This proves the first claim in~\ref{pt:DeTurck_gauge_covar}.
The second claim follows in the same manner.

\ref{pt:ic_gauge_covar}
This follows from~\eqref{eq:mcS_gauge_covar},~\eqref{eq:mrd_U_g},
and the fact that $h$ is constant in time.
%The process $V_t \eqdef (U^{(u)}_t)^h$  satisfies
%\begin{equ}
%\dt V = (\dt U^{(u)})^h = -\frac12 \nabla \mcS(V)\mrd t + V(\circ \dt \tilde \BM^h)\;,
%\end{equ}
%where we used again~\eqref{eq:mcS_gauge_covar},~\eqref{eq:mrd_U_g},~\eqref{eq:UX_g}, and that $h$ is constant in time.
\end{proof}

\subsection{Assumptions on the actions}
\label{subsec:assumptions}

%Restricting to $d=2$, 
We will work with a general function $S_N$ satisfying the properties mentioned at the start of Section~\ref{sec:approx-YM} and 
\begin{equ}\label{e:S=xR}
%\log e^{t\Delta}(x) 
S_N (x)= \eps^{d-4} |x,\id |_G^2+ R_{\eps}(x)\;,
\end{equ}
where $R_\e$ is a remainder that we soon discuss 
and $|\act ,\act|_G$ is the geodesic distance on $G$.
By standard considerations, note that $|x,\id|_G = |\log x|_\mfg$ for all $x$ in a small neighbourhood of $\id$.
%Recall that for a compact Lie group $G$ endowed with a bi-invariant metric, 
% the Lie-theoretic exponential map for $G$ coincides with the exponential map of this Riemannian metric;
% and thus if $x=\exp v$ for $v$ in a small neighbourhood of $0\in \mfg$, then we have 
% a geodesic $\gamma$ such that $\gamma(0)=\id$, $\gamma(1) = x$, $\dot \gamma(0)= v$.
%So,
%\begin{equ}
% |x,\id |_G =  |\gamma(1),\gamma(0) |_G  = |\dot \gamma(0)|_{\mfg} = |v|_\mfg  =|\log x|_\mfg  \;,
%\end{equ}
%where the second step is by definition of geodesic.

\begin{remark}
Since our metric on $\mfg$
is $\Ad$-invariant, and we assumed that $S_N$ is a class function invariant under inversion,
all the three terms in \eqref{e:S=xR} are class functions invariant under inversion for all $x$ outside of a null set.
\end{remark}
The assumption we make on $R_\e$ is the following. Define the pullback function
\begin{equ}
\hat R_\e \colon \mfg \to \R,\quad \hat R_\e (X) = R_\e(e^X)\;.
\end{equ}

\begin{assumption}\label{assump:R}
There exist linear maps $E^{(1)}_{\e} \in L(\mfg , \mfg^*)$ and $E^{(2)}_{\e}\in L(\mfg^{\otimes 3} , \mfg^*)$,
such that 
$
\sup_{\e\in (0,1]} |E^{(1)}_\e| + |E^{(2)}_\e| < \infty
$
and, uniformly in $\eps\in (0,1]$ and $X\in\mathring V$,
\begin{equ}[eq:DR_expan]
D\hat R_{\e}(X) = E^{(1)}_{\e}(X) + \e^{d-4} E^{(2)}_{\e}(X^{\otimes 3}) + O(X^2) + \e^{d-4}O(X^4)\;.
\end{equ}
We furthermore assume that $E^{(2)}_\e$ is $\Ad_G$ covariant in the sense that
\[
E^{(2)}(\Ad_g X_1\otimes\Ad_g X_2\otimes \Ad_gX_3)(\Ad_gY) = E^{(2)}( X_1\otimes X_2\otimes X_3)(Y)\;,
\]
for all $X_i\in V$, $Y\in\mfg$, and $g\in G$.
\end{assumption}
\begin{remark}\label{rem:big-O}
In \eqref{eq:DR_expan} and the rest of the paper,
for a normed space $(\CB,\|\cdot\|)$, $X\in \CB$, and $n\geq 1$, we let
$O(X^n)$ denote $Y\in \CB$ such that $\|Y\|=O(\|X\|^n)$.
\end{remark}

See Section~\ref{subsec:rem_action} and Remark~\ref{rem:power_of_eps} for the motivation behind Assumption \ref{assump:R}.

Remark that the final assumption on $\Ad_G$ covariance is not a restriction and can be assumed without loss of generality.
This is because $S_N$ is a class function which implies $\hat R_{\e}(X) = \hat R_{\e}(\Ad_g X)$ and thus $D\hat R_{\e}(\Ad_g X)(\Ad_g Y) = D\hat R_{\e}(X)(Y)$.
Therefore, we can replace every term $f(X)$ on the right-hand side of \eqref{eq:DR_expan} by $\tilde f(X)(Y) = \int_G f(\Ad_h X)(\Ad_h Y)\mrd h$, which is now $\Ad_G$ covariant.

We show below that some common actions, namely the Manton, Wilson, and Villain actions, satisfy Assumption~\ref{assump:R}.

\begin{remark}
%It is natural that we assume no term proportional to $X^2$ appears in the expansion for $D\hat R_\eps(X)$
%since $\hat R_\eps$ is even in $X$, and therefore $D\hat R_\eps$ is odd.
In Assumption~\ref{assump:R} 
the ``leading terms'' $E^{(1)}_{\e}(X)$ and $\e^{d-4} E^{(2)}_{\e}(X^{\otimes 3})$ are odd in $X$.  
This is natural 
since $\hat R_\eps$ is even in $X$, and therefore $D\hat R_\eps$ is odd.
\end{remark}

\subsubsection{Examples}
\label{subsubsec:example_actions}

\textbf{Manton.}
The simplest action that satisfies Assumption~\ref{assump:R} is the Manton action~\cite{Manton80} where $S_N(x)=\e^{d-4}|x,\id|_G^2$, so that $\hat R_\e(x)=0$ and $E^{(1)}_\e=0$ and $E^{(2)}_\e=0$.
Remark that $S_N$ in this case is \textit{not} differentiable on all of $G$, e.g. for $G=\U(1)$, $S_N$ is not differentiable at $-1$.

\medskip
\noindent
\textbf{Wilson.}
Suppose that the inner product on $\mfg\subset \mfu(n)$ is given by $\scal{X,Y}_\mfg = -\frac12\Trace{XY}$.
The Wilson action, which is a smooth function $S_N\colon G\to [0,\infty)$, is
\[
S_N(x) = \eps^{d-4}\Re\Trace (\id-x)\;.
\]
%In view of Remark~\ref{rem:Up}, this formally approximates the $L^2$ norm of curvature $F(A)$: 
%\begin{equ}
%\eps^{d-4} \Re\Trace(U(\partial p) - \id) = \frac{1}{2}\e^d\Trace(F_{jk}(A)(x)^2)+O(\e^{d+1})\,,
%\end{equ}
%noting that $\Re \Trace F(A) =0$ since $F(A)$ is skew-Hermitian.

\begin{remark}
See \cite{SSZloop,SZZ22} for explicit forms of the Langevin dynamic arising from the Wilson action for $\SO(n)$, $\U(n)$, and $\SU(n)$.
The results in this paper yield convergence of 
these dynamics on the two dimensional torus. 
\end{remark}
%Taking $d=2$, 
We have for $x\in \mathring W$ 
\begin{equs}
S(x) &= \eps^{d-4}\Re\Trace (\id-x)
= \eps^{d-4}\Re\Trace(\id - e^{\log x})
\\
&= -\eps^{d-4}\Re \Trace (\log(x) + (\log x)^2/2! + (\log x)^3/3! + (\log x)^4/4! + \ldots)
\\
&= \eps^{d-4} |\log x|_\mfg^2 - \eps^{d-4} \sum_{k\geq 2}  \Trace (\log x)^{2k}/(2k)!
= \eps^{d-4} |\log x|_\mfg^2 + R_{\eps}(x) \;,
\end{equs}
where we used that $|X|_\mfg^2 = -\frac12\Trace(X^2)$.
%Therefore
%\begin{equ}
%R_{\eps}(x) = - \eps^{-2} \sum_{k\geq 2}  \Trace (\log x)^{2k}/(2k)!
%\end{equ}
%and
Here $R_{\eps}(x)$ is defined in the last line and
\begin{equ}
\hat R_\e(X) = R_\e(e^X) 
= - \eps^{d-4} \sum_{k\geq 2}  \frac{1}{(2k)!}\Trace X^{2k} 
= -\eps^{d-4}\Trace (\cosh(X) - 1 - \frac{X^2}{2!})\;.
\end{equ}
It follows that, for all $Y\in\mfg$,  
\begin{equ}
D\hat R_\e(X)(Y) = -\eps^{d-4} \frac{1}{3!} \Trace (X^3 Y) + \eps^{d-4}O(X^5 Y)\;.
\end{equ}
Therefore the Wilson action satisfies Assumption~\ref{assump:R} with $E^{(1)}_\e=0$ and
\begin{equ}
E^{(2)}_\e(X_1\otimes X_2 \otimes X_3)(Y) 
%=  -\sum_{i,j,k}  \frac{1}{(3!)^2}\Trace(X_iX_jX_kY)\;,
=- \frac{1}{3!} \Trace (X_1X_2X_3 Y)\;.
\end{equ}

\medskip
\noindent
\textbf{Villain.}
Restricting to $d\leq 3$, the Villain (heat kernel) action $S_N\colon G\to\R$ is
\begin{equ}
e^{-S_N(x)-c_\e} \propto e^{\frac14\e^{4-d} \Delta}(x)\;,
\end{equ}
where $\Delta$ is the Laplace--Beltrami operator on $G$,
$e^{t\Delta}$ is the associated heat kernel at time $t>0$,
and $c_\e$ does not depend on $x$ and is chosen so that $S_N(\id)=0$.
Like the Wilson action, the Villain action is smooth on $G$.
Recall the Varadhan formula
\begin{equ}\label{eq:Varadhan_formula}
\log e^{\frac14\e^{4-d} \Delta}(x) = -\e^{d-4}|\log x|_\mfg^2 - \bar R_\e(x)
\end{equ}
where $\lim_{\e\to0}\e^{4-d} \bar R_{\e}(x) = 0$ uniformly in $x$.
Remark that $\bar R_\e(x) = R_\eps(x) + c_\e$ for $R_\e$ as in~\eqref{e:S=xR}.
Stroock--Turetsky~\cite[Thm.~4.1]{Stroock_Turetsky_97} (see also~\cite{Malliavin_Stroock_96,Stroock_Turetsky_98})
showed furthermore that the limit $\e^{4-d} \bar R_\e(x) \to 0$ commutes with gradients:
their result implies that, for any collection of smooth vector fields $X_1,\ldots,X_k$ on $G$,
\begin{equ}\label{eq:R_grad_bound}
\sup_{\e\in(0,1]}\sup_{x\in\mathring W}|(X_k\circ \cdots \circ X_1) \bar R_\e (x)| <\infty \;.
\end{equ}
Since $\bar R_\e$ and $R_\e$ differ by a constant,~\eqref{eq:R_grad_bound} holds also for $R_\e$, which implies, for all smooth vector fields $Y_1,\ldots, Y_k$ on $\mfg$,
$
\sup_{\e\in(0,1]}\sup_{X\in\mathring V}|(Y_k\circ\cdots\circ Y_1)\hat R_\e(X)| <\infty
$,
which in particular implies for every $k\geq 1$
\begin{equ}\label{eq:R_deriv_bound}
\sup_{\e\in(0,1]}\sup_{X\in\mathring V}|D^k \hat R_\e(X)| <\infty \;,
\end{equ}
where $D^k \hat R_\e \colon \mfg \to L(\mfg^{\otimes k},\R)$ is the $k$-th derivative.
%and the norm $|D^k \hat R_\e(X)|$ is taken as the operator norm once we equip $\mfg^{\otimes k}$ with an arbitrary cross norm (e.g. the Hilbert tensor product).
It follows that
\begin{equ}
D\hat R_\e(X) = D\hat R_\e(0) + D^2\hat R_\e(0)(X) + O(|X|^2)\;,
\end{equ}
where the proportionality constant in $O(|X|^2)$ is uniform in $\e\in (0,1]$ and $X\in\mathring V$ due to~\eqref{eq:R_deriv_bound}.
Due to the symmetry $\hat R_\e(X)=\hat R_\e(-X)$, which itself follows from $e^{\frac14\e^{4-d}\Delta}(x)=e^{\frac14\e^{4-d}\Delta}(x^{-1})$ and $|\log x|_\mfg=|\log x^{-1}|_\mfg$,
we obtain
$D\hat R_\e(0) = 0$.
Since $\sup_{\e\in (0,1]}|D^2\hat R_\e(0)| < \infty$
we conclude that
$\hat R_{\e}$ satisfies Assumption \ref{assump:R}
with $E_{\e}^{(1)} \eqdef D^2\hat R_{\e}(0) \in L(\mfg\otimes \mfg,\R)\simeq L(\mfg,\mfg^*)$ and $E^{(2)}_\e=0$.
%\begin{equ}
%D\hat R_{\e}(X) = E_\e^{(1)}(X) +
% O(|X|^2)\;,
%\end{equ}
%where  
%$E_{\e}^{(1)} \eqdef D^2\hat R_{\e}(0) \in L(\mfg\otimes \mfg,\R)\simeq L(\mfg,\mfg^*)$ and $\sup_{\e\in (0,1]} |E_{\e}^{(1)}|<\infty$,
%so Assumption~\ref{assump:R} is satisfied with $E^{(2)}_\e=0$.

\subsection{Lattice and continuum \texorpdfstring{$1$}{1}-forms}
\label{sec:norms}

We recall in this section the norms on lattice and continuum $1$-forms from~\cite{Chevyrev19YM, CCHS_2D}.
Let $d\geq 1$ and define the sets
\begin{equs}
\olines_N  &= \{(x,ke_i 2^{-N}) \in \T^d\times \R^d\,:\, x\in\Lambda_N \,,\, k\in \{0,\ldots, 2^{N-1}\}\,,\, i\in [d]\}\;,
\\
\lines_N &= \olines_N \cup \{(x,-h)\,:\, (x,h)\in\olines_N\}\;.
\end{equs}
For $\ell=(x,ke_i 2^{-N})\in\olines_N$, we call $|\ell| \eqdef |k|2^{-N}$ the length of $\ell$.
The set $\olines_N$ can be interpreted as the set of (positively oriented) lines on the lattice $\Lambda_N$ of length at most $\frac12$.
We will frequently identify every $\ell=(x,h)\in\olines_N$ with the set $\{x+th\,:\,t\in[0,1]\}\subset \T^d$ (which completely characterises $\ell$).

We say that $\ell,\bar \ell\in\olines_N$ are parallel, and write $\ell\parallel\bar \ell$, whenever $\ell=(x,ke_i2^{-N})$ and $\bar\ell=(x+m e_j 2^{-N}, ke_i2^{-N})$ for some $j\neq i$ and $|m|\leq 2^{N-1}$.
We write $d(\ell,\bar\ell)\eqdef m2^{-N}$, which is the Hausdorff distance between $\ell$ by $\bar \ell$ treated as subsets of $\T^d$.
We further define the quantity
\begin{equ}
\rho(\ell,\bar\ell) = |\ell|^{1/2}d(\ell,\bar\ell)^{1/2}\;.
\end{equ}
The area of the (smallest) rectangle with two of its sides given by $\ell,\bar\ell$ is $\rho(\ell,\bar\ell)^2$.

Every $A\in \mfq_N$ defines a function still denoted by $A\colon \olines_N \to \mfg$ as
\begin{equ}\label{eq:A_def_1_form}
A(x,ke_i2^{-N}) \eqdef \sum_{j=1}^{k}A(x+(j-1)\e_{i},x+j\e_i)\;.
\end{equ}
We denote by $\Omega_N$   \label{page:Omega_N}
the set of all functions $A\colon \olines_N \to \mfg$ arising this way, which is clearly in bijection with $\mfq_N\simeq \mfg^{\obonds_N}$.
We extend every $A\in\Omega_N$ to a function $A\colon \lines_N \to \mfg$ given by requiring $A(\overleftarrow{\ell})=-A(\ell)$ where $\overleftarrow{(x,h)} = (x+h,-h) \in\lines_N$ is the reversal of $(x,h)\in\lines_N$.

Recalling $\log\colon G\to \mfg$ from Section~\ref{sec:Notation},
every $U\in Q_N$ defines a function $A\in\mfq_N$ by $A(b) = \log U(b)$ for all $b\in\obonds_N$ and thus defines an element of $\Omega_N$.
Conversely, every $A\in\Omega_N$ defines an element $U\in Q_N$ by $U(b)=e^{A(b)}$ for all $b\in\bonds_N$.
We correspondingly write $A\sim B$ for $A,B\in\Omega_N$ whenever $e^A\sim e^B$
and denote $[A] \eqdef \{\omega\in\Omega_N\,:\,\omega\sim A\}$.
For $g\in\mfG_N$ and $A\in\Omega_N$ we further denote $A^g \eqdef \log [(e^A)^g]$.
Remark that $A^g\sim A$, but, if $B\in \Omega_N$ is outside the image of $\log$, it is possible that $B\sim A$ while $B \neq A^g$ for any $g\in \mfG_N$.
We denote by $\pi\colon\Omega_N\to\Omega_N/{\sim}$ the projection
$\pi(A)=[A]$.\label{page:pi2}

Fix $\alpha\in [0,1]$ throughout this subsection.
We define on $\Omega_N$ the H\"older-type norm $|\cdot|_{N;\gr\alpha}$ and semi-norm $|\cdot|_{N;\alpha;\rho}$ by
\begin{equs}
|A|_{N;\gr\alpha} &= \sup\{ |\ell|^{-\alpha}|A(\ell)| \,:\, \ell\in\olines_N\,,\, |\ell|>0\}\;,
\\
|A|_{N;\alpha;\rho} &= \sup\{ \rho(\ell,\bar\ell)^{-\alpha} |A(\ell)-A(\bar\ell)| \,:\, \ell,\bar\ell\in\olines_N\,,\,
\ell\parallel\bar\ell\}\;.
\end{equs}
We then denote by $\Omega_{N;\alpha}$ \label{page:Omega_Nalpha}
the space $\Omega_N$ equipped with the norm
$
|\cdot|_{N;\alpha} \eqdef |\cdot|_{N;\gr\alpha}+|\cdot|_{N;\alpha;\rho}
$.
For $M\leq N$, we clearly have inclusions $\olines_M\subset \olines_N$, which gives the natural projections $\pi_{N,M}\colon\Omega_{N} \twoheadrightarrow \Omega_{M}$ by restriction.
Explicitly,		\label{page:pi_NM}
\begin{equ}
(\pi_{N,M}A)(x,e_i 2^{-M}) = A(x,e_i 2^{-M})
\end{equ}
for all $x\in \Lambda_M$ and $i\in[d]$.
Furthermore $\pi_{N,M}$ is clearly a contraction $\Omega_{N;\alpha}\to \Omega_{M;\alpha}$.
This endows $\{\Omega_{N;\alpha}\}_{N\geq 1}$ with the structure of a projective system of Banach spaces.
The projective limit $\varprojlim_{N\to\infty} \Omega_{N;\alpha}$ is naturally described as follows.

\begin{definition}\label{def:barL}
We denote by $\olines = \cup_{N=1}^\infty \olines_N$ the set of positively oriented lines in $\T^d$ of length $\leq \frac12$ whose starting and ending points are dyadic rationals.
We say that $\ell,\bar\ell\in\olines$ are \textit{joinable} if $\ell=(x,e_ih)$ and $\bar\ell = (x+e_ih,e_i\bar h)$ for some $x\in\Lambda_N$, $i\in[d]$ and $h,\bar h\geq 0$ such that, treated as subset of $\T^d$, $\ell\cup\bar\ell\in\olines$,
i.e. if $|h+\bar h| \leq \frac12$.
We say that a function $A\colon \olines \to \mfg$ is additive if $A(\ell\cup\bar\ell)=A(\ell)+A(\bar\ell)$ for all joinable $\ell,\bar\ell\in\olines$.
We denote by $\Omega$ the vector space of additive functions.
\end{definition}
With these definitions, it is clear that $\Omega = \varprojlim_{N\to\infty} \Omega_N$ as a projective limit of vector spaces.
We define the (extended) norm $|\cdot|_{\gr\alpha} = \lim_{N\to \infty}|\cdot|_{N;\gr\alpha}$ and (extended) semi-norm $|\cdot|_{\alpha;\rho} = \lim_{N\to \infty}|\cdot|_{N;\alpha;\rho}$ on $\Omega$, which are well-defined since $|\cdot|_{N;\gr\alpha}$ and $|\cdot|_{N;\alpha;\rho}$ are increasing in $N$.
Define the Banach spaces
\begin{equ}
\Omega_{\gr\alpha} = \{A\in\Omega\,:\, |A|_{\gr\alpha}<\infty\}\;,\quad
\Omega_{\alpha} = \{A\in\Omega\,:\, |A|_{\gr\alpha}+|A|_{\alpha;\rho}<\infty\}\;,
\end{equ}
where $\Omega_{\gr\alpha}$ is equipped with the norm $|\cdot|_{\gr\alpha}$,
and $\Omega_{\alpha}$ is
equipped with the norm
$|\cdot|_{\alpha}\eqdef |\cdot|_{\gr\alpha}+|\cdot|_{\alpha;\rho}$.
Then $\Omega_{\alpha} = \varprojlim_{N\to\infty} \Omega_{N;\alpha}$ as a projective limit of Banach spaces.\label{page:Omega_alpha}
We let $\pi_N \colon \Omega_\alpha\to \Omega_{N;\alpha}$ denote the corresponding projection.

%The set
%\begin{equ}
%\alllines = \{(x,he_i)\,:\,x\in\T^d\,,\,h\in[0,1/2]\,,\, i\in[d]\}
%\end{equ}
%of all positively oriented lines of length at most $\frac12$ is 
%the closure of $\olines$ in the Hausdorff distance $d_\Haus$
%(as before, we identify $\ell=(x,he_i)\in\alllines$ with the subset $\{x+te_ih\,:\,t\in[0,1]\}\subset \T^d$).
%One can show (see~\cite[Prop.~3.9]{Chevyrev19YM}) that there exists $C>0$ such that, for all $\alpha\in[0,1]$ and $\ell,\bar\ell\in\olines$,
%$
%|A(\ell)-A(\bar\ell)| \leq C|A|_\alpha d_\Haus(\ell,\bar\ell)^{\alpha/2}
%$.
%Therefore, every $A\in\Omega_\alpha$ with $\alpha>0$ extends uniquely to a continuous map $A\colon(\alllines,d_\Haus) \to \mfg$
%(the same is not true for elements of $\Omega_{\gr\alpha}$).

For a vector space $\mcE$ of $\mfg$-valued distributions (in the sense of Schwartz),
let $\Omega \mcE$\label{page:Omega_mcE} denote the space of $\mfg$-valued distributional $1$-forms $A=(A_1,\ldots,A_d)$ with $A_i\in \mcE$.
If $\mcE$ carries a norm $|\cdot|_\mcE$, we denote
$|A|_{\mcE}\eqdef \max_{i\in[d]}|A_i|_{\mcE}$.
For  the space $\CC$  of $\mfg$-valued continuous functions
there exists an injection $\imath \colon \Omega\CC\to \Omega$
given for $A=(A_1,\ldots,A_d)\in\Omega\CC$ by
$
(\imath A)(x,e_ih) = \int_0^{|h|} A_i(x+te_i) \mrd t
$.
It is easy to see that
\begin{equ}\label{eq:cont_embedding}
|\imath A|_{\alpha} \lesssim |A|_{\CC^{\alpha/2}}\;.
\end{equ}
Here and below, we let $\CC^\gamma$ for $\gamma\in\R$ denote the Banach space of H\"older--Besov distributions on $\T^d$ as in~\cite[Sec.~1.5]{CCHS_2D}.
Conversely, if $\alpha>0$, then one also has for $A\in\Omega\CC$
\begin{equ}[eq:dist_embedding]
|A|_{\CC^{\alpha-1}} \lesssim |\imath A|_{\gr\alpha}
\end{equ}
(see \cite[Prop.~3.21]{Chevyrev19YM}).
It follows from~\eqref{eq:cont_embedding} that one can view $\Omega\CC^\infty$ as a subspace of $\Omega_\alpha$ via $\imath\colon \Omega\CC^\infty \to \Omega_\alpha$
and we let $\Omega^1_\alpha$\label{page:Omega_1_alpha} denote the closure of $\Omega\CC^\infty$ in $\Omega_\alpha$.
Moreover, by~\eqref{eq:dist_embedding},
one can view $\Omega^1_\alpha$ as a subspace of $\Omega\CC^{\alpha-1}$ via
$\imath^{-1}\colon\Omega^1_\alpha\to\Omega\CC^{\alpha-1}$.
In the sequel, we will not write $\imath$ and $\imath^{-1}$ explicitly and will simply treat $\Omega^1_\alpha$
as a space of distributional $1$-forms.

%More generally, every $A\in\Omega_{\gr\alpha}$ such that $A$ is continuous on $(\olines,d_\Haus)$
%canonically defines a distributional $1$-form $(A_1,\ldots, A_d)$ with $A_i\in\mcS'(\T^d,\mfg)$ for which $\max_{i\in [d]}|A_i|_{\CC^{\alpha-1}}\lesssim |A|_{\gr\alpha}$ and for which $A_i=0$ for all $i\in[d]$ if and only if $A\equiv 0$ -- see~\cite[Propositions~3.15,~3.21]{Chevyrev19YM}.

Denote $\mfG^{\alpha}=\CC^{\alpha}(\T^d,G)$\label{mfG_page_ref}
and let $\mfG^{0,\alpha}$\label{mfG_0_page_ref} denote the closure of smooth functions in $\mfG^\alpha$.
It follows from the exact same argument
as in~\cite[Sections~3.4-3.6]{CCHS_2D} that, for $\alpha\in(\frac23,1]$, there is a locally uniformly continuous group action $\Omega_\alpha\times \mfG^\alpha \to\Omega_\alpha$, denoted $(A,g)\mapsto A^g$,
which restricts to a group action $\Omega^1_\alpha\times \mfG^{0,\alpha} \to\Omega^1_\alpha$
for which $\mfO_\alpha \eqdef \Omega^1_\alpha/ \mfG^{0,\alpha}$\label{page:mfO} is a Polish space.
For smooth $A,g$, the $1$-form $A^g$ takes the familiar form $A^g = \Ad_g A - (\mrd g) g^{-1}$.
As usual, for $A,B\in\Omega^1_\alpha$, we write $B\sim A$ if $B=A^g$ for some $g\in\mfG^{0,\alpha}$
and denote $[A] = \{\omega\in\Omega^1_\alpha\,:\,\omega\sim A\}$.
We denote by $\pi\colon\Omega^1_\alpha\to\mfO_\alpha$ the projection
$\pi(A)=[A]$.\label{page:pi3}

\begin{remark}\label{rem:diff_with_CCHS_2D}
The spaces $\Omega_{\gr\alpha}$, $\Omega_\alpha$, $\Omega^1_\alpha$, $\mfO_\alpha$ differ from
those denoted by the same symbols in~\cite[Sec.~3]{CCHS_2D}: our spaces are weaker as we only consider axis-parallel lines (vs. all straight lines in~\cite{CCHS_2D}).
However, it is straightforward to verify that all results on these spaces shown in~\cite[Sec.~3]{CCHS_2D}
remain true for our definitions
and we will use them without constantly highlighting this distinction.
\end{remark}

\subsection{Main results}
\label{subsec:main_results}

We define our $\mfq$-valued Brownian motion $\BM$ and the white noise $\xi$ on the same probability space by setting
for $e\in\obonds_i$, $i\in [d]$,
\begin{equ}[e:couple-noise]
\dt \BM(e) = \e^{1-d}  %\e^{-1}
 \langle \xi_i(t,\cdot),1_{B(e,\e)} \rangle
\end{equ}
where $B(e,\e)$ is the $d$-cube centred at the midpoint of $e$ with side length $\e$.
The scaling in \eqref{e:couple-noise} is such that  $\dt \BM$ is a discrete white noise under the metric \eqref{e:mfq-inner}.

A mollifier is a smooth function $\moll\in\CC^\infty(\R\times\R^d,\R)$ such that $\int\moll = 1$ and with
support contained in $\{z\,:\, |z|<\frac14\}$.
We say that $\moll$ is non-anticipative if the support of $\moll$ is contained in $\{(t,x)\in\R\times\R^d\,:\, t\geq0\}$.

We equip $\mfg^d$ with the (diagonal) adjoint action of $G$, given for $X\in\mfg^d$ and $g\in G$ by
\[
\Ad_g(X_1,\ldots,X_d) = (\Ad_g X_1,\ldots,\Ad_g X_d)\;.
\]
We denote
\begin{equ}[eq:L_G]
L_G(\mfg^d,\mfg^d) = \{C\in L(\mfg^d,\mfg^d) \,:\, \Ad_g C(X) = C(\Ad_g X) \text{ for all }
g\in G\,,\, X\in\mfg^d\}
\;,
\end{equ}
i.e. the subset of maps that commute with the adjoint action of $G$.
We tacitly identify $L(\mfg,\mfg)$ with a subspace of $L(\mfg^d,\mfg^d)$ by
extending every $C\in L(\mfg,\mfg)$ to a `diagonal' map $C \in L(\mfg^d,\mfg^d)$ by $C(A_1,\ldots,A_d)=(CA_1,\ldots,CA_d)$.

Throughout this subsection, we fix a non-anticipative mollifier $\moll$ and
independent and identically distributed (i.i.d) $\mfg$-valued white noises $\xi=(\xi_1,\ldots,\xi_d)$ on $\R\times \T^d$.

For the remainder of the section, we take $d=2$.
We further fix $\alpha\in (\frac45,1)$.
For $C\in L(\mfg^2,\mfg^2)$ and $a\in\Omega^1_{\alpha}$, let $\SYM_t(C,a)$\label{page:SYM_t} denote the solution at time $t>0$ to
the stochastic YM equations (SYM) driven by $\xi$ with mollifier $\moll$, mass renormalisation $C$, and initial condition $a$.
That is,
$
\SYM(C,a) = \lim_{\e \downarrow 0} A^\e
$,
where $A^\e$ solves~\eqref{eq:SYM_moll}
and the limit is taken in the metric space $(\Omega_{\alpha}^1)^\sol$ that allows from blow-up, see~\cite[Sec.~1.5.1]{CCHS_2D}.
Therefore $\{\SYM_t(C,a)\}_{t\geq 0}$ is a continuous stochastic process with values in $\hat\Omega_{\alpha}^1\eqdef \Omega_{\alpha}^1\sqcup\{\skull\}$.
We recall from~\cite[Sec.~1.5.1]{CCHS_2D} that, for a metric space $(F,\rho)$, we define $\hat F \eqdef F\sqcup \{\skull\}$\label{page:F_hat}
equipped with a metric $\hat\rho$ which allows points to approximate $\skull$ given by
\begin{equ}
\hat\rho(f,g) = \rho(f,g)\wedge (h(f)+h(g))
\end{equ}
where $h\colon \hat F \to [0,1]$, $h(f)\eqdef (1+\rho(f,o))^{-1}$ for an arbitrary fixed $o\in F$.
We further let $\bar C\in L_G(\mfg,\mfg)$ be the $\moll$-dependent operator as in~\cite[Thm.~2.9(i)]{CCHS_2D}.

Our first main result shows convergence of a class of discrete dynamics on the lattice to $\SYM(\bar C,\cdot)$,
which is a form of universality for the YM Langevin dynamic.
For a normed space $(\CB, \|\cdot \|)$, $T>0$, and $\eta\in\R$, let $\CC_\eta^T(\CB)$ denote the space of functions $f\in\CC((0,T],\CB)$ for which
\begin{equ}
\|f\|_{\CC_\eta^T(\CB)}  \eqdef \sup_{t\in (0,T]} t^{-\eta/2} \|f(t)\|< \infty\;.
\end{equ}

\begin{theorem}[Universality of the dynamic]\label{thm:discrete_dynamics}
Suppose Assumption~\ref{assump:R} holds.
Let $\eta<\alpha-1$, $a\in\Omega_{\alpha}^1$, and $a^{(N)} \in \log(\conf_N) \subset \mfq_N$
such that
\begin{equ}
\lim_{N \to \infty}|a^{(N)} - \pi_N a|_{N;\alpha} = 0\;.
\end{equ}
Consider the $\mfq$-valued Markov process $A^{(N)} \eqdef \log U^{(N)}$ with initial condition $a^{(N)}$.
Denote by $T^*>0$ the blow-up time of $\SYM(\bar C,a)$.
Then there exists an increasing sequence of stopping times $T_N$ such that $\lim_{N\to\infty}T_N=T^*$ almost surely and
\begin{equ}	
\lim_{N\to\infty}\|A^{(N)} - \pi_N \SYM(\bar C,a)\|_{\CC^{T_N}_ {\eta-\alpha/2}(\Omega_{N;\alpha})} = 0 \quad \text{in probability.}
\end{equ}
\end{theorem}

\begin{remark}\label{rem:T*}
The stopping times $T_N$ can be taken as $T_N=K_N\wedge\inf\{t>0\,:\,|\SYM(\bar C,a)|_\alpha>K_N\}$
for some increasing sequence $K_N$ with $\lim_{N\to\infty} K_N=\infty$.
\end{remark}

\begin{remark}\label{rem:local_condition}
Assumption~\ref{assump:R} is purely local around $\id\in G$.
The reason such a local assumption on $S_N$ suffices is that the probability of $e^{A^{(N)}}$ leaving the neighbourhood $W$ in a short amount of time becomes small as $N\to\infty$.
Consequently, the only role that the exit time~\eqref{eq:varpi_def} plays is to ensure that $U$ is well-defined
and Theorem~\ref{thm:discrete_dynamics} remains true with $U^{(N)}$ replaced by $\check U^{(N)}$.
\end{remark}
We give the proof of Theorem~\ref{thm:discrete_dynamics} at the end of Section~\ref{sec:gauge-covar}.
The proof is based on regularity structures and involves several steps.
In Section~\ref{sec:lattice_dynamics} we derive the dynamics for $A^{(N)}$ on the lattice.
The main idea is to work directly on the Lie algebra $\mfq$ by pulling back the Riemannian structure of $G$ to $\mfg$ via the derivative of the exponential map.
In Sections~\ref{sec:DiscreteRS} and~\ref{sec:Aeps}, we show local well-posedness of $A^{(N)}$ uniform in $N\geq1$ by means of discrete regularity structures. 
The main work in Section~\ref{sec:Aeps} is to show moment bounds on discrete models and that there is a \textit{finite} mass renormalisation arising from the lattice approximations.
We emphasise that the equation for $A^{(N)}$ is non-polynomial for finite $N$ with a multiplicative noise
-- this forces us to work with a substantially larger regularity structure than the one considered in~\cite[Sec.~6]{CCHS_2D}, see Section~\ref{sec:tr-reg}.
In Section~\ref{sec:SHE} we analyse the stochastic heat equation, which is the roughest part of $\SYM$.

In Section~\ref{sec:diagonal_argument} we use a diagonal argument to show that $A^{(N)}$ (along subsequences)
converges to $\SYM(C,\cdot)$ for \textit{some} operator $C\in L_G(\mfg^2,\mfg^2)$.
Finally, in Section~\ref{sec:gauge-covar} we show that $C=\bar C$ by arguing that $\bar C$ is the unique operator in $L_G(\mfg^2,\mfg^2)$ for which $\SYM(C,\cdot)$ is gauge covariant.
This final point is of independent interest,
resolving another problem left open in~\cite{CCHS_2D}, and can be summarised in the following theorem (see Theorem~\ref{thm:A_tilde_A} for a more precise statement).
\begin{theorem}[Uniqueness of gauge-covariant constant]\label{thm:C_unique}
Let $C\in L_G(\mfg^2,\mfg^2)$ for which $C\neq \bar C$.
There exists a loop $\ell\in\CC^\infty(S^1,\T^2)$ with the following property.
For all $t>0$ sufficiently small, there exists $g\in \CC^\infty(\T^2,G)$ such that
\begin{equ}\label{eq:W_diff_lower}
|\E W_\ell(\SYM_t(C,0)) - \E W_\ell(\SYM_t(C,0^{g}))| \gtrsim t^2\;,
\end{equ}
where
$W_\ell\colon\Omega_{\alpha}^1\sqcup\{\skull\} \to \C$  is the Wilson loop defined by $W_\ell(\omega) = \Trace\hol(\omega,\ell)$ for $\omega\in\Omega^1_{\alpha}$ and $W_\ell(\skull)=0$.
\end{theorem}
Here $\hol(\omega,\ell)\in G$ is the holonomy of $\omega$ along $\ell$, see \cite[Sec.~3.5]{CCHS_2D} or~\eqref{eq:hol_def} below.
If $G$ is Abelian, a similar result is much simpler to show~\cite[Sec.~I.E]{Chevyrev22YM}
(see also Remark \ref{rem:Abelian_vs_simply_conn} that contrasts the Abelian and simply connected cases).

The lower bound $t^2$ in~\eqref{eq:W_diff_lower} is important as it rules out differing explosion times as the source of non-equality between $\E W_\ell(\SYM_t(C,0))$ and $\E W_\ell(\SYM_t(C,0^{g}))$ for small $t$.
That is, although $\SYM(\bar C,\cdot)$ is `gauge-covariant' in a suitable sense, $\SYM(\bar C,0^{g})$ and $\SYM(\bar C,0)$ may still explode at finite (and different!) times,
so we do not know if $\E W_\ell(\SYM_t(\bar C,0)) = \E W_\ell(\SYM_t(\bar C,0^{g}))$ holds for any $t>0$.
However, the probability that either process explodes before time $t$ is of order $O(t^M)$ for any $M>0$,
so one does have
\begin{equ}\label{eq:best_bound_bar_C}
|\E W_\ell(\SYM_t(\bar C,0)) - \E W_\ell(\SYM_t(\bar C,0^{g}))|  = O(t^M)\;.
\end{equ}
Therefore,~\eqref{eq:W_diff_lower} shows that the characterising feature of $\bar C$ is that one can take $M>2$ in~\eqref{eq:best_bound_bar_C} while every other choice of $C$ at best allows $M=2$.

As a corollary of Theorem~\ref{thm:C_unique}, we show that $\bar C$ is the unique operator for which `generative probability measures' (in the language of~\cite{CCHS_2D}) give rise to a Markov process on gauge orbits, see Corollary~\ref{cor:unique_C}.

Our next main result is on the invariant measure of $\SYM(\bar C,\cdot)$ projected to gauge orbits.
To state this result, let $X$ be the time homogenous Markov process
taking values in $\mfO_\alpha$ from~\cite[Thm.~2.13(ii)]{CCHS_2D}
(one should think of $X_t=[\SYM_t(\bar C,\cdot)]$, although this is only formal because $\SYM(\bar C,\cdot)$ might blow-up before $X$ does).
Letting $X^x$ denote the law of $X$ with initial condition $x\in\mfO_\alpha$,
recall that $X^x$ has a.s. continuous sample paths with values in $\hat\mfO_\alpha \eqdef \mfO_\alpha\sqcup \{\skull\}$
and its law is $\pi_* \nu$, where 
$\nu$ is a generative probability measure with initial condition $x$.

\begin{theorem}[Invariant measure]\label{thm:invar_measure}
\begin{enumerate}
\item There exists a unique probability measure $\mu$ on $\mfO_\alpha$ such that its finite dimensional distributions, seen as a process indexed by loops,
agree with those of the YM measure 
for the trivial principal $G$-bundle over $\T^2$ in the sense of \cite{Levy06}.

\item $\mu$ is the unique invariant probability measure on $\mfO_\alpha$ of $X$. Furthermore, $\mu$ has full support in $\mfO_\alpha$ and $X$ is reversible with respect to $\mu$.
\end{enumerate}
\end{theorem}
The proof of Theorem~\ref{thm:invar_measure}
is given in Section~\ref{sec:invar_measure}
and employs Bourgain's invariant measure argument, lattice gauge-fixing, and ergodicity properties of $X$ (the latter following from recent ergodicity results in singular SPDE~\cite{HM16,HS22}).
The first part of Theorem~\ref{thm:invar_measure} (for simply connected $G$) is essentially due to~\cite{Chevyrev19YM},
so the main new content is in the second part.
We prove the theorem first for the case that $G$ is simply connected, so that every principal $G$-bundle is trivial,
and then deduce the general result by an appropriate decomposition and projection.

\begin{remark}
We make no claims on the invariant measure of $\SYM(\bar C,\cdot)$ itself (in fact, without modifications, $\SYM(\bar C,\cdot)$ does not have an invariant probability measure already in the Abelian case, see~\cite[Rem.~1.17]{Chevyrev22YM}).
We do not even know if $\SYM(\bar C,\cdot)$ has global-in-time solutions.
\end{remark}
\begin{remark}\label{rem:CCHS22_compare}
The construction of the Markov process $X$ is one of the main results of~\cite{CCHS_2D}
and is shown therein through
a coupling argument with time-dependent gauge-transformations.
We believe our results, with further technical effort,\footnote{Specifically, one would need to improve the 
$O(1)$ bounds in Sections~\ref{sec:renorm_constants} and~\ref{subsec:some_C_proof} to convergences.}
can recover independently this construction via lattice approximations;
specifically, one should be able to prove the existence and uniqueness (of pushforwards) of
generative probability measures for \textit{some} mass renormalisation $C\in L_G(\mfg,\mfg)$
(see~\cite[Thm.~2.13]{CCHS_2D}).
We do, however, crucially rely on the coupling argument of~\cite[Sec.~7]{CCHS_2D} in Section~\ref{sec:gauge-covar} to prove Theorem~\ref{thm:C_unique}.
which is used to identify the continuum limit of our lattice approximations and, perhaps more importantly,
to show that the Markov process is entirely canonical in that $\SYM$ does not yield
a Markov process on gauge orbits for any other mass renormalisation in $L_G(\mfg^2,\mfg^2)$.
\end{remark}
Finally, we turn to several consequences of Theorem~\ref{thm:invar_measure}.

\begin{corollary}[Long-time existence]\label{cor:long_time}
The Markov process $X^x$ does not blow up for every initial condition $x\in \mfO_\alpha$.
\end{corollary}

\begin{proof}
Since $\mu$ has full support in $\mfO_\alpha$ and is invariant for $X$,
the conclusion follows from the strong Feller property of $X$ (see Proposition~\ref{prop:X_strong_Feller} below).
\end{proof}
We give the proofs of the next two corollaries in Section~\ref{subsec:proofs_cor}.
The first provides a decomposition of the YM measure into the GFF plus an almost Lipschitz remainder.
\begin{definition}\label{def:GFF}
A Gaussian free field (GFF) is a random Gaussian  $\mfg$-valued distributional $1$-form $\Psi=(\Psi_1,\Psi_2)$ such that $\scal{\Psi_i,1}=0$ a.s., $\Psi_1,\Psi_2$ are i.i.d, and
$
\E\scal{\Psi_i,\phi}\scal{\Psi_i,\bar\phi} = \scal{\phi,(-\Delta)^{-1}\bar\phi}
$
for all $\phi,\bar\phi\in\CC^\infty(\T^2,\mfg)$ with $\int \phi = \int\bar\phi=0$.
\end{definition}

\begin{corollary}[Gauge-fixed decomposition]\label{cor:decomposition}
There exist, on the same probability space, random distributional $1$-forms $\Psi,B$ such that $\Psi$ is a GFF
and $|B|_{\CC^{1-\kappa}}$ has moments of all order for each $\kappa>0$,
and $\mbox{Law}([\Psi+B]) = \mu$.
\end{corollary}

\begin{remark}\label{rem:Che19_compare}
By the same (and simpler) arguments as in~\cite[Sec.~4.2]{CCHS_2D},
a GFF $\Psi$ admits a modification with $|\Psi|_\alpha < \infty$, and thus $[\Psi+B]$ indeed makes sense as an $\mfO_\alpha$-valued random variable.

Furthermore, let  $|\cdot|_{\bullet}$ denote any norm from~\cite[Sec.~3]{CCHS_2D}, e.g.
$|\cdot|_{\v\alpha}$ or $|\cdot|_{\gr\alpha}$ (note that the norms $|\cdot|_{\bullet}$ are isotropic, i.e.
 not restricted to axis-parallel lines,
in particular $|\cdot|_{\gr\alpha}$ in~\cite{CCHS_2D} denotes
a stronger norm than $|\cdot|_{\gr\alpha}$ here).
Then $\E|\Psi|^p_{\bullet} < \infty$ for all $p\geq 1$.
Therefore, Corollary~\ref{cor:decomposition} shows that $\mu$ admits a gauge-fixed representation $A=\Psi+B$ with strong regularity properties,
namely with $\E |A|^p_{\bullet} < \infty$ for all $p\geq 1$.
This significantly improves~\cite[Thm.~1.1]{Chevyrev19YM} which only shows that there exists $A$ with $|A|_\alpha < \infty$ a.s. and $\mbox{Law}([A])= \mu$.
\end{remark}

\begin{remark}
If $\mu$ were the invariant measure of $\SYM(\bar C,\cdot)$ itself, then
Corollary~\ref{cor:decomposition} would be a straightforward consequence of decompositions in regularity structures.
% (see the proof of~\cite[Thm.~2.4]{CCHS_2D}).
However, because the Markov process with invariant measure $\mu$ is defined by restarting $\SYM(\bar C,\cdot)$ at appropriate stopping times, care is needed in applying this decomposition.
\end{remark}
The final corollary that we present is a universality result for discrete approximations of the YM measure.
Since we are working only with trivial principal bundles, stating the result for non-simply connected $G$ requires some preparation.

Recall that $G$ is connected and compact. By the structure Theorem~\ref{thm:structure_compact_groups},  there exists $n\geq 0$ and a simply connected $L$ such that $G\simeq \T^n\times L / Z$ where $Z$ is a finite subgroup of the centre of $\T^n\times L$.
Denote $H\eqdef \T^n\times L$
and let $p\colon H\to G\simeq H/Z$ be the canonical projection.  
(We recommend the reader first considers the case that $G$ is simply connected, which is already interesting and in which case $n=0$ and $G\simeq H \simeq L$.)

We consider a family of probability measures $\mu_N$ on gauge fields $G^{\obonds_N}$ as follows.
Let $\mu_{N,\T^n}$ be the discrete YM probability measure on $(\T^n)^{\obonds_N}$
associated to the trivial principal $\T^n$-bundle over $\T^2$ as in~\cite[Sec.~2.3]{Levy06}.
We let $\mu_{N,L}$ be any probability measure on $L^{\obonds_N}$ as in~\eqref{eq:mu_N} with $S_{N,L}\colon L\to \R$ satisfying Assumption~\ref{assump:R} and
Assumption~\ref{assum:P_N} below.
Define $\mu_{N,H}=\mu_{N,\T^n}\times \mu_{N,L}$, understood canonically as a probability measure on $H^{\obonds}$.
Finally, let $\mu_N \eqdef p_* \mu_{N,H}$
be the pushforward of $\mu_{N,H}$ under $p\colon H^{\obonds} \to G^{\obonds}$.

\begin{remark}
The condition that $\mu_{N,\T^n}$ is distributed exactly as the Abelian YM measure for the trivial bundle can be significantly relaxed.
One way to do this is to work with families of actions $S_{N,\R^n}\colon \R^n\to \R$
and condition the corresponding plaquette variables to sum to zero, akin to~\cite[Thm.~1]{Levy06}.
Since the interesting component is $\mu_{N,L}$, we do not elaborate on this point.
\end{remark}
We say that a function $f\colon G^k\to \R$ is $\Ad$-invariant if, for all $h,g_1,\ldots, g_k\in G$,
$
f(g_1,\ldots,g_k) = f(hg_1h^{-1},\ldots, hg_kh^{-1})
$.
\begin{definition}\label{def:lattice_loop}
A path $\ell \in \CC([0,1],\T^2)$ is called a \textit{lattice loop} if there exist $N\geq 1$ and $x\in\Lambda_N$ such that $\ell(0)=\ell(1)=x$
and its image is contained in $\cup_{b\in\obonds_N}\iota(b)$ where $\iota(x,y) = \{x+t(y-x)\,:\,t\in[0,1]\} \subset \T^2$ is the natural embedding of $(x,y)\in\obonds_N$ into $\T^2$.
We call $x$ the \textit{base} of $\ell$.
\end{definition}
Observe that, for such $\ell$ and any $U\in Q_N$ and $A\in\Omega_{\gr\alpha}$, the holonomies  $\hol(U,\ell)$ and $\hol(A,\ell)$ are well-defined.

For $\Ad$-invariant $f\colon G^k\to\R$ and $k$-tuple of lattice loops $\ell=(\ell_1,\ldots, \ell_k)$,
define
$f_\ell\colon\Omega_{\gr\alpha}\to \R$ by
\begin{equ}
f_{\ell}(A) = f(\hol(A,\ell_1),\ldots,\hol(A,\ell_k))\;.
\end{equ}
By the same expression, we define $f_\ell\colon Q_N \to \R$, which is well-defined provided $N\geq 1$ is sufficiently large (depending on $\ell$).

Observe that for $A,B\in \Omega^1_\alpha$, the following statements are equivalent:
\begin{enumerate}[label=(\roman*)]
\item $A\sim B$
%\item $f_\ell(A)=f_\ell(B)$ for all $n\geq 1$,
%smooth $\Ad$-invariant functions $f\colon G^n\to \R$, and $n$-tuples $\ell=(\ell_1,\ldots,\ell_n)$ of lattice loops with base $0\in\T^2$,
\item\label{pt:all_f_ell} $f_\ell(A)=f_\ell(B)$ for all $k\geq 1$, continuous $\Ad$-invariant functions $f\colon G^k\to \R$, and $k$-tuples $\ell=(\ell_1,\ldots,\ell_k)$ of lattice loops with the same base. 
\end{enumerate}
In particular, every $f_\ell$ as in~\ref{pt:all_f_ell} descends to a function $f_\ell\colon\mfO_\alpha \to \R$.
A similar statement holds for $U,V\in Q_N$ (provided we restrict to loops supported on $\obonds_N$).
(For the proof of this equivalence for $\Omega^1_{\alpha}$, see~\cite[Prop.~3.35]{CCHS_2D}; for the case of $Q_N$, the proof is similar and simpler.)
The implication (ii) $\Rightarrow$ (i) implies that, for probability measures $\mu,\nu$ on $\mfO_\alpha$, one has $\mu=\nu \Leftrightarrow \mu(f_\ell)=\nu(f_\ell)$ for all $f,\ell$, see \cite[Exercise 7.14.79]{Bogachev07}.

The following corollary can be interpreted as a functional central limit theorem for random connections on $\T^2$.

\begin{corollary}[Universality of the measure]\label{cor:universality}
Let $k\geq 1$, $f\colon G^k\to \R$ a continuous $\Ad$-invariant function, and $\ell$ an $k$-tuple of lattice loops with a common base.
Then, for $\mu_N$ as above,
$
\lim_{N\to\infty}\mu_N(f_\ell) = \mu(f_\ell)
$,
where $\mu$ is the YM measure on $\mfO_\alpha$ from Theorem~\ref{thm:invar_measure}.
\end{corollary}

\section{Lattice dynamics in the Lie algebra}
\label{sec:lattice_dynamics}

In this section, we derive a discrete SPDE
for $A\eqdef\e^{-1}A^{(N)}$
with $A^{(N)}$ as in Theorem~\ref{thm:discrete_dynamics}
(i.e. $e^{\eps A}=U$ with $U=U^{(N)}$ defined by~\eqref{eq:discrete_U})
which is in a form that ``approximates'' the limiting SPDE.
Recall that $\eps=2^{-N}$.
The derivation will be in {\it arbitrary} $d\ge 1$, although
we will only prove convergence in $d=2$.

This derivation requires several steps.
We first introduce some further notation to do with the lattice.
Some of the notation (e.g. \eqref{e:EWNS}-\eqref{e:def-CEtimes}) may depend on a fixed plane in our $d$-dimensional lattice that is parallel with the $i$th and $j$th axes for some $i\neq j\in [d]$, but we will often omit the dependence on $i,j$ in our notation.

\subsection{More notation}
\label{sec:more_notation}

%When $d=2$, 
For any edge $e=(x,y)\in \obonds_i$, and $j\neq i$, we write
\begin{equ}[e:EWNS]
%e^\east = e + \e_i\;,
%\quad 
%e^\west = e - \e_i\;,&
%\quad 
%e^\north = e + \e_j\;,
%\quad 
%e^\south = e - \e_j\;,
%\\
e^\northeast = (y,y+\e_j)\;,
\;\;
e^\northwest = (x,x+\e_j)\;,
\;\;
e^\southeast = (y-\e_j,y)\;,
\;\;
e^\southwest = (x-\e_j,x)\;.
\end{equ}
%In particular, $\{e^\east ,e^\west,e^\north,e^\south\}  \subset \obonds_i$ and 
One has $\{ e^\northeast ,e^\northwest ,e^\southeast,e^\southwest \}  \subset \obonds_j$,
%The notation is such that if we view $e$ as a ``horizontal'' bond
%%(by reflection and rotation by $\frac\pi2$ if necessary),
%then these bonds are its nearest neighbour ``to the east, west, etc.'', 
see Figure~\ref{fig:CE_times}.
%
%Still in $d=2$, 
Again for $i\neq j\in [d]$, %in addition to $\CE_\pm$
we write
\begin{equ}[e:def-CEtimes]
\CE_\times \eqdef \CE_\times^{(i,j)} \eqdef 
\{\tfrac12 (s_1 \e_i+ s_2 \e_j) : s_1,s_2\in \{\pm 1\}\}\;.
\end{equ}
If we identify a bond $e\in\obonds$ with its midpoint, then
$
\{ e^\northeast ,e^\northwest ,e^\southeast,e^\southwest \}
=\{e+\be:\be\in \CE_\times\}
$.
Finally for $e=(x,x+\e_i)\in \obonds_i$, we write $p\succ e$ if $p=(x,\e,\bar\e)\in \plaq$ with $\e=\e_i$. For a given $e$ there are $2(d-1)$ plaquettes $p$ such that $p\succ e$. 

\begin{figure}[h]
\centering
\begin{tikzpicture}[scale = 1.3]
\node [draw,dot,name=x] at (0,1) {};
\node at (0.15,1.15) {$x$};
\node [draw, dot, name=d00] at (0,0) {};
\node [draw, dot, name=d10] at (1,0) {};
\node [draw, dot, name=d11] at (1,1) {}; \node at (1.15,1.15) {$y$};
\node [draw, dot, name=d02] at (0,2) {};
\node [draw, dot, name=d12] at (1,2) {};

\node [draw,dot,name=L] at (-1,1) {};
\node [draw,dot,name=R] at (2,1) {};

\draw[-{Latex[length=3mm]}] (x) -- (d11)
node[draw=none, midway, above=0.5]{$e$};
\draw[-{Latex[length=3mm]}] (d11) -- (d12) node[draw=none, midway, right=0.5]{$e^\northeast$};
\draw[-{Latex[length=3mm]}] (d02) --(d12) node[draw=none, midway, below=0.5]{};
\draw[-{Latex[length=3mm]}] (x) -- (d02) node[draw=none, midway, left=0.5]{$e^{\northwest}$};

\draw[-{Latex[length=3mm]}] (d10) -- (d11) node[draw=none, midway, right=0.5]{$e^{\southeast}$};
\draw[-{Latex[length=3mm]}] (d00) -- (d10) node[draw=none, midway, above=0.5]{};
\draw[-{Latex[length=3mm]}] (d00) -- (x) node[draw=none, midway, left=0.5]{$e^{\southwest}$};

\draw[-{Latex[length=3mm]}] (L) -- (x) node[draw=none, midway, below=0.5]{};
\draw[-{Latex[length=3mm]}] (d11) -- (R) node[draw=none, midway, below=0.5]{};
\end{tikzpicture}
\qquad
\begin{tikzpicture}[scale = 1.3,baseline = {(0, 0.5)}]
\node [draw,dot,name=x] at (1,1) {};
\node at (1.15,1.15) {$x$};
\node [draw, dot, name=d10] at (1,0) {};
\node [draw, dot, name=d01] at (0,1) {};
\node [draw, dot, name=y] at (1,2) {};
\node at (1.15,1.85) {$y$};
\node [draw, dot, name=d02] at (0,2) {};
\node [draw, dot, name=d13] at (1,3) {};
\node [draw, dot, name=d22] at (2,2) {};
\node [draw, dot, name=d21] at (2,1) {};

\draw[-{Latex[length=3mm]}] (x) -- (y)
node[draw=none, midway, left=0.5]{$e$};
\draw[-{Latex[length=3mm]}] (d01) -- (x) node[draw=none, midway, below=0.5]{$e^\southwest$};
\draw[-{Latex[length=3mm]}] (d01) -- (d02) node[draw=none, midway, right=0.5]{};
\draw[-{Latex[length=3mm]}] (d02) -- (y) node[draw=none, midway, above=0.5]{$e^\southeast$};
\draw[-{Latex[length=3mm]}] (y) -- (d22) node[draw=none, midway, above=0.5]{$e^\northeast$};
\draw[-{Latex[length=3mm]}] (d21) -- (d22) node[draw=none, midway, left=0.5]{};
\draw[-{Latex[length=3mm]}] (x) -- (d21) node[draw=none, midway, below=0.5]{$e^\northwest$};

\draw[-{Latex[length=3mm]}] (d10) -- (x) node[draw=none, midway, right=0.5]{};
\draw[-{Latex[length=3mm]}] (y) -- (d13) node[draw=none, midway, right=0.5]{};
\end{tikzpicture}
\caption{
We have $i=1$, $j=2$   in the left figure, 
 and $i=2$, $j=1$ in the right figure.
 (Here horizontal bonds are in $\obonds_1$ and vertical bonds are in $\obonds_2$.)
}
\label{fig:CE_times}
\end{figure}
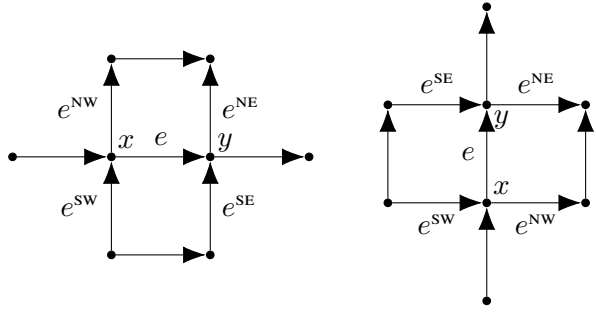

\subsubsection{Discrete derivatives}
\label{subsubsec:discrete_der}

We define a collection of discrete differentiation operators. Note that we will omit their dependence on $\e$ 
in our notation, since it will be clear from the context whether we are dealing with discrete or continuous operators in the sequel.

Consider $A\in\mfq$.
For $j\in[d]$, we define the ``forward and backward'' difference operator $\nabla^{\pm}_j A\in\mfq$ by
$
%\nabla^\pm_j A(x,y) &= A(x\pm\e_j,y\pm \e_j) - A(x,y)\;,
%\\
%\mbox{i.e.} \quad 
\nabla^\pm_j A(e) = A(e\pm\e_j) - A(e) 
$.
We define its rescaled version by $\partial^\pm_j A \eqdef \eps^{-1} \nabla^\pm_j A$. 
Similarly, for $f\colon\Lambda\to \mfg$, we define $\nabla^{\pm}_j f, \partial^\pm_j f\colon\Lambda\to \mfg$ by the same expression with $e$ replaced by $x\in\Lambda$.
%\begin{equ}
%\nabla^\pm_j f(x) = f(x\pm\e_j) - f(x)\;,
%\qquad
%\mbox{and}
%\qquad
%\partial^\pm_j f \eqdef \eps^{-1} \nabla^\pm_j f \;.
%\end{equ}

Moreover we define the second order derivative and the discrete Laplacian
\begin{equ}[e:laplacian]
\partial_j^2 A (e) \eqdef \frac{1}{\e^2}(A(e+\e_j)+A(e-\e_j) -2 A(e)),
\quad
\Delta A (e) \eqdef  \sum_{j=1}^d \partial_j^2 A (e). 
%=\frac{\e^{-2}}{d} \sum_{j=1}^d (A(e+\e_j)+A(e-\e_j) -2 A(e))\;.
\end{equ}
For $A\in\mfq$, we define the discrete divergence $\nabla\cdot A \colon\Lambda\to\mfg$ by   
\begin{equ}
\nabla \cdot A (x) \eqdef
\sum_{j=1}^d \nabla^+_j A(x-\e_j,x)
=\sum_{j=1}^d \big[ A(x,x+\e_j) - A(x-\e_j,x) \big]\;.
\end{equ}
%and its rescaled version $\div A \eqdef \eps^{-1} \nabla\cdot A$.

%For $A\in\mfq$, we define the discrete curvature $f_A\colon \plaq \to\mfg$
%by 
%\begin{equ}\label{eq:discrete_curv}
%f_A(p) = \sum_{i=1}^4 A(e_i) + \frac14[A(e_1)-A(e_3),A(e_2)-A(e_4)]\;,
%\end{equ}
%where $p=(e_1,e_2,e_3,e_4)$. Remark that $f_A(p)=-f_A(\overleftarrow{p})$ which mimics the anti-symmetry $F_{ij}=-F_{ji}$ of $2$-forms.
%
%Given $F\colon\plaq\to \mfg$ and $A\in\mfq$, we define the discrete adjoint covariant derivative $\nabla^*_A F\in\mfq$ by
%\begin{equ}[eq:nabalF]
%\nabla^*_A F(e) = \sum_{p\succ e}
%\Big(
%F(p)+
%\frac14[A(e_2)-A(e_4),F(p)]
%\Big)\;,
%\end{equ}
%where $e\in \obonds$ and
%we write $p=(e_1,\ldots,e_4)$.
%
%\begin{remark}\label{rem:d*fA}
%For any $A\in\mfq$ and $p=(e_1,\ldots, e_4)$, denoting $A_k=A(e_k)$,
%\begin{equs}
%f_A(p)
%+
%\frac14[A_2-A_4,f_A(p)]
%&=\sum_{k=1}^4 A_k + \frac12\Big([A_2,A_3]+[A_2,A_4]+[A_3,A_4]\Big)
%\\
%&\qquad\qquad + \frac1{16}[A_2-A_4,[A_1-A_3,A_2-A_4]]\;.
%\end{equs}
%\end{remark}

\subsubsection{Symmetrised derivatives and averaging}
\label{sec:ab-notation}

%Assume $d=2$ and $j\neq i\in[2]$.
Fix  $j\neq i\in[d]$.
For any $A_j\in \mfq_j$ and $e\in \obonds_i$, recall that 
$A_j(e)$ is not defined since $j\neq i$, but
 we will write\footnote{This ``averaging'' notation is prior to products, for instance,
 $(A_j (e^{(a)}))^2$ just means the square of $A_j (e^{(a)})$, 
rather than
$\frac14 \sum_{\be \in \CE_\times} A_j(e+\be)^2$.}
\begin{equ}[e:def-Aea]
A_j (e^{(a)}) 
\eqdef
\frac14 \sum_{\be \in \CE_\times} A_j(e+\be)\;.
\end{equ}
Namely, the script $(a)$  indicates
an ``average'' of $A_j$ over the four edges in direction $j$
which are neighbours of $e$.
Thus $\obonds_i \ni e \mapsto A_j(e^{(a)}) \in\mfg$ is an element of $\mfq_i$.

We also define a {\it symmetrised} difference for  functions on $\Lambda$ or $\obonds$
\begin{equ}[e:sym-diff]
\nabla_j \eqdef \frac12 (\nabla^+_j - \nabla^-_j)\;,
\qquad
\partial_j \eqdef \eps^{-1} \nabla_j \;.
\end{equ}
%That is,  %if $d=2$, 
%with $i\neq j$ and  notation \eqref{e:EWNS},
%for a function $A$ on $\obonds$ and $e \in \obonds_i$ 
%\[
%(\partial_j A_i) (e) = \frac{1}{2\eps} (A_i(e^\north) -A_i(e^\south) )
% \;, \qquad 
% (\partial_i A_i) (e) = \frac{1}{2\eps} (A_i(e^\east) -A_i(e^\west) )\;.
%\]

Again
with $i\neq j$, for a function $A_j \in\mfq_j$ and edge $e \in \obonds_i$, we define additional symmetrised
differences as follows:
\begin{equ}
(\bar\nabla_i A_j) (e) 
\eqdef 
\frac{1}{2} \Big(  A_j (e^\northeast)+A_j (e^\southeast)
- A_j (e^\northwest)-A_j (e^\southwest)\Big)\;,
\quad \bar\partial_i = \e^{-1}\bar\nabla_i\;,
\end{equ}
and\footnote{\eqref{e:djAj} will only be used
in the {\it ``remainder''} terms of our discrete equation.}
\begin{equ}[e:djAj]
(\bar\nabla_j A_j) (e)
 \eqdef 
\frac{1}{2} \Big(  A_j (e^\northeast)+A_j (e^\northwest)
- A_j (e^\southeast )- A_j (e^\southwest)
\Big)\;,
\quad
\bar\partial_j = \e^{-1}\bar\nabla_j\;.
\end{equ}
The assumption $j\neq i$ in the two definitions is important.
If $A\in\mfq$, we sometimes omit the index in $A_j$, e.g. write 
$(\bar\partial_j A) (e) \equiv (\bar\partial_j A_j) (e)$ for \eqref{e:djAj}.

\begin{remark}\label{rem:bar-partial}
The `bar' in the notation $\bar\partial_i$ indicates a ``switch'' of bond sets,
i.e. $\bar\partial_i, \bar\partial_j \colon \mfq_j \to\mfq_i$ for $i\neq j$.
\end{remark}

\begin{remark}\label{rem:loss-odd}
We remark that if $K$ is an even (resp. odd) function 
in $x_j$ (the $j$-th coordinate of space variable),
then $\partial_j K$ and  $\bar \partial_j K$ are odd  (resp. even) function in $x_j$,
but this is not true for $\partial^+_j K$.
Such `parity' considerations on symmetrised v.s. unsymmetrised derivatives
will be crucial when deriving the form of the renormalisations in later sections.
\end{remark}

\subsubsection{Plaquettes and \texorpdfstring{$1$}{1}-forms}
\label{subsubsec:plaq}

Recall from Section~\ref{sec:Notation} that we have fixed neighbourhoods $W = \mathring W^{\obonds_N}\subset\conf$
and $V = \mathring V^{\obonds_N} \subset \mfq$
for which $\exp \colon V\to W$
is a diffeomorphism.
Throughout this section, we only care about our actions on $W$ and $V$.

For the rest of this section, for  $U\in W$, we write $\A\eqdef \log U\in V$, and $A\eqdef \eps^{-1}\A \in \mfq$.
Consider a bond $e=(x,y)$ where $y=x+\eps_i$ for some $i\in [d]$.
Then
\begin{equ}
U(x,y)= e^{\A(x,y)}\;,\qquad \mbox{and}\quad \A (x,y) = \eps A(x,y)\;.
\end{equ}
When we fix $i\neq j\in [d]$, there are two plaquettes $p\succ e$ for a given $e\in\obonds_i$, which we often write $p$ and $\bar p$, namely
$
p=(x,\e_i,\e_j)$ and  $\bar p=(x,\e_i, - \e_j)$.
To simplify notation, we use the conventions $\A_{1,2,\cdots,9}$ for the values of $\A$ at the bonds around $p$ and $\bar p$ as in Figure~\ref{fig:plaquettes}.
(This should be distinguished from the notation $A_i$ 
where $i\in [d]$ is the spatial index.) For instance, for $i\neq j \in [d]$,
$
\A_1 = \A(x,x+\eps_i)$,
$\A_2 = \A(x+\eps_i,x+\eps_i+\eps_j)$, 
$\A_3 = \A(x+\eps_i+\eps_j,x+\eps_j)$, etc.
%$\A_4 = \A(x+\eps_j,x)$.

\begin{figure}[h]
\centering
\begin{tikzpicture}[scale = 1.5]
\node [draw,dot,name=x] at (0,1) {};  \node at (-0.15,0.85) {$x$};
\node [draw, dot, name=d00] at (0,0) {};
\node [draw, dot, name=d10] at (1,0) {};
\node [draw, dot, name=d11] at (1,1) {}; \node at (1.15,0.85) {$y$};
\node [draw, dot, name=d02] at (0,2) {};
\node [draw, dot, name=d12] at (1,2) {};

\node [draw,dot,name=L] at (-1,1) {};
\node [draw,dot,name=R] at (2,1) {};

\draw[-{Latex[length=3mm]}] (x) -- (d11) node[draw=none, midway, below=0.5]{$\A_1$}
node[draw=none, midway, above=0.5]{$e$};
\draw[-{Latex[length=3mm]}] (d11) -- (d12) node[draw=none, midway, right=0.5]{$\A_2$};
\draw[-{Latex[length=3mm]}] (d12) -- (d02) node[draw=none, midway, above=0.5]{$\A_3$};
\draw[-{Latex[length=3mm]}] (d02) -- (x) node[draw=none, midway, left=0.5]{$\A_4$};

\draw[-{Latex[length=3mm]}] (d11) -- (d10) node[draw=none, midway, right=0.5]{$\A_5$};
\draw[-{Latex[length=3mm]}] (d10) -- (d00) node[draw=none, midway, below=0.5]{$\A_6$};
\draw[-{Latex[length=3mm]}] (d00) -- (x) node[draw=none, midway, left=0.5]{$\A_7$};

\draw[-{Latex[length=3mm]}] (L) -- (x) node[draw=none, midway, below=0.5]{$\A_8$};
\draw[-{Latex[length=3mm]}] (d11) -- (R) node[draw=none, midway, below=0.5]{$\A_9$};

\node (bp) at (0.5,0.5) {$\bar p$};
\node (p) at (0.5,1.5) {$p$};
\end{tikzpicture}
\qquad
\begin{tikzpicture}[scale = 1.5,baseline = {(0, 0.3)}]
\node [draw,dot,name=x] at (1,1) {};
\node at (1.15,1.15) {$x$};
\node [draw, dot, name=d10] at (1,0) {};
\node [draw, dot, name=d01] at (0,1) {};
\node [draw, dot, name=y] at (1,2) {};
\node at (1.15,1.85) {$y$};
\node [draw, dot, name=d02] at (0,2) {};
\node [draw, dot, name=d13] at (1,3) {};
\node [draw, dot, name=d22] at (2,2) {};
\node [draw, dot, name=d21] at (2,1) {};

\draw[-{Latex[length=3mm]}] (x) -- (y)
node[draw=none, midway, right=0.5]{$\A_1$}
node[draw=none, midway, left=0.5]{$e$};
\draw[-{Latex[length=3mm]}] (d01) -- (x) node[draw=none, midway, below=0.5]{$\A_7$};
\draw[-{Latex[length=3mm]}] (d02) -- (d01) node[draw=none, midway, left=0.5]{$\A_6$};
\draw[-{Latex[length=3mm]}] (y) -- (d02) node[draw=none, midway, above=0.5]{$\A_5$};
\draw[-{Latex[length=3mm]}] (y) -- (d22) node[draw=none, midway, above=0.5]{$\A_2$};
\draw[-{Latex[length=3mm]}] (d22) -- (d21) node[draw=none, midway, right=0.5]{$\A_3$};
\draw[-{Latex[length=3mm]}] (d21) -- (x) node[draw=none, midway, below=0.5]{$\A_4$};

\draw[-{Latex[length=3mm]}] (d10) -- (x) node[draw=none, midway, right=0.5]{$\A_8$};
\draw[-{Latex[length=3mm]}] (y) -- (d13) node[draw=none, midway, right=0.5]{$\A_9$};

\node (bp) at (0.5,1.5) {$\bar p$};
\node (p) at (1.5,1.5) {$p$};
\end{tikzpicture}
\caption{
The plaquettes $p,\bar p$ with conventions for $\A_{1,2,\cdots,9}$,
%The left figure and right figure correspond to $i=1$ and $i=2$ respectively.
where $(i,j)$ in the two figures are as in Fig.~\ref{fig:CE_times}.
We omit the dependence of $\A_{1,2,\cdots,9}$ on $i,j$ in our notation.
}
\label{fig:plaquettes}
\end{figure}
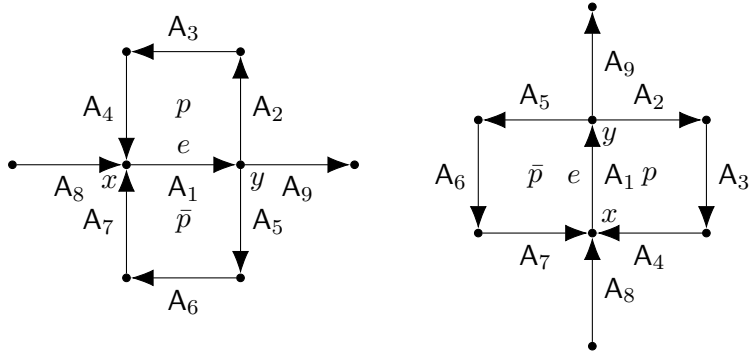

In calculating the cubic terms in the equation for $A$, it will be important to keep track of polynomials that contain factors $\A_1+\A_3$ and $\A_2+\A_4$ and other sums of oppositely pointing parallel bonds.
To this end, let $\CL(\R^m)$ denote the free Lie algebra generated by $\{\A_i\}_{i=1}^m$, where we view $\{\A_i\}_{i=1}^m$ as the canonical basis of $\R^m$.
We next define a class of polynomials called $\mcI_3$, which will allow us to treat error terms of the form $\e A^2\partial A$
(after rescaling) all together in Section~\ref{subsec:Mass renormalisation_1} without calculating their precise forms.

\begin{definition}\label{def:ideal}
Let $\mcI$ denote the Lie ideal in $\CL(\R^9)$ generated by $\{\A_i+\A_j\,:\,i\in \{1,8,9\}, j\in \{3,6\}\}$ and $\{\A_i+\A_j\,:\,i \in \{2,7\},j\in \{4,5\}\}$.
Let $\mcI_3$ denote the homogenous degree 3 polynomials in $\mcI$.

Given $i\neq j$, $e\in\obonds_i$ and  $Q\in\mcL(\R^9)$, we denote by the same symbol $Q\colon \mfq\to\mfg$ the map $Q(\A)=Q(\A_1,\ldots,\A_9)$ with $\A_{1,2,\ldots,9}$ related to $e$ and $(i,j)$ as in Fig.~\ref{fig:plaquettes}.

We also say that a polynomial $I\colon \mfq\to\mfq$ is in $\mcI_3$ to mean that there exist $I_{ij}\in \mcL(\R^9)$ with $I_{ij}\in \mcI_3$
for $i\neq j$ such that $I(\A)(e)=\sum_{j\neq i}I_{ij}(\A_1,\ldots\A_9)$ for all $e\in\obonds_i$, where $\A_{1,2,\ldots,9}$ are related to $e$ and $(i,j)$ as in Fig.~\ref{fig:plaquettes}.
\end{definition}
The following lemma is helpful in determining when Lie polynomials are in $\mcI$.
\begin{lemma}\label{lem:poly_vanish}
Suppose a Lie polynomial $Q(\A_1,\ldots,\A_9)$ vanishes upon the substitutions $\A_1,-\A_3,-\A_6,\A_8,\A_9\mapsto B_1$ and
$\A_2,-\A_4,-\A_5,\A_7\mapsto B_2$.
Then $Q\in \mcI$.
\end{lemma}

\begin{proof}
Let $M\colon \R^m \to \R^n$ be linear and denote by $\hat M\colon \CL(\R^m)\to\CL(\R^n)$ its extension to a Lie algebra morphism.
We claim that $\Ker(\hat M)$ is the Lie ideal generated by $\Ker(M)$.
Indeed, by changing coordinates,
we may suppose $M(X_i)=Y_i$ for $1\leq i\leq m'$ and $M(X_i)=0$ for $m'<i\leq m$ for some $0\leq m'\leq m$,
where $\{X_i\}_{i=1}^{m}$, $\{Y_i\}_{i=1}^n$ are canonical bases for $\R^{m},\R^n$.
By freeness of $\CL(\R^n)$, $\Ker(\hat M)$ consists of all Lie polynomials in which $X_i$ occurs for $m'<i\leq m$, which is the ideal generated by
$\Ker(M)$, thus proving the claim.

To conclude the proof, it suffices to apply the claim to the linear map $M\colon\R^9\to\R^2$ defined by the mappings in the statement.
\end{proof}

\subsection{Pull-back of Riemannian metric to Lie algebra}
\label{subsec:metric_pullback}

\begin{definition}
Consider $U \in W$ and $B\in T_{U}W$.
For $\A\eqdef \log U\in V$
we say that $\fancyC\in T_\A V \simeq \mfq$
is the \emph{pull-back} of $B$ if $D \exp_\A (\fancyC)=B$.
Similarly, denoting by $\eps\colon\mfq\to\mfq$ the map $A\mapsto \eps A$,
for $A\eqdef \eps^{-1} \log U$
we say that $C\in T_A\mfq\simeq\mfq$ is the \emph{rescaled pull-back} of $B$ if $D( \exp\circ\eps)_A(C) =B$,
i.e. $C=\e^{-1}\fancyC$.
\end{definition}
To find the discrete SPDE for $A=\e^{-1}A^{(N)}$, for $A^{(N)}$ as in Theorem~\ref{thm:discrete_dynamics},
we need to find the rescaled pull-back of the right-hand side of~\eqref{eq:discrete_hat_U}.
This can, in principle, be done directly, by computing the right-hand side in $T_U \conf$ and finding its pull-back.
This ends up being a rather lengthy
calculation, so we opt for a somewhat different way.
Namely, we first pull back the Riemannian metric on $W$, coming from the bi-invariant Riemannian metric $\scal{\act,\act}$ on $G$,
to $V$ via the 
Lie exponential map,
so that $\exp\colon V\to W$ becomes an isometry of Riemannian manifolds
(not just a diffeomorphism). 
This endows $V$ with a Riemannian metric $g$
(that differs from its Euclidean metric)
and we work with $V$ instead of $W$ right from the start.

\begin{remark}
For the rest of this subsection, $A,B,C$ denote generic elements in the Lie algebra which are not necessarily related to our equation.
\end{remark}
To write down the Riemannian metric $g$ on
$V$, recall the derivative of the exponential map \cite[Sec.~1.2, Thm.~5]{RossmannBook} 
\begin{equ}[e:d-exp]
\frac{\mrd}{\mrd \tau} e^{A(\tau)} =  e^{A(\tau)} \Phi(A(\tau)) \frac{\mrd}{\mrd \tau} A(\tau)\;,
\end{equ}
where
\begin{equ}
\Phi(A) = \frac{1-e^{-\ad_A}}{\ad_A}
= \sum_{k \geq 0} \frac{(-1)^k}{(k+1)!}(\ad_A)^k
\in L(\mfq,\mfq)\;.
\end{equ}
Therefore, given tangent vectors $B,C\in T_A\mfq \simeq \mfq$,
the required inner product is
\begin{equ}[e:def-g_A]
g_A(B,C) = \scal{e^A \Phi(A) B, e^A \Phi(A) C}_{e^A} = \scal{\Phi(A) B, \Phi(A) C}_{\mfq}\;,
\end{equ}
where in the final equality we used the identification of $T_{e^A} \conf \simeq \mfq$ as left invariant vector fields.
We have thus put all the structure from $W$ that we require onto $V$ and can take gradients of actions on $V$ with respective to the metric $g$.
The picture we keep in mind is Figure~\ref{fig:Riemannian_metric}.
The right-hand side of~\eqref{eq:discrete_hat_U} essentially contains three (unrelated) terms: the DeTurck term, the gradient flow, and the noise.
We derive each of these separately.
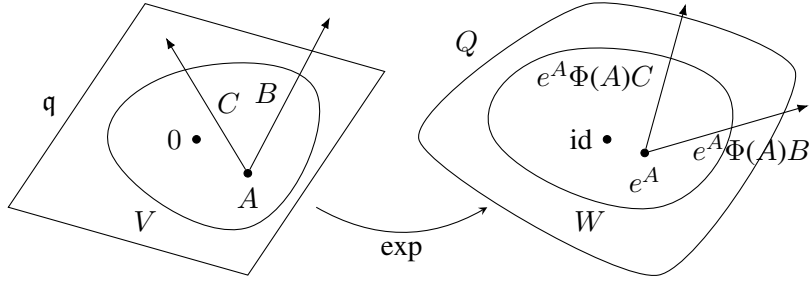
\begin{figure}[ht]
\centering
\begin{tikzpicture}[scale=0.9]
%mfq
\draw[smooth cycle, tension=0, fill=white] plot coordinates{(1.5,2) (-0.5,-1) (3,-2) (5,1)} node at (0.1,0.5) {$\mfq$};

%V
\draw[smooth cycle, tension=0.8, fill=white] plot coordinates{(2,1) (1,-0.2) (3,-1.3) (4,0.7)} node at (1.5,-1.2) {$V$};

%zero
\node [draw,dot,name=zero,label={180:$0$}] at (2.25,0) {};

%A
\node [draw,dot,name=A,label={-90:$A$}] at (3,-0.5) {};

%B arrow
\draw[-{Latex[length=1.5mm]}] (A) -- (4.2,1.8) node[draw=none, midway, left=0.5]{$B$};

%C arrow
\draw[-{Latex[length=1.5mm]}] (A) -- (1.8,1.5) node[draw=none, midway, right=0.5]{$C$};

%Q
\draw[smooth cycle, tension=0.4, fill=white] plot coordinates{(8,2) (5.5,0) (9,-2) (11,1)} node at (6.2,1.4) {$Q$};

%W
\draw[smooth cycle, tension=0.8, fill=white] plot coordinates{(7.6,1.3) (6.6,0) (9,-1) (10,0.7)} node at (8,-1.2) {$W$};

%id
\node [draw,dot,name=id,label={180:$\id$}] at (8.25,0) {};

%e^A
\node [draw,dot,name=eA,label={-90:$e^A$}] at (8.8,-0.2) {};

%eB arrow
\draw[-{Latex[length=1.5mm]}] (eA) -- (11.2,0.5) node[draw=none, midway, below,xshift=3mm,yshift=0.5mm]{$e^A\Phi(A)B$};

%eC arrow
\draw[-{Latex[length=1.5mm]}] (eA) -- (9.4,2) node[draw=none, midway, left=0.5]{$e^A\Phi(A)C$};

%exp
\path[->] (4, -1) edge [bend right] node[below] {$\exp$} (6.5, -1);
    %\draw[white,fill=white] (0.06,-0.57) circle (.15cm);
\end{tikzpicture}
\caption{Illustration of sets $V$, $W$ and tangent vectors $B,C\in T_A\mfq\simeq \mfq$.}
\label{fig:Riemannian_metric}
\end{figure}

\subsection{DeTurck term}

The DeTurck term in~\eqref{eq:discrete_hat_U} is $-\mrd_U\mrd^* U\in T_U\conf$.	
Recall our notation for discrete derivatives from Section~\ref{subsubsec:discrete_der}.
By~\eqref{eq:disc_adjoint_der} one has $(\mrd^* e^\A)(x)= \eps^{-1} (\nabla \cdot \A) (x)$.
%\begin{equ}
%(\mrd^* e^\A)(x)
%= \eps^{-1}\sum_{j=1}^d \big[ \A (x,x+\e_j) - \A(x-\e_j,x) \big]
%= \eps^{-1} (\nabla \cdot \A) (x)\;.
%\end{equ}
Together with \eqref{e:d_U-omega}
one has
\begin{equ}[e:divAeA]
(\mrd_{e^{\A}} \mrd^* e^{\A})(x,y)
 = \eps^{-2}
\big(
(\nabla \cdot \A) (x)\, e^{\A}(x,y) 
- e^{\A}(x,y)\, (\nabla \cdot \A) (y)
\big)
\;.
\end{equ}
Since $-\mrd_U\mrd^* U$ is in the tangent space $T_U \conf$,
we need to pull it back to $\A$ via the $\exp$ map.
More precisely, recalling \eqref{e:d-exp},
the pull back of the DeTurck term to $\A$ is the element  $h \in T_{\A}\mfq$ such that
$-\mrd_{e^{\A}} \mrd^* e^{\A} = e^{\A} \Phi(\A) h$, namely
\begin{equ}[e:-h]
-h = \Phi(\A)^{-1}e^{-\A}(\mrd_{e^{\A}} \mrd^* e^{\A})\;.
\end{equ}
To obtain a series representation for $\Phi(\A)^{-1}$, note that
\begin{equ}
\Phi(\A)^{-1} = (1+\psi(\A))^{-1} = 1+\sum\nolimits_{n\geq1}(-1)^n\psi(\A)^n\;,
\end{equ}
where
$\psi(\A) = \sum_{k \geq 1} \frac{(-1)^k}{(k+1)!}(\ad_{\A})^k
=-\frac12\ad_{\A} + \frac16(\ad_{\A})^2 + O(\A^3)$.
Hence
\begin{equ}\label{eq:Phi_inverse}
\Phi(\A)^{-1}
= 1+\frac12\ad_{\A} + \frac1{12}(\ad_{\A})^2
+O(\A^3)\;.
\end{equ}
We now list the contributions to $-\eps^2 h(e)$ for $e=(x,y)$ of the first three orders, using formulas \eqref{e:divAeA}, \eqref{e:-h}, and  \eqref{eq:Phi_inverse}
%Recall the notation $\A_1=\A(e)$ as in Figure~\ref{fig:plaquettes}.

\textbf{Order 1:}
Taking $e^{\A} = 1 + O(\A)$ and $\Phi(\A)^{-1}=1+O(\A)$, we find
\begin{equ}\label{eq:DeTurck_1}
\Delta_1(e)
\eqdef
 (\nabla\cdot \A)(x) - (\nabla\cdot \A)(y) = -\nabla^+_i (\nabla\cdot \A)(x)\;.
\end{equ}

\textbf{Order 2:}
From the series expansion of $\Phi(\A)^{-1}$,
$e^{-\A}$ and $e^{\A}$, it is easy to see that the 2nd order terms are
\begin{equs}
\Delta_2 (e)
%&\eqdef
%\Delta_2 (\A)(e)
%\\
&\eqdef
\frac12[\A(e),\Delta_1(e)]
-\A(e) \Delta_1(e)
+(\nabla\cdot \A)(x)\A(e) - \A(e)  (\nabla\cdot \A)(y)
\\
&=
-\frac12[\A(e),(\nabla\cdot \A)(x)+(\nabla\cdot \A)(y)]\;.\label{eq:DeTurck_2}
\end{equs}

\textbf{Order 3:} Due to $(\nabla \cdot \A) (x)$ and $(\nabla \cdot \A) (y)$ in \eqref{e:divAeA}, clearly the order 3 part of $\e^2h$ is in $\mcI_3$ (in the sense of Definition \ref{def:ideal}).
(One can verify that the order 3 term is $\Delta_3(e)
\eqdef
\frac1{12}[\A (e),[ \A (e),\Delta_1(e)]]$,
but we do not need its exact expression.)
%The contributions from the first two terms 
%in \eqref{eq:Phi_inverse} are
%\begin{equs}
% e^{-\A} (e) \,&
%\Big(
%(\nabla \cdot \A) (x)\,  e^{\A}(x,y) 
% - e^{\A}(x,y)\, (\nabla \cdot \A) (y)
%\Big)\in \mcI +  O(\A^4)
%\\
%&=
%\frac12 \A(e)^2\, \Delta_1(e)
%-\A(e)\Big((\nabla \cdot \A) (x)\,\A(e) - \A(e)\, (\nabla \cdot \A)(y)\Big)
%\\
%&\quad +\frac12 \Big((\nabla \cdot \A)(x)\, \A(e)^2 - \A(e)^2 \,(\nabla \cdot \A)(y)\Big) +O(\A^4)
%\end{equs}
%and
%\begin{equs}
%{}&\frac12 \ad_{\A(e)} 
%\Big( e^{-\A(e)} \,
%\Big(
%(\nabla \cdot \A) (x)\,  e^{\A}(x,y) 
% - e^{\A}(x,y)\, (\nabla \cdot \A) (y)
%\Big)\Big)\in \mcI + O(\A^4)\;.
%\\
%&=\frac12[ \A(e),-\A(e)\Delta_1(e)
%+(\nabla \cdot \A)(x)\A(e) - \A(e) (\nabla \cdot \A) (y)] +O(\A^4) \;.
%\end{equs}
%Straightforward calculation shows that the above two 
%expressions add up to zero, up to the error term $O(\A^4)$.
%Hence the final third order term only arises from
%\begin{equ}
% \frac1{12}(\ad_{\A(e)})^2
% \Big(
%(\nabla \cdot \A) (x)\,  e^{\A}(x,y) 
% - e^{\A}(x,y)\, (\nabla \cdot \A) (y)
%\Big)
%\end{equ}
%which is
%\begin{equ}\label{eq:DeTurck_3}
%\Delta_3(e)
%\eqdef
%\Delta_3(\A)(e)
%\eqdef
%\frac1{12}[\A (e),[ \A (e),\Delta_1(e)]]\;.
%\end{equ}

\begin{lemma}\label{lem:DT_terms}
There exists $\Delta_3\colon \mfq\to\mfq$ in $\mcI_3$
%(in the sense of Definition \ref{def:ideal})
and analytic $r_h\colon V \to\mfq$ such that $r_h(\A) = O(\A^4)$ 
and for every $e\in \obonds_i$,
\begin{equs}
\eps^2 h (e) &= -\Delta_1(e)-\Delta_2(e)-\Delta_3(\A)(e)+r_h(\A)(e)\;.
\label{e:Delta123}
%\\
%&=
%\nabla^+_i (\nabla\cdot \A)(x)
%+\frac12[\A(e),(\nabla\cdot \A)(x)+(\nabla\cdot \A)(y)] - \Delta_3(e) + r_h(\A)(e)\;,
\end{equs}
%\begin{equs}[e:Delta123]
%\eps^2 h (e) &= -\Delta_1(e)-\Delta_2(e)-\Delta_3(e)+r_h(\A)(e)
%\\
%&=
%\nabla^+_i (\nabla\cdot \A)(x)
%+\frac12[\A(e),(\nabla\cdot \A)(x)+(\nabla\cdot \A)(y)]
%\\
%&\qquad\qquad
%+\frac1{12}[\A(e),[\A(e), \nabla^+_i (\nabla\cdot \A)(x)]]
%+r_h(\A) (e)
%\end{equs}
\end{lemma}

\begin{proof}
The leading terms are obtained by the above derivation.
The remainder 
 $r_h(\A)$ is $O(\A^4)$ (recall Remark~\ref{rem:big-O})
by analyticity of the functions $x\mapsto e^x, 1/(1+x), (1-e^{-x})/x$
on $\R$ in a neighborhood of the origin.
\end{proof}

\subsection{Gradient term}
\label{subsec:gradient_term}

We now turn to the gradient term in~\eqref{eq:discrete_hat_U}.
In view of Sections~\ref{sec:approx-YM} and~\ref{subsec:assumptions},
the action
$\mcS \colon W \to \R$  is  given by 
\begin{equ}[e:lpU]
\mcS(U) = \frac18\sum_{p\in\plaq} \{\eps^{d-4} l^p(U) + R^p_{\eps}(U)\}
\end{equ}
where (recalling $R_\e$ in \eqref{e:S=xR} and $U(\partial p)$ in \eqref{e:U-pp})
\begin{equ}
l^p(U) \eqdef |\log U(\partial p)|_\mfg^2\;, \qquad R^p_\e(U)\eqdef R_{\eps}(U(\d p))\;.
\end{equ}

For a function $f\colon W\to \R$ we write $\hat f \colon V\to \R$ for the pullback $\hat f(\A) \eqdef f(e^\A)$.
By construction, the gradient $\nabla \hat f(\A)\in T_\A V\simeq \mfq$ with respect to the Riemannian metric $g_\A$ is the pull-back of $\nabla f(e^\A)\in T_{e^\A}W$.

For any bond $e\in\obonds$ there are $8$ plaquettes containing $e$ and which enclose the same square in $\T^d$. Therefore
\begin{equ}\label{eq:gradient_mcS}
-\frac12\nabla\hat\mcS(\A)(e) 
= -\frac{\eps^{d-4}}2 \sum_{p\succ e} \nabla \hat l^p(\A)(e)
 - \frac12 \sum_{p\succ e}\nabla\hat R^p_\eps(\A)(e) \;.
\end{equ}
Focusing first on $
 -\frac12{\eps^{d-4}} \sum_{p\succ e} \nabla \hat l^p(\A)
\in T_{\A} V\simeq \mfq
$, which is the 
pull-back of $-\frac12\eps^{d-4} \sum_{p\succ e} \nabla l^p(U) \in T_U W$,
we compute in this subsection its linear, quadratic, and cubic terms.
The terms with $R^p_{\eps}(U)$  in \eqref{eq:gradient_mcS} will be dealt with in Section~\ref{subsec:rem_action}.

We can write
$\hat l^p(\A) = |P(\A)|_\mfg^2$,
where $P\colon V\to\mfg$ is the analytic function 
$P(\A) = \log (e^{\A}(\partial p))$.
To find the gradient $\nabla \hat l^p(\A) \in T_\A \mfq$, we compute for $B\in T_\A\mfq \simeq \mfq$,
\begin{equ}
D \hat l^p(\A) (B) = 2\scal{DP(\A)(B), P(\A)}_\mfg = \scal{B, 2DP(\A)^* P(\A)}_\mfq
\end{equ}
where $DP(\A) \in L(\mfq,\mfg)$ with adjoint $DP(\A)^* \in L(\mfg^*,\mfq^*) \simeq L(\mfg,\mfq)$
and where we used the inner products on $\mfg$ and $\mfq$ to identify them with their duals.

On the other hand by \eqref{e:def-g_A}
\begin{equ}
g_\A(B,\nabla \hat l^p(\A)) = \scal{\Phi(\A) B,\Phi(\A)\nabla \hat l^p(\A)}_\mfq
= \scal{B,\Phi(\A)^* \Phi(\A)\nabla \hat l^p(\A)}_\mfq
\end{equ}
where $\Phi(\A)^*\in L(\mfq^*,\mfq^*) \simeq L(\mfq,\mfq)$.
Hence $2DP(\A)^* P(\A) = \Phi(\A)^* \Phi(\A)\nabla \hat l^p(\A)$
and therefore:		
\begin{lemma}\label{lem:computeDlp}
$\frac12\nabla \hat l^p(\A) = (\Phi(\A)^* \Phi(\A))^{-1} DP(\A)^* (P(\A))$.
\end{lemma}
We proceed to calculate the series expansion of $\frac12\nabla \hat l^p(\A)(e)$.
For this we need to have 
more explicit expressions for $(\Phi(\A)^* \Phi(\A))^{-1}$
and $P(\A)$.

Since $\ad_\A$ is skew-Hermitian, i.e. $\ad_\A^*=-\ad_\A$, one has $(e^{-\ad_\A})^*=e^{\ad_\A}$ and
$\Phi(\A)^*
%= 1+\sum_{k\geq 1} \frac{1}{(k+1)!}(\ad_\A)^k
= \frac{e^{\ad_\A}-1}{\ad_\A}$.
Hence  
\begin{equ}
\Phi(\A)^* \Phi(\A)
=
\Big(\frac{e^{\ad_\A}-1}{\ad_\A}\Big)
\Big(\frac{1-e^{-\ad_\A}}{\ad_\A}\Big)
=
\frac{e^{-\ad_\A}+e^{\ad_\A}-2}{(\ad_\A)^2}
=
1 + \phi(\A)\;,
\end{equ}
where $
\phi(\A) \eqdef
\sum_{k \geq 1}\frac{2}{(2k+2)!} (\ad_\A)^{2k}
$.
Hence expanding $(1+\phi(\A))^{-1}$ in a Neumann series one has
\begin{equ}[e:PhiA2-1]
(\Phi(\A)^* \Phi(\A))^{-1}
%= (1+\phi(\A))^{-1}
= 1 + \sum_{k\geq 1} (-1)^k \phi(\A)^k =
1-\frac{1}{12}\ad_{\A}^2+O(\A^3)\;.
\end{equ}
We calculate $P(\A)$ using
the following  
Baker-Campbell-Hausdorff (BCH) formula.

\begin{lemma}\label{lem:BCH}
For any $A_1,A_2,A_3,A_4\in \mfg$ we have
\begin{equs}
e^{A_1}e^{A_2}e^{A_3}e^{A_4} 
& = \exp\Big(\sum_a A_a + \frac12 \sum_{a<b} [A_a,A_b] + \frac{1}{12} \sum_{a,b}[A_a,[A_a,A_b]]		%\label{e:4BCH}  
\\
&\qquad\qquad
+ \frac16 \sum_{a<b<c} \left([A_a,[A_b,A_c]] + [A_c,[A_b,A_a]]\right) +O(A^4) \Big)	.	
\end{equs}
Here $a,b,c \in \{1,2,3,4\}$
and $O(A^4) $ is a remainder in the sense of Remark~\ref{rem:big-O}
and is analytic in $\{A_i\}_{i=1}^4$ in a neighbourhood of $0$.
\end{lemma}

\begin{proof}
This can be checked via direct calculation, for instance starting from the well-known formula for two matrices, so we omit the details.
%One could also obtain this identity from \cite[Lem.~3.10]{reutenauer93}. 
By %\cite[Prop.~2.2]{MR4137047},
\cite[p.24]{RossmannBook}, the series in the parenthesis on the right-hand side is absolutely convergent for $A_i$ sufficiently small and 
the truncated error is $O(A^4)$ in the sense of Remark~\ref{rem:big-O}.
\end{proof}
Since $(\Phi(\A)^* \Phi(\A))^{-1}$ acts on the ``diagonal'', to calculate $\frac12\nabla \hat l^p(\A)(e)$
we need to find $DP(\A)^*(X)(e)$ for an element $X\in\mfg$
(which is a placeholder for $P(\A)$).
Consider $\bar{\A}\in \mfq \cong\mfg^{\obonds_N}$ such that $\supp \bar{\A}= \{e\}$.
Then, recalling \eqref{e:mfq-inner},
\begin{equ}
\e^{d-2} \scal{DP(\A)^*(X)(e),\bar \A(e)}_\mfg
=\scal{DP(\A)^*(X),\bar \A}_\mfq
= \scal{X,DP(\A)(\bar \A)}_\mfg\;.
\end{equ}
If $p$ does not contain $e$, one has $DP(\A)(\bar \A)=0$ and thus $\nabla \hat l^p(\A)(e)=0$.
It therefore suffices to consider one plaquette $p$ as in Figure~\ref{fig:plaquettes}, and the derivation for the other $p\succ e$ (such as  $\bar p$  Figure~\ref{fig:plaquettes}) simply follows by suitable flipping of components.
Recall our notation $\A_{1,\cdots,4}$ as in Figure~\ref{fig:plaquettes}.
Since $P(\A) = \log (e^{\A_1}e^{\A_2}e^{\A_3}e^{\A_4})$, by the BCH formula (Lemma~\ref{lem:BCH}),
%\begin{equ}
%P(\A) =
%f_{\A}(p) + R_2^p(\A)+ O(\A^3)	\;,	\label{e:PA}
%\end{equ}
%where we recall $f_\A(p)$ from~\eqref{eq:discrete_curv} and where
%$
%R_2^p(\A)=\frac{1}{4}\sum_{a=1}^4[\A_a,\A_2+2\A_3+3\A_4] \in \mcI
%$.
\begin{equ}
P(\A) =
\sum_a \A_a +\frac12 \sum_{a<b}[\A_a,\A_b]+ O(\A^3)	\;.	\label{e:PA}
\end{equ}
Hence
\begin{equ}\label{eq:DP(A)(bar_A)}
DP(\A) (\bar{\A}) 
=
\bar{\A}(e) 
+ \frac12\sum_{b=2}^4 [\bar{\A}(e),\A_b] 
+O(\A^2\bar{\A}(e))\;.
%\\
%&+\frac1{12} \sum_{b>1} 
%\Big([\A_b,[\A_b, \bar{\A}(e)]] 
%+
%[\bar{\A}(e),[\A_1,\A_b]] + [\A_1,[\bar{\A}(e),\A_b]]\Big)
%\\
%&+ \frac16 \sum_{1<b<c} 
%\Big([\bar{\A}(e),[\A_b,\A_c]] + [\A_c,[\A_b,\bar{\A}(e)]]\Big)
%+O(\A^3\bar{\A})
%\\
%&=
%\bar{\A}(e) 
%+ \frac14[\bar{\A}(e),\A_2-\A_4] 
%+\frac{1}{4}[\bar \A(e),\A_2+2\A_3+3\A_4] + O(\A^2\bar \A)\;.
\end{equ}
%\begin{equs}
%\e^{d-2} DP(\A)^* & (X)(e)		\label{e:DPAstar}
%=
%X
%- \frac12\sum_{b=2}^4 [X,\A_b] + O(\A^2 X)
%%\\
%%&+\frac1{12} \sum_{b>1} 
%%\Big([\A_b,[\A_b, X]] 
%%-
%%[X,[\A_1,\A_b]] + [[\A_1,X],\A_b]\Big)
%%\\
%%&+ \frac16 \sum_{1<b<c} 
%%\Big(-[X,[\A_b,\A_c]] + [\A_b,[\A_c,X]]\Big)
%%+O(\A^3X)
%\\
%&=
%X - \frac14[X,\A_2-\A_4]
%-\frac{1}{4}[X,\A_2+2\A_3+3\A_4]
% + O(\A^2 X)\;.
%\end{equs}

\begin{lemma}\label{lem:third_order}
Recalling Definition \ref{def:ideal}, there exist $R^p_3 \in\mcI_3$ and analytic $r^p_l\colon V \to\mfq$ such that $r^p_l(\A)=O(\A^4)$ 
and
%\begin{equ}[e:le]
%- \frac{\e^{d-2}}{2} \sum_{p \succ e}  \nabla \hat l^p(\A) (e)
% = -\nabla^*_\A f_\A(e) -  \sum_{p \succ e}   R_3^p(\A)+ r_l(\A)(e)\;.
%\end{equ}
\begin{equs}
- \frac{\e^{d-2}}{2} \nabla \hat l^p(\A) (e)
 &=
-\sum_{a=1}^4 \A_a - \frac12\big([\A_2,\A_3] + [\A_2,\A_4]+[\A_3,\A_4]\big)\\
&\qquad\qquad + \frac12 [\A_2,[\A_2,\A_1]] 
+ R^p_3(\A)(e) + r^p_l(\A)(e)\;.
\end{equs}
\end{lemma}

\begin{proof}
By \eqref{eq:DP(A)(bar_A)}, we obtain
\begin{equs}
\e^{d-2} DP(\A)^* & (X)(e)		\label{e:DPAstar}
=
X
- \frac12\sum_{b=2}^4 [X,\A_b] + O(\A^2 X)\;.
%\\
%&+\frac1{12} \sum_{b>1} 
%\Big([\A_b,[\A_b, X]] 
%-
%[X,[\A_1,\A_b]] + [[\A_1,X],\A_b]\Big)
%\\
%&+ \frac16 \sum_{1<b<c} 
%\Big(-[X,[\A_b,\A_c]] + [\A_b,[\A_c,X]]\Big)
%+O(\A^3X)
\end{equs}
\textbf{Orders 1 and 2:}
Lemma \ref{lem:computeDlp} and \eqref{e:PhiA2-1}+\eqref{e:PA}+\eqref{e:DPAstar}
immediately imply that
\begin{equs}[eq:1-2_terms_A]
 \frac{\e^{d-2}}{2} \nabla \hat l^p(\A)(e)
%&=\e^{d-2} DP(\A)^*(P(\A))(e)+O(\A^3)
%\\
&= \sum_a \A_a + \frac12\sum_{a<b}[\A_a,\A_b] - \frac12\sum_{a}\sum_{b=2}^4 [\A_a,\A_b]
+O(\A^3)
\\
&=  \sum_a \A_a + \frac12\big([\A_2,\A_3]+ [\A_2,\A_4]+[\A_3,\A_4]\big)
+ O(\A^3)\;.
\end{equs}
\textbf{Order 3:}
The terms at order 3 are obtained as follows:
\begin{enumerate}
\item
The $\A^2$ term of $(\Phi(\A)^* \Phi(\A))^{-1}$ with the $\A$ part (i.e. linear part) of $P(\A)$. 
\item The quadratic part of $P(\A)$ with the linear part of $DP(\A)^*$.
\item
The $\A$ part of $P(\A)$ with the $\A^2$ part of $DP(\A)^*$.
\item
The $\A^3$ part of $P(\A)$.
\end{enumerate}
Since the linear part of $P(\A)$ is in $\mcI$, the 1st and 3rd cases yield terms in $\mcI_3$.
The 4th case gives 
$ \frac12\big([\A_1,[\A_1,\A_2]] -  [\A_2,[\A_2,\A_1]]\big) + \mcI_3$ as can be easily checked by Lemmas~\ref{lem:poly_vanish} and~\ref{lem:BCH}.
Indeed, by Lemma~\ref{lem:poly_vanish} we can replace $\A_{1,\cdots,9}$ by either $\pm \A_1$ or $\pm \A_2$ up to terms in $\mcI_3$, which gives the result.
%The 2nd case gives $- \frac14\sum_{b=2}^4 [ [\A_2,\A_3]+ [\A_2+\A_3,\A_4],\A_b]$,
%which, by Lemma~\ref{lem:poly_vanish}, equals $- \frac14[ [\A_2,\A_3]+ [\A_3,\A_4],\A_3]+\mcI_3=-\frac12 [\A_1,[\A_1,\A_2]]+\mcI_3$,

The 2nd case gives
\[
-\frac14 \sum_{c=2}^4 \Big[\sum_{a<b} [\A_a,\A_b], \A_c\Big] \;,
\]
and again using Lemma~\ref{lem:poly_vanish} in the same way as in the 4th case above, this is equal to $-\frac12 [\A_1,[\A_1,\A_2]]+\mcI_3$,
which cancels the first term in the contribution from the 4th case modulo $\mcI_3$.

%\eqref{e:le} follows from \eqref{eq:1-2_terms_A}-\eqref{eq:A^2_A_part1} and the error bound $r_l(\A)=O(\A^4)$
%follows from Lemma~\ref{lem:BCH}
%and analyticity of $(e^{-x}+e^x-2)/x^2$ and $1/(1+x)$.
We thus obtain the desired equality in which the error bound $r^p_l(\A)=O(\A^4)$
follows from Lemma~\ref{lem:BCH}
and analyticity of $(e^{-x}+e^x-2)/x^2$ and $1/(1+x)$.
\end{proof}

%Note that for any $p\in \plaq$ with $p=(e_1,e_2,e_3,e_4)$ and any $A\in \mfq$, we can define $R_2^p(A)$ as in  \eqref{eq:R2def} and
% $R_3^p(A)$ as in \eqref{eq:def-R3p} with $\A_{1,\cdots,4}$ therein
%replaced by $A(e_1),\cdots,A(e_4)$.

\subsection{Gradient of the remainder}\label{subsec:rem_action}

We now turn to the term $R^p_{\eps}(U)$ in the action~\eqref{e:lpU}.
Recall the pullback  to $V$
\begin{equ}
\hat R_\e^p \colon V\to \R, \quad \hat R_\e^p(\A)= R^p_\e(e^\A) = R_{\eps}(e^\A(\d p))\;.
\end{equ}
To complete the computation of the gradient term, we find $\nabla \hat R_\e^p(\A) \in T_\A V \simeq \mfq$.

\begin{lemma}\label{lem:gradient_R}
Suppose that Assumption~\ref{assump:R} holds and
let $E_{\eps}^{(1)}\in L(\mfg,\mfg^*)$ and $E_{\eps}^{(2)}\in L(\mfg^{\otimes 3},\mfg^*)$ be the corresponding linear operators.
For $p \succ e$ as in Fig. \ref{fig:plaquettes},
\begin{equs}[eq:R_grad_expansion]
\nabla \hat R^p_\e(\A)(e) 
=\e^{2-d} E_{\eps}^{(1)}\Big(\sum_{a=1}^4 \A_a\Big)^* 
&+ \e^{-2} E_{\eps}^{(2)}\Big[\Big(\sum_{a=1}^4 \A_a\Big)^{\otimes 3}\Big]^*
\\
&+ O(\e^{2-d}\A^2) 
+ \e^{-2}O(\A^4)
\end{equs}
where  $F^*\in\mfg$ is the dual of $F\in \mfg^*$, i.e. $\scal{F^*,Y}_\mfg=F(Y)$.
\end{lemma}

\begin{proof}
By a similar derivation as above Lemma~\ref{lem:computeDlp}, for $\bar \A\in T_\A\mfq \simeq \mfq$,
\begin{equ}
D\hat R_\e^p(\A)(\bar \A)
= g_\A(\bar \A,\nabla \hat R_\e^p(\A))
=\scal{\bar \A,\Phi(\A)^*\Phi(\A)\nabla \hat R_\e^p(\A)}_\mfq\;,
\end{equ}
and therefore
\begin{equ}\label{eq:nabla_hat_R_p}
\nabla \hat R^p_\e(\A) = [\Phi(\A)^*\Phi(\A)]^{-1}[D\hat R_\e^p(\A)]^* \in \mfq\;,
\end{equ}
where $F^*\in\mfq$ is the dual of $F\in\mfq^*$.
Since $[\Phi(\A)^*\Phi(\A)]^{-1}$ acts on the diagonal,
it suffices to find $[D\hat R^p_\e(\A)]^*(e)\in\mfg$, which is characterised by
\begin{equ}
\e^{d-2} \scal{[D\hat R^p_\e(\A)]^*(e), \bar \A(e)}_\mfg = D\hat R^p_\e(\A)(\bar \A)
\end{equ}
for all $\bar \A\in\mfq$ supported on $\{e\}$.
Fixing $\bar \A$ supported on $\{e\}$,
by the chain rule
\begin{equ}\label{eq:R_chain_rule}
D\hat R^p_\e(\A)(\bar \A) = [D\hat R_\e(P(\A))](DP(\A)(\bar \A))
\end{equ}
where $P(\A) = \log e^\A (\d p)$ and $\hat R_\e(X) = R_{\eps}(e^X)$ as before.
By~\eqref{eq:DP(A)(bar_A)}, we have
\begin{equ}\label{eq:DPA_bar_A}
DP(\A)(\bar \A) = \bar \A(e) + O(\A\bar \A)\;.
\end{equ}
Furthermore, by Assumption~\ref{assump:R}, we have that $D\hat R_\e(P(\A))](Y)$ for $Y\in\mfg$ equals
\begin{equs}{}
&E_{\eps}^{(1)}(P(\A))(Y) + \e^{d-4}E_\e^{(2)}(P(\A)^3)(Y)
+ O(P(\A)^2Y) + \e^{d-4}O(P(\A)^4Y)
\\
&= E_{\eps}^{(1)}\Big(\sum_{a=1}^4 \A_a\Big)(Y)
+\e^{d-4} E_{\eps}^{(2)}\Big[\Big(\sum_{a=1}^4 \A_a\Big)^{\otimes 3}\Big](Y)
+ O(\A^2Y)
+ \e^{d-4}O(\A^4Y)
\end{equs}
where we used $P(\A)=\sum_{a=1}^4 \A_a + O(\A^2)$ in the final equality.
Combining with~\eqref{eq:R_chain_rule} and~\eqref{eq:DPA_bar_A}, we obtain
\begin{equs}[eq:D_hat_R_A_bar_A]
D\hat R^p_\e(\A)(\bar \A)
= E_{\eps}^{(1)}\Big(\sum_{a=1}^4 \A_a\Big)(\bar \A(e)) &+ 
\e^{d-4} E_{\eps}^{(2)}\Big[\Big(\sum_{a=1}^4 \A_a\Big)^{\otimes 3}\Big](\bar \A(e))
\\
&+O(\A^2\bar \A)
+\e^{d-4}O(\A^4\bar\A)\;.
\end{equs}
Finally, recall from~\eqref{e:PhiA2-1} that
$[\Phi(\A)^*\Phi(\A)]^{-1} = 1+ O(\A^2)$.
The desired expansion~\eqref{eq:R_grad_expansion}
thus follows from~\eqref{eq:nabla_hat_R_p}
and~\eqref{eq:D_hat_R_A_bar_A}.
\end{proof}

\subsection{Noise term}
\label{sec:Noise term}

The final term in~\eqref{eq:discrete_hat_U}
is the noise term $U \circ \dt \BM$.
%Since each $\BM(e)$ for $e\in\obonds$ is a standard $\mfg$-valued Brownian motion,
%$\eps^{-1} \dt \BM(e)$ approximates the value of a white noise $\xi$ at $e$ as $\eps\to0$.
%We therefore write $\xi= \eps^{-1}\circ\dt \BM$.
Let $\xi^\e= \eps^{-1}\dt \BM$ (cf. \eqref{e:couple-noise}).
% $\xi^\e_e = \e^{-d}\scal{\xi,1_{B(e,\e)}}$, which is the average of $\xi$ near $e$.
Below we omit $\circ$ in Stratonovich products.
As with the DeTurck term, we need to pull the vector $U\eps \xi^\e\in T_U \conf$ back to $\A$ via the $\exp$ map.
More precisely, recalling \eqref{e:d-exp},
the pull back of the noise term to $\A$ is the element  $z \in T_{\A}\mfq$ such that
$e^{\A}\eps\xi^\e= e^{\A} \Phi(\A) z$, namely
$z = \eps\Phi(\A)^{-1}\xi^\e$.
We rewrite this as 
\begin{equ}[e:z-terms]
z(e)
=\eps\xi^\e (e)+
\e F(\A(e))[\A(e),\xi^\e(e)]
\end{equ}
where (after possibly shrinking $\mathring V\subset\mfg$), $F\colon \mathring V\to L(\mfg,\mfg)$ is the analytic function
\begin{equ}\label{eq:def_F}
F(X) = \frac{\Phi(X)^{-1}-1}{\ad_X} = \frac12 + \frac{1}{12}\ad_X + O(X^2)\;,
\end{equ}
where we recall the series expansion for $\Phi(X)^{-1}$ from~\eqref{eq:Phi_inverse}.
Below, whenever clear from the context, we often write $\xi$ for $\xi^\e$.
\begin{remark}\label{rem:noise_terms}
Recalling further that $\eps A=\A$ and that $Z\eqdef\eps^{-1}z$ is the rescaled pull-back of $U \circ \dt \BM$,
we obtain
\begin{equs}[e:Z-terms]
Z(e) &= \xi(e) +  \e F(\e A(e))[A(e),\xi(e)]
\\
&= \xi(e)+ \frac\eps2[A(e),\xi(e)] + \frac{\eps^2}{12}[A(e),[A(e),\xi(e)]] + O((\eps A)^3\xi)\;.
\end{equs}
Note that the second and third terms in the 2nd line of~\eqref{e:Z-terms} have negative power counting
and that even $O((\eps A)^3\xi)$ is classically ill-defined (requiring the Stratonovich product in time to make sense).
%Therefore, besides the additive noise $\xi$, we also have a multiplicative noise term.
To handle this `multiplicative noise term', we will introduce a new `abstract' noise modelling $\eps^{1-\kappa} \xi$ and a `multiplication by $\eps$' map sending $A$ to a function-like modelled distribution $\e A$,
and show that $\e F(\e A(e))[A(e),\xi(e)]$ vanishes on the level of modelled distributions.
\end{remark}

\subsection{Rescaled dynamics in Lie algebra}

We collect all the terms obtained from previous subsections.
In the sequel, we let $\A = \log \check U^{(N)}$ for $\check U^{(N)}$ as in~\eqref{eq:discrete_hat_U}.
Recall the exit time $\varpi$ of $\check U$ from $W$.

We rewrite our equation in terms of the ``macroscopic'' variable $A(t,e)= \eps^{-1}\A (t, e)$
%\begin{equ}[e:macro-A]
% A(t,e)= \eps^{-1}\A (\e^2 t, e) 
%\end{equ}
  which will have a form that at least formally ``approximates'' the limiting SPDE,
 modulo various ``remainder'' terms.
 Recall our notation $\Delta$ for discrete Laplacian in Section~\ref{subsubsec:discrete_der} and
various notation
  in Section~\ref{sec:ab-notation}.
%We also recall that 
%$R_3^{\bar p}$ is  defined 
%as in \eqref{eq:def-R3p} but with everything reflected, namely with the replacement 
%$(\A_2,\A_3,\A_4) \mapsto (\A_5,\A_6,\A_7)$.
%
%For the rest of the paper
%the following shorthand notation will be convenient.
%For a function $R^p\colon V\to \mfg$ depending on $p\in\plaq$ we denote
%$R(A) (e) \eqdef \sum_{p\succ e} R^p(A)$ so that $R\colon V\to\mfq$.
We further define
%($p,\bar p$ are again related to $e$ as before)
\begin{equ}[e:from-pp-to-e]
\hat R^{\nabla}_{\eps}\colon V\to \mfq\;\qquad
 \hat R^\nabla_{\eps}(A)(e)
\eqdef 
 \sum_{p\succ e}  \nabla \hat R^p_{\eps}(A)(e) \;.
\end{equ}

\begin{proposition}\label{prop:rescaled-equ}
%Suppose that Assumption~\ref{assump:R} holds. 
For $i\in [d]$ and every $e\in \obonds_i$,
one has on the interval $[0,\varpi]$
\begin{equs}[e:Aeps]
\partial_t A_i  (e)
&= \Delta A_i (e)
+ \sum_{j\neq i}
\big[A_j (e^{(a)}), 
 (2\partial_j A_i)(e)  - (\bar\partial_i A_j) (e) \big] 
\\
&+ [A_i (e), (\partial_i A_i)(e) ]
+  \sum_{j\neq i}\big[A_j(e^{(a)}) ,[A_j(e^{(a)}) ,A_i(e) ]\big]
+ \xi_i (e) 
\\
& 
 + \e F(\e A_i(e))[A_i(e),\xi_i(e)]
%+ \mathrm{Rem}_{e} (A)
+\tilde{R}_2(A)(e)
+I_3(A)(e)
+ \e^{-1} r(\e A)(e)\;,
\end{equs}
%Here, $\mathrm{Rem}_{e}(A)
%=
%\tilde{R}_2(A)(e)
%- R_3(A)(e)
%-\Delta_3 (A)(e)
%+\tilde{R}_3(A)(e)$ 
where $\tilde{R}_2(A)$ is defined
in \eqref{e:tildeR2}
and
$I_3\colon\mfq\to\mfq$ is in $\mcI_3$ (in the sense of Definition \ref{def:ideal}).
%\begin{itemize}
%\item
%$\tilde{R}_2(A)$, $\tilde{R}_3(A)$ are given in \eqref{e:tildeR2} and \eqref{e:tildeR3} below;
%\item
%$R_3(A)$ is  defined 
%by \eqref{e:from-pp-to-e};
%\item
%$\Delta_3(A)$ is  given  in \eqref{e:Delta3r} below.
%\end{itemize}
Moreover, with
%$\hat R^{\nabla}_{\eps}$ as in \eqref{e:from-pp-to-e},
$r_h$ as in \eqref{e:Delta123} and $r_l = \sum_{p\succ e} r_l^p$ for $r_l$ as in Lemma \ref{lem:third_order},
\begin{equ}[e:def-rem-r]
\e^{-1} r(\e A) (e)
=
 -\frac{\e^{-1}}2  \hat R^{\nabla}_{\eps}(\e A)(e)
 +  \e^{-3} r_h ( \e A)(e)
 +  \e^{-3}  r_l(\e A)  (e)\;.
\end{equ}
%$\nabla \hat R^p_{\eps}$ and $\nabla \hat R^{\bar p}_{\eps}$ are as in  \eqref{eq:R_grad_expansion}, and
\end{proposition}

\begin{proof}
For most of the proof 
we omit dependence on $e$ 
and write e.g. $r(\A)=r(\A)(e)$.
Following~\eqref{eq:gradient_mcS}, 
$
\partial_t \A(e) 
= - \frac{\e^{d-4}}{2} \sum_{p\succ e} \nabla \hat l^p(\A) 
-\frac12 \hat R^{\nabla}_\e(\A)
 + h(e) + z(e)
$.
We now add up the linear, quadratic and cubic terms.
We  claim that
\begin{equ}[e:LQT]
\partial_t \A (e) = %L (\A)
\Delta \A (e)
+ Q (\A)+ T(\A) +r(\A) + \eps \xi + \e F(\A(e))[\A(e),\xi]
\end{equ}
where $F$ is given by~\eqref{eq:def_F}.
Here,
$Q(\A) = \sum_{j\neq i}Q_j(\A)$ where for each $j\neq i$
\begin{equs}
 Q_j(\A) &\eqdef 
\e^{-2} \Big[ \frac14  (\A_2-\A_4 -\A_5+\A_7)\; , \;  (\A_6 - \A_3)  -  \frac12 (\A_2 + \A_4 - \A_5 -\A_7)\Big]
 \\
& +\frac{\e^{-2}}{2(d-1)}[\A_1,\A_9 - \A_8] 			\label{e:def-Q}
\\
& +\frac{\e^{-2}}4 [2\A_1 + \A_3 + \A_6 , \A_2+\A_5-\A_4-\A_7] 
 - \frac{\e^{-2}}4 [\A_2 +\A_5, \A_4+\A_7]  
\end{equs}
with
$\A_{1,2,\cdots,9}$
as in Fig.~\ref{fig:plaquettes} (understood as in the $(i,j)$ plane), 
and $T(\A)$ is equal to $\big(\sum_{j\neq i}T_j(\A)\big) -\e^{-2} \big( \sum_{p\succ e}R_3^p(\A) + \Delta_3(\A)\big)$ where $R^p_3$ and $\Delta_3$ are as in Lemmas \ref{lem:third_order}
and \ref{lem:DT_terms} respectively, and
\begin{equ}[e:def-T]
 T_j(\A) 
 \eqdef 
\frac{\e^{-2}}{2} 
\big(
[\A_2, [\A_2, \A_1]]		
+
[\A_5 , [\A_5 , \A_1]]
\big)\;.
\end{equ}
Finally, 
$ r (\A)\eqdef
 -\frac{1}{2} \hat R^{\nabla}_\e(\A)
 +   \e^{-2} r_h (\A)+  \e^{-2}  r_l(\A) $.

To prove the claim,
the linear terms in $- \frac{\e^{d-4}}{2} \sum_{p\succ e} \nabla\hat l^p(\A)$ in Lemma~\ref{lem:third_order}
% are 
%\begin{equ}
%- 2 \A_1 -\sum_{a=2}^7 \A_a\;.
%\end{equ}
plus the 1st order DeTurck term  $-\e^{-2}\Delta_1$
from Lemma~\ref{lem:DT_terms} or \eqref{eq:DeTurck_1}
%\[
%-\Delta_1 = \nabla^+_i(\nabla \cdot \A)(x) = 
% (\A_2-\A_1+\A_5+\A_9)+ (\A_4-\A_1+\A_7+\A_8)
%\]
%we obtain $\e^2 L(\A)$.
clearly gives the discrete Laplacian $\Delta \A (e)$.

For $j\neq i$,
 the quadratic terms from
$- \frac{\e^{d-4}}{2} \sum_{p\succ e} \nabla\hat l^p(\A)$
in the $(i,j)$ plane are
\begin{equ}\label{eq:2nd_order_A}
-\frac{\e^{-2}}2
\Big(
 [\A_2,\A_3] + [\A_2, \A_4] + [\A_3,\A_4] + [\A_5,\A_6] + [\A_5,\A_7]  + [\A_6,\A_7]
\Big)\;.
\end{equ}
The 2nd order DeTurck term $-\e^{-2}\Delta_2$ in Lemma~\ref{lem:DT_terms} is the sum over $j\neq i$ of
\begin{equ}\label{eq:DT_2_As}
%-\Delta_2=
 \frac{\e^{-2}}2
\Big[\A_1\;,\; \frac{1}{d-1}\A_9+\A_2+\A_5 - \Big(\A_4+\frac{1}{d-1}\A_8 + \A_7\Big)\Big]\;.
\end{equ}
By straightforward computation, 
the sum of \eqref{eq:2nd_order_A} and~\eqref{eq:DT_2_As} is equal to $Q_j(\A)$ as in \eqref{e:def-Q}. Here, note that 
the term $[\A_1,\A_9-\A_8]$ is shared 
among all choices of $j\neq i$, 
which is responsible for the factor $1/(d-1)$.

The form of the cubic term  $T(\A)$
in \eqref{e:def-T}
follows from Lemma~\ref{lem:third_order}.
The noise terms and the term
$r(\A)$
also clearly follow from the above calculations.
%The remainders are of the claimed orders by Lemma~\ref{lem:DT_terms}, Lemma~\ref{lem:third_order}.
%and \eqref{e:z-terms}.
This proves the claim.
%%%%%

Next, by Lemma \ref{lem:poly_vanish}, we can 
rewrite $T_j(\A)$ in \eqref{e:def-T}
as
\begin{equ}\label{e:T3}
\e^2 T_j (\A) = 
\frac1{16} [\A_2-\A_4+\A_7-\A_5 , [\A_2-\A_4+\A_7-\A_5,\A_1]]+\tilde{R}_{3,j}\;,
%\frac1{16}  [\A_2-\A_4 , [\A_7-\A_5 , \A_1]]	
%+
%\frac1{16}  [\A_2-\A_4 , [\A_2-\A_4 , \A_1]]	
%\\
%&+\frac1{16}  [\A_7-\A_5 , [\A_2-\A_4 , \A_1]]
%+
%\frac1{16}  [\A_7-\A_5 , [\A_7-\A_5 , \A_1]] +\tilde{R}_{3,j}(\A)
\end{equ}
%which equals $\frac1{16} [\A_2-\A_4+\A_7-\A_5 , [\A_2-\A_4+\A_7-\A_5,\A_1]]+\tilde{R}_{3,j}$
where
$\tilde{R}_{3,j}\in \mcI_3$.
Remark that $\frac14(\A_2-\A_4+\A_7-\A_5) = \e A_j(e^{(a)})$.
%\begin{equs}[e:def-tildeR3]
%\tilde{R}_{3,j}(\A)(e)
%&\eqdef
%\frac1{16}  [\A_2-\A_4 , [\A_2-\A_4 +\A_5-\A_7 , \A_1]]	
%\\
%&-
%\frac1{16}  [\A_2-\A_4 , [\A_2-\A_4 , \A_1+\A_3]]	
%\\
%&+\frac1{16}  [\A_7-\A_5 , [\A_7-\A_5 -\A_2+\A_4 , \A_1]]
%\\
%&-
%\frac1{16}  [\A_7-\A_5 , [\A_7-\A_5 , \A_1+\A_6]] \;.
%\end{equs}
%For instance, the 
%first two terms on RHS of \eqref{e:T3} + the first two terms in $\tilde R_{3,j}(\A)$
%is equal to the first term of RHS of \eqref{e:def-T}.

%%%%
Let $j\neq i$ for the rest of the proof.
%As in the proof of Proposition~\ref{prop:LQT}, 
%all unspecified terms are evaluated at $e$.
We rewrite the above equation 
%obtained in Proposition~\ref{prop:LQT}
with $A= \eps^{-1}\A $.
The linear term remains as the discrete Laplacian.
By \eqref{e:def-Q} and recalling the notation
 $A_j (e^{(a)})$, $\partial$, $\bar\partial$, $\d^2$, from Section \ref{sec:more_notation},
the term $\e^{-1}Q(\e A)$ is equal to 
the 2nd and 3rd term on the right-hand side of \eqref{e:Aeps} plus
\begin{equ}[e:tildeR2]
\tilde{R}_2(A)(e) \eqdef  \sum_{j\neq i}
- \frac{\eps^2}{2} [\partial_j^2 A_i(e), \bar\partial_j A_j(e)]
+ \frac{\eps}4 
[ (\partial^+_j A_j) (e^\southeast)   , 
(\partial^+_j A_j)(e^\southwest)    ] 
\end{equ}
(which arises from the last line in  \eqref{e:def-Q}).
%(recall the notation $\bar\partial_j A_j$ in  \eqref{e:djAj} and $\d_j^2 A_i$ in~\eqref{e:laplacian})
%\begin{equ}
% \big[ A_j (e^{(a)}), 
%(2\partial_j A_i)(e)  - (\bar\partial_i A_j) (e) \big] 
%+ [A_i (e), (\partial_i A_i)(e) ] + \tilde{R}_2(A)\;.
%\end{equ}
Moreover, since all the terms on the right-hand side of \eqref{e:T3} are cubic in $\A$,
we see that $\e^{-1} T(\e A) $ is equal to the cubic terms in
\eqref{e:Aeps} with $I_3(A)(e) =  \sum_{p\succ e} R_3^p(A) - \Delta_3(A)(e) +\sum_{j\neq i}\tilde{R}_{3,j}(A)(e)$.
%\[
%\e^{-1} T(\e A) = \Big(\sum_{j\neq i} \big[A_j(e^{(a)}) ,[A_j(e^{(a)}) ,A_i(e) ]\big]\Big)
%-   R_3(A) - \Delta_3(A) +\tilde{R}_3(A)\;.
%\]
%where by \eqref{e:def-tildeR3} we can rewrite
%\begin{equs}[e:tildeR3]
%\tilde{R}_3(A)(e)
%&=\sum_{j\neq i}
%\frac\eps{16}  [A_j(e^\northeast)+A_j(e^\northwest)\; , \; 
%		[\bar\partial_j A_j(e), A_i(e)]]	
%\\
%&\qquad+
%\frac\eps{16} [A_j(e^\northeast)+A_j(e^\northwest) \; ,\;
%		 [A_j(e^\northeast)+A_j(e^\northwest) , \partial_j^+ A_i(e)]]	
%\\
%&\qquad+
%\frac\eps{16}  [A_j(e^\southeast)+A_j(e^\southwest ),
%		 [-\bar\partial_j A_j(e) , A_i(e)]]
%\\
%&\qquad+
%\frac\eps{16}  [A_j(e^\southeast)+A_j(e^\southwest ) ,
%		 [A_j(e^\southeast)+A_j(e^\southwest ) ,  \partial_j^- A_i(e)]] \;.
%\end{equs}
%Recall from
%\eqref{eq:DeTurck_3} that
%\begin{equ}%[e:Delta3-2]
%\Delta_3 (\A)(e) =
%- \frac1{12}[\A_1,[ \A_1,
%\A_4+\A_2+\A_7+\A_5 ]]
%+
%\frac1{12}[\A_1,[ \A_1,
%2\A_1 -\A_8  - \A_9]]
%\end{equ}
%Thus
%Recalling \eqref{eq:DeTurck_3}, one has, for $e=(x,y)$,
%\begin{equ}[e:Delta3r]
%-\Delta_3 (A)(e)=
%% \frac{\eps}{6}[A_i(e),[A_i(e),
%%\bar\partial_j A_j(e) ]]
%%+\frac{\eps^2}{12}[A_i(e),[A_i(e),
%%\partial_i^2 A_i (e)]] \;.
% \frac{\eps^2}{12}[A_i(e),[A_i(e),
%\partial_i^+ \div A(x) ]] \;.
%\end{equ}
Finally the expression for $\e^{-1} r(\e A) $
follows obviously
and the terms with $\xi$ follow from~\eqref{e:LQT} (or from~\eqref{e:Z-terms}).
\end{proof}

\begin{remark}\label{rem:power_of_eps}
%The first two lines of \eqref{e:Aeps} obviously have a form which 
%``approximates'' \eqref{eq:SYM_DeTurck_coord} in any $d\ge 1$.
%Every term in  $\mathrm{Rem}_{e}$ actually contains at least one power
%of $\eps$. Indeed, 
%this is clear 
%for $\tilde{R}_2, \tilde{R}_3$ and $\Delta_3$ from their expressions.  
%It will also be clear for 
%$R_3$ when we rewrite them in Section~\ref{subsec:Mass renormalisation_1}.

The terms in \eqref{e:def-rem-r} all contain a power of $\e$.
Indeed,
this is the case for $\e^{-3} r_h ( \e A)$ and  $\e^{-3}  r_l(\e A)$ 
since $r_h(\A)$ in  \eqref{e:Delta123} and $r_l(\A)$ in  Lemma \ref{lem:third_order} are both $O(\A^4)$.
Regarding $\e^{-1}  \hat R^{\nabla}_{\eps}(\e A)$, under Assumption~\ref{assump:R} and by
Lemma~\ref{lem:gradient_R},
$ \hat R^{\nabla}_{\eps}(\A)$ is of the form 
\eqref{eq:R_grad_expansion}.
So assuming $d=2$, 
the terms  $O(\A^2)$ and $\e^{-2}O(\A^4)$ in \eqref{eq:R_grad_expansion} contribute a term $\eps O(A^2)$ and $\eps O(A^4)$
 to $\e^{-1} r(\e A) $,
which contains a power of $\eps$.
For the other terms in \eqref{eq:R_grad_expansion},
upon identifying $\mfg^*$ with $\mfg$,
the term $E^{(1)}_{\eps}(\sum_{a=1}^4 \A_a)$ contributes to $\e^{-1} r(\e A) $ a term
$\e E^{(1)}_{\eps}(\d_i A_j - \d_j A_i)$
which does have a factor $\e$.
Lastly, the term $\e^{-2} E^{(2)}_{\eps}(\sum_{a=1}^4 \A_a)^{\otimes 3}$ contributes to $\e^{-1} r(\e A) $ a term
$
\e E^{(2)}_{\eps}\big[(\d_i A_j - \d_j A_i)\otimes A\otimes A\big]
$
again containing a power of $\e$.
\end{remark}

%\begin{remark}
%The first two lines of \eqref{e:def-Q} 
%are the ``natural symmetric discretizations''
%to the quadratic terms in the $B_i$ component of the DeTurck-YM equation
%\begin{equ}\label{eq:2nd_order_B}
%[B_j, 2\partial_jB_i   -\partial_iB_j] + [B_i, \partial_i B_i ] \;, \quad j\neq i\;.
%\end{equ}
%\end{remark}
%

\section{Discrete regularity structures}
\label{sec:DiscreteRS}

\subsection{Discrete function spaces}\label{sec:funcSpaces}

Denote by $\T_\e^d \subset \T^d$ a generic square lattice with lattice spacing $\e$ which is embedded in $\T^d$.
Note that $\Lambda_N$ (where $\e=2^{-N}$) and $\obonds_j$
for $j\in [d]$ are all incidences of
 $\T_\e^d$, which are embedded in $\T^d$ in the natural way (i.e. identify midpoint of a bond with a point in $\T^d$).
We will write $A\asymp B$ if $A$ is bounded from above and below by $B$ up to proportionality constants uniform in $\e$ (or equivalently, in $N$).

We will write $\int_{\T_\e^d}$ for $\e^d \sum_{\T_\e^d}$.
Let $(E, |\cdot |)$ be a Banach space.
Given a function $f^\e\in E^{\T_\e^d}$ and  a smooth function $\phi\colon \T^d\to \R$,
we write $ \langle f^\eps, \varphi \rangle_\eps $ or $f^\e(\varphi)$
for $ \int_{\T_\eps^d} f^\eps(y) \varphi(y)\mrd y$.
Whenever it is clear from the context, 
for space-time functions $f^\e$ on $[0,T]\times \T_\e^d$ and smooth $\phi$ on $[0,T]\times \T^d$
we also write pairings as $ \langle f^\eps, \varphi \rangle_\eps $ or $f^\e(\varphi)$
with an additional integral in time.

Define the parabolic scaling $\s=(\s_0,\cdots,\s_d)=(2,1,\cdots,1)$ and write $|\s|=d+1$.
For a multi-index $k=(k_0,\ldots,k_d)\in \N^{d+1}$,
we write $|k|_\s = \sum_{i=0}^{d}\s_ik_i$.
For $z=(z_0,\cdots,z_d)\in\R^{d+1}$, let
\[
\|z\|_\s = \sup\{ |z_i|^{1/\s_i}:0\le i\le d\}\;,
\qquad
\CS_\s^\lambda z = (\lambda^{-\s_0} z_0,\cdots, \lambda^{-\s_d} z_d)\;.
\]
For $z=(z_0,\cdots,z_d)\in\R\times \T^d$, we define $\|z\|_\s$ in the same way.

For $\phi:\R^{d+1}\to \R$, we write $\phi^\lambda_z (y) = \lambda^{-|\s|} \phi(\CS_\s^\lambda(y-z))$.
We will also use rescaled test functions in time or space only:

(1) For $\phi:\R\to \R$, recalling $\s_0=2$, we write
 $\CS_{2,t}^{\lambda}\phi = \lambda^{-2} \phi(\frac{\cdot - t}{\lambda^2})$
for a rescaled temporal test function centred at $t$.

(2) For $\phi:\R^d\to \R$, we write
 $\CS_{x}^{\lambda}\phi = \lambda^{-d} \phi(\frac{\cdot - x}{\lambda})$
for a rescaled spatial test function centred at $x$,
and with a slight abuse of notation we also write it as $\phi^\lambda_x$.

Now we define some norms and distances 
for functions over space or over space-time.
Let $\alpha\in [0,1]$.
To compare functions $f\in\CC_\s^\alpha ([0,T]\times \T^d,E)$ 
 and $f^\e\colon [0,T]\times \T_\e^d \to E$, we define
\begin{equs}[e:discHolderT]
\|f;f^\e\|_{\CC_\e^{\alpha,T}}
&\eqdef 
\sup_{z\in [0,T]\times \T_\e^d}|f(z)-f^\e(z)|
+ \sup_{\substack{z, w\in [0,T]\times\T^d\\ \|z-w\|_\s < \e}}
\frac{|f(z)-f(w)|}{\|z-w\|_\s^\alpha}
\\
& + \sup_{\substack{z, w\in [0,T]\times\T_\e^d\\ \|z-w\|_\s \geq \e}} 
	\frac{|(f(z)-f(w))-(f^\e(z)-f^\e(w))|}{\|z-w\|_\s^\alpha}\;.
\end{equs}
We then define $\|f^\e\|_{\CC_\e^{\alpha,T}} = \|0;f^\e\|_{\CC_\e^{\alpha,T}}$ 
and spaces $\CC_\e^{\alpha,T}$ with this norm.

For any $f^\e\in E^{\T_\e^d}$ and $f\in\CC^\alpha(\T^d,E)$ in the usual H\"older space, 
we define $\|f;f^\e\|_{\CC^\alpha_\e}$ as in \eqref{e:discHolderT} 
but with $\|z-w\|_\s$ replaced by $|z-w|$ and with supremums over $\T^d_\e$ and $\T^d$.
%\begin{equs}[e:discHolder]
%\|f;f^\e\|_{\CC^\alpha_\e}
%\eqdef
%&\sup_{x\in\T_\e^d}|f(x)-f^\e(x)| +\sup_{\substack{x\neq y\in\R^d:\, |x-y|< \e}}
%\frac{|f(x)-f(y)|}{|x-y|^\alpha}\\
% &+ \sup_{x\neq y\in\T_\e^d}
%\frac{|(f(x)-f(y))-(f^\e(x)-f^\e(y))|}{|x-y|^\alpha}\;.
%\end{equs}
We then define $\|f^\e\|_{\CC_\e^\alpha}=\|0;f^\e\|_{\CC_\e^\alpha}$
and define discrete H\"older spaces $\CC_\e^\alpha(\T_\e^d,E)$ as the space of all elements $f^\e\in E^{\T_\e^d}$ with this norm.

Note that $\|f^\e\|_{\CC_\e^{0,T}} \asymp \|f^\e \|_{L^\infty([0,T]\times \T^d_\e)}$,
and likewise for the spatial norms.
We sometimes write $L^\infty_\e$ for $\CC^0_\e$.

\begin{remark}\label{rem:small_scales}
Although our time is continuous, the norms on $\CC_\e^{\alpha,T}$ do not measure oscillations when
$z=(t,x)$ and $w=(s,x)$
get very close (i.e. when $|t-s|<\e^2$). 
These oscillations in very small time scales will be measured by 
another collection of norms $\|\cdot \|_{\alpha;\K_\eps;z;\eps}$
defined in \eqref{eq:seminorm}.
\end{remark}
Now let $\alpha<0$.
For discrete distributions $f^\e$ on $[-1,T+1]\times\T^d_\e$
and distributions $f$ on $[-1,T+1]\times\T^d$ 
we define 
\begin{equ}\label{eq:f_e_f_alpha_neg}
\Vert f^\e ;f \Vert_{\CC^{\alpha,T}_\e} 
\eqdef 
\sup_{\varphi \in \CB^r_0} 
\sup_{z \in [0,T]\times \T_\e^d} 
\sup_{\lambda \in [\e,1]} 
\lambda^{-\alpha} 
|\langle f^\e, \varphi_z^\lambda \rangle_\e
-\langle f, \varphi_z^\lambda \rangle | \;,
\end{equ}
where $r$ is the smallest integer such that $r>-\alpha$ and $\CB^r_0$ is the set of smooth functions $\phi\in \CC^\infty(\R\times \T^d)$ with $|\phi|_{\CC^r}\leq 1$
and $\supp(\phi) \subset \{z\in\R\times\T^d\, :\, \|z\|_\s\leq 1/4\}$.
We define $\Vert f^\e \Vert_{\CC^{\alpha,T}_\e} =\Vert f^\e;0 \Vert_{\CC^{\alpha,T}_\e} $.

For discrete functions $f^\e$ on $\T^d_\e$ and distributions $f$ on $\T^d$
we define
%\begin{equ}[eq:f_e_f_alpha_neg_spatial]
%\Vert f^\e ;f \Vert_{\CC^{\alpha}_\e}
%\eqdef 
%\sup_{\varphi \in \CB^r_{\T^d}} 
%\sup_{z \in \T_\e^d} 
%\sup_{\lambda \in [\e,1]} 
%\lambda^{-\alpha} 
%|\langle f^\e, \varphi_z^\lambda \rangle_\e
%-\langle f, \varphi_z^\lambda \rangle | \;,
%\end{equ}
$\Vert f^\e ;f \Vert_{\CC^{\alpha}_\e}$ as in \eqref{eq:f_e_f_alpha_neg}
but with $z\in \T^d_\e$ and $\phi \in \CB^r_{\T^d}$,
where 
$\CB^r_{\T^d}$ is the set of smooth functions $\phi\in \CC^\infty(\T^d)$ with $|\phi|_{\CC^r}\leq 1$
and $\supp(\phi) \subset \{z\in\T^d\, :\, |z|\leq 1/4\}$.
We define $\Vert f^\e \Vert_{\CC^{\alpha}_\e} =\Vert f^\e;0 \Vert_{\CC^{\alpha}_\e} $.
Note that if we take $\alpha=0$ in this definition, it defines a norm equivalent to $\|f^\eps\|_{L^\infty(\T^d_\e)}$ uniformly in $\eps$ (which is then also equivalent with $\CC_\e^0$ defined above).
In contrast, if we take $\alpha=0$ and $f=0$ in~\eqref{eq:f_e_f_alpha_neg}, this does \textit{not} define a norm equivalent with $\|f\|_{L^\infty([0,T]\times \T^d_\e)}$.

%\begin{remark}
%When specifically referring to $\Lambda_N$ or $\obonds_i$ instead of a generic lattice $\T^d_\e$,
%we may write $\CC_\e^\alpha(\Lambda_N)$ or $\CC_\e^\alpha(\obonds_i)$, or omit these specific lattices when it is clear from the context.
%\end{remark}
We write $\CC_\e^\alpha(\Lambda_N)$ or $\CC_\e^\alpha(\obonds_i)$
when specifically referring to a particular lattice. 

The following simple facts will be helpful in the bounds for
the ``remainder terms'' showing up in our discrete equation. We will assume that $f$ is either a function on $\T^d_\e$ or a distribution on $[0,T]\times \T^d_\e$, and 
consider
a general discrete derivate $D^\eps f =\frac1{\eps}\sum_k a_k f(\cdot+k)$ with $\sum_k a_k=0$, where $k$ sums over  a finite set of  $\eps$-vectors.
For instance,  $k$ sums over $\{0,\eps_j\}$ for $\partial_j^+$.
%and over $\{0,\eps_1,\eps_2,\eps_1+\eps_2\}$ for $\bar\partial_1$.

\begin{lemma}\label{lem:eps-improve-reg}

\begin{enumerate}[label=(\roman*)] 
\item \label{pt:eps_f_vanish} There exists $C>0$ such that, for all $\alpha \leq \bar\alpha\leq  0$,
one has $\|f\|_{\CC^{\bar \alpha}_\eps} \leq C \eps^{\alpha-\bar\alpha}\|f\|_{\CC^\alpha_\eps}$
(if $\bar\alpha<0$, we can take $C=1$).
The same holds for the space-time norm $\|f\|_{\CC^{\alpha,T}_\e}$ provided that $\alpha\le\bar\alpha<0$.

\item\label{pt:eps_d_f_vanish} 
One has $\| \e D^\eps f \|_{\CC^{\bar\alpha}_\e}
\lesssim \e^{\alpha-\bar\alpha}\|f\|_{\CC^\alpha_\eps}$
for $\alpha,\bar\alpha\in[0,1]$.
The same holds for the space-time norm $\|f\|_{\CC^{\alpha,T}_\e}$.

\item\label{pt:eps_d_f} 
One has $\|D^\eps f\|_{\CC_\e^{-1}} \lesssim \|f\|_{L^\infty(\T^d_\eps)}$.
Moreover,
 $\| \e D^\eps f \|_{\CC^{\bar\alpha}_\e}
\lesssim \e^{\alpha-\bar\alpha}\|f\|_{\CC^\alpha_\eps}$
 for $\alpha,\bar\alpha\in [-1,0]$.
\end{enumerate}
\end{lemma}

\begin{proof}
%\ref{pt:eps_f_vanish} essentially follows from definitions; one needs only recall that $\lambda\in[\e,1]$ in the definition of $\|\cdot\|_{\CC^\alpha}$ for $\alpha<0$.
For \ref{pt:eps_f_vanish}, recalling the definition \eqref{eq:f_e_f_alpha_neg}
where $\lambda \ge \e$,
if $\alpha\leq\bar\alpha<0$
the bound follows from $\lambda^{-\bar\alpha}\le \e^{\alpha-\bar\alpha}\lambda^{-\alpha}$.
If $\alpha\leq \bar\alpha=0$, the case $\alpha=\bar\alpha$ is trivial, while if $\alpha<\bar\alpha$, then
taking $\varphi \in \CB^r_{\T^d}$ with $\varphi(0)=1$,
one has $|f(x)|=|\langle f, \varphi_x^\e \rangle_\e| 
\le \e^{\alpha}\|f\|_{\CC^\alpha_\eps}$.
For \ref{pt:eps_d_f_vanish}-\ref{pt:eps_d_f}, note that it suffices to consider $D^\e=\d^{\pm}_j$ due to the triangle inequality.
% and because in general $D^\eps f (x)=\frac1{\eps}\sum_k a_k (f(x+k)-f(x))$.
For \ref{pt:eps_d_f_vanish}, recalling \eqref{e:discHolderT} 
\begin{equs}
\|\e \partial_j^\pm f\|_{L^\infty(\T^d_\e)}
&= \| f(\cdot \pm \e_j)-f(\cdot)\|_{L^\infty(\T^d_\e)}
 \le \e^\alpha \|f\|_{\CC^\alpha_\e}
\\
\frac{|\eps\partial_j^\pm f(x)-\eps\partial_j^\pm f(y)|}{|x-y|^{\bar\alpha}}
&\le 
\frac{|f(x)-f(x\pm \e_j)| + |f(y)-f(y\pm \e_j)|}{|x-y|^{\bar\alpha}}
\lesssim \frac{ \e^\alpha \|f\|_{\CC^\alpha_\e}}{\e^{\bar\alpha}}.
\end{equs}
% and $|x-y|\geq \e$
%for all $x\neq y\in\T^d_\e$, which leads to the desired bound on $\| \e \partial f \|_{\CC^{\bar\alpha}_\e}$ by the triangle inequality.
The first bound in \ref{pt:eps_d_f} follows from summation by parts and the fact that $\|\phi^\lambda_z\|_{\CC^1}$ appearing in \eqref{eq:f_e_f_alpha_neg}
is bounded by $\lambda^{-1}$.
The second bound in \ref{pt:eps_d_f} then follows from
\[
\|\e \d_j^{\pm} f\|_{\CC^{\bar\alpha}_\e}
\lesssim \e^{-\bar\alpha}\|\d_j^{\pm} f\|_{\CC^{-1}_\e} \lesssim \e^{-\bar\alpha}\|f\|_{L^\infty(\T^d_\e)} \lesssim \e^{\alpha-\bar\alpha}\|f\|_{\CC_\e^{\alpha}}\;,
\]
where the first and third bounds follow from \ref{pt:eps_f_vanish}.
\end{proof}
We will also need ``inhomogeneous'' norms which distinguish time and space and allow possible blow-up at $t=0$.
For $f^\e\colon [0,T]\times \T^d_\e \to E$, $\alpha\leq 1$, $\eta \le 0$, define
\begin{equ}[e:def-C-alpha-eta]
\| f^\e\|_{\CC^{\alpha,T}_{\eta, \e}} \eqdef \sup_{t \in (0, T]}
(t^{\frac12} \wedge 1)^{-\eta} \|f^\e (t)\|_{\CC^\alpha_\e} \;.
\end{equ}
We remark that Young's theorem holds 
for our discrete spaces on $\T^d_\e$. 

\begin{lemma}\label{lem:discrete_Young}
Let $\alpha<0$ and $\beta \in (0,1)$ with  $\alpha+\beta>0$. 
Let $f^\e,g^\e\colon \T^d_\e\to \R$.
Then $\|f^\eps g^\eps\|_{\CC^\alpha_\e} \lesssim \|f^\e \|_{\CC^\alpha_\e} \|g^\e\|_{\CC^\beta_\e} $ uniformly in $\e$.
\end{lemma}

\begin{proof}
The proof follows in a straightforward way by extending $f^\e,g^\e$ to the continuum via piecewise constant and bilinear interpolation respectively and applying (the continuum) Young's product theorem.
%We extend $f^\e$ to $f\in L^\infty(\T^d)$ via piecewise constant (nearest neighbour) extension and extend $g^\e$ to $g\in L^\infty(\T^d)$ via bilinear interpolation.
%Note that $\|f\|_{\CC^\alpha} \lesssim \|f^\e \|_{\CC^\alpha_\e}$ and $\|g\|_{\CC^\beta}\lesssim  \|g^\e\|_{\CC^\beta_\e}$
%uniformly in $\eps$.
%Furthermore, for any $\phi\in \CB^1_{\T^d}$ and $\lambda \in [\eps,1]$, $z\in \T^d$, $x\in\T^d_\e$, and $y$ in the hypercube centred at $x$ with side length $\e$, we have $f(x)=f(y)$ and
%\begin{equ}
%|g(x)-g(y)|\lesssim \e^\beta\|g^\e\|_{\CC^\beta_\e}\;,
%\quad
%|\phi^\lambda_z(x)-\phi^\lambda_z(y)| \lesssim
%\lambda^{-d-1}\e \|\phi\|_{\CC^1}\;,
%\end{equ}
%from which it readily follows, using $\|f^\e\|_\infty \lesssim \eps^{\alpha}\|f^\e\|_{\CC^\alpha_\e}$, that
%\begin{equ}
%\lambda^{-\alpha}|\scal{f^\e g^\e,\phi^\lambda_z}_\e| \lesssim
%\lambda^{-\alpha}|\scal{fg,\phi^\lambda_z}|
%+
%\e^{\alpha+\beta}\lambda^{-\alpha-\beta}\|f^\e\|_{\CC^\alpha_\e}\|g^\e\|_{\CC^\beta_\e}\|\phi\|_{\CC^1} \;.
%\end{equ}
%Since $\alpha+\beta>0$, $\lambda\geq \eps$, and
%$\lambda^{-\alpha}|\scal{fg,\phi^\lambda_z}|\lesssim \|f\|_{\CC^\alpha}\|g\|_{\CC^\beta}$
%by the classical Young's product theorem, the conclusion follows.
\end{proof}

\subsubsection{Comparison with line integral spaces}
\label{subsubsec:compare_line_ints}

Recall the discrete spaces $\Omega_{N;\alpha}$ with projective limit $\Omega_\alpha$ from Section~\ref{sec:norms}.
We will identify every $A\in\Omega_N$ with the family of functions $(A_i)_{i\in[d]}$ with $A_i\colon  \obonds_i\to \mfg$ given by $A_i(b) = \e^{-1} A(b)$.
Note that $A_i$ differs from $A\restr_{\obonds_i}$ for $A\in\mfg^{\obonds}\simeq \Omega_N$ (see under \eqref{eq:A_def_1_form})
by a factor of $\e^{-1}$; we will always denote this difference through the subscript $A_i$.

We also denote $\|A\|_{\CC^\eta_\e}=\max_{i\in[d]}\|A_i\|_{\CC^\eta_\e}$ for $A\in\Omega_{N}$,
%where $\T^d_{\e,i}$ is the lattice given by the midpoints of $\obonds_i$ and $b_m$ is the midpoint of $b\in \obonds_i$.
for which we have the bounds  $|A|_{N;\alpha} \lesssim \|A\|_{\CC_\e^{\alpha/2}}$
and $\|A\|_{\CC^{\alpha-1}_\e} \lesssim  |A|_{N;\alpha}$
uniformly in $\e=2^{-N}$, which are the discrete analogues of~\eqref{eq:cont_embedding} and \eqref{eq:dist_embedding} respectively
(the latter of these is proved in the same way as~\cite[Prop.~3.21]{Chevyrev19YM}, see also the proof of Lemma~\ref{lem:A_pi_N_A} below).
We therefore have the uniform in $\eps$ embeddings
\begin{equ}[eq:embeddings]
\CC^{\alpha/2}_\e \hookrightarrow \Omega_{N;\alpha} \hookrightarrow \CC^{\alpha-1}_\e\;.
\end{equ}
The following lemma 
compares $A\in\Omega_\alpha^1$ with its projection $\pi_N A\in\Omega_N$.

\begin{lemma}\label{lem:A_pi_N_A}
For $\alpha\in (0,1]$, $\eta\in [\frac\alpha2-1,\alpha-1]$, $i\in[d]$,
and $A=(A_1,\ldots,A_d)\in\Omega_\alpha^1$, we have $\|A_i;\pi_N A_i\|_{\CC^\eta_\e} \lesssim \e^{\alpha-1-\eta}|A|_\alpha$ uniformly in $\e$.
\end{lemma}
For the proof, let $\CD^{\var p}([a,b],E)$ denote the space of c\`adl\`ag (right-continuous with left-limits) functions $X\colon[a,b]\to E$ with finite $p$-variation\label{page:variation_norm}
\begin{equ}[eq:p-var]
|X|_{\var p;[a,b]} \eqdef \sup_{P\subset [a,b]} \Big(\sum_{[s,t]\in P} |X_t - X_s|^p\Big)^{1/p}
\end{equ}
for $p\in[1,\infty)$ and where the $\sup$ is over all partitions $P$ of $[a,b]$ into disjoint (modulo endpoints) intervals.
We also denote $|X|_{\var\infty}= \sup_{s,t\in[a,b]} |X_t-X_s|$.
Denote by $\CC^{\var p}\subset \CD^{\var p}$ the subspace of continuous functions.

Consider $1\leq p,q  \leq \infty$
such that $\frac1p+\frac1q\geq 1$ where equality is allowed only if $p,q\in\{1,\infty\}$.
For $Y\in\CD^{\var q}([a,b],\R)$ and $X\in\CD^{\var p}([a,b],E)$,
recall Young's estimate for integrals
(see \cite[Prop.~1.9,~2.4]{FZ18} and also \cite[Thm.~6.8]{FV10})
\begin{equ}[eq:Young_int]
\Big|\int_a^\cdot Y_t\mrd X_t\Big|_{\var p;[a,b]} \leq C_{p,q} (|Y_0|+|Y|_{\var q;[a,b]})|X|_{\var p;[a,b]}\;,
\end{equ}
where the integral is understood as the limit of Riemann sums
with $Y$ evaluated at the left ends of the subintervals.
%\begin{equ}
%\int_a^s Y_t\mrd X_t = \lim_{|P| \to 0} \sum_{[u,v]\in P} Y^-_u(X_v-X_u)\;,
%\end{equ}
%where $|P|=\max_{[u,v]\in P}|v-u|$ is the mesh of a partition $P$ of $[a,s]$
%and $Y_u^- \eqdef \lim_{r\uparrow u} Y_r$.

\begin{proof}[of Lemma~\ref{lem:A_pi_N_A}]
For $\phi\in \CB^1(\T^d)$, $\lambda\in[\eps,1]$, and $x\in \T^d$,
%as in the proof of~\cite[Prop.~3.21]{Chevyrev19YM},
we can write $\scal{A_i,\phi^\lambda_x} = \int_{\bar B}\bar L(\bar z) \bar \nu(\mrd \bar z)$
where $\bar B\subset \T^d$ is the $(d-1)$-dimensional hypercube of side-length $\asymp \lambda$ centred at $x$ and perpendicular to $e_i$ and $\bar \nu$ is the $(d-1)$-dimensional Lebesgue measure.
Here $\bar L(\bar z) = \int_{-1}^1 \bar Y^{\bar z}_t\mrd  \bar X^{\bar z}_t$ where
$\bar Y^{\bar z} \in \CC^{\var 1}([-1,1],\R)$
and $\bar X^{\bar z} \in \CC^{\var{1/\alpha}}([-1,1],\mfg)$
are defined by
\begin{equ}
\bar Y^{\bar z}_t = \phi^\lambda_x(\bar z+t\lambda e_i)\;,
\qquad
\bar X^{\bar z}_t = \int_{-1}^t \lambda A_i(z+t\lambda e_i)\mrd t\;,
\end{equ}
where the second term is well-defined since $A\in\Omega^1_\alpha$.
It is simple to verify that
\begin{equ}[eq:bar_YX_var_bounds]
|\bar Y^{\bar z}|_{\var 1} \lesssim \lambda^{-d}|\phi|_{\CC^1}\;,
\qquad
|\bar X^{\bar z}|_{\var{1/\alpha}}\lesssim \lambda^{\alpha}|A|_{\gr\alpha}\;.
\end{equ}

We can likewise write $\scal{\pi_N A_i,\phi^\lambda_x}_\e=\int_{B} L(z) \nu(\mrd z)$, where $B$ is the intersection of $\bar B$ with $\obonds_i$ (recall we identify bonds with their midpoints)
and
$\nu$ is $\eps^{d-1}$ times the counting measure on $B$.
Here $L(z)=\int_{-1}^1 Y^z_t\mrd X^z_t$ where
$Y^z \in \CD^{\var 1}([-1,1],\R)$
and $X^z \in \CD^{\var{1/\alpha}}([-1,1],\mfg)$
are defined by
\begin{equ}
Y_t^z = \phi^\lambda_x(z + \floor{t 2^N \lambda}\e_i)\;,
\qquad
X^{z}_t = \sum_{j=\floor{-2^N\lambda}}^{\floor{t 2^N\lambda}} \e (\pi_N A)_i(z+(j-1)\e_i)\;,
\end{equ}
and we have the bounds as~\eqref{eq:bar_YX_var_bounds} for $Y^z,X^z$.

Furthermore, we can take $B,\bar B$ such that
nearest neighbour approximation partitions $\bar B$ into subsets $\{P_z\}_{z\in B}$,
each with volume $\bar\nu (P_z) = \e^{d-1}$.
Then for every $z\in B$ and $\bar z\in P_z$,
\begin{equ}
|Y^z-\bar Y^{\bar z}|_\infty \lesssim \eps\lambda^{-d-1}|\phi|_{\CC^1}\;,\qquad
|X^z-\bar X^{\bar z}|_\infty \lesssim (\eps\lambda)^{\alpha/2}|A|_\alpha\;.
\end{equ}
Recall the interpolation estimate for $1\leq\gamma\leq\beta\leq \infty$ 
\begin{equ}[eq:var_interpolate]
|Z|_{\var\beta}
\leq |Z|_{\var\gamma}^{\gamma/\beta}|Z|^{1-\gamma/\beta}_{\var\infty}
\lesssim |Z|_{\var\gamma}^{\gamma/\beta}|Z|^{1-\gamma/\beta}_{\infty}\;.
\end{equ}
Applying~\eqref{eq:var_interpolate}
with $\gamma=1$,
we obtain, for any $\beta\in[1,\infty]$ and $z\in B,\bar z\in P_z$
\begin{equ}
|Y^z-\bar Y^{\bar z}|_{\var\beta}
\lesssim \lambda^{-d/\beta} \eps^{1-1/\beta}\lambda^{(-d-1)(1-1/\beta)}|\phi|_{\CC^1}
= \eps^{1-1/\beta}\lambda^{-d - (1-1/\beta)}|\phi|_{\CC^1}\;.
\end{equ}
Furthermore, applying~\eqref{eq:var_interpolate}
with $\gamma=1/\alpha$, we obtain for any $\beta\in[\frac1\alpha,\infty]$
\begin{equ}
|X^z-\bar X^{\bar z}|_{\var\beta}
\lesssim \lambda^{1/\beta} (\eps\lambda)^{\frac\alpha2(1-\frac1{\alpha\beta})}
|A|_\alpha
= \eps^{\frac\alpha2-\frac1{2\beta}}
\lambda^{\frac\alpha2 +\frac1{2\beta}}|A|_\alpha\;.
\end{equ}
Therefore, taking $1/\beta=1-\alpha+\kappa$ for $\kappa\in(0,\alpha)$, we obtain from~\eqref{eq:Young_int}
\begin{equ}
\Big|\int_{-1}^1(\bar Y^{\bar z}_t -Y^z_t) \mrd X^z_t\Big| \lesssim \e^{\alpha-\kappa}\lambda^{-d+\kappa}|\phi|_{\CC^1}|A|_\alpha\;,
\end{equ}
while taking $\beta\geq 1/\alpha$, we obtain from~\eqref{eq:Young_int} for any $\kappa = \frac1{2\beta}\in[0,\frac\alpha2]$
\begin{equ}
\Big|\int_0^1 \bar Y^{\bar z}_t \mrd(X^z_t-\bar X^{\bar z}_t)\Big|
\lesssim \e^{\frac\alpha2-\kappa}\lambda^{-d+\frac\alpha2+\kappa}|\phi|_{\CC^1}|A|_\alpha\;.
\end{equ}
Since $\frac\alpha2-\kappa$ is the smaller exponent of $\eps$ and $\lambda\geq \eps$,
we obtain for any $z\in B$ and $\bar z \in P_z$ that $|L(z) - \bar L(\bar z)|\lesssim \e^{\frac\alpha2-\kappa}\lambda^{-d+\frac\alpha2+\kappa}$.
Therefore, since $|B|\asymp (\lambda/\e)^{d-1}$, we obtain for $\eta \in [\frac\alpha2-1,\alpha-1]$ (playing the role of $\frac\alpha2+\kappa-1$)
\begin{equ}
|\scal{\pi_N A_i,\phi^\lambda_x}_\e-\scal{A_i,\phi^\lambda_x}|
\leq \sum_{z\in B} \int_{P_z} |L(z) -\bar L(\bar z)| \bar \nu(\mrd z)
\lesssim
% \e^{\frac\alpha2-\kappa}\lambda^{\frac\alpha2+\kappa-1}
%\\
%&\leq
\e^{\alpha-1-\eta} \lambda^\eta\;.
\end{equ}
\end{proof}

\subsection{Trees and regularity structure}
\label{sec:tr-reg}

In the rest of the article, we consider only $d=2$.
In this subsection, we define the regularity structure used to analyse the discrete equation~\eqref{e:Aeps}.
Our construction uses ideas from the continuum 2D YM in~\cite[Sections~5-6]{CCHS_2D} but with a number of significant differences.
We first recall briefly the important steps in the construction in \cite{CCHS_2D}.
In \cite[Sec.~6.1]{CCHS_2D}, we define the label set $\mfL^\YM=\mfL_-^\YM\sqcup\mfL_+$ with noise labels 
$\mfL^\YM_-=\{\mfl_1,\mfl_2\}$ and solution/kernel labels $\mfL_+=\{\mfa_1,\mfa_2\}$,
and consider trees associated with these labels.
We use a rule coming from the equation~\eqref{eq:SYM_moll} to build an ``abstract'' regularity structure $\ST$ consisting of trees.
The trees with negative degree are of the following form (the degree of a tree is given below it).\footnote{With
the formulation in \cite{BHZ19,CH16} one would also have $\<I'XXi_notriangle>$, $\<IXiI'XXi_notriangle>$, $\<IXXiI'Xi_notriangle>$ where crossed circles $\<XXi>$ denote 
$\mbX \Xi$. Here, since we essentially build the trees by hand as in \cite{Hairer14}, we drop these trees.}

\begin{equ}[e:tree-list-YM]
\begin{tabular}{ccccccc}
\toprule
$\<Xi>$
&
$\<IXiI'Xi_notriangle>$
&
 $\<I'Xi_notriangle>$
&
$\begin{array}{c}
\<IXi^3_notriangle>\ \<IXiI'[IXiI'Xi]_notriangle> \  \<I[IXiI'Xi]I'Xi_notriangle>
\end{array}$
&
$\begin{array}{c}
\<I'[IXiI'Xi]_notriangle>\  \<I[I'Xi]I'Xi_notriangle>\ \<IXiI'[I'Xi]_notriangle> \ \<IXi^2>
\end{array}$
&
$\begin{array}{c}
\<IXi> \; \<I'[I'Xi]_notriangle> \;\; {\bf X}\<I'Xi_notriangle>
\end{array}$
\\
\midrule
$-2-\kappa$
&
$-1-2\kappa$
&
$-1-\kappa$
&
$-3\kappa$
&
$-2\kappa$
&
$-\kappa$
\\
\bottomrule
\end{tabular}
\end{equ}
Here, the noises which will be denoted by $\Xi_{i}$ are circles $\<Xi>$, 
edges $\mcb{I}_{i}$ and $\partial_j\mcb{I}_{i}$ 
denoting abstract integration operators
are thin and thick grey lines respectively,
and $\kappa>0$ is sufficiently small and fixed, e.g. $\kappa=\frac1{100}$ suffices.
We follow standard conventions that joining the roots of two trees denotes multiplication which is commutative, i.e. $\<IXiI'Xi_notriangle>=\<I'XiIXi_notriangle>$.
Noises are treated as edges that are terminal, so we draw them as simply as $\<Xi>$ (and later as other symbols, e.g. $\<Xi1>$).
See~\cite[Sec.~1.2]{ESI_lectures} for the formal correspondence between graphical trees and symbols.
Recall that each symbol above actually corresponds to a family of trees, determined by assigning indices in a way that conforms to the rule that generates the trees. 
For instance, when we say that $\tau$ is of the form $\<I[IXiI'Xi]I'Xi_notriangle> $, then $\tau$ could be 
the following tree
\begin{equ}[e:exam-chicken-tree]
\mcb{I}_{1}\big(\mcb{I}_{2}(\Xi_{2}) \partial_{1}\mcb{I}_{2}(\Xi_{2})\big)
\partial_{1}\mcb{I}_{1}(\Xi_{1})\;,
\end{equ}
or other choices of the indices conforming to the structure of the quadratic nonlinear terms in our equation.

One then defines target space assignment  $W_\mft$ and corresponding space assignment $V_\mft$
(in the terminologies as in \cite[Sec.~5]{CCHS_2D}) by
\begin{equ}[e:space-assig]
W_\mft=\mfg\;,\quad (\mft\in\mfL^\YM)\;,
\qquad
V_\mfa=\R \;, \quad (\mfa\in\mfL_+)\;,
\qquad
V_\mfl=\mfg^*\;,  \quad (\mfl\in\mfL^\YM_-)\;.
\end{equ}
Recall that the ``concrete'' regularity structure $\CT$ is then given by a functor $\Func_V$ applied to $\ST$ (likewise for all other algebraic objects).
The concrete regularity structure is a direct sum $\bigoplus_\tau\CT[\tau]$ of vector spaces indexed by trees where $\CT[\tau]=\Func_V(\tau)$.
For example, $\Func_V( \<IXi^2> )\simeq (\mfg^*)^{\otimes_s 2}$, the space of symmetric $2$-tensors of $\mfg^*$.
We write $\CT_\alpha = \bigoplus_{|\tau|=\alpha}\CT[\tau]$ for the subspace of degree $\alpha$.

There are, however, several reasons that 
our regularity structure will actually be an enlargement of the above one used in \cite[Sec.~6.1]{CCHS_2D} for the continuum YM,
which we describe now.

First,
our regularity structure $\CT$ will have a splitting  
%\begin{equ}[e:T1T2_heuristic]
$\CT=\CT^{(1)}\oplus \CT^{(2)}$.
%\end{equ}
This splitting encodes the information on whether the trees will be realised as functions on horizontal or vertical bonds,
which will be important for defining models. In particular we will have two copies of abstract polynomials,
written as $\bone^{(1)},\X_1^{(1)},\X_2^{(1)} \in \CT^{(1)}$, and
$\bone^{(2)},\X_1^{(2)},\X_2^{(2)}\in \CT^{(2)}$.   (We will only need abstract polynomials of degree at most $1$ in this paper.)
We recall at this stage that each polynomial encodes a tree $\tau$ with a single vertex and that $\CT[\tau]\simeq \R$.
We will not distinguish between $\R \tau$ and $\CT[\tau]$, i.e. we treat $\bone^{(1)}, \X_1^{(1)}$, etc.
both as trees and as basis vectors for $\CT[\bone^{(1)}]$, $\CT[\X_1^{(1)}]$ and so on.

More importantly, we will need to include more trees due to the ``error'' terms
\begin{equ}[e:deg-1errors]
\e F(\e A)[A,\xi] \;\; \mbox{in} \;\; \eqref{e:Z-terms}
\qquad \mbox{and} \qquad
\tilde{R}_2(A) \;\; \mbox{in} \;\; \eqref{e:tildeR2}\;.
\end{equ}

\begin{remark}
The quadratic terms in \eqref{e:deg-1errors} are the most singular ``error'' terms, which have scaling dimension $-1$.
% and which do not appear amongst the ``main'' terms of the discrete SPDE  \eqref{e:Aeps}.
Since the choice of the discrete DeTurck term $-\mrd_U \mrd^* U$ in~\eqref{eq:discrete_hat_U} is not unique, 
%it is not clear to us
one may wonder if a different choice would simplify  the  terms in \eqref{e:deg-1errors}.
In any case, the general form the DeTurck term is $-\mrd_U (\mrd^* U + \omega)$ for a function $\omega\in\mfg^\Lambda$,
which, up to 2nd order, would amount to replacing $\nabla \cdot \A$ by $\nabla \cdot \A + \omega$ in~\eqref{eq:DeTurck_1}-\eqref{eq:DeTurck_2}.
Considering that~\eqref{e:Z-terms} and~\eqref{e:tildeR2} consist of terms of the form $\e A_i\xi_i$, $\e(\partial_j A_j)^2$, and $\e^2\d_j^2A_i\d_jA_j$ with $j\neq i$,
it seems unlikely that a significant simplification is possible.
\end{remark}
Moreover, our discrete SPDE has nonlocal nonlinearities. To encode this, we introduce shifting operators $\CS_{\be}$, indexed by $\be\in\CE_\times$ (recall the notation in Section~\ref{sec:more_notation}), that will be defined on sectors $V_{\be}$ of our regularity structure.
These maps will satisfy the following properties:
\begin{itemize}
\item
for all $i\neq j\in\{1,2\}$ and $\alpha\in\R$,
$\CS_{\be}$ maps $\CT_\alpha^{(i)}\cap V_{\be}$ to $\CT_\alpha^{(j)}$,
\item
$\CS_{\be}$ is norm preserving,

\item for each $i\neq j\in \{1,2\}$,
 $\CS_{\be}$ maps $\bone^{(i)}$  to $\bone^{(j)}$ and 
 $ {\bf X}^{(i)}_k$ to ${\bf X}^{(j)}_k$.
\end{itemize}
Finally, we have multiple types of derivatives appearing in our discrete equation which will be encoded by the set of abstract differentiation 
operators 
\begin{equ}\label{eq:Diff_def}
\Diff \eqdef  \{ \CD_1,\CD_2,\bar\CD_1,\bar\CD_2 , \CD_1^+ , \CD_2^+, \CD_1^-, \CD_2^-\}\;,
\end{equ}
which are defined on sectors of $\CT$ that we soon specify.
We impose that, for every $\CD\in \Diff\backslash \{\bar\CD_1,\bar\CD_2\}$ with domain $V_\CD$,
$\CD\colon V_\CD\cap \CT_\alpha^{(i)} \to \CT_{\alpha-1}^{(i)}$,
while for $\CD\in \{\bar\CD_1,\bar\CD_2\}$
we have
$\CD\colon V_\CD\cap \CT_\alpha^{(i)} \to \CT_{\alpha-1}^{(3-i)}$.
In particular, %unlike e.g. \cite{Hairer14} where one only has $\CD_i$ such that $\CD_i {\bf X}_j = \delta_{ij}\bone$, here 
\begin{equ}[e:DX=1]
\CD_i^\pm {\bf X}_j^{(k)}  = \delta_{ij} \bone^{(k)}\;,
\qquad
\CD_i {\bf X}_j^{(k)}  = \delta_{ij} \bone^{(k)}\;,
\qquad
\bar\CD_i {\bf X}_j^{(k)}  = \delta_{ij} \bone^{(3-k)}\;,
\end{equ}
for $i,j,k\in \{1,2\}$.  See Remark~\ref{rem:bar-partial} for motivation, as well as Remark~\ref{rem:PiDX}.

\begin{remark}\label{rem:splitting}
The picture one should keep in mind is that $\CT^{(1)}$ and $\CT^{(2)}$ are isomorphic but distinct regularity structures, with isomorphic structure groups,
that communicate only through the shifts $\CS_\be$ and through the ``switching'' derivatives $\bar\CD_i$.
In particular, products of the form $\tau_1\tau_2$ with $\tau_1\in\CT^{(1)}$, $\tau_2\in\CT^{(2)}$ never appear,
and the polynomials $\bone^{(1)}$ and $\bone^{(2)}$ are units for multiplication in $\CT^{(1)}$ and $\CT^{(2)}$ respectively (i.e. $\bone^{(i)}\tau=\tau$ for $\tau\in\CT^{(i)}$).
\end{remark}
We next describe the full regularity structure $\CT$ for our discrete equation.
We use a recursive construction akin to~\cite[Sec.~8]{Hairer14} (though we could have
equivalently used the
formalism of rules~\cite{BHZ19}).
For clarity, we proceed in three steps
(it is possible to combine these steps, but we
separate them for clarity).

\subsubsection{Step 1: continuum YM trees}
\label{subsubsec:step1}

We recall the usual noise labels
%\begin{equ}\label{eq:mfL_YM}
$\mfL^{\YM}_-=\{\mfl_1,\mfl_2\}$
%\end{equ}
corresponding to the continuum 2D YM (we will later add more noise labels to form $\mfL_- = \mfL^\YM_- \sqcup\mfL^\rem_-$),
as well as the solution/kernel labels $\mfL_+=\{\mfa_1,\mfa_2\}$. 
%We define degrees as usual
%\begin{equ}
%\deg(\mfl_i) = -2-\kappa \;,
%\quad
%\deg(\mfa_i) = 2\;.
%\end{equ}
In the first step towards defining $\CT$, we describe a set of trees $\mfT^\YM=\mfT_1^\YM\sqcup \mfT_2^\YM$\label{page:mfT_YM_-}
which are similar to the trees for the continuum 2D YM.
For motivation behind the construction below, we recommend the reader reviews the nonlinearities in Proposition~\ref{prop:rescaled-equ} (see also the `abstract' fixed-point problem~\eqref{e:fix-pt-Ae} below).

\smallskip
\noindent
\textbf{Base Case (Step 1).}
For each $i \in \{1,2\}$, we let
\begin{equ}[e:rule-main1]
\bar{\mfT}_i \eqdef \{\bone^{(i)}, {\bf X}^{(i)}_1, {\bf X}^{(i)}_2 \} \subset \mfT_i^\YM \;,   \qquad
 \Xi_i \eqdef \Xi_{\mfl_i} \in \mfT_i^\YM \;.
\end{equ}

\smallskip
\noindent
\textbf{Integration (Step 1).}
Let $\mcb{I}_1$, $\mcb{I}_2$ be abstract integration operators.
For all $i\neq j\in\{1,2\}$, $\be\in\CE_\times$ and $\tau \in \mfT^\YM_i$ with $|\tau|<-1$, we enforce that,
\begin{equ}[e:rule-main3]
\mcb{I}_i (\tau) \in \mfT^\YM_i \;,
\qquad 
\CS_\be \mcb{I}_i (\tau) \in \mfT^\YM_j \;.
\end{equ}

\smallskip
\noindent
\textbf{Nonlinearity (Step 1).}
For all
\begin{itemize}
\item $i \neq j\in \{1,2\}$, $\be,\be'\in\CE_\times$, and $\bw\in\{\Northwest,\Southwest\}$,

\item $\tau_i,\sigma_i\in \mfT^\YM_i$ of the form $\tau_i=\mcb{I}_i(\bar\tau_i),\sigma_i=\mcb{I}_i(\bar\sigma_i)$ or $\tau_i,\sigma_i \in \bar{\mfT}_i$, and

\item $\tau_j,\sigma_j,\omega_j\in \mfT^\YM_j$ of the form 
$\tau_j=\mcb{I}_j(\bar\tau_j),\sigma_j=\mcb{I}_j(\bar\sigma_j),\omega_j=\mcb{I}_j(\bar\omega_j)$
where $\bar\omega_j\neq \Xi_j$, or $\tau_j,\sigma_j,\omega_j \in \bar{\mfT}_j$,
\end{itemize}
we enforce that every
\begin{equs}[e:rule-main2]
\sigma \in
&\{\;
(\CS_{\be} \tau_j )(\CD_j \tau_i) \;, \;
(\CS_{\be} \tau_j )(\bar\CD_i \CI_j\Xi_j) \;,\;
(\CS_{\be} \tau_j )(\CD_i^+ \CS_\bw \omega_j) \;,\;
\\
&\quad
\tau_i \CD_i \sigma_i\;,\;
(\CS_{\be} \tau_j )(\CS_{\be'} \sigma_j )\tau_i\;\}
\end{equs}
belongs to $\mfT_i^\YM$ provided that $|\sigma| \leq 0$.

Above, we assign degrees using the rules
\begin{equ}\label{eq:degree_step_1}
|\CI(\tau)| = |\tau|+2\;,\quad |\CD\tau|=|\tau|-1\;, \quad
|\CS_\be\tau|=|\tau|\;,
\quad |\tau_1\tau_2| = |\tau_1|+|\tau_2|\;,
\end{equ}
along with the base cases $|\Xi|=-2-\kappa$, $|\bone|=0$, and $|\X|=1$.

\begin{remark}
The terms in~\eqref{e:rule-main2} arise from the main nonlinearities in~\eqref{e:Aeps}.
Since $\tau_i,\sigma_i,\tau_j,\sigma_j$ can be any polynomial and $\CS_\be\colon\bone^{(j)} \mapsto \bone^{(i)}, \X^{(j)}\mapsto \X^{(i)}$, the right-hand side of~\eqref{e:rule-main2} includes $\CD_i \tau_i$,
$\CD_j \tau_i$,
$\bar\CD_i \CI_j\Xi_j$, $(\CS_\be\tau_j)\tau_i$, $\X^{(i)}\CD_j\tau_i$, and so on.
\end{remark}

\begin{remark}\label{rem:barCD_term1}
Since $\bar\d_iA_j = \frac12(\d_i^+A_j(e^{\northwest})+\d_i^+A_j(e^{\southwest}))$,
the average of the 3rd term in $\{\ldots\}$ in~\eqref{e:rule-main2} over $\bw\in\{\Northwest,\Southwest\}$
is \textit{formally} $(\CS_{\be}\tau_j) (\bar\CD_i \CI_j\omega_j)$, which would be a more natural term to include (cf. the 2nd term);
our current formulation of the 3rd term,
however, assists in deriving the renormalised equation in Section~\ref{subsec:renorm_eq} (see in particular~\eqref{e:quad-exp-3}-\eqref{eq:non_local_A_j}).
\end{remark}
\textit{Integration (Step 1)} and \textit{Nonlinearity (Step 1)}, i.e.~\eqref{e:rule-main3}-\eqref{e:rule-main2},
are implemented recursively with~\eqref{e:rule-main1} as the base case.
One can check that the recursion terminates after applying~\eqref{e:rule-main3}-\eqref{e:rule-main2} two times and that
every symbol in $\mfT^\YM_i$ corresponds to either a polynomial, a tree of the form given in~\eqref{e:tree-list-YM}, or a tree of the form
\begin{equ}\label{eq:pos_trees-YM}
\<I(cherry)>\;,\qquad \<IDPsi>\;,
\end{equ}
with degrees $1-2\kappa$, $1-\kappa$ respectively.
We use the same graphic notation as  in~\eqref{e:tree-list-YM} to describe the form of the trees (in particular the  graphic notation 
does not indicate insertions of shifting operator or choice of differentiation operator).

The set of trees obtained this way 
is strictly larger than that for the continuum YM because we have trees of the form  $\CS_\be\mcb{I}(\cdots)$, four types of directional derivatives $\CD_i,\bar \CD_i,\CD_i^{\pm}$,
and the abstract monomials are doubled.
But otherwise,  the trees obtained here are just the trees for continuum YM in~\cite[Sec.~6]{CCHS_2D} with proper insertions of shifting operators
and choice of differentiation operators.
For instance,
\eqref{e:exam-chicken-tree} may now become
$
\mcb{I}_{1}\big(\CS_\be \mcb{I}_{2}(\Xi_{2}) \bar\CD_{1}\mcb{I}_{2}(\Xi_{2})\big)
\CD_{1}\mcb{I}_{1}(\Xi_{1})
$.
We define the negative trees $\mfT^\YM_- = \{\tau\in\mfT^\YM\,:\,|\tau|<0\}$.

\subsubsection{Step 2: remainder trees}
\label{subsubsec:step2}

In our second step towards defining $\CT$, we add more trees which will allow us to handle the ``remainder terms'' in \eqref{e:deg-1errors}.
We introduce additional noise labels
\begin{equ}\label{eq:mfL_rem}
\mfL^{\rem}_-=\{\bar\mfl_1,\bar\mfl_2\}\sqcup \{ \mfl^{=}_i,\mfl^{+}_i,\mfl^{-}_i\}_{i=1}^2\;.
\end{equ}
In what follows, we assign degree to symbols using using~\eqref{eq:degree_step_1} together with $|\Xi_{\bar\mfl_i}| = -1-2\kappa$
and
$|\Xi_{\mfl^{=,\pm}_i}| = -1-3\kappa$.

The label $\bar\mfl_i$
corresponds to the noise $\e^{1-\kappa} \xi_i$.
The motivation is that we can write $ \eps [ A_i,\xi_i]  = \eps^\kappa [A_i, \eps^{1-\kappa}\xi_i] $
in the term $\eps F(\e A)[A,\xi]$ in~\eqref{e:Aeps}.
The labels $\mfl^{=,\pm}_{1,2}$
are motivated by the terms in $\tilde{R}_2$ in~\eqref{e:tildeR2}, the first of which we can write as
\begin{equ}\label{eq:R_2_other_form}
- \frac{\eps^2}{2} [\partial_j^2 A_i(e), \bar\partial_j A_j(e)]
=
- \frac{\eps}{2} [(\partial_j^+  +\partial_j^-)A_i(e), \bar\partial_j A_j(e)] \;,
\end{equ}
so $\tilde{R}_2$ has a form $\e \partial A \partial A$.
With $A=\Psi+v$ where $\Psi$ is SHE,
 we can then write
\begin{equ}\label{eq:edAdA}
\e \partial A \partial A = \e^\kappa \partial (\e^{\frac{1-\kappa}{2}}\Psi) \partial (\e^{\frac{1-\kappa}{2}}\Psi)
+ \e^\kappa \partial (\e^{1-\kappa}\Psi) \partial v
+ \e^\kappa \partial v \partial (\e^{1-\kappa}\Psi) 
+ \e \partial v\partial v\;.
\end{equ}
The 6 noises $\mfl^{=,\pm}_{1,2}$ arise from  $\d(\e^{\frac{1-\kappa}{2}}\Psi) \partial (\e^{\frac{1-\kappa}{2}}\Psi)$
with the different choices for
$\d$ and $\Psi$.
(We will treat the last term $ \e \partial v\partial v$ in the ``classical'' way and do not lift it to the space of modelled distributions.
See also~\eqref{e:fix-pt-Ae}-\eqref{e:def-tildeCR2} below for this decomposition at the level of modelled distributions.)

With these motivations, we now describe the full set of trees $\mfT = \mfT_1\sqcup \mfT_2$.

\smallskip
\noindent
\textbf{Base case (Step 2).}
We enforce $\mfT^\YM_i\subset \mfT_i$, and
for each $i\neq j\in\{1,2\}$
\begin{equ}\label{eq:base_case_2}
\bar\Xi_i \eqdef \Xi_{\bar\mfl_i}\;,
\;\;
\big(\CD_j^\pm \mcb{I}\tilde\Xi_i \big)  \big(\bar{\CD}_j \mcb{I}\tilde\Xi_j\big) \eqdef \Xi_{\mfl^\pm_i} \;,
\;\;
\big(\CD_j^+ \CS_{\southeast} \mcb{I}\tilde\Xi_j \big) \big( \CD_j^+ \CS_{\southwest} \mcb{I}\tilde\Xi_j \big) \eqdef \Xi_{\mfl^=_i}
\end{equ}
all belong to $\mfT_i$.
Here $\tilde\Xi_{i}$, which formally models $\eps^{\frac12-\frac\kappa2}\xi_{i}$, is purely notational and we do \textit{not} have a tree $\tilde\Xi_{i}$ in $\mfT_i$.
When drawing trees, we will denote $\bar\Xi_i$ by $\<Xi1>$ and $\tilde\Xi_i$ by $\tikz \node[xi2] at (0,0) {};$,
so that $\Xi_{\mfl_{1,2}^{=,\pm}}$ are all drawn as $\<R2-1new>$ (understood as a single terminal edge).

\smallskip
\noindent
\textbf{Integration (Step 2).}
For each $i\neq j\in\{1,2\}$, if $\tau\in\mfT_i$ with $|\tau|<-1$, we enforce that~\eqref{e:rule-main3} holds with $\mfT^\YM$ replaced by $\mfT$.

\smallskip
\noindent
\textbf{Nonlinearity (Step 2).}
To encode the main nonlinearities,
we enforce that \textit{Nonlinearity (Step 1)} holds with $\mfT^\YM$ replaced by $\mfT$.

Next, to encode the term $\tilde{R}_2$ in~\eqref{e:tildeR2},
for each $i\neq j\in\{1,2\}$ and $\tau_i\in\mfT_i$, $\tau_j\in\mfT_j$ of the form $\tau_i=\CI_i(\bar\tau_i)$ and $\tau_j=\CI_j(\bar\tau_j)$
we enforce that every
\begin{equ}[e:rule-more3]
\sigma \in \{\;\bar\CD_j \CI_j\Xi_j \;,\;
\CD_j^+ \CS_{\southwest} \tau_j \;, \;
 \CD_j^+ \CS_{\southeast} \tau_j \;,\;
\CD_j^\pm \tau_i\;\}
\end{equ}
belongs to $\mfT_i$ provided that $|\sigma|\leq 0$.
Furthermore, in view of \eqref{e:tildeR2}, \eqref{eq:R_2_other_form}, and the 2nd and 3rd terms on the right-hand side of~\eqref{eq:edAdA},
for all
\begin{itemize}
\item $i\neq j \in \{1,2\}$ and $\bs\in\{\Southeast,\Southwest\}$,
\item $\tau_i=\mcb{I}_i (\bar\tau_i) \in \mfT_i \backslash \{\mcb{I}_i (\Xi_i)\}$  or $\tau_i \in \bar{\mfT}_i$,
\item $\tau_j=\mcb{I}_j (\bar\tau_j)\in \mfT_j  \backslash \{\mcb{I}_j (\Xi_j)\}$   or $\tau_j \in \bar{\mfT}_j$, 
\end{itemize}
we enforce that every
\begin{equs}[e:rule-more2]
\sigma \in \big\{\;
&(\CD_j^\pm \mcb{I}_i{\bar\Xi_i}) (\CD_j^+ \CS_{\bs} \tau_j)\;,
\;\;
(\CD_j^\pm \tau_i )(\bar\CD_j  \mcb{I}_j{\bar\Xi_j})\;,
\\
&( \CD_j^+ \CS_{\southeast} \mcb{I}_j{\bar\Xi_j}) ( \CD_j^+ \CS_{\southwest} \tau_j) \;,\;\;
(\CD_j^+ \CS_{\southeast} \tau_j  ) ( \CD_j^+ \CS_{\southwest} \mcb{I}_j{\bar\Xi_j})\;
\big\}
\end{equs}
belongs to $\mfT_i$ provided that $|\sigma|\leq0$.

%\begin{remark}\label{rem:dVdV_trees}
%The terms in~\eqref{e:rule-more2} arise from the 2nd and 3rd terms on the right-hand side of~\eqref{eq:edAdA} (recall also~\eqref{e:tildeR2}).
%Note that $\tau=\CI(\Xi)$ is excluded from~\eqref{e:rule-more2} because it corresponds to the 1st term on the right-hand side in~\eqref{eq:edAdA}
%and is taken care of by~\eqref{eq:base_case_2}.
%Likewise, we do not include terms like
%$(\CD_j^\pm \tau_i )(\bar\CD_j  \tau_j)$ or $(\CD_j^+ \CS_{\southeast} \tau_j  ) (\CD_j^+ \CS_{\southwest} \sigma_j)$
%as these would correspond to the final term $\e\d v \d v$ in~\eqref{eq:edAdA}, which we instead deal with classically.
%\end{remark}

\begin{remark}\label{rem:barCD_term2}
Like in Remark~\ref{rem:barCD_term1},
the average of the 1st term in $\{\ldots\}$ in~\eqref{e:rule-more2} over $\bs\in\{\Southeast,\Southwest\}$
is \textit{formally} $(\CD_j^\pm \mcb{I}_i{\bar\Xi_i}) (\bar\CD_j \tau_j)$, which may appear more natural;
the current choice, however, helps in deriving the renormalised equation.
\end{remark}
To encode the nonlinearity $\e F(\e A)[A,\xi]$,
for each $i \in \{1,2\}$,
$\chi \in \bar\mfT_i$, and $\tau,\omega \in \mfT_i$ of the form
$\tau=\CI_i(\bar\tau),\omega=\CI_i(\bar\omega)\in\mfT_i$  or
$\tau,\omega\in\bar\mfT_i$
we enforce that every
\begin{equ}\label{eq:F_trees}
\sigma \in \{\;\tau\bar \Xi_i\;,
\;
\tau\chi\bar\Xi_i\;,
\;
\tau  \CI_i (\omega\bar \Xi_i) \bar \Xi_i
\;
\}
\end{equ}
belongs to $\mfT_i$ provided that $|\sigma|\leq 0$.

\begin{remark}\label{rem:F_trees}
The final definition \eqref{eq:F_trees} requires some explanation.
Recalling the expansion~\eqref{e:Z-terms},
the term $\tau\bar \Xi_i$ is due to the product $\e^\kappa[A,\e^{1-\kappa}\xi]$.
The next order term, $\e^{1-\kappa}A^2\e^{1-\kappa}\xi$, should naively
entail a general term of the form $\CI_i(\bar\tau)\CI_i (\omega)\bar \Xi_i$.
However, by defining  a `multiplication by $\eps$' map on our regularity structure (see Section~\ref{subsubsec:mult_eps_models})
we can use the remaining $\e^{1-\kappa}$ to increase the degree of $\omega$ by $1-\kappa$
\textit{unless} $\omega$ has $\bar\Xi_i$ at the root --
this is because the degree $|\bar\Xi| = -1-2\kappa$ is limited by the temporal regularity of $\e^{1-\kappa}\xi$ (white-in-time)
and further multiplication by $\eps$ will not increase this small-scale regularity (see also Remark~\ref{rem:cE_motivation}).
The higher powers $O(\e^3A^3\xi)$ in~\eqref{e:Z-terms} do not contribute any new trees for the same reason.
\end{remark}
We implement recursively \textit{Integration (Step 2)}
and \textit{Nonlinearity (Step 2)},
i.e. \eqref{e:rule-main3}+\eqref{e:rule-main2}+\eqref{e:rule-more3}+\eqref{e:rule-more2}+\eqref{eq:F_trees}  
together with \textit{Base case (Step 2)}.
One can check that the recursion terminates after applying \textit{Integration (Step 2)}
and \textit{Nonlinearity (Step 2)}
two times each.
The new trees with degrees $-1-$ (i.e. just below $-1$) are of the following form (degrees are in parentheses):\footnote{The noises in  $\<R2-1new>$  are $\tikz \node[xi2] at (0,0) {};$, {\it not} $\<Xi1>$ !
In other words in our graphic notation, rotating lines and thick lines by $\frac{\pi}{4}$ does not change their meaning, but we never rotate the notation for noises.}
\begin{equ}[e:trees-deg-1-rem]
\<Xi1>\quad(-1-2\kappa)\;,\qquad
\<PsiXi1>\quad (-1-3\kappa)\;,\qquad
\<R2-1new> \quad (-1-3\kappa)\;.
\end{equ}
Here the second tree arises from $\e A \xi$, i.e.~\eqref{eq:F_trees}, and the third one 
 arises from $\e \partial A \partial A$, i.e.~\eqref{eq:base_case_2}.
The new trees with degrees $0-$ are of the following form:
\begin{equ}[e:trees-deg-0-rem]
\begin{tabular}{cc}
\toprule
\hypertarget{(1-2)}{(1-2)}:
&
$ \<I'XiI'[IXiI'Xi]less>$
\qquad
$\<I'XiI'[IXiI'Xi]>$
\qquad
$\<I'(R2-1)new> $
\qquad
$\<IXiI'(R2-1)new>$ 
\qquad
$ \<I'XiI(R2-1)new>$
\\
\midrule
\hypertarget{(1-3)}{(1-3)}:
&
 $\<cherryY>$
\quad
$\<cherryYless>$
\quad
$\<I'Xibar>$
\quad
$\<IXiI'Xibar>$ 
\quad 
$\<I'XiIXibar>$ 
\quad 
$\<supercherry1less>$
\quad
$\<supercherry1>$ 
\quad  
$\<supercherry2>$
\quad
$ \X^{(i)} \bar\Xi_i$
\quad
$\X^{(i)} \<PsiXi1>$
\\
\midrule
\hypertarget{(2-2)}{(2-2)}:
&
$ \<I'Xi1I'(R2-1)new>$
\quad 
\hypertarget{(2-3)}{(2-3)}:
\quad
$\<I(R2-1)Xi1new>$
\quad
$\<cherry231>$
\quad
$\<cherry232>$
\quad
\hypertarget{(3-3)}{(3-3)}:
\quad
$\<I(PsiXi1)Xi1>$
\quad
$\<I(PsiXi1)Xi1less>$
\quad
$\<PsiPsibarXibar>$
\quad
$\<r_z_large>$
\\
\bottomrule
\end{tabular}
\end{equ}
In plain words, these trees are obtained by substitutions among three types of terms:
\begin{equ}
\text{(1)}\;\; A\partial A\;,\quad \text{(2)}\;\; \e \partial A \partial A\;,\quad 
\text{(3)}\;\;
\e^\kappa F(\e A) [A,\e^{1-\kappa} \xi]\;,
\end{equ}
subject to the conditions in~\eqref{e:rule-more2} and~\eqref{eq:F_trees} 
%(see also Remarks~\ref{rem:dVdV_trees} and~\ref{rem:F_trees}).
(see also Remark~\ref{rem:F_trees}).
We classified them into 5 groups, with group labels indicating how we obtained the trees in the group,
for instance trees in group \hyperlink{(1-3)}{(1-3)} are obtained by 
substituting between (1) $A\partial A$ and (3) $\e^\kappa F(\e A) [A,\e^{1-\kappa} \xi]$,
and group \hyperlink{(3-3)}{(3-3)} are obtained by 
substituting  (3) $\e^\kappa F(\e A) [A,\e^{1-\kappa} \xi]$ into itself.
%Remark that some trees appearing from naively substituting the non-linearities into each other, like $\<r_z_avoid1>$, do not appear
%due to the particularity of~\eqref{eq:F_trees}.

%(see Remarks~\ref{rem:F_trees} and~\ref{rem:cE_motivation}).

Remark also that, due to \eqref{e:rule-more3},
 we should in principle also include $\<I'[IXiI'Xi]_notriangle>$, $\<I'[I'Xi]_notriangle>$ in group  \hyperlink{(1-2)}{(1-2)},
 but these graphical trees are already in \eqref{e:tree-list-YM},
except that the differentiation at the bottom of the trees  should be the differentiations appearing in \eqref{e:rule-more3},
so we do not duplicately list them in  \hyperlink{(1-2)}{(1-2)}.
Similarly, $\<I'(R2-1)new> $ in  \hyperlink{(1-2)}{(1-2)}
should also be  in \hyperlink{(2-2)}{(2-2)},
 and $\<I'Xibar>$ and $\<supercherry1less>$ in \hyperlink{(1-3)}{(1-3)} should also be in \hyperlink{(2-3)}{(2-3)}.

\subsubsection{Step 3: two more trees}
\label{subsubsec:Step_3}

Finally, we add two more types of trees to $\mfT_i$ for reasons apart from the non-linearity.

The first is due to (negative) renormalisation.
There is one tree in $\mfT_i$ of the form $\<r_z_large>$, namely $\Psi_i\bar\Xi_i\CI_i[\Psi_i\bar\Xi_i]$ where we write $\Psi_i\eqdef\CI_i\Xi_i\in\mfT_i^\YM$.
These trees have a sub-divergence $\bar\Xi_i\CI_i[\bar\Xi_i]$ of the form $\<I(PsiXi1)Xi1less>$ that we will be forced to renormalise (see Section~\ref{subsec:renorm_group}) and which produces a contracted tree $\Psi_i\Psi_i=\Psi_i^2$ of the form $\<IXi^2>$, which we now add to $\mfT_i$ for $i\in\{1,2\}$.
(We already added trees of the form $\<IXi^2>$ to $\mfT_i$ due to~\eqref{e:rule-main2}, but these are of the form $\CS_\be\Psi_j\CS_{\be'}\Psi_j$ or $\CS_\be\Psi_j \Psi_i$ for $j\neq i$.)

The second type of tree we add is of the form $\<IXiI'Xibar>$
which we obtain by replacing the factor $\Xi$ `on the right' in all trees of the form $\<IXiI'Xi_notriangle>$ in $\mfT^\YM$ by $\bar\Xi$.
Specifically,
recalling \eqref{e:rule-main2}, the new trees appearing in $\mfT_i$
are $\CS_{\be} \Psi_j \CD_j \bar\Psi_i$,
$\CS_{\be} \Psi_j  \bar\CD_i \bar \Psi_j$,
and
$\Psi_i \CD_i \bar \Psi_i$,
where $\bar\Psi_i = \CI_i\bar\Xi_i$,
$i\neq j$, and $\be \in\CE_\times$.
Note $\<IXiI'Xibar>$ already appears in \eqref{e:trees-deg-0-rem} but these new trees did not arise in Step 2.
We add these trees to make a multiplication by $\e$ map in Section \ref{subsubsec:mult_eps_models} well-defined (see also Remark \ref{rem:F_trees}).

This concludes our definition of $\mfT = \mfT_1\sqcup \mfT_2$.
We denote $\mfT_- = \{\tau\in\mfT\,:\,|\tau|<0\}$,
$\mfT^\ren = \{\Psi_i^2\in\mfT\,:\,i\in\{1,2\}\}\subset\mfT_-$,
and
$\mfT^\rem = \mfT\setminus(\mfT^\YM\sqcup\mfT^\ren)$ and $\mfT^\rem_- = \mfT_- \setminus (\mfT^\YM_-\sqcup \mfT^\ren)$.
Note that the trees of the form $\<IXiI'Xibar>$ added in Step 3 are in $\mfT^\rem$
and we will treat them as belonging to group \hyperlink{(1-3)}{(1-3)}.
The positive trees in $\mfT^\rem$ are of the form (recall that we enforce $|\sigma|\leq 0$ in \eqref{e:rule-more2})
\begin{equ}\label{eq:pos_trees_rem}
 \<IXibar>\;,\qquad \<I(R2-1new)>\;, \qquad \<IPsiXi1>\;.
\end{equ}
%and the positive non-planted trees in $\mfT^\rem$ are of the form

\subsubsection{Regularity structure}
\label{subsubsec:RS}

To construct our regularity structure, we denote $\mfL_- = \mfL^\YM_-\sqcup \mfL^\rem_-$,
%with $\mfL^\YM_-$ and $\mfL^\rem_-$\label{page:mfT_rem_-}
%defined in~\eqref{eq:mfL_YM} and~\eqref{eq:mfL_rem},
with $\mfL^\rem_-$\label{page:mfT_rem_-}
defined in~\eqref{eq:mfL_rem},
and extend the space assignment~\eqref{e:space-assig}
by
$
W_{\bar\mfl_i} = \mfg$ and  $W_{\mfl^{=,\pm}_i} = \mfg\otimes \mfg
$
so that, in the notation of~\cite[Sec.~5.6]{CCHS_2D},
$
V_{\bar\mfl_i} = \mfg^*$ and  $V_{\mfl^{=,\pm}_i} = \mfg^* \otimes \mfg^*
$.

The regularity structure is then constructed by applying the functor $\Func_V$, i.e. 
$\CT \eqdef \bigoplus_{\tau\in\mfT} \CT[\tau]$
where
$
\CT[\tau]\eqdef \Func_V[\tau]
$.
More precisely, we consider symmetric sets built from finite labelled rooted trees exactly as in~\cite[Sec.~5.4]{CCHS_2D} with the addition that kernel edges may have an additional decoration $\CS_{\be}$ and several types of derivatives from $\Diff$ (recall~\eqref{eq:Diff_def}),
and that isomorphisms between trees are required to preserve these additional decorations (in addition to the usual node and edge decorations).
For example,
$
\CT[ \CD_j^+ \CS_{\southeast} \CI_j(\sigma_j)\CD^+_j\CS_{\southwest}\CI_j\bar\Xi_j] \simeq \mfg^*\otimes\mfg^* \otimes \mfg^*
$
where $\sigma_j = \Xi_{\mfl_j^=}= \big(\CD_i^+ \CS_{\southeast} \mcb{I}\tilde\Xi_i \big) \big( \CD_i^+ \CS_{\southwest} \mcb{I}\tilde\Xi_i \big) \in \mfT_j$.
(The reader may take these outputs of \cite[Sec.~5]{CCHS_2D} 
for granted since the precise definition of $\CT[\tau]$ is seldom used in this article.)
For $L\in\{\ym,\rem,\ren\}$,
we define
\begin{equ}
\CT^{L} = \bigoplus_{\tau\in\mfT^L}\CT[\tau]\;,\quad
\CT^{L}_- = \bigoplus_{\tau\in\mfT^L_-}\CT[\tau]\;,\quad
\CT_- = \CT^{\YM}_-\oplus \CT^{\rem}_-\oplus \CT^{\ren}_-\;.
\end{equ}
The splitting discussed in the beginning of Sec.~\ref{sec:tr-reg} %\eqref{e:T1T2_heuristic} 
is given by
\begin{equ}[e:T1T2]
\CT=\CT^{(1)}\oplus\CT^{(2)}\;,\qquad
\CT^{(i)} = \bigoplus_{\tau\in\mfT_i} \CT[\tau]\;.
\end{equ}
Recalling $\bar\mfT_i$ from~\eqref{e:rule-main1},
the polynomial regularity structure is denoted by
$
\bar \CT = \bar\CT^{(1)}\oplus \bar\CT^{(2)}
$
where
$\bar\CT^{(i)}= \bigoplus_{\tau\in\bar\mfT_i}\CT[\tau] \simeq \R^3\;.
$
For $H \in \{\CI_i,\CD,\CS_\be\}$, where $i\in\{1,2\}$, $\CD\in\Diff$ and $\be\in\CE_\times$,
we view $H$ as a map $D_H\to\mfT$ defined on the subset $D_H\eqdef\{\tau\in\mfT\,:\,H\tau\in\mfT\}$.
Then $H$ lifts to an operator, denoted by the same symbol, $H\colon V_H \to \CT$
via the functor $\Func_V$ as in~\cite[Sec.~5.2.1]{CCHS_2D},
where $V_H\eqdef\bigoplus_{\tau\in D_H}\CT[\tau]$.
We use the shorthand $V_\be = V_{\CS_\be}$.

The structure group $\CG$ is determined by the usual Connes--Kreimer-type procedure as in~\cite[Sec.~8.1]{Hairer14} or~\cite[Sec.~2]{ESI_lectures}
upon treating $\CD \CI_i$, $\CS_\be \CI_i$, and $\CD\CS_\be\CI_i$ for $\CD\in\Diff$, $\be\in\CE_\times$, and $i\in\{1,2\}$ as normal kernel edge decorations
(recall also Remark~\ref{rem:splitting}).

For example, denoting $\Psi_i=\CI_i\Xi_i$, 
for $i\neq j\in\{1,2\}$ and $\tau\in\CT[\CS_\be\Psi_i] \subset\CT^{(j)}$ we have $\Gamma\tau=\tau$ for all $\Gamma\in \CG$,
while for $\sigma =  \CS_\be\CI_i[(\CS_{\be'}\Psi_j)(\CD_j\Psi_i)]\bar\CD_j\Psi_i \in\mfT_j$
(which is of the form $\<I[IXiI'Xi]I'Xi_notriangle>$), and
\begin{equ}
\CT^{(j)}\ni\tau = \tau_1\otimes\tau_2\otimes \tau_3\in \CT[\sigma] \simeq (\mfg^*)^{\otimes 3}
\end{equ}
we have
$
\Gamma \sigma = \sigma + h \otimes \tau_3 \in \CT[\sigma]\oplus \CT[\bar\CD_j\Psi_i]
$,
where $h\in\CT[\bone^{(j)}] \simeq \R$ and $\tau_3 \in \CT[\bar\CD_j\Psi_i ]\subset \CT^{(j)}$.

One can readily check that the domain $V_H$ of $H\in \{\CI_i,\CS_\be,\CD\}$ with $\CD\in\Diff$ are all sectors (i.e. subspace closed under the action of $\CG$).
Furthermore, for all $\Gamma\in\CG$, $\CD\in\Diff$, and $\tau\in V_{\CD}$, we have 
\begin{equ}\label{eq:Gamma_Diff_commute}
\Gamma\CD\tau = \CD\Gamma\tau\;.
\end{equ}
The following is an elementary but useful lemma.

\begin{lemma}\label{lem:Gamma_simple}
Let $\tau\in\mfT\setminus\bar\mfT$ with $|\tau|>0$.
Then $
\Gamma\sigma=\sigma + \CT[\bone]
$ for all
$\Gamma\in\CG$ and $\sigma\in\CT[\tau]$.
\end{lemma}

\begin{proof}
All non-polynomial positive trees are of the form~\eqref{eq:pos_trees-YM} or~\eqref{eq:pos_trees_rem},
and the proof follows from the fact that $\Gamma \bar\tau = \bar\tau$ for all $\bar\tau\in\CT_\alpha$ with $\alpha<-1$.
\end{proof}

\subsection{Discrete framework}
\label{sec:Discrete framework}

In  \cite{EH19}, given a regularity structure, a discretization is a collection of data 
\begin{equ}[e:framework]
(\CX_\eps,\|\cdot \|_{\alpha;\K_\eps;z;\eps}, \$\cdot\$_{\gamma;\eta;\K;\eps})
\end{equ}
which we now specify.
We set 
$\CX_\eps^{(i)}= \CC^{-1+\kappa}(\R,\R^{\obonds_i})$ for $i\in \{1,2\}$.
%where we recall that $\obonds_i$ is the set of positively oriented bonds in the $i$-th direction.
%We view $\obonds_i$ as a copy of $\T^2_\e$.
Write $\CX_\e\eqdef \CX_\eps^{(1)}\oplus \CX_\eps^{(2)}$.
As in \cite[Sec.~2]{EH19} we view $\CX_\eps^{(i)}$ as a subspace 
of the space of distributions,
namely we define an injection map $\iota_\eps \colon\CX_\eps^{(i)} \to \CD'(\R\times\T^2)$ by setting
\begin{equ}
\label{eq:iota}
(\iota_\eps f)(\varphi)
= \eps^2\sum_{e\in\obonds_i}\int_{\R} f(t,e)\varphi(t, e)\mrd t\;,
\qquad
\forall \phi\in\CC^\infty_c(\R\times \T^2)
\end{equ}
 in distributional sense.
We will often omit $\iota_\eps$
and simply write $f(\varphi)$. % for $(\iota_\eps f)(\varphi)$.

We further endow $\CX_\eps$ with a family of seminorms $\|\cdot\|_{\alpha;\K_\eps;z;\eps}$, where $\alpha\in \R$, $\K_\eps$ are compact subsets in $\R\times \T^2$ with diameter bounded by $2\eps$, and $z=(t,x)\in\R\times \obonds$ is such that $z\in\K_\eps$. 
We then define for $A=(A_1,A_2)\in \CX_\eps$ with $A_i \in \CX_\eps^{(i)}$  
\begin{equ}[eq:seminorm]
\|A_i\|_{\alpha;\K_\eps;z;\eps}^{(i)}= \sup_{\lambda\in (0,\eps]}\sup_{\phi\in \Phi_{\eps,z}^{\lambda}}
\lambda^{-\alpha}\Big|\int_\R A_i(s,x)(\CS_{2,t}^{\lambda}\phi)(s)\mrd s\Big|\;,
\end{equ}
where $z=(t,x)\in\R\times \obonds_i$, and $\|A_j\|_{\alpha;\K_\eps;z;\eps}^{(j)}=0$ for $j\neq i$.
Here, for $\lambda\in (0,\eps]$ we denote by $\Phi_{\eps,z}^{\lambda}$
the set of all functions $\phi\colon\R\to\R$ with $\|\phi\|_{\CC^3}\leq 1$
and support contained in the unit interval,
so that  moreover the support of 
$\CS_{2,t}^{\lambda}\phi \eqdef \lambda^{-2} \phi(\frac{\cdot - t}{\lambda^2})$
 is contained in the set $\{s\in\R\,:\, (s,x)\in\K_\eps\}$.
We write $\|A\|_{\alpha;\K_\eps;z;\eps}=\sum_{i=1}^2 \|A_i\|_{\alpha;\K_\eps;z;\eps}^{(i)}$.

\begin{remark}\label{rem:small-scale-norm}
 In~\cite{EH19}
a simpler example of seminorms was given in the ``semi-discrete case'' as
$\|f\|_{\alpha;\K_\eps;z;\eps}
=
\eps^{-\alpha} \sup_{y} |f(y)|$ where
$y$ is over lattice points in $\K_\eps$.
Here we instead take
\eqref{eq:seminorm} from \cite[(2.3)]{EH21}. 
One reason is that our noise $\dt \BM$ is not bounded in $L^\infty$.
Also,
choosing  \eqref{eq:seminorm}  has the advantage  that the bounds on large scales, i.e. scales larger than $\eps$,  and on small scales, can be obtained in the same way, which simplifies moment bounds for discrete models, see
Section~\ref{subsec:moments_discrete_models}.
\end{remark}

\begin{remark}
We remark that our choice of the small scale norm \eqref{eq:seminorm} actually does not 
satisfy the exact statement of \cite[Assumption~5.4]{EH19} (which is used to prove  \cite[Proposition~5.6]{EH19}) 
for the operators $\partial_i^\pm,\partial_i,\bar\partial_i$
defined in Sec.~\ref{sec:ab-notation}. It was assumed therein that 
for all functions $f$ on $\obonds$,  all $\alpha\in\R$, all compact sets $\K_\eps$, and all $z\in\R^d$,
\begin{equ}\label{eq:Dismallscale}
\|\partial_i f\|_{\alpha-1;\K_\eps;z;\eps}\lesssim \sup_{h\in\CB_\s(0,\eps)} \|f\|_{\alpha;\tau_h\K_{\eps};z;\eps}
\end{equ}
uniformly in $\eps$,
with $\tau_h\K_{\eps} = \K_{\eps}+h$,
and likewise for $\bar\partial_i$ and $\partial_i^\pm$, which does not hold for us
because the right-hand side of \eqref{eq:seminorm} only depends on $A_i(\cdot,x)$ where $x$ is the spatial coordinate of the given point $z$, but not on the values of $A_i$ at neighboring points. 
However, it can be shown that \eqref{eq:Dismallscale} holds with $z$ on the right-hand side modified to $z+h$,
and that with this modification  \cite[Proposition~5.6]{EH19} is still valid.
In any case, we will not use  \cite[Proposition~5.6]{EH19}, and instead we will prove Lemma~\ref{lem:DRcommute} -- see Remark~\ref{rem:D-nonlocal} on the differences between our construction of differentiations and that of \cite{EH19}.
%It is easy to check that the small scale norm \eqref{eq:seminorm} behaves in the right way with respect to  these  differentiations
%(\cite[Assumption~5.4]{EH19}).
%\haoText{Just trying to understand what is going wrong:
%recalling $z=(t,x)$,
%\begin{equs}
%\|\partial_i f\|_{\alpha-1;\K_\eps;z;\eps}
%=\sup_{\lambda\in (0,\eps]}\sup_{\phi\in \Phi_{\eps,z}^{\lambda}}
%\lambda^{-\alpha+1}\Big|\int_\R 
%\frac{1}{\eps} (f(s,x+\eps)-f(s,x))(\CS_{2,t}^{\lambda}\phi)(s)\mrd s\Big|
%\\
%\lesssim
%\sup_{\lambda\in (0,\eps]}\sup_{\phi\in \Phi_{\eps,z}^{\lambda}}
%\lambda^{-\alpha}\Big|\int_\R 
%(f(s,x+\eps)-f(s,x))(\CS_{2,t}^{\lambda}\phi)(s)\mrd s\Big|
%\end{equs}
%Here, by definition, the integral doesn't have any contribution unless $s$ is such that $(s,x)\in \K_\eps$. 
%Maybe we want to bound the two terms separately. But note that by definition
%\begin{equs}
% \|f\|_{\alpha;\K_{\eps}+\eps ;z;\eps}
% =\sup_{\lambda\in (0,\eps]}\sup_{\phi\in \tilde\Phi_{\eps,z}^{\lambda}}
%\lambda^{-\alpha}\Big|\int_\R 
%f(s,x)(\CS_{2,t}^{\lambda}\phi)(s)\mrd s\Big|
%\end{equs}
%where $\tilde\Phi_{\eps,z}^{\lambda}$ is defined as for $\Phi_{\eps,z}^{\lambda}$ with $\K_{\eps}$ replaced by $\K_{\eps}+\eps$.
%So basically the problem is that the RHS of \eqref{eq:Dismallscale} 
%doesn't involve $f(\cdot,x+\eps)$ at all. 
%
%But then the following bound should hold?
%\begin{equ}
%\|\partial_i f\|_{\alpha-1;\K_\eps;z;\eps}\lesssim \sup_{h\in\CB_\s(0,\eps)} \|f\|_{\alpha;\tau_h\K_{\eps};z+h;\eps}
%\end{equ}
%}
\end{remark}
The fixed point problem will be
 posed in a space $\cD_\e^{\gamma,\eta}$
  of $\CT_{<\gamma} \eqdef \bigoplus_{\alpha <\gamma} \CT_{\alpha}$ valued 
functions, which allows for a blow-up of order $\eta$ near $t=0$. 
More precisely, we write $P = \{(0,x)\,:\,x\in\R^2\}$ for the ``$t=0$''-subspace.
We further let 
$\|z\|_{P}= 1\wedge \inf_{y \in P} \|z-y\|_\s$ and $\|y,z\|_{P}= \|y\|_{P}\wedge \|z\|_{P}$,
as well as
\begin{equation}
\label{eq:kp}
\K_P=\{(y,z)\in (\K\setminus P)^2:\, y\neq z,\, \|y-z\|_\s\leq \|y,z\|_{P}\}.
\end{equation}
Consider now a function $f\colon\R^3\setminus P\to\CT_{<\gamma}$  for $\gamma>0$.
By  \eqref{e:T1T2} we write $f=(f_1,f_2)$ with $f_i $ valued in $\CT^{(i)}$.
As the last component of the discrete setup \eqref{e:framework}, we define for any compact set $\K\subset\R^3$
and map $\Gamma^\e \colon\R^3\times\R^3\to \CG$ 
\begin{equs} [e:smallscale-norm]
\$ f_i\$_{\gamma,\eta;\K;\e}^{(i)}
&\eqdef
\sup_{\substack{z\in\{\K\setminus P\}\cap\{\R\times \obonds_i\} \\ \|z\|_P < \e}}
\sup_{\beta <\gamma}
\frac{\|f_i(z)\|_\beta}{\|z\|_{P}^{(\eta-\beta)\wedge 0}}
\\
&+ \sup_{\substack{y,z\in\K_P\cap\{\R\times \obonds_i\}  \\ \|y-z\|_\s<\e}}
\sup_{\beta <\gamma}
\frac{\|f_i(z)-\Gamma^\e_{zy} f_i(y)\|_\beta}{\|y-z\|_\s^{\gamma-\beta}\|y,z\|_{P}^{\eta-\gamma}}\;,
\end{equs}
which only depends on the values of $f_i$ in $\K$.  
We then write $\$ f\$_{\gamma,\eta;\K;\e}=\$ f_1\$_{\gamma,\eta;\K;\e}^{(1)}+\$ f_2\$_{\gamma,\eta;\K;\e}^{(2)}$.
We also define  
$\$ f\$_{\gamma;\K;\e}^{(i)}$ as in  \eqref{e:smallscale-norm}, but with $\eta=\gamma$ and with $\K\setminus P$ and $\K_P$ replaced by $\K$,
and similarly for $\$ f\$_{\gamma;\K;\e}$.

Given another map $\bar\Gamma^\e\colon\R^3\times\R^3\to \CG$ and $\bar f\colon\R^3\setminus P \to \CT_{<\gamma}$, we define the small scale distance $\$ f;\bar f\$_{\gamma,\eta;\K;\eps}$ in the same way as~\eqref{e:smallscale-norm} except, in the first term, $f_i$ is replaced by $f_i-\bar f_i$ in the first term, and, in the second term,
\begin{equ}\label{eq:increment_shorthand}
f_i(z,y) \eqdef f_i(z)-\Gamma^\e_{zy}f_i(y)
\end{equ}
is replaced by $f_i(z,y) - \bar f_i(z,y)$ where $\bar f_i(z,y) = \bar f_i(z)-\bar\Gamma^\e_{zy}\bar f_i(y)$.

\subsection{Models, integration, shifting and multiplication by \texorpdfstring{$\eps$}{epsilon}}
\label{sec:Models}

A discrete model $Z^\eps=(\Pi^\eps,\Gamma^\eps)$ consists of a collection of maps 
\begin{equ}[e:model-split]
\R^3\ni
z\mapsto \Pi_z^\eps\in 
L(\CT^{(1)}, \CX_\eps^{(1)})\oplus L(\CT^{(2)}, \CX_\eps^{(2)})
\subset L(\CT, \CX_\eps)
\end{equ}
 and $\Gamma^\eps\colon\R^{3}\times\R^{3}\to \CG$ satisfying the usual algebraic conditions
\begin{equ}\label{eq:model_algeba}
\Pi_z \Gamma_{zy} = \Pi_y\;,\qquad
\Gamma_{zy}\Gamma_{yx} = \Gamma_{zx}\;,
\end{equ}
 and for any compact set $\K\subset\R^3$ every $i\in \{1,2\}$ and every $\tau\in \CT_{\ell}^{(i)} \cap \{\CT^\YM\oplus\CT^\rem\}$ one imposes
the analytical estimates
\begin{equs}   \label{eq:Pi}
\bigl|\bigl(\iota_\eps\Pi^{\eps}_z \tau\bigr)(\phi_{z}^{\lambda})\bigr|
 \lesssim \|\tau\|_\ell\lambda^{\ell}\;,&
 \quad 
\|\Pi^{\eps}_z \tau\|_{\ell;\K_\eps,z;\eps}
 \lesssim \|\tau\|_\ell\;,
\\
\label{eq:Gamma}
\|\Gamma^\eps_{z\bar{z}}\tau\|_m\lesssim \|\tau\|_\ell \|z-\bar{z}\|_\s^{\ell-m}\;,&\quad
\$ y\mapsto \Gamma^{\eps}_{yz}\tau\$_{\ell;\K; \eps}\lesssim \|\tau\|_\ell\;.
\end{equs}
For some fixed $\gamma > 0$, these bounds are assumed to be uniform over $\lambda \in (\eps,1]$,
all $\phi \in \CB^r_0$,
all
$\ell <\gamma$, and $m<\ell$, all $\tau\in \CT_{\ell}\cap\{\CT^\YM\oplus\CT^\rem\}$,
and uniform over $z,\bar{z}\in\K$ such that $\|z-\bar{z}\|_\s\in (\eps,1]$.
The second bound in \eqref{eq:Pi} is also uniform over all compact subsets $\K_\eps$ of diameter bounded by $2\eps$ and over $z\in\K_\eps$.

Note that for $\tau=(\tau_1,\tau_2)\in \CT$ where $\tau_i \in \CT^{(i)}$, $(\Pi_z^\e \tau)(\phi)$ 
is understood as  $((\Pi_z^\e \tau_1)(\phi),(\Pi_z^\e \tau_2)(\phi))$.
Furthermore,
we require that $\Pi^\eps$
acts on abstract polynomials of order at most $1$ by
\begin{equ}[eq:poly_model]
\Pi_z^\eps \bone^{(i)} (y) = 1\;,
\quad
\Pi_z^\eps \mbX^{(i)}_j (y)=y_j-z_j \;,
\qquad
y\in \R\times \obonds_i \;,
\quad
i,j \in \{1,2\}\;,
\end{equ}
where in the second expression we recall that we identify a bond with its midpoint,
and that
\begin{equ}\label{eq:Gamma_polys}
\Gamma_{zy}^\eps\bone^{(i)}=\bone^{(i)}\;,
\quad
\Gamma_{zy}^\eps\X^{(i)}_j=\X^{(i)}_j + (z_j-y_j) \bone^{(i)}\;.
\end{equ}
(As mentioned earlier,
we will only need polynomials up to order $1$.)

We denote by $\|\Pi^\eps\|_{\gamma;\K}^{(\eps)}$ and $\|\Gamma^\eps\|_{\gamma;\K}^{(\eps)}$ the smallest proportionality constants for which~\eqref{eq:Pi} and~\eqref{eq:Gamma} hold respectively
and set $\$Z^\eps\$_{\gamma;\K}^{(\eps)}=\|\Pi^\eps\|_{\gamma;\K}^{(\eps)}+ 
\|\Gamma^\eps\|_{\gamma;\K}^{(\eps)}$.
We can compare two discrete models by $\$Z^\eps;\bar{Z}^\eps\$_{\gamma;\K}^{(\eps)}$ and a discrete model with a continuous model
by  $\$Z;Z^\eps\$_{\gamma;\K;\ge \e}^{(\eps)}$ as 
in   \cite[Remarks~2.9-2.10]{EH19}.

\begin{remark}\label{rem:model_bounds_sector}
We impose~\eqref{eq:Pi}-\eqref{eq:Gamma} only on the sector $\CT^\YM\oplus\CT^{\rem}\subset \CT$ since
the fixed point problem associated to~\eqref{e:Aeps} will be formulated on this sector.
\end{remark}
\begin{definition}[\cite{EH19} Def.~5.5]\label{def:compat_diff}
A model $(\Pi,\Gamma)$ is called
 \textit{compatible with differentiations} if for all $z\in \R^3$, $i \in \{1,2\}$,
\begin{equ}[e:compatible-D]
\partial_i^\pm \Pi_z  =  \Pi_z  \CD_i^\pm \;,
\qquad
\partial_i \Pi_z  =  \Pi_z  \CD_i \;,
\qquad
\bar\partial_i \Pi_z  =  \Pi_z  \bar\CD_i\;.
\end{equ}
\end{definition}
\begin{remark}\label{rem:PiDX}
The specifications~\eqref{e:DX=1} and~\eqref{eq:Gamma_polys} are consistent in that~\eqref{eq:Gamma_Diff_commute}
is satisfied on the polynomial sector.
Likewise the specifications~\eqref{e:DX=1} and~\eqref{e:compatible-D} are consistent in that, restricted to the polynomial sector, every model is compatible with differentiations.
For instance, ${\bf X}_1^{(2)}$ encodes $x_1$ living on $\obonds_2$,
and by \eqref{e:DX=1} $\bar\CD_1 {\bf X}_1^{(2)}$ should yield $1$ on $\obonds_1$. 
Indeed, 
%recalling the definition of $\bar\partial_1$,
for any $z\in \R^3$ and $e\in \R\times \obonds_1$, by \eqref{e:compatible-D},
\begin{equ}
\Pi_z (\bar\CD_1 {\bf X}_1^{(2)} ) (e) 
=\tfrac{1}{2\e}
\big(\Pi_z  {\bf X}_1^{(2)} (e^\northeast) 
+\Pi_z  {\bf X}_1^{(2)} (e^\southeast) 
-\Pi_z  {\bf X}_1^{(2)} (e^\northwest) 
-\Pi_z  {\bf X}_1^{(2)} (e^\southwest) \big)
\end{equ}
which is $\frac{1}{2\e}
(\frac\e2+\frac\e2-(-\frac\e2)-(-\frac\e2))=1$.
\end{remark}
In Section~\ref{sec:tr-reg} we built the regularity structure $\CT$ with
 integration operators $\mcb{I}_i$  and shifting operators $\CS_\be$.
Below we describe how they interact with models.

\subsubsection{Shifting}

Recall
from Section~\ref{subsubsec:RS} that we have
shifting operators $\CS_{\be}\colon V_\be\to \CT$ for $\be\in\mcE_\times$
defined on a sector
 $V_\be\subset \CT$.

\begin{definition}\label{def:compat_shifts}
We say that a model $(\Pi,\Gamma)$ is \textit{compatible with shifts}
if for all $\be\in \CE_\times$, $z,x,y\in\R^3$, and test functions $\phi$,
\minilab{Model-sf}
\begin{equs}  	 \label{e:Pi-sf}
(\Pi_z \CS_{\be}\tau) (\phi) 
& =  (\Pi_{z+\be} \tau) (\phi(\cdot -\be)) \;,
\\
\Gamma_{xy} \left( \CS_{\be}\tau \right)
 &=  
\CS_{\be}\left( \Gamma_{x+\be , y+\be} \tau \right) \;.
			\label{e:Gamma-sf}
\end{equs}
\end{definition}
\begin{remark}
Since $V_\be$ is a sector, $\CS_\be \Gamma_{x+\be , y+\be} \tau$ is well-defined whenever $\CS_\be\tau$ is.
\end{remark}
If $\Pi_z\tau$ is continuous for all $z\in\R^3$ and $\tau\in V_\be$,
then taking $\phi=\delta_y$,
the relation~\eqref{e:Pi-sf} implies 
\[
(\Pi_z \CS_{\be}\tau) (y) =  (\Pi_{z+\be} \tau) (y+\be)\;.
\]
Here $y\in \R\times \obonds_i$ if $\CS_{\be}\tau \in \CT^{(i)}$.
The fact that the ``base points'' $z$ are also shifted is crucial. 
Here we recall that $\Pi_z \tau$ is a vector notation (i.e. $1$-form),
namely if $\tau\in \CT^{(i)}$, $\be\in \CE_{\times}$, then $\CS_{\be} \tau\in  \CT^{(j)}$,
and then  the above expression is understood as $y\in \obonds_j$ and $y+\be\in\obonds_i$,
where $i\neq j$.

\subsubsection{Admissible models}
\label{sec:Admissible models}

%As in \cite[Sec.~5]{HM18}
%or \cite{EH21} 
%% \cite[Sec.~A.1]{KPZJeremy}
% it is possible to decompose the Green's function $P^\e$
% of  $\partial_t-\Delta$ 
% on $\R\times \T_\e^2$ as 
%$P^\e=K^\eps+R^\eps$
%where $K^\eps= \sum_{n=1}^{N}K_n^\e$ (recall $\e=2^{-N}$).
%Here each $K_n^\e$ and $R^\e$ can be extended to functions on $\R^3$ with bounded derivatives up to second order (see \cite[Sec.~5]{HM18}),
%and $K_n^\e$ has support in $\{z\in\R^3: \|z\|_\s \lesssim 2^{-n}\}$,
%is non-anticipative, i.e. $K_n^\eps(\cdot,t)=0$ for $t<0$,
%annihilates polynomials of degree $k$ with $|k|_\s \leq 3$,
%and
%\begin{equ}  
%|D^kK_n^\e(z)|\lesssim 2^{n(2+|k|_\s)} \quad (n<N,|k|_\s\leq 3)\;,
%\qquad
%|K_N^\e(z)|\lesssim 2^{2N} \;,
%\end{equ}
%uniformly in all parameters. Here $D^k$ is the usual (i.e. smooth) derivative.
%Furthermore, $R^\eps$ is supported outside a neighbourhood of zero, is non-anticipative, and $\|R^\eps\|_{\CC^3}$ is bounded uniformly in $\eps$.
%
%In our case, $\obonds_1,\obonds_2$ are two copies of $\T_\e^2$, so 
%we have two kernels $K^{1;\eps}$ and $K^{2;\eps}$ for horizontal and vertical bonds respectively
%and we extend $K^{i;\eps}(x,y)=K^{i;\eps}(y-x)$ to all of $\R^3$.
%(As functions of a single variable $K^{1;\eps}(z)=K^{2;\eps}(z)$.)
%We also write $K_n^{i;\e}$ and $R^{i;\e}$ for $i\in \{1,2\}$.		\label{page:K12}

As in \cite[Sec.~5]{HM18}
or \cite{EH21} 
% \cite[Sec.~A.1]{KPZJeremy}
 it is possible to decompose the Green's function $P^\e$
 of  $\partial_t-\Delta$ 
 on $\R\times \T_\e^2$ as 
$P^\e=K^\eps+R^\eps$,
where $K^\eps$ is supported in a ball of radius $\frac1{10}$
and has a multiscale decomposition satisfying suitable bounds.
Furthermore, $R^\eps$ is supported outside a neighbourhood of zero, is non-anticipative, and $\|R^\eps\|_{\CC^3}$ is bounded uniformly in $\eps$.
In our case, $\obonds_1,\obonds_2$ are two copies of $\T_\e^2$, so 
we have two kernels $K^{1;\eps}$ and $K^{2;\eps}$ correspondingly,
and we extend $K^{i;\eps}(x,y)=K^{i;\eps}(y-x)$ to all of $\R^3$.
(As functions of a single variable $K^{1;\eps}(z)=K^{2;\eps}(z)$.)
We also have $R^{i;\e}$ for $i\in \{1,2\}$.		\label{page:K12}

%Recall that $\bar\CT=\bar\CT^{(1)}\oplus\bar\CT^{(2)}$ is the polynomial regularity structure.
%For $\zeta\in\R$ and $i\in\{1,2\}$,
%define the linear map $T^{i;\e}_\zeta \colon \CX_\eps^{(i)} \to (\bar\CT^{(i)})^{\R^3}$ by
%\begin{equs}
%(T^{i;\e}_\zeta f)(z) &\eqdef \sum_{|k|_\s <\zeta+2} \frac{(\X^{(i)})^k}{k!}   \CQ_k\big[\big(T^{i;\e}_{\zeta}f\big) (z) \big]\;,
%\label{eq:T_eps_def}
%\\
%\CQ_k\big[\big(T^{i;\e}_{\zeta}f\big) (z) \big] &\eqdef
% \sum_{n<N} (D_1^k K_n^{i;\e}*_{(i)} f)(z)
% + \delta_{k=0}( K_N^{i;\e}*_{(i)} f)(z)\;,
%\end{equs}
%where $(\X^{(i)})^k\eqdef 0$ if $|k|_\s>1$
%and where the second line has the
%``Taylor terms''  as in  
%\cite[Sec.~2]{EH21}, and they satisfy
%\cite[Assumptions 4.3 and 4.5]{EH19}.
%Here $*_{(i)}$ is (semi)discrete convolution over $\R\times \obonds_i$,   \label{page:convolution-i}
%then evaluated at $z\in \R^3$ since $K_n^{i;\e}$ are extended.
%When the index $i$ is clear for the context, we will simply write $*$ for $*_{(i)}$.
%
%Following~\cite[Sec.~4]{EH19}, a model $(\Pi^\eps,\Gamma^\eps)$ is called \textit{admissible} if
%for all $\zeta\in\R$, $\tau\in\cbT^{(i)}_\zeta$ with $\CI_i\tau\in\CT^{(i)}$,
%$z\in\R^3$, $i\in\{1,2\}$, and $w\in \R\times \obonds_i$,
%\begin{equ}[eq:admissible]
%(\Pi^\e_z\cbI_i \tau ) (w) = (K^{i;\eps} *_{(i)} \Pi^\e_z\tau)(w)
%-\big(\Pi_z^\e \big[(T^{i;\e}_\zeta \Pi^\e_z\tau)(z)\big]\big)(w)\;.
%\end{equ}

Following~\cite[Sec.~4]{EH19}, a model $(\Pi^\eps,\Gamma^\eps)$ is called \textit{admissible} if
for all  $\tau\in\cbT^{(i)}_\zeta$ with $\CI_i\tau\in\CT^{(i)}$,
$z\in\R^3$, $i\in\{1,2\}$, and $w\in \R\times \obonds_i$,
\begin{equ}[eq:admissible]
(\Pi^\e_z\cbI_i \tau ) (w) = (K^{i;\eps} *_{(i)} \Pi^\e_z\tau)(w)
- \mbox{``Taylor terms''}
\end{equ}
where the ``Taylor terms'' are as in \cite[Sec.~4]{EH19} which satisfies \cite[Assumptions 4.3 and 4.5]{EH19} (see, however, footnote~\ref{foot:Rhys}).
For the rest of the article $*_{(i)}$ denotes (semi)discrete convolution over $\R\times \obonds_i$,   \label{page:convolution-i}
and when it is clear from the context, we simply write $*$ for $*_{(i)}$.

\begin{remark}
\label{rem:admissible_Holder}
Recall that $\CX_{\e} = \CC^{-1+\kappa}(\R,\R^{\obonds})$.
Admissibility of $(\Pi^\e,\Gamma^\e)$ therefore implies that $\Pi_x\CI_i\tau$ is $\kappa$-H\"older continuous (in time) for all $\CI_i\tau\in\CT$ and $x\in\R^3$.
\end{remark}

\subsubsection{Multiplication by \texorpdfstring{$\eps$}{epsilon}}
\label{subsubsec:mult_eps_models}

Recall from~\eqref{e:rule-main1} and~\eqref{eq:base_case_2} that we have three types of noises,
$\Xi_i = \Xi_{\mfl_i}$, $\bar\Xi_i = \Xi_{\bar\mfl_i}$, and $\Xi_{\mfl^{=,\pm}_i}$, corresponding to $\xi$, $\eps^{1-\kappa}\xi$, and a (quadratic function of)
$\eps^{\frac12-\frac\kappa2}\xi$ respectively.
These noises should therefore be treated as closely related, and we do this by
imposing suitable conditions on models.

Recall that a graphical tree $\sigma$ denotes the set of trees $\{\tau_1,\ldots,\tau_n\}\subset \mfT$
which are of the form $\sigma$, e.g. $\<IXibar>$ denotes the set
\begin{equ}
\{\CI_1\bar\Xi_1, \CI_2\bar\Xi_2\}\cup \{\CS_\be \CI_i\bar\Xi_i\,:\,i\in\{1,2\},\be\in\CE_\times\}\;,
\end{equ}
which contains $10$ trees.
To lighten notation, we write $\CT[\sigma]\eqdef \bigoplus_{\tau \in \sigma}\CT[\tau]$,
e.g.
$
\CT[\<IXibar>]=\bigoplus_{\tau \in \<IXibar>}\CT[\tau] \simeq (\mfg^*)^{\oplus 10}
$.
Consider the sets of negative trees
\begin{equ}\label{eq:def_mfT_eps}
\mfT_\eps \eqdef \{ \tau \in \sigma\,:\, \sigma\in \{\<IXiI'Xi_notriangle>\;,\;
\<I'Xi_notriangle>\;,\;
\<R2-1new>\}\}\;,
\quad \mfT_{\<IXi>} \eqdef \{ \tau \in \<IXi>\}
\end{equ}
(see Remark~\ref{rem:cE_motivation} for motivation).
Define the map $\sigma\mapsto \sigma_\eps$, $\mfT_\eps\sqcup \mfT_{\<IXi>} \to \mfT$, by
\begin{equ}\label{eq:tau_to_bar_tau}
\<IXiI'Xi_notriangle> \mapsto \<IXiI'Xibar>\;,\quad
\<I'Xi_notriangle> \mapsto \<I'Xibar>\;,\quad
\<R2-1new> \mapsto \<cherry232>\;,
\quad
\<IXi> \mapsto \<IXibar>\;,
\end{equ}
understood as replacing $\Xi_i$ by $\bar\Xi_i$ for the 1st, 2nd, and 4th terms, and by replacing every $\tilde\Xi_{i}$ by $\bar \Xi_{i}$ in~\eqref{eq:base_case_2} for the 3rd term.
We denote in the same way $\tau \mapsto \tau_\eps$ the lifting of this map to $\bigoplus_{\sigma\in \mfT_\eps\sqcup\mfT_{\<IXi>}} \CT[\sigma]\to \CT$ via the functor $\Func_V$.
Note that $\tau\mapsto \tau_\eps$
is well-defined
due to the trees $\<IXiI'Xibar>$ added in Step 3 in Section \ref{subsubsec:Step_3} and
since the terms in~\eqref{eq:base_case_2} with $\tilde \Xi_i$ replaced by $\bar\Xi_i$ are in $\mfT_-$.
The map  $\tau\mapsto \tau_\eps$ is furthermore an isomorphism onto its image.

For the following definition, recall from Remark~\ref{rem:admissible_Holder} that
$\Pi_x\tau$ is $\kappa$-H\"older continuous in time for all $\tau\in\CT[\<IXi>]\simeq \mfg^*\oplus\mfg^*$, $x\in\R^3$, and admissible models $(\Pi,\Gamma)$.
In particular, pointwise evaluation
$\Pi_x\tau(x)$ makes sense.

\begin{definition}\label{def:compat}
A model $(\Pi,\Gamma)$ is called \textit{compatible with multiplication by $\eps$} if
\begin{itemize}
\item it is admissible,
\item for all $\tau\in \bigoplus_{\sigma\in \mfT_\eps}\CT[\sigma]$ and $x\in \R^3$, one has
$
\Pi_x \tau_\eps = \eps^{1-\kappa} \Pi_x \tau
$,
\item for all $\tau\in\CT[\<IXi>]$ and all $x,y\in\R^3$
\begin{equs}[eq:IXi_compat]
\Pi_x \tau_\eps &= \eps^{1-\kappa} \Pi_x \tau - \eps^{1-\kappa} (\Pi_x \tau) (x)\;,
\\
\Gamma_{xy}\tau_\eps &= \tau_\eps + \{\eps^{1-\kappa}(\Pi_x\tau)(x) - \eps^{1-\kappa} (\Pi_y\tau)(y)\} \in \CT[\tau_\eps]\oplus \CT[\bone]\;,
\end{equs}
where, recalling the splitting $\tau=(\tau^{(1)},\tau^{(2)})\in\CT^{(1)}\oplus\CT^{(2)}$,
the term in the parentheses $\{\ldots\}$
is understood as an element of $\R^2$ that is canonically identified with an element of
$\CT[\bone]$.
\end{itemize}
We call $(\Pi,\Gamma)$ simply \textit{compatible} if it is compatible with differentiations (Definition~\ref{def:compat_diff}), with shifts (Definition~\ref{def:compat_shifts}),
and with multiplication by $\eps$.
\end{definition}
\begin{remark}
Condition~\eqref{eq:IXi_compat}, using $\Pi_x \tau= \Pi_y\tau$, gives the usual algebraic identity
$
\Pi_x \Gamma_{xy}\tau_\eps = \Pi_y \tau_\eps
$.
\end{remark}

\subsubsection{Canonical models}
\label{sec:canonical}

Consider a $\mfq$-valued Brownian motion $\BM=\{\BM_e\}_{e\in\obonds}$ where $\BM_e\in\CC(\R,\mfg)$ as in Section \ref{sec:approx-YM}.
Let $\xi^\e = \e^{-1} \dt \BM$.
We use $\xi^\e$ to build a canonical model as follows.
For every $x\in \R\times\T^2$ we set 
\begin{equ}[e:Pi-noise]
\Pi^\e_x\Xi = \xi^\e \;,
\qquad
\Pi^\e_x \bar\Xi = \e^{1-\kappa} \xi^\e\;,
\qquad
\Pi^\e_x \tilde\Xi = \e^{\frac12-\frac\kappa{2}} \xi^\e
\end{equ}
 understood in the natural way -- for every $\tau=(\tau_1,\tau_2)\in\cbT[\Xi_1]\oplus\cbT[\Xi_2] \simeq \mfg^*\oplus\mfg^*$ 
we define $\Pi_{x}^\e \tau\in\CX_\eps=\CC^{-1}(\R,\R^{\obonds})$ by
\begin{equ}[e:Pi noise]
(\Pi_{x}^\e \tau)(s,e) = \tau_{i}(\xi^\e (s,e)) \in \R\;,\qquad \mbox{if }e\in \obonds_i
\end{equ}
and similarly for $\bar\Xi $, $\tilde\Xi$.
We emphasise that $\tilde\Xi$ is \textit{not} part of our regularity structure, and the final term in~\eqref{e:Pi-noise} is only used as input into a recursive definition as follows,
from which $\Pi^\e_x \Xi_{\mfl^{=,\pm}_i}$ are defined.

The canonical model is then 
defined recursively by $\Pi_x^\e \tau \bar \tau = \Pi_x^\e\tau \Pi_x^\e\bar \tau$
for any  $\tau,\bar\tau\in\cbT$ for which $\tau\bar\tau\in\cbT$, where the product $\Pi_x^\e\tau \Pi_x^\e\bar \tau$ is either classical (e.g. for $\<IXiI'Xi_notriangle>$) or in the Stratonovich sense (e.g. for $\<PsiXi1>$), together with~\eqref{e:compatible-D},  \eqref{Model-sf},
\eqref{eq:poly_model}, \eqref{eq:admissible}, \eqref{e:Pi-noise}.
These conditions on $\Pi^\eps_x$ uniquely determine $\Gamma^\eps_{xy}$ in the usual way as in~\cite[Sec.~8.2]{Hairer14} or~\cite[Sec.~2.2]{ESI_lectures}.

\begin{proposition}\label{prop:canonical_model}
The recursion above yields a model $(\Pi^\e,\Gamma^\e)$ that is compatible.
\end{proposition}

\begin{proof}
Admissibility and compatibility with differentiations is clear since we impose~\eqref{eq:admissible} and~\eqref{e:compatible-D} at every step in the recursion.
Compatibility with multiplication by $\eps$ is also clear by~\eqref{e:Pi-noise}
(only the relation~\eqref{eq:IXi_compat} is not completely trivial but follows from $|\<IXibar>|\in (0,1)$).

Compatibility with shifts is clear since we impose \eqref{Model-sf},
and the only thing to verify is that the recursion indeed yields a model.
To verify this, consider $\CS_\be\tau\in \CT^{(i)}_\ell \cap \CT[\CS_{\be}\sigma]$ for some $\CS_\be\sigma\in\mfT_i$ and $i\in\{1,2\}$, $\ell\in\R$.
By the recursive procedure, we have a uniform bound
$|\left( \Pi^\eps_z \tau\right) (\varphi^\lambda_z)|\lesssim \lambda^{|\tau|}$,
and therefore by \eqref{e:Pi-sf}
\begin{equ}
|\left( \Pi^\eps_z \CS_{\be}\tau  \right) (\varphi^\lambda_z)|
= |\left(\Pi^\eps_{z+\be} \tau \right) (\varphi^\lambda_z(\cdot-\be)) |
=  |\left(\Pi^\eps_{z+\be} \tau \right)(\varphi^\lambda_{z+ \be} ) |
\lesssim \lambda^{|\tau|}\;,
\end{equ}
where the second equality is by definition of $\varphi^\lambda_z$.
The required bounds \eqref{eq:Gamma} for $\Gamma^\eps$ 
also hold simply because $|(x+\be)-(y+\be)|=|x-y|$.
In fact the optimal constants in these bounds
are the same for $\tau$ and $\CS_{\be}\tau$ since $\|\tau\|_\ell=\|\CS_{\be}\tau\|_\ell$.
The required bound on 
$\|\Pi^\eps_z \CS_{\be}\tau\|_{\ell;\K_\eps,z;\eps}$
also clearly holds given the bound on 
$\|\Pi^{\eps}_z \tau\|_{\ell;\K_\eps,z;\eps}$.

Assuming that the algebraic requirements 
 hold for $\tau$, one has
\begin{equs}
\Gamma^\eps_{xy} \Gamma^\eps_{yz}\left(\CS_{\be} \tau \right)
&  =
\Gamma^\eps_{xy}  \CS_{\be}\left( \Gamma^\eps_{y+ \be , z+\be} \tau \right)
=
\CS_{\be} \left( \Gamma^\eps_{x+ \be,y+ \be}   \Gamma^\eps_{y+ \be , z+\be} \tau \right)\\
& =
\CS_{\be} \left(    \Gamma^\eps_{x+ \be , z+\be} \tau \right)  
= 
\Gamma^\eps_{x z} \left(  \CS_{\be} \tau \right) \;,
\end{equs}
and
\begin{equs}
\Pi^\eps_x \Gamma^\eps_{xy} \left( \CS_{\be}\tau \right) (\phi)
& =
\Pi^\eps_x \CS_{\be} \left( \Gamma^\eps_{x+ \be , y+\be} \tau \right) (\phi)
=
\Pi^\eps_{x+ \be} \left( \Gamma^\eps_{x+ \be , y+\be} \tau \right) (\phi(\cdot-\be)) 
\\
& =
\left( \Pi^\eps_{y+ \be} \tau \right)(\phi (\cdot-\be))
 =
\left(\Pi^\eps_{y} \CS_{\be} \tau\right) (\phi) \;.
\end{equs}
So the algebraic relations also hold for $ \CS_{\be}\tau$.
\end{proof}

\subsection{Modelled distributions and reconstruction}
\label{sec:reconstruction}

Recall the notation $\$ f\$_{\gamma,\eta;\K;\eps}$ from \eqref{e:smallscale-norm}, and $\|z\|_P$, $\|y,z\|_P$ therein.

\begin{definition}
\label{def:Dgamma}
Let $\gamma,\eta\in\R$ and fix a discrete model $(\Pi^{\eps},\Gamma^{\eps})$. The space $\cD_\eps^{\gamma,\eta} \equiv \cD_{\Gamma^\e,\eps}^{\gamma,\eta}$
consists of all maps $f\colon\R^3\setminus P\to \CT_{<\gamma}$ 
such that
$\$ f \$^{(\eps)}_{\gamma,\eta;\K}  <\infty$
 for every compact set $\K\subset\R^3$ where $\$ f \$^{(\eps)}_{\gamma,\eta;\K}$ is defined as
\begin{equ} [e:Dgamma]
\sup_{\substack{z\in \K \backslash P\\
\|z\|_P\geq \eps}}
\sup_{\beta<\gamma} 
\frac{\|f(z)\|_{\beta}}{\|z\|_P^{(\eta-\beta)\wedge 0}}
+
\sup_{\substack{(y,z)\in \K_P\\
\eps \leq \|y-z\|_\s\leq 1}}
\sup_{\beta<\gamma} 
\frac{\|f(z)-\Gamma^{\eps}_{zy}f(y)\|_{\beta}}{\|y-z\|_\s^{\gamma-\beta} \|y,z\|_P^{\eta-\gamma}}
+
\$ f\$_{\gamma,\eta;\K;\eps} \;.
\end{equ}
We say that $f\in\cD_\eps^{\gamma,\eta}$ is a modelled distribution. We write 
$\cD_\eps^{\gamma,\eta;i}$ for the space of such $f$ taking values in $\CT_{<\gamma}^{(i)}$,
and $\cD_{\alpha,\eps}^{\gamma,\eta;i}$ if $\|f(z)\|_\beta=0$ for all $\beta<\alpha$ and $z\in\R^3$.
If $f$ takes values in a sector $V\subset\CT$, we write $f\in \cD_\eps^{\gamma,\eta}(V)$.
We also write $f\in \cD^\gamma_\e$ if $\$ f \$^{(\eps)}_{\gamma;\K} <\infty$ where $\$ f \$^{(\eps)}_{\gamma;\K}$
is defined as in  \eqref{e:Dgamma} with $\eta=\gamma$ and with $\K\setminus P$ and $\K_P$ replaced by $\K$ in the first and second terms respectively and with
$\$ f\$_{\gamma;\K;\eps}$ in the last term.

For another discrete model $(\bar\Pi^\eps,\bar\Gamma^\eps)$ with modelled distribution $\bar f$, we define the distance $\$f;\bar f\$^{(\e)}_{\gamma,\eta;\K}$ in the
same way as~\eqref{e:Dgamma} except $f(z)$ is replaced by $f(z)-\bar f(z)$ in the first term,
$f(z,y)$ is replaced by $f(z,y)-\bar f(z,y)$ in the second term (where we recall the notation~\eqref{eq:increment_shorthand}),
and $\$ f\$_{\gamma,\eta;\K;\eps}$ is replaced by $\$ f;\bar f\$_{\gamma,\eta;\K;\eps}$.
\end{definition}
Now we define the reconstruction operator $\CR^\eps\colon\cD_\eps^{\gamma}\to\CX_\eps$.
The processes we consider are in general distributional in time.
For each fixed $x \in \obonds$, by the reconstruction theorem
\cite[Thm.~3.10]{Hairer14}, there is an operator 
$\CR_x$ (depending on $x$) such that for $f\in \cD_\e^\gamma$
\[
|(\CR_x (f(\cdot,x)) - \Pi^\e_{(t,x)} f(t,x))(\CS_{2,t}^{\lambda}\phi)| \lesssim \lambda^\gamma
\]
uniformly in $\lambda\in(0,1)$ where $\phi$ is a test function only in time.
Moreover, if $f$ takes values in a sector $V$ and
$\Pi^\e_{(t,x)} \tau$ is continuous for all $\tau\in V$, 
then
\[
\CR_x (f(\cdot,x))(t) = (\Pi_{(t,x)}^\eps f(t,x))(t,x)\;.
\]
Using this family $(\CR_x)_{x \in \obonds}$ 
we then define, for any space-time test function $\psi$ and $f\in \cD_\eps^{\gamma}$,
\[
(\CR^\e f)(\psi) = \Big(\e^2 \sum_{x \in \obonds_1} (\CR_x (f_1(\cdot,x)))(\psi(\cdot, x))\;,\;
\e^2 \sum_{x \in \obonds_2} (\CR_x (f_2(\cdot,x)))(\psi(\cdot, x))\Big)\;.
\]
If $f$ takes values in a sector $V$ and
$\Pi^\e_{(t,x)} \tau$ is continuous for all $\tau\in V$, then
\begin{equ}\label{eq:rec}
(\CR^\eps f)(z)= (\Pi_z^\eps f_i(z))(z)\;,
\qquad \forall i\in\{1,2\}\;, \; \forall z\in\R\times \obonds_i\;.
\end{equ}
Remark that, due to our definition of models~\eqref{e:model-split},
if $f \in \cD_\eps^{\gamma;i}$,
then $\CR^\eps f \in \CX_\eps^{(i)}$.

\begin{lemma}\label{lem:reconstruct}
For $\gamma>0$, $\alpha\in(-2,0]$, $\eta>-2$, $f\in \cD^{\gamma,\eta}_{\alpha,\e}$, a compact set $\K\subset\R^3$ of diameter at most $2\eps$, and $z\in\K$, one has
\begin{equs}
\|\CR^\e f - \Pi^\e_z f(z)\|_{\gamma;\K_\e;z;\e}
&\lesssim
\|\Pi^\e\|_{\gamma;\bar\K_\e}^{(\e)} 
\$ f\$_{\gamma;\K_\e;\e}\;,	\label{e:CR-approx-small-scale}
\\
\|\CR^\e f\|_{\alpha\wedge\eta;\K_\e;z;\e}
&\lesssim
\|\Pi^\e\|_{\gamma;\bar\K_\e}^{(\e)}\$ f\$_{\gamma,\eta;\K_\e;\e}\;,
\label{e:CR-small-scale}
\end{equs}
uniformly in $z$, $\K_\e$, and $\eps>0$, where $\bar\K_\e$ is the $1$-fattening of $\K_\e$.
\end{lemma}

\begin{proof}
The bound~\eqref{e:CR-approx-small-scale}, namely
$\CR^\e$ approximates $\Pi^\e_z f(z)$ up to order $\lambda^\gamma$
at scales smaller than $\e$,
follows from reconstruction theorem for $\CR^\e_x$.
The small scale bound \eqref{e:CR-small-scale} on $\CR^\e f$
follows from the reconstruction theorem for singular modelled distributions~\cite[Prop.~6.9]{Hairer14} together with the fact that
the definition \eqref{eq:seminorm} coincides with the usual way of measuring $\alpha$ regularity
(at small scales).
\end{proof}
By \eqref{e:CR-approx-small-scale}, \cite[Assumption~3.1]{EH19}  is satisfied.
\cite[Assumption~3.4]{EH19}  can be checked in the same way.
Furthermore, by~\eqref{e:CR-small-scale} and similar bounds involving a second discrete model and continuous models,~\cite[Assumption~3.13]{EH19}
is satisfied.
It follows that the reconstruction theorem for $\CR^\e$ holds,
with continuity properties in $f^\e\in \cD^{\gamma,\eta}_{\alpha,\e}$
(for the same $\gamma,\alpha,\eta$
as in Lemma~\ref{lem:reconstruct})
and in discrete models, and approximates the limiting reconstruction operator $\CR$,
in the sense of \cite[Theorems~3.6,~3.15]{EH19}.
\begin{remark}\label{rem:time_restrict}
For $f\in\cD^\gamma_\alpha$, due to our definitions $\CX_\e \eqdef \CC^{-1+\kappa}(\R,\R^{\obonds})$ and~\eqref{e:model-split}, which imply $\CR^\e f \in \CC^{-1+\kappa}(\R,\R^{\obonds})$,
we can multiply $\CR^\e f$ by the indicator function $1_{t>0}$ even if $\alpha\leq -2$
(which is not true in general in the continuum setting),
but $f\mapsto 1_{t>0}\CR^\e f$ may lack continuity properties that are uniform in $\eps>0$.
\end{remark}
\begin{remark}\label{rem:canonical_R_mult}
If $(\Pi^\e,\Gamma^\e)$ is the canonical model from Section~\ref{sec:canonical},
then $\CR^\e$ is multiplicative.
More precisely, if $f\in\cD^\gamma_\alpha,g\in\cD^{\bar\gamma}_{\bar\alpha}$
with $(\gamma+\bar\alpha)\wedge (\bar\gamma+\alpha)>0$ and taking values in sectors such that $fg$ is well-defined, then
$\CR^\e (fg)=(\CR^\e f) (\CR^\e g)$ almost surely,
where the product is either classical or in the sense of Stratonovich.
\end{remark}
%
%The integration operator $\CK^{i;\eps}_{\gamma}$ is given by, for $f\in\cD^{\gamma;i}_\eps$,
%\begin{equ}\label{eq:conv}
%\CK_{\gamma}^{i;\eps} f(z)=\mcb{I}_i f(z) +\sum_{\zeta\in A}\big( T_{\zeta}^{i;\eps}\big[\Pi_z^\eps f(z)\big]\big)(z)
%+\big(T_{\gamma}^{i;\eps}\big[\CR^\eps f-\Pi_z^\eps f(z)\big]\big)(z)\;.
% \end{equ}
%Here $A=\{|\tau|\,:\, \tau\in\mfT\}$ and
%$T^{i;\eps}_\zeta \colon \mcX_\e^{(i)} \to (\bar\CT^{(i)})^{\R^3}$ is given by~\eqref{eq:T_eps_def}.

For admissible models,  as in \cite[Section~4]{EH19}, there is an integration operator $\CP^{i;\eps}_{\gamma}$ for $i\in\{1,2\}$ on $\cD^{\gamma;i}_\eps$, such that 
by \cite[Remarks~4.13-4.14]{EH19}, one has 
$\CR^\e \CP_\gamma^{i;\e} f = P^{i;\e} *_{(i)} \CR^\e f$,
and the multi-level Schauder estimate \cite[Theorems~4.9,~4.26]{EH19} hold.
Namely, for $\gamma\in (0,1/2)$, $\alpha\in(-2,0]$ and $\eta>-2$,
\begin{equ}\label{eq:conv_K}
\CP^{i;\e}  \eqdef \CP^{i;\e}_\gamma 
\colon \cD^{\gamma,\eta;i}_{\alpha,\e} \to \cD^{\gamma+3/2,(\eta\wedge\alpha)+2}_{(\alpha+2)\wedge 0,\e}
\end{equ}
is a continuous operator \label{CP_page_ref}
with continuity properties in the discrete model.
%
%\begin{remark}
The restriction $\gamma<1/2$ and the gain in regularity $\gamma \to \gamma+\frac32$  rather than the usual $\gamma \to \gamma+2$ arises from our choice to only keep degree $0$ and $1$ polynomials in $\CT$.
%Indeed, this choice
%implies that $T_\zeta^{i;\eps}$ never produces polynomials of degree $2$ and higher, and thus $\CK^{i,\eps}_\gamma$ cannot in general output a function in $\cD^\gamma$ with $\gamma\geq 2$.
%\end{remark}

\subsubsection{Differentiation}
\label{sec:Differentiation}

Let $\gamma\in (1,2)$.
Fix  an admissible (discrete) model $Z=(\Pi^\e,\Gamma^\e)$ which is compatible with differentiations (Definition~\ref{def:compat_diff}).
Recall from Remark~\ref{rem:admissible_Holder} that
admissibility implies H\"older continuity of $\Pi^\e_x\CI_i\tau$ , which,
by compatibility with differentiations, further implies H\"older continuity (in time) of $\Pi^\e_x\CD\CI_i\tau$
% for $\CD\in\Diff$
for $\CD\CI_i\tau\in\CT$.

For $f\in \cD^\gamma$, we 
define  (nonlocal) differentiation operators $\mcb{D}_i^\pm$, $\mcb{D}_i$ and $\bar{\mcb{D}}_i$
by 
\[
\mcb{D}_i^\pm f \eqdef 
\CD_i^\pm f  +
 H_i^\pm f \bone \;,
 \qquad
\mcb{D}_i f \eqdef 
\CD_i f  +
 H_i f \bone \;,
\qquad
\bar{\mcb{D}}_i f \eqdef 
\bar\CD_i f  +
 \bar H_i f \bone \;,
\]
where we suppose that $f$ takes values in the domain of the respective abstract derivative
and, for $z\in \R\times \T^2_\e$,
%\begin{equ}\label{eq:H_Diff_def}
$H_i f (z)
\eqdef 
\partial_i
\left( \CR^\e f
 - \Pi_z^\e  f(z) \right) (z)
$
%\end{equ}
and is extended to a function on $ \R\times \T^2$ by bilinear interpolation.
(We let $\T^2_\e$, $\bone$ denote in this subsection either $\obonds_1$, $\bone^{(1)}$ or $\obonds_2$, $\bone^{(2)}$ respectively depending on whether we look at $f^{(1)}$ or $f^{(2)}$ in the decomposition %$f=(f^{(1)},f^{(2)})$ 
arising from %$\CT=\CT^{(1)}\oplus\CT^{(2)})$
\eqref{e:T1T2}).
$H_i^{\pm}$ and $\bar H_i$ are defined in the same way 
with $\partial_i$ replaced by $\partial_i^\pm$ and $\bar\partial_i$ respectively.
Note that
%~\eqref{eq:H_Diff_def}
$H_if$ is well-defined since $\Pi_z^\e f(z)$ is continuous and~\eqref{eq:rec} holds.

\begin{remark}\label{rem:D-nonlocal}
 \cite[Sec.~5.3]{EH19} defines $\mcb{D}$ to be the local operator $\CD$, which then does not exactly commute with $\CR^\e$ (see Remark~5.10 therein). 
 Here %our definition of these differentiation operators is a modification  
we modify the definition  by incorporating the discrepancy
to ensure 
\eqref{e:RDDR} which will be convenient to derive the renormalised equation.
%(so that the reconstruction would not give  the desired SPDE ``up to an error''). 
(See~\cite[Remark~4.13]{EH19} for related ideas of modifying integration operators.)
We only focus on the case $\gamma\in (1,2)$
(so only $\bone$ appears above)
 which is sufficient for our purpose.
\end{remark}
Let $\cD_{\<IXi>,\e}^{\gamma,\eta}$ denote the subspace of all modelled distributions $f\in\cD^{\gamma,\eta}_\e$
of the form
\begin{equ}\label{eq:A_expansion}
f = f_{\<IXi>} + f_\bone + \sum\nolimits_{0<|\sigma|<1} f_{\sigma} + f_\X\;,
\end{equ}
where $\sigma$ refers to a graphical tree (so the sum over $\sigma$ contains 5 terms given
by~\eqref{eq:pos_trees-YM} and~\eqref{eq:pos_trees_rem})
and $f_{\sigma}\colon \{\R\times \T^2\}\setminus P \to \CT[\sigma]$
(we follow the conventions from Section~\ref{subsubsec:mult_eps_models})
and where $f_{\<IXi>}$ is constant on each of the two connected components of $\{\R\times \T^2\}\setminus P$.
Here $\CT[\bone]\simeq \R^2$ and $\CT[\X]\simeq \R^4$ denote the 
polynomial subspaces spanned by $\bone^{(i)}$ and $\X^{(i)}_j$, $i,j\in\{1,2\}$, respectively.
Note that $|\sigma|$ is well-defined as $|\tau|$ is constant for all $\tau\in\sigma$.

\begin{lemma}\label{lem:DRcommute}
Let $\gamma \in (1,2-4\kappa)$, $\eta\in\R$, and
$f\in \cD^{\gamma,\eta}_{\<IXi>,\e}$.
Let $\K = I\times\T^2$ where $I\subset\R$ is compact.
Then $\mcb{D}_i f \in \cD^{\gamma-1,\eta-1}_{-1-\kappa,\e}$ for $i\in\{1,2\}$ with $\$\mcb{D}_i f\$^{(\e)}_{\gamma-1,\eta-1;\K} \lesssim \$f\$_{\gamma,\eta;\K}^{(\e)}\|\Gamma^\e\|^{(\e)}_{\gamma;\K}$, 
and likewise for $\mcb{D}_i^\pm f ,\bar{\mcb{D}}_if$, and
\begin{equ}[e:RDDR]
\CR^\e \mcb{D}_i^\pm f
= \partial_i^\pm \CR^\e f \;,
\qquad
\CR^\e \mcb{D}_i f
= \partial_i \CR^\e f \;,
\quad
\CR^\e \bar{\mcb{D}}_i f 
= \bar\partial_i \CR^\e f\;,
\end{equ}
provided that $f$ takes values in the domain of the respective abstract derivative.
If $\bar Z=(\bar\Pi^\eps,\bar\Gamma^\eps)$ is another model that is compatible with differentiations and $\bar f\in \cD^{\gamma,\eta}_{\bar\Gamma^\e,\<IXi>,\e}$ such that $\mcb{D}_i\bar f$ is well-defined, then
\begin{equ}\label{eq:Df_diff_bound}
\$\mcb{D}_if;\mcb{D}_i\bar f\$_{\gamma-1,\eta-1;\K}^{(\e)} \lesssim \| \Gamma^\e\|^{(\e)}_{\gamma;\K}\$f;\bar f\$_{\gamma,\eta;\K}^{(\e)} + 
\|\Gamma^\e-\bar \Gamma^\e\|^{(\e)}_{\gamma;\K} \$\bar f\$_{\gamma,\eta;\K}^{(\e)} \;,
\end{equ}
%where $\bar\K$ is the $1$-fattening of $\K$,
and likewise for $\mcb{D}_i^\pm,\bar{\mcb{D}}_i$.
The proportionality constants do not depend on $Z,\bar Z,f,\bar f$.
\end{lemma}

\begin{remark}
We believe the condition $f\in\cD^{\gamma,\eta}_{\<IXi>,\e}$ in Lemma~\ref{lem:DRcommute} is not fundamental,
but it suffices for our purposes and allows for an elementary proof without further assumptions.
Moreover, the space $\cD^{\gamma,\eta}_{\<IXi>,\e}$ plays an important role in Section~\ref{subsubsec:multiplication_by_eps}.
\end{remark}

\begin{proof}
We prove the lemma for $\mcb{D}_i$ and it follows in the same way for the other differentiation operators.
We also prove without $\eta$ and the general case with the ``$t=0$''-subspace follows analogously.
We drop reference to $\eps$ for simplicity.

Using the compatibility of the model with $\CD_i$ and the fact that~\eqref{eq:rec} holds for $f\in \cD^\gamma(V_{\CD_i})$ and for $\CD_i f$, we have for
$z\in \R\times \T^2_\e$
\begin{equ}
\CR \CD_i f (z) 
 = \Pi_z (\CD_i(f(z)))(z) 
 = \partial_i \left(\Pi_z (f(z))\right)(z) \;.
\end{equ}
Hence $\CR \CD_i f (z) + H_i f (z)$ is precisely $\partial_i \CR f(z)$, which proves \eqref{e:RDDR}.

To prove 
$\mcb{D}_i f \in \cD^{\gamma-1}_{\alpha-1}$,
recall the norm $\$ \cdot \$^{(\eps)}_{\gamma;\K} $ on $\cD^\gamma_\e$ defined below \eqref{e:Dgamma}.
By these definitions and %the fact that every $\Gamma\in\CG$ commutes with $\CD_i$ by
\eqref{eq:Gamma_Diff_commute},
it is clear that the ``local term'' $\CD_i f \in \cD^{\gamma-1}_{\alpha-1}$ and that we have~\eqref{eq:Df_diff_bound} with $\mcb{D}_i$ replaced by $\CD_i$.

It remains to bound $\$ H_i f\bone \$_{\gamma-1;\K} $ and $\$ H_i f\bone ;H_i\bar f\bone\$_{\gamma-1;\K} $.
For $\tau\in\CT$, we write $\scal{\tau,\bone}$ for the projection of $\tau$ to $\CT[\bone]$ under the decomposition $\CT=\bigoplus_{\sigma\in\mfT}\CT[\sigma]$.
By our choice of $\K$, note that if $z\in(\R\times \T^2_\e)\cap\K$, then also $z\pm \e_i \in (\R\times \T^2_\e)\cap\K$.
Therefore, by~\eqref{eq:rec},
one has
\begin{equs}[eq:Rf_Pi_f_grad]
|\left( \CR f
- \Pi_z  f(z) \right)   (z + \e_i)|
&=
|\left( \Pi_{z+\e_i}\{ f(z+\e_i) - \Gamma_{z+ \e_i,z} f(z)\}\right)(z+ \e_i)|
\\
&= |\scal{f(z+\e_i) - \Gamma_{z+ \e_i,z} f(z),\bone }|
\leq \e^\gamma\$ f\$_{\gamma;\K}
\end{equs}
and likewise for $z-\e_i$,
where we used Lemma~\ref{lem:Gamma_simple} and the fact that $f_{\<IXi>}$ is constant.
Therefore 
\begin{equ}[e:dR-Pi-bound]
|H_i f (z)|=
\left|
\partial_i
\left( \CR f
 - \Pi_z  f(z) \right) (z)   
 \right|
 \leq  \e^{\gamma-1} \$ f\$_{\gamma;\K}\;,
\end{equ}
and thus
\begin{equs}
\sup_{\substack{(y,z)\in  (\R\times \T^2_\e)\cap \K\\
\eps \leq \|y-z\|_\s\leq 1}}
\frac{\|H_i f(z)\bone- H_i f(y)\bone\|_0}{\|y-z\|_\s^{\gamma-1}}
\lesssim
\sup_{\substack{(y,z)\in  (\R\times \T^2_\e)\cap \K\\
\eps \leq \|y-z\|_\s\leq 1}}
\frac{\e^{\gamma-1}\$ f\$_{\gamma;\K}
}{\|y-z\|_\s^{\gamma-1}}
\leq
\$ f\$_{\gamma;\K}\;.
\end{equs}
We then bound 
 $\$ H_i f \bone\$_{\gamma-1;\K;\eps} $ given by \eqref{e:smallscale-norm},
 which amounts to proving that 
$ H_i f $ is $(\gamma-1)$-H\"older at scales smaller than $\e$.
For this time H\"older regularity, write $z+t\eqdef (s+t,x)\in\K\cap(\R\times\T^2_\e)$ for $z=(s,x)\in\K\cap(\R\times\T^2_\e)$ and $t\in\R$.
By a similar argument as in~\eqref{eq:Rf_Pi_f_grad},
the desired bound follows once we show, for $t\leq \eps^2$,
\begin{equs}{}
&|\scal{\Delta_1-\Delta_2,\bone}|\lesssim \e t^{(\gamma-1)/2}\$f\$_{\gamma;\K}\|\Gamma\|_{\gamma;\K}\;,\label{eq:Deltas_diff}
\\
&\Delta_1\eqdef f(z+\e_i) - \Gamma_{z+\e_i,z}f(z)\;,\quad \Delta_2 \eqdef f(z+\e_i+t) - \Gamma_{z+\e_i+t,z+t}f(z+t) \;.
\end{equs}
Remark that $\|\Delta_2\|_\beta \leq \e^{\gamma-\beta}\$f\$_{\gamma;\K}$ for all $\beta<\gamma$
and thus
\begin{equs}
|\scal{\Delta_2-\Gamma_{z+\e_i,z+\e_i+t}\Delta_2,\bone}|
&\leq \sum_{0<\beta<\gamma} \e^{\gamma-\beta}\$f\$_{\gamma;\K} t^{\beta/2} \|\Gamma\|_{\gamma;\K}
\\
&\lesssim
\e t^{(\gamma-1)/2}\$f\$_{\gamma;\K}\|\Gamma\|_{\gamma;\K}\;,
\end{equs}
where we used that $\beta > 1-4\kappa$ in the sum and $\gamma<2-4\kappa$.
On the other hand
$
|\scal{f(z+\e_i)-\Gamma_{z+\e_i,z+\e_i+t}f(z+\e_i+t),\bone}| \leq t^\gamma\$f\$_{\gamma;\K}
$
and
\begin{equs}{}
&|\scal{\Gamma_{z+\e_i,z}f(z) - \Gamma_{z+\e_i,z+\e_i+t}\Gamma_{z+\e_i+t,z+t}f(z+t),\bone}|
\\
&\qquad =
|\scal{\Gamma_{z+\e_i,z}(f(z) - \Gamma_{z,z+t}f(z+t)),\bone}|
\\
&\qquad \leq 
\sum_{\beta<\gamma} \e^\beta \|\Gamma\|_{\gamma;\K}
 t^{(\gamma-\beta)/2}\$f\$_{\gamma;\K}
\lesssim \e t^{(\gamma-1)/2}\|\Gamma\|_{\gamma;\K}
\$f\$_{\gamma;\K}
\end{equs}
where we used that $\beta\leq 1$ in the sum.
Combining the above estimates proves~\eqref{eq:Deltas_diff}.
For this H\"older regularity in space, recalling that we have extended   $ H_i f $ with bilinear interpolation,
for $z,y\in\K$ with $\|z-y\|_\s \leq \e$ and identical time coordinate, one can find $\bar z,\bar y \in (\R\times \T^2_\e)\cap \K$ with
$\|\bar y-\bar z\|_\s$ being a positive multiple of $\e$
such that
\[
\frac{|H_i f(z)- H_i f(y)|}{\|y-z\|_\s^{\gamma-1}}
\lesssim
\frac{|H_i f(\bar z)- H_i f(\bar y)|}{\|y-z\|_\s^{\gamma-1}} \frac{\|y-z\|_\s}{\|\bar y-\bar z\|_\s}
\lesssim
\|y-z\|_\s^{2-\gamma} \e^{\gamma-2} \leq 1
\]
where we used \eqref{e:dR-Pi-bound} and $\gamma<2$ and the second proportionality constant is a multiple of $\$ f\$_{\gamma;\K}$. 
The bound~\eqref{eq:Df_diff_bound} for $\$H_i f\bone;H_i\bar f\bone\$_{\gamma;\K}$ follows similarly.
\end{proof}

\subsubsection{Shifting operator}

Fix %throughout this subsection 
a model $Z=(\Pi^\eps,\Gamma^\eps)$ that is compatible with shifts.
%We define shifting operators $\cS_\be$  on the space $\cD^{\gamma,\eta}_\e$ as follows.
Recall that $\CS_{\be}$ is defined on a sector $V_\be$ for $\be\in  \CE_\times$ (where $\CE_\times$ is defined in \eqref{e:def-CEtimes}).
For $f\in\cD^{\gamma,\eta}_\e(V_\be)$
we define the shifting operator
\begin{equ} [e:shift-Dga]
(\cS_{\be} f) (x)\eqdef \CS_{\be} (f(x+\be)) \;, \qquad \forall x\in \R^{3}\;.
\end{equ}
We will also often write  $\cS_{\northwest}$,  $\cS_{\northeast}$,  $\cS_{\southwest}$,  $\cS_{\southeast}$ understood in the obvious way.

\begin{lemma} \label{lem:sf-is-Lip}
For $\be\in\CE_\times$, another model $\bar Z=(\bar\Pi^\eps,\bar\Gamma^\eps)$ that is compatible with shifts,
and $f\in \cD^{\gamma,\eta}_{\Gamma^\e,\e}(V_\be)$ and $\bar f\in \cD^{\gamma,\eta}_{\bar\Gamma^\e,\e}(V_\be)$,
one has
\begin{equ}
\$\cS_{\be}f\$_{\gamma,\eta;\K}^{(\eps)} = \$f\$_{\gamma,\eta;\K+\be}^{(\eps)}\;,
\quad 
\$\cS_\be f;\cS_\be\bar f\$_{\gamma,\eta;\K}^{(\eps)} = \$ f; \bar f\$_{\gamma,\eta;\K+\be}^{(\eps)}\;.
\end{equ}
\end{lemma}

\begin{proof}
For all $z,y\in \R^{3}$, by \eqref{e:Gamma-sf} and the definition \eqref{e:shift-Dga},
\begin{equs}
\|\cS_{\be}f(z) & -\Gamma^{\eps}_{zy} \cS_{\be}f(y)\|_{\beta}
=
\|\CS_{\be}f(z+\be) -\Gamma^{\eps}_{zy} \CS_{\be}f(y+\be)\|_{\beta}
\\
&=
\|\CS_{\be}f(z+\be) - \CS_{\be} \Gamma^{\eps}_{z+\be,y+\be} f(y+\be)\|_{\beta}
\;,
\end{equs}
and since $\CS_{\be}$
preserves the norm on $\CT_\beta$,
it follows that all the terms for $\cS_\be f$ in~\eqref{e:Dgamma} are equal to the corresponding terms for $f$ with the change of variable $x\mapsto x+\be$,
and likewise for $\$\cS_\be f;\cS_\be\bar f\$_{\gamma,\eta;\K}^{(\eps)}$.
\end{proof}
The following `average' of shiftings will be useful.
We
define\footnote{Here, at $x\in \R^3$, for each $\be\in\CE_\times$,  we shift the tree by $\CS_{\be}$ and {\it also} shift $x$ by $\be$ (see \eqref{e:shift-Dga}).
This is the reason why we did not define $\frac14 \sum_{\be\in \CE_\times} \CS_{\be}$ as  ``one single'' operator on the level of regularity structures (which one might think is convenient).}
\begin{equ}[e:shift-a]
\cS_{a} f \eqdef \frac14 \sum_{\be\in \CE_\times} \cS_{\be} f\;,
\qquad
\forall f\in \cap_{\be\in\CE_\times}\cD^{\gamma,\eta}_\e(V_\be)\;.
\end{equ}
%In the next lemma we assume that the reconstruction of the given modelled distribution yields continuous functions, which is sufficient for our purpose.

\begin{lemma}\label{lem:RecShft}
Let $\be\in\CE_\times$, $j\in\{1,2\}$, and $\CA_j^\e \in \cD^{\gamma,\eta;j}_\e(V_\be)$ and suppose that~\eqref{eq:rec} holds for $\CA^\e_j$.
Then for $i\neq j$ and any $e\in \obonds_i$, one has
$(\CR^\e \cS_{\be} \CA_j^\e  )(e)=A_j^{\eps} (e+\be)$ 
where $A_j^\e \eqdef \CR^\e \CA_j^\e$.
If furthermore $\CA_j^\e \in\cap_{\be\in\CE_\times} \cD^{\gamma,\eta;j}_\e(V_\be)$, then
$(\CR^\e \cS_{a} \CA_j^\e  )(e)=A_j^{\eps} (e^{(a)})$
where we recall the notation $A_j^{\eps} (e^{(a)})$ from~\eqref{e:def-Aea}.
\end{lemma}

\begin{proof}
Fixing $\be\in \CE_\times$, one has
\begin{equs}
(\CR^\e \cS_{\be} \CA_j^\e  )(e)
&
 \stackrel{\eqref{eq:rec}}{=}
 \Pi_{e}  \left(  (\cS_{\be} \CA_j^\e)(e)\right) (e)
 \stackrel{\eqref{e:shift-Dga}}{=}
 \Pi_{e}  \left( \CS_{\be} (\CA_j^\e(e+\be))  \right) (e)
 \\
&\stackrel{\eqref{e:Pi-sf}}{=}
 \Pi_{e+\be}  \left( \CA_j^\e(e+\be)  \right) (e+\be) 
  \stackrel{\eqref{eq:rec}}{=}
A_j^{\eps} (e+\be) \;.
\end{equs}
From this and \eqref{e:shift-a} one also obtains the identity for $\cS_{a}$.
%\begin{equs}
%(\CR^\e \cS_{a} \CA_j^\e  )(e)
%=\frac14 \sum_{\be \in \CE_\times} 
%	(\CR^\e \cS_{\be} \CA_j^\e  )(e)
% =
% \frac14 \sum_{\be \in \CE_\times} 
% A_j^{\eps} (e+\be)=
%A_j^{\eps} (e^{(a)})\;.
%\end{equs}
\end{proof}

\subsubsection{Multiplication by \texorpdfstring{$\eps$}{epsilon}}
\label{subsubsec:multiplication_by_eps}

To handle the terms $\e F(\e A)[A,\xi]$ in Proposition~\ref{prop:rescaled-equ},
we represent $\eps A$ as a function-like modelled distribution which we can compose with the function $F$.

We fix a model $Z=(\Pi,\Gamma)$ throughout this subsection that we suppose is compatible with multiplication by $\eps$ (Definition~\ref{def:compat}).
Recall the space $\cD^{\gamma,\eta}_{\<IXi>,\e}$ from Section~\ref{sec:Differentiation}.
We define a linear map $\cE \colon \cD^{\gamma,\eta}_{\<IXi>,\e} \to \cD^{\gamma,\eta}_{0,\e}$ with the property that $\CR\cE f = \eps \CR f$.
Recalling~\eqref{eq:pos_trees-YM} and~\eqref{eq:pos_trees_rem}, there are 5 graphical trees in the sum over $\sigma$ in~\eqref{eq:A_expansion} that we group into the two sets
\begin{equ}
T_1 = \{\<IXibar>\;,\;\;
\<IPsiXi1>\}\;,\quad
T_2 = \{\<I(cherry)>\;,\;\;
\<IDPsi>\;,\;\;
\<I(R2-1new)>\}
\end{equ}
(note the relation between $T_2$ and $\mfT_\eps$ from~\eqref{eq:def_mfT_eps}).
For $f$ as in~\eqref{eq:A_expansion},
we define
\begin{equs}[eq:B_expand]
\cE f &\eqdef b_{\bone} + b_{\<IXibar>} + b_{\<IPsiXi1>}  + b_{\X}\;,
\\
b_\bone(x) &= \eps f_\bone(x) + \eps [\Pi_x f_{\<IXi>} (x)](x) \in \CT[\bone]\;,
\\
b_{\<IXibar>}(x) &= \e^\kappa [f_{\<IXi>}(x)]_\e + \e f_{\<IXibar>}(x)\;,
\quad
b_{\<IPsiXi1>}(x) = \e f_{\<IPsiXi1>}(x)\;,\quad
b_\X(x) = \eps f_\X(x)\;,
\end{equs}
where, in the definition of $b_\bone(x)$, $\eps [\Pi_x f_{\<IXi>} (x)](x)$ is
understood as an element of $\CT[\bone]$ in the same way as described after~\eqref{eq:IXi_compat},
and, in the definition of $b_{\<IXibar>}(x)$, we recall the linear isomorphism
$\tau\mapsto \tau_\eps, \CT[\<IXi>] \to \CT[\<IXibar>]$ from Section~\ref{subsubsec:mult_eps_models}.
Note that $b_\bone(x)$ is well-defined since $\Pi_x\tau$ is continuous for all $\tau\in\<IXi>$ (Remark~\ref{rem:admissible_Holder}).

\begin{remark}\label{rem:cE_motivation}
Trees in $T_1$ are kept in the definition of $\cE f$ in~\eqref{eq:B_expand}
since the map $\tau\mapsto \tau_\eps$ is not defined on these trees;
for $\<IXibar>$, this is because multiplication by $\eps^{1-\kappa}$ does not increase the small scale regularity of $\eps^{1-\kappa}\xi$ (Remark~\ref{rem:F_trees}),
while for $\<IPsiXi1>$, this is because multiplication by $\eps^{1-\kappa}$ does not commute with $\Pi$, i.e.
$\Pi_x \<I(PsiXi1)Xi1less> \neq \eps^{1-\kappa} \Pi_x \<PsiXi1>$.

The reason that the trees in $T_2$ are dropped in~\eqref{eq:B_expand} is that multiplication by $\eps^{1-\kappa}$ 
on these trees \textit{does} increase degree by $1-\kappa$ (and we can indeed prove the corresponding analytical bounds)
and commutes with $\Pi$ (see~\eqref{eq:Gamma_one_comp} below where this is used).
%This is not true for $\<Xi1>$ because $\eps^{2-\kappa}\xi$ in our setting
%essentially has the same small scale regularity as $\eps^{1-\kappa}\xi$ (i.e. that of white-in-time noise), so it is not sensible to define a new symbol representing $\eps^{2-\kappa}\xi$ with degree $-2\kappa$; this is a manifestation of the fact that multiplication by $\eps$ increases spatial but not temporal regularity.
Note that we could have alternatively kept \textit{all} trees $\sigma$ for $0<|\sigma|<1$ in~\eqref{eq:B_expand} by defining $b_\sigma = \eps f_\sigma$ for $\sigma\in T_2$,
but this would require enlarging our regularity structure by including the negative trees $\<r_z_avoid1>$, $\<r_z_avoid2>$, and $\<r_z_avoid3>$
that would belong to the groups
\hyperlink{(1-3)}{(1-3)}, \hyperlink{(1-3)}{(1-3)}, and \hyperlink{(2-3)}{(2-3)} respectively in table~\eqref{e:trees-deg-0-rem}.
\end{remark}

\begin{remark}
For a `multiplication by $\eps$' map on modelled distributions with similarities to ours, see~\cite{KPZJeremy}.
Note that Lemma~\ref{lem:B_bound_by_A}
is stronger but more specialised than the analogous result~\cite[Prop.~3.16]{KPZJeremy}.
We believe our construction can be substantially generalised but refrain from doing so here.
\end{remark}

\begin{lemma}\label{lem:B_bound_by_A}
Consider $\gamma\in (1,2-4\kappa)$, $\eta\in\R$, a compact set $\K\subset\R\times\T^2$, and $f\in\cD^{\gamma,\eta}_{\<IXi>,\e}$.
Then
\begin{equ}\label{eq:B_bound_A}
\$\cE f\$_{\gamma,\eta;\K}^{(\eps)} \lesssim \eps^\kappa \$ f\$_{\gamma,\eta;\K}^{(\eps)}(1+\|\Pi\|^{(\eps)}_{\gamma;\bar \K})\;,
\end{equ}
where $\bar\K$ is the $1$-fattening of $\K$.
Furthermore, if $\bar Z = (\bar\Pi,\bar\Gamma)$ is another model that is compatible with multiplication by $\eps$ and $\bar f\in\cD^{\gamma,\eta}_{\bar\Gamma,\<IXi>,\e}$, then
\begin{equ}\label{eq:B_bar_B_bound}
\$\cE f;\cE\bar f\$_{\gamma,\eta;\K}^{(\eps)} \lesssim 
\e^\kappa\| f;\bar f\|_{\gamma,\eta;\K}^{(\eps)}(1+\|\Pi\|^{(\eps)}_{\gamma;\bar \K})
+ \e^\kappa \|\bar f\|_{\gamma,\eta;\K}^{(\e)}\|\Pi-\bar\Pi\|_{\gamma;\bar\K}^{(\e)}\;.
\end{equ}
The proportionality constants do not depend on $Z,\bar Z,  f$ or $\bar  f$.
\end{lemma}

\begin{proof}
To prove~\eqref{eq:B_bound_A}, dropping reference to $\e$ in the norms,
the first terms in~\eqref{e:Dgamma} and~\eqref{e:smallscale-norm} for $\|\cE f\|_{\gamma,\eta;\K}$ are bounded by the corresponding terms for $\| f\|_{\gamma,\eta;\K}$.
We handle the second terms in~\eqref{e:Dgamma} and~\eqref{e:smallscale-norm} simultaneously.

Consider $(x,y)\in \K_P$.
Then $\cE f(x) - \Gamma_{xy} \cE f(y)$ has non-zero components only of degrees $|\bone|=0$, $|\<IXibar>|=1-2\kappa$, $|\<IPsiXi1>|=1-3\kappa$, and $|\X|=1$.
Clearly
\begin{equs}
\|\cE f(x) - \Gamma_{xy} \cE f(y)\|_{1-2\kappa}
&\leq |b_{\<IXibar>}(x) - b_{\<IXibar>}(y)|
\\
&\leq \eps|f_{\<IXibar>}(x) - f_{\<IXibar>}(y)|
+ \eps^\kappa|f_{\<IXi>}(x)-f_{\<IXi>}(y)|
\\
&\leq
\eps\| f(x)-\Gamma_{xy} f(y)\|_{1-2\kappa}
%\\
%&\qquad+ \eps^\kappa\| f(x)-\Gamma_{xy} f(y)\|_{-\kappa}
\\
&\leq
\eps\|x-y\|_\s^{\gamma-1+2\kappa}
%+ \eps^\kappa\|x-y\|_\s^\gamma)
%\\
%&\qquad
\|x,y\|^{\eta-\gamma}_P\|f\|_{\gamma,\eta;\K}\;,
\end{equs}
where in the third inequality we used $f_{\<IXi>}(x)=f_{\<IXi>}(y)$ since $f\in\cD^{\gamma,\eta}_{\<IXi>,\e}$
and Lemma~\ref{lem:Gamma_simple} to see that, in $f$, no terms except $f_{\<IXi>}$ contribute to $\|\Gamma_{xy} f(y)\|_{-\kappa}$
and no terms except $f_{\<IXibar>}$
contribute to the subspace $\CT[\<IXibar>]$.
The bounds at degrees $1$ and $1-3\kappa$ follow in a similar and simpler way.
%Furthermore
%\begin{equ}
%\|\cE f(x) - \Gamma_{xy} \cE f(y)\|_1 \leq |b_\X(x) - b_\X (y)| 
%%= \eps |\nabla a(x) - \nabla a(y)|
%\leq \eps \|x-y\|_\s^{\gamma-1}\|x,y\|_P^{\eta-\gamma}\| f\|_{\gamma,\eta;\K}\;,
%\end{equ}
%and
%\begin{equs}
%\|\cE f(x) - \Gamma_{xy} \cE f(y)\|_{1-3\kappa}
%&\leq |b_{\<IPsiXi1>}(x) - b_{\<IPsiXi1>}(y)|
%= \eps|f_{\<IPsiXi1>}(x) - f_{\<IPsiXi1>}(y)|
%\\
%&\leq
%\eps\| f(x)-\Gamma_{xy} f(y)\|_{1-3\kappa}
%\\
%&\leq
%\eps\|x-y\|_\s^{\gamma-1+3\kappa}\|x,y\|_P^{\eta-\gamma}
%\|f\|_{\gamma,\eta;\K}\;,
%\end{equs}
%where we again used Lemma~\ref{lem:Gamma_simple} in the third bound as before.

It remains to consider the component in $\CT[\bone]$. For $\tau\in\CT$, we write $\scal{\tau,\bone}$ for the projection of $\tau$ to $\CT[\bone]$ under the decomposition $\CT=\bigoplus_{\sigma\in\mfT}\CT[\sigma]$.
Then
\begin{equs}[eq:CB_zero]
\scal{\cE f(x) - \Gamma_{xy} \cE f(y),\bone} &= b_\bone(x) - b_\bone(y) -\scal{\Gamma_{xy}b_{\<IXibar>}(y) +  \Gamma_{xy}b_{\<IPsiXi1>}(y) ,\bone}
\\
&\quad - \scal{b_\X(y), x-y}
\\
&= \scal{\e f(x) - \e\Gamma_{xy} f(y) + \eps \Gamma_{xy}\sum_{\sigma\in T_2}f_\sigma(y),\bone}\;,
\end{equs}
where the final equality follows from~\eqref{eq:IXi_compat},
the definition of $b_\sigma$ %in~\eqref{eq:b_def},
in~\eqref{eq:B_expand},
and the fact that $f_{\<IXi>}(x)=f_{\<IXi>}(y)$ since $f\in\cD^{\gamma,\eta}_{\<IXi>,\e}$.
Next, for all $\sigma \in T_2$ and $\CI[\tau]\in\CT[\sigma]$ with $\CI\in\{\CI_1,\CI_2\}$,
it follows from  compatibility with multiplication by $\eps$ of $(\Pi,\Gamma)$ that
\begin{equs}[eq:Gamma_one_comp]
\scal{\eps\Gamma_{xy} \CI[\tau],\bone}
&= \eps[(K*\Pi \tau)(x) - (K*\Pi \tau)(y)]
\\
&= \eps^\kappa [(K*\Pi \tau_\eps)(x) - (K*\Pi \tau_\eps)(y)]\;,
\end{equs}
where $\tau\mapsto \tau_\eps$ is the mapping $\bigoplus_{\sigma\in \mfT_\eps}\CT[\sigma] \to \CT_-$ from Section~\ref{subsubsec:mult_eps_models}.
%which is pictorially given by
%\begin{equ}
%\<IXiI'Xi_notriangle> \mapsto \<IXiI'Xibar>\;,\quad
%\<I'Xi_notriangle> \mapsto \<I'Xibar>\;,\quad
%\<R2-1new> \mapsto \<cherry232>\;.
%\end{equ}
In~\eqref{eq:Gamma_one_comp} we drop the index $x$ in $\Pi_x$ since $\Pi_x\tau = \Pi_y\tau$ and $\Pi_x\tau_\eps = \Pi_y\tau_\eps$ for all $\tau\in\bigoplus_{\sigma\in \mfT_\eps}\CT[\sigma]$.
We now remark that, for all $\tau\in\bigoplus_{\sigma\in \mfT_\eps}\CT[\sigma]$, the 1st line in~\eqref{eq:Gamma_one_comp}
and the fact that $|\sigma|\geq -1-3\kappa$ for all $\sigma\in\mfT_\e$ implies
\begin{equ}
\scal{\eps\Gamma_{xy} \mcb{I}[\tau],\bone}\lesssim \eps\|x-y\|_\s^{1-3\kappa}\|\Pi\|_{\gamma;\bar\K} \;.
\end{equ}
On the other hand, if $\|x-y\|_\s < \eps$ (which means in particular that $x,y$ differ only in their time coordinates) then the 2nd line in~\eqref{eq:Gamma_one_comp}
and the fact that $|\sigma_\e|\geq -4\kappa$ for all $\sigma\in\mfT_\e$
implies \begin{equ}[eq:eGamma_bound]
\scal{\eps\Gamma_{xy} \mcb{I}[\tau],\bone}\lesssim \eps^\kappa \|x-y\|_\s^{2-4\kappa}\|\Pi\|_{\gamma;\bar\K}\;.
\end{equ}
%(This final part uses already existing assumptions on the model, namely that $\Pi\<IXiI'Xibar>$ has the correct bounds at order $-3\kappa$,
%which itself exploits that the irregularity of $\<IXiI'Xi_notriangle>$ is primarily spatial).
Consequently, for all $\tau\in\bigoplus_{\sigma\in \mfT_\eps}\CT[\sigma]$ and $(x,y)\in\K_P$, we have~\eqref{eq:eGamma_bound}.
%\begin{equ}
%\scal{\eps\Gamma_{xy} \mcb{I}[\tau],\bone} \lesssim \eps^\kappa\|x-y\|_\s^{2-4\kappa}\|\Pi\|_{\gamma;\bar\K}\;.
%\end{equ}
Now \eqref{eq:CB_zero} implies the following bound which concludes the proof of~\eqref{eq:B_bound_A}:
\begin{multline*}
\scal{\cE f(x) - \Gamma_{xy} \cE f(y),\bone}
\\
\lesssim
\e\|x-y\|_\s^\gamma\|x,y\|_P^{\eta-\gamma}\| f\|_{\gamma,\eta;\K}
+\e^\kappa\|x-y\|_\s^{2-4\kappa}\| f\|_{\gamma,\eta;\K}\|\Pi\|_{\gamma;\bar\K}\;.
\end{multline*}

The proof of~\eqref{eq:B_bar_B_bound} is almost identical except the final term in~\eqref{eq:CB_zero} becomes
$
\scal{\eps \Gamma_{xy}\sum_{\sigma\in T_2}f_\sigma(y) - \eps \bar\Gamma_{xy}\sum_{\sigma\in T_2}\bar f_\sigma(y),\bone}
$,
which is bounded above by a multiple of
$
\e^\kappa\|x-y\|^{2-4\kappa}_\s(\| f;\bar f\|_{\gamma,\eta;\K}\|\Pi\|_{\gamma;\bar\K} + \|\bar f\|_{\gamma,\eta;\K}\|\Pi-\bar\Pi\|_{\gamma;\K})
$.
\end{proof}

\begin{lemma}\label{lem:Rec_Mult_eps}
Suppose we are in the setting of Lemma~\ref{lem:B_bound_by_A}. Then $\CR\cE f = \eps \CR f$.
\end{lemma}

\begin{proof}
Obvious from the definition and the identities $\CR  f(x) = [\Pi_x f_{\<IXi>}(x)](x) + f_\bone(x)$ and $\CR \cE f(x) = b_\bone(x)$.
\end{proof}

\subsection{Renormalisation group}
\label{subsec:renorm_group}

%Unfortunately, the canonical models $(\Pi^\e,\Gamma^\e)$ from Section~\ref{sec:canonical} are not bounded uniformly in $\eps>0$.
We now define the renormalisation maps $M^\e \in L(\CT,\CT)$
that we later apply to $(\Pi^\e,\Gamma^\e)$ to obtain models that are bounded uniformly in $\e>0$.
These maps are invertible and form an Abelian group
isomorphic to $(\mfg^{\otimes_s 2})^{\oplus 4}\oplus(\mfg^{\otimes 2})^{\oplus 50}\oplus\mfg^{\otimes 4}$,
elements of which have 55 components denoted by 
%symmetric tensors: 4
%general 2-tensors: 6+4*2+4*2+1+1+4+1+1+2*2+2+4+4*2+1+1
%4-tensors: 1
\begin{equs}[eq:all_C]{}
&\bar C^\e_{\be\be'}\;, \;
\hat C^\e_{1,\be}\;,\;
\hat C^\e_{2,\be}\;,\;
\hat C^\e_{3,\be,\bw}\;,\;
\hat C^\e_4\;,\;
\hat C^\e_5 \;,\;
\tilde C_\be\;, \;
C^\e [ \,\<PsiXi1>\, ]\;,\;
C^\e [\,\<I(PsiXi1)Xi1less>\,] \;
 \in \mfg^{\otimes 2}\;,
\\
&C_{\bs,\pm}^\e[ \, \<I'XiI'[IXiI'Xi]less>\,]\;,\;
C_{\bs}^\e[ \, \<I'XiI'[IXiI'Xi]less>\,]\;,\;
C^\e_\be [\<I'XiIXibar>]\;,\; C^\e_{\be,\bw}[\<IXiI'Xibar>] \;,\;
C^\e [\<R2-1new>] \;,\;
C^\e [\<cherry232>]  \; \in \mfg^{\otimes 2} \;,
\\
& C^\e [\<r_z_large>]\in\mfg^{\otimes 4}\;,
\end{equs}
where $\be,\be'\in\mcE_\times$, %$j\in\{1,2\}$, 
$\bw\in\{\Northwest,\Southwest\}$,
and $\bs\in\{\Southeast,\Southwest\}$.
Here the tensors $\bar C^\e_{\be\be'}$ satisfy
$\bar C^\e_{\be'\be} = T(\bar C^\e_{\be\be'})$, where $T\in L(\mfg^{\otimes 2},\mfg^{\otimes 2})$ is the switching map $T(x\otimes y)=y\otimes x$.
In particular, $\bar C^\e_{\be\be}\in \mfg^{\otimes_s 2}$.
We will later choose these components as in Section~\ref{sec:renorm_constants}, but for now they are arbitrary.

For $\tau\in \CT$, $M^\e(\tau)$ is of the form
\begin{equ}[eq:M_tau_def]
M^\e(\tau) = \tau - \text{contractions}\;,
\end{equ}
where `contractions' are
terms obtained by contracting components of $\tau$ with suitable constants in~\eqref{eq:all_C}.
We describe these contractions in several steps.
Write $\Psi_i =\mcb{I}_i\Xi_i\in\mfT^\YM_-$ and $\bar\Psi_i =\mcb{I}_i\bar\Xi_i\in\mfT^\rem_-$ as shorthands.

We first define $M^\e\colon \CT^\YM \cap \CT^{(i)} \to \CT^\YM \cap \CT^{(i)}$.
Consider $\tau\in \CT[\sigma]$ for $\sigma\in\mfT_i^\YM$.
First of all,
$M^\e$
contracts occurrences of the following expressions in $\tau$
and replaces them by the right-hand sides:
\begin{equs}[e:M-on-YM]
(\CS_{\be} \Psi_j)(\CS_{\be'} \Psi_j ) & \mapsto \bar C^\e_{\be\be'} \bone^{(i)}\;,
\\
(\CS_{\be} \mcb{I}_j (\CS_{\be'} (\cdot ) \bar\CD_j \Psi_i))  &   (\CD_j \Psi_i) \mapsto \hat C^\e_{1,\be} \bone^{(i)}\;, 
\\
(\bar\CD_i \Psi_j)  ( \CS_{\be} \mcb{I}_j (\CS_{\be'}(\cdot) & \CD_i \Psi_j  ))   \mapsto \hat C^\e_{2,\be} \bone^{(i)}\;, 
\\
(\CS_{\be} \Psi_j )
( \CD_i^+ \CS_{\bw} \mcb{I}_j (\CS_{\be'} (\cdot ) & \CD_i \Psi_j  )) \mapsto \hat C^\e_{3,\be,\bw} \bone^{(i)}\;,
\\ 
  \CD_i \Psi_i  (\mcb{I}_i \CD_i \Psi_i ) \mapsto \hat C^\e_4 \bone^{(i)}\;,
& \qquad
 \Psi_i  \CD_i \mcb{I}_i (\CD_i \Psi_i) \mapsto \hat C^\e_5 \bone^{(i)}\;,
\end{equs}
where $i\neq j\in\{1,2\}$, $\be,\be'\in \CE_\times$, and $\bw\in \{\Northwest,\Southwest\}$.
The contractions are made with suitable copies of $\mfg^*\otimes\mfg^*$ that appear inside 
$\CT[\sigma]$ in the way that respects the order in which we write them on the left-hand side
(see Example~\ref{ex:M_tau}
and recall the construction of $\CT=\bigoplus_{\tau\in\mfT}\CT[\tau]$ from Section~\ref{subsubsec:RS}).
A more correct notation would therefore be, e.g.
\begin{equ}
(\mfg^*)^{\otimes 2}\simeq \CT[(\CS_{\be} \Psi_j)(\CS_{\be'} \Psi_j )]\ni a \mapsto \scal{a,\bar C^\e_{\be\be'}} \bone^{(i)}\in \CT[\bone^{(i)}]\;,
\end{equ}
but we prefer to use the shorthand as in~\eqref{e:M-on-YM}. 
Note that the terms on the left-hand side of~\eqref{e:M-on-YM} correspond respectively to trees of the form
$
\<IXi^2>\,,\, \<I[I'Xi]I'Xi_notriangle>\,,\,
\<I[I'Xi]I'Xi_notriangle>\,,\, \<IXiI'[I'Xi]_notriangle>\,,\,
\<I[I'Xi]I'Xi_notriangle>\,,\,
\<IXiI'[I'Xi]_notriangle>\,$.
\begin{example}\label{ex:M_tau}
Suppose
$\mfT\ni\sigma = \CS_{\be} \mcb{I}_j (\CS_{\be'} \Psi_i  \bar{\CD}_j \Psi_i )
 \CD_j \Psi_i   \sim \<I[IXiI'Xi]I'Xi_notriangle>$, and
\begin{equ}
\tau = x_1\otimes x_2 \otimes x_3 \in \CT[\sigma ] \simeq (\mfg^*)^{\otimes 3}\;.
\end{equ}
Then the 2nd line of~\eqref{e:M-on-YM} implies (noting the minus sign in~\eqref{eq:M_tau_def})
\begin{equ}
M^\e(\tau) = \tau - \scal{x_2\otimes x_3,\hat C^\e_{1,\be}} x_1 \in \CT[\sigma] \oplus 
\CT[\Psi_i] \;.
\end{equ}
Next, suppose
\begin{equ}
\sigma = \CS_{\be} \mcb{I}_j (\bar{\CD}_j \Psi_i ) \CD_j \Psi_i 
\sim \<I[I'Xi]I'Xi_notriangle>\;,
\qquad
\tau = x_2 \otimes x_3 \in \CT[\sigma] \simeq (\mfg^*)^{\otimes 2}\;.
\end{equ}
Then, writing $\sigma=\CS_{\be} \mcb{I}_j (\CS_{\be'} \bone^{(i)} \bar{\CD}_j \Psi_i ) \CD_j \Psi_i$ where we used
$\bar{\CD}_j \Psi_i=\bone^{(j)}\bar{\CD}_j \Psi_i=\CS_{\be'}\bone^{(i)}\bar{\CD}_j \Psi_i$ for an arbitrary $\be'\in\CE_\times$,
the 2nd line of~\eqref{e:M-on-YM} now implies
\begin{equ}
M^\e(\tau) = \tau - \scal{x_2\otimes x_3, \hat C^\e_{1,\be}}\bone^{(i)} \in \CT[\sigma]\oplus \CT[\bone^{(i)}]\;.
\end{equ}
\end{example}
Moreover, 
%we also introduce  ``superficial renormalisations'' by imposing that  $M^\e$
we also impose that $M^\e$
makes the following contractions %in $\tau$ 
for $\be\in\CE_\times$
\begin{equ}[eq:tilde_C_renorm]
\CS_{\be} \Psi_j \bar\CD_i \Psi_j   \mapsto \tilde C_{\be} \bone^{(i)} \;,
\quad
\CS_{\be} \Psi_j \bar\CD_i \bar\Psi_j   \mapsto 
\eps^{1-\kappa}
\tilde C_{\be} \bone^{(i)} \;.
\end{equ}

We now define $M^\e \colon \CT^\rem \cap \CT^{(i)} \to \CT^\rem \cap \CT^{(i)}$.
For $\sigma\in\mfT^{\rem}$ and $\tau\in\CT[\sigma]$,
the map $M^\e$ makes the following contractions in $\tau$
\begin{equ}[e:M-XiIXi]
\bar\Xi_i \Psi_i \mapsto C^\e [ \,\<PsiXi1>\, ] \bone^{(i)}\;,
\qquad
\bar\Xi_i \bar\Psi_i
\mapsto
C^\e [\,\<I(PsiXi1)Xi1less>\,] \bone^{(i)}\;,
\end{equ}
where $C^\e [ \,\<PsiXi1>\, ],C^\e [\,\<I(PsiXi1)Xi1less>\,] \in\mfg^{\otimes 2}$
as well as, for $\bs\in\{\Southeast,\Southwest\}$
and $j\neq i$,
%(they are all of the form $\<I'XiI'[IXiI'Xi]less>$ in group \hyperlink{(1-2)}{(1-2)}):
\begin{equs}[e:div-R2]
( \CD_j^{\pm} \bar\Psi_i)
\CD_j^+ \CS_{\bs}
  \mcb{I}_j (\CS_{\be} (\cdot)\bar\CD_j \Psi_i)  &\mapsto C_{\bs,\pm}^\e[ \, \<I'XiI'[IXiI'Xi]less>\,] \bone^{(i)}\;, 
\\
(\CD_j^+ \CS_{\southeast} \bar\Psi_j)
 (\CD_j^+ \CS_{\southwest}
  \mcb{I}_j  (\CS_{\be} (\cdot) \CD_i \Psi_j ))
&\mapsto C_{\southwest}^\e[ \, \<I'XiI'[IXiI'Xi]less>\,] \bone^{(i)}\;, 
\\
( \CD_j^+ \CS_{\southwest} \bar\Psi_j)
( \CD_j^+ \CS_{\southeast} 
  \mcb{I}_j  (\CS_{\be} (\cdot )\CD_i \Psi_j ) )
 &\mapsto C_{\southeast}^\e[ \, \<I'XiI'[IXiI'Xi]less>\,] \bone^{(i)}\;,
\end{equs}
and
\begin{equ}[e:M-four]
\mcb{I}_i (\Psi_i \bar\Xi_i)(\Psi_i \bar\Xi_i) \mapsto C^\e [\<r_z_large>] \bone^{(i)}\;.
\end{equ}
We recall that $C_{\bs,\pm}^\e[ \, \<I'XiI'[IXiI'Xi]less>\,], C_{\bs}^\e[ \, \<I'XiI'[IXiI'Xi]less>\,] \in \mfg^{\otimes 2}$ and are independent of $\be$ appearing on the left-hand sides, and that $C^\e [\<r_z_large>]\in\mfg^{\otimes 4}$.
These contractions are understood in the same sense as described following~\eqref{e:M-on-YM}.
Furthermore, for $\bw\in\{\Northwest,\Southwest\}$ and $j\neq i$,
$M^\e$ makes the contractions
\begin{equs}[eq:square_cherry_consts]
\bar\CD_i \Psi_j \CS_\be \CI_j(\cdot\, \bar\Xi_j ) &\mapsto C^\e_\be [\<I'XiIXibar>] \CS_\be (\cdot)\;,
\\
\CS_{\be}\Psi_j \CD_i^+\CS_{\bw}\CI_j(\cdot\, \bar\Xi_j) & \mapsto C^\e_{\be,\bw}[\<IXiI'Xibar>]\CS_{\bw}(\cdot)\;,
\\
\CD_j^+\CS_{\southeast} \tilde\Psi_j
 \CD_j^+\CS_{\southwest} \tilde\Psi_j
&\mapsto
C^\e [\<R2-1new>] \bone^{(i)}\;,
\\
\CD_j^+\CS_{\southeast} \bar\Psi_j
 \CD_j^+\CS_{\southwest} \bar\Psi_j
&\mapsto 
C^\e [\<cherry232>] \bone^{(i)} \;.
\end{equs} 
Finally, we define $M^\e=\id$ on $\CT^\ren$, which concludes the definition of $M^\e$ on $\CT$.

A tree $\tau\in\mfT$ is called \textit{planted} if it is of the form $\tau=\CS_\be\CI_i[\bar\tau]$ or $\tau=\CI_i[\bar\tau]$ for some $i\in\{1,2\}$, $\be\in\CE_\times$, and $\bar\tau\in\mfT$.

\begin{remark}\label{rem:renorm_planted}
Since we did not include the symbol $\CI_i\bone^{(i)}$ and its derivatives in $\mfT$,
$M^\e\tau=\tau$ all $\tau\in\CT[\sigma]$ where $\sigma$ is planted or is the derivative of a planted tree.
\end{remark}

%\begin{remark}
%The set of all $M^\e\in L(\CT,\CT)$ forms a group.
%This group is furthermore Abelian, which may be surprising due to trees $\sigma$ of the form $\<r_z_large>$.
%Indeed, suppose $M_1$ has only non-zero component
%$C^\e [\,\<I(PsiXi1)Xi1less>\,]$, which means $M_1[\tau] \in \tau + \CT[\<IXi^2>]$.
%Then it may appear that applying another map $M_2$ with only non-zero component $\bar C^\e$ would yield an element in $\CT[\bone]$,
%while this would not be the case if we first apply $M_2$ then $M_1$.
%The reason this does not happen in our setting is that $\bar C^\e$ \textit{only} contracts cherries with shifts, while the tree $\<IXi^2>$ above does not have shifts.
%\end{remark}

\begin{remark}\label{rem:superificial_renorm}
The renormalisations \eqref{eq:tilde_C_renorm} will not contribute to the renormalised equation due to Lie brackets (see~\eqref{e:quad-exp-3}).
We later also define their values 
so that  $\sum_{\be \in \CE_\times} \tilde C_{\be} =0$
(as one would guess by parity),
which is another reason that they do not contribute to the renormalised equation.
Likewise,  $\<PsiPsibarXibar>$ and $\<r_z_large>$  in group \hyperlink{(3-3)}{(3-3)}
are renormalised by  $C^\e [ \,\<PsiXi1>\, ]$, but 
this will not have any effect on the renormalised equation, thanks to the Lie brackets;
the same happens for $\X^{(i)} \<PsiXi1>$ (alternatively, due to $\X$).
The tree  $\<R2-1new>$  also will not contribute, again due to the Lie bracket,
and $C^\e [\<cherry232>] $ will cancel out, see \eqref{e:tildeR2-exp}.
Finally, 
some trees in group \hyperlink{(1-3)}{(1-3)} 
%will satisfy an exact parity symmetry thanks to the derivative, so no renormalisation is needed.
are renormalised by \eqref{eq:square_cherry_consts}, but they will eventually 
renormalise the equation
by a term of the form $\e \partial A$
 which vanishes in the limit (see \eqref{eq:bar_delta_def} and the proof of Proposition~\ref{prop:sol-abs}).
\end{remark}

\begin{remark}\label{rem:no_renorm}
The above renormalisation
is sufficient although 
some of the other trees % in $\mfT$ in \eqref{e:trees-deg-1-rem}-\eqref{e:trees-deg-0-rem},
appear to have divergences.
Here we give a quick ``preview'' for the arguments in the forthcoming sections.
%Recall from Remark~\ref{rem:loss-odd} that discretisations may not preserve exact parity symmetry.
\begin{enumerate}
%\item
%The tree 
% $\<R2-1new> = (\CD_j^+ \CS_{\southeast}\tilde\Psi_j) ( \CD_j^+ \CS_{\southwest} \tilde\Psi_j ) $
% does not need renormalisation,
%since the two factors are evaluated at {\it distinct} points.\footnote{This is reminiscent to 
%a cancellation mechanism in proofs of KPZ equation universality, e.g. 
%\cite[Section~4.2, Lemma~A.1]{BG97}.}
%The same applies to $\<cherry232>$. % in group \hyperlink{(2-3)}{(2-3)}.
\item
Trees  of the form $\<I'XiI'[IXiI'Xi]less>$ %in group \hyperlink{(1-2)}{(1-2)} 
do not need renormalisation
either by independence of $\xi_i$ and $\xi_j$,
or by {\it exact} parity,
 except for the ones in \eqref{e:div-R2} with only {\it approximate} parity.
This is subtle: see Remarks~\ref{rem:loss-odd} and~\ref{rem:C_non_vanish} and proof of Lemma~\ref{lem:mom-13}.
%for instance, one might expect a divergence from the tree
%$(\CD_j^+ \CS_{\southeast}\bar\Psi_j)
% (\CD_j^+ \CS_{\southwest}
%  \mcb{I}  (\Psi_j \CD_j \Psi_j ))$
%which would ultimately contribute to a mass renormalisation of the form $cA_j$ in the equation for $A_i$ ($j\neq i$). 
%It turns out that such superficial  divergences will be cancelled by exact parity, while 
%the divergences in \eqref{e:div-R2} only satisfy approximate parity and will lead to mass renormalisation of the form $cA_i$.
\item
Trees in group \hyperlink{(2-2)}{(2-2)} and several trees such as $\<I'XiI(R2-1)new>$ in group \hyperlink{(1-2)}{(1-2)}
do not have divergent substructures
thanks to the introduction of $\tilde\Xi$.
\item We choose not to renormalise trees in $\mfT^\ren=\{\Psi_1^2,\Psi_2^2\}$
since they
only appear after the renormalisation of $\<r_z_large>$ in the derivation of the renormalised equation.
In particular, they do not appear in the solution of
the abstract fixed point problem % associated to~\eqref{e:Aeps} 
(see also Remark~\ref{rem:model_bounds_sector}).
This choice simplifies slightly the renormalisation group, e.g. rendering it Abelian.
\end{enumerate}
\end{remark}
\begin{definition}\label{def:renorm_model}
Consider a discrete model $(\Pi^\e,\Gamma^\e)$ and a map $M^\e\in L(\CT,\CT)$ as above.
Let $(\hat\Pi,\hat\Gamma)=\{\hat\Pi^\e_x,\hat\Gamma^\e_{xy}\}_{x,y\in \R^3}$
be the collection of linear maps $\hat\Pi^\e_x\colon \CT\to \CX_\e$ and
operators $\hat\Gamma^\e_{xy} \in\CG$
defined by $\hat\Pi^\e_x \tau =\Pi^\e_x M^\e \tau $
and $\hat\Gamma^\e_{xy}\tau = (M^\e)^{-1} \Gamma^\e_{xy} M^\e\tau$
for every $\tau\in \CT$.
\end{definition}
%
%\begin{remark}
%It is non-trivial that $(\hat\Pi^\e,\hat\Gamma^\e)$ is a model, and we show this in Proposition~\ref{prop:renormalised_model}.
%The reason we can take such a simple definition for $(\hat\Pi^\e, \hat\Gamma^\e)$ is 
%that our regularity structure is rather small and satisfies the co-interaction property 
%of~\cite{BHZ19} without the need for so-called extended decorations.
%(See~\cite[Proof of Theorem 3.36]{ESI_lectures} for how these definitions follow from co-interaction.)
%This is specific to 2D and such a simple definition would not suffice in 3D.
%
%
%We could have followed a more general construction and defined $(\hat\Pi^\e,\hat\Gamma^\e)$ via a renormalised admissible (uncentred) map $\PPi^\e M^\e$, which always yields a model.
%However, with this definition,
%it would not be a priori clear that we have the identity $\hat\Pi^\e_x=\Pi^\e_x M^\e$,
%which we use in Section~\ref{subsec:renorm_eq} to derive the renormalised equation
%and in Section~\ref{subsec:moments_discrete_models} to get uniform estimates on our models (since
%the general results of~\cite{BCCH17, CH16} are not available to us)
%and the proof of which would amount to that of Proposition~\ref{prop:renormalised_model}.
%We therefore choose to define $\hat\Pi^\e_x = \Pi^\e_x M^\e$ right away and avoid introducing $\PPi^\e$.
%\end{remark}
%
It is non-trivial that $(\hat\Pi^\e,\hat\Gamma^\e)$ is a model, and we show this next.
Recall the notion of compatibility of a model (Definition~\ref{def:compat}).
\begin{proposition}\label{prop:renormalised_model}
If $(\Pi^\e,\Gamma^\e)$ is a model, then so is $(\hat\Pi^\e,\hat\Gamma^\e)$.
Furthermore, if $(\Pi^\e,\Gamma^\e)$ is compatible, then so is $(\hat\Pi^\e,\hat\Gamma^\e)$.
\end{proposition}
For the proof, we require the following (essentially combinatorial) lemma.
\begin{lemma}\label{lem:Gamma_M}
Consider a tree $\sigma\in\mfT$ and let $\tau\in\CT[\sigma]$, $\Gamma\in \CG$, and $M\in L(\CT,\CT)$ a map as above.
\begin{enumerate}[label=(\roman*)]
\item\label{pt:Gamma_M_commute} 
Suppose $\sigma$ is not of the form $\X^{(i)} \<PsiXi1>,\,\<PsiPsibarXibar>$ or $\<r_z_large>\,$.
Then
%\begin{equ}\label{eq:Gamma_M_commute}
$M\Gamma\tau = \Gamma M\tau$.
%\end{equ}

\item\label{pt:Gamma_M_commute_almost} Suppose $\sigma$ is of the form $\X^{(i)} \<PsiXi1>,\,\<PsiPsibarXibar>$ or $\<r_z_large>\,$.
Then
\begin{equ}\label{eq:Gamma_M_commute_almost}
M^{-1}\Gamma M\tau = \tau + \hat\tau\;,\qquad 
\hat\tau\in\CT[ \<PsiXi1>]\;.
\end{equ}
\end{enumerate}
\end{lemma}

\begin{proof}
\ref{pt:Gamma_M_commute}
Clearly $\Gamma\bar\CT=\bar\CT$ and $M$ acts trivially on $\bar\CT$, hence~\ref{pt:Gamma_M_commute} holds for all $\tau\in\bar\CT$.
It therefore suffices to consider $\sigma$ non-polynomial.
Suppose $\sigma\in\mfT\setminus\bar\mfT$ is such that $\CG$ acts non-trivially on $\CT[\sigma]$.
Then either $|\sigma|>0$, i.e. $\sigma$ is of the form
\begin{equ}[eq:pos_trees]
\<I(cherry)>\;,\quad
\<IDPsi>\;, \quad
 \<IXibar>\;,\quad
\<I(R2-1new)>\;, \quad
\<IPsiXi1>
\end{equ}
(we recall here \eqref{eq:pos_trees-YM} and \eqref{eq:pos_trees_rem} ), or $\sigma$ is of the form $\sigma=\chi\bar \Xi$, i.e.
\begin{equ}[eq:trees_Xi_bar]
\<cherryY>\;,\quad
\<cherryYless>\;,\quad
\X^{(i)} \bar\Xi_i\;,\quad
\X^{(i)} \<PsiXi1>\;,\quad
\<I(R2-1)Xi1new>\;,\quad
\<I(PsiXi1)Xi1>\;,\quad
\<I(PsiXi1)Xi1less>\;,\quad
\<PsiPsibarXibar>\;\quad
\<r_z_large>\;,
\end{equ}
or $\sigma$ is of the form $\sigma = \CI[\bar\sigma]\CD\CI[\Xi]$ with $|\bar\sigma|\in(-2,-1)$, i.e.
\begin{equ}\label{eq:IDI_trees}
\<I[IXiI'Xi]I'Xi_notriangle>\;,\quad
\<I[I'Xi]I'Xi_notriangle>\;,\quad
\<I'XiI(R2-1)new>\;,\quad
\<I'XiIXibar>\;,\quad
\<supercherry2>\;.
\end{equ}
Let $\mfT^{\nfix}\subset\mfT$ denote the set of trees of the above form
and $\mfT^{\fix} = \mfT\setminus\{ \bar\mfT \cup \mfT^{\nfix}\}$.

We claim that, if $\sigma\in \mfT^\fix$, then $M\tau \in \CT^\fix \eqdef \bigoplus_{\sigma\in\mfT^\fix}\CT[\sigma]$.
Indeed, every tree in $\mfT^{\nfix}$ except $\<IDPsi>\;,\;
 \<IXibar>\;,\;\X^{(i)} \bar\Xi_i$
 contains at least two noises
 (while we usually treat $\<R2-1new>$ as a single noise, for the purpose of this argument, we treat it as having two noises).
If $b\in\mfT^{\nfix}$ has at least two noises and $M\tau$ has a (non-zero) component in $\CT[b]$,
then $\sigma$ must contain at least four noises, which implies $\sigma\sim\<r_z_large>$ and thus $\sigma\in\mfT^{\nfix}$.
Turning to the trees $b\in\mfT^\nfix$ with a single noise,
clearly $M\tau$ does not have a (non-zero) component in $\CT[\X^{(i)} \bar\Xi_i]$
(since only $\sigma=\X^{(i)} \bar\Xi_i\in\mfT^{\nfix}$ would allow such a component). 
Similarly $M\tau$ cannot have a component in $\CT[\<IDPsi>]$ because it would require contracting a tree of degree $<-1$ at the root
(this would only have been possible with $\sigma \sim \<r_z_avoid2>$ which does not appear in $\mfT$).
Finally, $M\tau$ can have a component in $\CT[\<IXibar>]$ only if $\sigma \sim \<IXibar>$ or $\sigma\sim\<PsiPsibarXibar>$, but these are in $\mfT^{\nfix}$.
This proves the claim and shows~\ref{pt:Gamma_M_commute} for $\tau\in\CT^\fix$.

We now consider the case $\sigma\in\mfT^{\nfix}$.
Suppose $|\sigma|>0$, i.e. $\sigma$ is of the form in~\eqref{eq:pos_trees}.
Then clearly $\Gamma\tau=\sum_i\tau_i$ with $M\tau_i=\tau_i$ for all $\Gamma\in\CG$
(in particular $M\tau=\tau$), which proves~\ref{pt:Gamma_M_commute}.

Suppose now $\sigma$ is of the form in~\eqref{eq:trees_Xi_bar}.
Except for $\sigma\sim\X^{(i)} \<PsiXi1>\,,\, \<PsiPsibarXibar>\,,\,
\<r_z_large>$, we have
$M\tau \in \tau + \CT[\bone] + \CT[\<IXi>]$
and $\Gamma\tau \in \tau + \CT[\bar\Xi_i]$,
which shows~\ref{pt:Gamma_M_commute}
since $\Gamma$ acts trivially on $\CT[\bone]\oplus\CT[\<IXi>]$ and $M$ acts trivially on $\CT[\bar\Xi_i]$.

Finally, suppose $\sigma$ is of the form~\eqref{eq:IDI_trees}.
Then $M\tau \in \tau + \CT[\bone] + \CT[\<IXi>]$ and $\Gamma \tau \in \tau + \CT[\<I'Xi_notriangle>]$,
and since $M$ acts trivially on $\CT[\<I'Xi_notriangle>]$ and $\Gamma$ acts trivially on $\CT[\bone]\oplus\CT[\<IXi>]$,
we obtain the desired identity, %~\eqref{eq:Gamma_M_commute}, 
which concludes the proof.

\ref{pt:Gamma_M_commute_almost}
Suppose $\sigma \sim \<r_z_large>$.
Then $M\tau \in \tau + \bar\tau + \CT[\<IXi^2>]  + \CT[\bone]$
where $\bar\tau\in \CT[\<IPsiXi1>]$
(note that $\CT[\<IXi>]$ does not appear since we exclude contractions of three noises).
The only components of $M \tau$ on which $\Gamma$ does not act trivially are $\tau\in\CT[\<r_z_large>]$
and $\bar\tau\in\CT[\<IPsiXi1>]$, for which we have $\Gamma\tau = \tau + \hat\tau$ with $\hat\tau \in \CT[\<PsiXi1>]$ and $\Gamma\bar\tau \in \bar\tau + \CT[\bone]$.
On the other hand, $M^{-1}\hat\tau \in \hat\tau + \CT[\bone]$ where the component in $\CT[\bone]$ carries the opposite sign to that of
the component in $\CT[\bone]$ of 
$\Gamma\bar\tau$, which shows $M^{-1}\Gamma M\tau =\tau + \hat\tau$ and thus~\eqref{eq:Gamma_M_commute_almost}.
The proof for $\sigma$ of the form $\X^{(i)} \<PsiXi1>$ and $\<PsiPsibarXibar>$
is almost the same once we note that, in both cases, $\Gamma\tau = \tau + \hat\tau$ with $\hat\tau \in \CT[\<PsiXi1>]$
and, in the first case,
\begin{equ}
M\tau \in \tau + \CT[\X^{(i)}]\;,\quad
 \Gamma\X^{(i)} = \X^{(i)} + \CT[\bone]\;.
\end{equ}
and in the second case (in somewhat informal notation),
$M\tau \in \tau + \CT[\<IXi>] + \CT[\<IXibar>]$,
$\Gamma \<IXi> = \<IXi>$, and
$\Gamma \,\<IXibar> = \<IXibar> + \CT[\bone]$.
\end{proof}
\begin{proof}[of Proposition~\ref{prop:renormalised_model}]
Dropping reference to $\eps$,
we first verify that $(\hat\Pi,\hat\Gamma)$ is a model.
Since $M$ preserves $\CT^{(i)}$ for each $i\in\{1,2\}$,
it follows that $\hat\Pi$ is of the desired form~\eqref{e:model-split}.
The algebraic conditions~\eqref{eq:model_algeba} on $(\hat\Pi,\hat\Gamma)$ are obviously satisfied as are the analytic conditions for $\hat\Pi$ since $M\tau = \tau + \sum_i\tau_i$ for $\tau\in\CT_\beta$ with $\tau_i\in \CT_{\alpha}$ for $\alpha>\beta$.
The analytic conditions on $\Gamma$ are satisfied due to Lemma~\ref{lem:Gamma_M},
and thus $(\hat\Pi,\hat\Gamma)$ is a model.

The fact that $(\hat\Pi,\hat\Gamma)$ is admissible, compatible with shifts, and compatible with derivatives is clear from the respective properties of $(\Pi,\Gamma)$
because $M\tau=\tau$ for any $\tau\in\CT[\CI \sigma]$ where $\sigma$ is a planted tree or the derivative of a planted tree (Remark~\ref{rem:renorm_planted}).
The same applies to compatibility with multiplication by $\eps$
because $M\tau=\tau$ and $M \tau_\eps = \tau_\eps$ for every $\tau\in\bigoplus_{\sigma\in \{\<I'Xi_notriangle>,
\<R2-1new>, \<IXi>\}}\CT[\sigma]$
(recall the map $\tau\mapsto \tau_\e$ from~\eqref{eq:tau_to_bar_tau}),
while for $\tau \in \CT[\<IXiI'Xi_notriangle>]$,
$M\tau = \tau + C\bone^{(i)}$
and $M\tau_\e = \tau_\e + \e^{1-\kappa}C\bone^{(i)}$
for some $C\in\mfg^{\otimes 2}$
due to \eqref{eq:tilde_C_renorm}.
\end{proof}

\section{Solution theory of the discrete dynamic \texorpdfstring{$A^\eps$}{A eps}}
\label{sec:Aeps}

In this section we discuss the solution theory of \eqref{e:Aeps} for $d=2$.
Throughout the section, we suppose that Assumption~\ref{assump:R} holds.

\subsection{Remainders of degree \texorpdfstring{$0-$}{0-} and mass renormalisation}
\label{subsec:Mass renormalisation_1}

We first analyse the remainder terms of degree (slightly below) $0$.
These terms are more regular than the terms \eqref{e:deg-1errors} (which have degree below $-1$)
and they do not require using regularity structures.
Nonetheless, they impact our formulation of the fixed point problem in Section~\ref{subsec:FPP}.

Fix throughout this subsection $r\geq 1,T\in(0,1]$, and $\Psi_i,v_i \colon [-1,2] \times\obonds_i \to \mfg$ for $i\in\{1,2\}$ and write $A_i=\Psi_i+v_i$ (all functions are implicitly indexed by $\eps$).
We will later take $\Psi_i = K^{i;\e}*_{(i)}\xi_i^\e$ and $A_i$ as the solution to our fixed point problem.
We suppose that, for some $\eta\in (-\frac14,-2\kappa)$, for all $\e>0$,
\begin{equ}[eq:v_bound]
\|v\|_{\CC^{1+\eta,T}_{-1,\eps}} + \|v\|_{\CC^{-\eta,T}_{2\eta,\eps}} + \|v\|_{\CC^{0,T}_{\eta,\eps}} 
+ \|v\|_{\CC^{\eta,T}_{0,\eps}}< r\;.
\end{equ}
Our assumption on $\Psi$ is as follows (see Lemma~\ref{lem:Psi_Wick} for the case $\Psi_i = K^{i;\e}*_{(i)}\xi_i^\e$).

\begin{assumption}\label{as:models_abstract}
For all $\eps>0$,
\begin{equ}[e:assump-Psi-1]
 \sup_{t\in[-1,2]} \|\Psi_i(t)\|_{\CC^{-\kappa}_\e} <r\;.
\end{equ}
Furthermore, let $\partial\in \{\partial_j,\bar\partial_j, \partial_j^\pm\}$ for some $j\in \{1,2\}$,
$i_1,i_2,i_3\in \{1,2\}$, and $h_k\in  \R^2$
 for $k\in\{2,3\}$ such that $|h_k|\leq 4$ and $\obonds_{i_1}+\e h_k =\obonds_{i_k}$.
Then there exist $c_{23},c_{13} \in \mfg\otimes\mfg$ (depending on $\d,i_k,h_k,\e$) such that $c_{23}=O(1),c_{13}=O(1)$ uniformly in $\eps$
and
\begin{equs}[e:assump-Psi-2]
\sup_{t\in[-1,2]}  
\{\| W_{c,h} (\Psi_{i_1}, \Psi_{i_2}, \partial\Psi_{i_3})(t)\|_{\CC^{-\kappa}_\eps}
&+
\| W^{(13)}_{c,h} (\Psi_{i_1}, \partial\Psi_{i_3})(t)\|_{\CC^{-\kappa}_\eps}
\\
&+
\| W^{(23)}_{c,h} (\Psi_{i_2}, \partial\Psi_{i_3})(t)\|_{\CC^{-\kappa}_\eps}\} < r\e^\theta\;,
\end{equs}
for some fixed $\theta\in(0,\kappa]$ and where, for functions $f_{k} \in \mfg^{\obonds_{i_k}}$, $k=1,2,3$,
we define
\begin{equs}[eq:W_ch]
W_{c,h} (f_1,f_2,f_3) &\eqdef \e f_1\otimes f_2(\cdot+\e h_2)\otimes f_3(\cdot+\e h_3) - f_1 \otimes c_{23}
\\
&\qquad - \sum_{\ell=1}^nc^{(\ell)}_1\otimes f_2(\cdot+\e h_2) \otimes c^{(\ell)}_3 \in (\mfg^{\otimes 3})^{\obonds_{i_1}}\;,
\end{equs}
where we write $c_{13} = \sum_{\ell=1}^n c^{(\ell)}_1\otimes c^{(\ell)}_3$ for $c^{(\ell)}_1, c^{(\ell)}_3\in\mfg$,
and
\begin{equs}
W^{(13)}_{c,h} (f_1,f_3)
&\eqdef \e f_1\otimes f_3(\cdot+\e h_3) - c_{13} \in (\mfg\otimes \mfg)^{\obonds_{i_1}}\;,
\\
W^{(23)}_{c,h} (f_2,f_3)
&\eqdef \e f_2(\cdot+\e h_2)\otimes f_3(\cdot+\e h_3) - c_{23} \in (\mfg\otimes \mfg)^{\obonds_{i_1}}\;.
\end{equs}
\end{assumption}
We will show that the ``remainder'' terms arising from $R_3$, $\tilde R_3$, $\Delta_3$, $\hat R^\nabla_{\eps}$ in~\eqref{e:Aeps} fall into the scope of the following lemma.

\begin{lemma}\label{lem:eps-A2-dA}
Let $\partial, i_1,i_2,i_3, h_2,h_3$ be as in Assumption~\ref{as:models_abstract}.
Then, uniformly over $\Psi,v$ satisfying~\eqref{eq:v_bound} and
Assumption~\ref{as:models_abstract},
\begin{equs}
\sup_{t\in (0,T]} t^{\frac12}
\|W_{c,h}(A_{i_1},A_{i_2},\d A_{i_3})(t)\|_{\CC^{-2\kappa}_\eps}
\lesssim \e^\theta r^3\;.
\end{equs}
\end{lemma}

\begin{proof}
In this proof we will write
$A_{i_1}  A_{i_2}  \partial A_{i_3}$ or 
simply $A^2 \partial A$ as shorthands for the tensor product.
Expanding $A=\Psi+v$, we have 6 types of terms in $W_{c,h}(A_{i_1},A_{i_2},\d A_{i_3})$ to consider.

Recall $\eta\in (-\frac14,-2\kappa)$ and $\theta\in (0,\kappa]$,
and recall Young's theorem in our discrete setting in Lemma~\ref{lem:discrete_Young}.
Omitting the $t$ variables in the notation and with all the norms understood as spatial norms, and applying \eqref{eq:v_bound} for the ``$v$ factors'' and 
\eqref{e:assump-Psi-1}-\eqref{e:assump-Psi-2}  for the ``$\Psi$ factors'' appearing below,
one has
\begin{equs}
%%%%% 1 %%%%%
{} & \| \e \Psi_{i_1} \Psi_{i_2} \partial \Psi_{i_3} 
- \Psi_{i_1}c_{23}^\e
- c_1^{\e,(\ell)} \Psi_{i_2} c_3^{\e,(\ell)}\|_{\CC^{-\kappa}_\e}
\stackrel{\eqref{e:assump-Psi-2}}{<} r\e^\theta \;,
\\
%%%%% 2 %%%%%
{} & \|\e \Psi^2   \partial v\|_{L_\e^\infty}
\le 
\e^{1-2\kappa+\eta}
\| \eps^{\kappa}\Psi \|^2_{L^\infty_{\e}}
\|\e^{-\eta} \partial v\|_{L_\e^\infty}
\stackrel{\textnormal{\ref{pt:eps_f_vanish}\ref{pt:eps_d_f_vanish}}}\lesssim
\e^{1-2\kappa+\eta}
\| \Psi \|^2_{\CC_\e^{-\kappa}}
\|v\|_{\CC_\e^{1+\eta}}
\\
&\qquad\qquad\qquad\qquad\qquad\qquad\qquad\qquad\qquad
\lesssim
 \e^{1-2\kappa+\eta}r^3 t^{-\frac12} \;,
\\
%%%%% 3 %%%%%
{} & \|\e \Psi \partial \Psi v - cv\|_{\CC_\e^{-\kappa}} 
\le
\|\eps \Psi \partial \Psi - c\|_{\CC_\e^{-\kappa}}
\|v\|_{\CC_\e^{1+\eta}}
\lesssim 
r \e^\theta\cdot r  t^{-\frac12} \;,
\end{equs}
where $\e \Psi \partial \Psi v - cv$ corresponds to two types of terms
with $c$ standing for $c_{23}^\e$ or $c_{13}^\e$,
and where
 \ref{pt:eps_f_vanish}\ref{pt:eps_d_f_vanish}\ref{pt:eps_d_f} refer 
to Lemma~\ref{lem:eps-improve-reg}\ref{pt:eps_f_vanish}\ref{pt:eps_d_f_vanish}\ref{pt:eps_d_f}.
Moreover,
since
\begin{equ}[eq:eDPsi-eDv]
\|\e \partial\Psi\|_{\CC^{-2\kappa}_\e}
\stackrel{\textnormal{\ref{pt:eps_d_f}}}{\lesssim} 
\e^{\kappa}
\|\Psi\|_{\CC_\e^{-\kappa}}
\lesssim
\e^{\kappa} r,
\quad
\| \e \partial v \|_{L^\infty_\e}
\stackrel{\textnormal{\ref{pt:eps_d_f_vanish}}}\lesssim
\e^{-\eta} 
\| v \|_{\CC^{-\eta}_\e}
\le
\e^{-\eta}
r t^{\eta} ,
\end{equ}
one has, similarly as above,
\begin{equs}
%%%%% 4 %%%%%
{} &
\| \e \Psi v \partial v \|_{L_\e^\infty} 
\le 
\e^{-\kappa}
\|\e^\kappa \Psi\|_{L^\infty_\e}
\|v \|_{L^\infty_\e}
\| \e \partial v \|_{L^\infty_\e}
%\stackrel{\ref{pt:eps_f_vanish}\ref{pt:eps_d_f_vanish}}\lesssim
%\e^{-\kappa-\eta}
%\| \Psi\|_{\CC_\e^{-\kappa}}
%\|v \|_{L^\infty_\e}
%\|v \|_{\CC_\e^{-\eta}}
%\\
%&\qquad\qquad\qquad\qquad\qquad\qquad\qquad\qquad\qquad
\lesssim
 r^3\e^{-\eta-\kappa}t^{3\eta/2}\;, %\leq r^3\e^{\theta} t^{-\frac12}\;,
\\
%%%%% 5 %%%%%
{} & 
\| \e v^2 \partial\Psi \|_{\CC_\e^{-2\kappa}}
\le 
\|\e \partial\Psi\|_{\CC_\e^{-2\kappa}}
\| v^2\|_{\CC_\e^{-\eta}}
%\stackrel{\ref{pt:eps_d_f}}{\lesssim} 
%\e^{\kappa}
%\|\Psi\|_{\CC_\e^{-\kappa}}
%\| v^2\|_{\CC_\e^{-\eta}}
\lesssim
\e^{\kappa} r  \cdot r^2 t^{2\eta}\;,
%\\
%&\qquad\qquad\qquad\qquad\qquad\qquad\qquad\qquad\qquad
%\lesssim 
%\e^{\kappa}  r^3 t^{-\frac12}\;,
\\
%%%%% 6 %%%%%
{} &
\| \e v^2 \partial v\|_{L_\e^\infty}
\le
\|v^2\|_{L_\e^\infty}
\|\e \partial v\|_{L_\e^\infty} 
%\stackrel{\ref{pt:eps_d_f_vanish}}\lesssim
%\e^{-\eta}
%\|v^2\|_{L_\e^\infty}
%\|v\|_{\CC^{-\eta}_\e} 
\lesssim 
\e^{-\eta}
r^3 t^{2\eta} \;.
\end{equs}
These bounds are all uniformly in $t\in (0,T]$,
and are all bounded by $r^3\e^{\theta} t^{-\frac12}$
by our assumptions on $\eta$ and $\theta$.
\end{proof}
For the rest of the subsection, we suppose~\eqref{eq:v_bound} and Assumption~\ref{as:models_abstract} hold.
For $c\in L(\mfg^2,\mfg^2)$, we write $c_i\in L(\mfg^2,\mfg)$ for the $i$-th component of $c$ and $c_i B = c_i^{(1)}B_1 + c_i^{(2)} B_2$ with $c^{(j)}_i\in L(\mfg,\mfg)$.
We extend every $c\in L(\mfg^2,\mfg^2)$ to a map $c\in L(\mfq,\mfq)$ by $cA(e)=c_i^{(i)}A_i(e)+c_i^{(j)}A_j(e^{(a)})$
for $e\in\obonds_i$ and $j\neq i$,
where $A_j(e^{(a)})$ is defined by~\eqref{e:def-Aea}.

\begin{lemma}\label{lem:I3}
For $I_3$ from~\eqref{e:Aeps},
there exists  $c^\e\in L(\mfg^2,\mfg^2)$ such that $c^\e=O(1)$ and
$
\|I_3(A) - c^\e A\|_{\CC^{-2\kappa,T}_{-1,\e}} \lesssim \e^\theta r^3
$.
\end{lemma}

\begin{proof}
Since $I_3$ is in $\mcI_3$ (in the sense of 
Definition \ref{def:ideal}), $I_3$ is of the form $\e A^2\d A$.
By Lemma~\ref{lem:eps-A2-dA}, for every monomial in $I_3$,
there exists $c\in L(\mfg^2,\mfg^2)$, $c=O(1)$, and $h\in \R^2$, $|h|<4$,
such that by subtracting $cA(\cdot +\e h)$, one has the claimed bound for this monomial.
The difference between $A(e +\e h)$ and $A(e)$
can always be written into terms of the form $\e \partial A$, and
%by~\eqref{eq:e_d_Psi_bound} and~\eqref{eq:e_p_v_bound},
by~\eqref{eq:eDPsi-eDv},
% Lemma~\ref{lem:eps-improve-reg}\ref{pt:eps_f_vanish}-\ref{pt:eps_d_f_vanish} and the assumptions~\eqref{eq:v_bound}-\eqref{e:assump-Psi-1} on $A= \Psi+v$,
\begin{equ}[eq:edA_to_zero]
\|\e\d A\|_{\CC^{-2\kappa,T}_{2\eta,\e}}  \lesssim \|\e\d \Psi\|_{\CC^{-2\kappa,T}_{0,\e}}
+ \|\e\d v\|_{\CC^{0,T}_{2\eta,\e}}
\lesssim r\e^\kappa\;,
\end{equ}
which means that all renormalisation terms can be evaluated either at $e\in\obonds_i$ or $e^{(a)}$ (vs. other neighbouring points).
\end{proof}
\begin{lemma}\label{lem:grad-R-hat}
For $ \hat R^\nabla_{\eps}$ from
\eqref{e:from-pp-to-e},
there exists $c^\e\in L(\mfg^2,\mfg^2)$, such that $c^\e=O(1)$ and
$
\|\e^{-1} \hat R^{\nabla}_{\eps} (\e A)- c^\e A \|_{\CC^{-2\kappa,T}_{-1,\e}}
\lesssim
\eps^\theta r^4
$
whenever $\e A(t) \in V$ for all $t\in[0,T]$.
\end{lemma}

\begin{remark}\label{rem:eA_small}
The final condition $\e A(t) \in V$ is necessary to make $\nabla \hat R^p_\e(\e A)$ that appears in \eqref{e:from-pp-to-e} well-defined (recall $\hat R_\e$ is assumed differentiable only in $\mathring V$).
This condition is not so stringent:
Assumptions~\eqref{eq:v_bound} and~\eqref{e:assump-Psi-1} and Lemma~\ref{lem:eps-improve-reg}\ref{pt:eps_f_vanish}
imply $\|A(t)\|_{L^\infty_\e} \lesssim \e^\eta \|A(t)\|_{\CC^\eta_\e} \leq 2r\e^\eta$
for all $t\in[0,T]$.
In particular, since $\eta>-1$,
for $\e$ sufficiently small, $\e A$ takes values to $V\subset\mfq$.
\end{remark}

%\begin{remark}\label{rem:sanity}
%For the proof of Theorem~\ref{thm:discrete_dynamics},
%it would suffice in the final statement of Lemma~\ref{lem:grad-R-hat} that $c^{(j)}_i=0$ only for $i\neq j$,
%just like in Lemma~\ref{lem:Rp3},
%while having $c_i^{(i)}$ remain a general element of $L(\mfg,\mfg)$ that can in principle depend on the choice of remainder $R_\e$ in~\eqref{e:S=xR} that appears in the 
%action $S_N$.
%However, $\hat R^{\nabla}_{\eps}$ is the only term in the discrete equation~\eqref{e:Aeps} that
%depends on $R_\e$, and when we ultimately take $A$ as the solution to~\eqref{e:Aeps},
%having a renormalisation $-c_i^{(i)}A_i$ that is dependent on $R_\e$
%would give rise to limiting SPDEs with mass terms that depend on the particular choice of lattice gauge theory,
%which we prove is not possible in Section~\ref{sec:gauge-covar}.
%The current statement of Lemma~\ref{lem:grad-R-hat} (with $c_i^{(i)}=0$)
%verifies directly that the choice of $R_\eps$ does not impact the limiting dynamic, thus serving as a sanity check.
%\end{remark}

\begin{proof}
Consider $\A=\e A$ and $\e^{-1} \nabla \hat R^p_{\eps}(\A)$ where $\nabla \hat R^p_\e(\A)$ admits the expansion \eqref{eq:R_grad_expansion}.
Recall
that $E_{\eps}^{(1)}$ and $E_{\eps}^{(2)}$ there are linear.
As discussed in Remark~\ref{rem:power_of_eps},
the terms  $O(\A^2)$ and $\e^{-2}O(\A^4)$ in \eqref{eq:R_grad_expansion} 
contribute  $\eps O(A^2)$ and $\eps O(A^4)$  respectively to $\e^{-1} \nabla \hat R^p_{\eps}(\e A)$ (recall our notation in Remark~\ref{rem:big-O}).
Since $\|A(t)\|_{L_\e^\infty} \lesssim r\e^{\eta}$ by Remark~\ref{rem:eA_small}, these terms are of order $\e^{2-2\kappa}r^4$ in $\CC^{0,T}_{0,\e} \hookrightarrow \CC^{-2\kappa,T}_{-1,\e}$.
Also upon identifying $\mfg^*$ with $\mfg$,
the term $E^{(1)}_{\eps}(\sum_{a=1}^4 \A_a)$ contributes  to $\e^{-1} \nabla \hat R^p_{\eps}(\e A)$ a term
$
\e E^{(1)}_{\eps}(\d_i A_j - \d_j A_i)
$,
which is of order $r\e^\kappa$ in $\CC^{-2\kappa,T}_{2\eta,\e}$  by~\eqref{eq:edA_to_zero}.

The only term in~\eqref{eq:R_grad_expansion} that remains to be considered is
$
 \e^{-3} E^{(2)}_{\eps}( \A_1+\A_2+\A_3+\A_4)^{\otimes 3}
% +
%  \e^{-3} E^{(2)}_{\eps}\Big( \A_1+\A_5+\A_6+\A_7\Big)^{\otimes 3}
$
(we have $\e^{-3}$ instead of $\e^{-2}$ as we are considering $\e^{-1} \nabla \hat R^p_{\eps}(\e A)$ instead of $\nabla \hat R^p_{\eps}(\e A)$).
%which will yield a mass renormalisation. 
%We rewrite them in the form of \eqref{e:epsE2}.
%
Denoting $F (p)= \A_1+\A_2+\A_3+\A_4= \e^2 (\partial_i^+ A_j (e^\northwest) - \partial_j^+ A_i (e))$,
%and $F (\bar p)= \A_1+\A_5+\A_6+\A_7$.
%Also write $(\bar 1,\bar 2,\bar 3, \bar 4)= (1,5,6,7)$.
we rewrite this term as
$
\sum_{m,n=1}^4 \e^{-3} E^{(2)}_{\eps} ( F(p) \otimes \A_m \otimes \A_n)
%+
% \e^{-3} E^{(2)}_{\eps} (F(\bar p) \otimes \A_{\bar m} \otimes \A_{\bar n}) \;.
$.
%Upon rescaling $\A=\e A$, we have
%$
%F(p) = \e^2 (\partial_i^+ A_j (e^\northwest) - \partial_j^+ A_i (e))
%$.
The conclusion follows from Lemma~\ref{lem:eps-A2-dA}
together with~\eqref{eq:edA_to_zero} which implies that only evaluations at $e$ and $e^{(a)}$ need to be considered.
\end{proof}

\begin{remark}\label{rem:sanity_check}
With a more careful calculation and natural symmetry assumptions on the coefficients $c_{23},c_{13}$ (e.g. invariance under flipping coordinates and reflections) it is possible to show that, in Lemma \ref{lem:I3},
$c^\e$ takes the block-diagonal form $c^\e = (c_1,c_2)$ with $c^{(1)}_1=c^{(2)}_2$ and $c^{(j)}_i=0$ for $i\neq j$,
and that in Lemma~\ref{lem:grad-R-hat},
$c^\e=0$.
Since
$\hat R^{\nabla}_{\eps}$ is the only term in the discrete equation~\eqref{e:Aeps} that
depends on $R_\e$,
this remark about $c^\e$ in Lemma~\ref{lem:grad-R-hat}
implies that
the choice of $R_\eps$ (i.e. the choice of lattice gauge theory) does not impact the limiting dynamic when we 
ultimately take $A$ as the solution to~\eqref{e:Aeps}.
We do not use these remarks going forward, however, as we give an alternative, more general argument in Section \ref{sec:gauge-covar}.
\end{remark}

\subsection{Renormalisation constants}
\label{sec:renorm_constants}

We now specify the values of the constants in~\eqref{eq:all_C}. % that we use for the rest of the section.
%We first specify $\bar{C}^\e, \hat{C}^{\e}_{1,\be},\cdots, \hat{C}^{\e}_5\in\mfg^{\otimes 2}$.
Write $\Psi_i=K^{i;\e} *\xi_i^\e$. For $e\in \obonds_i$,  $j\neq i$, $\be,\be'\in \CE_\times$, $\bw\in\{\Northwest,\Southwest\}$,
define (with $\otimes$ implicit)\footnote{These are 
the discrete analogues of $\bar C^{\eps},
\hat{C}^{\eps}$ in
\cite[Eq.~(6.12)]{CCHS_2D}.}
\begin{equs}
\hat{C}^{\e}_{1,\be} &\eqdef \E [ (K^{j;\e} * \bar\partial_j \Psi_i) (e+\be) \partial_j\Psi_i(e)]\;,
\quad
\hat{C}^{\e}_4 \eqdef \E [ \partial_i\Psi_i(e) (K^{i;\e} * \partial_i \Psi_i) (e) ]\;,
\\
\hat{C}^{\e}_{2,\be} &\eqdef \E [ \bar\partial_i\Psi_j(e) (K^{j;\e} * \partial_i \Psi_j) (e+\be) ]\;,
\quad
\hat{C}^{\e}_5 \eqdef \E [\Psi_i(e) ( \partial_i K^{i;\e} * \partial_i \Psi_i) (e) ]\;,
\\
%\hat{C}^{\e}_{3,\be} &\eqdef \E [ \Psi_j(e+\be) (\bar\partial_i K^{j;\e} * \partial_i \Psi_j) (e) ]\;,
\hat{C}^{\e}_{3,\be,\bw} &\eqdef \E [ \Psi_j(e+\be) (\partial_i^+ K^{j;\e} * \partial_i \Psi_j) (e^\bw) ]\;,
\quad
\bar{C}^\e_{\be\be'}  \eqdef % \frac{1}{16} \sum_{\be,\be'\in \CE_\times}
 \E [ \Psi_j (e+\be) \Psi_j (e+\be') ]\;.
 \end{equs}
Their values do not depend on the choice of $e\in \obonds_i$ or $i\in \{1,2\}$ or time variable.

\begin{remark}\label{rem:graphs-constants}
For a reader familiar with graphic notation, it is helpful to represent
$ \hat{C}^{\e}_{1,\be}$,
$\hat{C}^{\e}_{2,\be}$,
$\frac12 \sum_{\bw\in \{\southwest,\northwest\}}  \hat{C}^{\e}_{3,\be,\bw}$,
$\hat{C}^{\e}_4$,
 $\hat{C}^{\e}_5$,
$\bar{C}^\e_{\be\be'}$
as  
%\[
%\begin{tikzpicture}  [baseline=10,scale=0.7]		
%\node[notgreen] at (0,0) {};
%\node[notorange] at (-1,1) {};
%\draw[very thick] (0,0) -- (1,1);
%%\draw (-.2,.2) -- (-1,1);
%%\draw[thick] (-1+.2,1.2) -- (0,2);
%%\node[dot] at (1,1) {};
%\node[notgreen] at (0,2) {};
%%\draw[bend right =20, dotted]  (1,1) to (0,2) ;
%\end{tikzpicture}
%\]
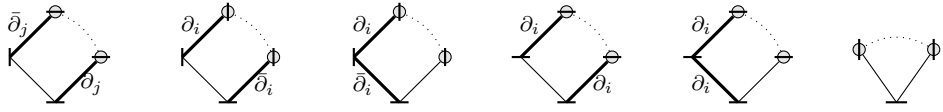
\[
\begin{tikzpicture}  [baseline=10,scale=0.6]		
\draw[thick]   (-.2, 0) -- (.2,0) ;
\draw[thick]   (-1,  1-.2) -- (-1,  1+.2) ;
\node at (1,1)  [var] (a)  {};  \draw[thick]   (1-.2, 1) -- (1.2,1) ;
\node at (0,2) [var] (b) {}; \draw[thick]   (-.2, 2) -- (.2,2) ;

\draw[very thick] (0,0) -- (a);
\draw (0,0) -- (-1,1);
\draw[very thick] (-1,1) -- (b);

\draw[bend right =20, dotted]  (a) to (b) ;
\node at (-0.8,1.7) {\scriptsize $\bar\partial_j$};
\node at (0.8,0.4) {\scriptsize $\partial_j$};
\end{tikzpicture}
\qquad
\begin{tikzpicture}  [baseline=10,scale=0.6]		
\draw[thick]   (-.2, 0) -- (.2,0) ;
\draw[thick]   (-1,  1-.2) -- (-1,  1+.2) ;
\node at (1,1)  [var] (a)  {};  \draw[thick]   (1, 1-.2) -- (1,1.2) ;
\node at (0,2) [var] (b) {}; \draw[thick]   (0, 2-.2) -- (0,2.2) ;

\draw[very thick] (0,0) -- (a);
\draw (0,0) -- (-1,1);
\draw[very thick] (-1,1) -- (b);

\draw[bend right =20, dotted]  (a) to (b) ;
\node at (-0.8,1.7) {\scriptsize $\partial_i$};
\node at (0.8,0.4) {\scriptsize $\bar\partial_i$};
\end{tikzpicture}
\qquad
\begin{tikzpicture}  [baseline=10,scale=0.6]		
\draw[thick]   (-.2, 0) -- (.2,0) ;
\draw[thick]   (-1,  1-.2) -- (-1,  1+.2) ;
\node at (1,1)  [var] (a)  {};  \draw[thick]   (1, 1-.2) -- (1,1.2) ;
\node at (0,2) [var] (b) {};   \draw[thick]   (0, 2-.2) -- (0,2.2) ;

\draw (0,0) -- (a);
\draw[very thick] (0,0) -- (-1,1);
\draw[very thick] (-1,1) -- (b);

\draw[bend right =20, dotted]  (a) to (b) ;
\node at (-0.8,1.7) {\scriptsize $\partial_i$};
\node at (-0.8,0.4) {\scriptsize $\bar\partial_i$};
\end{tikzpicture}
\qquad
\begin{tikzpicture}  [baseline=10,scale=0.6]		
\draw[thick]   (-.2, 0) -- (.2,0) ;
\draw[thick]   (-1-.2,  1) -- (-1+.2,  1) ;
\node at (1,1)  [var] (a)  {};  \draw[thick]   (1-.2, 1) -- (1.2,1) ;
\node at (0,2) [var] (b) {}; \draw[thick]   (-.2, 2) -- (.2,2) ;

\draw[very thick] (0,0) -- (a);
\draw (0,0) -- (-1,1);
\draw[very thick] (-1,1) -- (b);

\draw[bend right =20, dotted]  (a) to (b) ;
\node at (-0.8,1.7) {\scriptsize $\partial_i$};
\node at (0.8,0.4) {\scriptsize $\partial_i$};
\end{tikzpicture}
\qquad
\begin{tikzpicture}  [baseline=10,scale=0.6]		
\draw[thick]   (-.2, 0) -- (.2,0) ;
\draw[thick]   (-1-.2,  1) -- (-1+.2,  1) ;
\node at (1,1)  [var] (a)  {};  \draw[thick]   (1-.2, 1) -- (1.2,1) ;
\node at (0,2) [var] (b) {};   \draw[thick]   (-.2, 2) -- (.2,2) ;

\draw (0,0) -- (a);
\draw[very thick] (0,0) -- (-1,1);
\draw[very thick] (-1,1) -- (b);

\draw[bend right =20, dotted]  (a) to (b) ;
\node at (-0.8,1.7) {\scriptsize $\partial_i$};
\node at (-0.8,0.4) {\scriptsize $\partial_i$};
\end{tikzpicture}
\qquad
\begin{tikzpicture}  [baseline=10,scale=0.7]		
\draw[thick]   (-.2, 0) -- (.2,0) ;
%\draw[thick]   (-1-.2,  1) -- (-1+.2,  1) ;
\node at (.7,1)  [var] (a)  {};  \draw[thick]   (.7,1-.2) -- (.7,1.2) ;
\node at (-.7,1) [var] (b) {};   \draw[thick]   (-.7,1-.2) -- (-.7,1.2) ;

\draw (0,0) -- (a);
\draw (0,0) -- (b);

\draw[bend right =30, dotted]  (a) to (b) ;
\end{tikzpicture}
\]
Here we specified the ``orientation'' of each bond ($i$:horizontal; $j$:vertical).
Shifting by some $\be\in \CE_\times$ is implicitly applied when a line connects two bonds of different orientations
(unless the line carries a $\bar\partial$).
We do not precisely explain this graphic notation here
but we will explain in Section~\ref{subsec:moments_discrete_models}
where graphic notation is systematically used (without drawing orientations of bonds there).
\end{remark}
We further define, for $k\in\{1,2\}$,
\begin{equ}[e:def-C16]
\bar C^\e \eqdef \frac{1}{16}\sum_{\be,\be'\in \CE_\times}\bar C^\e_{\be\be'} \;,
\quad
\hat C^\e_{k} \eqdef \frac{1}{4}\sum_{\be\in \CE_\times}\hat C^\e_{k,\be} \;,
\quad
\hat C^\e_{3} \eqdef \frac{1}{8}\sum_{\be\in \CE_\times}\sum_{\bw\in\{\Northwest,\Southwest\}}\hat C^\e_{3,\be,\bw} \;.
\end{equ}
These are defined in such a way that,
 for the expressions in \eqref{e:M-on-YM} with $\CS$,
if we replace $\CS_{\be},\CS_{\be'},\CS_\bw$
 by the ``averages'' $\frac14\sum_{\be\in\CE_\times}\CS_{\be}$ and $\frac12\sum_{\bw\in\{\Northwest,\Southwest\}}\CS_{\bw}$,
then
 the map $M^\e$ contracts them into $\bar C^\e \bone^{(i)}$, $\hat C^\e_{k} \bone^{(i)}$ $(k=1,2,3)$.
We then define
\begin{equ}[e:def-CSYM]
C_{\sym}^{\eps} 
\eqdef 
%  4\hat{C}^{\bar\eps} 
2 \hat{C}^{\e}_1
+2 \hat{C}^{\e}_2
-2 \hat{C}^{\e}_3
- \hat{C}^{\e}_4
+\hat{C}^{\e}_5
 - \bar C^{\eps} \;.
\end{equ}
%
%For the `superficial' renormalisation, we define,
For $e\in \obonds_i$,  $j\neq i$, $\be\in \CE_\times$, we set 
%\begin{equ}[e:def-C67etc]
$ \tilde C_{\be} 
\eqdef
\E[ \Psi_j (e+\be) \bar\partial_i \Psi_j(e)]$.
%\end{equ}
 %
Moreover, recalling $K^\e (0)=P^\e(0) =\e^{-2}$ (since the heat kernel on $\Z^2$ equals $1_{x=0}$ at $t=0$), we define
\begin{equs}[e:C-XiIXi]
C^\e [ \,\<PsiXi1>\, ] 
&\eqdef 
\e^{1-\kappa} \E [\xi^\e_i (e)\Psi_i (e)]
= \e^{1-\kappa} K^\e (0)\Cas = \e^{-1-\kappa}\Cas\;,
\\
C^\e [ \,\<I(PsiXi1)Xi1less>\, ] 
&\eqdef 
\e^{2-2\kappa} \E [\xi^\e_i (e)\Psi_i (e)]
= \e^{2-2\kappa} K^\e (0)\Cas = \e^{-2\kappa}\Cas\;,
\end{equs}
where $\Cas\in \mfg\otimes\mfg$ is the quadratic Casimir, i.e. $\Cas=\sum_l e_l \otimes e_l$ for an orthonormal basis $(e_l)_l$ of $\mfg$.
For trees of the form
$\<I'XiI'[IXiI'Xi]less>$,
we also define, for $e\in \obonds_i$, $j\neq i$, and $\bs\in\{\Southeast,\Southwest\}$, the following $6$ renormalisation constants 
(recall that $\CD_j^{\pm}$ includes two cases $\CD_j^{+},\CD_j^{-}$) as appeared in \eqref{e:div-R2}:
\minilab{e:def-C-triangle}
\begin{equs}
{} &  
 C_{\bs,\pm}^\e[ \, \<I'XiI'[IXiI'Xi]less>\,]  
 \eqdef
  \e^{1-\kappa}
\E [\partial_j^{\pm} \Psi_i (e) (\partial_j^+ K^{j;\e} * \bar\partial_j \Psi_i )(e^{\bs}) ] \;,		\label{e:def-C-triangle1}
\\
& 
 C_{\southwest}^\e[ \, \<I'XiI'[IXiI'Xi]less>\,]  
\eqdef
  \e^{1-\kappa}
\E[\partial_j^+ \Psi_j (e^{\southeast})  (\partial_j^+ K^{j;\e} * \partial_i \Psi_j) (e^{\southwest})  ]\;,	\label{e:def-C-triangle2}
\\
 & 
 C_{\southeast}^\e[ \, \<I'XiI'[IXiI'Xi]less>\,]  
 \eqdef
  \e^{1-\kappa}
 \E [\partial_j^+ \Psi_j (e^{\southwest}) ( \partial_j^+ K^{j;\e}* \partial_i \Psi_j )(e^{\southeast}) ]\;.	  \label{e:def-C-triangle3}
\end{equs}
We will write
\begin{equs}[e:def-CREM]
C_{\rem}^\e \eqdef 
-\frac{\e^{2\kappa}}{6}  C^\e [ \,\<I(PsiXi1)Xi1less>\, ] 
- \frac{\e^\kappa}{4} 
(
&C_{\southeast,+}^\e [\, \<I'XiI'[IXiI'Xi]less>\,]
+
C_{\southeast,-}^\e [\, \<I'XiI'[IXiI'Xi]less>\,] 
+
C_{\southwest,+}^\e [\, \<I'XiI'[IXiI'Xi]less>\,]
\\
&+
C_{\southwest,-}^\e [\, \<I'XiI'[IXiI'Xi]less>\,]
% +C_2^\e [\, \<I'XiI'[IXiI'Xi]less>\,]
-2C_\southwest^\e [\, \<I'XiI'[IXiI'Xi]less>\,]+2C_\southeast^\e [\, \<I'XiI'[IXiI'Xi]less>\,] )\;.
\end{equs}
Here, the first term is equal to $-\frac16 \Cas$ by \eqref{e:C-XiIXi}.

\begin{remark}\label{rem:C_non_vanish}
Expressions in \eqref{e:def-C-triangle}  only have ``approximate'' parity property.
For instance,   \eqref{e:def-C-triangle2} is equal to $\Cas$ times 
\begin{equ}[e:def-C-triang]
%C^\e [\, \<I'XiI'[IXiI'Xi]less>\,]
%=
\e^{1-\kappa} 
\int_{(\R\times \obonds_j)^2}
\partial_j^+ K^\e(z^\southeast -y) 
\partial_j^+ K^\e(z^\southwest -w)
\partial_i K^\e(w-y)
\mrd w \mrd y
\end{equ}
where $z\in \R\times \obonds_i$, which does {\it not} vanish.
(It would vanish if $z^\southeast,z^\southwest$ were replaced by $z\in \obonds_j$ as can be seen by flipping the sign of the $i$th coordinates).
\end{remark}
Furthermore, writing $(e_\ell)_{\ell}$ for an orthonormal basis of $\mfg$, 
we define 
\begin{equs}[e:def-C-four]{}
&c^\e_1 [\<r_z_large>] \eqdef \e^{-2\kappa} \int_{\R\times (\e \Z)^2} K^\e(w)^2\mrd w  \;,
\quad
c^\e_2 [\<r_z_large>]  \eqdef \e^{2-2\kappa} \int_{\R\times (\e \Z)^2} K^\e(w)^3\mrd w \;,
\\
&C^\e [\<r_z_large>]
\eqdef  \sum_{j,k} \Big\{c^\e_1 [\<r_z_large>]\,e_j\otimes e_k\otimes e_j\otimes e_k
+
c^\e_2 [\<r_z_large>] \, e_j\otimes e_k\otimes e_k\otimes e_j \Big\}\;,
\end{equs}
where $C^\e [\<r_z_large>]\in \mfg^{\otimes 4}$.
Finally, for $\be\in\CE_\times$ and $\bw\in\{\Northwest,\Southwest\}$, we define
\begin{equs}[eq:C_square_cherries]
C^\e_\be [\<I'XiIXibar>] &\eqdef \e^{1-\kappa} \E [(\bar\d_i \Psi_j)(e) \Psi_j(e+\be)] \;,
\\
C^\e_{\be,\bw}[\<IXiI'Xibar>] &\eqdef 
\e^{1-\kappa} \E [\Psi_j(e+\be) \d_i^+\Psi_j(e^\bw) ]\;,
\\
C^\e [\<R2-1new>] &\eqdef 
\e^{1-2\kappa} \E [\d_j^+ \Psi_j(e^\southwest) \d_j^+\Psi_j(e^\southeast) ]\;,
\\
C^\e [\<cherry232>]  &\eqdef 
\e^{2-2\kappa} \E [\d_j^+ \Psi_j(e^\southwest) \d_j^+\Psi_j(e^\southeast) ]\;.
\end{equs} 
Some of these definitions of renormalisation constants may seem somewhat arbitrary at first, but their purposes will become clear in Section~\ref{subsec:moments_discrete_models}; for instance, $c^\e_1 [\<r_z_large>] $, $c^\e_2 [\<r_z_large>] $ are introduced to cancel the divergent terms in \eqref{e:c1c2-cancel}.

\subsubsection{Behaviour of renormalisation constants}
\label{subsubsec:behaviour}

For $k=(k_0,k_1,k_2)\in\N^3$, 
 we write  $D^k_\eps\eqdef \partial_t^{k_0} (\partial_1^+)^{k_1} (\partial_2^+)^{k_2}$ where $\partial_i^+$
 are the discrete derivatives as in Section~\ref{sec:Notation}. 
For a family of functions $F^\e$ on $\R \times \T_\e^2$ supported in a ball around the origin, we say that it is of order $\zeta\in\R$ if for some $m\in\N$ the quantity (\cite[(5.16)]{Shen18}, \cite[Def.~5.7]{CM18})
\begin{equ}[e:DSingKer]
\VERT F^\e \VERT_{\zeta;m}^{(\eps)}
\eqdef \max_{|k|_\s\leq m}\sup_{z\in \R \times \T_\e^2} \frac{\left| D^k_\e F^\e(z)\right|}{\|z\|_\e^{\zeta-|k|_\s}}
\qquad
(\mbox{where } \|z\|_\e \eqdef \|z\|_\s \vee \e)
\end{equ}
is bounded uniformly in $\eps$. 
In our case, the truncated heat kernel $K^\e$  is of order $-2$.

Let $\Psi_i=K^{i;\e} *_{(i)}\xi_i^\e$ and recall that $\Psi_1,\Psi_2$ are independent. We will often encounter constants of the  form 
``$\e \E [\Psi\partial\Psi]$''
as discussed in the next lemma, where $O(1)$ is a bounded element of $\mfg^{\otimes 2}$ uniformly in $\e$.

\begin{lemma}\label{lem:ePsiDPsi}
Let $j,i\in \{1,2\}$, and $a>1$.
Then, 
\begin{equs}[eq:PsiDPsi]
\e \, \E [\Psi_{i}(e_1)\partial^\pm_j \Psi_{i}(e_2)] =  O(1)\;,
\qquad
e_1, e_2\in \obonds_{i} \;,
\\
\e \, \E [\Psi_{i}(e_1)\partial_j \Psi_{i}(e_2)] = 1_{e_1\neq e_2} O(1) \;,
\qquad
e_1, e_2 \in \obonds_{i}   \;,
\\
\e \, \E [\Psi_{i}(e_1)\bar\partial_j \Psi_{i}(e_2)] = O(1)\;,
\qquad
e_1\in \obonds_{i},e_2\in \obonds_{3-i}  \;,
\end{equs}
uniformly in $\eps\in (0,1)$ and $|e_1 - e_2| \le  a \e$ (where we identify the midpoint of $e_i $ as a point in  $\T^2$).
%Moreover, the second quantity is $0$ if $e_1=e_2$.
%Also, if one of the $\Psi_i$ is replaced by $\Psi_{i'}$ for $i'\neq i$ then the expectation is zero.
Also, the last expectation vanishes if $\Psi_i(e_1)$ is replaced by $\sum_{\be \in \CE_\times} \Psi_i(e_2+\be)$.
Here all the functions are evaluated at the same time variable which is omitted.
\end{lemma}

\begin{proof}
%Writing $\Psi=\Psi_i$ and $\int=\int_{\R\times\T^2_\e}$ for simplicity,
For $e_1, e_2\in \obonds_{i}$, $\e \E [\Psi_i (e_1) \partial^+_j \Psi_i (e_2)] $ is equal to an integral over
$\R\times (\e \Z)^2$ where the integrand is bounded by  $\e  (\|z\|_\s \vee \e)^{-5}$  for $\|z\|_\s\lesssim 1$
and vanishes for $\|z\|_\s \gg 1$.
%\[
%\e \E [\Psi_i (e_1) \partial^+_j \Psi_i (e_2)] 
%= \Cas\int_{\R\times (\e \Z)^2}  \e K^\e ( (0,e_1-e_2)+z) \, \partial_j^+ K^\e (z)  \mrd z   \;.
%\]
%Since $K$ is of order $-2$ in the sense of \eqref{e:DSingKer}       
%it is easy to see (for instance by Lemma~\ref{lem:shift-K} below) that 
%the integrand is bounded by  $\e  (\|z\|_\s \vee \e)^{-5}$ for $\|z\|_\s\lesssim 1$, so the claim follows.
The same bound holds for  $\partial_j^-, \partial_j, \bar\partial_j$.

For the symmetrised derivative $\partial_j$ and $e\in \obonds_i$,
$ \E [\Psi_i (e) \partial_j \Psi_i(e)]$ is equal to
$ \Cas\int_{\R\times (\e \Z)^2} K^\e(z)  \partial_j K^\e(z) \mrd z   =0$
  since the integrand is odd by  Remark~\ref{rem:loss-odd}.
%summing by parts (or Remark~\ref{rem:loss-odd})
%\[
%\e \E [\Psi \partial_j \Psi] 
%= \e \int K(z)  \partial_j K(z) \mrd z
%= - \e\int \partial_j K(z)  K(z) \mrd z=0 \;.
%\]

For the claim about the term with $\bar\partial_j$ being $0$,
for $e\in \obonds_{3-i} $,
\begin{equs}
\sum_{\be\in\CE_\times} \E [\Psi_{i}(e+\be)\bar\partial_j \Psi_{i}(e)] 
&=
\Cas\int_{\R\times \obonds_i} \sum_{\be\in\CE_\times}  K^{i;\e} (e+\be-\bar e) 
\bar\partial_j  K^{i;\e}(e-\bar e) \mrd \bar e 
\end{equs}
which vanishes
since the integrand is odd in the $j$th coordinate of $e-\bar e$.
\end{proof}

\begin{remark}
\label{rem:trans_invar}
By translation-invariance, the terms on the left-hand side of~\eqref{eq:PsiDPsi} depend only on $e_1-e_2$.
If $K^\e$ is replaced by the full heat kernel $P^\e$, one may find their exact values,
for instance by elementary properties of $P^\e$, one has $\e \E [\Psi(e) \partial^+_j \Psi(e)] = -\frac18\Cas$. We will not need these values though.
\end{remark}

%\begin{remark}
%If $K^\e$ is replaced by the full heat kernel $P^\e$, and $e_1=e_2$, the value of $\e \, \E [\Psi_{i}(e_1)\partial^+_j \Psi_{i}(e_2)]$
%can be calculated explicitly (and is non-zero). Indeed summation by parts yields $ \int P^\e(z)  \partial_j^+ P^\e(z) \mrd z
%=\int \partial_j^- P^\e(z)  P^\e(z) \mrd z$.
%Therefore $\e \E [\Psi \partial^+_j \Psi]$ is $\Cas$ times
%\begin{equs}{}
%&\frac12 \int P^\e(z)  \Big(P(z+\e_j) +P^\e(z-\e_j) - 2 P^\e(z)\Big) \mrd z
%%=\frac{\eps^2}{2d} \int P^\e(t,x) 
%%\Delta P^\e(t,x) \mrd t\mrd x
%\\
%&=
%\frac{\eps^2}{4d} \int \partial_t(P^\e(t,x)^2) \mrd t \mrd x
%=
%-\frac{\eps^2}{4d} \int  P^\e(0,x)^2 \mrd x
%= 
%-\frac{1}{4d\e^{d-2}}
%=
%-\frac{1}{8}\;.
%\end{equs}
%We will not use this explicit value, but 
%we only point out that the parity symmetry may be lost on lattice if the derivative has a ``biased'' direction.
%\end{remark}
%
As a consequence of Lemma~\ref{lem:ePsiDPsi}, 
for the constants in \eqref{eq:C_square_cherries} one has
\begin{equ}[eq:C_be_vanish]
\sum_{\be\in \CE_\times} C^\e_\be [\<I'XiIXibar>] = 0\;.
\end{equ}
Furthermore, we have $C^\e_{\be,\northwest}[\<IXiI'Xibar>]+C^\e_{\be,\southwest}[\<IXiI'Xibar>] = 2\e^{1-\kappa} \E [\Psi_j(e+\be) \bar\d_i\Psi_j(e) ]$,
and thus Lemma~\ref{lem:ePsiDPsi} again implies
\begin{equ}[eq:C_bs_be_vanish]
\sum_{\be\in \CE_\times}\sum_{\bw\in\{\Northwest,\Southwest\}} C^\e_{\be,\bw}[\<IXiI'Xibar>] = 0\;.
\end{equ}
Lemma~\ref{lem:ePsiDPsi} also implies $
\sum_{\be \in \CE_\times} \tilde C_{\be} =0$,
as claimed in Remark~\ref{rem:superificial_renorm}, but we will not use this.

The next lemma will be used in multiple places. For any kernel $K$ we will write $K^{*2} \eqdef K*\hat K$
 where $\hat K$ the space-time reflection: $\hat K(z)=K(-z)$.
\begin{lemma} \label{lem:identity}
Let $P^\e$ be the heat kernel on $\R\times (\e\Z)^d$,
namely $\partial_t P^\e = \Delta P^\e$ with $P^\e_0(x)= \e^{-d}\mathbf{1}_{x=0}$ and $P^\e_t(x)=0$ for $t<0$. Then 
for every $j\in [d]$, $s,s'\in\R$ and $y,y'\in (\e\Z)^d$, 
\begin{equ} \label{e:d-dim-id}
2\sum_{x\in (\e\Z)^d} \e^d
\sum_{j=1}^d
  \int_{\R}
	 \partial_j^+ P^\e_{t+s} (x+y) \partial_j^+ P^\e_{t+s'} (x+y') \mrd t
= P^\e_{|s-s'|}(y-y') \;. %+p_{s'-s}(y-y') \;.
\end{equ}
\end{lemma}

\begin{proof}
By scaling it suffices to prove the case $\e=1$ and we write $P=P^\e$.
Let $\mathscr F_x,\mathscr F_t ,\mathscr F $ denote the Fourier transform operators
in space, time, and space-time respectively.
Since  $\Delta e^{ik\cdot x} = \lambda_k e^{ik\cdot x}$
with
$\lambda_k \eqdef 2 \sum_{n=1}^d (\cos k_n-1)$, 
by definition of $P$,
we have
$(\mathscr F P)\, (\omega,k)    = (-\lambda_k+ i\omega)^{-1}$.
The left-hand side of \eqref{e:d-dim-id} equals
$2\sum_j (\nabla_j^+ P )^{*2}_{s-s'}(y-y') $
whose Fourier transform
at $(\omega,k)$ is equal to
\begin{equ}[e:idenLHS]
2\sum_{j=1}^d\Big| \frac{e^{ik_j}-1}{-\lambda_k+ i\omega}   \Big|^2
= 
 \frac{4 \sum_{j=1}^d (1-\cos k_j)}{\lambda_k^2+ \omega^2} 	\;.
\end{equ}
For the right-hand side of \eqref{e:d-dim-id}, one has
$
(\mathscr F_x P_{|t|})\, (k)  = e^{\lambda_k |t|}
	%=e^{-\frac{|t|}{d} \sum_{n=1}^d (1-\cos k_n)}
$.
Recall the fact that for any $a>0$,
$\mathscr F_t e^{-a|t|} =\frac{2a}{a^2+\omega^2}$.
So
$\mathscr F_t (\mathscr F_x P_{|\cdot|}(\cdot))\, (\omega,k)  = \frac{-2\lambda_k}{\lambda_k^2+\omega^2} =\eqref{e:idenLHS} $.
\end{proof}

\begin{lemma}\label{lem:C_sym-finite}
$C_{\sym}^{\eps}=O(1)$ %converges as $\e\to 0$. 	
uniformly in $\e\in(0,1)$, where $C_{\sym}^{\eps}$ is as in \eqref{e:def-CSYM}.
\end{lemma}

\begin{proof}
The continuous version can be found in~\cite[Lem.~6.9]{CCHS_2D}, so we give a brief proof.
First, 
we can write $\bar{C}^\e$ as 
$\E [ \Psi_j (e) \Psi_j (e) ]$ plus an error of the form 
``$\e \E [\Psi\partial\Psi]$''
which falls into the scope of Lemma~\ref{lem:ePsiDPsi}.
We then write this as 
$\bar{C}^\e_{\approx} \eqdef \Cas\int_{\R\times (\e\Z)^2} K^\e(z) P^\e(z) \mrd z$
plus a converging error. 
Note that replacing $K^\e$ by $P^\e$ in these integrals  yields
errors which are integrals of (discrete) smooth functions
so these errors converge.
%Also note that $P^\e(t,x) \lesssim t^{-d/2}=t^{-1}$ for $|x|\lesssim 1$ so $(P^\e)^2$ is integrable for each $\e>0$.
%
%\ilyaText{but is it then obvious what to do for $|x|\gg 1$?
%Even for the heat kernel $P$ on $\R\times\R^2$,
%we have $P(x,t)\sim t^{-1}$ for $|x|\lesssim\sqrt t$.
%But for fixed $t$, $\{x:|x|\lesssim\sqrt t\}$ has volume $\sim t$, so
%$\int_{|x|\lesssim\sqrt t} P(x,t)^2\mrd x \sim t^{-1}$.
%And then we have to integrate in $t$,
%which diverges for large $t$ like a $\log$. (??)}
%
%\haoText{I think you're right that there's some infrared problem.
%For instance $P(z)^2 \sim \frac{1}{t^2} e^{-x^2/t}$, and 
%$\int \frac{1}{t} e^{-x^2/t} dx$ is basically $1$, we're left with $1/t$ which is not integrable.
%
%We can simply change to $\bar{C}^\e_{\approx} \eqdef \Cas\int_{\R\times (\e\Z)^2} K^\e(z) P^\e(z) \mrd z$ and 
%\[
%\hat C^\e_{\approx} \eqdef \Cas\int_{\R\times (\e\Z)^2} K^\e (z) (\partial^+_j P^\e)^{*2}(z)\mrd z
%\]
% and everything should be okay then?}

Likewise, we can write $\hat C^\e_1$, $\hat C^\e_2$, $\hat C^\e_4$ as 
$\hat C^\e_{\approx} \eqdef \Cas\int_{\R\times (\e\Z)^2} K^\e (z) (\partial^+_j P^\e)^{*2}(z)\mrd z $
 plus a converging error
and an error of the form 
``$\e \E [\Psi\partial\Psi]$''
which again falls into scope of Lemma~\ref{lem:ePsiDPsi}.
For $\hat C^\e_3$, $\hat C^\e_5$,
by summation by parts, we can write them as $-\hat C^\e_{\approx}$, up to converging errors.
(These are easier to see by observing the graphs in Remark~\ref{rem:graphs-constants}.)
By Lemma~\ref{lem:identity}, since $\hat C^\e_{\approx}$ does not depend on $j$,
one has $\hat C^\e_{\approx} = \frac14 \bar{C}^\e_{\approx}$.
Therefore,
%$C_{\sym}^{\eps} $ defined in \eqref{e:def-CSYM} equals
%\begin{equ}
%2 \hat{C}^{\e}_1
%& +2 \hat{C}^{\e}_2
%-2 \hat{C}^{\e}_3
%- \hat{C}^{\e}_4
%+\hat{C}^{\e}_5
% - \bar C^{\eps} 
%\\ 
%&=
\eqref{e:def-CSYM} equals
$
\frac14 \big(
2+2+2-1-1
\big)  \bar{C}^\e_{\approx} -  \bar{C}^\e_{\approx} + c_\e =c_\e
$
%\end{equ}
for some $c_\e=O(1)$. %which converge as $\e\to 0$.
\end{proof}

\begin{lemma}\label{lem:thick-trig-bahave}
$\e^\kappa C_{\bs,\pm}^\e [\, \<I'XiI'[IXiI'Xi]less>\,] = O(1)$
and
$\e^\kappa C_\bs^\e [\, \<I'XiI'[IXiI'Xi]less>\,] = O(1)$ 
%converges as $\e\to 0$,
 for every $\bs\in \{\Southeast,\Southwest\}$.
 In particular $C_{\rem}^\e $ defined in 
 \eqref{e:def-CREM}  is $O(1)$.
% converges  as $\e\to 0$.
\end{lemma}

\begin{proof}
This  follows from the fact that $\partial K$ has degree $-3$, 
Lemma~\ref{lem:OpDSingKer} below, and
$\int \e  (\|z\|_\s \vee \e)^{-5} \mrd z \lesssim 1$.
%It alternatively follows from Lemmas~\ref{lem:ePsiDPsi} and~\ref{lem:identity}.
%We give an alternative proof which is of independent interest.
%In fact the constant is essentially the same as in Lemma~\ref{lem:ePsiDPsi} in disguise.
%Take \eqref{e:def-C-triang} for example.
%By translation invariance, we instead fix $w$ and integrate $z,y$.
%Applying Lemma~\ref{lem:identity} to the convolution in $z$,
%we see that $\e^\kappa C_\southwest^\e [\, \<I'XiI'[IXiI'Xi]less>\,]$ is equal to, up to a bounded part
%caused by replacing $K^\e$ by $P^\e$ in \eqref{e:def-C-triang},
%\begin{equ}
%\Cas\,\e
%\int_{\R\times \obonds_j}
%K^\e (w-y+\e_i)
%\partial_i K^\e(w-y)
%\mrd w \mrd y  
%\end{equ}
%which falls into scope of Lemma~\ref{lem:ePsiDPsi}. The same arguments apply to the other constants in \eqref{e:def-C-triangle}.
\end{proof}
Finally, we note that the definitions~\eqref{e:def-C-four} and~\eqref{eq:C_square_cherries} imply the bounds
\begin{equs}[eq:C_bounds_large_cherries]
\e^{3\kappa}c^\e_k [\<r_z_large>] =o(1)\;,\quad \e^{\kappa}C^\e_\be [\<I'XiIXibar>]=O(1)\;,\quad \e^{\kappa}C^\e_{\be,\bw}[\<IXiI'Xibar>]  =O(1)
\\
\e^{1+2\kappa} C^\e [\<R2-1new>] =O(1) \;,
\qquad
\e^{2\kappa} C^\e [\<cherry232>] = O(1)
\end{equs}
for all $k\in\{1,2\}$, $\be\in\CE_\times$, and $\bw\in\{\Northwest,\Southwest\}$.

The next lemma shows that %the processes $\Psi_i=K^{i;\e} *_{(i)}\xi_i^\e$ 
$\Psi_i$ satisfy Assumption~\ref{as:models_abstract} with high probability.

\begin{lemma}\label{lem:Psi_Wick}
For every $p\geq 1$, uniformly in $\eps$,
\begin{equ}[eq:Psi_p_bound]
\E \sup_{t\in[-1,2]} \|\Psi_i(t)\|_{\CC^{-\kappa}_\e}^p  \lesssim 1\;.
\end{equ}
Furthermore, let $\partial, i_1,i_2,i_3, h_2, h_3$ be as in Assumption~\ref{as:models_abstract}.
Then there exist  $c_{23},c_{13}\in \mfg\otimes\mfg$ (depending on $\partial, i_k,
h_k,\e$) such that\footnote{In some cases $c_{13}$ or $c_{23}$ may vanish, but we do not need this.}
$c_{23}=O(1),c_{13}=O(1)$ uniformly in $\eps$ and
$\E [\textnormal{left-hand side of } \eqref{e:assump-Psi-2}] \lesssim \e^{\kappa/2}$.  
%\begin{equs}
%\E \sup_{t\in[1,2]}  
%\{\| W_{c,h} (\Psi_{i_1}, \Psi_{i_2}, \partial\Psi_{i_3})(t)\|_{\CC^{-\kappa}_\eps}
%&+
%\| W^{(13)}_{c,h} (\Psi_{i_1}, \partial\Psi_{i_3})(t)\|_{\CC^{-\kappa}_\eps}
%\\
%&+
%\| W^{(23)}_{c,h} (\Psi_{i_2}, \partial\Psi_{i_3})(t)\|_{\CC^{-\kappa}_\eps}\}
%\lesssim \e^{\kappa/2}\;.
%\end{equs}
\end{lemma}

%\begin{remark}
%In Lemma~\ref{lem:Psi_Wick},
%one can show that, when $\partial\in \{\partial_j,\bar\partial_j\}$, one has $c_{23}=0$ if $h_2=h_3$  and $c_{13}=0$ if $h_3=0$, but this will not be used later.
%\end{remark}

\begin{proof}
The first statement is standard (see, e.g. Proposition~\ref{prop:SHE} for a stronger bound).
For the second statement,
as in the proof of Lemma~\ref{lem:eps-A2-dA}, we write $\Psi_{i_1} \Psi_{i_2}\partial \Psi_{i_3}$ or simply
$\Psi^2 \partial \Psi$ as shorthand for the tensor product.
Recall that in the definition of the discrete spaces $\CC^\alpha_\eps$ with $\alpha<0$,
the scaling parameter $\lambda$ is restricted to $\lambda\in [\e,1]$.

\noindent
1. Consider first
$
W_{c,h} (\Psi_{i_1}, \Psi_{i_2}, \partial\Psi_{i_3}) = \e \Psi_{i_1} \Psi_{i_2} \partial \Psi_{i_3} - \Psi_{i_1}c_{23} -c_1^{(\ell)} \Psi_{i_2} c_3^{(\ell)}
$,
which we decompose into third and first chaos.

\begin{enumerate}[label=1(\alph*)]
\item\label{pt:Psi2dPsi_Wick}
The second moment of the 3rd chaos $\e \Wick{\Psi^2 \partial \Psi}$ integrated against a spatial rescaled test function
$\phi^\lambda$ is easily bounded by $\e^2 \lambda^{-2} \le \e^{\kappa/2} \lambda^{-\kappa/2}$.
Combined with a suitable Kolmogorov criterion (in both time and space), it follows that $\E\sup_{t\in [1,2]} \|\e \Wick{\Psi^2 \partial \Psi}(t)\|_{\CC_\e^{-\kappa}} \lesssim \eps^{\kappa/2}$.
\item
By Lemma~\ref{lem:eps-improve-reg}\ref{pt:eps_d_f} and~\eqref{eq:Psi_p_bound} (applied to $\kappa/4$), and using $\E\Psi^2 \lesssim \eps^{-\kappa/4}$,
\begin{equ}
\E \sup_{t\in[-1,2]}\|\e \E[\Psi^2] \partial \Psi(t)\|_{\CC_\e^{-\kappa}} \lesssim\e^{\kappa/2}\;.
\end{equ}
\item\label{pt:EPsidPsi}
The other two terms in the first chaos are
zero provided we choose $c_{13} = \e \E[\Psi_{i_1}  \partial \Psi_{i_3}]$
and $c_{23} = \e \E[\Psi_{i_2}  \partial \Psi_{i_3}]$.
% $
%$ \e \E[\Psi_{i_1}  \partial \Psi_{i_3}] \Psi_{i_2} - c_1^{(\ell)} \Psi_{i_2} c_3^{(\ell)}$ and
% $
% \e \E[\Psi_{i_2}  \partial \Psi_{i_3}] \Psi_{i_1} - \Psi_{i_1}c_{23}
% $, which
By Lemma~\ref{lem:ePsiDPsi}, $c_{23},c_{13} = O(1)$,
and by Remark~\ref{rem:trans_invar}, $c_{23}, c_{13}$ only depend on $\partial, i_k, h_k,\e$.
\end{enumerate}

\noindent
2. Consider now $W^{(13)}_{c,h} (\Psi_{i_1}, \partial\Psi_{i_3}) = \e\Psi_{i_1}\d \Psi_{i_3} - c_{13}$
with
$\Psi\partial \Psi =\Wick{\Psi\partial \Psi}+\E[\Psi\partial \Psi]$. Then
$\E[\Psi  \partial \Psi] - c_{13} = 0$ for the same $c_{13}$ as in~\ref{pt:EPsidPsi},
while
by a similar argument as in~\ref{pt:Psi2dPsi_Wick}, $
\E\sup_{t\in[-1,2]}\|\e \Wick{\Psi  \partial \Psi}(t)\|_{\CC_\e^{-\kappa}}\lesssim \e^{\kappa/2}$.

\noindent
3. The term $W^{(23)}_{c,h} (\Psi_{i_2}, \partial\Psi_{i_3})$ is handled in an identical way.
\end{proof}

\subsection{Fixed point problem}
\label{subsec:FPP}

Fix throughout this subsection (discrete) compatible models $Z^\e=(\Pi^\e,\Gamma^\e)$ (see Definition~\ref{def:compat}) indexed by $\e=2^{-N}$
on the regularity structure $\CT$ from Section~\ref{sec:tr-reg} with a
reconstruction operator $\CR$ as in Section~\ref{sec:reconstruction}. 
Denote by $(e_\ell)_{\ell}$ an orthonormal basis of $\mfg$.
We write $\bXi_i,\bar\bXi_i\colon \R\times\T^2 \to\CT\otimes \mfg$ for the modelled distributions 
which take the constant values
$
\id = \sum_\ell e_\ell^*\otimes e_\ell
$ as elements of 
$
\CT[\Xi_i]\otimes \mfg\simeq \mfg^*\otimes \mfg
$
and
$\CT[\bar\Xi_i]\otimes \mfg\simeq \mfg^*\otimes \mfg
$ 
respectively.
We also write $\bPsi_i =\mcb{I}_i \bXi_i$, $\bar\bPsi_i =\mcb{I}_i \bar\bXi_i$,
also understood as constant modelled distributions equal to $\id$ in the appropriate spaces.
%Recall $P^{i;\e} = K^{i;\e}+R^{i;\e}$ from Section~\ref{sec:Admissible models}
%and 
%$\CK^{i;\eps}$ from \eqref{eq:conv}.
%Write $\CP^{i;\eps}  = \CK^{i;\eps}+R^{i;\eps}*_{(i)} \CR$.

For $A\in\mfq_i$ we define $\delta_i A\in\mfq_i$ as follows.
For $e\in \obonds_i$, 
\begin{equs}[eq:delta_A_def]{}
&(\delta_i A) (e) 
\eqdef 
 \frac{1}{8} \!\! \sum_{\be,\be'\in \CE_\times} \!\!   
 (\ad_{\hat{C}^\e_{1,\be}}+\ad_{\hat{C}^\e_{2,\be}})
  (A(e+\be+\be')-A(e))
\\
&\quad
-\frac{1}{16} \sum_{\be,\be'\in \CE_\times}\sum_{\bw\in\{\Northwest,\Southwest\}} \ad_{\hat{C}^\e_{3,\be,\bw}}(A(e^\bw + \be')-A(e))
%\\
%&\quad
%-
%\frac{\e^\kappa}{16}\sum_{\be\in\CE_\times} (
%\ad_{C_{\southeast,+}^\e [\, \<I'XiI'[IXiI'Xi]less>\,]}
%+
%\ad_{C_{\southeast,-}^\e [\, \<I'XiI'[IXiI'Xi]less>\,]}
%+2\ad_{C_\southeast^\e [\, \<I'XiI'[IXiI'Xi]less>\,]} 
%)
%(A(e^{\southeast}+\be) - A(e))
%\\
%&\quad
%-
%\frac{\e^\kappa}{16}\sum_{\be\in\CE_\times} (
%\ad_{C_{\southwest,+}^\e [\, \<I'XiI'[IXiI'Xi]less>\,]}
%+
%\ad_{C_{\southwest,-}^\e [\, \<I'XiI'[IXiI'Xi]less>\,]}
%-2\ad_{C_\southwest^\e [\, \<I'XiI'[IXiI'Xi]less>\,]} 
%)
%(A(e^{\southwest}+\be) - A(e))\;,
\\
&\quad
-
\frac{\e^\kappa}{16}\sum_{\be\in\CE_\times}\sum_{\bs\in\{\southeast,\southwest\}} (
\ad_{C_{\bs,+}^\e [\, \<I'XiI'[IXiI'Xi]less>\,]}
+
\ad_{C_{\bs,-}^\e [\, \<I'XiI'[IXiI'Xi]less>\,]}
\pm 2 \ad_{C_\bs^\e [\, \<I'XiI'[IXiI'Xi]less>\,]} 
)
(A(e^{\bs}+\be) - A(e))
\end{equs}
where $\pm$ is $+$ for $\bs=\Southeast$ and $-$ for $\bs=\Southwest$,
and we extend $\ad$ to the universal enveloping algebra of $\mfg$, i.e. for $X=X_1\otimes\cdots\otimes X_k \in\mfg^{\otimes k}$ we define $\ad_{X}\in L(\mfg,\mfg)$ by $\ad_X = [X_1,[\ldots,[X_k,\cdot]]]$.
For $A\in\mfq_j$ and $i\neq j$, we also define $\bar \delta_i A \in\mfq_i$ by
\begin{equ}[eq:bar_delta_def]
(\bar \delta_i A) (e) 
\eqdef 
\frac{\e^{\kappa}}{4} \sum_{\be\in\CE_\times}
\Big\{
2\ad_{C^\e_{\be}[\<I'XiIXibar>]}A(e+\be)
-
\sum_{\bw\in\{\Northwest,\Southwest\}}\ad_{C^\e_{\be,\bw}[\<IXiI'Xibar>]}A(e^\bw)
\Big\}\;.
\end{equ}
Recall the analytic function $F\colon \mathring V\to L(\mfg,\mfg)$ defined in \eqref{eq:def_F},
which we arbitrarily extend in this section to a smooth compactly supported function $F\colon \mfg\to L(\mfg,\mfg)$.
Define the smooth functions $h^\e_\ell \colon \mfg \to \mfg$ for $\ell\in\{1,2,3\}$ by 
\begin{equs}[e:def-h123A]
h_1^\e (A) &\eqdef 
%\e^{2\kappa}C^\e [\<I(PsiXi1)Xi1less>]
\sum_{j} 
(F'(\e A)(e_j) - \frac{1}{12}  \ad_{e_j})
[ A,e_j] \;,
\\
h_2^\e (A) &\eqdef 
%\frac{\e^{1+2\kappa}}{2}C^\e [\<I(PsiXi1)Xi1less>]
\frac{\e}{2}
\sum_{j}  F'(\e A)([ A,e_j])[  A ,e_j] \;,
\\
h_3^\e (A) &\eqdef 
 \frac{\e^{1+2\kappa}}{2} (c^\e_1[\<r_z_large>]-c^\e_2[\<r_z_large>])
\sum_{j,k}  F'(\e A) ([e_j,e_k]) [e_j,e_k]\;.
\end{equs}
Here the notation $F'(X)(Y) \in L(\mfg,\mfg)$ means the derivative in direction $Y\in \mfg$.

We formulate our fixed point problem 
on (recall $\cD^{\gamma,\eta}_{\<IXi>,\eps}$
from Section~\ref{sec:Differentiation})
\begin{equ}\label{eq:cD^sol}
\cD^\sol \eqdef \{ \CA^\e \in \cD^{\gamma,\eta}_{\<IXi>,\e}\otimes \mfg  \; : \; \CA^\e = \bone_+\bPsi + \CV^\e \;, \; \CV^\e \in  \cD_{0,\e}^{\gamma,\eta}\otimes \mfg \}\;,
\end{equ}
where $\bone_+$ is the positive time indicator function $\bone_+=1_{t>0}$
and $\gamma\in (1,2)$ and $\eta>-\frac12$ (the values of which we specify later), as follows.
Consider an initial condition $a^\eps=(a^\eps_1,a^\eps_2)\in\Omega_N$ (recalling $\eps=2^{-N}$).
For $j\neq i$, consider\footnote{$\bar{\mcb{D}}_i \bPsi_j +\hat{\mcb{D}}_i \CV^\e_j$ 
%in~\eqref{e:fix-pt-Ae} 
formally combine 
into $\bar{\mcb{D}}_i \CA^\e_j$ corresponding to $\bar\d_iA_j$ in~\eqref{e:Aeps} (see Remark~\ref{rem:barCD_term1}).}
\begin{equs}[e:fix-pt-Ae]
\CA^\e_i &= 
\CP^{i;\eps} \bone_+\Big(
[ \cS_{a} \CA_j^\e ,
2\mcb{D}_j \CA^\e_i   - \bar{\mcb{D}}_i \bPsi_j -
 \hat{\mcb{D}}_i \CV^\e_j] 
+ [\CA^\e_i, \mcb{D}_i \CA^\e_i ]
\\
& \;\;
+[\cS_{a} \CA_j^\e ,[\cS_{a} \CA_j^\e ,\CA_i^\e]]
+ c^{\eps}_{\star,i} \CA^\e
%+ \bdelta_i^\e
+ \bXi_i
+ \e^\kappa F(\cE\CA_i^\e)[\CA^\e_i, \bar\bXi_i]
\\
&\;\;
+  \CQ^{\e}_i
+ \boldR_i
+ \boldH_i
\Big)
- \CL_1 P^{i;\e} *_{(i)} (\delta_i \CR \CV^\e_i + \bar\delta_i \CR \CA^\e_j)
+ \tilde{\boldR}_i
+ G^{i;\eps} a_i^\e\;,
\end{equs}
where we used the shorthand $\hat{\mcb{D}}_i \eqdef \frac12\mcb{D}_i^+(\cS_{\northwest}+\cS_{\southwest})$,
$\cE$ is the linear map on $\cD^{\gamma,\eta}_{\<IXi>,\e}$  defined in Section~\ref{subsubsec:multiplication_by_eps},
$G^{i;\eps}a^\eps_i$ is the harmonic extension of $a^\eps_i$ lifted to $ \cD^{\gamma,\eta}_{0,\e}$, and
\begin{equs}[e:def-tildeCR2]
 \CQ^{\e}_i
  &= 
  -\frac{\eps^\kappa}{2}  [(\mcb{D}_j^+ +\mcb{D}_j^-) \tilde\bPsi_i, \bar{\mcb{D}}_j \tilde\bPsi_j]
 \\
 &\qquad
-\frac{\eps^\kappa}{2}  [(\mcb{D}_j^+ +\mcb{D}_j^-) \bar\bPsi_i, \hat{\mcb{D}}_j \CV_j^\e]
-\frac{\eps^\kappa}{2}  [(\mcb{D}_j^+ +\mcb{D}_j^-) \CV_i^\e, \bar{\mcb{D}}_j \bar\bPsi_j]
\\
&    
 \qquad + \frac{\e^\kappa}{4} [\mcb{D}_j^+\cS_{\southeast} \tilde\bPsi_j, \mcb{D}_j^+\cS_{\southwest} \tilde\bPsi_j]
\\
&\qquad
+ \frac{\e^\kappa}{4} [\mcb{D}_j^+\cS_{\southeast} \bar\bPsi_j, \mcb{D}_j^+\cS_{\southwest} \CV^\e_j]
+ \frac{\e^\kappa}{4} [ \mcb{D}_j^+ \cS_{\southeast} \CV^\e_j  ,  \mcb{D}_j^+ \cS_{\southwest} \bar\bPsi_j] \;,
\end{equs}
where we used another shorthand $\hat{\mcb{D}}_j \eqdef \frac12\mcb{D}_j^+(\cS_{\southeast}+\cS_{\southwest})$.

\begin{remark}\label{rem:quad-tildeXi}
Recall that we built our regularity structure by including 
$\Xi_{\mfl_i^\pm}$ and $\Xi_{\mfl_i^=}$
%\[
%\Xi_{\mfl_i^\pm} = \big(\CD_j^\pm \mcb{I}\tilde\Xi_i \big)  \big(\bar{\CD}_j \mcb{I}\tilde\Xi_j\big) \;,
%\qquad
%\Xi_{\mfl_i^=}=\big(\CD_j^+\CS_{\southeast} \mcb{I}\tilde\Xi_j \big) \big( \CD_j^+ \CS_{\southwest} \mcb{I}\tilde\Xi_j \big)
%\]
 %in $\mfT$
 in \eqref{eq:base_case_2}, rather than starting with $\tilde\Xi$.
The terms in \eqref{e:def-tildeCR2} which are ``quadratic in $\tilde\bPsi$''
should be interpreted as, for instance,
$
 [\mcb{D}_j^+  \tilde\bPsi_i, \bar{\mcb{D}}_j \tilde\bPsi_j]
=
\sum_{m,n} e_m^*\otimes e_n^* \otimes [e_m,e_n]
$
as an element of 
$
 \CT[\big(\CD_j^+ \mcb{I}\tilde\Xi_i \big)  \big(\bar{\CD}_j \mcb{I}\tilde\Xi_j\big)] \otimes \mfg 
 \simeq \mfg^*\otimes \mfg^* \otimes \mfg
$.
\end{remark}
\begin{remark}\label{rem:P_bone_+_xi}
We interpret 
$\CP^\e \bone_+\bXi$  in~\eqref{e:fix-pt-Ae} as  $\bone_+\bPsi -\mathcal L_\gamma K^\eps *(1_{t\le 0} \xi^\eps)  + \mathcal L_\gamma R^\eps *(\bone_+\xi^\eps)$ where $\mathcal L_\gamma$ is the lift to the polynomial regularity structure up to order $\gamma \in (1,2)$.
In this way our fixed point problem will not involve $\Xi$ (i.e. we effectively excluded the additive noise $\Xi$ from our regularity structure $\CT$).

We take this interpretation for two reasons.
First, $\bone_+\bXi \in \cD^{\infty,\infty}_{-2-\kappa}$ is formally outside the domain of $\CP^\e$ from~\eqref{eq:conv_K} (although this problem can be circumvented using a discrete analogue of \cite[Lemma~A.4]{CCHS_2D}).
Second, having $\Xi$ in our regularity structure $\CT$ makes the lowest degree below $-2$,
%but we choose to truncate the degree from above at $\gamma\in (1,2)$
which causes several technical subtleties in the application of~\cite[Sec.~4]{EH19}.\footnote{It was pointed out to us by Rhys Steele that, if the lowest 
degree in the regularity structure is $-2$ or lower, then the assumption \cite[Assump.~4.3.2]{EH19} is not verified with the choice of small-scale norm and heat 
kernel decomposition of~\cite{EH21,HM18} that we take in~\eqref{eq:seminorm} and Section~\ref{sec:Admissible models}.\label{foot:Rhys}}
\end{remark}
Moreover, recalling the remainder terms in Proposition~\ref{prop:rescaled-equ} and \eqref{e:def-h123A},
\begin{equ}[e:def-boldR]
\boldR_i 
= \CL_0 
\big( \e^{-3} r_h ( \e \CR\CA^\e)_i
 +  \e^{-3}  r_l(\e \CR\CA^\e)_i  \big)\;,
%R(e,\CR \CA^\e)
\end{equ}
%where $\xi^\e = \CR\bXi$,
\begin{equ}[e:def-boldH]
\boldH_i 
= \CL_0
\big(
h^\e_1(\CR\CA^\e_i)+h^\e_2(\CR\CA^\e_i)+h^\e_3(\CR\CA^\e_i)
\big)\;,
\end{equ}
and recalling $I_3$ in Proposition~\ref{prop:rescaled-equ},
again for $j\neq i$,
\begin{equs}[e:def-tilde-boldR]
\widetilde\boldR_i
= \CL_1  P^{i;\e}   &  *_{(i)} 
\Big\{
I_3(\CR\CA^\e)_i
 -\frac{\e^{-1}}2  \hat{R}^\nabla_\e (\e \CR\CA^\e)_i %(\nabla \hat R^p_{\eps}(\e \CR\CA^\e) + \nabla \hat R^{\bar p}_{\eps}(\e \CR\CA^\e) )
- \tilde{c}_{\star,i}^\e \CR\CA^\e
\\
& -\frac{\eps^2}{2}  [\partial_j^2 \CR\CV_i^\e, \bar\partial_j \CR\CV_j^\e]
  + \frac{\e}{4} [ \partial_j^+ \CR\CV^\e_j (\cdot_\southeast) ,  \partial_j^+ \CR\CV^\e_j (\cdot_\southwest)]
 \Big\}\;.
\end{equs}
Here $\CL_\gamma$ is the lift to the polynomial regularity structure of order $\gamma$ (after  bilinear interpolation into functions in continuum). 

Finally,
$c^\e_\star\in L(\mfg^2,\mfg^2) $ is, for now,  arbitrary,
$\tilde c^\e_\star\in L(\mfg^2,\mfg^2)$ will be chosen in Definition~\ref{def:tildec},
and we denote as usual $c^\e_\star X= (c^\e_{\star,1} X, c^\e_{\star,2} X)$ and 
$c^\e_{\star,i}X=\sum_{j=1}^2c^{\e,(j)}_{\star,i}X_j$ for $c^\e_{\star,i}\in L(\mfg^2,\mfg)$ and $c^{\e,(j)}_{\star,i}\in L(\mfg,\mfg)$ and $X\in\mfg^2$.
The meaning of $c_iA$ is the same as above Lemma \ref{lem:I3}, i.e.
$
c^{\eps}_{\star,i} \CA^\e
=c^{\eps,(i)}_{\star,i} \CA^\e_i  + c^{\eps,(j)}_{\star,i} \cS_a \CA^\e_j
$
and similarly for $\tilde{c}_{\star,i}^\e \CR\CA^\e$.
%
%\begin{remark}
%The terms $- \bar{\mcb{D}}_i \bPsi_j - \hat{\mcb{D}}_i \CV^\e_j$ 
%in~\eqref{e:fix-pt-Ae} formally combine 
%to give $- \bar{\mcb{D}}_i \CA^\e_j$ (see Remark~\ref{rem:barCD_term1}) which corresponds to $\bar\d_iA_j$ in~\eqref{e:Aeps}.
%\end{remark}

\begin{remark}
Recalling~\eqref{eq:R_2_other_form} and the discussion thereafter,~\eqref{e:def-tildeCR2} and the last line of \eqref{e:def-tilde-boldR} correspond to $\tilde R_2$ 
in Proposition~\ref{prop:rescaled-equ}.
The second line of~\eqref{e:def-tildeCR2} arises from rewriting $\bar\d_j A_j(e)$ as
$\frac12 \{(\d^+_j A_j)(e^{\southeast})+(\d^+_j A_j)(e^{\southwest})\}$
(see Remark~\ref{rem:barCD_term2}).
%We will show that the term $\CQ_i^\e$ vanishes in the limit as a modelled distribution,
%the arguments of $\CL_0$ in~\eqref{e:def-boldR} and~\eqref{e:def-boldH} for $\boldR_i$ and $\boldH_i$ respectively vanish in $L_\e^\infty$,
%and for $\widetilde\boldR_i$ the terms on the right-hand side of \eqref{e:def-tilde-boldR} 
%before convolution with heat kernel $P^\e$ vanish in $\CC_\e^\alpha$ with $\alpha<0$
%provided that $ \tilde{c}^\e_\star $ is suitably chosen.
\end{remark}

%\begin{remark}
%The terms $c^\eps_\star \CA^\e_i$, $ \CL_1 P^{i;\e} *_{(i)} (\delta_i \CR \CV^\e_i + \bar\delta_i \CR \CA^\e_j)$,  and $\boldH_i$ all play a similar role,
%which is to ensure the reconstruction of $\CA^\e$ for the renormalised model yields the original discrete Langevin dynamic without any renormalisation (Lemma~\ref{lem:renormalised-equ}).
%It turns out that $c^\eps_\star = O(1)$.
%\end{remark}

%\begin{remark}\label{rem:HQ18_compare}
%It is helpful to compare with 
%the study of universality for KPZ
%in the weakly asymmetric regime in \cite{KPZJeremy},
%where one has `remainder' terms of the form $\e^k (\partial h)^{2k+2}$ with $\d h$ having negative regularity
%and one needs to introduce an abstract `multiplication by $\eps^k$' map for the fixed point argument to close.
%We introduce $\cE$ for a related but different reason, namely to make $\cE\CA$ function-like so that $F(\cE\CA_i^\e)$ makes sense.
%We in particular do not need $\cE$ to handle the `remainder' terms~\eqref{e:def-tildeCR2}
%since, after decomposing $\CA=\bPsi+\CV$, and introducing $\bar\bPsi$ and $\tilde\bPsi$ with enhanced degrees,
%the lowest degree of
%$(\partial \CV)^2$ is $0-$, so we can simply view $\e$ as a coefficient in~\eqref{e:def-tilde-boldR}
%and the fixed point argument will still close.
%\end{remark}

\begin{remark}
If $(\Pi^\e,\Gamma^\e)$ is the canonical model from Section~\ref{sec:canonical}
and $\CA^\eps$ solves the above 
fixed point problem \eqref{e:fix-pt-Ae} with $c^\e_\star=\tilde{c}_\star=0$ and with $\boldH_i$, $\delta_iA$, and $\bar\delta_i A$ redefined to be $0$,
then $\CR\CA^\eps$
 solves the  discrete SPDE \eqref{e:Aeps} up to the exit time of $\CR\CA^\eps$ from the $V$ (see Section~\ref{sec:Notation}).
This follows from the fact that $(\Pi^\e,\Gamma^\e)$ is compatible 
(Proposition~\ref{prop:canonical_model}),
in particular $\CR\CP^{i;\e} = P^{i;\e}*_{(i)}\CR$ % where $P^\e$ is the random walk transition probability on $\obonds_i$. 
%Also, $(\CR^\e \CA_i^\e)(e) = A_i(e)$. 
and~\eqref{eq:rec} holds for $\CA$ (Remark~\ref{rem:admissible_Holder}),
together with Lemmas~\ref{lem:DRcommute},~\ref{lem:RecShft}, and~\ref{lem:Rec_Mult_eps}
and the fact that
$\Pi^\e$ commutes with products.

Similarly, the terms in \eqref{e:fix-pt-Ae} are designed so that (a) they have good analytic properties, and (b)
the reconstruction of $\CA^\e$ for the \textit{renormalised model} yields the original discrete Langevin dynamic \eqref{e:Aeps} without any renormalisation for a suitable $c^\eps_\star $ (Lemma~\ref{lem:renormalised-equ}).
It turns out that $c^\eps_\star = O(1)$.
\end{remark}
We will need the fact that the
fixed point problem \eqref{e:fix-pt-Ae} is indeed locally well-posed
and close to another fixed point with no error terms.
To this end, we consider on $\cD^\sol$ the fixed point problem, with $\CB=\bone_+\bPsi + \CW^\e$,
\begin{equs}[e:fix-pt-Be]
\CB^\e_i &= 
\CP^{i;\eps} \bone_+\Big(
\Big[
\cS_{a} \CB_j^\e ,
2\mcb{D}_j \CB^\e_i   - \bar{\mcb{D}}_i \bPsi_j - \hat{\mcb{D}}_i \CW^\e_j
\Big] 
+ [\CB^\e_i, \mcb{D}_i \CB^\e_i ]
\\
& \qquad 
+[\cS_{a} \CB_j^\e ,[\cS_{a} \CB_j^\e ,\CB_i^\e]]
+ c^{\eps}_{\star,i} \CB^\e
+ \bXi_i
\Big)
%-\CL_1 P^{i;\e} *_{(i)} \delta \CR \CB^\e_i
+ G^{i;\eps} a_i^\e\;,
\end{equs}
where all the terms have the same meaning.

\begin{assumption}\label{as:models}
There exists $r>0$ such that $\$Z^\eps\$_{\gamma;\K}^{(\eps)}<r$ and $|c^\e_\star|
< r$ for all $\eps>0$,
where $c^\e_\star\in L(\mfg^2,\mfg^2)$ is in~\eqref{e:fix-pt-Ae}.
Furthermore, for $i\in\{1,2\}$, denote
\begin{equ}[eq:Psi_def]
\Psi_i \colon[0,1] \to L^\infty_\e\;,\quad \Psi_i(t) \eqdef \Pi^\e_{(t,\cdot)}\CI_i\Xi_i(t,\cdot)\;.
\end{equ}
Then $\Psi$ satisfies Assumption~\ref{as:models_abstract} for some $\theta>0$.
\end{assumption}

\begin{definition}\label{def:tildec}
Let $\tilde{c}^\e_\star$ be the sum  
of the maps $c^\e$
in Lemmas
\ref{lem:I3} and
\ref{lem:grad-R-hat}.
%From these lemmas, $\tilde{c}^\e_\star$ converges as $\e\to 0$.
\end{definition}

\begin{proposition}  \label{prop:sol-abs}
Let $\gamma\in (1+10\kappa,\frac32)$, $\eta\in (-\frac14,-2\kappa)$, and $r>0$.
Suppose Assumption~\ref{as:models} holds and that $\|a^\e\|_{\CC^\eta_\e} < r$.
Then there exists $\e_0>0$, depending on $r$, such that for all $\e\in (0,\e_0)$,
there exists a unique solution
\begin{equ}
\CA^\e = \CP^\e \bone_+ \bXi + G^\e a^\e + \tilde \CA^\e
\end{equ}
to the fixed point problem \eqref{e:fix-pt-Ae} in 
$\cD^\sol$ on $(0,T]$.
The existence time $T$ can be chosen uniformly over $\eps \in (0,\e_0)$,  
over initial conditions satisfying $\|a^\e\|_{\CC^\eta_\e} < r$,
and over compatible models $Z^\e=(\Pi^\e,\Gamma^\e)$ 
%with \eqref{e:assump-model} uniformly bounded.
satisfying Assumption~\ref{as:models}.

Moreover, there exists a unique solution
\begin{equ}
\CB^\e = \CP^\e \bone_+ \bXi + G^\e a^\e + \tilde \CB^\e
\end{equ}
to~\eqref{e:fix-pt-Be} in $\cD^\sol$ on $(0,T]$ (this holds for all $\e\in(0,1]$ and only $|c^\e_\star|<r$, $\$Z^\eps\$_{\gamma;\K}^{(\eps)}<r$, and~\eqref{e:assump-Psi-1} from Assumption~\ref{as:models} are necessary for this).
Furthermore $\tilde\CB^\e$ is in  $\cD^{\gamma,2\eta+1}_{0,\e}$ and is
locally Lipschitz continuous in the initial condition $a^\e\in\CC^\eta_\e$ and the model under the metric
\begin{equ}[eq:metric_model_extra]
\$Z^\e;\bar Z^\e\$_{\gamma;[-1,2]\times\T^2}^{(\e)} + \sup_{t\in [-1,2]} \|\Psi(t)-\bar\Psi(t)\|_{\CC^{-\kappa}_\e}\;,
\end{equ}
where $\Psi,\bar\Psi$ are defined as in~\eqref{eq:Psi_def} from $Z^\e$ and $\bar Z^\e=(\bar\Pi^\e,\bar\Gamma^\e)$.

Finally, for $\e\in(0,\e_0)$, the remainders $\tilde \CA^\e, \tilde \CB^\e$ are close in $\cD^{\gamma,3\eta+1}_{0,\e}$ in that
\begin{equ}\label{eq:CA_CB_close}
\lim_{\e\to 0}\$\tilde \CA^\e;\tilde\CB^\e\$_{\gamma,3\eta+1;T}^{(\e)} \lesssim \e^\theta\;,
\end{equ}
uniformly over bounded sets of data as above.
\end{proposition}

\begin{proof}
We first show well-posedness of $\CB^\e$ in $\cD^\sol$ and that $\tilde\CB^\e$ is in $\cD^{\gamma,2\eta+1}_{0,\e}$ (we can assume $\e\in(0,1]$ for this).
Recalling Remark~\ref{rem:P_bone_+_xi} and our definition of the metric~\eqref{eq:metric_model_extra}, it follows that $\CP^\e\bone_+\bXi \in \cD^{\gamma,\eta}_{-\kappa,\e}$ is Lipschitz continuous in the model.
%(see \cite[Appendix~A.3]{CCHS_2D} or \cite[Prop.~9.8]{Hairer14} and also Remark~\ref{rem:P_bone_+_xi}).
Since furthermore $ G^\e a^\e \in \cD^{\gamma,\eta}(\bar\CT)$,  it suffices to show that the map
\begin{equs}
\cD^\sol\ni \CA &\mapsto \CP^{i;\eps} \bone_+\Big(
\Big[
\cS_{a} \CA_j^\e ,
2\mcb{D}_j \CA^\e_i   - \bar{\mcb{D}}_i \bPsi_j - \hat{\mcb{D}}_i \CV^\e_j
\Big] 
\\
&\qquad
+ [\CA^\e_i, \mcb{D}_i \CA^\e_i ]
+[\cS_{a} \CA_j^\e ,[\cS_{a} \CA_j^\e ,\CA_i^\e]]
+ c^\eps_{\star,i} \CA^\e \Big)
%+ \bdelta_i^\e
%- \CL_1 P^{i;\e} *_{(i)} \delta \CR \CA^\e_i
\in \cD^{\gamma,2\eta+1}_{0,\e}
\end{equs}
maps every ball in $\cD^\sol$ to a ball of radius $T^{\bar\kappa}$ in $\cD^{\gamma,2\eta+1}_{0,\e}$
and furthermore has Lipschitz constant $T^{\bar\kappa}$ for some $\bar\kappa>0$ and
all $T>0$ sufficient small when all modelled distributions are restricted to $(0,T]$.
To this end, for $\CA\in\cD^\sol$,
the terms in the parentheses $(\ldots)$ are clearly bounded in $\cD^{\gamma-1-\kappa,2\eta-1}_{-1-2\kappa,\e}$ due to Lemmas~\ref{lem:DRcommute} and~\ref{lem:sf-is-Lip} and~\cite[Sec.~5.1]{EH19}.
The well-posedness of $\CB^\e$ in $\cD^\sol$ and the fact that $\tilde \CB^\e$ is in $\cD^{\gamma,2\eta+1}_{0,\e}$ and is locally Lipschitz in the model and initial condition
follow by standard short-time convolution estimates and their multi-level versions~\cite[Lem.~6.2]{EH19} (see also~\eqref{eq:conv_K}).

To show well-posedness of $\CA^\e$ in $\cD^\sol$ and~\eqref{eq:CA_CB_close},
it remains to consider % the remaining terms in~\eqref{e:fix-pt-Ae}, i.e.
\[
\CP^{i;\e} \bone_+ \Big(
% \frac{\e^\kappa}{2} [\CA^\e_i,\bar\bXi_i]
\e^\kappa F(\cE\CA_i^\e)[\CA^\e_i, \bar\bXi_i]
+  \CQ^{\e}_i
+ \boldR_i
+\boldH_i
\Big)
- \CL_1 P^{i;\e} *_{(i)} (\delta_i \CR \CV^\e_i + \bar\delta_i \CR \CA^\e_j)
+ \widetilde{\boldR}_i \;,
\]
and show that these are well-defined and of order $\e^\theta$ in $\cD^{\gamma,3\eta+1}_{0,\e}$ 
uniformly over bounded sets of data and provided $\e<\e_0$ is sufficiently small.
Since  
$\bar\bXi\in \cD^{\infty,\infty}_{-1-2\kappa,\e}$
and
$\CA^\e \in \cD^{\gamma,\eta}_{-\kappa,\e}$,
one has $[\CA^\e_i, \bar\bXi_i] \in \cD^{\gamma-1-2\kappa,\eta-1-2\kappa;i}_{-1-3\kappa,\e}$.
By Lemma~\ref{lem:B_bound_by_A}, $F(\cE\CA_i^\e)\in \cD^{\gamma,\eta;i}_{0,\e}$, and
\[
F(\cE\CA_i^\e)[\CA^\e_i, \bar\bXi_i] \in 
\cD^{\gamma-1-3\kappa, \, (\eta-1-3\kappa)\wedge (2\eta-1-2\kappa); \,i}_{-1-3\kappa,\e} \;.
\]
Therefore under our assumptions on $\eta,\kappa$ one has
\begin{equ}[e:eps-kappa-PAXi]
 \VERT  \e^\kappa \CP^{i;\e} \bone_+ 
% [\CA^\e_i,\bar\bXi_i]   
\big (F(\cE\CA_i^\e)[\CA^\e_i, \bar\bXi_i] \big)
  \VERT_{ \gamma,3\eta+1}^{(\e)}
\lesssim \eps^\kappa \;.
\end{equ}
For the term with $\CQ_i^{\e}$
as defined in \eqref{e:def-tildeCR2}, since
$\CV^\e  \in \cD_{0,\e}^{\gamma,\eta}$, 
$\mcb{D}\CV^\e  \in \cD_{-3\kappa,\e}^{\gamma-1,\eta-1}$ and 
\[
\mcb{D} \bPsi \in \cD_{-1-\kappa,\e}^{\infty,\infty} \;,
\qquad
\mcb{D}  \bar\bPsi \in \cD_{-2\kappa,\e}^{\infty,\infty} \;,
\qquad
[\mcb{D}  \tilde\bPsi, \tilde{\mcb{D}} \tilde\bPsi] \in \cD_{-1-3\kappa,\e}^{\infty,\infty} 
\]
for various differentiation operators $\mcb{D},\tilde{\mcb{D}}$ and shifts (see Remark~\ref{rem:quad-tildeXi} for the meaning of the last term)
the terms of the form $\e^\kappa (\mcb{D}  \tilde\bPsi)(\mcb{D} \tilde\bPsi)$ (with shifts)
are clearly bounded in $\cD_{-1-3\kappa}^{\infty,\infty}$ by a multiple of $\eps^\kappa$,
and, by Lemmas~\ref{lem:DRcommute} and~\ref{lem:sf-is-Lip}, the terms of the form $\e^\kappa (\mcb{D}  \bar\bPsi)(\mcb{D}\CV)$ (with shifts)
are bounded in $ \cD_{-5\kappa,\e}^{\gamma-1-2\kappa,\eta-1-2\kappa}$ by a multiple of $\eps^\kappa$.
So we obtain $\$  \CP^\e  \bone_+ \CQ_i^{\e}  \$_{ \gamma,2\eta+1}^{(\e)} \lesssim \eps^\kappa$ uniformly in $\eps$.

One can also show that $\$\CP^\e \bone_+  \boldR_i\$_{\gamma,2\eta+1}^{(\e)} \lesssim \eps^{1-4\kappa}$.
Indeed,  recalling the small scales bound $\|\CR\CA^\e\|_{\eta;\K_\e;z;\e}\lesssim\$\CA^\e\$_{\gamma,\eta;\K_\e;\e}$ by Lemma~\ref{lem:reconstruct},
the (linear) map
\begin{equ}[eq:A_to_RA]
\cD^\sol\ni\CA^\e \mapsto \CR \CA^\e(t) \in \CC^\eta_\e
\end{equ}
is continuous, uniformly over $t\in(0,T]$, $\e>0$, over bounded balls in $\cD^{\sol}$, initial conditions in $\CC^\eta_\e$, and compatible models
%\eqref{e:assump-model} uniformly bounded.
satisfying Assumption~\ref{as:models} for a fixed $r>0$.
(We implicitly used the
additional bound $\|\Psi_i(t)\|_{L^\infty_\e}\lesssim \e^{-\kappa}$ from~\eqref{e:assump-Psi-1} to obtain~\eqref{eq:A_to_RA}.)
Lemma~\ref{lem:eps-improve-reg}\ref{pt:eps_f_vanish} thus implies $\|\CR\CA^\e(t)\|_{L^\infty_\e} \lesssim \e^{\eta}$.
Hence, for $\e<\e_0$ sufficiently small, which ensures $\e\CR\CA^\e \in V$
and thus $r_h ( \e \CR\CA^\e)$ and $r_l(\e \CR\CA^\e)$ are well-defined,
we obtain by analyticity of $r_h$, $r_l$ and Remark~\ref{rem:power_of_eps},
that the map
\begin{equ}
\cD^\sol\ni\CA^\e \mapsto \e^{-4+4\kappa} \big( r_h ( \e \CR\CA^\e) +  r_l(\e \CR\CA^\e)\big)(t)\in L^\infty_\e
\end{equ}
is locally Lipschitz, uniformly over the same data as above,
and so $\e^{-3} r_h ( \e \CR\CA^\e)
 +  \e^{-3}  r_l(\e \CR\CA^\e) $ can be
 lifted to polynomial regularity structure of order $0$ with $\cD^{0,0}$ norm bounded by $\eps^{1-4\kappa}$ uniformly in $\eps < \e_0$.
The bound $\$\CP^\e \bone_+  \boldR_i\$_{\gamma,2\eta+1}^{(\e)} \lesssim \eps^{1-4\kappa}$ follows from short-time convolution estimates
and the definition of $\boldR_i$ in~\eqref{e:def-boldR}.
%Note that we have not yet used the bounds~\eqref{e:assump-model_2}.

For $\boldH_i$, we obtain $\$\CP^\e \bone_+  \boldH_i\$_{\gamma,2\eta+1}^{(\e)} \lesssim \eps^{1-2\kappa}$ in the same way as for $\boldR_i$ upon using the bound $|c_j^\e[\<r_z_large>]|\lesssim\e^{-3\kappa}$
from~\eqref{eq:C_bounds_large_cherries}
and remarking that $h^\e_1 (A)=O(\e A^2)$, $h^\e_2 (A)=O(\e A^2)$, and $h^\e_3 (A)=O(\e^{1-\kappa})$.

Turning to the term $\widetilde{\boldR}_i$, recall 
$\|\CR\CV^\e\|_{\eta;\K_\e;z;\e}\lesssim\$\CV^\e\$_{\gamma,\eta;\K_\e;\e}$ by Lemma~\ref{lem:reconstruct}.
One has, for any $\beta\in[0,1]$,
$\CR\CV^\e \in \CC^{\beta+\eta}_{-\beta,\e}$ (with blow-up due to initial condition).
Applying Lemma~\ref{lem:eps-improve-reg}\ref{pt:eps_d_f_vanish} with $\alpha=\beta+\eta$, $\beta\in[-\eta,1]$, therein,
\begin{equ}[eq:bounds_on_V]
\|\e^{1-\alpha}\partial\CR\CV^\e(t)\|_{L^\infty_\e}\lesssim t^{-\beta/2}\;,\qquad
\|\e^{2-\alpha}\partial^2\CR\CV^\e(t)\|_{L^\infty_\e}\lesssim t^{-\beta/2}\;.
\end{equ}
Taking $\beta = \frac12-\eta+\kappa$, so $1-\alpha=\frac12-\kappa$, we obtain a bound of order $\e^{2\kappa}$
in $\CC^{0,T}_{-2\beta,\e}=\CC^{0,T}_{-1+2\eta-2\kappa,\e}$ on the final two terms in the parenthesis $\{\ldots\}$ in~\eqref{e:def-tilde-boldR},
which gives a bound of order $\e^{2\kappa}$ in $\cD^{\gamma,3\eta+1}_{0,\e}$ after convolution with the heat kernel.

The rest of the terms in $\{\dots\}$ in \eqref{e:def-tilde-boldR} are of order $\e^\theta$
in $\CC^{-2\kappa,T}_{-1,\e}$ for $\e<\e_0$ by our choice of $\tilde c^\e_\star$ in Definition~\ref{def:tildec} and by Lemmas~\ref{lem:I3} and
\ref{lem:grad-R-hat}
together with Remark~\ref{rem:eA_small}
and the fact that $\CR\CV^\e$ satisfies~\eqref{eq:v_bound} (for a possibly larger $r$).
After convolution with the heat kernel, these terms, and thus $\widetilde{\boldR}_i$, are of order $\e^\theta$ in $\cD^{\gamma,1}_{0,\e}$.

Recalling $\delta_i$ from~\eqref{eq:delta_A_def}, by using the first bound in~\eqref{eq:bounds_on_V} with $\alpha=2\kappa$, $\beta = 2\kappa-\eta$ (and its obvious extension to derivatives of the form $\e^{-1}\{f(\cdot+\be+\be')-f(\cdot)\}$)
and $|\hat C^\e_{j,\be}|
+|\hat{C}^\e_{3,\be,\bw}|
+|C_{\bs,\pm}^\e [\, \<I'XiI'[IXiI'Xi]less>\,]|
+|C_{\bs}^\e [\, \<I'XiI'[IXiI'Xi]less>\,]| \lesssim \e^{-\kappa}$ (Lemma~\ref{lem:thick-trig-bahave}), we obtain
\begin{equ}[eq:delta_term_vanish]
\|(\delta_i \CR\CV^\e_i)(t)\|_{L^\infty_{\e}} \lesssim \e^{\kappa}t^{-\kappa+\eta/2}\;.
\end{equ}
After convolution with the heat kernel, this term is of order $\e^\kappa$ in $\cD^{\gamma,1}_{0,\e}$.

Finally, we recall $\bar\delta_i$ from~\eqref{eq:bar_delta_def}, and $|C^\e_{\be,\bw}[\<IXiI'Xibar>]|,|C^\e_{\be}[\<I'XiIXibar>]|\lesssim \eps^{-\kappa}$ by~\eqref{eq:C_bounds_large_cherries}.
Note the elementary identity
$\sum_{i=1}^n\lambda_iy_i = \sum_{i=1}^{n-1} (\sum_{j=1}^i\lambda_j)(y_{i}-y_{i+1})$
whenever 
$\sum_{i=1}^n\lambda_i=0$.
%\begin{equ}[e:element-identity]
%\sum_{i=1}^n\lambda_iy_i = \sum_{i=1}^{n-1} (\sum_{j=1}^i\lambda_j)(y_{i}-y_{i+1}) \quad \mbox{whenever}
%\quad \sum_{i=1}^n\lambda_i=0
%\end{equ}
Using this identity with $n=4$,
and recalling~\eqref{eq:C_be_vanish},
the first term  in~\eqref{eq:bar_delta_def}
$\e^{\kappa} \sum_{\be\in\CE_\times} \ad_{C^\e_{\be}[\<I'XiIXibar>]}A(e+\be)$
can be written as linear combinations of terms of the form $\e\d A$.
Using the above identity with $n=2$ 
 and recalling~\eqref{eq:C_bs_be_vanish},
the second term  in~\eqref{eq:bar_delta_def} 
$\e^{\kappa} 
\sum_{\bw\in\{\Northwest,\Southwest\}} \big( \sum_{\be\in\CE_\times} \ad_{C^\e_{\be,\bw} [\<IXiI'Xibar>]}\big)A(e^\bw)$
can be also written in the form $\e\d A$.
We have $\|\CR\CA^\e(t)\|_{L^\infty_\e}\lesssim t^{\eta/2}\e^{-\kappa}$ by~\eqref{eq:A_to_RA},
which implies by Lemma~\ref{lem:eps-improve-reg}\ref{pt:eps_d_f}
$
\|\bar\delta_i\CR\CA_j^\e\|_{\CC^{-2\kappa}_\e}\lesssim \|\e\d \CR\CA_j^\e(t)\|_{\CC^{-2\kappa}_\e}\lesssim t^{\eta/2}\e^\kappa$.
After convolution with the heat kernel, this term is of order $\e^\theta$ in $\cD^{\gamma,1}_{0,\e}$.
\end{proof}

\subsection{Renormalised equation}
\label{subsec:renorm_eq}

Let $(\Pi^\e,\Gamma^\e)$ be the canonical model from Section~\ref{sec:canonical} and 
let $(\hat\Pi^\e,\hat\Gamma^\e)$ be as in Definition~\ref{def:renorm_model}.
We call $(\hat\Pi^\e,\hat\Gamma^\e)$ the \textit{renormalised model}.
Let $\tilde c^\e_\star$ be as in Definition~\ref{def:tildec} for $\Psi_i = K^{i;\e}*_{(i)}\xi^\e_i$ (which exists due to Lemma~\ref{lem:Psi_Wick}).
By Propositions~\ref{prop:canonical_model} and~\ref{prop:renormalised_model}, $(\Pi^\e,\Gamma^\e)$ and $(\hat\Pi^\e,\hat\Gamma^\e)$ are both compatible models.
Let $\hat\CR$ be the reconstruction operator for $(\hat\Pi^\e,\hat\Gamma^\e)$.
Recall $V\subset \mfg^{\obonds}$ from Section~\ref{sec:Notation}.

\begin{lemma}\label{lem:renormalised-equ}
There exists $c^\eps_\star\in L(\mfg^2,\mfg^2)$
such that $c^\eps_\star=O(1)$ and, if $\CA^\eps$ solves the 
fixed point problem \eqref{e:fix-pt-Ae},
then $\hat\CR\CA^\eps$ solves the  discrete SPDE \eqref{e:Aeps}
up to the exit time of $\hat\CR\CA^\eps$ from the set $V$.
\end{lemma}
An important outcome of the following proof is that
in the equation for $A^\e$, the renormalisation map only produces 
1)  vanishing terms corresponding to \eqref{eq:delta_A_def}-\eqref{e:def-h123A}
and 2) linear terms in $A^\e$.
Both of these are eventually cancelled thanks to our formulation of the fixed point problem in Section~\ref{subsec:FPP}.

%An important outcome of the following proof is that
%in the equation for $A^\e (e)$ where $e\in\obonds_i$ the renormalisation map only produces 
%1) linear terms  evaluated at the same bond $e$  or the neighboring bonds of $e$ in $\obonds_i$, but {\it not} $\obonds_j$
%(with only one  exception which leads to a ``$\bar\delta$ term'' \eqref{eq:bar_delta_def} but this vanishes as we saw below \eqref{e:element-identity}); 2) vanishing terms corresponding to \eqref{e:def-h123A}. 
%These terms produced by renormalisation are eventually all cancelled thanks to our formulation of the fixed point problem in Section~\ref{subsec:FPP}, so that our discrete Langevin dynamic  \eqref{e:Aeps} remains untouched.

\begin{proof}
We do not have the general results as in \cite{BCCH17} or \cite[Prop.~5.68]{CCHS_2D}   so we proceed by hand. Throughout the proof we assume $j\neq i$.
In view of the fixed point problem \eqref{e:fix-pt-Ae},
the solution $\CA^\e$ has the following expansion up to degree $1$: 
\begin{equs}[e:CA-expand]
\CA^\e_i 
&=
 \mcb{I} \bXi_i + a_i \bone^{(i)}
+
 \mcb{I} \big[\cS_{a}\bPsi_j +\cS_{a} (a_j \bone^{(j)}), 
2\CD_j \bPsi_i \big]
\\
& \quad
+\mcb{I}  \big[\cS_{a}\bPsi_j +\cS_{a} (a_j \bone^{(j)}), 
 \bar{\CD}_i \bPsi_j \big] 
 +\mcb{I} \big[ \bPsi_i+ a_i \bone^{(i)} , \CD_i \bPsi_i \big]
\\
&\quad
+
\frac{\e^\kappa}{2}  \mcb{I} [  \bPsi_i +a_i \bone^{(i)},   \bar\bXi_i ]
- \frac{\eps^\kappa}{2} \mcb{I}  [(\CD_j^+ +\CD_j^-) \tilde\bPsi_i, \bar{\CD}_j \tilde\bPsi_j]
\\
&\quad
 + \frac{\e^\kappa}{4}\mcb{I} [\CD_j^+\cS_{\southeast} \tilde\bPsi_j, \CD_j^+ \cS_{\southwest}\tilde\bPsi_j]
+ 
\langle \nabla a_i, {\bf X}^{(i)}\rangle\;.
\end{equs}
Here, we write $\bPsi_i =\mcb{I} \bXi_i$, $\bar\bPsi_i =\mcb{I} \bar\bXi_i$, $\tilde\bPsi_i =\mcb{I} \tilde\bXi_i$
as in Section~\ref{subsec:FPP}
(with the terms involving $\tilde\bPsi_i$ understood as in Remark~\ref{rem:quad-tildeXi}),
and $\nabla a_i = (\nabla_1 a_i,\nabla_2 a_i)$
is a vector-valued coefficient 
which is not necessarily the derivative of $a_i$.
Note that % the terms such as $\CD_j \bPsi_i$ are the same as $\mcb{D}_j \bPsi_i$ since $\bPsi$ is constant modelled distribution. On the other hand 
$\cS$ in \eqref{e:CA-expand} is a nonlocal operator which we treat carefully below (see e.g. \eqref{e:calc-double-shifts}). 

Remark that the nonlocal parts of differentiation operators in  \eqref{e:fix-pt-Ae}
do not contribute to \eqref{e:CA-expand} at degree $\le 1$: for instance, for the terms of the form 
$\CA^\e \mcb{D}\CA^\e = \CA^\e \CD \CA^\e+ \CA^\e (H \CA^\e \bone)$,
 the last term  upon hitting by $\mcb{I}$ has degree $>1$.

By the definition \eqref{e:shift-Dga} of $\cS$,
and the definition of  $\CS_{\be}$ on abstract polynomials,
\begin{equ}[e:Sa1]
\cS_{a} (a_j \bone^{(j)}) (e) = 
\frac14\sum_{\be\in \CE_\times} a_j (e+\be) \bone^{(i)} \;,
\qquad
e\in \obonds_i\;.
\end{equ}
We  remark that all the terms in \eqref{e:CA-expand} take values in $\CT^{(i)}\otimes \mfg$ and the Lie brackets act on the $\mfg$ component,
where $\CT^{(i)}$ is the $i$th component in our splitting
\eqref{e:T1T2}.

We now go through all the terms in~\eqref{e:fix-pt-Ae} that require renormalisation, expand them up to degree $0$,
and apply the map $M^\e$ to them in order to compute the 
right-hand side of our equation.
The cubic term is expanded as
\begin{equs}[e:cubic-exp]
{} &
[\bPsi_j^{(a)}, [\bPsi_j^{(a)}, \bPsi_i]]
+ [a_j^{(a)}, [\bPsi_j^{(a)}, \bPsi_i]]
+ [\bPsi_j^{(a)}, [a_j^{(a)}, \bPsi_i]]
+ [\bPsi_j^{(a)}, [\bPsi_j^{(a)}, a_i]]
\\
&
+ [a_j^{(a)}, [a_j^{(a)}, \bPsi_i]]
+ [a_j^{(a)}, [\bPsi_j^{(a)}, a_i]]
+ [\bPsi_j^{(a)}, [a_j^{(a)}, a_i]]
+ [a_j^{(a)}, [a_j^{(a)}, a_i]]
\end{equs}
where $a_j^{(a)}$ is a shorthand for \eqref{e:Sa1},
and $a_i=a_i\bone^{(i)}$, $\bPsi_j^{(a)} =\cS_{a} \bPsi_j$ as shorthands.
Upon applying the renormalisation map $M^\e$, only the 1st and 4th terms are affected, and the expression transforms as
\begin{equ}[eq:cubic_ren_eq]
\eqref{e:cubic-exp} \mapsto 
 \eqref{e:cubic-exp}
 %- C[\<IXi^2>] \,\Cas 
 -\ad_{\bar{C}^\e}
 \,(\bPsi_i+a_i \bone^{(i)})\;,
\end{equ}
where we recall the definition of $\bar C^\e$~\eqref{e:def-C16} and of $\ad_{X\otimes Y}=[X,[Y,\cdot]]$ for $X,Y\in\mfg$ following~\eqref{eq:delta_A_def}
and the minus sign in~\eqref{eq:M_tau_def}.
Here, $\bPsi_i+a_i \bone^{(i)}$ is just $\CA^\e_i$ truncated at degree $0$ and one has $\hat\CR \CA^\e_i =\hat\CR(\bPsi_i+a_i \bone^{(i)})$.

For $[\CA^\e_i, \mcb{D}_i \CA^\e_i ]$, we have $\mcb{D}_i \CA^\e_i = \CD_i \CA^\e_i + b$, where $b\in \CT[\bone^{(i)}]\otimes \mfg$.
The product of $\CA^\e_i$ and $b$ is not renormalised, so it suffices to consider $[\CA^\e_i, \CD_i \CA^\e_i ]$,
which is expanded as,  up to degree $0$,
\begin{equs}[e:quad-exp-1]
{}&
[\bPsi_i + a_i  , \CD_i \bPsi_i + \nabla_i a_i \bone^{(i)}]
+ [\langle \nabla a_i,{\bf X}^{(i)} \rangle, \CD_i \bPsi_i]
\\
&+\Big[
\mcb{I} [a_j^{(a)} +\bPsi_j^{(a)} ,  2 \CD_j \bPsi_i ]  
+ \mcb{I} [a_j^{(a)} +\bPsi_j^{(a)} , - \bar{\CD}_i \bPsi_j  ] 
\\
&\qquad\qquad\qquad\qquad\qquad\qquad
+ \mcb{I} [a_i  +\bPsi_i ,\CD_i \bPsi_i  ] \; , \;\; \CD_i \bPsi_i 
\Big]
\\
&+ \Big[ \bPsi_i +a_i \; ,  \;\; 
\CD_i \mcb{I} [a_j^{(a)} +\bPsi_j^{(a)} ,  2 \CD_j \bPsi_i ]  
\\
&\qquad\qquad
+\CD_i  \mcb{I} [a_j^{(a)} +\bPsi_j^{(a)} , - \bar{\CD}_i \bPsi_j  ] 
+\CD_i \mcb{I} [a_i + \bPsi_i , \CD_i \bPsi_i] 
\Big] 
+ \mcH^{(1)}\;.
%\\
%&+\frac{\e^\kappa}{2}
% [\bPsi_i+a_i , \CD_i\mcb{I} [\bPsi_i+a_i  , \bar\bXi_i]]
%+\frac{\e^\kappa}{2}
%[\mcb{I} [\bPsi_i+a_i , \bar\bXi_i], \CD_i \bPsi_i] 
%\\
%&-\frac{\e^\kappa}{2}
% \Big[\bPsi_i+a_i  , \CD_i\mcb{I}  [(\CD_j^+ +\CD_j^-) \tilde\bPsi_i, \bar{\CD}_j \tilde\bPsi_j]
% -\frac12 \CD_i\mcb{I}  [\CD_j^+ \cS_{\southeast} \tilde\bPsi_j, \CD_j^+\cS_{\southwest} \tilde\bPsi_j]\Big]
% \\
%&+\frac{\e^\kappa}{2}
%\Big[\mcb{I}  [(\CD_j^+ +\CD_j^-) \tilde\bPsi_i, \bar{\CD}_j \tilde\bPsi_j]
% -\frac12  \mcb{I}  [\CD_j^+ \cS_{\southeast} \tilde\bPsi_j, \CD_j^+ \cS_{\southwest} \tilde\bPsi_j], \CD_i \bPsi_i\Big] \;.
\end{equs}
The final term $\mcH^{(1)}$ involves trees
$\<IXiI'Xibar>$, $\<supercherry1>$, $\<I'XiIXibar>$, $\<supercherry2>$
in the group \hyperlink{(1-3)}{(1-3)}
and trees $\<I'(R2-1)new> $,
$\<IXiI'(R2-1)new>$,
$ \<I'XiI(R2-1)new>$ in the group \hyperlink{(1-2)}{(1-2)}, all of which are not renormalised (for trees in \hyperlink{(1-3)}{(1-3)} this is because they are not of the form in~\eqref{eq:square_cherry_consts}).
The other terms only involve $\mfT^{\YM}_-$. 
Upon applying the renormalisation map, only $[\mcb{I} [a_i  +\bPsi_i ,\CD_i \bPsi_i  ] , \;\CD_i \bPsi_i ]$ and $[\bPsi_i,\CD_i \mcb{I} [a_i + \bPsi_i , \CD_i \bPsi_i]]$ are affected, and the expression transforms as
\begin{equ}[eq:quad1_ren_eq]
\eqref{e:quad-exp-1} \; 
\mapsto\;
 \eqref{e:quad-exp-1} 
%-  \left( C[ (\mcb{I} \CD_i \Psi_i ) \CD_i \Psi_i ]
%- C[ \Psi_i  \CD_i \mcb{I} (\CD_i \Psi_i)]\right)\, \Cas
-(\ad_{\hat{C}^\e_4}-\ad_{\hat{C}^\e_5})
\, (\bPsi_i+a_i \bone^{(i)})\;.
\end{equ}

For the term
$[ \cS_{a} \CA_j^\e , 2\mcb{D}_j \CA^\e_i  ] $, as before, it suffices to consider $[ \cS_{a} \CA_j^\e , 2\CD_j \CA^\e_i  ] $, which expands as,  up to degree $0$,
\begin{equs}[e:quad-exp-2]
{}&
[\bPsi_j^{(a)} + a_j^{(a)} , 2\CD_j \bPsi_i + 2\nabla_j a_i \bone^{(i)}]
+ [\cS_a \langle \nabla a_j,{\bf X}^{(j)} \rangle, 2\CD_j \bPsi_i]
\\
&+\Big[
\cS_a \mcb{I} [a_i^{(a)} +\bPsi_i^{(a)} ,  2 \CD_i \bPsi_j ]  
+ \cS_a \mcb{I} [a_i^{(a)} +\bPsi_i^{(a)} , - \bar{\CD}_j \bPsi_i ] 
\\
&\qquad\qquad\qquad \qquad\qquad\qquad
+  \cS_a \mcb{I} [a_j  +\bPsi_j ,\CD_j \bPsi_j  ] \;, \;\; 2 \CD_j \bPsi_i 
\Big]
\\
&+ \Big[ \bPsi_j^{(a)} +a_j^{(a)} \;,\;\; 
2 \CD_j \mcb{I} [a_j^{(a)} +\bPsi_j^{(a)} ,  2 \CD_j \bPsi_i ]  
\\
&\qquad\qquad
+2\CD_j  \mcb{I} [a_j^{(a)} +\bPsi_j^{(a)} , - \bar{\CD}_i \bPsi_j  ] 
+2\CD_j \mcb{I} [a_i + \bPsi_i , \CD_i \bPsi_i] 
\Big] + \mcH^{(2)}\;.
%\\
%&+\e^\kappa
% [\bPsi_j^{(a)}+a_j^{(a)}  , \CD_j\mcb{I} [\bPsi_i+a_i , \bar\bXi_i]]
%+\e^\kappa
%[\cS_a\mcb{I} [\bPsi_j+a_j , \bar\bXi_j], \CD_j \bPsi_i] 
%\\
%&-\e^\kappa
% \Big[\bPsi_j^{(a)}+a_j^{(a)} , \CD_j\mcb{I}  [(\CD_j^+ +\CD_j^-) \tilde\bPsi_i, \bar{\CD}_j \tilde\bPsi_j]
% -\frac12 \CD_j\mcb{I}  [\CD_j^+\cS_{\southeast} \tilde\bPsi_j, \CD_j^+\cS_{\southwest} \tilde\bPsi_j]\Big]
% \\
%&+\frac{\e^\kappa}{2}
%\Big[\cS_a \mcb{I}  [(\CD_i^+ +\CD_i^-) \tilde\bPsi_j, \bar{\CD}_i \tilde\bPsi_i]
% -\frac12 \cS_a \mcb{I}  [\CD_i^+\cS_{\southeast} \tilde\bPsi_i, \CD_i^+ \cS_{\southwest} \tilde\bPsi_i], \CD_j \bPsi_i\Big] \;.
\end{equs}
The final term $\mcH^{(2)}$ again denotes a term involving %the same trees 
trees of the same form
as $\mcH^{(1)}$ in \eqref{e:quad-exp-1} -- again not renormalised (the trees in \hyperlink{(1-3)}{(1-3)} are again not of the form in \eqref{eq:square_cherry_consts}).
One can check that $M^\e$ affects only 
$[\cS_a \mcb{I} [a_i^{(a)} +\bPsi_i^{(a)} , - \bar{\CD}_j \bPsi_i ], 2\CD_j\bPsi_i]$,
which at $e\in \obonds_i$ is equal to
(recalling \eqref{e:Sa1}, the definition of $\cS_a$, and the fact that $\bPsi_i$ is a constant modelled distribution)
%taking the constant value $\Psi_i\eqdef \id\in \CT[\CI\bXi_i]\otimes \mfg \simeq \mfg^*\otimes \mfg$, note that
\begin{equs}[e:calc-double-shifts]
%{} &\Big[ \cS_a \mcb{I} [a_i^{(a)} +\bPsi_i^{(a)} , - \bar{\CD}_j \bPsi_i ] 
%, 2 \CD_j \bPsi_i \Big] (e)
%\\
{}&\frac{1}{16} \sum_{\be,\be'\in \CE_\times} \!\! 
\Big(a_i(e+\be+\be') \Big[ \CS_{\be} \mcb{I} [\CS_{\be'} \bone^{(i)}, - \bar{\CD}_j \Psi_i ] 
, 2 \CD_j \Psi_i \Big]
\\
&\qquad \qquad\qquad\qquad\qquad
 +\Big[ \CS_{\be} \mcb{I} [\CS_{\be'} \Psi_i, - \bar{\CD}_j \Psi_i ] 
, 2 \CD_j \Psi_i \Big]\Big)\;.
\end{equs}
Upon applying the renormalisation map  (in particular see Example~\ref{ex:M_tau}), 
\begin{equs}[eq:C_1_ren_eq]
\eqref{e:quad-exp-2} 
\mapsto
 \eqref{e:quad-exp-2} 
&+2 \ad_{\hat{C}^\e_1}
  \,(\bPsi_i+a_i \bone^{(i)})
  \\
 &+  \frac{1}{16}\!\! \sum_{\be,\be'\in \CE_\times} \!\!  2 \ad_{\hat{C}^\e_{1,\be}} (a_i(\cdot+\be+\be')-a_i)\bone^{(i)}
 \;.
\end{equs}
Note that $a_i(e+\be+\be')-a_i(e)=(\hat\CR \CV_i)(e+\be+\be')-(\hat\CR \CV_i)(e)$.

For the term
$[ \cS_{a} \CA_j^\e ,  - \bar{\mcb{D}}_i\bPsi_j -\hat{\mcb{D}}_i\CV^\e_j] $,
where we recall that $\hat{\mcb{D}}_i=\frac12
\mcb{D}_i^+ (\cS_{\northwest}+\cS_{\southwest})$,
we again  only need to consider $[ \cS_{a} \CA_j^\e ,  - \bar{\CD}_i \bPsi_j -\frac12
\CD_i^+ (\cS_{\northwest}+\cS_{\southwest}) \CV^\e_j] $, which expands as, up to degree $0$ and
with $b$ denoting a generic polynomial term that does not enter any terms that are renormalised,
\begin{equs}[e:quad-exp-3]
{}&
[\bPsi_j^{(a)} + a_j^{(a)} , -\bar{\CD}_i \bPsi_j +b]
+ [b, \bar \CD_i \bPsi_j]
\\
&+\Big[
\cS_a \mcb{I} [a_i^{(a)} +\bPsi_i^{(a)} ,  2 \CD_i \bPsi_j ]  
+ \cS_a \mcb{I} [a_i^{(a)} +\bPsi_i^{(a)} , - \bar{\CD}_j \bPsi_i ] 
\\
&\qquad\qquad\qquad \qquad
+  \cS_a \mcb{I} [a_j+\bPsi_j ,\CD_j \bPsi_j  ] \;, \;\;  -\bar{\CD}_i \bPsi_j
\Big]
\\
&+ \Big[ \bPsi_j^{(a)} +a_j^{(a)} \;,\;\; 
-\frac12
\CD_i^+ (\cS_{\northwest}+\cS_{\southwest}) \Big\{
 \mcb{I} [a_i^{(a)} +\bPsi_i^{(a)} ,  2 \CD_i \bPsi_j ]  
\\
&\qquad\qquad\qquad
+  \mcb{I} [a_i^{(a)} +\bPsi_i^{(a)} , - \bar{\CD}_j \bPsi_i ] 
+ \mcb{I} [a_j + \bPsi_j , \CD_j \bPsi_j] 
\Big\}
\Big] 
\\
&-\frac{\e^\kappa}{2}
 [\bPsi_j^{(a)}+a_j^{(a)}  , \frac12
\CD_i^+ (\cS_{\northwest}+\cS_{\southwest})
\mcb{I} [\bPsi_j+a_j  , \bar\bXi_j]]
\\
&\qquad -\frac{\e^\kappa}{2}
[\cS_a\mcb{I} [\bPsi_j+a_j , \bar\bXi_j], \bar{\CD}_i \bPsi_j] + \mcH^{(3)}\;,
%\\
%&+\frac{\e^\kappa}{2}
% \Big[\bPsi_j^{(a)}+a_j^{(a)} , \hat{\CD}_i\mcb{I}  [(\CD_i^+ +\CD_i^-) \tilde\bPsi_j, \bar{\CD}_i \tilde\bPsi_i]
% + \hat{\CD}_i\mcb{I}  [\CD_i^+ \cS_{\southeast} \tilde\bPsi_i, \CD_i^+ \cS_{\southwest} \tilde\bPsi_i]\Big]
% \\
%&-\frac{\e^\kappa}{4}
%\Big[\cS_a \mcb{I}  [(\CD_i^+ +\CD_i^-) \tilde\bPsi_j, \bar{\CD}_i \tilde\bPsi_i]
% + \cS_a \mcb{I}  [\CD_i^+ \cS_{\southeast} \tilde\bPsi_i, \CD_i^+ \cS_{\southwest} \tilde\bPsi_i], \bar{\CD}_i \bPsi_j\Big] \;,
\end{equs}
where $\mcH^{(3)}$ denotes a term involving trees of the form $\<I'(R2-1)new> $,
$\<IXiI'(R2-1)new>$,
$ \<I'XiI(R2-1)new>$ in the group \hyperlink{(1-2)}{(1-2)}, which are not renormalised.
%we used the shorthand $\hat\CD_i \eqdef \frac12 \CD_i^+(\cS_{\northwest}+\cS_{\southwest})$ in terms that are not renormalised.
The other terms in the final two lines involve trees $\<IXiI'Xibar>$, $\<supercherry1>$, $\<I'XiIXibar>$, $\<supercherry2>$ from the group \hyperlink{(1-3)}{(1-3)},
which {\it are} now renormalised according to~\eqref{eq:square_cherry_consts}.
The only terms from the first five lines that the renormalisation map affects are $[\bPsi_j^{(a)} , -\bar{\CD}_i \bPsi_j ]$,
$[\cS_a \mcb{I} [a_i^{(a)} +\bPsi_i^{(a)} ,  2 \CD_i \bPsi_j ]  , -\bar{\CD}_i \bPsi_j]$,
and
$[\bPsi_j^{(a)},
 -\frac12
\CD_i^+ (\cS_{\northwest}+\cS_{\southwest}) \mcb{I} [a_i^{(a)} +\bPsi_i^{(a)} ,  2 \CD_i \bPsi_j ]  ]$.
Recalling~\eqref{eq:tilde_C_renorm}, the first term does not contribute anything to the renormalised equation
thanks to the Lie bracket.
Therefore,
upon applying the renormalisation map, \eqref{e:quad-exp-3} is mapped to itself plus
\begin{equs}[eq:non_local_A_j]{}
&\frac{1}{16}\sum_{\be,\be'\in \CE_\times}
\!\!
\Big\{
2\ad_{\hat{C}^\e_{2,\be}}
 \,(\bPsi_i+a_i(\cdot+\be+\be'))
-
\!\!\!\!
\sum_{\bw\in\{\Northwest,\Southwest\}}
\!\!\!\!
\ad_{\hat{C}^\e_{3,\be,\bw}}(\bPsi_i+a_i(\cdot^\bw+\be'))
\Big\}
\\
&+ \frac{\e^{\kappa}}{4} \sum_{\be\in\CE_\times}
\!\!
\Big\{
2\ad_{C^\e_{\be}[\<I'XiIXibar>]}(\CS_{\be}\Psi_j + a_j(\cdot+\be))
-
\!\!\!\!
\sum_{\bw\in\{\Northwest,\Southwest\}}
\!\!\!\!
\ad_{C^\e_{\be,\bw}[\<IXiI'Xibar>]}(\CS_{\bw}\Psi_j + a_j(\cdot^\bw))
\Big\}\;.
\end{equs}
By a similar calculation as in \eqref{e:calc-double-shifts}-\eqref{eq:C_1_ren_eq},
the first line in~\eqref{eq:non_local_A_j} can be written as a local term plus a difference as
\begin{equs}{}
&2\ad_{\hat{C}^\e_2}
 \,(\bPsi_i+a_i )
-
2\ad_{\hat{C}^\e_3}(\bPsi_i+a_i)
\label{eq:C_2_C_3_ren_eq}
\\
&+  \frac{1}{16} \!\! \sum_{\be,\be'\in \CE_\times} \!\! 
\Big\{
 2\ad_{\hat{C}^\e_{2,\be}} (a_i(\cdot+\be+\be')-a_i)
-
\!\!\!
\sum_{\bw\in\{\Northwest,\Southwest\}}
\!\!\!
\ad_{\hat C^\e_{3,\be,\bw}}(a_i(\cdot^\bw+\be')-a_i)\Big\}\;.
\end{equs}
Again, all the terms above take values in $\CT^{(i)}\otimes \mfg$.
For instance, for a term $\mcb{I} [a_j + \bPsi_j , \CD_j \bPsi_j] $  in \eqref{e:quad-exp-3}
it take values in $\CT^{(j)}\otimes \mfg$,
but then $\CD_i^+(\cS_{\northwest}+\cS_{\southwest})$ maps it into $\CT^{(i)}\otimes \mfg$
according to Section~\ref{sec:tr-reg}.

By \eqref{e:CA-expand} and the definition of $\cE$ in \eqref{eq:B_expand}, one has, up to degree $1$,
\begin{equs}
\cE \CA^\e_i
 = (\e a_i + \e \Pi \bPsi_i ) \bone^{(i)}
+\e^\kappa\bar\bPsi_i
+\frac{\e^{1+\kappa}}{2}  \mcb{I} [  \bPsi_i +a_i \bone^{(i)},   \bar\bXi_i ]
+\e\langle \nabla a_i, {\bf X}^{(i)}\rangle\;.
\end{equs}
Below we write $\e A_i = \e a_i + \e \Pi \bPsi_i = \e \hat\CR \CA_i^\e$ 
(the second equality follows from Lemma~\ref{lem:Rec_Mult_eps}).
So, up to degree $1$,
\begin{equs}
F(\cE \CA^\e_i)
 &= F(\e A_i) \bone^{(i)} 
 \\
 &\quad + F'(\e A_i)
\Big(\e^\kappa\bar\bPsi_i
+\frac{\e^{1+\kappa}}{2}  \mcb{I} [  \bPsi_i +a_i \bone^{(i)},   \bar\bXi_i ]
+\e\langle \nabla a_i, {\bf X}^{(i)}\rangle \Big)\;.
\end{equs}
Here, recalling $F$ from \eqref{eq:def_F}, 
$F'(\e A_i)$   takes values in $L(\mfg\otimes \mfg,\mfg)$ and
the right-hand side of the above identity takes values in $\CT^{(i)}\otimes L(\mfg,\mfg)$.
Moreover the term $ [\CA^\e_i,\bar\bXi_i]$
is expanded as,  up to degree $0$,
\begin{equs}[e:multi-noise-exp]
{}&   \Big[
 \bPsi_i + a_i
+
 \mcb{I} \big[\bPsi_j^{(a)} + a_j^{(a)}, 
2\CD_j \bPsi_i +\bar{\CD}_i \bPsi_j \big]
+
\mcb{I} \big[ \bPsi_i +a_i  , \CD_i \bPsi_i \big]
\;, \;\;
\bar\bXi_i \Big] 
\\
&+
 \Big[
\frac{\e^\kappa}{2}  \mcb{I} [  \bPsi_i +a_i ,   \bar\bXi_i ] \;, \;\; \bar\bXi_i \Big] 
\\
&+\Big[ - \frac{\eps^\kappa}{2} \mcb{I}  [(\CD_j^+ +\CD_j^-) \tilde\bPsi_i, \bar{\CD}_j \tilde\bPsi_j]
 + \frac{\e^\kappa}{4}\mcb{I} [\CD_j^+ \cS_{\southeast} \tilde\bPsi_j, \CD_j^+ \cS_{\southwest} \tilde\bPsi_j]\;,
 \;\;
\bar\bXi_i \Big]\;.
\end{equs}
The first line involves the trees $\<Xi1>\,$,
$  \<PsiXi1>$, 
 $\<cherryY>$, and
$\<cherryYless>$ from group \hyperlink{(1-3)}{(1-3)}.
They are not renormalised,
except for $\<PsiXi1>\,$, but this does not contribute to the renormalised equation due to the Lie bracket
$[  \bPsi_i , \bar\bXi_i ] $.
The second line
involves 
$\<I(PsiXi1)Xi1>$ and $\<I(PsiXi1)Xi1less>$ from group \hyperlink{(3-3)}{(3-3)} which require renormalisation by  \eqref{e:M-XiIXi}.
The last line involves trees of the form $\<I(R2-1)Xi1new>$ from group \hyperlink{(2-3)}{(2-3)} which do not require renormalisation.

Therefore,  up to degree $0$,
\begin{equs}\label{e:F-multi-noise-exp}
\e^\kappa & F(\cE \CA^\e_i)[\CA^\e_i,\bar\bXi_i]
= 
\e^\kappa  F(\e A_i)  \eqref{e:multi-noise-exp}
\\
& + \e^\kappa 
 F'(\e A_i)
\Big(
\e^\kappa\bar\bPsi_i
+\frac{\e^{1+\kappa}}{2}  \mcb{I} [  \bPsi_i 
+a_i ,   \bar\bXi_i ]
+\e\langle \nabla a_i, {\bf X}^{(i)}\rangle \Big)
[ \bPsi_i + a_i,
\bar\bXi_i ] \;.
\end{equs}
Here 
the right-hand side takes values in $\CT^{(i)}\otimes \mfg$.
The trees involved in the first term on the right-hand side are explained below \eqref{e:multi-noise-exp}.
The second line
involves the trees  $\<I(PsiXi1)Xi1>$,
$\<I(PsiXi1)Xi1less>$, $\<r_z_large>$, $\<PsiPsibarXibar>$ from group \hyperlink{(3-3)}{(3-3)}
 and 
 $ \X^{(i)} \bar\Xi_i$,
$\X^{(i)} \<PsiXi1>$  from group \hyperlink{(1-3)}{(1-3)}.\footnote{The second line
should be understood as $\e^\kappa$ times the derivative
of $F$ at $\e A_i$ in the direction given by the expression in the parenthesis, which gives us an element of $\CT\otimes L(\mfg,\mfg)$, and then we apply this to $[ \bPsi_i+a_i,\bar\bXi_i ] \in \CT\otimes \mfg$
which yields an element of $\CT\otimes \mfg$.}

To find the action of renormalisation map on \eqref{e:F-multi-noise-exp},
we note that 
the contractions of  $\<PsiXi1>$ specified  in \eqref{e:M-XiIXi}
in the trees $\<r_z_large>$, $\<PsiPsibarXibar>$ and
$\X^{(i)} \<PsiXi1>$
do not have any contribution to the renormalised equation thanks to the last Lie bracket in 
\eqref{e:F-multi-noise-exp}.
Only the contractions of  $\<I(PsiXi1)Xi1less>$  in \eqref{e:M-XiIXi} and 
of  $\<r_z_large>$
in \eqref{e:M-four} need to be considered. 

Recall from \eqref{eq:def_F} and \eqref{eq:Phi_inverse} that
\begin{equ}\label{eq:def_F-rep}
F(X) Z %= \frac{\Phi(X)^{-1}-1}{\ad_X}
= \sum_{k=0}^\infty c_k  \ad_X^k Z
 = \frac12 Z + \frac{1}{12}\ad_X Z %+ r_F(X)\;, \qquad  r_F(X) = O(X^2)
 +O(X^2)Z
\end{equ}
for some coefficients $c_k\in\R$. So,
\begin{equ}\label{eq:F-prime}
F'(X)(Y)Z 
=\sum_{k=1}^\infty c_k \sum_{m=0}^{k-1} \ad_X^m \, \ad_Y \, \ad_X^{k-m-1} Z 
= \frac{1}{12}\ad_Y Z + O(XY)Z\;. %+ r'_F(X)(Y)\;.
\end{equ}
Consider the following term from the first line on the right-hand side of \eqref{e:F-multi-noise-exp}:
\begin{equ}[e:F-multi-noise-1]
\frac{\e^{2\kappa}}{2}  F(\e A_i)  
[\mcb{I} [  \bPsi_i +a_i \bone^{(i)},   \bar\bXi_i ],  \bar\bXi_i ] \;.
\end{equ}
Upon applying by $M^\e$,  by  \eqref{e:M-XiIXi} and the exact value of $C^\e[ \<I(PsiXi1)Xi1less> ]= \e^{-2\kappa}\Cas$ given in \eqref{e:C-XiIXi},\footnote{Heuristically, 
the renormalisation $-\frac{1}{4} \ad_{\Cas}(\bPsi_i+a_i \bone^{(i)})$ for \eqref{e:F-multi-noise-1} could be understood 
as obtained by ``substituting $A$ in $\frac\eps2[A,\xi]$ in 
\eqref{e:Z-terms} by $\frac\eps2 \mcb{I}[A,\xi]$''.}
\begin{equs}[e:F-multi-noise-1-result]
\eqref{e:F-multi-noise-1} & \mapsto
 \eqref{e:F-multi-noise-1} 
- \frac{1}{2} F(\e A_i)  
\ad_{\Cas}(\bPsi_i+a_i \bone^{(i)}) 
\\
&=
 \eqref{e:F-multi-noise-1} 
- \frac{1}{4} 
\ad_{\Cas}(\bPsi_i+a_i \bone^{(i)})
%- \frac{1}{2} (\frac{1}{12}\ad_{\e A_i} + r_F(\e A_i))
- \frac{1}{2} (F(\e A_i)-\frac12)\,
\ad_{\Cas}(\bPsi_i+a_i \bone^{(i)}) \;.
\end{equs}
It is important to note that 
the last term 
will not contribute to the renormalised equation.
Indeed, this follows from  $\e A_i = \e\hat\CR \CA^\e_i = \e a_i + \e \Pi \bPsi_i $,
%and, recalling our shorthand  $\e A_i = \e a_i + \e \Pi \bPsi_i $,
%one has 
%$\e A_i = \e \hat\CR \CA^\e_i$
and the fact that, in \eqref{eq:def_F-rep}, the summand for every $k\ge 1$ vanishes if $X=Z$.

Next, 
writing again $(e_l)_l$ for an orthonormal basis of $\mfg$,
consider the following term from the second line of \eqref{e:F-multi-noise-exp},
corresponding to the trees $\<PsiPsibarXibar>$ and $\<I(PsiXi1)Xi1less>\,$:
\begin{equs}\label{e:F-multi-noise-2}
\e^{2\kappa} &
 F'(\e A_i)
(\bar\bPsi_i)
[ \bPsi_i + a_i \bone^{(i)},
\bar\bXi_i ] 
\\
&=
\e^{2\kappa} 
\sum_{j,n} 
e_j^*\otimes e_n^* \otimes
 F'(\e A_i)
(e_j)
[\bPsi_i + a_i\bone^{(i)},
e_n]  
\in ( \CT[\<PsiPsibarXibar>] \oplus \CT[\<I(PsiXi1)Xi1less>] )  \otimes\mfg\;.
\end{equs}
Here and below, for $f\colon \R\times\T^2\to\CT\otimes \mfg$ and $\tilde\CT\subset \CT$, we write $f\in \tilde\CT  \otimes\mfg$ to mean $f$ takes values in $\tilde\CT  \otimes\mfg$.
We remark that the $\CT[\<PsiPsibarXibar>]   \otimes\mfg$ component can be alternatively written as 
$\e^{2\kappa} 
\sum_{j,m,n} 
e_j^*\otimes e_m^* \otimes e_n^*\otimes
 F'(\e A_i) (e_j)
[e_m,e_n]  $.
Contracting $\<I(PsiXi1)Xi1less>$ we have
\begin{equs}
\e^{2\kappa} 
&\sum_{j,n} 
\langle 
e_j^* \otimes e_n^* , C^\e[ \<I(PsiXi1)Xi1less> ]\,
\rangle
 F'(\e A_i)
(e_j)
[ \bPsi_i + a_i \bone^{(i)},
e_n] 
\\ 
&=
\sum_{j}    
 F'(\e A_i)
(e_j)
[ \bPsi_i + a_i \bone^{(i)},
e_j]  \;\;\in ( \CT[\<IXi>] \oplus \CT[\bone^{(i)}] )  \otimes\mfg\;,
\end{equs}
where we used again the value of $C^\e[ \<I(PsiXi1)Xi1less> ]$ in \eqref{e:C-XiIXi}.
Therefore,
upon applying $M^\e$,\footnote{Again heuristically, the renormalisation $\frac{1}{12} \ad_{\Cas}(\bPsi_i+a_i \bone^{(i)})$  for \eqref{e:F-multi-noise-2} can be seen as contracting 
the first $A$ and $\xi$ in $\frac{\eps^2}{12}[A,[A,\xi]]$ in \eqref{e:Z-terms}.}
\begin{equs}[e:F-multi-noise-2-result]
\eqref{e:F-multi-noise-2}
&\mapsto
\eqref{e:F-multi-noise-2}
- \sum_{j} F'(\e A_i)(e_j)[\bPsi_i + a_i \bone^{(i)},e_j]
\\
&=
\eqref{e:F-multi-noise-2}
+ \frac{1}{12} 
\ad_{\Cas}(\bPsi_i+a_i \bone^{(i)})
\\
&\qquad\qquad
- \sum_{j} %r'_F(\e A_i) (e_j)
(F'(\e A_i)(e_j) - \frac{1}{12}  \ad_{e_j})
[\bPsi_i + a_i \bone^{(i)},e_j]\;.
\end{equs}
Next, consider the following terms from the second line of \eqref{e:F-multi-noise-exp},
\begin{equs}\label{e:F-multi-noise-3}
{}& \frac{\e^{1+2\kappa}}{2}  F'(\e A_i ) ( \mcb{I} [  \bPsi_i + a_i \bone^{(i)} ,  \bar\bXi_i ] ) 
[ \bPsi_i + a_i \bone^{(i)},
\bar\bXi_i ]
\\
%&= \frac{\e^{1+2\kappa}}{2}  \sum_{j,n}  e_j^*\otimes  e_n^* F'(\e A_i)([\bPsi_i+a_i \bone^{(i)},e_j])[\bPsi_i+a_i \bone^{(i)},e_n]
%\\
&\qquad\qquad \in (\CT[\<r_z_large>]\oplus \CT[\<PsiPsibarXibar>] \oplus \CT[\<I(PsiXi1)Xi1>] \oplus \CT[\<I(PsiXi1)Xi1less>] )\otimes\mfg\;.
\end{equs}
Contracting $\<I(PsiXi1)Xi1less>$ and $\<r_z_large>$ with the basis of $\mfg$ as above  using the definitions of  $C^\e[\<I(PsiXi1)Xi1less>]$  in \eqref{e:C-XiIXi}, and $C^\e[\<r_z_large>]$
in \eqref{e:def-C-four},
one has that, upon applying $M^\e$, 
\begin{equs}[e:F-multi-noise-3-result]
\eqref{e:F-multi-noise-3}
\mapsto
\eqref{e:F-multi-noise-3}
& -  \frac{\e}{2} \sum_{j}  F'(\e A_i)([\bPsi_i+a_i \bone^{(i)},e_j])[\bPsi_i+a_i \bone^{(i)},e_j] 
\\
& - \frac{\e^{1+2\kappa}}{2} ( c_1^\e[\<r_z_large>]-c_2^\e[\<r_z_large>]) \sum_{j,k}  F'(\e A_i) ([e_j,e_k]) [e_j,e_k]\;.
\end{equs}
Here the second term on the right-hand side belongs to
$( \CT[ \<IXi^2>]\oplus \CT[ \<IXi>] \oplus \CT[ \bone^{(i)}] )\otimes\mfg$,
and the third term belongs to $\CT[ \bone^{(i)}] \otimes\mfg$.

Combining~\eqref{e:F-multi-noise-1-result},~\eqref{e:F-multi-noise-2-result}, and~\eqref{e:F-multi-noise-3-result},
we find that upon applying $M^\e$,
\begin{equ}[eq:F_ren_eq]
\eqref{e:F-multi-noise-exp} \mapsto \eqref{e:F-multi-noise-exp} 
- \frac{1}{6} 
 \ad_{\Cas}(\bPsi_i+a_i \bone^{(i)})
+ (\cdots) \;,
\end{equ}
where $- \frac{1}{6} =- \frac{1}{4}+ \frac{1}{12}$ and
$ (\cdots)$ denotes the sum of
the last term in \eqref{e:F-multi-noise-1-result} (which will not contribute 
to the renormalised equation upon reconstruction by the remark following~\eqref{e:F-multi-noise-1-result}),
the last term in \eqref{e:F-multi-noise-2-result}, and the last two terms  in \eqref{e:F-multi-noise-3-result}.
Since $\e A_i = \e \hat\CR \CA^\e_i$, the contribution of $(\cdots)$
to the renormalised equation (upon reconstruction) is precisely
$
-h_1^\e(\hat\CR \CA^\e_i)
-h_2^\e(\hat\CR \CA^\e_i)
-h_3^\e(\hat\CR \CA^\e_i)
$
for $h^\e_\ell$ defined by~\eqref{e:def-h123A},
which is cancelled by $\boldH_i$ in~\eqref{e:fix-pt-Ae}.

Finally, for $ \CQ^{\e}_i$ defined by~\eqref{e:def-tildeCR2},
we have $\mcb{D} \CV = \CD\CV + b$ for the various differentiations, where $b\in \CT[\bone]\otimes \mfg$.
Since the product of $\mcb{D}\bar\bPsi$ and $b$ is not renormalised,
it suffices to expand $ \CQ^{\e}_i$ with all $\mcb{D}$ replaced by $\CD$,
which gives, up to degree $0$, 
\begin{equs}[e:tildeR2-exp]
  {}& 
%  \mcH^{(4)}
  - \frac{\eps^\kappa}{2}  [(\CD_j^+ +\CD_j^-) \tilde\bPsi_i, \bar{\CD}_j \tilde\bPsi_j]
   + \frac{\e^\kappa}{4} [\CD_j^+\cS_{\southeast} \tilde\bPsi_j, \CD_j^+\cS_{\southwest} \tilde\bPsi_j]
 \\
 &
 -\frac{\eps^\kappa}{4}  
 \Big[(\CD_j^+ +\CD_j^-) \bar\bPsi_i, 
{\CD}_j^+ (\cS_{\southeast}+\cS_{\southwest})\CI_j
\CH_i\{\<IXiI'Xi_notriangle>\}
 \Big]
 \\
&-\frac{\eps^\kappa}{2}  
 \Big[
 (\CD_j^+ +\CD_j^-) \CI_i
\CH_j\{\<IXiI'Xi_notriangle>\}, 
 \bar{\CD}_j \bar\bPsi_j
 \Big]
\\ 
&
+ \frac{\e^\kappa}{4} 
\Big[
\CD_j^+ \cS_{\southeast}\bar\bPsi_j, \CD_j^+\cS_{\southwest}\CI_j 
\CH_i\{\<IXiI'Xi_notriangle>\}
 \Big]
+ \frac{\e^\kappa}{4} 
\Big[
  \CD_j^+\cS_{\southeast}\CI_j 
\CH_i\{\<IXiI'Xi_notriangle>\}
  ,  \CD_j^+ \cS_{\southwest}\bar\bPsi_j
\Big]
 \\
 &
 -\frac{\eps^\kappa}{2}  
 \Big[(\CD_j^+ +\CD_j^-) \bar\bPsi_i, 
\frac{\eps^\kappa}{4}  
{\CD}_j^+(\cS_{\southeast}+\cS_{\southwest})
\mcb{I}_j [\bPsi_j+a_j,\bar\bXi_j]
 \Big]
 \\
 &
  -\frac{\eps^\kappa}{2}  
 \Big[
 \frac{\eps^\kappa}{2}  
 (\CD_j^+ +\CD_j^-) 
\mcb{I} [\bPsi_i+a_i,\bar\bXi_i],
 \bar{\CD}_j \bar\bPsi_j
 \Big]
 \\ 
&
+ \frac{\e^\kappa}{4} 
\Big[
\CD_j^+ \cS_{\southeast}\bar\bPsi_j,
\frac{\eps^\kappa}{2}  
 \CD_j^+ \cS_{\southwest}
\mcb{I} [\bPsi_j+a_j,\bar\bXi_j]
 \Big]
 \\
 &
+ \frac{\e^\kappa}{4} 
\Big[
\frac{\eps^\kappa}{2}  
\CD_j^+  \cS_{\southeast}
\mcb{I} [\bPsi_j+a_j,\bar\bXi_j]
  ,  \CD_j^+ \cS_{\southwest} \bar\bPsi_j
\Big]
+ \CH_i^{(4)}\{\<I'Xi1I'(R2-1)new>\}
\end{equs}
where 
\begin{equ}[e:def-CHi]
\CH_i \{\<IXiI'Xi_notriangle>\}\eqdef
  \big[\bPsi_i^{(a)} + a_i^{(a)}, 
2\CD_i \bPsi_j +\bar{\CD}_j \bPsi_i \big]
+
\big[ \bPsi_j +a_j \bone^{(j)} , \CD_j \bPsi_j \big]\;,
\quad j\eqdef 3-i
\end{equ}
and $\CH_j \{\<IXiI'Xi_notriangle>\}$
is defined in the same way with $i,j$ exchanged.
The  term
$\CH_i^{(4)}\{\<I'Xi1I'(R2-1)new>\}$
involves trees of the form $ \<I'Xi1I'(R2-1)new>$
in group \hyperlink{(2-2)}{(2-2)},
which is not renormalised.
%$\mcH^{(4)}$
The first line involves trees of the form $\<R2-1new>$ in \eqref{e:trees-deg-1-rem}; 
and the last four lines (excluding $\CH_i^{(4)}\{\<I'Xi1I'(R2-1)new>\}$)
involve $\<cherry231>$ and
$\<cherry232>$ in group \hyperlink{(2-3)}{(2-3)},
in which case the first and the second terms are not renormalised since $i\neq j$.
The remaining terms are all of the form
$\<I'XiI'[IXiI'Xi]>$ and $\<I'XiI'[IXiI'Xi]less>$ in group \hyperlink{(1-2)}{(1-2)}.

By definition of $M^\e$ in Section~\ref{subsec:renorm_group},
especially  \eqref{e:div-R2} for trees of the type $\<I'XiI'[IXiI'Xi]less>$
and \eqref{eq:square_cherry_consts} for $\<cherry232>$,
% only $6$ terms in \eqref{e:tildeR2-exp} will be renormalised:
the  terms in \eqref{e:tildeR2-exp} which will be renormalised are:
\begin{itemize}
\item
%they are precisely 
the $6$ terms containing $\CH_i \{\<IXiI'Xi_notriangle>\}$
(but not $\CH_j \{\<IXiI'Xi_notriangle>\}$);
\item
the last two lines (excluding $\CH_i^{(4)}\{\<I'Xi1I'(R2-1)new>\}$).
\end{itemize}
Upon applying $M^\e$ \eqref{e:tildeR2-exp} is mapped to itself plus
\begin{equs}{}
&- \frac{\e^\kappa}{4} 
(
\ad_{C_{\southeast,+}^\e [\, \<I'XiI'[IXiI'Xi]less>\,]}
+
\ad_{C_{\southeast,-}^\e [\, \<I'XiI'[IXiI'Xi]less>\,]}
+2\ad_{C_\southeast^\e [\, \<I'XiI'[IXiI'Xi]less>\,]} 
)
(\bPsi_i + a_i^{(a)}(\cdot^{\southeast}))
\\
&- \frac{\e^\kappa}{4}
(\ad_{C_{\southwest,+}^\e [\, \<I'XiI'[IXiI'Xi]less>\,]}
+\ad_{C_{\southwest,-}^\e [\, \<I'XiI'[IXiI'Xi]less>\,]}
-2\ad_{C_\southwest^\e [\, \<I'XiI'[IXiI'Xi]less>\,]}
)
(\bPsi_i + a_i^{(a)}(\cdot^{\southwest}))
\\
& + \frac{\e^{2\kappa}}{8}
\ad_{C^\e [\<cherry232>] } (\bPsi_j + a_j)
-\frac{\e^{2\kappa}}{8}
\ad_{C^\e [\<cherry232>] } (\bPsi_j + a_j)
\;.
\end{equs}
Note that the last two terms cancel. The first line above
can be written as a local term plus a difference as
\begin{equs}[e:C_tree_ren_eq]{}
&- \frac{\e^\kappa}{4} 
(
\ad_{C_{\southeast,+}^\e [\, \<I'XiI'[IXiI'Xi]less>\,]}
+
\ad_{C_{\southeast,-}^\e [\, \<I'XiI'[IXiI'Xi]less>\,]}
+2\ad_{C_\southeast^\e [\, \<I'XiI'[IXiI'Xi]less>\,]} 
)
(\bPsi_i + a_i)
\\
&-
\frac14\sum_{\be\in\CE_\times} \frac{\e^\kappa}{4} (
\ad_{C_{\southeast,+}^\e [\, \<I'XiI'[IXiI'Xi]less>\,]}
+
\ad_{C_{\southeast,-}^\e [\, \<I'XiI'[IXiI'Xi]less>\,]}
+2\ad_{C_\southeast^\e [\, \<I'XiI'[IXiI'Xi]less>\,]} 
)
(a_i(\cdot^{\southeast}+\be) - a_i)\;,
\end{equs}
and likewise for the second line with $\Southeast$ and $2\ad_{C_\southeast^\e [\, \<I'XiI'[IXiI'Xi]less>\,]}$ replaced by $\Southwest$ and $-2\ad_{C_\southwest^\e [\, \<I'XiI'[IXiI'Xi]less>\,]}$.

We now observe that $\delta_i\hat\CR\CV_i^\e$ in~\eqref{e:fix-pt-Ae}, by its definition~\eqref{eq:delta_A_def}, upon reconstruction cancels the 
`non-local' renormalisations above involving $a_i$, i.e. the second lines of~\eqref{eq:C_1_ren_eq},~\eqref{eq:C_2_C_3_ren_eq}, and~\eqref{e:C_tree_ren_eq} (and its analogue for $\Southwest$).
Likewise, $\bar\delta_i\hat\CR\CA_j^\e$ in~\eqref{e:fix-pt-Ae}, by its definition~\eqref{eq:bar_delta_def},
cancels the second line of~\eqref{eq:non_local_A_j}.
Furthermore, recalling
$C^\e_{\sym}$ defined in \eqref{e:def-CSYM},
the reconstruction of the sum of the renormalisation terms in~\eqref{eq:cubic_ren_eq},~\eqref{eq:quad1_ren_eq},
and the first lines of~\eqref{eq:C_1_ren_eq} and~\eqref{eq:C_2_C_3_ren_eq}
is $\ad_{C^\e_{\sym}} \hat\CR\CA_i^\e$.
Finally, recalling $C^\e_{\rem}$ defined in~\eqref{e:def-CREM},
the reconstruction of the sum of the renormalisation terms in the first lines of~\eqref{eq:F_ren_eq} and~\eqref{e:C_tree_ren_eq} (and its analogue for $\Southwest$) is
$\ad_{C^\e_{\rem}} \hat\CR\CA_i^\e$.

To finally find the equation that $\hat\CR\CA^\e$ solves,
note that $(\hat\Pi^\e,\hat\Gamma^\e)$ is admissible, and thus $\hat\CR \CP^{i;\e} = P^{i;\e} *_{(i)} \hat\CR$.
It therefore suffices to find $\hat\CR f$ as a function of $\hat\CR\CA^\e$, where $f$ is any one of the terms in the parentheses $(\cdots)$ in~\eqref{e:fix-pt-Ae}.
To this end, for all such $f$ and all $x\in\R\times\obonds_i$,
\begin{equ}
\hat \CR f(x) = (\hat \Pi_x f(x))(x) = (\Pi_x M^\e f(x))(x) = \CR (M^\e f)(x)\;,
\end{equ}
where $\CR$ is the reconstruction operator associated to the canonical model $(\Pi^\e,\Gamma^\e)$
(this identity holds as written if $\hat \CR f$ is continuous in time
and needs to be understood in the distributional sense otherwise, i.e. 
for $f=\bXi_i$ or $f=\e^\kappa F(\cE\CA_i^\e)[\CA^\e_i, \bar\bXi_i]$).
Moreover $\CR$ commutes with derivatives (Lemma~\ref{lem:DRcommute}), shiftings (Lemma~\ref{lem:RecShft}), products (Remark~\ref{rem:canonical_R_mult}), and with composition of smooth functions,
and we have
\begin{equ}
\CR\CA_i^\e(x)=(\Pi_x \CA_i^\e(x))(x) = (\hat \Pi_x \CA_i^\e(x))(x) = \hat\CR \CA_i^\e(x)\;,
\end{equ}
and likewise for $\CV_i^\e, \bXi_i, \bar \bXi_i$ and for $\mcb{D}\cS\CA^\e$ for any derivative and shift operators $\mcb{D},\cS$.
We conclude from the above computations that $\hat\CR\CA^\eps$ solves the discrete SPDE~\eqref{e:Aeps} plus the following terms on the right-hand side:
\begin{equ}[e:Ae-renormalised-mass]
(\ad_{C^\e_{\sym}}+\ad_{C^\e_{\rem}})A_i^\e +  (c^\eps_{\star,i} -  \tilde{c}^\e_{\star,i} ) A^\e\;.
\end{equ}
Note that $C^\e_{\sym}=O(1)$ and $C^\e_{\rem}=O(1)$ by Lemmas~\ref{lem:C_sym-finite} and~\ref{lem:thick-trig-bahave} respectively,
while $\tilde c^\e_\star = O(1)$ by
Definition~\ref{def:tildec} and since each $c^\e$ in Lemmas~\ref{lem:I3} and \ref{lem:grad-R-hat} is $O(1)$ due to Lemma~\ref{lem:Psi_Wick}.
The proof is therefore done with the choice
\begin{equ}[e:specify-c_star]
c^\e_\star \eqdef - \ad_{C^\e_{\sym}} - \ad_{C^\e_{\rem}} + \tilde{c}_\star^\e = O(1)\;.
\end{equ}
\end{proof}

\subsection{Uniform bounds on renormalised discrete models}
\label{subsec:moments_discrete_models}

In this section we derive uniform moment bounds on the renormalised models
$\hat Z^\eps = (\hat\Pi^\e,\hat\Gamma^\e)$ from Section~\ref{subsec:renorm_eq}.
\begin{proposition}\label{prop:Ze_moments}
For all $p\geq 1$, $\gamma>0$, and compact $\K\subset \R\times\T^2$, one has $\sup_{\e>0} \E |\$\hat Z^\e\$^{(\e)}_{\gamma;\K}|^p < \infty$.
\end{proposition}
\begin{proof}
The proof follows from Lemmas~\ref{lem:mom-YM},~\ref{lem:mom-xi},~\ref{lem:mom-12},~\ref{lem:mom-13} and~\ref{lem:mom-222333}
and a Kolmogorov-type criterion (see~\cite[Thm.~6.1]{HM18} or~\cite[Thm.~10.7]{Hairer14}).
\end{proof}
The main ingredient in the above proof is the bounds
\minilab{e:mom-bounds}
\begin{equs}
\E[(\hat\Pi_{z_*}^\e\tau)(\phi_{z_*}^{\lambda})^2]
&\lesssim 
\lambda^{2(|\tau|+\kappa)},			\label{e:mom-bounds1}
\\
\E\Big[\Big(\int (\hat\Pi_{z_*}^\e\tau)(s,x_*)(\CS_{2,t_*}^{\lambda}\phi)(s)\mrd s\Big)^2\Big]
&\lesssim \lambda^{2(|\tau|+\kappa)}\;,		\label{e:mom-bounds2}
\end{equs}
for each  $\tau\in \CT^{(i)}_-\cap \{\CT^\YM\oplus\CT^\rem\}$
(recall Remark~\ref{rem:model_bounds_sector}), uniformly in $z_*=(t_*,x_*)\in \R\times \obonds_i$, 
and $\e > 0$. Here the second moment bounds suffice since our noise is Gaussian.
The first bound is also uniform in $\lambda\in (\e,1]$ 
and all space-time test functions $\phi$ as in~\eqref{eq:Pi},
and the second bound is uniform in $\lambda\in (0,\e]$ and all temporal  functions $\phi$ as in~\eqref{eq:seminorm}.
The rest of this subsection is devoted to the proof of these bounds.

We will need to verify \eqref{e:mom-bounds} for each of our trees in Section~\ref{sec:tr-reg}
{\it one by one} for several reasons.
First, in our discrete setting, we do not have the general ``black-box'' results for probabilistic bounds as in \cite{CH16,HS23_BPHZ,LOTT21} (see also~\cite{BB23,BN22_discrete,BN22,KPZJeremy}).
Also we do not  use the general discrete bounds in the appendix of \cite{EH21},
since in our setting, verifying the conditions therein would  not be easier.
Another, more important reason
is that 
since we introduced the noises such as $\bar\xi=\e^{1-\kappa} \xi$, the models
will yield functions possibly containing factors of $\e$ which are crucial for us to obtain  \eqref{e:mom-bounds}.
Moreover,
we need to carefully examine parity, i.e. whether 
certain functions in the integrands are exactly or only approximately odd,
due to discrete derivatives (see Remark~\ref{rem:loss-odd}) and the shifting operators. 
%and this will further relate with the verification that the equation for $A_i$
%does not involve any mass renormalisation term of the form $cA_j$ for $j\neq i$
%(so in a sense the general argument as in \cite[Section~4]{CCHS_3D} does not directly apply to the discrete setting here).
Fortunately, in 2D  the number of  trees is  relatively small
so it is manageable to obtain the bounds by hand.

Recall  the kernel norm $\VERT \cdot \VERT_{\zeta;m}^{(\eps)}$
and the notation $\|z\|_\e$ from \eqref{e:DSingKer}. For the rest of this subsection we drop the $\e$-dependence in our kernel notation, e.g. $K=K^\e$.
%The following lemma states that slightly shifting a kernel will essentially preserve its norm.
\begin{lemma}\label{lem:shift-K}
Let $\zeta<0$. %For $\be\in \CE_+\cup \CE_\times$, $\tilde K(z) = K(z+\be)$,
For two kernels $\tilde K$ and $K$ such that $\tilde K(z) = K(z+a)$ where $a=a_1\e_1+a_2\e_2$ with $|a_1|+|a_2|<10$,
one has
$\VERT \tilde K \VERT_{\zeta;m}^{(\eps)}
\asymp
\VERT K \VERT_{\zeta;m}^{(\eps)}$.
\end{lemma}

\begin{proof}
This is immediate from the definition \eqref{e:DSingKer} and $\|z+a \|_\e \asymp \|z\|_\e$.
\end{proof}
We write $K_1\bigcdot_\e K_2$  for any  product of the form $K_1(z)K_2(z+a)$,
and $K_1*_\e K_2$ for any  convolution  of the form $K_1(\cdot ) * K_2(\cdot+a)$,
%where $\be$ is a generic element of $ \CE_+\cup \CE_\times$.
where $a$ is as in Lemma~\ref{lem:shift-K}.
Here $*$ is the (semi-discrete) convolution over space-time.
We will also write $K^{*_\e 2} \eqdef K*_\e\hat K$,
where $\hat K$ the reflection of $K$ as in Section~\ref{subsubsec:behaviour}.
Recall  $\s=(2,1,1)$ and $|\s|=4$.

\begin{lemma}\label{lem:OpDSingKer}
Let $K_1$ and $K_2$ be functions of order $\zeta_1$ and $\zeta_2$ respectively. 
Let $m\in\N$, $\bar{\zeta}\eqdef \zeta_1+\zeta_2+|\s|$.
Then one has, uniformly in $\eps$ for $\e$ sufficiently small,
\begin{enumerate}
\item
(Multiplication bound):
$\VERT K_1 \bigcdot_\e K_2 \VERT^{(\eps)}_{\zeta_1+\zeta_2;m}
\lesssim 
\VERT K_1 \VERT^{(\eps)}_{\zeta_1;m}
\VERT K_2 \VERT^{(\eps)}_{\zeta_2;m}$.
\item 			\label{item:convolution}
(Convolution bound): $\VERT K_1 \ast_\eps K_2 \VERT^{(\eps)}_{\bar{\zeta};m}
\lesssim 
\VERT K_1\VERT^{(\eps)}_{\zeta_1;m}
\VERT K_2 \VERT^{(\eps)}_{\zeta_2;m}$ provided that $\zeta_1\wedge \zeta_2>-|\s|$ and $\bar{\zeta}<0$.
%\end{equation}
%\item 			\label{item:positive}
%If $\bar{\zeta}\in (0,2) \setminus\N$, then
%$\VERT \bar{K}^\eps\VERT^{(\eps)}_{\bar{\zeta};m}\lesssim \VERT K_1^\eps\VERT^{(\eps)}_{\zeta_1;\bar{m}}\VERT K_2^\eps\VERT^{(\eps)}_{\zeta_2;\bar{m}}$ where
%$\bar{m}=m\vee(\lfloor\bar{\zeta}\rfloor +2)$ and
%\[
%\bar{K}^\eps(z)
%\eqdef K_1^\eps\ast_\eps K_2^\eps(z)
%	-\sum_{|k|_\s<\bar{\zeta}} 
%	\frac{z^k}{k!} D_\eps^k (K_1^\eps\ast_\eps K_2^\eps) (0)\;.
%%\quad (t,x)_{k,\eps} \eqdef t^{k_0} \prod_{i\neq 0}\prod_{0\le j<k_i} (x_i - \eps j)\;.
%\]
%%\begin{equation}\label{b:ConvDSingKer2}. 
\item				\label{item:renormalized}
(Renormalised kernel):
If
$\zeta_1 \in \bigl(-|\s|-1, -|\s|\bigr]$ and $\zeta_2 \in \bigl(-2 |\s|-\zeta_1, 0\bigr]$, 
\begin{equ} [e:RenorConv]
\VERT{\bigl(\mathscr{R} K_1\bigr) \ast_\eps K_2}\VERT_{\bar{\zeta}; m} \lesssim \VERT{K_1}\VERT_{\zeta_1; m} \VERT{K_2}\VERT_{\zeta_2; m + 2}
\end{equ}
where 
$\left(\mathscr{R} K\right)(\varphi) \eqdef \int_{\R \times (\e\Z)^2} K(z) \left( \varphi(z) - \varphi(0) \right) dz$
 for any $\varphi \in \CC^\infty_0(\R^{3})$.
\end{enumerate}
\end{lemma}
\begin{proof}
This follows from Lemma~\ref{lem:shift-K} and \cite[Lemmas~7.3, 7.5]{HM18}.
\end{proof}
%
%Our first step in the proof of the bounds \eqref{e:mom-bounds} is the same as that in~\cite{Hairer14,HM18}. Namely, 
As in~\cite{Hairer14,HM18},
for each $\tau$ we have a Wiener chaos decomposition for
$\hat\Pi^{\eps}_{z_*} \tau (\phi) = \sum_k I_k^\e ((\CW^{(\eps; k)}\tau)(\phi))$
where $I_k^\e$ is the $k$th order Wiener integral with respect to $\xi^\e$,
and $\CW^{(\eps; k)}\tau$ are kernels such that
$(\CW^{(\eps; k)}\tau)(z) \in H_\e^{\otimes k}$   
with $H_\e$ being the Hilbert space for $\xi^\e$.
Similarly, for a temporal test function $\phi$,
\[
\int_{\R} (\hat\Pi_{z_*}^\e\tau)(s,x_*)\phi(s)\mrd s = 
\sum_k I_k^\e 
\int_{\R} 
(\CW^{(\eps; k)}\tau)(s,x_*) \phi(s)\mrd s\;.
\]
Then we consider their ``symmetric pairings''
\begin{equs}[e:CovDef]
\bigl(\CK^{(\eps; k)}\tau\bigr)(z_1, z_2) \eqdef \langle \bigl(\CW^{(\eps; k)}\tau\bigr)(z_1), \bigl(\CW^{(\eps; k)}\tau\bigr)(z_2) \rangle_{H_\eps^{\otimes k}}\;.
\end{equs}
Below we will always assume $z_1=(t_1,x_1)$ and $z_2=(t_2,x_2)$.
We will also assume $z_*=(t_*,x_*)=(0,0)$ since our model is stationary, but  we often still write 
$z_*=(t_*,x_*)$ as an intuitive placeholder.
%In \cite[Prop.~10.11]{Hairer14}  the uniform bounds on models amounts to finding the bounds
Consider the bound
 \begin{equ}[e:CovarianceBounds]
|\big(\CK^{(\eps; k)}\tau\big)\big(z_1, z_2\big)| 
\lesssim 
\sum \e^{n_\e}
\big( \|z_1\|_\s + \|z_2\|_\s\big)^\zeta 
\|z_1- z_2\|_\e^\alpha   %{2|\tau| +\kappa - \zeta}
\,\delta_\e (z_1-z_2)^{n_\delta}\;,
\end{equ}
where the sum is over finitely many values of $(\zeta,\alpha,n_\delta,n_\e)$ 
such that 
$\alpha\in (-|\s|,0]$, $\zeta\ge 0$, $n_\delta\in \{0,1\}$, $n_\e \ge 0$, 
and 
\[
\alpha+n_\e \in [\, |\s|n_\delta-|\s|, n_\delta d \,) \;, \qquad
\alpha+\zeta-|\s|n_\delta + n_\e = 2|\tau|+\kappa\;.
\]
Here, $\delta_\e(z_1-z_2)  \eqdef \delta(t_1-t_2) \e^{-d} \bone(x_1-x_2)$  where $\bone$ is the indicator of the origin.
We claim that \eqref{e:CovarianceBounds} suffices
to prove both bounds in \eqref{e:mom-bounds}.
Indeed, \eqref{e:CovarianceBounds} $\Rightarrow$ \eqref{e:mom-bounds1} in the regime $\lambda > \e$ 
follows as in \cite[Prop.~10.11]{Hairer14}.
(Note that when $n_\delta=1$, $\e^{n_\e} \|z_1- z_2\|_\e^\alpha \delta_\e (z_1-z_2)$ is integrable uniformly in $\e$ thanks to $\alpha+n_\e \ge 0$.)
Regarding  \eqref{e:mom-bounds2} in the regime  $\lambda\le \e$,
the 2nd moment of the $k$th chaos  $I_k^\e \int_{\R} (\CW^{(\eps; k)}\tau)(t,x_*)(\CS_{2,t_*}^{\lambda}\phi)(t)\mrd t$ is bounded by
\begin{equs}
%\E & \Big[\Big( I_k^\e \int_{\R} (\CW^{(\eps; k)}\tau)(t,x_*)(\CS_{2,t_*}^{\lambda}\phi)(t)\mrd t\Big)^2\Big]
%\\
{}&
\int_{\R^2}  \langle (\CW^{(\eps; k)}\tau)(t_1,x_*), (\CW^{(\eps; k)}\tau)(t_2,x_*)\rangle
(\CS_{2,t_*}^{\lambda}\phi)(t_1)(\CS_{2,t_*}^{\lambda}\phi)(t_2)
\mrd t_1\mrd t_2
\\
&\lesssim
 \sum%_{\zeta \geq 0}
 \int_{\R^2} \e^{n_\e}  \big( \|(t_1,x_*)\|_\s + \|(t_2,x_*)\|_\s\big)^\zeta 
(|t_1- t_2|^{\frac12}\vee \e)^{\alpha}
\\
&\qquad\qquad\qquad 
\cdot  (\e^{-d} \delta (t_1-t_2))^{n_\delta}   (\CS_{2,t_*}^{\lambda}\phi)(t_1)(\CS_{2,t_*}^{\lambda}\phi)(t_2)
\mrd t_1\mrd t_2
%\label{e:IkWk-2nd}
%\\
%&\lesssim
% \sum %_{\zeta \geq 0}
%  \lambda^\zeta
% \int_{\R^2}
%(|t_1- t_2|^{\frac12}\vee \e)^{2|\tau|+\kappa - \zeta}
%(\CS_{2,t_*}^{\lambda}\phi)(t_1)(\CS_{2,t_*}^{\lambda}\phi)(t_2)
%\mrd t_1\mrd t_2
%\\
%&\lesssim 
% \sum %_{\zeta \geq 0} 
% \lambda^\zeta
% \lambda^{2|\tau| +\kappa - \zeta}
%\lesssim
% \lambda^{2 |\tau|+\kappa}\;,	
\\
&\lesssim 
\begin{cases}
 \int_{\R^2}  \lambda^\zeta (|t_1- t_2|^{\frac12}\vee \e)^{\alpha+n_\e}   (\CS_{2,t_*}^{\lambda}\phi)(t_1)(\CS_{2,t_*}^{\lambda}\phi)(t_2)\mrd t_1\mrd t_2 \lesssim \lambda^{\zeta+\alpha+n_\e}
 \\
  \lambda^\zeta \e^{n_\e} \int_\R \e^\alpha \e^{-d} (\CS_{2,t_*}^{\lambda}\phi)(t)^2 \mrd t 
  \lesssim \lambda^{\zeta-2} \e^{\alpha-d+n_\e} \lesssim \lambda^{\zeta+\alpha-|\s|+n_\e}
\end{cases}
\end{equs}
where  the first bound is for $n_\delta =0$ where we used $\alpha+n_\e <0$ in the last step;
the second bound is for $n_\delta =1$ where we used $\alpha+n_\e <d$ in the last step.
In both cases we have the bound $\lambda^{2|\tau|+\kappa}$ as required, which proves the claim.
%Below 
%we will write  (the contribution of a certain chaos to) the left-hand sides of 
% \eqref{e:mom-bounds1} and  \eqref{e:mom-bounds2}
%as  ``L.S.'' and ``S.S.''
% respectively,
%which stand for ``large scale'' and ``small scale'' bounds.
%We always assume $\lambda > \e$ when bounding L.S. and $\lambda\le \e$  when bounding S.S.

In the next lemmas, we keep in mind that $i\neq j$, and that $\kappa>0$ is an arbitrarily small parameter.
We will introduce graphic notation:
$\tikz[baseline=-3]\draw (0,0)--(.8,0);$ stands for the truncated heat kernel $K$,\footnote{understood
as $K(z-w)$ where $z$ is a point lower than $w$ in the graph; same for other kernels.}
$\tikz[baseline=-3]\draw[very thick] (0,0)--(.8,0);$ stands for a generic derivative of $K$,
$\tikz\node at (0,0) [var] {}; $ stands for $\xi^\e$,
$\tikz\node at (0,0) [var1] {}; $ stands for $\e^{1-\kappa}\xi^\e$.
Joining these lines correspond to multiplications or convolutions  possibly with shifts (i.e. 
they can be 
$K_1\bigcdot_\e K_2$  
and $K_1*_\e K_2$ described above), but we do not explicitly show these shifts in the graphs.
Moreover,
$\tikz[baseline=-3]\draw[kernel1] (0,0) --(.8,0);$ stands for the ``positively renormalised'' truncated heat kernel 
$K(z-w)-K(-w)$.% where $z$ is a point lower than $w$ in the graph.

%The following fact will be useful. 
%Let $F(w)$ be a function of $w \in\R\times \obonds_i$.
%Let $J(z,w)$ be a function with arguments 
%$z\in  \R\times \obonds$ and
%$w\in \R\times \obonds_j$, which only depends on the distance $\|z-w\|_\s$ (recall that we identify the bonds $z,w$ with their midpoints).
%Assume that $j\neq i$. We can then write, for any $z\in \R\times \obonds$,
%\[
%\frac14 \sum_{\be\in \CE_\times}(J * _{(j)} F(\cdot + \be))(z) 
%=\frac14 \sum_{\be\in \CE_\times} \int_{\R\times \obonds_j} J(z,w) F(w + \be) \mrd w \;.
%\]
%If we define $J^\times$ by 
%$J^\times (z,w) \eqdef \frac14 \sum_{\be\in \CE_\times} J(z,w-\be)$ where 
%$w\in \R\times \obonds_i$, then
%\begin{equ}[e:shift-in-conv]
%\frac14 \sum_{\be\in \CE_\times}(J * _{(j)} F(\cdot + \be))(z) 
%=
%(J^\times * _{(i)} F)(z) \;.
%\end{equ}

\begin{remark}
In Lemmas~\ref{lem:mom-YM}--\ref{lem:mom-222333} below
we will frequently apply the multiplication and convolution bounds in Lemma~\ref{lem:OpDSingKer}.
For instance to obtain \eqref{e:mom-cherry},
the convolution of the two kernels of degree $-3$ (resp. $-2$) yields a new kernel of degree $-2$ (resp. $-\kappa$),
%and the convolution of the two kernels of degree $-2$ yields a new kernel of degree $-\kappa$,
and the multiplication of the two new kernels is of degree $-2-\kappa$.
This procedure is elementary and standard, so we will directly give the resulting bound
without repeatedly referring to Lemma~\ref{lem:OpDSingKer}.
Also, we will slightly abuse notation
and write $\Psi,\bar\Psi,\tilde\Psi$ for the trees $\mcb{I}\Xi$,  $\mcb{I}\bar\Xi$,  $\mcb{I}\tilde\Xi$, 
as well as the Gaussian processes they realise;
the meaning is always clear from the context.
\end{remark}

\begin{lemma}
	\label{lem:mom-YM}
For every $\tau\in \CT^{\YM}_-$,  the bounds \eqref{e:mom-bounds} hold.
\end{lemma}

\begin{proof}
Consider $\tau=\Xi$ and  recall $\xi^\e= \e^{-1}\dt \BM$ as in Section~\ref{sec:canonical}.
%recall that $\xi^\e = \e^{-1}\dt \BM$ where $\BM$ is a collection of i.i.d. Brownian motions.
One has $\big(\CK^{(\eps; 1)}\tau\big)\big(z_1, z_2\big) = \delta_\e(z_1-z_2)$.
%\begin{equ}[e:def-delta-eps]
%\big(\CK^{(\eps; 1)}\tau\big)\big(z_1, z_2\big) = \delta_\e(z_1-z_2)
% \eqdef \delta(t_1-t_2) \e^{-2} \bone(x_1-x_2)
%\end{equ}
%where $\bone$ equals $1$ at the origin and $0$ otherwise.
%This is not exactly the case as in \eqref{e:CovarianceBounds} due to $\delta(t_1-t_2) $.
%However one still has the desired bounds
%\begin{equs}[e:xiS22]
%\mbox{L.S.} &= \int (\phi^\lambda_{z_*}(z))^2 \mrd z \lesssim \lambda^{-4}\;,
%\\
%%\E \Big[\Big(\int_\R \xi^\e (t,x) \CS^\lambda_2 \phi(t) \mrd t \Big)^2 \Big]
%\mbox{S.S.} 
%%&\le \int_\R \e^{-2}  \delta(t_1-t_2)  \CS^\lambda_{2,t_*} \phi(t_1)\CS^\lambda_{2,t_*} \phi(t_2) \mrd t_1\mrd t_2
%%\\
%&= \int_{\R^2} \e^{-2} \CS^\lambda_{2,t_*} \phi(t)^2 \mrd t
%\lesssim \e^{-2} \lambda^{-2} \lesssim \lambda^{-4}\;, \quad (\mbox{by }\lambda\le \e)\;.
%\end{equs}
%%where in the last step we used $\lambda\le \e$.
So \eqref{e:CovarianceBounds} applies with $(\zeta,\alpha,n_\delta,n_\e)=(0,0,1,0)$.

%Consider $\tau=\<XXi>$. One has
%\[
%\big|\big(\CK^{(\eps; 1)}\tau\big)\big(z_1, z_2\big)\big|
% \lesssim \|z_1\|_\s \|z_2\|_\s   \delta_\e(z_1-z_2)\;.
%\]
%Similarly as \eqref{e:xiS22} we have $\mbox{L.S.} \lesssim \lambda^{-2} $,
% and
%$\mbox{S.S.} \lesssim \e^{-2} \le \lambda^{-2}$,
%where we bound $\|z_1\|_\s \|z_2\|_\s \le \lambda^2$ as in
%\eqref{e:WkSt-general}.
%
%The same argument shows that
%for $\tau=X^2 \Xi$ 
%we have $\mbox{L.S.} \lesssim \lambda^{-\kappa} $,
% and
%$\mbox{S.S.} \lesssim \e^{-\kappa} \le \lambda^{-\kappa}$.

Consider $\tau=\<IXiI'Xi_notriangle>$
which includes
$\CS_{\be} \Psi_j \CD_j \Psi_i $,
$\CS_{\be} \Psi_j \bar\CD_i \Psi_j$, and
$\Psi_i \CD_i \Psi_i $ for  $\be\in \CE_\times$.  %\CS \in \{\CS_a,\CS_{\be},\be\in \CE_\times\}
The $0$th chaos of $(\hat\Pi_z^\e\tau)(\phi_z^{\lambda})$ for each of them is $0$. Indeed,
this holds for $\CS_{\be} \Psi_j \CD_j \Psi_i$ since $i\neq j$,
for $\Psi_i \CD_i \Psi_i $ by Lemma~\ref{lem:ePsiDPsi} thanks to 
the symmetrised derivative $\partial_i$,
and for $\CS_{\be} \Psi_j \bar\CD_i \Psi_j$ by the choice of $\tilde C_{\be}$. %~\eqref{e:def-C67etc}.
For the 2nd chaos, 
\footnote{Below the dotted lines represent 
%$\delta_\e$ as in \eqref{e:CovarianceBounds}.
the parings, i.e. replacing the pair of noises by their covariances.} \vspace{-5pt}
\begin{equ}[e:mom-cherry]
\begin{tikzpicture}  [scale=0.4,baseline=5]		
\node at (.6,1)  [var] (a)  {};  
\node at (-.6,1) [var] (b) {}; 
\draw[very thick] (0,0) -- (a);
\draw (0,0) -- (b);
\node at (2.3+.6,1)  [var] (c)  {};  
\node at (2.3-.6,1) [var] (d) {}; 
\draw[very thick] (2.3,0) -- (d);
\draw (2.3,0) -- (c);

\draw[bend left =40, delta]  (a) to (d) ;
\draw[bend left =40, delta]  (b) to (c) ;
\end{tikzpicture}
\qquad
\big(\CK^{(\eps; 2)}\tau\big)\big(z_1, z_2\big) \lesssim \|z_1-z_2\|_{\e}^{-2-\kappa}
\end{equ}
%one has
%$\mbox{L.S.} \lesssim \lambda^{-2} $ and
%$\mbox{S.S.} \lesssim \e^{-2} \le \lambda^{-2}$.
% where for S.S. we brutally bound $ \|z\|_{\e}^{-2}$ by $\e^{-2}$.
as required by \eqref{e:CovarianceBounds}.
For $\tau=\<I'Xi_notriangle>$, one has essentially the same bound \eqref{e:mom-cherry}.

Consider $\tau = \<IXi^3_notriangle>$,
namely %$(\CS_a \Psi_j)(\CS_a \Psi_j)\Psi_i$ and 
$(\CS_{\be} \Psi_j)(\CS_{\be'} \Psi_j)\Psi_i$.
The 1st chaos vanishes by our choice of renormalisation
%$\bar C^\e$ and 
$\bar C^\e_{\be\be'}$.
For the 3rd chaos, \vspace{-5pt}
\begin{equ}[e:cubic-pairing]
\begin{tikzpicture}  [scale=0.45,baseline=5]		
\node at (.7, .8)  [var] (a)  {};  
\node at (-.7,.8) [var] (b) {}; 
\node at (0,1) [var] (x) {}; 
\draw (0,0) -- (a);
\draw (0,0) -- (b);
\draw (0,0) -- (x);
\node at (2.3+.7,.8)  [var] (c)  {};  
\node at (2.3-.7,.8) [var] (d) {}; 
\node at (2.3,1)  [var] (y)  {};  
\draw (2.3,0) -- (d);
\draw (2.3,0) -- (c);
\draw (2.3,0) -- (y);

\draw[bend left =50, delta]  (a) to (d) ;
\draw[bend left =60, delta]  (b) to (c) ;
\draw[bend left =40, delta]  (x) to (y) ;
\end{tikzpicture}
\qquad
\big(\CK^{(\eps; 3)}\tau\big)\big(z_1, z_2\big) \lesssim \|z_1-z_2\|_{\e}^{-\kappa}
\end{equ}
as required by \eqref{e:CovarianceBounds}.
The bounds for $\<IXi^2>$ and  $\<IXi>$ are similar and simpler.
%Consider $\tau=\<IXi^2>$.
%The 0th chaos  vanishes by the choice of renormalisation.
%For the 2nd chaos, 
%\begin{equ}[e:mom-V]
%\begin{tikzpicture}  [scale=0.6,baseline=10]		
%\node at (.6,1)  [var] (a)  {};  
%\node at (-.6,1) [var] (b) {}; 
%\draw (0,0) -- (a);
%\draw (0,0) -- (b);
%\node at (2.6,1)  [var] (c)  {};  
%\node at (2-.6,1) [var] (d) {}; 
%\draw (2,0) -- (d);
%\draw (2,0) -- (c);
%
%\draw[bend left =40, dotted]  (a) to (d) ;
%\draw[bend left =40, dotted]  (b) to (c) ;
%\end{tikzpicture}
%\qquad
%\big(\CK^{(\eps; 2)}\tau\big)\big(z_1, z_2\big) \lesssim \|z_1-z_2\|_{\e}^{-\kappa}
%\end{equ}
%as required by \eqref{e:CovarianceBounds}.
%The case $\<IXi>$ is even simpler.

Consider $\<IXiI'[IXiI'Xi]_notriangle>$, which consists of the following trees
\begin{gather*}
{} {\bf(1)}\;\; \Psi_i \CD_i \mcb{I} (\CS_{\be} \Psi_j \CD_j \Psi_i)\;,
\quad
{\bf(2)}\;\; \Psi_i \CD_i \mcb{I} (\CS_{\be} \Psi_j \bar\CD_i \Psi_j)\;, 
\quad
{\bf(3)}\;\; \Psi_i \CD_i \mcb{I} ( \Psi_i \CD_i \Psi_i)\;,
\\
{\bf(4)}\;  \CS_{\be} \Psi_j \CD_j \mcb{I} (\CS_{\be'} \Psi_j \CD_j \Psi_i),
\;
{\bf(5)}\; \CS_{\be} \Psi_j \CD_j \mcb{I} (\CS_{\be'} \Psi_j \bar\CD_i \Psi_j),
\;
{\bf(6)}\; \CS_{\be} \Psi_j \CD_j \mcb{I} (\Psi_i\CD_i \Psi_i),
\\
 {\bf(7)}\; \CS_{\be} \Psi_j \CD_i^+ \CS_{\bw} \mcb{I} (\CS_{\be'} \Psi_i \CD_i \Psi_j),
\;
{\bf(8)}\; \CS_{\be} \Psi_j \CD_i^+\CS_{\bw} \mcb{I} (\CS_{\be'} \Psi_i \bar\CD_j \Psi_i),
\;
 {\bf(9)}\; \CS_{\be} \Psi_j \CD_i^+ \CS_{\bw} \mcb{I} (\Psi_j \CD_j \Psi_j).
\end{gather*}
Here $\be,\be'\in \CE_\times$ and $\bw\in\{\Northwest,\Southwest\}$.
For the 3rd chaos, as well as one term of the 1st chaos of the form \tikz[scale=0.25,baseline=-5]{\node at (.6, 1)  [var] (a)  {};  
\node at (-.6,1) [var] (b) {}; 
\node at (-1.2,0) [var] (x) {}; 
\draw[very thick] (0,0) -- (a);
\draw (0,0) -- (b);
\draw[very thick] (0,0) -- (-0.6,-1);
\draw (x) -- (-0.6,-1); \draw[bend right =50, delta,thick]  (b) to (x);},
%obtained by contracting 
%the $\Psi$ outside $\mcb{I}$ and the $\Psi$ (without derivative) inside $\mcb{I}$,
%their second moments are represented by the following graphs;
proceeding as above
we again get a bound by $\|z_1-z_2\|_{\e}^{-\kappa}$. 
%\[
%\begin{tikzpicture}  [scale=0.35,baseline=10]		
%\node at (.6, 1)  [var] (a)  {};  
%\node at (-.6,1) [var] (b) {}; 
%\node at (1.2,0) [var] (x) {}; 
%\draw[very thick] (0,0) -- (a);
%\draw (0,0) -- (b);
%\draw[very thick] (0,0) -- (0.6,-1);
%\draw (x) -- (0.6,-1);
%
%\node at (4.6,1)  [var] (c)  {};  
%\node at (4-.6,1) [var] (d) {}; 
%\node at (4-1.2,0) [var] (y) {}; 
%\draw[very thick] (4,0) -- (d);
%\draw (4,0) -- (c);
%\draw[very thick] (4,0) -- (4-.6,-1);
%\draw (y) -- (4-0.6,-1);
%
%\draw[bend left =40, dotted]  (b) to (c) ;
%\draw[bend left =40, dotted]  (x) to (y) ;
%\draw[bend left =40, dotted]  (a) to (d) ;
%\end{tikzpicture}
%\qquad\qquad
%\begin{tikzpicture}  [scale=0.35,baseline=10]		
%\node at (.6, 1)  [var] (a)  {};  
%\node at (-.6,1) [var] (b) {}; 
%\node at (-1.2,0) [var] (x) {}; 
%\draw[very thick] (0,0) -- (a);
%\draw (0,0) -- (b);
%\draw[very thick] (0,0) -- (-0.6,-1);
%\draw (x) -- (-0.6,-1);
%
%\node at (3.6,1)  [var] (c)  {};  
%\node at (3-.6,1) [var] (d) {}; 
%\node at (3+1.2,0) [var] (y) {}; 
%\draw[very thick] (3,0) -- (d);
%\draw (3,0) -- (c);
%\draw[very thick] (3,0) -- (3.6,-1);
%\draw (y) -- (3+0.6,-1);
%
%\draw[bend right =50, dotted]  (b) to (x) ;
%\draw[bend left =50, dotted]  (c) to (y) ;
%\draw[bend left =40, dotted]  (a) to (d) ;
%\end{tikzpicture}
%\]
For each of these trees there is another term in the 1st chaos component of $\hat\Pi_{z_*}^\e\tau (z)$, denoted by $I_1^\e(\bar\CW^{(\eps; 1)}\tau)(z)$,
% obtained by contracting 
%the $\Psi$ outside $\mcb{I}$ with the derivative of $\Psi$ inside $\mcb{I}$,
of the form
\tikz[scale=0.25,baseline=-5]{
\node at (.6, 1)  [var] (a)  {};  
\node at (-.6,1) [var] (b) {}; 
\node at (-1.2,0) [var] (x) {}; 
\draw (0,0) -- (a);
\draw[very thick] (0,0) -- (b);
\draw[very thick] (0,0) -- (-0.6,-1);
\draw (x) -- (-0.6,-1); \draw[bend right =50, delta,thick]  (b) to (x);},
 which we focus on now.\footnote{The term in the 1st chaos
obtained by contracting the two noises on the top is zero, since the kernel $\d K$
annihilates constants. We will often use this fact below, together with the fact that the positively renormalised kernel $K(z-w)-K(-w)$ also kills constants, without explicitly mentioning.}

For (2)(4)(6)(8): $I_1^\e(\bar\CW^{(\eps; 1)}\tau)=0$ by  independence
of $\Psi_i$ and $\Psi_j$.

For (1)(5)(9): 
$I_1^\e(\bar\CW^{(\eps; 1)}\tau)(z)$ for $z\in \R\times \obonds_i$ is given by 
%\begin{equ}[e:J-convolve-Psi]
$
(J *_{(j)} \Psi_j) (z) = \int_{\R\times \obonds_j} J(z-w) \Psi_j (w) \mrd w
$
where for $z\in \R\times \obonds_i$ and $w\in \R\times \obonds_j$
\begin{equ}[e:def-J1]
J(z-w) \eqdef 
\begin{cases}
  \E [\Psi_i(z) \partial_i K (z+\be-w) \partial_j \Psi_i(w-\be)] \;,\\
  \E [\Psi_j(z+\be) \partial_j K (z+\be'-w) \bar\partial_i \Psi_j(w-\be')]\;,\\
  \E [\Psi_j(z+\be) \partial_i^+ K (z^{\bw}-w) \partial_j \Psi_j(w)] 
\end{cases}
\end{equ}
respectively for the three trees.
Now one has
\begin{equ}[e:RJ-convolve-Psi]
%I_1^\e (\bar\CW^{(\eps; 1)}\tau)(z)  = 
(J *_{(j)} \Psi_j) (z) - \mathrm{Int}[J] \Psi_j (z)
= (\mathscr{R} J *_{(j)} \Psi_j) (z)\;.
\end{equ}
Here we write  $\mathrm{Int}[J]$ for $\int_{\R\times \obonds_j} J(z-w) \mrd w$
which is independent of $z\in \R\times \obonds_i$.
Since 
$J$ has degree $-4$,
by the bound for renormalised kernels in Lemma~\ref{lem:OpDSingKer} one  can
bound the second moment of $\mathscr{R} J *_{(j)} \Psi_j$ by  \vspace{-5pt}
\begin{equ}[e:mom-2424]
\begin{tikzpicture}  [scale=0.32,baseline=0]		
\node at (.6, 1)  [var] (a)  {};  
\draw (0,0) -- (a);
\draw[kernels] (0,0) -- (0.6,-1);

\node at (3-.6, 1)  [var] (b)  {};  
\draw (3,0) -- (b);
\draw[kernels] (3,0) -- (3-.6,-1);

\draw[bend right =50, delta]  (b) to (a) ;
\end{tikzpicture}
\qquad
\qquad
\big(\CK^{(\eps; 1)}\tau\big)\big(z_1, z_2\big) \lesssim \|z_1-z_2\|_{\e}^{-\kappa}
\end{equ}
satisfying  \eqref{e:CovarianceBounds}.
Here $\tikz[baseline=-3]\draw[kernels] (0,0) --(.8,0);$ stands for $\mathscr{R}  J$.
Now it suffices to show that $\mathrm{Int}[J]$ is finite (recalling $i\neq j$).
Indeed, in the first case of \eqref{e:def-J1},
$\mathrm{Int}[J]$ can be written as
$
\int_{\R\times \obonds_i} \E [\Psi_i(z) \partial_i K (z-w) \partial_j \Psi_i(w)] \mrd w
=0$
by parity, namely,
$\partial_i K(w)$ is odd and  $\E[\Psi_i(z) \partial_j \Psi_i(z-w)]$ is even (in $w$)
in the $i$th spatial coordinate, noting that $\partial_i$ is a symmetrised derivative.
In the second case, we write $\mathrm{Int}[J]$ as
\begin{equ}
\frac12\sum\nolimits_{\sigma\in \{0,1\}}\int_{\R\times \obonds_i} 
\E [(\Psi_j(z+\be)+ (-1)^{\sigma}\Psi_j(z-\be)) \partial_j K (z-w) \bar\partial_i \Psi_j(w)]
%+\E [(\Psi_j(z+\be)-\Psi_j(z-\be)) \partial_j K (z-w) \bar\partial_i \Psi_j(w)]
 \mrd w\;.
\end{equ}
The $\sigma=0$ term vanishes by parity; the $\sigma=1$ term
 $\lesssim\int \e \|w\|_\e^{-5} \mrd w$
which is finite uniformly in $\e$.
The third case of \eqref{e:def-J1} follows in the same way.

For (3)(7): 
by \eqref{e:M-on-YM},  %\eqref{e:def-C-main},
$\hat\Pi_{z_*}^\e\tau (z)$ for $z\in \R\times \obonds_i$ is equal to
\begin{equs}
{} &\Psi_i (z) \,(\partial_i K*_{(i)}( \Psi_i \partial_i \Psi_i))(z)
- \E [\Psi_i(z) (\partial_i K*_{(i)} \partial_i \Psi_i)(z)]
\,\Psi_i (z)\;,
\\
& \Psi_j (z+\be)\, (\partial_i^+ K*_{(j)}( \Psi_i(\cdot+\be') \partial_i \Psi_j))(z^\bw)
\\
&\qquad\qquad\qquad \qquad
 - \E [ \Psi_j(z+\be) (\partial_i^+ K *_{(j)} \partial_i \Psi_j) (z^\bw) ]
\, \Psi_i (z)
\end{equs}
respectively, recalling $\hat{C}^{\e}_5$, $\hat C^\e_{3,\be,\bw}$ in Section \ref{sec:renorm_constants}.
So with
\[
J(z-w) \eqdef 
\begin{cases}
  \E [\Psi_i(z) \partial_i K (z-w) \partial_i \Psi_i(w)] \;,\\
  \E [\Psi_j(z+\be) \partial_i^+ K (z^\bw+\be'-w) \partial_i \Psi_j(w-\be')]
\end{cases}
\]
where $z,w\in \R\times \obonds_i$,
 one has 
 \begin{equ}[e:37RJ]
 I_1^\e (\bar\CW^{(\eps; 1)}\tau)(z) 
= (J *_{(i)} \Psi_i) (z) - \mathrm{Int}[J] \Psi_i (z)
 = (\mathscr{R} J *_{(i)} \Psi_i) (z)
 \end{equ}
  and one again has the bound \eqref{e:mom-2424}.

 Consider $\<I[IXiI'Xi]I'Xi_notriangle>$, which consists of the trees
\begin{gather*}
{}{\bf(1)}\;\;\mcb{I} (\CS_{\be}\Psi_j\CD_j \Psi_i) \CD_i \Psi_i\;,
\quad
{\bf(2)}\;\;\mcb{I} (\CS_{\be}\Psi_j\bar\CD_i \Psi_j) \CD_i \Psi_i\;,
\quad
{\bf(3)}\;\; \mcb{I} (\Psi_i\CD_i \Psi_i) \CD_i \Psi_i\;,
\\
{\bf(4)} \;\CS_{\be} \mcb{I} (\CS_{\be'}\Psi_i\CD_i \Psi_j) \CD_j \Psi_i,
\;\;
{\bf(5)} \; \CS_{\be} \mcb{I} (\CS_{\be'}\Psi_i \bar\CD_j \Psi_i) \CD_j \Psi_i,
\;\;
{\bf(6)} \;\CS_{\be} \mcb{I} (\Psi_j\CD_j \Psi_j) \CD_j \Psi_i,
\\
{\bf(7)} \;\CS_{\be} \mcb{I} (\CS_{\be'}\Psi_i\CD_i \Psi_j) \bar\CD_i \Psi_j,
\;\;
{\bf(8)}\;\CS_{\be} \mcb{I} (\CS_{\be'}\Psi_i \bar\CD_j \Psi_i) \bar\CD_i \Psi_j,
\;\;
{\bf(9)}\;\CS_{\be} \mcb{I} (\Psi_j\CD_j \Psi_j) \bar\CD_i \Psi_j.
\end{gather*}
For the 3rd chaos, we have \vspace{-5pt}
\begin{equ}[e:K3-01]
\begin{tikzpicture}  [scale=0.33,baseline=0]		
\node at (.6, 1)  [var] (a)  {};  
\node at (-.6,1) [var] (b) {}; 
\node at (1.2,0) [var] (x) {}; 
\draw[very thick] (0,0) -- (a);
\draw (0,0) -- (b);
\draw[kernel1] (0,0) -- (0.6,-1);
\draw[very thick] (x) -- (0.6,-1);

\node at (4.6,1)  [var] (c)  {};  
\node at (4-.6,1) [var] (d) {}; 
\node at (4-1.2,0) [var] (y) {}; 
\draw[very thick] (4,0) -- (d);
\draw (4,0) -- (c);
\draw[kernel1] (4,0) -- (4-.6,-1);
\draw[very thick] (y) -- (4-0.6,-1);

\draw[bend left =40, delta]  (b) to (c) ;
\draw[bend left =40, delta]  (x) to (y) ;
\draw[bend left =40, delta]  (a) to (d) ;
\end{tikzpicture}
\quad
\big(\CK^{(\eps; 3)}\tau\big)\big(z_1, z_2\big) \lesssim 
\|z_1-z_2\|_{\e}^{-2} (\|z_1\|_{\e} +\|z_2\|_{\e})^{2-2\kappa}
\end{equ}
%so that $\mbox{L.S.}\lesssim \lambda^{-\kappa}$
%and  $\mbox{S.S.}\lesssim \e^{-\kappa} + \e^{-1-\kappa}\lambda \lesssim \lambda^{-\kappa}$.
as required by \eqref{e:CovarianceBounds} with $\zeta=2-\kappa$ therein.
Here, the factor $\|z_1-z_2\|_{\e}^{-2} $
arises from pairing the two noises in the bottom. For the  kernel 
$\tikz[baseline=-3]\draw[kernel1] (0,0) --(.8,0);$ we bound
\begin{equ}[e:Diff-K]
|K(z-w)-K(-w)| \lesssim \|z\|_\e^\alpha (\|z-w\|_\e^{-2-\alpha} + \|w\|_\e^{-2-\alpha})\;,
\quad
\forall \alpha\in [0,1]
\end{equ}
which can be proved in the same way as 
\cite[Lem.~10.18]{Hairer14}, see also \cite[Lem.~7.4]{HM18}.
The above bound \eqref{e:K3-01} is then obtained by applying 
\eqref{e:Diff-K} to both $\tikz[baseline=-3]\draw[kernel1] (0,0) --(.8,0);$ %of the  ``positively renormalised'' kernels 
with $\alpha=1-\kappa$:
one has $ \|z_1\|_\e^{1-\kappa}  \|z_2\|_\e^{1-\kappa}\le 
(\|z_1\|_{\e} +\|z_2\|_{\e})^{2-2\kappa}$,
and the integration over the other variables is bounded due to our choice of $\alpha$.

Consider the term in the 1st chaos not requiring renormalisation. Graphically:\vspace{-5pt}
\begin{equ}[e:K3-02]
\begin{tikzpicture}  [scale=0.33,baseline=0]		
\node at (.6, 1)  [var] (a)  {};  
\node at (-.6,1) [var] (b) {}; 
\node at (-1.2,0) [var] (x) {}; 
\draw[very thick] (0,0) -- (a);
\draw (0,0) -- (b);
\draw[kernel1] (0,0) -- (-0.6,-1);
\draw[very thick] (x) -- (-0.6,-1);

\node at (3.6,1)  [var] (c)  {};  
\node at (3-.6,1) [var] (d) {}; 
\node at (3+1.2,0) [var] (y) {}; 
\draw[very thick] (3,0) -- (d);
\draw (3,0) -- (c);
\draw[kernel1] (3,0) -- (3.6,-1);
\draw[very thick] (y) -- (3+0.6,-1);

\draw[bend right =50, delta]  (b) to (x) ;
\draw[bend left =50, delta]  (c) to (y) ;
\draw[bend left =40, delta]  (a) to (d) ;
\end{tikzpicture}
\end{equ}
We bound the two terms in the ``positively renormalised'' kernel separately
(i.e. \eqref{e:Diff-K} with $\alpha=0$), which boils down to estimating the following two terms
\begin{equs}
\int 
 \|z_1-w_1\|_\e^{-3} \,& \|w_1-w_2\|_\e^{-2} 
 \|z_2-w_2\|_\e^{-3} \mrd w_1 \mrd w_2 \;,
\\
\int 
\|w_1\|_\e^{-2}\,\|z_1-w_1\|_\e^{-1}\, & \|w_1-w_2\|_\e^{-2} 
\|w_2\|_\e^{-2}\,\|z_2-w_2\|_\e^{-1} \mrd w_1 \mrd w_2\;.
\end{equs}
The first term is easily bounded by $\|z_1-z_2\|_\e^{-\kappa}$ using Lemma~\ref{lem:OpDSingKer}.
For the second term, by Young's inequality 
$\|w_1\|_\e^{-2}\|z_1-w_1\|_\e^{-1} \lesssim \|w_1\|_\e^{-3}+\|z_1-w_1\|_\e^{-3}$.
For the integral with the term $\|w_1\|_\e^{-3}$, Lemma~\ref{lem:OpDSingKer} leads to 
a bound by $\|z_2\|_\e^{-\kappa}$.
The integral with the term $\|z_1-w_1\|_\e^{-3}$ is bounded by
$
\int 
 \|z_1-w_2\|_\e^{-1} 
\|w_2\|_\e^{-2}\,\|z_2-w_2\|_\e^{-1}  \mrd w_2
$.
By Young's inequality as above with the subscript $1$ replaced by $2$, it is bounded by
$\|z_1-z_2\|_\e^{-\kappa}+ \|z_2\|_\e^{-\kappa}$.
Summarising, we obtain a bound
\begin{equ}[e:kappa3]
\|z_1-z_2\|_\e^{-\kappa}+\|z_1\|_\e^{-\kappa}+\|z_2\|_\e^{-\kappa}\;.
\end{equ}
The first term  has the form of  \eqref{e:CovarianceBounds}.
The  other two terms integrated against test functions are bounded by $\lambda^{-\kappa}$,
where in the regime $\lambda<\e$ one uses $\|z_i\|_\e^{-\kappa} \leq \e^{-\kappa}$.
%Regarding the other two terms, 
%one has  
%\begin{equs}[e:kappa3LSSS]
%\mbox{L.S.}
%&\lesssim 
%\int \phi^\lambda_{z_*}(z_1)\phi^\lambda_{z_*}(z_2) \|z_2\|_\e^{-\kappa} \mrd z_1\mrd z_2 \lesssim \lambda^{-\kappa}\;,
%\\
%\mbox{S.S.}
%&\lesssim 
%\int_{\R^2}  \|z_1\|_\e^{-\kappa}
%(\CS_{2,t_*}^{\lambda}\phi)(t_1)(\CS_{2,t_*}^{\lambda}\phi)(t_2)
%\mrd t_1\mrd t_2
%\lesssim
%\e^{-\kappa}
%\leq
%\lambda^{-\kappa}
%\end{equs}
%where in the second bound we used $  \|z_2\|_\e^{-\kappa} \leq \e^{-\kappa}$.

We now focus on
the other term in the 1st chaos which has a potential divergence, again denoted by $I_1^\e\bar\CW^{(\eps; 1)}\tau$.

For (2)(4)(6)(8): $I_1^\e(\bar\CW^{(\eps; 1)}\tau)=0$ by  independence
of $\Psi_i$ and $\Psi_j$.

For (1)(9): we proceed similarly as for the terms (1)(5)(9) of $\<IXiI'[IXiI'Xi]_notriangle>$.  
$I_1^\e (\bar\CW^{(\eps; 1)}\tau)(z)$ for $z\in \R\times \obonds_i$  is given by $(J *_{(j)} \Psi_j) (z)$ %as in  \eqref{e:J-convolve-Psi}
where for $z\in \R\times \obonds_i$ and $w\in \R\times \obonds_j$, 
\begin{equ}[e:def-J2]
J(z-w) \eqdef 
\begin{cases}
  \E [\partial_i\Psi_i(z)  K (z+\be-w) \partial_j \Psi_i(w-\be)] \;,\\
  \E [\bar\partial_i \Psi_j(z)  K (z+\be-w) \partial_j \Psi_j(w)] 
\end{cases}
\end{equ}
respectively, plus a remainder term due to the %``positively renormalised'' 
$\tikz[baseline=-3]\draw[kernel1] (0,0) --(.8,0);$ kernels. 
 We again have \eqref{e:RJ-convolve-Psi} 
and the bound \eqref{e:mom-2424} for $\mathscr{R} J *_{(j)} \Psi_j$.
One can then show that 
$\mathrm{Int}[J]$ is bounded uniformly in $\e$ by similar arguments as for the terms (1)(5)(9) of $\<IXiI'[IXiI'Xi]_notriangle>$.  
Now we only need to consider the remainder term,
which boils down to estimating 
\begin{equ}[e:w-z-w-w-w]
\int 
\|w_1\|_\e^{-2}\,\|z_1-w_1\|_\e^{-2}\,  \|w_1-w_2\|_\e^{-\kappa} \,
\|w_2\|_\e^{-2}\,\|z_2-w_2\|_\e^{-2} \mrd w_1 \mrd w_2\;.
\end{equ}
%Since $0=z_*,w_1,w_2$ form a triangle, $\|w_1\|_\e^{-2} \|w_1-w_2\|_\e^{-\kappa} \|w_2\|_\e^{-2}$ is bounded by
%\[
%\|w_1\|_\e^{-2-\frac{\kappa}2}\|w_2\|_\e^{-2-\frac{\kappa}2}
%+\|w_1\|_\e^{-2} \|w_1-w_2\|_\e^{-2-\kappa} 
%+\|w_2\|_\e^{-2} \|w_1-w_2\|_\e^{-2-\kappa}\;.
%\]
By Young's inequality 
$\|w_1\|_\e^{-2} \|w_1-w_2\|_\e^{-\kappa}
 \lesssim \|w_1\|_\e^{-2-\kappa}+\|w_1-w_2\|_\e^{-2-\kappa}$.
If the first term is chosen, the integral is bounded by $\|z_1\|_\e^{-\kappa}+\|z_2\|_\e^{-\kappa}$.
If the other term is chosen, we again apply Lemma~\ref{lem:OpDSingKer}  and Young's inequality to obtain
\eqref{e:kappa3}. % and thus \eqref{e:kappa3LSSS}.
 
For (3)(5)(7): we proceed similarly as for the terms (3)(7) of $\<IXiI'[IXiI'Xi]_notriangle>$.
By \eqref{e:M-on-YM} and the choices of $\hat C^\e_{4}$, $\hat C^\e_{1,\be}$,
$\hat C^\e_{2,\be}$  in Section \ref{sec:renorm_constants},
% \eqref{e:def-C-main}, 
 $\hat\Pi_{z_*}^\e\tau (z)$ for $z\in \R\times \obonds_i$ is equal to
\begin{equs}
\partial_i \Psi_i (z) (K*_{(i)}( \Psi_i \partial_i \Psi_i))(z)
&- \E [\partial_i\Psi_i(z) (K*_{(i)} \partial_i \Psi_i)(z)]
\Psi_i (z)\;,
\\
\partial_j \Psi_i (z) ( K*_{(j)}( \Psi_i(\cdot+\be') \bar\partial_j \Psi_i))(z+\be)
&- \E [\partial_j \Psi_i (z) ( K *_{(j)} \bar\partial_j \Psi_i) (z+\be) ]
\Psi_i (z)\;,
\\
\bar\partial_i \Psi_j (z) ( K*_{(j)}( \Psi_i(\cdot+\be') \partial_i \Psi_j))(z+\be)
&- \E [\bar\partial_i \Psi_j(z) ( K *_{(j)} \partial_i \Psi_j) (z+\be) ]
\Psi_i (z)
\end{equs}
respectively, plus remainder terms due to the ``positively renormalised'' kernels
which are bounded in exactly the same way as above.
So with
\[
J(z-w) \eqdef 
\begin{cases}
  \E [\partial_i  \Psi_i(z) K (z-w) \partial_i \Psi_i(w)] \;,\\
   \E [\partial_j \Psi_i (z) K (z+\be+\be'-w) \bar\partial_j \Psi_i (w-\be') ]\;,\\
  \E [\bar\partial_i \Psi_j(z)  K (z+\be+\be'-w) \partial_i \Psi_j(w-\be')]
\end{cases}
\]
where $z,w\in \R\times \obonds_i$,
the above three expressions can all be written as in \eqref{e:37RJ},
and one again has the bound \eqref{e:mom-2424}.

%Consider $\<I'[IXiI'Xi]_notriangle>$. It only has 2nd chaos and
%\begin{equ}
%\begin{tikzpicture}  [scale=0.4,baseline=0]		
%\node at (.6, 1)  [var] (a)  {};  
%\node at (-.6,1) [var] (b) {}; 
%\draw[very thick] (0,0) -- (a);
%\draw (0,0) -- (b);
%\draw[very thick] (0,0) -- (0.6,-1);
%
%\node at (3.6,1)  [var] (c)  {};  
%\node at (3-.6,1) [var] (d) {}; 
%\draw[very thick] (3,0) -- (d);
%\draw (3,0) -- (c);
%\draw[very thick] (3,0) -- (3-.6,-1);
%
%\draw[bend left =40, dotted]  (b) to (c) ;
%\draw[bend left =40, dotted]  (a) to (d) ;
%\end{tikzpicture}
%\qquad
%\big(\CK^{(\eps; 2)}\tau\big)\big(z_1, z_2\big) \lesssim \|z_1-z_2\|_{\e}^{-\kappa}
%\end{equ}
%%one has
%%$\mbox{L.S.} \lesssim \lambda^{-2} $ and
%%$\mbox{S.S.} \lesssim \e^{-2} \le \lambda^{-2}$.
%% where for S.S. we brutally bound $ \|z\|_{\e}^{-2}$ by $\e^{-2}$.
%as required by \eqref{e:CovarianceBounds}.
%The case $\<I'[I'Xi]_notriangle>$ satisfies the same bound.
The bounds for $\<I'[IXiI'Xi]_notriangle>$ and  $\<I'[I'Xi]_notriangle>$
follow from the bounds for $\<IXiI'Xi_notriangle>$ and $\<I'Xi_notriangle>$.
%
%Consider $ \<IXiI'[I'Xi]_notriangle>$. This is controlled in the same way as $\<IXiI'[IXiI'Xi]_notriangle>$.
%In particular, the two trees that need renormalisation are
%$ \Psi_i \CD_i \mcb{I} ( \bone^{(i)} \CD_i \Psi_i)$ 
%and 
%$\CS_a \Psi_j \bar\CD_i \mcb{I} (\CS_a \bone^{(i)} \CD_i \Psi_j)$,
%and for each of them
%the $0$th chaos exactly vanishes  by the choice of renormalisation.
%This is also the case for the other trees of this form
% either by independence
%of $\Psi_i$ and $\Psi_j$, or by exact spatial parity due to 
%symmetrised derivatives.
%
%The case $\<I[I'Xi]I'Xi_notriangle>$ is also treated in the same way as $\<I[IXiI'Xi]I'Xi_notriangle>$.
%
The trees  $\<IXiI'[I'Xi]_notriangle>$ and $\<I[I'Xi]I'Xi_notriangle>$
are treated in the same way as 
$\<IXiI'[IXiI'Xi]_notriangle>$ and
$\<I[IXiI'Xi]I'Xi_notriangle>$, and we omit the details.
\end{proof}

\begin{lemma}
	\label{lem:mom-xi}
%For every $\tau\in \CT^{\xi}_-$,  
For every $\tau$ in \eqref{e:trees-deg-1-rem},
the bounds \eqref{e:mom-bounds} hold.
\end{lemma}

\begin{proof}
For $\tau=\<Xi1>=\bar\Xi$ such that $\hat\Pi^\e\tau = \e^{1-\kappa} \xi^\e$,
similarly as $\Xi$, \eqref{e:CovarianceBounds} applies with $(\zeta,\alpha,n_\delta,n_\e)=(0,0,1,2-2\kappa)$.
%the same calculation as in \eqref{e:xiS22} yields the desired bounds:
%\begin{equ}
%\mbox{L.S.} \lesssim (\e^{1-\kappa})^2 \lambda^{-4} \le \lambda^{-2-2\kappa}\;,
%\qquad
%\mbox{S.S.} \lesssim (\e^{1-\kappa})^2 \e^{-2} \lambda^{-2} \lesssim \lambda^{-2-2\kappa}\;.
%\end{equ}

Let $\tau= \<PsiXi1>$.
The $0$th chaos vanishes by our choice of renormalisation. For the 2nd chaos,
$\bigl(  \CK^{(\eps; 2)}\tau \bigr)(z_1,z_2)  = \e^{2-2\kappa} \delta_\e(z_1-z_2)  K^{*2}(z_1-z_2)$,
so \eqref{e:CovarianceBounds} applies with $(\zeta,\alpha,n_\delta,n_\e)=(0,-\kappa,1,2-2\kappa)$.
%
%for which
%\begin{equs}[e:PsiXi1-mom]
%%\E [(\hat\Pi^\e \tau) (\phi^\lambda)^2]
%\mbox{L.S.}
%& \le  (\e^{1-\kappa})^2  \int \delta_\e(z_1-z_2)  \phi_{z_*}^\lambda(z_1)   \phi_{z_*}^\lambda(z_2)
%  K^{*2}(0) \mrd z_1 \mrd z_2
%\\
%& \lesssim \e^{2-2\kappa} \int \phi_{z_*}^\lambda(z)^2 \e^{-\kappa}\mrd z
% \lesssim  \e^{2-3\kappa} \lambda^{-4} 
% \le \lambda^{-2-3\kappa}\;.
%\end{equs}
%%For $\lambda<\e$, 
%Here we recall $\delta_\e$ in \eqref{e:def-delta-eps}. By similar calculation,
%\begin{equs}
%%\E \Big[\Big(\int_\R \hat\Pi^\e \tau (t,x) \CS^\lambda_2 \phi(t) \mrd t \Big)^2 \Big]
%\mbox{S.S.}
%\le  \e^{2-2\kappa}\int_\R \e^{-2} \CS^\lambda_{2,t_*} \phi(t)^2 \e^{-\kappa}\mrd t
%\lesssim \e^{-3\kappa} \lambda^{-2} 
%\le \lambda^{-2-3\kappa}\;.
%\end{equs}
%Here $ \e^{-2}$ inside the integral arises from the factor $\e^{-1}$ in $\e^{-1}dW$.
%In general, the bounds at small scales and large scales  
%can be obtained in essentially the same way:
%as in the above calculation one simply needs to replace a factor $\lambda^{-2}$
%in $ \e^{2-3\kappa} \lambda^{-4} $ for the large scale bound
%by a factor $\e^{-2}$ which is simply the right small scale bound.

Let $\tau = \<R2-1new> = \big(\CD_j^+ \CS_{\southeast} \mcb{I}_j \tilde\Xi_j \big) 
\big(\CD_j^+ \CS_{\southwest} \mcb{I}_j \tilde\Xi_j \big)$. For the 2nd chaos,
\begin{equs}
\bigl(  \CK^{(\eps; 2)}\tau \bigr)(z_1,z_2) 
& = \e^{2-2\kappa} (\partial_j^+ K)^{*_\e 2} (z_1-z_2) (\partial_j^+ K)^{*_\e 2} (z_1-z_2)
\\
& \lesssim \e^{2-2\kappa} \|z_1-z_2\|_\e^{-4-\kappa}
 \lesssim  \|z_1-z_2\|_\e^{-2-3\kappa} 
\end{equs}
so we  have \eqref{e:CovarianceBounds}.
%and we get the bound $\lambda^{-2-3\kappa}$ on both large and small scales.
The $0$th chaos is cancelled by our choice of 
$C^\e [\<R2-1new>]$
in \eqref{eq:C_square_cherries}.
%For the $0$th chaos, one has, for $z = (t,x)$ with $x \in \obonds_i$,
%\[
%\bigl(\CW^{(\eps; 0)}\tau \bigr)(z) = 
%\int_{\R} \e^2 \sum\nolimits_{y\in \obonds_j}
%\partial_j^+ K_{t-s}(x^\southeast - y) \,
%\partial_j^+ K_{t-s}(x^\southwest - y)\mrd s \;.
%\]
%By naive dimension counting, this appears divergent. However, it is actually bounded.
%This is because replacing $K$ above by $P$ only causes  an error uniformly bounded in $\e$.
%Furthermore,
%since $x^\southeast \neq x^\southwest$ and $P_0$
%vanishes away from origin, by Lemma~\ref{lem:identity}, 
%the above quantity with 
%$K$ replaced by $P$ vanishes.
\end{proof}

\begin{lemma}
	\label{lem:mom-12}
For every $\tau$ in group \hyperlink{(1-3)}{(1-3)},  the bounds \eqref{e:mom-bounds} hold.
\end{lemma}

\begin{proof}
Let 
$\tau= \<cherryY>$, consisting of three trees,
%We will see that they do not contribute to any mass renormalisation.
with $\big(\CK^{(\eps; 3)}\tau\big)\big(z_1, z_2\big)$ all bounded as \vspace{-5pt}
\begin{equ}[e:mom-Y333]
\begin{tikzpicture}  [scale=0.35,baseline=0]		
\node at (.6, 1)  [var] (a)  {};  
\node at (-.6,1) [var] (b) {}; 
\node at (.6,-1) [var1] (x) {}; 
\draw[very thick] (0,0) -- (a);
\draw (0,0) -- (b);
\draw[kernel1] (0,0) -- (x);

\node at (3.6,1)  [var] (c)  {};  
\node at (3-.6,1) [var] (d) {}; 
\node at (3-.6,-1) [var1] (y) {}; 
\draw[very thick] (3,0) -- (d);
\draw (3,0) -- (c);
\draw[kernel1] (3,0) -- (y);

\draw[bend left =40, delta]  (b) to (c) ;
\draw[delta]  (x) to (y) ;
\draw[bend left =40, delta]  (a) to (d) ;
\end{tikzpicture}
\qquad
\|z_1-z_2\|_{\e}^{-\kappa} (\|z_1\|_{\e}^2 +\|z_2\|_{\e}^2)
\delta_\e(z_1-z_2)  (\e^{1-\kappa})^2 
\end{equ}
so \eqref{e:CovarianceBounds} applies with $(\zeta,\alpha,n_\delta,n_\e)=(2,-\kappa,1,2-2\kappa)$.
%Therefore, since $\|z_1\|_{\e}^2 +\|z_2\|_{\e}^2\lesssim \lambda^2$
%on the supports of the rescaled test functions,
%\begin{equs}[e:mom-cheY3]
%\mbox{L.S.}
%&\lesssim
%\e^{-\kappa} \lambda^2 (\e^{1-\kappa})^2  
% \int \phi_{z_*}^\lambda(z)^2 \mrd z
% \lesssim
%\e^{-\kappa} \lambda^2 \e^{2-2\kappa}  \lambda^{-4}
%\le \lambda^{-3\kappa}\;,
%\\
%\mbox{S.S.}
%&\lesssim
%\e^{-\kappa} \lambda^2 (\e^{1-\kappa})^2  
%\int_\R \e^{-2} \CS^\lambda_{2,t_*} \phi(t)^2 \mrd t 
% \lesssim
% \e^{-\kappa} \lambda^2 \e^{2-2\kappa}\e^{-2}  \lambda^{-2}
%\le \lambda^{-3\kappa}\;.
%\end{equs}
For the 1st chaos, 
one term is treated as follows. Note that the contraction of two noises within each tree
yields a product of kernels with total degree $-4$, but we can use the powers of $\e$ to improve it, e.g.
$\e^{1-\kappa} \|z\|_\e^{-1} \le \|z\|_\e^{-\kappa}$. Then $\big(\CK^{(\eps; 1)}\tau\big)\big(z_1, z_2\big)$ is bounded by
\begin{equ}[e:mom-Y11]
\begin{tikzpicture}  [scale=0.35,baseline=0]		
\node at (.6, 1)  [var] (a)  {};  
\node at (-.6,1) [var] (b) {}; 
\node at (0,-1.2) [var1] (x) {}; 
\draw[very thick] (0,0) -- (a);
\draw (0,0) -- (b);
\draw[kernel1] (0,0) -- (x);

\node at (3.6,1)  [var] (c)  {};  
\node at (3-.6,1) [var] (d) {}; 
\node at (3,-1.2) [var1] (y) {}; 
\draw[very thick] (3,0) -- (d);
\draw (3,0) -- (c);
\draw[kernel1] (3,0) -- (y);

\draw[bend right =40, delta]  (b) to (x) ;
\draw[bend left =40, delta]  (c) to (y) ;
\draw[delta]  (a) to (d) ;
\end{tikzpicture}
\qquad
\|z_1-z_2\|_{\e}^{-2-\kappa} 
(\|z_1\|_{\e}^2 +\|z_2\|_{\e}^2+\|z_1-z_2\|_{\e}^2)
\end{equ}
so we get  \eqref{e:CovarianceBounds}.
There is then another term in the 1st chaos, with 2nd moment
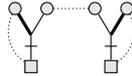
\[
\begin{tikzpicture}  [scale=0.35,baseline=0]		
\node at (.6, 1)  [var] (a)  {};  
\node at (-.6,1) [var] (b) {}; 
\node at (0,-1.2) [var1] (x) {}; 
\draw (0,0) -- (a);
\draw[very thick] (0,0) -- (b);
\draw[kernel1] (0,0) -- (x);

\node at (3.6,1)  [var] (c)  {};  
\node at (3-.6,1) [var] (d) {}; 
\node at (3,-1.2) [var1] (y) {}; 
\draw (3,0) -- (d);
\draw[very thick] (3,0) -- (c);
\draw[kernel1] (3,0) -- (y);

\draw[bend right =40, delta]  (b) to (x) ;
\draw[bend left =40, delta]  (c) to (y) ;
\draw[delta]  (a) to (d) ;
\end{tikzpicture}
%\quad
%\big(\CK^{(\eps; 1)}\tau\big)\big(z_1, z_2\big) \lesssim 
%\|z_1-z_2\|_{\e}^{-2-\kappa} 
%(\|z_1\|_{\e}^2 +\|z_2\|_{\e}^2+\|z_1-z_2\|_{\e}^2)
\]
Define $\bar K\eqdef \e^{1-\kappa} (K \partial K)$ (where $\partial$ could be replaced by  $\bar\partial$ depending on the particular tree).
One has
$|\bar K(z)| \lesssim \|z\|_\e^{-4-\kappa}$
and
 since the derivative is symmetrised, 
 $\bar K=\mathscr{R} \bar K$ and  $\bigl(\CK^{(\eps; 1)}\tau \bigr)(z_1,z_2)$ is given by  (recall $z_*=0$)
\begin{equ}
 \int 
\big(\mathscr{R} \bar K (z_1-w_1)
-\bar K (-w_1)\big)
K^{*2} (w_1- w_2) 
\big(\mathscr{R} \bar K (z_2-w_2)
-\bar K (-w_2)\big)
\mrd w_1\mrd w_2\;.
\end{equ}
We can derive the same bound \eqref{e:mom-Y11} for this case,
so one again has \eqref{e:CovarianceBounds}.
The case $\<cherryYless>$ follows in the same way.

Let 
$\tau = \<IXiI'Xibar>$ (which includes trees added in Step 3 in Section \ref{subsubsec:Step_3}).
The $0$th chaos vanishes 
thanks to our choice of renormalisation \eqref{eq:tilde_C_renorm}, \eqref{eq:square_cherry_consts}, \eqref{eq:C_square_cherries}.
%\[
%\bigl(\CK^{(\eps; 2)}\tau \bigr)(z,\bar z) = \e^{2-2\kappa} (K^\e)^{*2} (\partial K^\e)^{*2} (z-\bar z) \lesssim \e^{2-2\kappa} \|z-\bar z\|_\e^{-2-\kappa}\;.
%\]
%One gets a bound $\e^{2-2\kappa} \lambda^{-2-\kappa} \lesssim \lambda^{-3\kappa}$ for $\e<\lambda$.
The 2nd chaos is treated similarly as  $\<IXiI'Xi_notriangle>$,
and by absorbing powers of $\e$ into kernels $\e^{1-\kappa} \|z\|_\e^{-1} \le \|z\|_\e^{-\kappa}$.
The cases $\tau=\<I'Xibar>$  and $\tau=\<I'XiIXibar>$ follow similarly. 
%This is similar as the last one, except that 
%\[
%\bigl(\CK^{(\eps; 2)}\tau \bigr)(z,\bar z) 
% \lesssim
%  \e^{2-2\kappa} \|z-\bar z\|_\e^{-2-\kappa}
% +  \e^{2-2\kappa} \|z-\bar z\|_\e^{-3-\kappa} (\|z\|_\e + \|\bar z\|_\e)\;.
%\]
%So we get the same bound.   The zeroth chaos also vanishes.

Consider $\tau\in \{ \<supercherry1>, \<supercherry2>\}$.  This includes the following trees
\begin{equs}[e:supercherry1-list]
 \Psi_i \CD_i\mcb{I} (\Psi_i \bar\Xi_i)\;,
&\qquad
\mcb{I} ( \Psi_i \bar\Xi_i) \CD_i \Psi_i\;,
\\
\CS_{\be}\Psi_j  \CD_j\mcb{I} (\Psi_i  \bar\Xi_i)\;,
&\qquad
\CS_{\be}\mcb{I} (\Psi_j \bar\Xi_j) \CD_j \Psi_i\;,
\\
 \CS_{\be} \Psi_j
\CD_i^+ \CS_{\bw}
\mcb{I} (\Psi_j \bar\Xi_j)\;,
&\qquad
\CS_{\be}\mcb{I} (\Psi_j \bar\Xi_j) \bar{\CD}_i \Psi_j\;,
\end{equs}
where $\be\in\CE_\times$ and $\bw\in\{\Northwest,\Southwest\}$.
%Remark that these trees showed up in \eqref{e:quad-exp-1}, \eqref{e:quad-exp-2}, \eqref{e:quad-exp-3}.
The 3rd chaos parts of the trees in the left column are all bounded as follows:
\begin{equ}[e:superch-3rd]
\begin{tikzpicture}  [scale=0.35,baseline=0]		
\node at (.6, 1)  [var] (a)  {};  
\node at (0,0) [var1] (b) {}; 
\node at (1.2,0) [var] (x) {}; 
\draw (b) -- (a);
\draw[very thick] (b) -- (0.6,-1);
\draw (x) -- (0.6,-1);

\node at (4,0)  [var1] (c)  {};  
\node at (4-.6,1) [var] (d) {}; 
\node at (4-1.2,0) [var] (y) {}; 
\draw (c) -- (d);
\draw[very thick] (c) -- (4-.6,-1);
\draw (y) -- (4-0.6,-1);

\draw[bend left =30, delta]  (b) to (c) ;
\draw[bend left =10, delta]  (x) to (y) ;
\draw[bend left =10, delta]  (a) to (d) ;
\end{tikzpicture}
\qquad
\big(\CK^{(\eps; 3)}\tau\big)\big(z_1, z_2\big) \lesssim 
  \e^{2-3\kappa} 
\|z_1- z_2\|_\e^{-2} \lesssim \|z_1- z_2\|_\e^{-3\kappa}
\end{equ}
where we bounded  
$|K^{*_\e 2}(0)| \lesssim \e^{-\kappa}$ at the top
(it is evaluated at $0$ due to the contraction of the two $\bar\xi$'s). 
For the 3rd chaos  of the trees in the right column:
\[
\begin{tikzpicture}  [scale=0.35,baseline=0]		
\node at (.6, 1)  [var] (a)  {};  
\node at (0,0) [var1] (b) {}; 
\node at (1.2,0) [var] (x) {}; 
\draw (b) -- (a);
\draw[kernel1] (b) -- (0.6,-1);
\draw[very thick] (x) -- (0.6,-1);

\node at (4,0)  [var1] (c)  {};  
\node at (4-.6,1) [var] (d) {}; 
\node at (4-1.2,0) [var] (y) {}; 
\draw (c) -- (d);
\draw[kernel1] (c) -- (4-.6,-1);
\draw[very thick] (y) -- (4-0.6,-1);

\draw[bend left =30, delta]  (b) to (c) ;
\draw[bend left =10, delta]  (x) to (y) ;
\draw[bend left =10, delta]  (a) to (d) ;
\end{tikzpicture}
\quad
\big(\CK^{(\eps; 3)}\tau\big)\big(z_1, z_2\big) 
\lesssim 
\|z_1-z_2\|_\e^{-1-3\kappa} 
(\|z_1-z_2\|_\e
+\|z_1\|_\s+\|z_2\|_\s)
\]
These two bounds are again in the scope of \eqref{e:CovarianceBounds}.
The 1st chaos of each of these trees consists of two terms,
one of which is bounded as follows
\begin{equ}[e:mom-supercheche]
\begin{tikzpicture}  [scale=0.35,baseline=0]		
\node at (.6, 1)  [var] (a)  {};  
\node at (1.2,0) [var1] (b) {}; 
\node at (0,0) [var] (x) {}; 
\draw (b) -- (a);
\draw[very thick] (b) -- (0.6,-1);
\draw (x) -- (0.6,-1);

\node at (4-.6,1) [var] (d) {}; 
\node at (4-1.2,0)  [var1] (c)  {};  
\node at (4,0) [var] (y) {}; 
\draw (c) -- (d);
\draw[very thick] (c) -- (4-.6,-1);
\draw (y) -- (4-0.6,-1);

\draw[bend right =40, delta]  (a) to (x) ;
\draw[delta]  (b) to (c) ;
\draw[bend left =40, delta]  (d) to (y) ;
\end{tikzpicture}
\qquad
\big(\CK^{(\eps; 1)}\tau\big)\big(z_1, z_2\big) 
\lesssim 
\e^{2-2\kappa} \|z_1-z_2\|_\e^{-2-2\kappa}
 \lesssim \|z_1- z_2\|_\e^{-4\kappa}
\end{equ}
\begin{equ}
\begin{tikzpicture}  [scale=0.4,baseline=0]		
\node at (.6, 1)  [var] (a)  {};  
\node at (1.2,0) [var1] (b) {}; 
\node at (0,0) [var] (x) {}; 
\draw (b) -- (a);
\draw[kernel1] (b) -- (0.6,-1);
\draw[very thick] (x) -- (0.6,-1);

\node at (3.7-.6,1) [var] (d) {}; 
\node at (3.7-1.2,0)  [var1] (c)  {};  
\node at (3.7,0) [var] (y) {}; 
\draw (c) -- (d);
\draw[kernel1] (c) -- (3.7-.6,-1);
\draw[very thick] (y) -- (3.7-0.6,-1);

\draw[bend right =40, delta]  (a) to (x) ;
\draw[delta]  (b) to (c) ;
\draw[bend left =40, delta]  (d) to (y) ;
\end{tikzpicture}
\qquad
\big(\CK^{(\eps; 1)}\tau\big)\big(z_1, z_2\big) 
 \lesssim \|z_1- z_2\|_\e^{-4\kappa}
+ \|z_1\|_\e^{-3\kappa}+ \|z_2\|_\e^{-3\kappa}
\end{equ}
so we are again in the scope of \eqref{e:CovarianceBounds} and \eqref{e:kappa3}.
The other term in the 1st chaos is 
%obtained by contracting the two noises in the bottom:
\tikz[scale=0.29,baseline=-5]
{\node at (.6, 1)  [var] (a)  {};  
\node at (1.2,0) [var1] (b) {}; 
\node at (0,0) [var] (x) {}; 
\draw (b) -- (a);
\draw[very thick] (b) -- (0.6,-1);
\draw (x) -- (0.6,-1);
\draw[very thick, delta]  (b) to (x) ;
}
 or 
\tikz[scale=0.29,baseline=-5]
{	
\node at (.6, 1)  [var] (a)  {};  
\node at (1.2,0) [var1] (b) {}; 
\node at (0,0) [var] (x) {}; 
\draw (b) -- (a);
\draw[kernel1] (b) -- (0.6,-1);
\draw[very thick] (x) -- (0.6,-1);
\draw[very thick, delta]  (b) to (x) ;
} .
This term vanishes for trees in the second row of \eqref{e:supercherry1-list}, by independence of $\xi_i,\xi_j$.
For trees in the first row of \eqref{e:supercherry1-list}, 
this term is equal to $J*_{(i)}\Psi_i$ (up to a remainder due to the  kernel $\tikz[baseline=-3]\draw[kernel1] (0,0) --(.8,0);$ in the case of $\mcb{I} ( \Psi_i \bar\Xi_i) \CD_i \Psi_i$) with
$J(z-u)$  being
$\e^{1-\kappa}  \E [\Psi_i(z)\xi_i(u) \partial_i K(z-u)]$
and
$\e^{1-\kappa} \E [\partial_i \Psi_i(z)\xi_i(u)  K(z-u)]$
for $z,u\in \R\times \obonds_i$. Since each $J$  here is odd in the $i$th coordinate, its integral vanishes and thus $J=\mathscr{R} J $.
Since $J$ is of order $-4-\kappa$,
we therefore have the bound 
\eqref{e:mom-2424}. The ``remainder'' is bounded exactly as in \eqref{e:w-z-w-w-w} for $\<I[IXiI'Xi]I'Xi_notriangle>$.

Consider now  the trees in last row of  \eqref{e:supercherry1-list},
which are renormalised as in  \eqref{eq:square_cherry_consts} and \eqref{eq:C_square_cherries}.
For $ \CS_{\be} \Psi_j\CD_i^+ \CS_{\bw}\mcb{I} (\Psi_j \bar\Xi_j)$,
the term of interest in the 1st chaos is
\begin{equs}
\e^{1-\kappa} &\int \E[\Psi_j (z+\be) \partial_i^+ K(z^\bw-u)  \xi_j(u)] \Psi_j(u) \mrd u
\\
&-
\e^{1-\kappa} \E [\Psi_j(z+\be) \d_i^+\Psi_j(z^\bw) ] \Psi_j(z^\bw) = (\mathscr{R} J *_{(j)} \Psi_j)(z^\bw)
\end{equs}
where for $\bar z,u\in  \R\times \obonds_j$, writing $\hat z$ for the element of $\R\times \obonds_i$ such that $\hat z^{\bw}=\bar z$,
\[
J(\bar z-u) = \e^{1-\kappa} \E[\Psi_j (\hat z+\be) \partial_i^+ K(\bar z-u)  \xi_j(u)]\;.
\]
We then again obtain  the bound 
\eqref{e:mom-2424}.  The other tree in the last row of  \eqref{e:supercherry1-list} follows in the same way,
with ``positively renormalised'' kernels treated as above.

For $\tau=\<supercherry1less>$ one only needs to bound the 2nd chaos and it is similar as \eqref{e:superch-3rd}.
%$\bigl(\CK^{(\eps; 1)}\tau\bigr)(z_1,z_2)$ equals
%\[
% \e^{2-2\kappa} 
%\int
%(\partial K^\e)(K^\e)^{*2} (z_1-y_1) \,
%(\partial K^\e)(K^\e)^{*2} (z_2- y_2) \,
%\delta_\e(y_1-y_2) 
%\mrd y_1 \mrd y_2
%\]
%One then gets a large scale bound $ \e^{2-2\kappa}  \lambda^{-2-2\kappa} \lesssim \lambda^{-4\kappa}$.

Finally, for $\tau=\X^{(i)} \<PsiXi1>$, the 0th chaos vanishes by parity.
For the 2nd chaos the bound follows as in $\<PsiXi1>$ above except that $\zeta=1$.
%Regarding the 2nd chaos, writing $z_j=(t_j,x_j)$ for $j\in \{1,2\}$ one has
%\begin{equ}
%\mbox{L.S.}
%\le  (\e^{1-\kappa})^2  \int \delta_\e(z_1-z_2) \, \| x_1^{(i)}\|_\e \| x_2^{(i)}\|_\e\, \phi_{z_*}^\lambda(z_1)   \phi_{z_*}^\lambda(z_2)
%  K^{*2}(0) \mrd z_1 \mrd z_2
%\end{equ}
%which is bounded by $\lambda^{-3\kappa}$
%similarly as \eqref{e:PsiXi1-mom} for $ \<PsiXi1>$, noting that $\|x_j^{(i)}\|_\e \le \lambda$ on the support of $\phi_{z_*}^\lambda$.
%The small scale bound is trivial since 
%in \eqref{e:IkWk-2nd}
%$\CW^{(\eps; 2)}\tau$ is evaluated at $(t,x_*)$ which is zero. 
The tree $\X^{(i)} \bar\Xi_i$ is even simpler.
\end{proof}

\begin{lemma}
	\label{lem:mom-13}
For every $\tau$ in group \hyperlink{(1-2)}{(1-2)},  the bounds \eqref{e:mom-bounds} hold.
\end{lemma}

\begin{proof}
Let $\tau=\<I'XiI'[IXiI'Xi]>$. Recall (below \eqref{e:trees-deg-0-rem}) that trees of this form arise from the term $\tilde R_2$ in \eqref{e:tildeR2}, \eqref{eq:R_2_other_form},
with one of the factors $A$ there replaced by an integration of $A\partial A$. (See also \eqref{e:tildeR2-exp}.)
% where they appear on the level of modelled distributions when we expand the term $ \CQ^{\e}_i$ defined in \eqref{e:def-tildeCR2}.
These trees have different properties in terms of parity, noise indices and renormalisation;
to enumerate these trees more easily and conceptually,  we will count 
the two trees  corresponding to $\CD^+ + \CD^-$ as ``one tree'', 
the two trees corresponding to $\sum_{\bw} \CD^+ \cS_{\bw}$ as ``one tree'', 
the two trees which only differ by simultaneously swapping all indices $i\to 3-i$ as ``one tree'',
and the four trees corresponding to $\cS_{a}$ defined in \eqref{e:shift-a} (see e.g. \eqref{e:def-CHi}) as ``one tree''.
Counting this way,
there are $2\times 2\times 3 = 12$ trees of this form (two terms in $\tilde R_2$, each of them is quadratic in $A$, and the equation for $A$ has $3$ terms of the form $A\partial A$).
The 3rd chaos can be all bounded as follows: \vspace{-5pt}
\[
\begin{tikzpicture}  [scale=0.33,baseline=0]		
\node at (.6, 1)  [var] (a)  {};  
\node at (-.6,1) [var] (b) {}; 
\node at (1.2,0) [var1] (x) {}; 
\draw[very thick] (0,0) -- (a);
\draw (0,0) -- (b);
\draw[very thick] (0,0) -- (0.6,-1);
\draw[very thick] (x) -- (0.6,-1);

\node at (4.6,1)  [var] (c)  {};  
\node at (4-.6,1) [var] (d) {}; 
\node at (4-1.2,0) [var1] (y) {}; 
\draw[very thick] (4,0) -- (d);
\draw (4,0) -- (c);
\draw[very thick] (4,0) -- (4-.6,-1);
\draw[very thick] (y) -- (4-0.6,-1);

\draw[bend left =40, delta]  (b) to (c) ;
\draw[bend left =40, delta]  (x) to (y) ;
\draw[bend left =40, delta]  (a) to (d) ;
\end{tikzpicture}
\qquad
\big(\CK^{(\eps; 3)}\tau\big)\big(z_1, z_2\big) \lesssim 
\e^{2-2\kappa}\|z_1-z_2\|_{\e}^{-2-\kappa}
\lesssim
\|z_1-z_2\|_{\e}^{-3\kappa}\;.
\]
The same bound holds for the term 
\tikz[scale=0.27,baseline=-5]
{
\node at (.6, 1)  [var] (a)  {};  
\node at (-.6,1) [var] (b) {}; 
\node at (-1.2,0) [var1] (x) {}; 
\draw[very thick] (0,0) -- (a);
\draw (0,0) -- (b);
\draw[very thick] (0,0) -- (-0.6,-1);
\draw[very thick] (x) -- (-0.6,-1);
\draw[bend right =40, delta,thick]  (b) to (x) ;
} in the 1st chaos
using $\e^{1-\kappa} \|z\|_\e^{-4} \le  \|z\|_\e^{-3-\kappa}$.
%For the 1st chaos, the following term is easy to bound: using $\e^{1-\kappa} \|z\|_\e^{-4} \le  \|z\|_\e^{-3-\kappa}$  \vspace{-5pt}
%\[
%\begin{tikzpicture}  [scale=0.33,baseline=0]		
%\node at (.6, 1)  [var] (a)  {};  
%\node at (-.6,1) [var] (b) {}; 
%\node at (-1.2,0) [var1] (x) {}; 
%\draw[very thick] (0,0) -- (a);
%\draw (0,0) -- (b);
%\draw[very thick] (0,0) -- (-0.6,-1);
%\draw[very thick] (x) -- (-0.6,-1);
%
%\node at (3.6,1)  [var] (c)  {};  
%\node at (3-.6,1) [var] (d) {}; 
%\node at (3+1.2,0) [var1] (y) {}; 
%\draw[very thick] (3,0) -- (d);
%\draw (3,0) -- (c);
%\draw[very thick] (3,0) -- (3+.6,-1);
%\draw[very thick] (y) -- (3+0.6,-1);
%
%\draw[bend right =40, delta]  (b) to (x) ;
%\draw[bend left =40, delta]  (a) to (d) ;
%\draw[bend left =40, delta]  (c) to (y) ;
%\end{tikzpicture}
%\qquad
%\big(\CK^{(\eps; 1)}\tau\big)\big(z_1, z_2\big) \lesssim 
%\|z_1-z_2\|_{\e}^{-3\kappa}\;.
%\]
Below we only need to focus on 
the term 
\tikz[scale=0.27,baseline=-5]
{
\node at (.6, 1)  [var] (a)  {};  
\node at (-.6,1) [var] (b) {}; 
\node at (-1.2,0) [var1] (x) {}; 
\draw (0,0) -- (a);
\draw[very thick] (0,0) -- (b);
\draw[very thick] (0,0) -- (-0.6,-1);
\draw[very thick] (x) -- (-0.6,-1);
\draw[bend right =40, delta,thick]  (b) to (x) ;
}
 in the 1st chaos.
% obtained by contracting the two noises in the subtree $\<I'XiI'[IXiI'Xi]less>$,
%which form a potential divergent subgraph  with 3 derivatives. 
By enumerating the $12$ trees, one finds that for $6$ of them 
the two noises being contracted %in $\<I'XiI'[IXiI'Xi]less>$ 
have distinct spatial indices,
% (the expansion of $\CA_i^\e$~\eqref{e:CA-expand} is helpful for this),
for instance
$\big(\CD_j^+ \CS_{\southeast} \mcb{I} (\CS_{\be}\Psi_i \CD_j \Psi_i)\big)
\big( \CD_j^+ \CS_{\southwest} \bar\Psi_j\big)$.
By independence
of $\xi_i,\xi_j$, 
 we only need to consider the other $6$ trees.
% 
%one might be able to use parity, but on lattice we must check if we have {\it exact} parity,
%so we look into each tree more explicitly.
First, consider 
$\big(\CD_j^+ \CS_{\southeast} \mcb{I} (\Psi_j \CD_j \Psi_j)\big)
\big( \CD_j^+ \CS_{\southwest} \bar\Psi_j\big)$.
The term of interest  in the 1st chaos has a potential divergence, where $z\in \R\times \obonds_i$:  %the form $\bar K * \Psi_j$ where 
\begin{equ}[e:mom-R2-DjK3]
\e^{1-\kappa} \!\!
\int_{(\R\times \obonds_j)^2} \!\!
\partial_j^+ K(z^\southwest - y)
\partial_j^+ K(z^\southeast - w)
\partial_j K( w-y) \mrd w \mrd y\;.
\end{equ}
By a change of variables $w_j \mapsto 2z_j-w_j$ and $y_j \mapsto 2z_j-y_j$,
and recalling that $K$ is even in each spatial direction and the definitions of $z^\southwest ,z^\southeast,\partial_j^+,\partial_j$,
we see that each of the three factors in the integrand flips the sign, so \eqref{e:mom-R2-DjK3} is zero. (Essentially
$\partial_j^+ K(z^\southwest) = \frac{1}{\e} (K (z^\northwest) - K (z^\southwest))$ is still a ``symmetrised'' derivative in the $j$th direction.)
%\hao{Probably don't need this sentence now. On the other hand, 
%if this is $\e^{-\kappa}$ instead of $0$,
%having $c\CA_j^\e$ is fine due to your new argument;
%it's just that $\e^{-\kappa}$ isn't allowed here to bound the model...
%In principle
%  we could also treat it in the same way as the 3 trees listed in the displayed equation below, namely we renormalize it (and we have analog of Lemma \ref{lem:thick-trig-bahave}), but that doesn't necessarily shorten things}
Now with the kernel
 $\bar K=\mathscr{R} \bar K $ of degree $-4-\kappa$
 denoted by $\tikz[baseline=-3]\draw[kernels] (0,0) --(.8,0);$ which is defined to be the above integral without integrating $w$,
 one  has the same bound as \eqref{e:mom-2424},
again satisfying  \eqref{e:CovarianceBounds}.
%
%
%\[
%\bigl(\CK^{(\eps; 1)}\tau \bigr)(z,\bar z) = \e^{2-2\kappa}\int \bar K^\e (z-w) (K^\e)^{*2} (w-\bar w) \bar K^\e (\bar z-\bar w)\mrd w\mrd \bar w
%\]
%and we get a bounded by $\lambda^{-2\kappa}$.
%
Exactly the same argument applies to 
$\big( \CD_j^+ \CS_{\southeast} \bar\Psi_j\big)
\big(\CD_j^+ \CS_{\southwest} \mcb{I} (\Psi_j \CD_j \Psi_j)\big)$.
We then turn to the case
$ (\CD_j^{\pm} \mcb{I} (\CS_{\be} \Psi_j \bar\CD_i \Psi_j) )(\bar\CD_j \bar\Psi_j)$.
The potential divergence (which is independent of $\be$) is an integral of a function which is odd in the $i$th direction, so we again have a bound \eqref{e:mom-2424}.
It then remains to 
consider 
\begin{equs}
(\CD_j^{\pm} \bar\Psi_i)  \CD_j^+ &\CS_{\bs}  \mcb{I} (\CS_{\be} \Psi_i \bar\CD_j \Psi_i))\;,
\\
\big( \CD_j^+ \CS_{\southeast} \bar\Psi_j\big)
\big(\CD_j^+ \CS_{\southwest} \mcb{I} (\CS_{\be}\Psi_i \CD_i \Psi_j)\big)
\;,
& \qquad
\big(\CD_j^+ \CS_{\southeast} \mcb{I} (\CS_{\be}\Psi_i \CD_i \Psi_j)\big)
\big( \CD_j^+ \CS_{\southwest} \bar\Psi_j\big)\;,
\end{equs}
where $\be\in\CE_\times$ and $\bs\in\{\Southeast,\Southwest\}$.
These are precisely the trees renormalised as in 
\eqref{e:div-R2} and
\eqref{e:def-C-triangle} by 
$C_{\bs,\pm}^\e[ \, \<I'XiI'[IXiI'Xi]less>\,]$,
%$C_{\southeast,\pm}^\e[ \, \<I'XiI'[IXiI'Xi]less>\,]$,
$C_{\southwest}^\e[ \, \<I'XiI'[IXiI'Xi]less>\,]$,
$C_{\southeast}^\e[ \, \<I'XiI'[IXiI'Xi]less>\,]$,
for which we can obtain bounds of the form \eqref{e:mom-2424}.
For instance if $\tau$ is the last tree,
we have (see \eqref{e:def-C-triangle3})
\begin{equs}[e:product-expectation]
\hat\Pi_{z_*}^\e\tau (z)
&=\e^{1-\kappa} 
\partial_j^+\Psi_j(z^\southwest)
 (\partial_j^+ K *( \Psi_i(\cdot+\be) \partial_i \Psi_j))(z^\southeast)
 \\
&\qquad -  \E [ \e^{1-\kappa} \partial_j^+\Psi_j(z^\southwest)
 (\partial_j^+ K * \partial_i \Psi_j)(z^\southeast)]
\Psi_i (z)\;,
\end{equs}
which is equal to $(\mathscr{R} J *_{(i)} \Psi_i)(z)$ with $J$ of degree $-4-\kappa$ as above.
The other trees follow similarly.
(These constants may not in general vanish, see Remark \ref{rem:C_non_vanish}.)\footnote{In these calculations, the 3 trees which do require renormalisation due to lack of exact parity 
are renormalised by $c\Psi_i$ which will contribute a term $c A_i$ to the renormalised equation;
the 3 trees which would potentially renormalise the equation by a term $c A_j$ for $j\neq i$ all enjoy exact parity so 
they are not renormalised. This is consistent with the fact that we show in Section \ref{sec:gauge-covar} (by other means) that we do not have a renormalisation term $c A_j$ in the $A_i$ equation.}
%Instead, we can replace the kernels $K^\e$ by the full heat kernel $P^\e$;
%the error of this replacement vanishes.
%Moreover, we can replace the $w$ integral in \eqref{e:mom-R2-DjK2DiK} by an integral over $z$.
%With these replacements, and 
%applying Lemma~\ref{lem:identity}, \eqref{e:mom-R2-DjK2DiK} is then turned into
%\begin{equ}
%\e^{1-\kappa} \!\!
%\int_{\R\times \obonds_j} \!\!
%P^\e(w - y - \e_i)
%\partial_i P^\e( w-y) \mrd y\;.
%\end{equ}
%By Lemma~\ref{lem:ePsiDPsi} this is $\e^{-\kappa} O(1)$.
%

The case  $\<I'XiI'[IXiI'Xi]less>$ is analogous to the above arguments for $\<I'XiI'[IXiI'Xi]>$.

The bound for $\<I'(R2-1)new>$ follows from that for $\<R2-1new>$.
%It only has the 2nd chaos  which is easily bounded:
%\begin{equ}[e:momR2-Ydiamond]
%\begin{tikzpicture}  [scale=0.4,baseline=0]		
%\node at (.6, 1)  [var2] (a)  {};  
%\node at (-.6,1) [var2] (b) {}; 
%\draw[very thick] (0,0) -- (a);
%\draw[very thick] (0,0) -- (b);
%\draw[very thick] (0,0) -- (0,-1);
%
%\node at (3.6,1)  [var2] (c)  {};  
%\node at (3-.6,1) [var2] (d) {}; 
%\draw[very thick] (3,0) -- (d);
%\draw[very thick] (3,0) -- (c);
%\draw[very thick] (3,0) -- (3,-1);
%
%\draw[bend left =40, dotted]  (b) to (c) ;
%\draw[bend left =40, dotted]  (a) to (d) ;
%\end{tikzpicture}
%\qquad
%\qquad
%\big(\CK^{(\eps; 2)}\tau\big)\big(z_1, z_2\big) \lesssim \|z_1-z_2\|_{\e}^{-\kappa}
%\end{equ}

Consider $\tau=\<IXiI'(R2-1)new>$ and  $\tau=\<I'XiI(R2-1)new>$. 
For the 3rd and 1st chaoses we can obtain 
%a bound of the form \eqref{e:momR2-Ydiamond} 
the desired bounds
by observing the following graphs
and absorbing powers of $\e$ into kernels as before: 
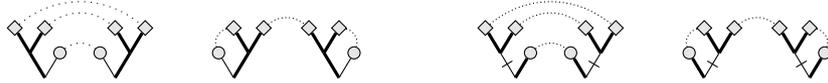
\vspace{-5pt}
\[
\begin{tikzpicture}  [scale=0.33,baseline=10]		
\node at (.6, 1)  [var2] (a)  {};  
\node at (-.6,1) [var2] (b) {}; 
\node at (1.2,0) [var] (x) {}; 
\draw[very thick] (0,0) -- (a);
\draw[very thick] (0,0) -- (b);
\draw[very thick] (0,0) -- (0.6,-1);
\draw (x) -- (0.6,-1);

\node at (4.6,1)  [var2] (c)  {};  
\node at (4-.6,1) [var2] (d) {}; 
\node at (4-1.2,0) [var] (y) {}; 
\draw[very thick] (4,0) -- (d);
\draw[very thick] (4,0) -- (c);
\draw[very thick] (4,0) -- (4-.6,-1);
\draw (y) -- (4-0.6,-1);

\draw[bend left =40, dotted]  (b) to (c) ;
\draw[bend left =40, dotted]  (x) to (y) ;
\draw[bend left =40, dotted]  (a) to (d) ;
\end{tikzpicture}
\qquad
\begin{tikzpicture}  [scale=0.33,baseline=10]		
\node at (.6, 1)  [var2] (a)  {};  
\node at (-.6,1) [var2] (b) {}; 
\node at (-1.2,0) [var] (x) {}; 
\draw[very thick] (0,0) -- (a);
\draw[very thick] (0,0) -- (b);
\draw[very thick] (0,0) -- (-0.6,-1);
\draw (x) -- (-0.6,-1);

\node at (3.6,1)  [var2] (c)  {};  
\node at (3-.6,1) [var2] (d) {}; 
\node at (3+1.2,0) [var] (y) {}; 
\draw[very thick] (3,0) -- (d);
\draw[very thick] (3,0) -- (c);
\draw[very thick] (3,0) -- (3.6,-1);
\draw (y) -- (3+0.6,-1);

\draw[bend right =50, delta]  (b) to (x) ;
\draw[bend left =50, delta]  (c) to (y) ;
\draw[bend left =40, delta]  (a) to (d) ;
\end{tikzpicture}
\qquad
\qquad
\begin{tikzpicture}  [scale=0.33,baseline=10]		
\node at (.6, 1)  [var2] (a)  {};  
\node at (-.6,1) [var2] (b) {}; 
\node at (1.2,0) [var] (x) {}; 
\draw[very thick] (0,0) -- (a);
\draw[very thick] (0,0) -- (b);
\draw[kernel1] (0,0) -- (0.6,-1);
\draw[very thick] (x) -- (0.6,-1);

\node at (4.6,1)  [var2] (c)  {};  
\node at (4-.6,1) [var2] (d) {}; 
\node at (4-1.2,0) [var] (y) {}; 
\draw[very thick] (4,0) -- (d);
\draw[very thick] (4,0) -- (c);
\draw[kernel1] (4,0) -- (4-.6,-1);
\draw[very thick] (y) -- (4-0.6,-1);

\draw[bend left =40, delta]  (b) to (c) ;
\draw[bend left =40, delta]  (x) to (y) ;
\draw[bend left =40, delta]  (a) to (d) ;
\end{tikzpicture}
\qquad
\begin{tikzpicture}  [scale=0.33,baseline=10]		
\node at (.6, 1)  [var2] (a)  {};  
\node at (-.6,1) [var2] (b) {}; 
\node at (-1.2,0) [var] (x) {}; 
\draw[very thick] (0,0) -- (a);
\draw[very thick] (0,0) -- (b);
\draw[kernel1] (0,0) -- (-0.6,-1);
\draw[very thick] (x) -- (-0.6,-1);

\node at (3.6,1)  [var2] (c)  {};  
\node at (3-.6,1) [var2] (d) {}; 
\node at (3+1.2,0) [var] (y) {}; 
\draw[very thick] (3,0) -- (d);
\draw[very thick] (3,0) -- (c);
\draw[kernel1] (3,0) -- (3.6,-1);
\draw[very thick] (y) -- (3+0.6,-1);

\draw[bend right =50, delta]  (b) to (x) ;
\draw[bend left =50, delta]  (c) to (y) ;
\draw[bend left =40, delta]  (a) to (d) ;
\end{tikzpicture}
\]
Here the 3rd graph has  ``positively renormalised'' kernels and is bounded in the same way as 
\eqref{e:K3-01} using \eqref{e:Diff-K}. 
The last graph is also treated as in \eqref{e:K3-02}.
%also satisfies the same bound \eqref{e:kappa3LSSS} as in the case \eqref{e:K3-02}.
\end{proof}

\begin{lemma}
	\label{lem:mom-222333}
For every $\tau$ in groups \hyperlink{(3-3)}{(3-3)}, \hyperlink{(2-3)}{(2-3)} and \hyperlink{(2-2)}{(2-2)},  the bounds \eqref{e:mom-bounds} hold.
\end{lemma}

\begin{proof}
In the proof 
we will often use the positive powers of $\e$ to improve the singularities of the kernels as before without 
explicitly mentioning.
Let 
$\tau=\<I(PsiXi1)Xi1>$.
The 3rd chaos is bounded as follows, using  \eqref{e:Diff-K} with 
$\alpha=1$: \footnote{Here \eqref{e:Diff-K} with 
$\alpha=0$ i.e. bounding the two terms in  $\tikz[baseline=-3]\draw[kernel1] (0,0) --(.8,0);$ separately
would lead to a bound $\big(\CK^{(\eps; 3)}\tau\big)\big(z_1, z_2\big)  \lesssim \e^{4-6\kappa}  \delta_\e(z_1-z_2)$, but this is not sufficient for the 
desired (small scale) bound.}
\begin{equ}[e:I(PsiXi1)Xi1-3rd]
\begin{tikzpicture}  [scale=0.33,baseline=0]		
\node at (.6, 1.2)  [var] (a)  {};  
\node at (1.2,0) [var1] (b) {}; 
\node at (0.6,-1.2) [var1] (x) {}; 
\draw (b) -- (a);
\draw[kernel1] (b) -- (x);

\node at (4-.6,1.2) [var] (d) {}; 
\node at (4-1.2,0)  [var1] (c)  {};  
\node at (4-.6,-1.2) [var1] (y) {}; 
\draw (c) -- (d);
\draw[kernel1] (c) -- (y);

\draw[delta]  (a) to (d) ;
\draw[delta]  (b) to (c) ;
\draw[delta]  (x) to (y) ;
\end{tikzpicture}
\qquad
\big(\CK^{(\eps; 3)}\tau\big)\big(z_1, z_2\big) 
 \lesssim  %\|z_1- z_2\|_\e^{-2-\kappa}
 (\|z_1\|_\e^{2} +\|z_2\|_\e^{2})
 \delta_\e(z_1-z_2)  \e^{2-6\kappa}.
\end{equ}
%From this we can deduce a bound $\lambda^{-6\kappa}$ as in \eqref{e:mom-cheY3}.
Now \eqref{e:CovarianceBounds} applies with $(\zeta,\alpha,n_\delta,n_\e)=(2,0,1,2-6\kappa)$.
%(in particular for the small scale we bound $\|z_1- z_2\|_\e^{-2-\kappa}\le \e^{-2-\kappa}$).

For one term in the 1st chaos the second moment is bounded as follows
\begin{equ}[e:I(PsiXi1)Xi1-1st]
\begin{tikzpicture}  [scale=0.33,baseline=0]		
\node at (.6, 1.2)  [var] (a)  {};  
\node at (1.2,0) [var1] (b) {}; 
\node at (0.6,-1.2) [var1] (x) {}; 
\draw (b) -- (a);
\draw[kernel1] (b) -- (x);

\node at (4-.6,1.2) [var] (d) {}; 
\node at (4-1.2,0)  [var1] (c)  {};  
\node at (4-.6,-1.2) [var1] (y) {}; 
\draw (c) -- (d);
\draw[kernel1] (c) -- (y);

\draw[bend right =40, delta]  (a) to (x) ;
\draw[bend left =40, delta]  (b) to (c) ;
\draw[bend left =40, delta]  (d) to (y) ;
\end{tikzpicture}
\qquad
\big(\CK^{(\eps; 1)}\tau\big)\big(z_1, z_2\big) 
 \lesssim 
% \|z_1- z_2\|_\e^{-2-4\kappa}
% (\|z_1\|_\e^{2} +\|z_2\|_\e^{2}+\|z_1-z_2\|_\e^{2})
 \|z_1- z_2\|_\e^{-4\kappa} +  \|z_1\|_\e^{-4\kappa}+ \| z_2\|_\e^{-4\kappa}
\end{equ}
%from which we obtain the desired bounds by \eqref{e:CovarianceBounds} and \eqref{e:kappa3LSSS}.
essentially the same as \eqref{e:kappa3}.
For the other term %which is graphically denoted as
%\begin{equ}[e:I(PsiXi1)Xi1-1st1]
%\begin{tikzpicture}  [scale=0.33,baseline=0]		
\tikz[scale=0.27,baseline=-7]
{
\node at (-.5, 0)  [var] (a)  {};  
\node at (1,0) [var1] (b) {}; 
\node at (0,-1.2) [var1] (x) {}; 
\draw (b) -- (a);
\draw[kernel1] (b) -- (x);
\draw[bend left = 60, delta, thick]  (b) to (x) ;
} ,
 recalling \eqref{e:M-XiIXi} and \eqref{e:C-XiIXi},,
upon cancellation by  $C^\e[\,\<I(PsiXi1)Xi1less>\,]$ we are left with
$\e^{2-2\kappa}\int K(-z)K(z-w)\xi(w)\mrd w$,
and the covariance is bounded by $\|z_1\|^{-3\kappa}\|z_2\|^{-3\kappa}$ which can be treated easily.
%bounded as in \eqref{e:kappa3LSSS}.
The bound for 
$\<I(PsiXi1)Xi1less>$ is similar and even simpler.

Let $\tau=\<r_z_large>$. 
%For the 4th chaos, $\big(\CK^{(\eps; 4)}\tau\big)\big(z_1, z_2\big)$ is equal to
%\begin{equ}
% \e^{4-4\kappa} (K^\e)^{*2} (0)^2 
%\int K^\e(z_1-w)K^\e(z_2-w) \delta_\e (z_1-z_2)\mrd w\mrd z_1\mrd z_2
%\end{equ}
%plus (up to possible minus sign) three other terms given by the same expression with $z_1$ or $z_2$ (or both) in the function $K^\e$
%replaced by $z_*$.
%Here $\delta_\e$ is as in \eqref{e:def-delta-eps}. Since $|K^\e * K^\e |\lesssim \e^{-\kappa}$,
The 4th chaos can be bounded in the same way as the 3rd chaos of 
$\<I(PsiXi1)Xi1>$ above. The 2nd chaos has $5$ terms:
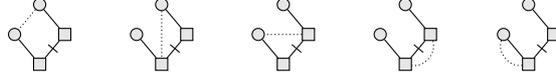
\[
\begin{tikzpicture}  [scale=0.33,baseline=0]		
\node at (0, 1.2)  [var] (a)  {};  
\node at (1,0) [var1] (b) {}; 
\node at (0,-1.2) [var1] (x) {}; 
\node at (-1,0) [var] (y) {}; 
\draw (b) -- (a);
\draw[kernel1] (b) -- (x);
\draw (x) -- (y);
\draw[delta]  (a) to (y) ;
\end{tikzpicture}
\qquad
\begin{tikzpicture}  [scale=0.33,baseline=0]		
\node at (0, 1.2)  [var] (a)  {};  
\node at (1,0) [var1] (b) {}; 
\node at (0,-1.2) [var1] (x) {}; 
\node at (-1,0) [var] (y) {}; 
\draw (b) -- (a);
\draw[kernel1] (b) -- (x);
\draw (x) -- (y);
\draw[delta]  (a) to (x) ;
\end{tikzpicture}
\qquad
\begin{tikzpicture}  [scale=0.33,baseline=0]		
\node at (0, 1.2)  [var] (a)  {};  
\node at (1,0) [var1] (b) {}; 
\node at (0,-1.2) [var1] (x) {}; 
\node at (-1,0) [var] (y) {}; 
\draw (b) -- (a);
\draw[kernel1] (b) -- (x);
\draw (x) -- (y);
\draw[delta]  (b) to (y) ;
\end{tikzpicture}
\qquad
\begin{tikzpicture}  [scale=0.33,baseline=0]		
\node at (0, 1.2)  [var] (a)  {};  
\node at (1,0) [var1] (b) {}; 
\node at (0,-1.2) [var1] (x) {}; 
\node at (-1,0) [var] (y) {}; 
\draw (b) -- (a);
\draw[kernel1] (b) -- (x);
\draw (x) -- (y);
\draw[bend left = 60, delta]  (b) to (x) ;
\end{tikzpicture}
\qquad
\begin{tikzpicture}  [scale=0.33,baseline=0]		
\node at (0, 1.2)  [var] (a)  {};  
\node at (1,0) [var1] (b) {}; 
\node at (0,-1.2) [var1] (x) {}; 
\node at (-1,0) [var] (y) {}; 
\draw (b) -- (a);
\draw[kernel1] (b) -- (x);
\draw (x) -- (y);
\draw[bend right = 60, delta]  (y) to (x) ;
\end{tikzpicture}
\]
For the first term,  the bound follows in the same way as \eqref{e:I(PsiXi1)Xi1-3rd}.
For the second term the bound follows in the same way as \eqref{e:I(PsiXi1)Xi1-1st}.
The fourth term follows in the same way as %\eqref{e:I(PsiXi1)Xi1-1st1}.
\tikz[scale=0.27,baseline=-7]
{
\node at (-.5, 0)  [var] (a)  {};  
\node at (1,0) [var1] (b) {}; 
\node at (0,-1.2) [var1] (x) {}; 
\draw (b) -- (a);
\draw[kernel1] (b) -- (x);
\draw[bend left = 60, delta, thick]  (b) to (x) ;
} above.
These are all because $|\e^\kappa K(-\cdot)*K|\lesssim 1$.
For the third term,
by  \eqref{e:Diff-K} with 
$\alpha=1$, one can deduce a bound on 
$\big(\CK^{(\eps; 2)}\tau\big)$ by the right-hand side of \eqref{e:I(PsiXi1)Xi1-3rd}.
The fifth term is cancelled by our choice of  $C^\e [ \,\<PsiXi1>\, ]$
in \eqref{e:M-XiIXi} and \eqref{e:C-XiIXi}.
Turning to the 0th chaos, there are two terms
\begin{equ}[e:c1c2-cancel]
\begin{tikzpicture}  [scale=0.33,baseline=0]		
\node at (0, 1.2)  [var] (a)  {};  
\node at (1,0) [var1] (b) {}; 
\node at (0,-1.2) [var1] (x) {}; 
\node at (-1,0) [var] (y) {}; 
\draw (b) -- (a);
\draw[kernel1] (b) -- (x);
\draw (x) -- (y);
\draw[delta]  (a) to (y) ;
\draw[bend left = 60, delta]  (b) to (x) ;
\end{tikzpicture}
\qquad\qquad
\begin{tikzpicture}  [scale=0.33,baseline=0]		
\node at (0, 1.2)  [var] (a)  {};  
\node at (0,0) [var1] (b) {}; 
\node at (0,-1.2) [var1] (x) {}; 
\node at (-1.2,0) [var] (y) {}; 
\draw (b) -- (a);
\draw[kernel1] (b) -- (x);
\draw (x) -- (y);
\draw[delta]  (b) to (y) ;
\draw[delta,bend left = 60]  (a) to (x) ;
\end{tikzpicture}
\end{equ}
Our choice of renormalisation \eqref{e:M-four} and \eqref{e:def-C-four}
precisely cancel these terms, up to remainders from the positively renormalised kernels,
which clearly satisfy the desired bounds.
The bounds for $\<PsiPsibarXibar>$ are similar and simpler.

For $\tau=\<I(R2-1)Xi1new>$, the desired bounds follow similarly as \eqref{e:mom-Y333} and \eqref{e:mom-Y11}.

Consider $\<cherry231>$ which includes the following trees (which appear in \eqref{e:tildeR2-exp}):
\minilab{e:cherry231-list}
\begin{equs}
(\CD_j^{\pm} \mcb{I} (\Psi_i \bar\Xi_i) )(\bar\CD_j \bar\Psi_j),
&\quad
(\CD_j^{\pm} \bar\Psi_i)( \CD_j^+ \CS_{\bs}  \mcb{I} (\Psi_j \bar\Xi_j)),		\label{e:cherry231-list1}
\\
(\CD_j^+ \CS_{\southeast} \mcb{I}  (\Psi_j \bar\Xi_j)) (\CD_j^+  \CS_{\southwest}\bar\Psi_j),
&\quad
(\CD_j^+  \CS_{\southeast} \bar\Psi_j) (\CD_j^+ \CS_{\southwest} \mcb{I}  (\Psi_j \bar\Xi_j)),	\label{e:cherry231-list2}
\end{equs}
where $\bs\in \{\Southwest,\Southeast\}$.
Recall that they are not renormalised.
Their 3rd chaoses can be all bounded as the {\it 1st} graph in \eqref{e:grouped-graph1} by $\big(\CK^{(\eps; 3)}\tau\big)\big(z_1, z_2\big) 
 \lesssim \|z_1- z_2\|_\e^{-4\kappa}$. 
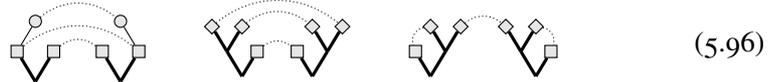
\vspace{-5pt}
\begin{equ}[e:grouped-graph1]
\begin{tikzpicture}  [scale=0.4,baseline=0]		
\node at (.6, 1)  [var] (a)  {};  
\node at (0,0) [var1] (b) {}; 
\node at (1.2,0) [var1] (x) {}; 
\draw (b) -- (a);
\draw[very thick] (b) -- (0.6,-1);
\draw[very thick] (x) -- (0.6,-1);

\node at (4,0)  [var1] (c)  {};  
\node at (4-.6,1) [var] (d) {}; 
\node at (4-1.2,0) [var1] (y) {}; 
\draw (c) -- (d);
\draw[very thick] (c) -- (4-.6,-1);
\draw[very thick] (y) -- (4-0.6,-1);

\draw[bend left =40, delta]  (b) to (c) ;
\draw[bend left =40, delta]  (x) to (y) ;
\draw[bend left =40, delta]  (a) to (d) ;
\end{tikzpicture}
\qquad
\begin{tikzpicture}  [scale=0.33,baseline=0]		
\node at (.6, 1)  [var2] (a)  {};  
\node at (-.6,1) [var2] (b) {}; 
\node at (1.2,0) [var1] (x) {}; 
\draw[very thick] (0,0) -- (a);
\draw[very thick] (0,0) -- (b);
\draw[very thick] (0,0) -- (0.6,-1);
\draw[very thick] (x) -- (0.6,-1);

\node at (4.6,1)  [var2] (c)  {};  
\node at (4-.6,1) [var2] (d) {}; 
\node at (4-1.2,0) [var1] (y) {}; 
\draw[very thick] (4,0) -- (d);
\draw[very thick] (4,0) -- (c);
\draw[very thick] (4,0) -- (4-.6,-1);
\draw[very thick] (y) -- (4-0.6,-1);

\draw[bend left =40, delta]  (b) to (c) ;
\draw[bend left =40, delta]  (x) to (y) ;
\draw[bend left =40, delta]  (a) to (d) ;
\end{tikzpicture}
\qquad
\begin{tikzpicture}  [scale=0.33,baseline=0]		
\node at (.6, 1)  [var2] (a)  {};  
\node at (-.6,1) [var2] (b) {}; 
\node at (-1.2,0) [var1] (x) {}; 
\draw[very thick] (0,0) -- (a);
\draw[very thick] (0,0) -- (b);
\draw[very thick] (0,0) -- (-0.6,-1);
\draw[very thick] (x) -- (-0.6,-1);

\node at (3.6,1)  [var2] (c)  {};  
\node at (3-.6,1) [var2] (d) {}; 
\node at (3+1.2,0) [var1] (y) {}; 
\draw[very thick] (3,0) -- (d);
\draw[very thick] (3,0) -- (c);
\draw[very thick] (3,0) -- (3.6,-1);
\draw[very thick] (y) -- (3+0.6,-1);

\draw[bend right =50, delta]  (b) to (x) ;
\draw[bend left =50, delta]  (c) to (y) ;
\draw[bend left =40, delta]  (a) to (d) ;
\end{tikzpicture}
\end{equ}
The 1st chaos of each of these trees
has two terms,
one of which,
given by contracting $\Psi$ and $\bar\Psi$,
can be essentially bounded as in \eqref{e:mom-supercheche}.
We now turn to the other term in the 1st chaos,
obtained by contracting $\bar\Psi$  and $\bar\Xi$.
This term vanishes for \eqref{e:cherry231-list1}
by independence of $\xi_i$ and $\xi_j$.
For \eqref{e:cherry231-list2},
%the same argument as for 
%$\<R2-1new>$ in the proof of Lemma~\ref{lem:mom-xi},
%which uses 
%Lemma~\ref{lem:identity}
%and the fact that $x^{\southeast}\neq x^{\southwest}$,
%applies here,
%so  we again have a bound of the form \eqref{e:mom-2424}.
the desired bound follows by our choice of $C^\e [\<cherry232>]$ in \eqref{eq:C_square_cherries}.
The case $\tau=\<cherry232>$ is similar (and simpler).

Finally, for $\<I'Xi1I'(R2-1)new>$
we also have
the desired bounds
by observing the {\it 2nd and 3rd} graphs in \eqref{e:grouped-graph1}
and absorbing powers of $\e$ into kernels as before.
\end{proof}

\section{Stochastic heat equation}
\label{sec:SHE}

In this section we establish uniform in $\e$ estimates on the discrete stochastic heat equation (SHE) as a function of time with values in $\Omega_{N;\alpha}$ (Proposition \ref{prop:SHE}).
Our results can be seen as the discrete versions of~\cite[Sec.~4]{CCHS_2D},
but several statements and proofs below differ
%to those of~\cite{CCHS_2D} 
and give sharper results (e.g. Lemmas~\ref{lem:SHE_gr}-\ref{lem:SHE_rho}).

%We fix $\e=2^{-N}$ throughout this section
%and let $\T^2_\e$ be a generic square lattice in $\T^2$ with spacing $\eps$.
Recall $\e=2^{-N}$ and  $\T^2_\e$ as in  Section~\ref{sec:funcSpaces}.
Equip $\T^2_\e$ with the natural measure $\mrd x = \e^{2} m(\mrd x)$ approximating the Lebesgue measure on $\T^2$, where $m$ is the counting measure on $\T^2_\e$.
Recall the discrete Laplace operator $\Delta\equiv\Delta_\e$ from~\eqref{e:laplacian} and let $\{e^{t\Delta}\}_{t\geq 0}$ denote the associated heat semigroup.
By translation invariance, we can view $e^{t\Delta}$ as a function
$e^{t\Delta}\colon \T^2_\e \to [0,\infty)$ which acts on $L^2(\T^2_\e)$ by convolution $e^{t\Delta}f(x) = \int_{\T^2_\e} e^{t\Delta}(x-y)f(y) \mrd y$.

Let $\xi$ be the white noise over the Hilbert space $L^2(\R\times \T^2_\e)$.
Consider the solution to the SHE
$\Psi\equiv \Psi_N \colon \R_+\times \T^2_\e\to \R$
\begin{equ}\label{eq:disc_SHE}
\d_t \Psi = \Delta \Psi + \xi\;,
\end{equ}
with some initial condition $\Psi(0,\cdot)\colon\T^2_\e \to \R$.

If $\Psi(0)=0$, then $\Psi(t,x) = \int_0^t\mrd s\int_{\T^2_\e} \mrd y e^{(t-s)\Delta}(x-y)\xi(s,y)$, and therefore, by It\^o isometry and the semi-group property of $e^{t\Delta}$,
for all $x\in\T^2_\e, t\geq 0$
\begin{equs}[eq:C_t_x]
C_t(x) &\eqdef
\E[\Psi(t,x)\Psi(t,0)] = \int_0^t \mrd s\int_{\T^2_\e} \mrd ye^{(t-s)\Delta}(x-y)e^{(t-s)\Delta}(y)
\\
&= \int_0^t\mrd s e^{2(t-s)\Delta}(x) 
= \frac12\int_0^{2t} e^{u\Delta}(x)\mrd u\;.
\end{equs}

\begin{lemma}
Uniformly over $x\in\T^2_\e$ and $t\in [0,1]$
\begin{equ}\label{eq:hk_upper}
e^{t\Delta}(x) \lesssim \eps^{-2} \wedge t^{-1} \wedge |x|^{-2}\;,
\end{equ}
and uniformly over $y\in\T^2_\e$ with $|y|\geq |x|$
\begin{equ}\label{eq:hk_diff}
|e^{t\Delta}(x) - e^{t\Delta}(y)| \lesssim
(\e^{-3}\wedge t^{-3/2} \wedge |x|^{-3})|x-y|\;.
\end{equ}
\end{lemma}

\begin{proof}
This follow from~\cite[Lem.~B.1]{SSSX21}.
\end{proof}
Note that $\e^{-2}$ in~\eqref{eq:hk_upper} is only useful when $x=0$ and $t<\e^2$.
%Furthermore, by the trivial bound $|e^{t\Delta}(x) - e^{t\Delta}(y)| \leq |e^{t\Delta}(x)| + |e^{t\Delta}(y)|$,
%if $t^{1/2}\leq |x-y|$ we can replace $t^{-3/2}|x-y|$ by $t^{-1}$.
The identity~\eqref{eq:C_t_x} and bound~\eqref{eq:hk_upper} imply that, for all $x\in\T^2_\e,t\geq 0$,
\begin{equ}
%\label{eq:C_0_bound}
%|C_t(0)|& \lesssim \int_0^{t} \e^{-2}\wedge s^{-1}\mrd s =
%\begin{cases}
%t\e^{-2} & \text{ if } t < \e^2\\
%1+\log(t\e^{-2}) & \text{ if } t \geq \e^2\;,
%\end{cases}
%\\
\label{eq:C_x_bound}
|C_t(x)| \lesssim \int_0^{t} s^{-1}\wedge |x|^{-2} \mrd s =
\begin{cases}
t|x|^{-2} & \text{ if } t < |x|^2\;,\\
1+\log(t|x|^{-2}) & \text{ if } t \geq |x|^2\;,
\end{cases}
\end{equ}
and if $x=0$, one can replace $|x|$ on the right-hand side by $\eps$.
Consider further now $|y|\geq|x|$ with $|x-y|^2 \leq t$
for which \eqref{eq:hk_upper}-\eqref{eq:hk_upper} imply
\begin{equs}
|C_t(x)-C_t(y)|
&\lesssim \int_0^{|x-y|^2} s^{-1} \wedge |x|^{-2} \mrd s
+ \int_{|x-y|^2}^{t} (s^{-3/2}\wedge |x|^{-3})|x-y| \mrd s\;.
\end{equs}
For the case $|x| < |x-y|$, we obtain
\begin{equs}[eq:C_x_C_y_1]
|C_t(x)-C_t(y)|
&\lesssim 
1+\log(|x-y||x|^{-1}) + |x-y|(|x-y|^{-1}-t^{-1/2})
\\
&\lesssim 1+\log(|x-y||x|^{-1})\;.
\end{equs}
For the case $ |x-y|\leq |x| < t^{1/2}$, $|C_t(x)-C_t(y)|$ is bounded by a multiple of
\begin{equ}{}
|x-y|^2|x|^{-2} + |x-y||x|^{-3}(|x|^2-|x-y|^2) + |x-y|(|x|^{-1}-t^{-1/2})\;.
\end{equ}
Finally, for the case $ |x-y|\leq t^{1/2}\leq |x|$, we obtain
\begin{equ}
|C_t(x)-C_t(y)| \lesssim |x-y|^2|x|^{-2} + |x-y||x|^{-3}(t-|x-y|^2)\;.
\end{equ}
Note that in the final two cases we obtain the bound
\begin{equ}\label{eq:crude_diff}
|C_t(x)-C_t(y)|\lesssim |x-y||x|^{-1}
\end{equ}
(which is sharp for the range $|x-y|\ll |x|\ll t^{1/2}$).
Finally, we remark that
\begin{equs}[eq:C_0_C_y]
|C_t(0)-C_t(y)|&\lesssim
\int_0^{\e^2\wedge t}\e^{-2}\mrd s 
+ \int_{\e^2\wedge t}^{|y|^2\wedge t} s^{-1}\mrd s
+ \int_{|y|^2\wedge t}^t s^{-3/2}|y|\mrd s
\\
&= \begin{cases}
t\e^{-2} & \text{ if } t<\e^2\;,\\
1+\log ((t\wedge |y|^2)\e^{-2}) & \text{ if } \e^2 \leq t \leq |y|^2\;,\\
1+ \log (|y|^2\e^{-2}) + (1-|y|t^{-1/2}) & \text{ if } |y|^2 < t\;.
\end{cases}
\end{equs}

\begin{lemma}\label{lem:SHE_gr}
Suppose $\ell\subset \T^2_\e$ is a line containing $k$ points in $\T^2_\e$, i.e. $h\eqdef |\ell|=k \eps$.
Let $\delta_\ell\colon L^2(\T^2_\e)\to\R$ be the linear map $\delta_\ell(f) = \e\sum_{x\in\ell} f(x)$, which we also identify with a function in $L^2(\T^2_\e)$ by duality $\scal{\delta_\ell,f}=\delta_\ell(f)$.
Suppose $\Psi(0)=0$.
Then
\begin{equs}
\E \scal{\delta_\ell,\Psi(t)}^2
\lesssim h\big(\big[h\log(th^{-2} + 1)\big] \wedge t^{1/2}\big)
=
\begin{cases}
h^2\log(th^{-2}+1) & \text{ if } h<\sqrt t\\
ht^{1/2}
& \text{ if } h \geq \sqrt t\;.
\end{cases}
\end{equs}
\end{lemma}

\begin{proof}
Using~\eqref{eq:C_x_bound} and its sharpening for $x=0$,
\begin{equs}
\E  & \scal{\delta_\ell,\Psi(t)}^2
= \scal{\delta_\ell, C_t * \delta_\ell}
=
\e^{2}\sum_{x,y\in\ell} C_t(x-y)
\lesssim
h\e\sum_{j=0}^{k-1}(C_t(j\e e_1))\label{eq:d_ell_Psi_bound}
\\
&\lesssim h\e\big[t\e^{-2}\wedge (1+|\log(t\e^{-2})|)\big]
 +h\int_0^{h}(tx^{-2} )\wedge (1+|\log(tx^{-2})|) \mrd x
\\
&\lesssim h\big[t\e^{-1}\wedge (\e+\e|\log(t\e^{-2})|)\big]
+
\begin{cases}
h^2\log(th^{-2}+1) & \text{ if } h<\sqrt t\\
ht^{1/2}
& \text{ if } h \geq \sqrt t\;.
\end{cases}
\end{equs}
%The conclusion follows by remarking that, 
We then note that
in the final line, the first term is controlled by the second.
\end{proof}

\begin{lemma}\label{lem:SHE_rho}
Suppose $\ell,\bar\ell\subset \T^2_\e$ are parallel lines each containing $k$ points in $\T^2_\e$ and distance $r$ apart, i.e. $h\eqdef |\ell|=k \eps$ and $\inf_{x\in\ell,y\in\bar\ell}|x-y|=r$.
Suppose $\Psi(0)=0$ and $r<h$.
Then
$
\E \scal{\delta_\ell-\delta_{\bar\ell},\Psi(t)}^2
\lesssim 
(ht^{1/2}) \wedge (hr(1+\log(h/r)))
%\begin{cases}
%ht^{1/2} & \text{ if } t<r^2\\
%hr(1+\log(h/r)) & \text{ if } r^2 \leq t\;.
%\end{cases} 
%=
%h(t^{1/2} \wedge r)\;.
$.
\end{lemma}

\begin{remark}\label{rem:correct_Thm_413}
We point out a typo in~\cite[Lem.~4.8,~4.12]{CCHS_2D} related to the above estimate.
First, $\kappa\in(0,1)$ therein should be $\kappa\in(0,\frac12)$,
and second, the second bound in~\cite[Lem.~4.8]{CCHS_2D} is incorrect.
Thus, in the proof of~\cite[Lem.~4.12]{CCHS_2D},
one can at best apply~\cite[Lem.~4.11]{CCHS_2D} with $M \sim C_\xi^{p/2} t^{p\kappa/2}$ and $\alpha=1-2\kappa$ with $\kappa\in (0,\frac12)$ 
(instead of $\alpha=1-\kappa$ as therein implicitly).
Since $\kappa = (1-\alpha)/2$,
this means that $t^{p(1-\alpha)/2}$ in the right-hand side of the bound in~\cite[Lem.~4.12]{CCHS_2D} should read $t^{p(1-\alpha)/4}$ for $\alpha\in (0,1)$.
We note, however, that this has no effect on the proof of~\cite[Thm.~4.13]{CCHS_2D} -- in fact,
the final term $|t-s|^{p(1-\beta)/2}$ now becomes $|t-s|^{p(1-\beta)/4}$ which is more natural
since it has the same homogeneity as the terms $|t-s|^{p(\alpha-\bar\alpha)/4}$ coming from the heat flow.
\end{remark}

%\begin{remark}\label{rem:log_factor}
%We expect the factor $(1+\log(h/r))$ can be dropped in Lemma~\ref{lem:SHE_rho} -- it arises from the term~\eqref{eq:log_factor} below that in turn follows from~\eqref{eq:crude_diff}. This final bound can likely be improved by sharpening~\eqref{eq:hk_diff}.
%\end{remark}

\begin{proof}
As in~\eqref{eq:d_ell_Psi_bound},
$
\E \scal{\delta_\ell-\delta_{\bar\ell},\Psi(t)}^2
%&= 2 \scal{\delta_\ell-\delta_{\bar\ell}, C_t * \delta_\ell}
%=
%2\e^{2}\sum_{x,y\in\ell}\sum_{\bar y\in \bar \ell} C_t(x-y)-C_t(x-\bar y)
%\\
\lesssim
h\e\sum_{j=0}^{k-1}(C_t(j\e e_1) - C_t(j\e e_1 + re_2))$.
In all cases, in particular when $r^2>t$, by Lemma~\ref{lem:SHE_gr}
\begin{equ}
\E \scal{\delta_\ell-\delta_{\bar\ell},\Psi(t)}^2 \lesssim \E \scal{\delta_\ell,\Psi(t)}^2
\lesssim h\big(\big[h\log(th^{-2} + 1)\big] \wedge t^{1/2}\big) \leq ht^{1/2}\;.
\end{equ}
If $t\geq r^2$, then we gain something by using sharper bounds for $|C_t(x)-C_t(y)|$ above.
Considering the sum over $j\geq 1$, for $j^2\e^2 < r^2$ we have by~\eqref{eq:C_x_C_y_1}
\begin{equ}
|C_t(j\e e_1) - C_t(j\e e_1 + re_2)| \lesssim 1+\log(r|j\e|^{-1})\;,
\end{equ}
which contributes a multiple of
$
\int_0^{r} \{1+\log(rx^{-1})\} \mrd x = 2r
$.
For $j^2\e^2 \geq r^2$, by~\eqref{eq:crude_diff}, we have
$
|C_t(j\e e_1) - C_t(j\e e_1 + re_2)| \lesssim 
r|j\e|^{-1}
$,
which contributes a multiple of
%\begin{equ}[eq:log_factor]
$
\int_r^{h} rx^{-1} \mrd x \asymp r \log (h/r)
$.
%\end{equ}
In conclusion, the sum over $j\geq 1$ is bounded by a multiple of $hr(1+\log (h/r))$.
It remains to add the contribution from $j=0$ for $t\geq r^2$
for which, by~\eqref{eq:C_0_C_y},
\begin{equ}
h\e |C_t(0)-C_t(re_2)| \lesssim h\e [1+ \log (r^2\e^{-2}) + (1-rt^{-1/2})]
\lesssim hr\;.
\end{equ}
The conclusion follows since $hr(1+\log(h/r)) < ht^{1/2}$ only for $r^2<t$.
\end{proof}

\begin{lemma}\label{lem:hk_flow}
Uniformly in $t\in (0,1)$, $0<\bar\alpha\leq \alpha\in[0,1]$, and $\e>0$, one has $
|A-e^{t\Delta}A|_{N;\bar\alpha}\lesssim t^{(\alpha-\bar\alpha)/4}|A|_{N;\alpha}$,
where we treat $e^{t\Delta}$ a linear operator $e^{t\Delta}\colon \mfq_N\to\mfq_N$ by viewing $\mfq_N$ as two copies of $\mfg^{\T^2_\e}$.
\end{lemma}

\begin{proof}
We first claim that $|f-e^{t\Delta}f|_{L^\infty} \lesssim t^{\alpha/2}|f|_{\CC^\alpha_\e}$.
Indeed,
\begin{equ}
|(e^{t\Delta}f)(x) - f(x)| = |\E[f(x+\e X_{N(t\e^{-2})}) - f(x)]| \leq |f|_{\CC^\alpha_\e}\E[|\e X_{N(t\e^{-2})}|^\alpha]\;,
\end{equ}
where $X_0,X_1\ldots,$ is a simple symmetric random walk on $\Z^2$ and $\{N(t)\}_{t\geq 0}$ is an independent Poisson process.
The final term is bounded by a multiple of $|f|_{\CC^\alpha_\e} t^{\alpha/2}$ since
$\E[|X_n|^\alpha]\lesssim n^{\alpha/2}$ and $\E N(t)^{\alpha/2} \lesssim t^{\alpha/2}$ uniformly in $n,t\geq 0,\alpha\in[0,1]$, which proves the claim.
The conclusion follows by an analogue of \cite[Prop.~4.1]{CCHS_2D} with $\Omega^1_\alpha$ replaced by $\Omega_{N,\alpha}$.
\end{proof}
\begin{proposition}\label{prop:SHE}
For $i\in\{1,2\}$,
consider a white noise $\xi_i$
on $\R\times\obonds_i$ 
and let $\Psi^{(i)}$ be the solutions to the SHE~\eqref{eq:disc_SHE}
with $\xi=\xi_i$ and $\T^2_\e=\obonds_i$ (as usual, we identify $e\in\obonds_i$ with its midpoint).
For $t\geq 0$, denote $\Psi(t)=(\Psi^{(1)}(t),\Psi^{(2)}(t))\in \Omega_N$.

Then for $0<\bar\alpha<\alpha\leq 1$, $\kappa\in (0,\frac{\alpha-\bar\alpha}{4})$, $p \geq 1$, and $T>0$, we have
\begin{equ}
\E\Big[\sup_{0\leq s<t\leq T} |t-s|^{-p\kappa}|\e\Psi(t)-\e\Psi(s)|^p_{N;\bar\alpha}\Big] \lesssim |\e\Psi(0)|_{\alpha} + 1\;,
\end{equ}
where the proportionality constant depends only on $\bar\alpha,\alpha,\kappa,p,T$.
\end{proposition}

\begin{proof}
Suppose first that $\Psi(0)=0$ and $\alpha\in (\frac12,1)$. Then, by equivalence of Gaussian moments, Lemmas~\ref{lem:SHE_gr} and~\ref{lem:SHE_rho} imply that, for all parallel $\ell,\bar\ell\in\olines_N$ and $p\geq 1$, one has
$
\E|\e\Psi(t)(\ell)|^p \lesssim |\ell|^{p\alpha}t^{p(1-\alpha)/2}
$,
and for any $\beta>0$
\begin{equ}
\E|\e\Psi(t)(\ell)-\e\Psi(t)(\bar\ell)|^p \lesssim \rho(\ell,\bar\ell)^{p(\alpha-\beta)}|\ell|^{p\beta + p(1-\alpha)/2} t^{p(1-\alpha)/4}\;.
\end{equ}
%where the proportionality constants depend only on $\alpha,p,\beta$.
It thus follows from a discrete version of the Kolmogorov criterion of \cite[Lem.~4.11]{CCHS_2D} (with the same proof) that
\begin{equ}\label{eq:SHE_t_control}
\E|\e\Psi(t)|^p_{N;\bar\alpha} \lesssim t^{p(1-\alpha)/4}\;.
%\;, \qquad \forall \bar\alpha \in (0,\alpha-\tfrac{16}{p})\;.
\end{equ}
The proof now follows in the same way as the proof of~\cite[Thm.~4.13]{CCHS_2D} with Lemma \ref{lem:hk_flow} and~\eqref{eq:SHE_t_control} substituting the role of~\cite[Cor.~4.2 \& Lem.~4.12]{CCHS_2D} respectively
(see also Remark~\ref{rem:correct_Thm_413}).
\end{proof}

\section{Convergence of discrete dynamic to a continuum limit}
\label{sec:diagonal_argument}

We fix throughout this section a mollifier $\moll$ (not necessarily non-anticipative).
We prove in this section that the discrete dynamic $A^\eps$ converges along a subsequence of $\e\downarrow0$ to $\SYM(C)$ for \textit{some} operator $C\in L(\mfg^2,\mfg^2)$.
We make the following assumption throughout this section.
\begin{assumption}\label{as:subsequence}
We fix a sequence $\e=2^{-N}\downarrow0$ such that $c^\e_\star$ from Lemma~\ref{lem:renormalised-equ} converges along this sequence.
We continue to write this sequence as $\e$.
\end{assumption}

Recall from \eqref{eq:L_G} that $L_G(\mfg^2,\mfg^2)$ denotes
the space of linear maps which commute with the (diagonal) adjoint action by $G$.
\begin{theorem}\label{thm:some_C}
There exists a subsequence of $\e\downarrow0$ and $C\in L_G(\mfg^2,\mfg^2)$
such that the statement of Theorem~\ref{thm:discrete_dynamics} holds along this subsequence with $\bar C$ replaced by $C$.
Furthermore, $\sup_{N\geq 1} \E |T_N|^{-p} < \infty$ for all $p >0$. 
\end{theorem}
In Section~\ref{sec:gauge-covar} we will show that $C=\bar C$ regardless of the subsequence $\eps$ from Assumption~\ref{as:subsequence}, thus proving Theorem~\ref{thm:discrete_dynamics}.
We prove Theorem~\ref{thm:some_C} by a diagonal argument.
This argument has been used in several recent works,
e.g. \cite[Sec.~6]{MourratWeber}, \cite[Sec.~6]{HS15},  \cite[Sec.~7]{HM18}, \cite[Sec.~6]{CM18}, \cite[Sec.~4]{EH21}, \cite[Sec.~4]{ChoukGairingPerkowski} (see also \cite[Sec.~5.1]{MR4029148}
for an alternative approach based on 
convergence of discrete multiple stochastic integrals \cite{MR3584558}).
%\footnote{Another interesting approach is
%\cite{MR3584558} which provides a general criterion for the {\it convergence} of discrete multiple stochastic integrals to multiple Wiener-It\^o integrals, which was adapted later to the setting of Wick products in \cite[Sec.~5.1]{MR4029148}.}
In our case we would like to preserve the exact gauge symmetry along the discrete approximations,
and thus are {\it not} allowed to renormalise the discrete dynamics;
this is why we first show 
Theorem~\ref{thm:some_C} and then identify $C$ in Section~\ref{sec:gauge-covar}.

\subsection{Regularised discrete process \texorpdfstring{$A^{\e,\bar\e}$}{A eps, bar eps}}
\label{subsec:regular_discrete_process}

In the following, we will write $A^\e$ for the solution to \eqref{e:Aeps}.

Let $\bar\eps \ge \eps$ and, for $x\in \T^2$, define
$
\chi^{\eps,\bar\eps} (t,x) \eqdef \eps^{-2} \int \chi^{\bar\eps} (t,y) % 1_{|y-x| \le \eps/2}
1_{B(x,\e)}(y)
 \mrd y
$, 
where $\chi^{\bar\eps} (t,y) = \bar\eps^{-4} \chi (\bar\eps^{-2} t, \bar\eps^{-1} y )$
and $1_{B(x,\e)}$ is  as in~\eqref{e:couple-noise}.
The above definition is particularly useful when  $x\in \obonds_i$
(for instance in situation such as \eqref{e:def-xi-ebe} we think of it as restricting to $\obonds_i$)
 but also makes sense if 
 $x\in \T^2$.
 
 We then define the ``smoothened'' discrete noise	
\begin{equ}[e:def-xi-ebe]
\xi^{\eps,\bar\eps}_i = \chi^{\eps,\bar\eps} *_{(i)} \xi^\eps_i\;,\qquad i\in \{1,2\}
\end{equ}
where $*_{(i)}$ is convolution on $\R\times \obonds_i$, 
%In other words, for $(t,x)\in \R\times \obonds_i$,
%\[
%\xi^{\eps,\bar\eps}_i (t,x) 
%=\eps^2 \int \mrd s \sum_{y\in  \obonds_i}  \chi^{\eps,\bar\eps} (t-s, x-y) \xi^\eps_i (s,y)\;,
%\]
i.e. $\eps^2 \int \sum_{y\in  \obonds_i}  \chi^{\eps,\bar\eps} (\cdot -s, \cdot -y) \xi^\eps_i (s,y) \mrd s$.
%where the time integral is in the distributional sense.

%\begin{remark}
%Another way would be to average $\xi^\e$ over the entire $\obonds$,
%namely
%for $(t,x)\in \R\times \obonds$,
%\[
%\xi^{\eps,\bar\eps} (t,x) 
%=\eps^2 \int ds \sum_{y\in  \obonds}  \chi^{\eps,\bar\eps} (t-s, x-y) \xi^\eps (s,y)\;.
%\]
%However in this way we would lose independence between $\xi^{\eps,\bar\eps}_1$ and $\xi^{\eps,\bar\eps}_2$
%which would be slightly inconvenient for models and renormalisation.
%\end{remark}

For $C\in L(\mfg^2,\mfg^2)$, let $A^{\eps,\bar\eps}$  be the solution to the equation for $e\in \obonds_i$, 
$j\neq i$,
\begin{equs}\label{e:Aeps-eps-bar}
{}&\partial_t  A_i^{\eps,\bar\eps}  (e)
= \Delta A_i^{\eps,\bar\eps} (e)
+ \big[A_j^{\eps,\bar\eps} (e^{(a)}), 
(2\partial_j A_i^{\eps,\bar\eps})(e)  
- (\bar\partial_i A_j^{\eps,\bar\eps}) (e) \big] 
\\
&+ [A_i^{\eps,\bar\eps} (e), (\partial_i A_i^{\eps,\bar\eps})(e) ]
+ \big[A_j^{\eps,\bar\eps}(e^{(a)}) ,[A_j^{\eps,\bar\eps}(e^{(a)}) ,A_i^{\eps,\bar\eps}(e) ]\big] 
+ C_iA^{\eps,\bar\eps}  (e)
+ \xi_i^{\eps,\bar\eps} (e).
\end{equs}

\subsection{Proof of Theorem~\ref{thm:some_C}}
\label{subsec:some_C_proof}

Throughout this subsection, let us fix initial conditions $a\in\Omega^1_\alpha$, $a^{\e,\bar\e}\in\Omega_N$, and $a^\e \equiv a^{(N)}\in\Omega_N$ such that, for some $\eta\in (\frac\alpha4-\frac12,\alpha-1)$,
%\begin{equ}[e:a_initial]
%(i) \; \sup_\e |a^{\e}|_{N;\alpha}<\infty\;,
%\quad
%(ii) \;\lim_{\bar\e \to 0}|a^\e- a^{\e,\bar\e}|_{N;\alpha} = 0\;,
%\quad
%(iii)\; \lim_{\bar\e\to0}\|a^{\e,\bar\e}; a\|_{\CC^\eta_\e} = 0
%\end{equ}
\begin{equs}
\sup_\e |a^{\e}|_{N;\alpha}&<\infty\;,
\label{eq:sup_a_N}
\\
\lim_{\bar\e \to 0}\sup_{\e\in(0,\bar\e)}|a^\e- a^{\e,\bar\e}|_{N;\alpha} &= 0\;,\label{eq:ae_aee}\
\\
\lim_{\bar\e\to0}\sup_{\e\in(0,\bar\e)}\|a^{\e,\bar\e}; a\|_{\CC^\eta_\e} &= 0\;,\label{eq:aee_a}
\end{equs}
%where~\eqref{eq:ae_aee} and~\eqref{eq:aee_a} 
%hold uniformly in $\e\in (0,\bar\e)$.
(We later take $a^{\e,\bar\e}=a^\e$ in the proof of Theorem~\ref{thm:some_C}, but do not need this for now.)
\begin{remark}\label{rem:a_vs_ea}
Recall our convention from Section~\ref{subsubsec:compare_line_ints} that $a^\e\in\Omega_N$ is identified with $\{a^\e_i\}_{i\in[2]} \in \CC^\eta_\e(\obonds_1)\oplus \CC^\eta_\e(\obonds_2)$ where $a^\e_i(b)=\e^{-1}a^\e(b)$ for $b\in\obonds_i$.
In particular, $\|a^{\e,\bar\e};a\|_{\CC^\eta_\e}$ refers to
$\max_{i\in[2]}\|a^{\e,\bar\e}_i;a_i\|_{\CC^\eta_\e}$
in the notation of Section~\ref{sec:funcSpaces}.
\end{remark}
We let $A_C=\SYM(C,a)$ and let $A^{0,\bar \e}_C$ be the solution to~\eqref{eq:SYM_moll} with mollifier $\moll^{\bar\e}$,
map $C\in L(\mfg^2,\mfg^2)$, and initial condition $a$.
We furthermore let $A^{\e,\bar\e}_C$ be the solution to~\eqref{e:Aeps-eps-bar} with map $C$ and initial condition $\{a^{\e,\bar\e}_i\}_{i\in[2]}$.

The idea behind the proof of Theorem~\ref{thm:some_C} is to bound $|A^{(N)}-\pi_N A_C|_{N;\alpha}$ by
\begin{equ}[eq:diagonal_bound]
|A_C - A^{0,\bar \e}_C|_{\alpha} + |\pi_N A^{0,\bar \e}_C - A^{\e,\bar\e}_C|_{N;\alpha} + |A^{\e,\bar\e}_C - B^\e|_{N;\alpha} + |B^\e - A^{(N)}|_{N;\alpha}\;,
\end{equ}
where $B^\e = \hat\CR\CB^\e$ for $\CB^\e$ solving~\eqref{e:fix-pt-Be} with initial condition $a^\e$.
Here and below, $\hat\CR$ is the reconstruction operator associated to the renormalised (random) discrete model $(\hat\Pi^\e,\hat\Gamma^\e)$
as in Section~\ref{subsec:renorm_eq}.

For any $C$, the first term in~\eqref{eq:diagonal_bound} vanishes as $\bar\eps \downarrow0$ by~\cite[Thm.~2.4]{CCHS_2D}
while the second term, for any \textit{fixed} $\bar\eps>0$, vanishes as $\e\downarrow0$ by classical numerical analysis since $A^{0,\bar \e}_C(t)$ is smooth for $t>0$ (Lemma~\ref{lem:1-2_terms}).
The fourth term vanishes as $\eps\downarrow0$ by the results of Section~\ref{sec:Aeps} (Lemma~\ref{lem:4th_term}).
We show in this subsection that, for suitable $C$, the third term vanishes as $\bar\eps\downarrow0$ uniformly over $\eps<\bar\eps$  (Lemma \ref{lem:3_term}).
%We thus prove Theorem~\ref{thm:some_C} by taking $\bar\e>0$ small such that the first and third terms are small for all $\eps<\bar\eps$, and then choose $\eps$ small (depending on $\bar\eps$) such that the second and fourth terms are small.

We suppose for the remainder of the section that Assumption~\ref{assump:R} holds.
We summarise the bounds on the first and second terms in~\eqref{eq:diagonal_bound} in the following lemma.

\begin{lemma}\label{lem:1-2_terms}
Let $C\in L(\mfg^2,\mfg^2)$.
There exist $\{K_{\bar\e}\}_{\bar\e\in(0,1)}$ such that 
$\lim_{\bar\e\downarrow0}K_{\bar\e}=\infty$ and such that,
for $\tau(\bar\e) \eqdef K_{\bar\e}\wedge \inf\{t>0\,:\, |A_C(t)|_\alpha\vee |A^{0,\bar \e}_C(t)|_\alpha >K_{\bar\e}\}$,
one has
\begin{equ}[eq:AC_A0eC]
\lim_{\bar\e\to0}\|A_C - A^{0,\bar \e}_C\|_{\CC([0,\tau(\bar\e)],\Omega^1_\alpha)} =0\quad \text{ in probability,}
\end{equ}
and, for every fixed $\bar\e>0$,
\begin{equ}[eq:AeeC_A0eC]
\lim_{\e\to 0} 
\|  A^{\e,\bar\e}_C - \pi_N A^{0,\bar\e}_C\|_{\CC^{\tau(\bar\e)}_{\eta-\alpha/2}(\CC^{\alpha/2}_\e)}
=0\quad \text{ in probability.}
\end{equ}
\end{lemma}

\begin{proof}
The existence of $K_{\bar\e}$ as in the statement for which~\eqref{eq:AC_A0eC} holds readily follows from~\cite[Thm.~2.4]{CCHS_2D}.
To prove~\eqref{eq:AeeC_A0eC}, since $\xi^{0,\bar\e}\eqdef\moll^{\bar\e}*\xi$ is smooth,
we have
\begin{equ}[eq:A0eC_apriori]
\sup_{t\in (0,\tau(\bar\e)]}
t^{\frac12(k-\eta)}\|A^{0,\bar\e}_C(t)\|_{\CC^{k}}<\infty
\end{equ}
for any $k\geq 0$.
Recall that
$A^{\e,\bar\e}_{i,C}$ solves~\eqref{e:Aeps-eps-bar} with forcing term
$\xi^{\eps,\bar\eps}_i = \chi^{\eps,\bar\eps} *_{(i)} \xi^\e_i$
as in 
\eqref{e:def-xi-ebe}.
With the coupling of $\xi_i^\e$ and $\xi_i$ in \eqref{e:couple-noise}
we can
 show that 
$
\E [ (\xi_i^\e - \xi_i) (\phi^\lambda)^2] \lesssim \e^{2\kappa} \lambda^{-4-3\kappa}
$
for space-time rescaled test function $\phi^\lambda$ where $\xi_i^\e$  is viewed as linear combinations of Dirac mass distributions.
One then has the convergence of $\xi_i^\e$ to $\xi_i$ in $\CC_{\e}^{-2-\kappa}$ in probability.
One can then prove that, for any $R>0$, in probability
\begin{equ}
%\lim_{\e\to 0} \|  \xi_i^{\e,\bar\e} - \xi_i^{0,\bar\e} \|_{L^\infty([-R,R],L^\infty_\e)} 
%\eqdef
\lim_{\e\to 0} \sup_{(t,x)\in [-R,R]\times \obonds_i} |\xi_i^{\e,\bar\e}(t,x) - \xi_i^{0,\bar\e} (t,x)| =0\;.
\end{equ}
This follows from the bound
$\|\chi^{\bar\e}-\chi^{\eps,\bar\eps}\|_{\CC^k_\e} \lesssim \e$ uniformly in $\e$
for any $k\ge 0$ and moment bounds on 
$\xi^\e$ in $\CC^{-2-\kappa}_{\e}$ and 
$\xi^{\e,\bar\e}_i$ in $\CC([-R,R]\times\obonds_i)$
uniformly in $\e$, see for instance a similar argument in \cite[Section~5]{EH21}.
Then~\eqref{eq:AeeC_A0eC} follows from convergence of initial conditions~\eqref{eq:aee_a},
the a priori bound~\eqref{eq:A0eC_apriori}, and standard results in numerical analysis (e.g.  \cite{MR2895081}).
 \end{proof}
The next lemma controls the fourth term in~\eqref{eq:diagonal_bound}.
Recall that $\CB^\e$ solves~\eqref{e:fix-pt-Be} 
with respect to the renormalised model 
$\hat Z^{\e}=(\hat\Pi^{\e},\hat\Gamma^{\e})$
and with $c^\e_{\star}$ taken as in Lemma~\ref{lem:renormalised-equ}.
Recall that the solution is given by Proposition~\ref{prop:sol-abs} with
\[
\CB^\e_i = \CP^{i;\e}\bone_+ \bXi_i + G^{i;\e} a^\e_i + \tilde \CB^\e_i \;, \qquad 
\tilde \CB^\e_i \in  \cD_{0,\e}^{\gamma,2\eta+1}\;,
\]
where the initial condition is $a^\e_i$ (recall Remark~\ref{rem:a_vs_ea} on $a^\e_i$ vs. $a^\e$).

\begin{lemma}\label{lem:4th_term}
Let $K>0$ and $\tau(\e,K)\eqdef K\wedge\inf\{t>0\,:\,|B^\e(t)|_{N;\alpha}>K\}$.
Then
$
\lim_{\e\to 0} \|B^\e - A^{(N)}\|_{\CC([0,\tau(\e,K)],\Omega_{N;\alpha})} = 0
$
in probability.
\end{lemma}

\begin{proof}
Recall the uniform in $\eps>0$ bounds on the models $(\hat\Pi^\e,\hat\Gamma^\e)$ in Proposition~\ref{prop:Ze_moments},
the bounds on the discrete SHE $\Psi^\e$ from Proposition~\ref{prop:SHE}, and the fact that $e^{t\Delta_\e}$ is a contraction on $\Omega_{N;\alpha}$.
Recall further that $A^{(N)}$, by Lemma~\ref{lem:renormalised-equ}, is the reconstruction $\hat\CR\CA^\e$ of $\CA^\e$ solving~\eqref{e:fix-pt-Ae} with our choice of $c^\e_\star$.
Finally, note that we are in the scope of Proposition~\ref{prop:sol-abs} (see in particular Lemma~\ref{lem:Psi_Wick} that verifies Assumption~\ref{as:models} for $\Psi^\e$).
The conclusion follows from the uniform bound on initial conditions~\eqref{eq:sup_a_N},
from Proposition~\ref{prop:sol-abs} (in particular~\eqref{eq:CA_CB_close}) upon taking $\eta=\alpha-1$ therein
so that $3\eta+1 > \alpha/2$ by the assumption $\alpha>\frac45$,
and from the uniform embeddings $\CC^{\alpha/2}_\e \hookrightarrow \Omega_{N;\alpha}\hookrightarrow \CC^{\alpha-1}_\e$
by~\eqref{eq:embeddings}.
\end{proof}
It remains to control the third term in~\eqref{eq:diagonal_bound}, which requires a special choice of $C\in L(\mfg^2,\mfg^2)$.
For the rest of this subsection, $\gamma,\eta$ are chosen as in  Proposition~\ref{prop:sol-abs} such that furthermore
$\eta\in (\frac\alpha4-\frac12,\alpha-1)$; such $\eta$ exists since $\alpha>\frac23$.
(The condition $\eta<\alpha-1$ is needed to apply Lemma~\ref{lem:A_pi_N_A} below
while $\frac\alpha4-\frac12<\eta$ is used in the proof of Lemma~\ref{lem:3_term}.)

Remark that~\eqref{e:fix-pt-Be} has no ``error'' terms (which is the point of introducing it) and thus $\CB^\e$ take values in the sector $\CT^\YM\subset \CT$.
Since~\eqref{e:fix-pt-Be} is the only fixed point problem that we analyse in what follows, we make the following convention.

\begin{notation}\label{not:model_restr}
For the rest of this section, all models are assumed to be defined on $\CT^\YM$ only.
In particular, all model-dependent quantities, e.g. $\$\hat Z^\e\$^{(\e)}_{\gamma;\K}$, $\$\hat Z^\e;\hat Z^{\e,\bar\e}\$^{(\e)}_{\gamma;\K}$,
are defined in terms of $\CT^\YM$.
\end{notation}
We now define a random discrete model $\hat Z^{\e,\bar\e} = (\hat\Pi^{\e,\bar\e},\hat\Gamma^{\e,\bar\e})$ on $\CT^\YM$ in the same way
as $(\hat\Pi^{\e},\hat\Gamma^{\e})$, but replacing each instance of $\xi^\e$ by $\xi^{\e,\bar\e}$, 
and each ``delta function'' on $\R\times \obonds_i$
%$\begin{tikzpicture} \node at (0,0)  [var] (a)  {};  \draw[thick]   (-.2, 0) -- (.2,0) ; \end{tikzpicture}$
%and 
%$\begin{tikzpicture}[baseline=-3] \node at (0,0)  [var] (a)  {};  \draw[thick]   (0, -.2) -- (0,.2) ; \end{tikzpicture}$
by the kernels
 $\langle \chi^{\e,\bar\e} (z-\cdot ), \chi^{\e,\bar\e} (-\cdot)\rangle_{L^2(\R\times \obonds_i)}$
with $i=1,2$ respectively.

More precisely,
each $\xi^\e$ in \eqref{e:Pi-noise}
is replaced by $\xi^{\e,\bar\e}$;
the renormalisation map $M^{\e,\bar\e}$ is defined as for $M^{\e}$ in Section~\ref{subsec:renorm_group}
with
\begin{equ}[e:C-e-bare]
\bar{C}^{\e,\bar\e}_{\be\be'}\;,
\;\; \hat{C}_{1,\be}^{\e,\bar\e}\;,
\;\; \hat{C}_{2,\be}^{\e,\bar\e}\;,
\;\; \hat{C}_{3,\be,\bw}^{\e,\bar\e}\;,
\;\; \hat{C}_{4}^{\e,\bar\e}\;,
\;\;
\hat{C}_{5}^{\e,\bar\e}\;,\;\;
\tilde C_\be \in \mfg^{\otimes 2}
\end{equ}
defined in the same way 
as in 
Section~\ref{sec:renorm_constants}
except that $\xi^\e$ is now replaced by $\xi^{\e,\bar\e}$ and
$\Psi$ is now defined as $\Psi=K^\e * \xi^{\e,\bar\e}$.
Following Notation~\ref{not:model_restr},
we ignore all other renormalisation constants in~\eqref{eq:all_C}
because they only affect the action of $M^{\e}$ on $\CT^\rem$, which has trivial intersection with $\CT^\YM$.
We then define
$C_{\sym}^{\e,\bar\e}$
as in \eqref{e:def-C16}-\eqref{e:def-CSYM}
but using the corresponding constants \eqref{e:C-e-bare}.

\begin{lemma}
$C_{\sym}^{\e,\bar\e} = O(1)$ uniformly in $\e\le\bar\e\in(0,1)$.
\end{lemma}

\begin{proof}
This follows  analogously as Lemma~\ref{lem:C_sym-finite}.
Indeed, it is easy to see from the proof of Lemma~\ref{lem:ePsiDPsi} that \eqref{eq:PsiDPsi} remains true 
with $\Psi$ therein replaced by $K^\e * \xi^{\e,\bar\e}$, uniformly  in $\e,\bar\e\in(0,1)$.
The claims about 
$\bar{C}^{\e}$, $\hat{C}^{\e}_k$
in the proof of Lemma~\ref{lem:C_sym-finite} are then still valid for $\bar{C}^{\e,\bar\e}$, $\hat{C}^{\e,\bar\e}_k$
with 
$\bar{C}^\e_{\approx}$  and   $\hat C^\e_{\approx}$
replaced by 
\begin{equs}
 \bar{C}^{\e,\bar\e}_{\approx} &\eqdef \Cas\int_{\R\times (\e\Z)^2} 
K^{\e,\bar\e}(z) P^{\e,\bar\e}(z) \mrd z \;,
\\
\hat C^{\e,\bar\e}_{\approx} & \eqdef
 \Cas\int_{\R\times (\e\Z)^2} K^\e (z) (\partial^+_j P^{\e,\bar\e})^{*2}(z)\mrd z 
 \end{equs}
 respectively,
 where $ P^{\e,\bar\e}= P^\e *\chi^{\eps,\bar\eps}$.
 Now by Lemma~\ref{lem:identity}, one has 
 $ (\partial^+_j P^{\e,\bar\e})^{*2}=\frac14 P^{\e}(|\cdot|,\cdot) *(\chi^{\e,\bar\e})^{*2}$.
%where $(\cdots)^{*2}$ is again the convolution of a function with its reflection.
Writing $P^{\e}(|t|,x) = P^{\e}(t,x)1_{t\ge 0} + P^{\e}(-t,-x)1_{t<0}$, we have 
$
\hat C^{\e,\bar\e}_{\approx} = \frac14 \bar{C}^{\e,\bar\e}_{\approx}
 +  \frac14 \Cas\int_{\R\times (\e\Z)^2} P^{\e,\bar\e}(z) (P^{\e}*\chi^{\e,\bar\e}(-\cdot)) (-z)\mrd z
$
and the last term is $O(1)$ since the integrand is bounded by $(\|z\|_\s + \bar\e)^{-4}$ but supported on $\{z=(t,x):|t|\le \bar\e^2\}$.
\end{proof}
\begin{definition}\label{def:CAee}
For $C\in L(\mfg^2,\mfg^2)$, let
\begin{equ}
\CA^{\e,\bar\e}_i = \bone_+ \bPsi_i + \CV^{\e,\bar\e}_i 
=\CP^{i;\e,\bar\e}\bone_+\bXi_i + G^{i;\e}a^{\e,\bar\e}_i+\tilde\CA^{\e,\bar\e}_i  \in\cD^\sol\otimes \mfg
\end{equ}
be the solution to the fixed point problem
\begin{equs}\label{e:fix-pt-Aee}
\CA^{\e,\bar\e}_i &= 
\CP^{i;\eps} \bone_+\Big(
\Big[
\cS_{a} \CA_j^{\e,\bar\e} ,
2\mcb{D}_j \CA^{\e,\bar\e}_i   - \bar{\mcb{D}}_i \bPsi_j -  \hat{\mcb{D}}_i \CV^{\e,\bar\e}_j
\Big]
\\
&\quad+ [\CA^{\e,\bar\e}_i, \mcb{D}_i \CA^{\e,\bar\e}_i ]
+[\cS_{a} \CA_j^{\e,\bar\e} ,[\cS_{a} \CA_j^{\e,\bar\e} ,\CA_i^{\e,\bar\e}]]
\\
&+ C_i\CA^{\e,\bar\e}_i -\ad_{C^{\e,\bar\e}_{\sym}}\CA^{\e,\bar\e}_i
+ \bXi_i
\Big)
 -\CL_1 P^{i;\e} *_{(i)} \tilde \delta \hat\CR^{\e,\bar\e} \CA^{\e,\bar\e}_i
+ G^{i;\eps} a_i^{\e,\bar\e}
\end{equs} 
with underlying model $\hat Z^{\e,\bar\e}=(\hat\Pi^{\e,\bar\e},\hat\Gamma^{\e,\bar\e})$
and where $\hat\CR^{\e,\bar\e}$ is the reconstruction operator for $\hat Z^{\e,\bar\e}$
and, for $A\in \mfq_i$, we define $\tilde \delta_i A\in\mfq_i$ by
the first two lines of \eqref{eq:delta_A_def} with
$\hat{C}^{\e}_{1,2,\be}$, $\hat{C}^{\e,}_{3,\be,\bw}$
replaced by $\hat{C}^{\e,\bar\e}_{1,2,\be}$, $\hat{C}^{\e,\bar\e}_{3,\be,\bw}$.
%\begin{equs}[eq:tilde_delta_A_def]{}
%&(\tilde \delta_i A)(e)
%\eqdef 
% \frac{1}{8} \!\! \sum_{\be,\be'\in \CE_\times} \!\!   
% (\ad_{\hat{C}^{\e,\bar\e}_{1,\be}} +\ad_{\hat{C}^{\e,\bar\e}_{2,\be}})
%  (A(e+\be+\be')-A(e))
%\\
%&\quad
%-\frac{1}{16} \sum_{\be,\be'\in \CE_\times}\sum_{\bw\in\{\Northwest,\Southwest\}} \ad_{\hat{C}^{\e,\bar\e}_{3,\be,\bw}}(A(e^\bw + \be')-A(e))\;.
%\end{equs}
\end{definition}
The only difference between~\eqref{e:fix-pt-Aee} and~\eqref{e:fix-pt-Be}
is that $c^\e_\star$ is replaced by $C_i-\ad_{C^{\e,\bar\e}_{\sym}}$
and the additional term $-\CL_1 P^{i;\e} *_{(i)} \tilde \delta \hat\CR^{\e,\bar\e} \CA^{\e,\bar\e}_i$ in~\eqref{e:fix-pt-Aee}.
The existence and uniqueness of $\CA^{\e,\bar\e}$ follows in the same way as Proposition~\ref{prop:sol-abs}
and we have $\tilde\CA^{\e,\bar\e} \in \cD^{\gamma,2\eta+1}_{0,\e}$.
Define $A^{\e,\bar\e} \eqdef \hat\CR^{\e,\bar\e}\CA^{\e,\bar\e}$.
The proof of the following lemma is similar to (and simpler than) that of Lemma~\ref{lem:renormalised-equ}
(see in particular \eqref{e:Ae-renormalised-mass}).

\begin{lemma}\label{lem:Aee_reconstruct}
For any $C\in L(\mfg^2,\mfg^2)$,
$A^{\e,\bar\e}=A^{\e,\bar\e}_{C}$,
i.e.
$A^{\e,\bar\e}$ solves
\eqref{e:Aeps-eps-bar}.
\end{lemma}
We define for the rest of this section
\begin{equ}[e:set-final-C]
c^0_\star = \lim_{\e\to0} c^\eps_\star\;,
\quad
C^{0,\bar\e}_{\sym} = \lim_{\e\to 0} 
C^{\e,\bar\e}_{\sym}\;,
\quad
C =
c^0_\star + \lim_{\bar\e\to 0}\ad_{C^{0,\bar\e}_{\sym}}\;,
\end{equ}
where the limits are understood in the following way: for each $\bar\e$, we take a subsequence $\e\downarrow0$ in $(0,\bar\e)$ (depending on $\bar\e$)
such that $C^{0,\bar\e}_{\sym}=\lim_{\e\downarrow0} C^{\e,\bar\e}_{\sym}$ exists,
and since $C^{0,\bar\e}_{\sym}=O(1)$ in $\bar\e$,
we further take a subsequence $\bar\e\downarrow0$ such that $\lim_{\bar\e\downarrow 0}C^{0,\bar\e}_{\sym}$ exists.
(Recall that $c^0_\star=\lim_{\e\downarrow0}c^\eps_\star$ exists by Assumption~\ref{as:subsequence}.)
By choosing another subsequence of 
each $\bar\e$-dependent sequence $\e\downarrow0$, we further suppose that
\begin{equ}[eq:unif_diff_lims]
\sup_\e
\{|C^{\e,\bar\e}_{\sym} - C^{0,\bar\e}_{\sym}|
+|c^\e_\star - c^0_\star|\}
\leq \bar\e^\kappa\;.
\end{equ}

%The final ingredient in the proof of Theorem~\ref{thm:some_C} is the following lemma that controls the third term in~\eqref{eq:diagonal_bound}.

\begin{lemma}\label{lem:3_term}
Let $K>0$ and $\tau(\e,K)$ as in Lemma~\ref{lem:4th_term}.
Then, for any $c>0$,
\begin{equ}
\lim_{\bar\e\downarrow 0} \sup_{\e} \P\big[\|B^\e - A^{\e,\bar\e}_C\|_{\CC([0,\tau(\e,K)],\Omega_{N;\alpha})} > c\big] = 0\;,
\end{equ}
where $\bar\e\downarrow0$ is the subsequence taken in the definition of $C$ following~\eqref{e:set-final-C}
and $\sup_{\e}$ is over the corresponding $\bar\e$-dependent subsequence $\e\downarrow0$ .
\end{lemma}

\begin{proof}
Consider the discrete SHE and regularised discrete SHE
\begin{equ}[eq:Psie_Psiee_def]
\Psi^\e_i=\hat\CR\CP^{i;\e}\bone_+\bXi_i + G^{i;\e} a^\e_i\;,\quad
\Psi^{\e,\bar\e}_i=\hat\CR^{\e,\bar\e}\CP^{i;\e,\bar\e}\bone_+\bXi_i+ G^{i;\e} a^{\e,\bar\e}_i
\end{equ}
with 
initial conditions $a^\e_i$ and $a^{\e,\bar\e}_i$ respectively.
By Lemma~\ref{lem:Psie_Psiee} below,
\begin{equ}
\lim_{\bar\e\to0}\sup_{\e}\P[\|\Psi^\e-\Psi^{\e,\bar\e}\|_{\CC([0,R],\Omega_{N;\alpha})} > c] = 0
\end{equ}
for any $R>0$,
where we also used convergence of initial conditions~\eqref{eq:ae_aee} and that  $e^{t\Delta_\e}$ is a contraction on $\Omega_{N;\alpha}$.
To control the remainders,
by the same argument as the first part of the proof of Proposition~\ref{prop:sol-abs},
\begin{equs}[e:Bee-Be]
\$ \tilde \CA^{\e,\bar\e};  \tilde\CB^\e  \$_{\gamma,2\eta+1;T}^{(\e)} 
&\lesssim
\$ \hat Z^{\e,\bar\e} ; \hat Z^\e  \$_{\gamma;O}^{(\e)} +
\|\Psi^\e-\Psi^{\e,\bar\e}\|_{\CC([-1,2],\CC^{-\kappa}_\e)}
\\
&\quad +
\|a^{\e,\bar\e}-a^\e\|_{\CC^\eta_\e} + O(\bar\e^\kappa)\;,
\end{equs}
where the proportionality constant and $T>0$ are uniform in the size of models, $\Psi^\e,\Psi^{\e,\bar\e}\in \CC([-1,2],\CC^{-\kappa}_\e)$, and initial conditions in $\CC^\eta_\e$.
Here $O=[-1,2]\times \T^2$ and
$\$ \hat Z^{\e,\bar\e} ; \hat Z^\e  \$_{\gamma;O}^{(\e)}$
is the distance between two {\it discrete} models defined in Section~\ref{sec:Models}.
Furthermore $O(\bar\e^\kappa)$ accounts for the difference between $C-\ad_{C^{\e,\bar\e}_{\sym}}$
and $c^\e_\star$, which is $O(\bar\e^\kappa)$ by~\eqref{eq:unif_diff_lims},
and for the extra `$\tilde \delta$ term' which appears in~\eqref{e:fix-pt-Aee}, which is
of order $\e^\kappa$
due to~\eqref{eq:delta_term_vanish}
and the bound $|\hat C^{\e,\bar\e}_{j,\be}|+|\hat C^{\e,\bar\e}_{3,\be,\bw}| \lesssim \e^{-\kappa}$.

By the condition $2\eta+1>\frac\alpha2 (\Leftrightarrow \eta>\frac\alpha4-\frac12)$
and by continuity of the reconstruction operator, it follows that
$\|\hat\CR \tilde\CB^\e - 
\hat\CR^{\e,\bar\e}\tilde\CA^{\e,\bar\e}\|_{\CC([0,T],\CC^{\alpha/2}_\e)}$
is bounded by a multiple of the right-hand side of~\eqref{e:Bee-Be}
uniformly over the same data.

By Lemma~\ref{lem:ZeeZe} below, for any $R>0$, one has 
$\lim_{\bar\e\to 0}\$ \hat Z^{\e,\bar\e} ; \hat Z^\e  \$_{\gamma;[-R,R]\times \T^2}^{(\e)} =0$
in probability uniformly in $\e\in (0,\bar\e)$.
Furthermore, by Proposition~\ref{prop:Ze_moments},
$\$\hat Z^\e\$^{(\e)}_{\gamma;[-R,R]\times\T^2}$ has every moment bounded uniformly in $\e$.

It follows from the embedding $\CC^{\alpha/2}_\e\hookrightarrow\Omega_{N;\alpha}$ by~\eqref{eq:embeddings}
and from the definition of $\tau(\e,K)$,
that we can iterate the bound~\eqref{e:Bee-Be} and the corresponding bound on the reconstructions
in a standard way to see that $B^\e$ and $A^{\e,\bar\e}= \hat\CR^{\e,\bar\e}\CA^{\e,\bar\e}$
are close in $\Omega_{N;\alpha}$ over $[0,\tau(\e,K)]$.
Finally, $A^{\e,\bar\e}=A^{\e,\bar\e}_C$
by Lemma~\ref{lem:Aee_reconstruct}.
\end{proof}
In the above proof, we used the following two lemmas.
\begin{lemma}\label{lem:Psie_Psiee}
Let $\Psi^\e$ and $\Psi^{\e,\bar\e}$ be defined by~\eqref{eq:Psie_Psiee_def} with $a^\e=a^{\e,\bar\e}=0$.
Then, for any $R,c>0$,
%\begin{equ}[eq:Psie_Psiee_conv]
one has
$\lim_{\bar\e\to0}\sup_{\e\in (0,\bar\e)}
\P[\|\Psi^\e-\Psi^{\e,\bar\e}\|_{\CC([0,R],\Omega_{N;\alpha})} > c] = 0$.
%\end{equ}
\end{lemma}

\begin{proof}
This can be shown in a way similar to the continuum case~\cite[Cor.~4.14]{CCHS_2D}, but we give a slightly different and shorter proof.
Let $P^{i;\e}\colon [0,\infty)\times\obonds_i \to\R$ be the discrete heat kernel
and let $\Phi^\e_i \eqdef P^{i;\e}*_{(i)}(\bone_{t\geq-1}\xi_i^\e)$ be the SHE with $0$ initial condition started at time $-1$.
Then, since $t\mapsto e^{t\Delta_\e}$ is a semi-group,
$
\Psi^{\e}(t)
= \Phi^\e(t) - e^{t\Delta_\e}\Phi^\e(0)
$
for $t\geq 0$.
Define also $\Phi^{\e,\bar\e} = \moll^{\e,\bar\e}*\Phi^\e$.
Then for $t\geq 0$
\begin{equs}
\Psi^{\e,\bar\e}(t)
&= \int_{-1}^t e^{(t-s)\Delta_\e} (\moll^{\e,\bar\e}*\xi)(s)\mrd s - e^{t\Delta_\e}\int_{-1}^0
e^{-s\Delta_\e} (\moll^{\e,\bar\e}*\xi)(s)\mrd s
\\
&=
\Phi^{\e,\bar\e}(t) - e^{t\Delta_\e}\Phi^{\e,\bar\e}(0)\;,
\end{equs}
where in the first equality we again used that $t\mapsto e^{t\Delta_\e}$
is a semi-group
and in the second equality we used $(\moll^{\e,\bar\e}*\xi)(s) = [\moll^{\e,\bar\e}*(\bone_{t>-1}\xi)](s)$
for $s>-1+\bar\e^2$.
Hence
\begin{equs}
\|\Psi^\e-\Psi^{\e,\bar\e}\|_{\CC([0,R],\Omega_{N;\alpha})}
&\leq 
\sup_{t\in[0,R]}\{|\Phi^\e(t)-\Phi^{\e,\bar\e}(t)|_{N;\alpha}\\
&\qquad\qquad+ |e^{t\Delta_\e}(\Phi^\e(0)-\Phi^{\e,\bar\e}(0))|_{N;\alpha}\}
\\
&\leq \sup_{t\in[0,R]}2|\Phi^\e(t)-\Phi^{\e,\bar\e}(t)|_{N;\alpha}\;,
\end{equs}
where in the second inequality we used that
$e^{t\Delta_\e}$ is a contraction on $\Omega_{N;\alpha}$.
By the H\"older-in-time bound in Proposition~\ref{prop:SHE}
and the same argument as in~\cite[Cor.~4.4]{CCHS_2D},
we obtain
$
\lim_{\bar\e\to0}\sup_{\e\in (0,\bar\e)}
\P[\|\Phi^\e-\Phi^{\e,\bar\e}\|_{\CC([0,R],\Omega_{N;\alpha})} > c] = 0$.
\end{proof}

\begin{lemma}\label{lem:ZeeZe}
There exists $\theta>0$ such that, for all compact $\K\subset \R\times\T^2$,
one has
$ \E \VERT \hat Z^{\e,\bar\e} ; \hat Z^\e  \VERT_{\gamma;\K}^{(\e)}
\lesssim \bar\e^\theta$ uniformly in $\e \in (0,\bar\e)$. 
\end{lemma}

\begin{proof}
As in the proof of Proposition~\ref{prop:Ze_moments}, it suffices to show 
%that for every $\tau\in \CT^{(i)}_-\cap\CT^\YM$ and $z_*=(t_*,x_*)\in\R\times\obonds_i$,
%\begin{equs}[e:mom-diff-bounds]
%\E[(\hat\Pi_{z_*}^\e  - \hat\Pi_{z_*}^{\e,\bar\e}) (\tau)(\phi_{z_*}^{\lambda})^2]
%\lesssim 
%\bar\e^{2\theta}
%\lambda^{2(|\tau|+\kappa-\theta)}\;,
%\\
%\E\Big[\Big(\int (\hat\Pi_{z_*}^\e   - \hat\Pi_{z_*}^{\e,\bar\e})   (\tau)(s,x)(\CS_{2,t_*}^{\lambda}\phi)(s)\mrd s\Big)^2\Big]
%\lesssim 
%\bar\e^{2\theta}
%\lambda^{2(|\tau|+\kappa-\theta)}\;,
%\end{equs}
%uniformly over $\e > 0$ and $\bar\e\in (\e,1)$. 
%The first bound is  uniform in $\lambda\in (\e,1]$ 
%and all space-time test functions $\phi$ as in~\eqref{eq:Pi},
%and the second bound is uniform in $\lambda\in (0,\e]$ and all temporal  functions $\phi$ as in~\eqref{eq:seminorm}.
\eqref{e:mom-bounds} but with bounds $\bar\e^{2\theta}\lambda^{2(|\tau|+\kappa-\theta)}$ on the right-hand sides,
and with $\hat\Pi_{z_*}^\e$ replaced by $\hat\Pi_{z_*}^\e  - \hat\Pi_{z_*}^{\e,\bar\e}$.
To prove this,
%To prove~\eqref{e:mom-diff-bounds}, 
by \cite[Lemma~7.5]{HM18} one can bound
$\VERT K^\e - K^\e \ast_\eps \psi^{\e,\bar\e} \VERT_{-2 - \kappa; m}$
 by $\bar{\eps}^\theta \VERT K^\e \VERT_{-2; m + 2}$
uniformly in $\e\le \bar\e$.
The proof then follows in exactly the same way as the proof of  \eqref{e:mom-bounds},
except that we can now extract a small positive power of $\bar{\eps}$ for each %element of our regularity structure 
$\tau\in \CT^{(i)}_-\cap\CT^\YM$
with at least one integration map.
The bounds for
$\<Xi>$
follow easily (e.g. by similar argument as \cite[Prop.~4.1]{EH21}).
\end{proof}

\begin{proof}[of Theorem~\ref{thm:some_C}]
Let $c>0$, $K_{\bar\e}$ as in Lemma~\ref{lem:1-2_terms}, and $a^{\e,\bar\e}=a^\e$. Note that~\eqref{eq:aee_a} holds due to $\eta<\alpha-1$ and Lemma~\ref{lem:A_pi_N_A}.
By~\eqref{eq:AC_A0eC}, for $\bar\e$ small,
\begin{equ}[eq:AC_A0eC_c]
\P[\|A_C - A^{0,\bar \e}_C\|_{\CC([0,\tau(\bar\e)],\Omega^1_\alpha)} > c ]<c
\end{equ}
for $\tau(\bar\e)$ as in Lemma~\ref{lem:1-2_terms}.
Furthermore, by Lemma~\ref{lem:3_term},
after possibly decreasing $K_{\bar\e}$ but such that still $\lim_{\bar\e\downarrow0}K_{\bar\e}=\infty$,
we have
\begin{equ}[eq:B_AC_c]
\sup_\e \P[\|B^\e - A^{\e,\bar\e}_C\|_{\CC([0,\tau(\e,K_{\bar\e})],\Omega_{N;\alpha})} > c] < c
\end{equ}
for $\tau(\e,K_{\bar\e})$ and $\sup_\e$ are as in Lemma~\ref{lem:3_term}.
Next, by~\eqref{eq:AeeC_A0eC} and the embedding
$\CC^{\alpha/2}_\e\hookrightarrow\Omega_{N;\alpha}$,
we have for $\e$ sufficiently small
\begin{equ}[eq:AeeC_A0eC_c]
\P\big[
\|  A^{\e,\bar\e}_C - \pi_N A^{0,\bar\e}_C\|_{\CC^{\tau(\bar\e)}_{\eta-\alpha/2}(\Omega_{N;\alpha})}
>c \big] < c\;.
\end{equ}
Finally, by Lemma~\ref{lem:4th_term}, again for $\e$ sufficiently small,
\begin{equ}[eq:B_AN_c]
\P\big[\|B^\e - A^{(N)}\|_{\CC([0,\tau(\e,K)],\Omega_{N;\alpha})}  > c \big] < c\;.
\end{equ}
%It follows that
%\begin{equ}
%\P[\|A^{(N)} - \pi_N A_C\|_{\CC([0,\tau(\bar\e) \wedge \tau(\e,K_{\bar\e})],\Omega_{N;\alpha})} > 4c] < 4c\;.
%\end{equ}
Finally, since~\eqref{eq:AC_A0eC_c},~\eqref{eq:B_AC_c},~\eqref{eq:AeeC_A0eC_c},~\eqref{eq:B_AN_c} control the differences between all four terms,
by choosing $\bar\e \downarrow0$ and $\e=2^{-N}\downarrow0$  as above,
it follows that there exists an increasing sequence $M_N$, with $\lim_{N\to\infty}M_N=\infty$,
such that, for all $c>0$,
\begin{equ}[eq:T_N_conv]
\lim_{N\to\infty}\P[\|A^{(N)} - \pi_N A_C\|_{\CC([0,T_N],\Omega_{N;\alpha})} > c] =0\;,
\end{equ}
where $T_N\eqdef M_N\wedge \inf\{t>0\,:\, |A_C(t)|_\alpha > M_N\}$.
Clearly $T_N$ is increasing in $N$ and $\lim_{N\to\infty} T_N = T^*$, the blow-up time of $A_C$.

It remains to show that $C\in L_G(\mfg^2,\mfg^2)$.
This follows by tracking all the contributions in the definition of $C$.
Indeed, by \eqref{e:set-final-C},
$C = c^0_\star + \lim_{\bar\e\to 0}\ad_{C^{0,\bar\e}_{\sym}}$
where
$c^0_\star = \lim_{\e\to0} c^\eps_\star$
and
$C^{0,\bar\e}_{\sym} = \lim_{\e\to 0}  C^{\e,\bar\e}_{\sym}$.
Recall that 
$C^{\e,\bar\e}_{\sym}$ is a linear combination of the renormalisation operators in 
\eqref{e:C-e-bare}, which are 
defined as in 
Section~\ref{sec:renorm_constants} with $\xi^\e$ replaced by $\xi^{\e,\bar\e}$ and
 $\Psi$ replaced by $K^\e * \xi^{\e,\bar\e}$;
they are all linear multiples of the quadratic Casimir $\Cas \in \mfg^{\otimes 2}$, and $\ad_{\Cas}$ commutes with the adjoint action of $G$, see e.g. \cite[Rem.~6.8]{CCHS_2D}.

It remains to consider $c^0_\star$.
Recall that, by \eqref{e:specify-c_star},
$c^\e_\star = - \ad_{C^\e_{\sym}} - \ad_{C^\e_{\rem}} + \tilde{c}_\star^\e $.
Here, the operators $C^\e_{\sym}$ and  $C_{\rem}^\e$ 
 are defined in  \eqref{e:def-CSYM} and \eqref{e:def-CREM}, all the terms in these definitions being
again multiples of the quadratic Casimir.

It only remains to consider $\tilde{c}_\star^\e $, which by 
Definition~\ref{def:tildec}
is the sum  
of the maps $c^\e$
in Lemmas
\ref{lem:I3} and
\ref{lem:grad-R-hat}. 
These maps are defined in terms of $c_{13},c_{23}$ appearing in Lemma~\ref{lem:Psi_Wick}, which, following its proof, are all of the form of the expectations in Lemma~\ref{lem:ePsiDPsi},
so $c_{13},c_{23}$ are multiples of the quadratic Casimir.
It now follows that each component $c^{\eps,(i)}_j$ of $c^\eps\in L(\mfg^2,\mfg^2)$ in Lemmas \ref{lem:I3} and
\ref{lem:grad-R-hat} are of the form
\[
c^{\eps,(i)}_j(X) = T(e_i\otimes e_i \otimes X)
\]
where $\Cas = e_i\otimes e_i$ for an orthonormal basis $e_i$ of $\mfg$ and implcit summation over $i$, and where $T\in L(\mfg^{\otimes 3})\to\mfg$ satisfies
\[
\Ad_g T(X_1\otimes X_2\otimes X_3) = T(\Ad_g X_1\otimes \Ad_g X_2\otimes \Ad_g X_3)\;.
\]
To see this, recall that $\Cas$ is $\Ad_G$ invariant.
Then the claimed form of $c^\eps$ follows for Lemma \ref{lem:I3} from the fact that $I_3$ is in $\CI_3$ and thus $c^\eps$ is a multiple of $[e_i,[e_i,X]]$,
while for Lemma \ref{lem:grad-R-hat} from the $\Ad_G$ covariance of $E^{(2)}_\e$ in Assumption \ref{assump:R} which implies that $\Ad_g E^{(2)}(X_1\otimes X_2\otimes X_3)^* = E^{(2)}(\Ad_g X_1\otimes \Ad_g X_2\otimes \Ad_g X_3)^*$.
(Remark that the map in 
\eqref{e:def-C-four} is an element of $\mfg^{\otimes 4}$ but 
by the argument below \eqref{eq:F_ren_eq} it does not contribute
to the renormalisation operator $C$.)
\end{proof}

\section{Identification of the limit}
\label{sec:gauge-covar}

In this section, we prove Theorem~\ref{thm:discrete_dynamics}.
We have shown in Theorem~\ref{thm:some_C} that the discrete dynamics converge
(along subsequences) to the continuum dynamic introduced in~\cite{CCHS_2D} with \textit{some} mass renormalisation $CA$ with $C\in L_G(\mfg^2,\mfg^2)$.
It remains to show that, if $\moll$ is non-anticipative,
then $C=\bar C$ for $\bar C\in L_G(\mfg,\mfg)$ from the gauge covariance theorems~\cite[Theorems~2.9,~2.13]{CCHS_2D}.
A potential way to show this is to compute all the finite mass shifts from the error terms,
i.e. find the $O(1)$ constants in the lemmas of Sections~\ref{subsec:Mass renormalisation_1},~\ref{sec:renorm_constants}, and~\ref{sec:diagonal_argument}.
However, this would require numerous computations that seem challenging by hand.

Our strategy instead is to show that $\bar C$ is the \textit{unique} element in $L_G(\mfg^2,\mfg^2)$ that makes the continuum dynamic gauge covariant,
which in addition answers a question left open in~\cite{CCHS_2D}, see Remark~\ref{rem:uniqueness_compare}.
% (in a sense much stronger than the uniqueness of $\bar C$ in~\cite[Thm~2.9]{CCHS_2D}).
In fact, we prove a quantitative version of this statement in Theorem~\ref{thm:A_tilde_A}.
This uniqueness of $\bar C$ allows us to show $\bar C = C$ since the lattice dynamic similarly preserves gauge covariance in the limit.
%We give the proof of Theorem~\ref{thm:discrete_dynamics} in Section~\ref{subsec:proof_C=bar_C}.
We believe this strategy also has a higher chance of generalising to 3D for which a direct computation of the mass shifts seems even more challenging.

With the exception of Section~\ref{subsec:proof_C=bar_C}, where we prove Theorem~\ref{thm:discrete_dynamics},
the results of this section are independent of those of Sections~\ref{sec:lattice_dynamics} - \ref{sec:diagonal_argument}.

\subsection{Quantitative gauge-covariance}

We fix throughout this section a non-anticipative mollifier $\moll$ and let $\bar C \in L_G(\mfg,\mfg)$ be the unique element from~\cite[Thm.~2.9(i)]{CCHS_2D}.
%
%\begin{remark}
%If $\mfg$ is simple, then $\bar C = \lim_{\eps\downarrow0} \lambda \int \moll^\eps (z)(K*K^\eps)(z)\mrd z$ where $\lambda\in\R$ is the Casimir in the adjoint representation.
%Equivalently, one can define $\check C$ like in~\cite[Eq.~(2.13)]{Chevyrev22YM}
%and consider the equation for $A$ with bare mass $\check C$, so that $\bar C = \check C+ \lim_{\eps\downarrow0} C^\eps_{{\tiny\mathrm{SYM}}}$.
%\end{remark}
%
Fix $\alpha\in(\frac12,1)$, a white noise $\xi$, and for $C\in L(\mfg^2,\mfg^2)$ and $a\in\Omega^1_{\alpha}$, recall the process $\{\SYM_t(C,a)\}$ from Section~\ref{subsec:main_results}.
The following is a detailed version of Theorem~\ref{thm:C_unique}.

\begin{theorem}\label{thm:A_tilde_A}
Let $C\in L_G(\mfg^2,\mfg^2)$ with $\bar C \neq C$.
There exists a loop $\ell\in\CC^\infty(S^1,\T^2)$ with the following property.
There exist $K,t^*,\sigma>0$, depending only on $\moll, C, G$,
such that, for all $t\in (0,t_*)$
there exists $\tilde g\in\CC^\infty(\T^2,G)$ with $|\tilde g|_{\CC^3}<K$ such that
$
|\E W_\ell(\SYM_t(C,0)) - \E W_\ell(\SYM_t(C,0^{\tilde g}))| \geq \sigma t^2
$.
\end{theorem}

\begin{remark}
The function $\tilde g$ in Theorem~\ref{thm:A_tilde_A} in general depends on $t\in (0,t_*)$. The point of $K$ is to show that the size of $\tilde g$ is bounded uniformly in $t$.
\end{remark}
Before proceeding, we state a corollary of Theorem~\ref{thm:A_tilde_A} which shows a very strong form of uniqueness for the mass renormalisation $\bar C$ for SYM in~\cite[Thm.~2.13]{CCHS_2D}.
For $C\in L(\mfg^2,\mfg^2)$ and $a\in\Omega^1_{\alpha}$, we say that a probability measure $\mu$ on functions $\R_+ \to \hat\Omega^1_{\alpha}$
is \emph{generative with mass $C$ and initial condition $a$}
if it satisfies~\cite[Def.~2.11]{CCHS_2D} with the given $C$ in item 3. therein.

\begin{corollary}\label{cor:unique_C}
For any $\bar C\neq C\in L_G(\mfg^2,\mfg^2)$, 
there exists $\tilde g\in \CC^\infty(\T^2,G)$ such that for any generative probability measures $\mu,\tilde\mu$
with mass $C$ and
initial conditions $0, 0^{\tilde g}$ respectively,
the pushforward measures $\pi_* \mu,\pi_*\tilde \mu$ on $\CC(\R_+,\hat\mfO_{\alpha})$
are not equal.
\end{corollary}

\begin{proof}
Let notation be as in Theorem~\ref{thm:A_tilde_A}.
There exists a random variable $\tau>0$ such that $\E\tau^{-p}<\infty$ for all $p\geq 1$
and
such that $\SYM_t(C,0)\neq \skull$ and $\SYM_t(C,0^{\tilde g}) \neq\skull$
for all $t\in (0,\tau)$ and all $|\tilde g|_{\CC^3} < K$ (see, e.g.~\eqref{eq:tau_inverse_bound} below).
By Markov's inequality,
$
\P[\SYM_t(C,0)=\skull] + \P[\SYM_t(C,0^{\tilde g})=\skull] \lesssim t^{\eta}
$
for any $\eta>0$.
Taking $\eta>2$, Theorem~\ref{thm:A_tilde_A} implies that for all $t$ sufficiently small, there exists $|\tilde g|_{\CC^3} < K$ such that
$
|\E W_\ell(A(t)) - \E W_\ell(\tilde A(t))| \gtrsim t^2
$,
where $A,\tilde A$ are the adapted processes as in~\cite[Def.~2.11]{CCHS_2D} associated to $\mu,\tilde \mu$ respectively.
Since $A\sim \mu$ and $\tilde A\sim\tilde \mu$ by definition, and since $W_\ell$ is a gauge-invariant function on $\hat\Omega^1_{\alpha}$, the conclusion follows.
\end{proof}

\begin{remark}
\label{rem:uniqueness_compare}
%To compare this with the uniqueness statement of~\cite{CCHS_2D},
It was shown in~\cite[Thm.~2.9]{CCHS_2D}
that $\bar C$ is the unique mass renormalisation that renders the dynamic gauge covariant under a particular coupling.
Whether other choices of $\bar C$ lead to gauge covariant dynamics (e.g. via different couplings) was left as an open question that Theorem~\ref{thm:A_tilde_A} and Corollary~\ref{cor:unique_C} answer.
\end{remark}
To prove Theorem~\ref{thm:A_tilde_A}, we reduce it to a simpler and more general statement (Proposition~\ref{prop:A_tilde_A}).
Fix for the rest of the section $C\in L(\mfg^2,\mfg^2)$.
Consider further $a\sim \bar a \in \Omega^1_{\alpha}$ with $\bar a = a^{g(0)}$ for some $g(0)\in \CC^\infty(\T^2,G)$.
Recall the heuristic notation \eqref{eq:SYM_heuristic} and that, by~\cite[Sec.~2.2]{CCHS_2D},
if $C\in L_G(\mfg^2,\mfg^2)$, then the solution $A$ to
\begin{equ}\label{eq:A_C}
\partial_t A = \Delta A + A\partial A + A^3 + \xi+ CA\;,\quad A(0)=a
\end{equ}
is pathwise gauge equivalent to $B$ that solves
\begin{equ}\label{eq:B_C}
\partial_t B = \Delta B + B\partial B + B^3 + \Ad_g\xi + CB + C\mrd g g^{-1}\;,\quad B(0)=\bar a
\end{equ}
(here and below, unless otherwise stated,
solutions to singular SPDEs like~\eqref{eq:A_C}-\eqref{eq:B_C} are understood as the $\eps\downarrow0$
limit of the solutions with $\xi$ replaced by $\xi^\eps= \moll^\eps*\xi$).
More precisely, $A^{g} = B$ up until the blow up of $(A,B,g)$,
where $g$ solves the (classically well-posed)
PDE with $B$ as a `driver'
\begin{equ}\label{eq:PDE_g}
\partial_t g = \Delta g - (\partial_j g) g^{-1} (\partial_j g) + [B_j, (\partial_j g) g^{-1}]g
\end{equ}
and initial condition $g(0)$.
Furthermore, for $\bar C$ as above (which depends on $\moll$) and any $C\in L(\mfg^2,\mfg^2)$, by~\cite[Thm.~2.9(i)]{CCHS_2D},\footnote{\cite[Thm.~2.9(i)]{CCHS_2D} considers only `diagonal' operators $C,\bar C\in L(\mfg,\mfg)$ that commute with $\Ad$, but an inspection of the proof reveals that this is not necessary, see also \cite[Thm.~1.14]{Chevyrev22YM}.}
$B$ is equal almost surely (up until blow-up) to $\bar A$ obtained as the $\eps\downarrow0$ limit of
\begin{equ}
\partial_t \bar A =\Delta \bar A + \bar A\partial \bar A + \bar A^3 +\moll^\eps* (\Ad_{\bar g} \xi) + C\bar A + (C-\bar C)\mrd \bar g \bar g^{-1}
\;,\quad \bar A(0)=\bar a\;,
\end{equ}
where $\bar g$ solves~\eqref{eq:PDE_g} but with `driver'
$\bar A$.
In the $\eps\downarrow0$ limit, clearly $(\bar A,\bar g)$, as a random variable in $(\Omega^1_{\alpha}\times \mfG^{0,\alpha})^\sol$,
is equal in law to,
\begin{equ}
\partial_t \tilde A = \Delta \tilde A + \tilde A\partial \tilde A + \tilde A^3 + \xi + C\tilde A + (C-\bar C)\mrd g g^{-1}
\;,\quad \tilde A(0)=\bar a\;,
\end{equ}
where $g$ solves~\eqref{eq:PDE_g} with `driver' $\tilde A$.
Theorem~\ref{thm:A_tilde_A} now follows from the following more general result.
Recall that we identify $\T^2$ with $[0,1)^2$.
For $c\in L(\mfg^2,\mfg^2)$, we write $c_i\in L(\mfg^2,\mfg)$ for the $i$-th component of $c$ for $i=1,2$.

\begin{proposition}\label{prop:A_tilde_A}
Consider the loop $\ell\in\CC^\infty([0,1],\T^2)$, $\ell(x)=(x,0)$,
and let $c,C\in L(\mfg^2,\mfg^2)$ with $c_1\neq 0$.
Then there exist $K,t_*,\sigma>0$, depending only on $\moll,c,C,G$, such that for all $t\in (0,t_*)$ there exists
$g(0)\in \CC^\infty(\T^2,G)$ with $|g(0)|_{\CC^3} < K$
such that
$
|\E W_\ell(A(t)) - \E W_\ell(\tilde A(t))| \geq \sigma t^2
$,
where $A,\tilde A$ solve
\begin{equs}
\partial_t A &= \Delta A + A\partial A + A^3 + \xi + CA\;, \qquad &A(0)=0\;,\label{eq:A_equ}
\\
\partial_t \tilde A &= \Delta\tilde A + \tilde A\partial \tilde A + \tilde A^3 + \xi + C\tilde A + c\mrd g g^{-1}\;,\qquad &\tilde A(0)=0\;.\label{eq:tilde_A_equ}
\end{equs}
As before, $g$ solves the PDE~\eqref{eq:PDE_g} with driver $\tilde A$ and initial condition $g(0)$
and we treat $(\tilde A, g)$ as a random variable in $(\Omega^1_{\alpha}\times \mfG^{0,\alpha})^\sol$
(in particular $\tilde A(t)=\skull$ if either $\tilde A$ or $g$ blow up before time $t$).
\end{proposition}

Remark that $c,C$ in Proposition \ref{prop:A_tilde_A} do need to be elements of $L_G$ (cf. Theorem \ref{thm:A_tilde_A}).
Moreover, $A$ and $\tilde A$ in Proposition \ref{prop:A_tilde_A} have the \emph{same} initial condition $A(0)=\tilde A(0)=0$.
In the proof of Theorem \ref{thm:A_tilde_A} below, $A$ in Proposition \ref{prop:A_tilde_A} will really correspond to $\SYM(C,0)$, i.e. the dynamic from \eqref{eq:A_C} with initial condition $\bar a = 0$,
which we compare to $\SYM(C,a)$ with $a= 0^{\tilde g}$ for suitable $\tilde g$.

\begin{proof}[of Theorem~\ref{thm:A_tilde_A}]
To prove Theorem~\ref{thm:A_tilde_A}, we assume without loss of generality that $C_1\neq \bar C_1$
and apply Proposition~\ref{prop:A_tilde_A} with $c=C-\bar C$.
To this end, let $g(0)$ be as in Proposition \ref{prop:A_tilde_A},
and take $\bar a=0$ and $a=0^{\tilde g}$ in the discussion above Proposition~\ref{prop:A_tilde_A}, where $\tilde g=g(0)^{-1}$. 
Then, with these choices of $\tilde g$ and $c$,
\begin{equs}{}
&|\E W_\ell(\SYM_t(C,0)) - \E W_\ell(\SYM_t(C,0^{\tilde g}))| 
\\
&\quad\geq
|\E W_\ell(\tilde A(t)) - \E W_\ell(\SYM_t(C,0))|-
 |\E W_\ell(\SYM_t(C,0^{\tilde g})) - \E W_\ell(\tilde A(t))|\;,
\end{equs}
where $\tilde A$ solves \eqref{eq:tilde_A_equ}.
The first term on the right-hand side is of order at least $t^2$ by Proposition \ref{prop:A_tilde_A}.
On the other hand, the second term is of order $t^\eta$ for any $\eta>2$, by the discussion above Proposition \ref{prop:A_tilde_A} 
(which shows  
gauge equivalence in law of 
$\SYM_t(C,0^{\tilde g})$ and
$\tilde A(t)$ and which is where we use that $C \in L_G(\mfg^2,\mfg^2)$), by the fact that 
Wilson loops are gauge invariant observables,
and by the argument in
the proof of Corollary~\ref{cor:unique_C} (i.e. the existence times for the above SPDEs have inverse moments of all orders).
\end{proof}
Section~\ref{subsec:proof_of_A_tilde_A} contains the proof of Proposition~\ref{prop:A_tilde_A}, which we now sketch.
The idea is to derive an Euler estimate for $A(t)$ and $\tilde A(t)$ for small times $t$
(these Euler estimates are motivated by those for rough differential equations~\cite{Davie_08,Friz_Victoir_08_Euler,FV10}).
It turns out that (see Proposition~\ref{prop:Euler_A}) 
\begin{equ}
\tilde A(t) = \Psi(t) + tch(0) + L_t(h(0)) + O_{[\mfg,\mfg]}(t^{2-}) + O(t^{5/2-})\;,
\end{equ}
where $h(0) = \mrd g(0) g(0)^{-1}$, $\Psi(t) = O(t^\beta)$ for some small $\beta>0$ and is \textit{independent} of $h(0)$, and $L_t(h(0))$ is a term linear in $h(0)$ for which $|L_t(h(0))|_{\CC^{0+}} \lesssim t^{3/2-}$.
Here $O(t^\eta)$ represent a term $\omega$, which in general depends on $h(0)$, such that $|\omega|_{\gr\alpha} \lesssim t^\eta$ uniformly in $t$,
and $O_{[\mfg,\mfg]}$ denotes a term taking values in the derived Lie algebra $[\mfg,\mfg]$.
The same expansion holds for $A(t)$ but with $h(0)=0$ and crucially with the \textit{same} $\Psi(t)$.
In particular (see Lemma \ref{lem:a_tilde_a_diff}),
\begin{equ}\label{eq:A_tilde_A_diff_prelim}
\tilde A(t) - A(t) = tch(0) + L_t(h(0)) + O_{[\mfg,\mfg]}(t^{2-}) + O(t^{5/2-})\;.
\end{equ}

Coming back to the loop $\ell$, we can expand $\hol(\tilde A(t),\ell)$ in terms of $\hol(A(t),\ell)$ using iterated integrals in the sense of Young
(this is another type of Euler estimate but for ODEs, see Lemma~\ref{lem:path_perturb}).
The errors in this expansion, by~\eqref{eq:A_tilde_A_diff_prelim},
are sufficiently large powers of $t$ and $|A(t)|_{\gr\alpha}$,
which allows us to quantify the difference $\hol(A(t),\ell) - \hol(\tilde A(t),\ell)$ in Lemma~\ref{lem:a_tilde_a_hol}.
In particular, we obtain
\begin{equ}\label{eq:hol_diff_non_simple}
\E \hol(A(t),\ell) - \E \hol(\tilde A(t),\ell) = tc_1 X + O(t^{1+})\;,
\end{equ}
where $X \in\mfg^2$ depends linearly on $h(0)$ and can be chosen so that $c_1 X \neq 0$.

If $\Trace (c_1 X) \neq 0$ (which implies in particular that $\mfg$ is not semi-simple since otherwise $\mfg=[\mfg,\mfg]$ and thus $\Trace \mfg=0$),
then we can make the trace of the right-hand side of~\eqref{eq:hol_diff_non_simple} at least of order $t$, which would prove Proposition~\ref{prop:A_tilde_A}.

If, on the other hand, $\Trace(c_1 X) =0$,
then we need to expand the holonomies to higher orders, use the cyclic property of trace, and use that the largest error in \eqref{eq:A_tilde_A_diff_prelim} takes values in $[\mfg,\mfg]$ so is killed by trace, to obtain
 \begin{equ}\label{eq:hol_diff_simple}
\Trace \hol(A(t),\ell) - \Trace\hol(\tilde A(t),\ell) = \bar L_t(h(0))+t^2 \Trace( ( c_1 X)^2) + O(t^{2+})\;,
\end{equ}
where $\bar L_t(h(0))$ is random but \textit{linear} in $h(0)$.
% (furthermore $|\bar L_t(h(0))| \lesssim t^{1+}$ but this is not important).
%(What is crucial is that the $t^2$ term is quadratic in $h(0)$ while $\bar L_t(h(0))$ is linear.)
We then carefully choose a reference $g(0)$, such that, depending on whether $\E \bar L_t(h(0)) \asymp t^2$ or $\E \bar L_t(h(0)) \gg t^2$ or $\E \bar L_t(h(0)) \ll t^2$,
we can adjust $g(0)$ (with the effect of multiplying $h(0)$ by a large scalar) to make the expectation of the right-hand side at least order $t^2$;
finding $g(0)$ in general requires the Chow--Rashevskii theorem (see Lemma~\ref{lem:curve_selection}).

\subsection{Euler estimate for SYM}\label{subsec:Euler}

Suppose we are in the setting of~\cite[Sec.~6.1]{CCHS_2D} with the regularity structure $\cT$ for the continuum 2D YM
(not to be confused with the regularity structure denoted by the same symbol from Section~\ref{sec:DiscreteRS}).
Consider the nonlinearity given by
\begin{equ}
F_{\mfa_i}(\mbA) = \xi_i + \sum_{j=1}^2 \{[A_j,2\partial_j A_i - \partial_iA_j + [A_j,A_i]] + C_i^{(j)}A_j\}\;,
\end{equ}
where $C^{(j)}_i\in L(\mfg,\mfg)$.
We further denote $H_{\mfa_i} = F_{\mfa_i}-\xi_i$.
One should think of $H$ as a polynomial of $(A,\partial A)$ of the form $H(A)=A+A\partial A + A^3$.

We fix henceforth an admissible model $Z$ on $\cT$
which we will later take to be the BPHZ model of the noise $\xi$.
Note that, when applying the results of this subsection to prove Proposition~\ref{prop:A_tilde_A},
we will work directly with the limiting model rather than considering $\eps>0$
and then taking the limit.

We will use the following notation throughout this section.

\begin{notation}\label{not:plus_minus}
For $a\in\R$, we let $a-$ and $a+$ denote $a-n\kappa$ and $a+N\kappa$ respectively,
where the values of $n$ and $N$ may change from term to term, but will always satisfy $1\leq n \ll N \ll \kappa^{-1}$.
In particular, one has
$(a-)+(b+) =(a+b)+$.
\end{notation}

%\begin{remark}
%Notation~\ref{not:plus_minus} implies $(a-)+(b+) =(a+b)+$.
%\end{remark}
%
Fix for the remainder of this section some non-zero $c\in L(\mfg^2,\mfg^2)$.
The fixed point problem we consider for now is 
\begin{equ}\label{eq:A_general_FPP}
A = B + \CP \bone_+ (H(A) + ch) \;,
\end{equ}
which we pose for $A\in\cD^{1+,0-}_{0-}$
and where $B \in \cD^{1+,0-}_{0-}$, at this stage, is
an arbitrary given modelled distribution.
Here $\CP$ denotes convolution with the heat kernel.
(We should think of $B= P A(0) +\CP^{1_+\xi}\bone_+\Xi$, where $PA(0)$ is the `harmonic extension' of $A(0)$ to positive times and where we emphasise that $\CP \bone_+\Xi$, is a priori ill-defined, so we insert
$1_+\xi$ compatible with $\bone_+\Xi$ playing the role of its reconstruction,
see~\cite[Appendix~A.3]{CCHS_2D}.
We could thus even take $B \in \cD^{\infty,0-}_{0-}$.)

Above, $h\in \cD^{1-,0}_0$ is the lift of $\mrd g g^{-1}$ to the polynomial regularity structure,
where $g$ solves the well-posed PDE \eqref{eq:PDE_g} with $B$ replaced by $A$.
%\begin{equ}
%\partial_t g = \Delta g - (\partial_j g) g^{-1} (\partial_j g) + [A_j, (\partial_j g) g^{-1}]g\;.
%\end{equ}
A straightforward calculation shows that (see~\cite[Lem.~7.2]{CCHS_2D})
\begin{equ}\label{eq:h_equ}
\partial_t h_{i} = 
\Delta h_i - [h_j,\partial_j h_i] + [[A_j, h_j],h_i] + \partial_i [A_j, h_j]\;.
\end{equ}
Observe that
\begin{equ}\label{eq:H_map}
H\colon \cD^{1+, 0-}_{0-} \to \cD^{0+,-1-}_{-1-} \quad \text{is locally Lipschitz.}
\end{equ}
%(even in the strong sense of~\cite[Sec.~7.3]{Hairer14}, though we don't use this).
Recall also that, for $\theta\geq0$, $t\in(0,1]$, and $f\in\cD^{\gamma,\eta}_\alpha$ with $\alpha\wedge\eta>-2$ and $\gamma>0$,
\begin{equ}\label{eq:short_time_conv}
|\CP\bone_+f|_{\cD^{\gamma+2,\bar\eta};t} \lesssim t^{\theta/2}|f|_{\cD^{\gamma,\eta};t}
\end{equ}
where $\bar\eta=(\eta\wedge\alpha)+2-\theta$
and the proportionality constant depends on the Greek letters and polynomially on $\$Z\$_{\gamma;O}$
(the polynomial dependence on $\$Z\$_{\gamma;O}$ follows from inspecting the proof of~\cite[Thm.~7.1]{Hairer14}).
Here $O\eqdef[-1,2]\times \T^2$ and we use the shorthand $|\cdot|_{\cD^{\gamma,\eta};t}\eqdef |\cdot|_{\cD^{\gamma,\eta};(0,t]\times\T^2}$.
The notations $\$Z\$_{\gamma;\K}$ and $|f|_{\cD^{\gamma,\eta};\K}$ are taken from~\cite[Sections~2.3 and~6]{Hairer14}
($|f|_{\cD^{\gamma,\eta};\K}$ is denoted by $\$f\$_{\gamma,\eta;\K}$ therein).

It is standard that there exists $\tau\in(0,1]$, with
\begin{equ}\label{eq:tau_inverse_bound}
\tau^{-1} \lesssim 1+\$Z\$^q_{\gamma;O} + |B|^q_{\cD^{1+,0-}_{0-};O}+|h(0)|_{\CC^{1-}(\T^2)}^q
\end{equ}
for some $q\geq 1$, such that~\eqref{eq:A_general_FPP}+\eqref{eq:h_equ} admits a unique fixed point
$
(A,h)\in \cD^{1+, 0-}_{0-;(0,\tau]\times\T^2}\times \CC^{1-}_{(0,\tau]}
$
such that
\begin{equ}\label{eq:A_h_bound}
|A|_{\cD^{1+,0-}_{0-};\tau} + |h|_{\CC^{1-};\tau} \lesssim 1+ |B|_{\cD^{1+,0-}_{0-};O} + |h(0)|_{\CC^{1-}(\T^2)}\;,
\end{equ}
uniformly in $B$, $h(0)$, and $Z$.
Here $|\cdot|_{\CC^{\alpha};\tau}$ and $\CC^{\alpha}_{(0,\tau]}$ denote the $\CC^\alpha$ norm and space (for the parabolic scaling) restricted to $(0,\tau]\times\T^2$.

\begin{notation}
In the estimates that follow, unless otherwise stated, we fix $t\in(0,\tau]$ and all space-time functions and modelled distributions are understood as restricted to $(0,t]\times\T^2$,
e.g. $\cD^{\gamma,\eta}_\alpha$ and $\CC^\gamma$ will mean $\cD^{\gamma,\eta}_{\alpha;(0,t]\times\T^2}$ and $\CC^\gamma((0,t]\times\T^2)$ respectively.
All implicit proportionality constants will be sufficiently large powers of
\begin{equ}\label{eq:poly_constant}
2+\$Z\$_{\gamma;O}+|B|_{\cD^{1+,0-}_{0-};O}+|h(0)|_{\CC^{2+}(\T^2)}
\end{equ}
and will \textit{not} depend on $t$.
Note that we require more regularity on $h(0)$ in \eqref{eq:poly_constant}
than in \eqref{eq:tau_inverse_bound}-\eqref{eq:A_h_bound}.

Furthermore, for a normed space $(\CB,\|\cdot\|)$, we let
$O_{\CB}(t^{\eta})$ denote a term $X\in\CB$ such that $\|X\|\lesssim t^{\eta}$ uniformly in $t$.
We also write $\cD^{\gamma,\eta}_\alpha[\mfg,\mfg]$
for the space of modelled distributions that reconstruct to $1$-forms with values in the derived Lie algebra $[\mfg,\mfg]$,
and likewise for $\CC^\gamma[\mfg,\mfg]$.
%When we estimate a term of the form $|r|_{\cD[\mfg,\mfg]}$, we implicitly mean $r\in \cD[\mfg,\mfg]$.
\end{notation}
The following is the main result of this subsection.

\begin{proposition}\label{prop:Euler_A}
Define the space-time function $\bar h\colon [0,t]\times\T^2\to\mfg^2$ by $\bar h(s,x) = sh(0)(x)$. Then
\begin{equ}\label{eq:A_Euler}
A = B + D + c \bar h + L_t(h(0)) + O_{\cD^{0+,0-}_0[\mfg,\mfg]}(t^{2-N\kappa})
+
O_{\cD^{1+,0-}_0}(t^{5/2-N\kappa})
\end{equ}
where
$D= O_{\cD^{1+,0-}_0}(t^{1/2-})$ depends only on $B$
and
$L_t(h(0))= O_{\CC^{0+}}(t^{3/2-N\kappa})$
is linear in $h(0)$ 
and takes values in the polynomial sector.
\end{proposition}
Proposition~\ref{prop:Euler_A}  is proven at the end of this subsection.
We  write~\eqref{eq:h_equ} as
\begin{equ}\label{eq:h_FPP}
h = P h(0) + \CP(h\partial h + A h^2 + \partial(Ah))\;.
\end{equ}
Here, for $f\in\CC(\T^2)$, we denote by $Pf \in\CC([0,t]\times \T^2)$ the harmonic extension of $f$ to positive times, i.e. $Pf(t,x) = (e^{t\Delta} f)(x)$.
\begin{lemma}\label{lem:h_Euler_1}
Interpreting $h(0)$ as a space-time function that is constant in time,
\begin{equ}\label{eq:h_Euler_1}
h = h(0) +  O_{\CC^{0+}[\mfg,\mfg]}(t^{1/2-N\kappa}) + R_t\;.
\end{equ}
where $R_t\eqdef P h(0) - h(0) =O_{\CC^{0+}}(t)$.
Furthermore
\begin{equ}\label{eq:CP_h_Euler_2}
|\CP h|_{\CC^{1+}} \lesssim t^{1-N\kappa}\;.
\end{equ}
\end{lemma}

\begin{proof}
The bound~\eqref{eq:h_Euler_1} follows from
\begin{equ}
|\CP\partial(Ah)|_{\CC^{0+}[\mfg,\mfg]} \lesssim t^{1/2-N\kappa}|\partial(Ah)|_{\CC^{-1-}[\mfg,\mfg]}
\lesssim t^{1/2-N\kappa}
%|A|_{\CC^{0-}} |h|_{\CC^{0+}}
\end{equ}
and
\begin{equ}\label{eq:h_better}
|\CP(h\partial h + Ah^2)|_{\CC^{0+}[\mfg,\mfg]} \lesssim t^{1-N\kappa}\;.
%(|h|_{\CC^{1-N\kappa}}^2+|h|_{\CC^{0+}}^2|A|_{\CC^{0-}})\;.
\end{equ}
Now~\eqref{eq:CP_h_Euler_2} follows from~\eqref{eq:h_Euler_1} and $|\CP h(0)|_{\CC^{1+}}\lesssim t$ and $|\CP f|_{\CC^{1+}} \lesssim t^{1/2}|f|_{\CC^{0+}}$.
\end{proof}
%
%\begin{lemma}\label{lem:Euler_1}
%\begin{equ}\label{eq:Euler_1}
%A = B + O_{\cD^{1+,0-}_0}(t^{1/2-})\;.
%\end{equ}
%\end{lemma}
%%
%\begin{proof}
%By~\eqref{eq:short_time_conv} and~\eqref{eq:CP_h_Euler_2}, we obtain
%\begin{equ}
%|A-B|_{\cD^{1+,0-}_0} = |\CP ( H(A)+ch)|_{\cD^{1+,0-}_0} \lesssim t^{1/2-}\;.
%\end{equ}
%\end{proof}
%
A consequence of~\eqref{eq:H_map},~\eqref{eq:short_time_conv}, and~\eqref{eq:CP_h_Euler_2} is
$
|\CP ( H(A)+ch)|_{\cD^{1+,0-}_0} \lesssim t^{1/2-}
$
and therefore, since $A$ solves the fixed point~\eqref{eq:A_general_FPP},
\begin{equ}\label{eq:Euler_1}
A = B + O_{\cD^{1+,0-}_0}(t^{1/2-})\;.
\end{equ}
In the proofs below, we will use without mention
the usual multiplication and differentiation rules for singular
modelled distributions~\cite[Propositions~6.12,~6.15]{Hairer14}.
\begin{lemma}\label{lem:Euler_2}
Let  $B_1 = \CP H(B)$. Then
\begin{equ}[e:Euler_2]
A = B+ B_1 + O_{\cD^{1+,0-}_0}(t^{1-N\kappa})\;.
\end{equ}
\end{lemma}
\begin{proof}
We Taylor expand $H(A)$ once to get $H(A)=H(B) + r$ where
\begin{equs}
|r|_{\cD^{0+,-1-}_{-1-}}
&\lesssim
|A-B|_{\cD^{1+,0-}_{0-}} + |Q(A,B)|_{\cD^{1+,0-}_{0-}}|(A-B)|_{\cD^{1+,0-}_{0-}}
\\
&\qquad + |B|_{\cD^{1+,0-}_{0-}}|\partial(A-B)|_{\cD^{0+,-1-}_{-1-}}
+ |\d A|_{\cD^{0+,-1-}_{-1-}}|A-B|_{\cD^{1+,0-}_{0-}}
\\
&\lesssim t^{1/2-}\;,	\label{eq:n=1_bound2}
\end{equs}
where $Q$ is quadratic
and where we used~\eqref{eq:Euler_1} in the second bound.
Substituting this into the fixed point~\eqref{eq:A_general_FPP} gives
$
A = B + \CP(H(B) + r + ch)
$.
Furthermore,
$|\CP r|_{\cD^{1+,0-}_{0}} \lesssim t^{1/2-}|r|_{\cD^{0+,-1-}_{-1-}} \lesssim t^{1-}$
where we used~\eqref{eq:short_time_conv} in the first bound
and~\eqref{eq:n=1_bound2} in the second bound.
We conclude by applying~\eqref{eq:CP_h_Euler_2}.
\end{proof}
To get a remainder term which is sufficiently small in $t$ and takes values in $[\mfg,\mfg]$, we need to expand the nonlinearity $H$ and substitute back into the fixed point~\eqref{eq:A_general_FPP} two more times.
We introduce the notation
\begin{equ}\label{eq:DH_notation}
DH(B)X = X+X\partial B + B\partial X + 3 B^2 X\;,
\end{equ}
where the products have the obvious meaning.
In particular, by~\eqref{eq:short_time_conv}, we have
\begin{equ}\label{eq:CP_DHB_bound}
\CP DH(B)X = O_{\cD^{1+,0-}_0}(t|X|_{\cD^{1+,0-}_0}) + O_{\cD^{1+,0-}_0[\mfg,\mfg]} (t^{1/2}|X|_{\cD^{1+,0-}_0})\;.
\end{equ}

\begin{lemma}\label{lem:Euler_3}
Let  $B_2= \CP  [DH(B)B_1]$.
Then
\begin{equ}
A = B + B_1 + B_2 + \CP  ch+ O_{\cD^{1+,0-}_0}(t^{3/2-N\kappa})\;.
\end{equ}
\end{lemma}

\begin{proof}
By Taylor expansion,
$
H(A) = H(B) + DH(B)(A-B) + r
$,
where, by~\eqref{eq:Euler_1},
$|r|_{\cD^{0+,-1-}_{0-}} \lesssim
|A- B|_{\cD^{1+,0-}_0}^2 + |B|_{\cD^{1+,0-}_{0-}}|A- B|_{\cD^{1+,0-}_0}^2 \lesssim t^{1-}$.
Therefore
\begin{equ}
A = B + \CP (H(B) + DH(B)(A-B) + ch) + \CP r\;,
\end{equ}
where, by~\eqref{eq:short_time_conv},
\begin{equ}\label{eq:n=2_bound}
|\CP r|_{\cD^{1+,0-}_0} \lesssim t^{1/2}|r|_{\cD^{0+,-1-}_{0-}} \lesssim t^{3/2-}\;.
\end{equ}
Then, by Lemma~\ref{lem:Euler_2},
\begin{equ}
A = B + B_1
+ \CP  [DH(B)B_1] + \CP [DH(B)O_{\cD^{1+,0-}_0}(t^{1-N\kappa})]
+ \CP  ch
+ \CP r
%\\
%&= B + B_1 + B_2 + \CP  ch+ \CE
\end{equ}
and the conclusion follows by~\eqref{eq:CP_DHB_bound} and~\eqref{eq:n=2_bound}.
%where now $|\CE|_{\cD^{1+,0-}_0} \lesssim t^{3/2-N\kappa}$
%because, by~\eqref{eq:CP_DHB_bound},
%\begin{equ}
%|\CP DH(B) O_{\cD^{1+,0-}_0}(t^{1-N\kappa})|_{\cD^{1+,0-}_0}\lesssim t^{3/2-N\kappa}\;.
%\end{equ}
\end{proof}
Similar to~\eqref{eq:DH_notation}, we introduce the notation
\begin{equ}
D^2H(B)(X,X) \eqdef D^2H(B)X^2 \eqdef 2X\d X+6BX^2
\end{equ}
where again the products have the obvious meaning.
\begin{lemma}\label{lem:Euler_4}
Let $B_3 = \CP \{DH(B)B_2 + 2B_1\d B_1\}$. Then
\begin{equs}
A &= B + B_1 + B_2 + B_3 + \CP \{DH(B)(\CP ch)\} + \CP  ch
\\
&\quad + O_{\cD^{1+,0-}_0[\mfg,\mfg]}(t^{2-N\kappa}) + O_{\cD^{1+,0-}_0}(t^{5/2-N\kappa})\;.
\end{equs}
\end{lemma}
\begin{proof}
Taylor expansion gives
\begin{equ}
H(A) = H(B) + DH(B)(A-B) + \frac12 D^2H(B)(A-B)^2 + (A-B)^3\;.
\end{equ}
We have $(A-B)^3=O_{\cD^{1+,0-}_{0-}[\mfg,\mfg]}( t^{3/2-})$ by~\eqref{eq:Euler_1}, so~\eqref{eq:short_time_conv} implies
$
\CP  (A-B)^3 = O_{\cD^{1+,0-}_{0-}[\mfg,\mfg]}(t^{5/2-})
$,
which is (better than) the desired order.
Furthermore, by \eqref{eq:Euler_1} and \eqref{e:Euler_2},
\begin{equ}
(A-B)\d(A-B) = B_1\d B_1 +  O_{\cD^{0+,-1-}_{-1-}[\mfg,\mfg]}(t^{3/2-N\kappa})\;,
\end{equ}
and, by \eqref{eq:Euler_1}, $(A-B)^2 = O_{\cD^{1+,0-}_{0}[\mfg,\mfg]}(t^{1-})$, which implies
%\begin{equ}
%D^2H(B)(A-B)^2 = \CP\{B_1\d B_1\} + O_{\cD^{1+,0-}_{0}[\mfg,\mfg]}(t^{2-N\kappa})\;,
%\end{equ}
%and thus
by~\eqref{eq:short_time_conv}
\begin{equ}
\CP D^2H(B)(A-B)^2 =  \CP\{2 B_1\d B_1\} + O_{\cD^{1+,0-}_{0}[\mfg,\mfg]}(t^{2-N\kappa})\;,
\end{equ}
which is the desired order.
Finally, applying Lemma~\ref{lem:Euler_3} and~\eqref{eq:CP_DHB_bound} we obtain
%\begin{equ}
%DH(B)(A-B) = DH(B)(B_1 + B_2 + \CP  ch+ O_{\cD^{1+,0-}_0[\mfg,\mfg]}(t^{3/2-N\kappa}) + O_{\cD^{1+,0-}_0}(t^{2-N\kappa}))\;.
%\end{equ}
%Using~\eqref{eq:CP_DHB_bound}, we obtain 
\begin{equs}
\CP[DH(B)(A-B)]
&= \CP \{DH(B)(B_1+B_2+\CP ch)\}
\\
&\quad+ O_{\cD^{1+,0-}_{0-}[\mfg,\mfg]}(t^{2-N\kappa}) + O_{\cD^{1+,0-}_{0-}}(t^{5/2-N\kappa})\;.
\end{equs}
\end{proof}
The point in Lemma~\ref{lem:Euler_4} is that everything except $h$ and the terms $O(\cdots)$ on the right-hand side 
is a function only of $B$, and that $h$ appears only linearly (except possibly in $O(\cdots)$).
We now approximate $\CP h$ further.
\begin{lemma}\label{lem:h_Euler}
One has
$
h = h(0) + \CP\partial(Bh(0))+ O_{\CC^{0+}[\mfg,\mfg]}(t^{1-N\kappa})
+ R_t$
for $R_t$ as in Lemma~\ref{lem:h_Euler_1}.
\end{lemma}

\begin{proof}
Since $h$ solves~\eqref{eq:h_FPP}, we have
\begin{equs}
h &= h(0) + \CP \partial(Ah) + O_{\CC^{0+}[\mfg,\mfg]}(t^{1-N\kappa}) + R_t
\\
&= h(0) + \CP\partial (A h(0)) + O_{\CC^{0+}[\mfg,\mfg]}(t^{1-N\kappa}) + R_t
\end{equs}
where we used~\eqref{eq:h_better} in the first equality
and, in the second equality,~\eqref{eq:h_Euler_1} combined with
$
\CP\partial (A O_{\CC^{0+}}(t^{1/2-N\kappa})) = O_{\CC^{0+}[\mfg,\mfg]}(t^{1-N\kappa})
$.
Finally, we have $A = B + O_{\CC^{0-}}(t^{1/2-})$ by~\eqref{eq:Euler_1},
and
$
\CP \partial (O_{\CC^{0-}}(t^{1/2-}) h(0)) = O_{\CC^{0+}[\mfg,\mfg]}(t^{1-N\kappa})
$,
from which the desired estimate follows.
\end{proof}

\begin{proof}[of Proposition~\ref{prop:Euler_A}]
Combining Lemmas~\ref{lem:Euler_4} and~\ref{lem:h_Euler} and
\begin{equs}
\CP \{DH(B)\CP  [O_{\CC^{0+}[\mfg,\mfg]}(t^{1-N\kappa})+ R_t]\} =  O_{\cD^{1+,0-}_0[\mfg,\mfg]}(t^{2-N\kappa}) + O_{\CC^{0+}}(t^3)
\end{equs}
(which follows from $|\CP f|_{\CC^{1+}}\lesssim t^{1/2}|f|_{\CC^{0+}}$ and~\eqref{eq:CP_DHB_bound})
and
\begin{equ}
\CP O_{\CC^{0+}[\mfg,\mfg]}(t^{1-N\kappa}) =  O_{\cD^{0+,0-}_0[\mfg,\mfg]}(t^{2-N\kappa})\;,
\end{equ}
we obtain
\begin{equs}
A &= B + B_1 + B_2 +B_3
+\CP \{DH(B)\CP c[ h(0)
+ \CP\partial (B h(0))]\}
+\CP c h(0)
\\
&\quad
+ \CP \{c\CP \partial(B h(0))\} + \CP c R_t
+O_{\cD^{0+,0-}_0[\mfg,\mfg]}(t^{2-N\kappa}) 
+O_{\cD^{1+,0-}_0}(t^{5/2-N\kappa}) \;.
\end{equs}
We have also $\CP\partial (B h(0)) = O_{\CC^{0+}}(t^{1/2-N\kappa})$ which implies
\begin{equs}[eq:L_t_2_terms]
\CP \{DH(B)\CP c[ h(0)
+ \CP\partial (B h(0))]\} &= O_{\CC^{0+}}(t^{3/2-N\kappa})\;,
\\
\CP \{c\CP \partial(B h(0))\} &=
O_{\CC^{0+}}(t^{3/2-N\kappa})\;.
\end{equs}
%so these terms can be absorbed into $L_t(h(0))$ in~\eqref{eq:A_Euler}.
Moreover, we can write $\CP h(0)= \bar h + \{\CP h(0) - \bar h\}$ where $\CP h(0) - \bar h= O_{\CC^{0+}}(t^2)$ is linear in $h(0)$.
The terms \eqref{eq:L_t_2_terms}, $\CP c R_t$, and $\CP ch(0) - c\bar h$ thus combine to give the desired $L_t(h(0))$ in~\eqref{eq:A_Euler}.
The conclusion follows since $B_1=O_{\cD^{1+,0-}_0}(t^{1/2-})$ by~\eqref{eq:Euler_1} and Lemma~\ref{lem:Euler_2},
and $B_{2,3}$  are of even smaller order.
\end{proof}

\subsection{Expansion of Wilson loops}
\label{subsec:hol_expansion}

We keep the same setting as the previous subsection.
In particular, we have a model $Z$ on the regularity structure $\cT$ and a given modelled distribution $B\in\cD^{1+,0-}_{0-}$.
We now consider $\tilde A$ and $A$ as the modelled distribution from the previous subsection
with different choices for the initial condition of $h$.
More specifically, we take $h(0)=0$ (and thus $h\equiv 0$) for $A$,
and generic smooth $h(0)$ for now for $\tilde A$.
We emphasise that $A,\tilde A$ solve~\eqref{eq:A_general_FPP} with different $h$ but with the \textit{same} model $Z$ and modelled distribution $B$.
We furthermore take the existence time $\tau$ common to $A,\tilde A$ and consider $t\in(0,\tau]$
and assume that the model $Z$ is such that the reconstructions of all modelled distributions can be evaluated at time-slices.

We denote by $a,\tilde a, b \in \CD'(\T^2)$ the reconstructions of $A,\tilde A, B$ respectively evaluated at time $t$.
The main result of this subsection is Lemma~\ref{lem:a_tilde_a_hol} which estimates the difference between $\hol(\tilde a,\ell)$ and $\hol(a,\ell)$.

%All function norms, e.g. $|\cdot|_{\CC^{\gamma}}$, are taken over $\T^2$.
As before, unless otherwise stated, all implicit proportionality constants below are sufficiently large powers of~\eqref{eq:poly_constant} and do not depend on $t$.
We have the following estimate on $a-b$ and $\tilde a- a$, which is an immediate consequence of Proposition~\ref{prop:Euler_A}.

\begin{lemma}\label{lem:a_tilde_a_diff}
One has $a = b + O_{\CC^{0+}(\T^2)}(t^{1/2-})$
and
\begin{equ}\label{eq:a_tilde_a_diff}
\tilde a = a + t c h(0) + L_t(h(0)) + O_{\CC^{0+}(\T^2,[\mfg,\mfg])}(t^{2-N\kappa})
+O_{\CC^{0+}(\T^2)}(t^{5/2-N\kappa}) \;,
\end{equ}
where $L_t(h(0))= O_{\CC^{0+}(\T^2)}(t^{3/2-N\kappa})$ is linear in $h(0)$.
\end{lemma}
Consider $p\in[1,2)$ and a path $\gamma\in\CC^{\var p}([0,1],\mfg)$,
where $\CC^{\var p}$ is defined below~\eqref{eq:p-var}.
Let $E^\gamma \in \CC^{\var p} ([0,1],G)$
denote the path ordered exponential of $\gamma$, i.e. the solution to the linear ODE
\begin{equ}\label{eq:V_ODE}
\mrd E^\gamma(x) = E^\gamma(x)\mrd \gamma(x)\;,\quad E^\gamma(0)=\id\;,
\end{equ}
which is well-posed as a Young ODE~\cite{Lyons94}.

%For $\gamma$ as above, since the ODE for $E^\gamma$ is linear, we have the expansion in terms of iterated (Young) integrals
%\begin{equs}%\label{eq:hol_expansion}
%E^\gamma(1) &= 1 + \int_0^1 \mrd \gamma(x) + \int_0^1\int_0^x \mrd \gamma(y)\mrd \gamma(x) + \ldots
%\\
%&= 1 + \sum_{k\geq1} \int_0^1\ldots\int_0^{x_{k-1}} \mrd \gamma(x_{k})\ldots \mrd \gamma(x_1)\;,
%\end{equs}
%which converges absolutely and, for all $k\geq 1$,
%\begin{equ}\label{eq:Young_iter_int}
%\Big|\int_0^1\ldots\int_0^{x_{k-1}} \mrd \gamma(x_{k})\ldots \mrd\gamma(x_1)\Big| \leq \frac{C|\gamma|_{\var p}^k}{(k/p)!}\;,
%\end{equ}
%where $C=C(p)>0$ is a constant,
%see e.g.~\cite[Thm.~3.7]{LyonsStFlour07}.

We will be interested in perturbing $\gamma$, for which we have the following result.
The lemma is a routine application of Young's estimate for integrals, but we give a proof for completeness.

\begin{lemma}\label{lem:path_perturb}
For $\zeta\in \CC^{\var p}([0,1],\mfg)$ and $L\geq 4$
we have
\begin{equs}
E^{\gamma+\zeta}(1) = E^\gamma(1) &+ \int_0^1 \mrd \zeta(x) + P^\gamma(\zeta)
+ \int_0^1\int_0^x \mrd \zeta(y)\mrd \zeta(x)
\\
&+ O\{w^L+w^{L+1} + v^{L+1} + v^2 (w + v + w^{L-3})\}
\end{equs}
where $v\eqdef |\zeta|_{\var p}$, $w\eqdef |\gamma|_{\var p}$, and $P^\gamma(\zeta) = O(v(w+w^{L-2}))$
is linear in $\zeta$, and
where the proportionality constants depend only on $p,L$.
\end{lemma}

\begin{proof}
Denoting $I^\gamma=\sum_{k=1}^{L-1} \int_0^1\ldots\int_0^{x_{k-1}} \mrd \gamma(x_{k})\ldots \mrd \gamma(x_1)$,
by linearity of the ODE~\eqref{eq:V_ODE}, for every $L\geq 1$
\begin{equ}
E^\gamma(1) = \id
+ I^\gamma +
\int_0^1\cdots\int_0^{x_{L-1}} E^{\gamma}(x_{L})\mrd \gamma(x_{L}) \mrd \gamma(x_{L-1})\ldots \mrd \gamma(x_1)\;.
\end{equ}
Since $E^\gamma$ takes values in a compact set, it follows that $|E^\gamma|_{\var p} \leq C_p |\gamma|_{\var p}$.
Recalling Young's estimate for integrals~\eqref{eq:Young_int} with $p=q$,
we obtain
\begin{equ}\label{eq:V_gamma_expansion}
E^\gamma(1) = \id
+ I^\gamma + O(w^L + w^{L+1})\;,
\end{equ}
where the proportionality constant depends only on $L,p$.
Likewise
\begin{equ}
E^{\gamma+\zeta}(1)
=
\id + I^{\gamma+\zeta}
+ O(w^L + v^L+ w^{L+1} + v^{L+1})\;.
\end{equ}
Expanding the iterated integrals in $I$, we further have
\begin{equs}
I^{\gamma+\zeta} &= I^\gamma + \int_0^1\mrd \zeta(x) + P^\gamma(\zeta)
+ \int_0^1\int_0^x \mrd \zeta(y)\mrd \zeta(x) 
\\
&\qquad + O\{v^2 (w + v + w^{L-3} + v^{L-3})\}
\end{equs}
where $P^\gamma(\zeta)$ arises from the integrals with one instance of $\zeta$ and at least one instance of $\gamma$, which is clearly linear in $\zeta$ and
satisfies $P^\gamma(\zeta) = O(v(w+w^{L-2}))$,
and $O\{v^2 (w + v + w^{L-3} + v^{L-3})\}$ arises from the $k$-fold integrals for $3\leq k \leq L-1$ which have at least two instances of $\zeta$.
%\ilya{This is where we use $L\geq 4$, otherwise this term is absent}
The conclusion now follows once we remark that $v^{L} + v^{L-1} = O(v^3 + v^{L+1})$.
\end{proof}

For metric spaces $(E,\rho),(E',\rho')$, a function $f\colon E\to E'$, and $\alpha\in[0,1]$, denote
\begin{equ}\label{eq:Hol_def}
|f|_{\Hol\alpha} \eqdef \sup_{x\neq y\in E} \rho(x,y)^{-\alpha}\rho'(f(x),f(y))\;.
\end{equ}
%Consider 
Suppose that we are given a line $\ell\colon [0,1]\to \T^2$,
$\ell(x)=z+xhe_i$, where $z\in\T^2$, $h\in[-1,1]$,
 and $i\in\{1,2\}$.
For a $1$-form $\omega\in\Omega^1_{\alpha}$
we define $\ell_\omega \colon [0,1] \to \mfg$
by $\ell_\omega(x)=\omega(\ell\restr_{[0,x]})$,
where $\ell\restr_{[0,x]}$ is understood canonically as a line segment in $\T^2$ and $\omega(\ell\restr_{[0,x]})$ is well-defined by additivity of $\omega$.
(If $h=0$ then $\ell$ is constant and thus $\ell_\omega\equiv0$ by additivity.)
A basic consequence of additivity is that
\begin{equ}\label{eq:ell_omega_Hol}
|\ell_\omega|_{\Hol\alpha} \leq 2^{1-\alpha}  |\ell|^\alpha|\omega|_{\gr\alpha}\;.
\end{equ}
(The factor $2^{1-\alpha}$ arises from our restriction to lines of length $\leq 1/2$ in the definition of $|\omega|_{\gr\alpha}$.)
For $\alpha\in(\frac12,1]$, we can thus define 
the holonomy $\hol(\omega,\ell)$ by \begin{equ}\label{eq:hol_def}
\hol(\omega,\ell)=E^{\ell_\omega}(1)\;,
\end{equ}
well-posed as a Young ODE since $|\cdot|_{\var{(1/\alpha)}} \leq |\cdot|_{\Hol\alpha}$.

%Consider
Below we fix the loop $\ell\colon[0,1]\to\T^2$, $\ell(x)=(x,0)$ as in Proposition~\ref{prop:A_tilde_A}.

\begin{lemma}\label{lem:a_tilde_a_hol}
For any $\alpha\in(\frac12,1]$ and $L\geq 4$ 
\begin{equ}\label{eq:hol_diff_1}
\hol(\tilde a,\ell) = \hol(a,\ell) + t c_1 \gamma(1) + O(t|a|_{\gr\alpha} + t^{3/2-N\kappa} + |a|_{\gr\alpha}^{L}+|a|_{\gr\alpha}^{L+1})\;,
\end{equ}
where $\gamma\in \CC^\infty([0,1],\mfg^2)$ is defined by $\gamma_i(x)=\int_0^x h_i(0)(y,0)\mrd y$.
Furthermore
\begin{equs}[eq:hol_diff_2]
\Trace\hol(\tilde a,\ell)
&= \Trace\hol(a,\ell)
+ t\Trace c_1\gamma(1)
+\bar L_t(h(0)) + \frac{t^2}2 \Trace \big[( c_1\gamma(1))^2\big]
\\
&\qquad + O(t^{2-N\kappa}|a|_{\gr\alpha} + |a|_{\gr\alpha}^{L}+|a|_{\gr\alpha}^{L+1}+ t^{5/2-N\kappa})\;,
\end{equs}
where
$\bar L_t(h(0))=O(t^{3/2-N\kappa}+t(|a|_{\gr\alpha}+|a|^{L-2}_{\gr\alpha}))$ %\ilya{added $t^{3/2-N\kappa}$ which I think was missing - see below. Also changed $|a|^{L-1}_{\gr\alpha}$ to $|a|^{L-2}_{\gr\alpha}$ here and in proof}
 is linear in $h(0)$.
\end{lemma}

% \begin{remark}
% The estimate $\bar L_t(h(0))=O(t(|a|_{\gr\alpha}+|a|^{L-1}_{\gr\alpha}))$ is not used later; what is crucial for the proof of Proposition~\ref{prop:A_tilde_A} is that $\bar L_t(h(0))$ is linear in $h(0)$.
% \end{remark}

\begin{proof}
By~\eqref{eq:a_tilde_a_diff}, we have $\tilde a = a + \omega+\eta$ where
\begin{equ}\label{eq:omega_def}
\omega= tc h(0) + L_t(h(0))\;,\quad \eta = O_{\CC^{0+}(\T^2,[\mfg,\mfg])}(t^{2-N\kappa})
+O_{\CC^{0+}(\T^2)}(t^{5/2-N\kappa})\;,
\end{equ}
and where $\omega = O_{\CC^{0+}(\T^2)}(t)$ and $L_t(h(0))=O_{\CC^{0+}(\T^2)}(t^{3/2-N\kappa})$ is linear in $h(0)$.
Since $\omega\in \CC^{0+}(\T^2)\hookrightarrow \Omega_{\gr1}$,
we have $|\ell_\omega|_{\Hol\alpha}\leq |\ell_\omega|_{\Hol1}\leq |\omega|_{\CC^{0+}(\T^2)}$ and similarly for $\eta$.
Likewise
$|\ell_a|_{\Hol{\alpha}} \lesssim |a|_{\gr{\alpha}}$ by~\eqref{eq:ell_omega_Hol}.
Furthermore $\ell_{\tilde a} = \ell_a + \ell_\omega + \ell_\eta$ by linearity of $X\mapsto \ell_X$,
and $E^{\ell_a}(1)=\hol(a,\ell)$.
Therefore,
by Lemma~\ref{lem:path_perturb} with $p=1/\alpha$, $\zeta=\ell_{\omega+\eta}$, and $\gamma=\ell_a$ therein (and thus $v=O(t)$ and $w\lesssim |a|_{\gr\alpha}$),
\begin{equs}[eq:hol_tilde_a_expan]
\hol(\tilde a,\ell) &= \hol(a,\ell) + \int_0^1 \mrd \ell_{\omega+\eta}(x)
+ P^{\ell_a}(\ell_\omega + \ell_\eta)
\\
&\quad + 
\int_0^1\int_0^x \mrd \ell_{\omega+\eta}(y)\mrd\ell_{\omega+\eta}(x)
+ O(|a|_{\gr\alpha}^{L} + |a|_{\gr\alpha}^{L+1} + t^3 + t^2|a|_{\gr\alpha})\;,
\end{equs}
where $P^{\ell_a}(\ell_\omega) = O(t(|a|_{\gr\alpha}+|a|_{\gr\alpha}^{L-2}))$ is linear in $\ell_\omega$ and thus in $h(0)$,
and $P^{\ell_a}(\ell_\eta) = O(t^{2-N\kappa}(|a|_{\gr\alpha}+|a|_{\gr\alpha}^{L-2}))$,
and where we used that $t^2|a|_{\gr\alpha}^{L-3} = O(t^2|a|_{\gr\alpha} + |a|_{\gr\alpha}^{L})$.
Since $h_i(0)(x,0) = \gamma'_i(x)$, we have
\begin{equ}\label{eq:omega_f}
|\ell_\omega - tc_1\gamma|_{\Hol1} = |\omega_1-tc_1h(0)|_{L^\infty(\T^2)}= O(t^{3/2-N\kappa})\;,
\end{equ}
where, in the final bound, we used~\eqref{eq:omega_def} and the estimate on $L_t(h(0))$ in the line below it.
Since the double integral in~\eqref{eq:hol_tilde_a_expan} is of order $t^2$,
this proves~\eqref{eq:hol_diff_1}.

To prove~\eqref{eq:hol_diff_2}, by the cyclic property of trace,
\begin{equs}
\Trace \int_0^1\int_0^x \mrd \ell_{\omega+\eta}(y)\mrd\ell_{\omega+\eta}(x)
%&= \frac12\Trace \int_{[0,1]^2}\mrd\ell_{\omega+\eta}(y)\mrd\ell_{\omega+\eta}(x)
%\\
&= \frac12 \Trace [(\ell_{\omega+\eta}(1))^2]\;.
\end{equs}
%where we used the cyclic property of trace in the first equality.
Furthermore, by~\eqref{eq:omega_f} and the bound on $\eta$ from \eqref{eq:omega_def},
\begin{equ}
\Trace [(\ell_{\omega+\eta}(1))^2] = t^2\Trace [( c_1 \gamma(1))^2] + O(t^{3-N\kappa})\;.
\end{equ}
%\ilya{This doesn't matter for the final statement, but I changed $O(t^{5/2-N\kappa})$ to $O(t^{3-N\kappa})$ since $\omega = O(t)$ and $\eta = O(t^{2-N\kappa})$ - it looks more natural to me now}
Finally, by \eqref{eq:omega_def}, $\Trace\int_0^1 \mrd \ell_\eta(x) = O(t^{5/2-N\kappa})$ since $\Trace[\mfg,\mfg]=0$.
The conclusion follows with $\bar L_t(h(0))= \Trace \{A+P^{\ell_a}(\ell_\omega)\}$ by taking trace on both sides of \eqref{eq:hol_tilde_a_expan},
where $A=\int_0^1 \mrd \ell_{\omega}(x) - t c_1\gamma(1) = O(t^{3/2-N\kappa})$ is linear in $h(0)$.
%\ilya{I added this extra term $A$ into $\bar L_t(h(0))$ - it was missing before which looks like a mistake!}
\end{proof}

\subsection{Proof of Proposition~\ref{prop:A_tilde_A}}
\label{subsec:proof_of_A_tilde_A}

We finally take $B$ as the modelled distribution $\CP^{1_+\xi}\bone_+\Xi$, which
reconstructs to the SHE with  zero initial condition.
We fix $\alpha\in(\frac12,1)$ and $0<\beta<(1-\alpha)/2$
throughout this subsection.
Then $\E|b(t)|^2_{\gr\alpha}\lesssim t^{\beta}$ by~\cite[Lem.~4.12]{CCHS_2D}
(see also Remark~\ref{rem:correct_Thm_413}),
and thus $t^{-\beta/2}|b(t)|_{\gr\alpha}$ has Gaussian tails uniform in $t\in (0,1]$.
Furthermore, $\$Z\$_{\gamma;O}$ and $|B|_{\cD^{1+,0-}_{0-};O}$ have moments of all orders.

We continue using the notation of Section~\ref{subsec:hol_expansion}.
Throughout this section, we consider $M\geq 2$ and, for $t\in(0,1)$, the event
\begin{equ}
Q_t \eqdef \{2(1+\$Z\$_{\gamma;O}+ |B|_{\cD^{1+,0-}_{0-};O} + t^{-\beta/2}|b(t)|_{\gr\alpha}) < M\}\;.
\end{equ}
We let $\tau\geq M^{-q}$ denote the guaranteed existence time from Section~\ref{subsec:Euler} on the event $\{2(1+\$Z\$_{\gamma;O}+ |B|_{\cD^{1+,0-}_{0-};O})<M\}\supset Q_t$ for $q$ sufficiently large.

\begin{lemma}\label{lem:restricted_hol}
In the notation of Lemma~\ref{lem:a_tilde_a_hol},
for all $t\in(0,\tau]$
\begin{equ}\label{eq:E_hol_1}
\E\hol(\tilde a,\ell)\bone_{Q_t} = \E \hol(a,\ell)\bone_{Q_t}
+\P[Q_t]\{t c_1\gamma(1) + O(t^{1+\frac\beta2})\}\;.
\end{equ}
Furthermore, for all $t\in(0,\tau]$
\begin{equs}[eq:E_hol_2]
\E\Trace \hol(\tilde a,\ell)\bone_{Q_t} &= \E\Trace \hol( a,\ell)\bone_{Q_t}
+t\Trace c_1\gamma(1)\P[Q_t]
+\E[L_t](h(0))
\\
&\qquad + \P[Q_t]\Big( \frac{t^2}{2} \Trace[( c_1\gamma(1))^2] + O(t^{2+\frac\beta2-N\kappa})\Big)\;,
\end{equs}
where $\E[L_t](h(0)) = O(t^{1+\beta/2})$ is linear in $h(0)$.
The proportionality constants in all bounds are $M^k$ for some $k >0 $.
\end{lemma}

\begin{proof}
We take $L\geq \frac6\beta$
 in Lemma~\ref{lem:a_tilde_a_hol}
so that $|b|_{\gr\alpha}^{L} = O(t^3)$ on the event $Q_t$.
Lemma \ref{lem:a_tilde_a_diff} implies
$
|a|_{\gr{\alpha}}\leq |b|_{\gr\alpha} + O(t^{1/2-})
$, so
we also obtain $|a|_{\gr\alpha}^{L} = O(t^3)$.
The conclusion follows from Lemma~\ref{lem:a_tilde_a_hol}
upon taking expectations.
\end{proof}

\begin{proof}[of Proposition~\ref{prop:A_tilde_A}]
The idea is to choose $t\downarrow0$ and $M\uparrow\infty$ together in a suitable way and use Lemma~\ref{lem:restricted_hol} with a suitable choice of $g(0)$ for each sufficiently small $t$.
% to exhibit a difference between $\E[W_\ell(\tilde a)] $ and $\E[W_\ell(a)]$.
We take for now general $g(0)\in \CC^\infty(\T^2,G)$
for which we will make specific choices in different cases below.
Note that $a,\tilde a$ from Lemma~\ref{lem:restricted_hol} are $A,\tilde A$
from the proposition statement.
We often omit dependence on $t$ in the proof.

Let $k$ be as in Lemma \ref{lem:restricted_hol}.
We let $\eta\geq q$ so that for $t\asymp M^{-\eta}$ we have $t<\tau$.
We can suppose $\frac\beta2-N\kappa>0$ and we further take $\eta$ large such that $(\frac\beta2 - N\kappa)\eta-k\gg 1$,
and thus $O(t^{\frac\beta2-N\kappa})$ in Lemma~\ref{lem:restricted_hol} becomes $O(t^{0+})$, this time \textit{uniform} in $M\geq 2$.
Furthermore, by Markov's inequality, we obtain for $t\asymp M^{-\eta}$
\begin{equ}\label{eq:Markov}
\E[W_\ell(\tilde a)\bone_{Q_t^c}] \lesssim \P[Q_t^c] \lesssim  M^{-3\eta} = O(t^3)\;,
\end{equ}
uniform in $M\geq 2$,
and similarly for $\E[W_\ell(a)\bone_{Q_t^c}]$.
Let us write, as usual, $c_1 A = c^{(1)}_1A_1 + c^{(2)}_1A_2$ with $c^{(i)}_1\in L(\mfg,\mfg)$.

\textit{Case 1: $c^{(1)}_1 \neq 0$.}
We define $u, \tilde u \in \CC^\infty(\T^2,G)$ as follows.
Let $\zeta\in\CC^\infty([0,1],\mfg)$ be as
in Lemma~\ref{lem:curve_selection}
with $J$ therein taken as $J\eqdef Q\circ c^{(1)}_1\in L(\mfg,\R)$ where $Q\in L(\mfg,\R)$ is such that $J\neq 0$ and arbitrary otherwise.
For all $x,y\in [0,1)$, we let $u(0,y)=\id$
and $(\d_1 u)u^{-1}(x,y) = \dot\zeta(x)$.
Remark that these conditions uniquely determine $u$, and the conditions that $\dot\zeta=0$ in a neighbourhood of $0$ and $1$ (Lemma~\ref{lem:curve_selection}\ref{pt:const_deriv})
and $L^\zeta(0)=L^\zeta(1)=\id$
(Lemma~\ref{lem:curve_selection}\ref{pt:targets})
ensure that $u\colon\T^2\to G$ is smooth.
Furthermore $\d_2 u(x,y)= 0$.
We define $\tilde u$ in exactly the same way except $(\d_1 \tilde u) \tilde u^{-1}(x,y)=4\dot\zeta(x)$, and remark that $\tilde u\colon \T^2\to G$ is smooth now due to Lemma~\ref{lem:curve_selection}\ref{pt:4targets}.

Denote $h = (\mrd u) u^{-1}$ and define $\gamma_j(x)=\int_0^x h_j(y,0)\mrd y$ for $j=1,2$ as in Lemma~\ref{lem:a_tilde_a_hol},
and likewise for $\tilde h$ and $\tilde\gamma_j$.
Remark that, by construction,
$h_2=\tilde h_2=0$, thus $\gamma_2 = \tilde\gamma_2=0$,
and further that $\tilde \gamma_1=4\gamma_1=4\zeta$ and $\tilde h_1=4h_1$,
and by the condition in 
Lemma~\ref{lem:curve_selection}\ref{pt:non-zero}
$
c_1\tilde \gamma(1) = 4c_1\gamma(1) =  4c_1^{(1)}\zeta(1) \neq 0
$.

\textit{Case 1A.}
Suppose that $\Trace c_1^{(1)}\zeta(1) \neq 0$.
We then take $g(0)=u$ and observe that, by Lemma~\ref{lem:restricted_hol}, uniformly in $M\geq 2$ for $t\asymp M^{-\eta}$
\begin{equ}
\E[W_\ell(\tilde a)\bone_{Q_t}] = \E[W_\ell(a)\bone_{Q_t}] + \P[Q_t]\{t\Trace c_1^{(1)}\zeta(1) + O(t^{1+})\}\;.
\end{equ}
Combining with~\eqref{eq:Markov}, we obtain uniformly in $M\geq2$ for $t\asymp M^{-\eta}$
\begin{equ}
\E[W_\ell(\tilde a)] = \E[W_\ell(a)] + t\Trace c_1^{(1)}\zeta(1) + O(t^{1+})\;.
\end{equ}
It follows that $|\E[W_\ell(\tilde a)] - \E[W_\ell(a)]| \geq \sigma t$ for some fixed $\sigma >0$ and all $t>0$ sufficiently small.
This completes the proof in \textit{Case 1A}.

\textit{Case 1B.} Suppose now that $\Trace c_1^{(1)}\zeta(1) = 0$.
Taking for now $g(0)= u$,
we obtain by Lemma~\ref{lem:restricted_hol} uniformly in $M\geq 2$ for $t\asymp M^{-\eta}$
\begin{equ}
\E[W_\ell(\tilde a)\bone_{Q_t}] = \E[W_\ell(a)\bone_{Q_t}] + \E{L_t}(h) + \P[Q_t]\Big\{\frac{t^2}{2}\Trace[(c_1^{(1)}\zeta(1))^2] + O(t^{2+}) \Big\}
\end{equ}
where $\E[L_t](h)=O(t^{1+})$ is linear in $h$.
Again combining with~\eqref{eq:Markov}, we obtain 
uniformly in $M\geq2$
\begin{equ}
\E[W_\ell(\tilde a)] = \E[W_\ell(a)] + \E{L_t}(h) + \frac{t^2}{2}\Trace[(c_1^{(1)}\zeta(1))^2] + O(t^{2+})\;.
\end{equ}
Let $r=|\Trace[(c_1^{(1)}\zeta(1))^2]|$ and note that $r>0$ since $c_1^{(1)}\zeta(1)\neq0$.

Now, for every $t$, we consider two cases.
The first case is $|\E{L_t}(h)| \notin [\frac14 t^2r,t^2r]$,
in which case $\E{L_t}(h) + \frac{t^2}{2}\Trace[(c_1^{(1)}\zeta(1))^2]$ is at least of order $t^2$,
so we are done due to our choice of $g(0)$.

The second case is $|\E{L_t}(h)| \in [\frac14 t^2r,t^2r]$,
in which case we take instead $g(0) = \tilde u$.
%Note that now
%$\gamma=\gamma^{(2)}$ for $\gamma$ as in Lemma~\ref{lem:a_tilde_a_hol}.
By linearity of $h\mapsto \E L_t (h)$
and the fact that $\tilde\gamma_1=4\gamma_1=4\zeta$ and $\tilde h=4h$ and $\tilde \gamma_2=0$,
we have
\begin{equ}
|\E{L_t}(\tilde h) + \frac{t^2}{2}\Trace[(c_1\tilde \gamma(1))^2]| \geq (4^2/2-4)t^2 r=4 t^2r\;.
\end{equ}
Therefore, for every $t>0$ sufficiently small, either $g(0)=u$ or $g(0)=\tilde u$ yields $|\E[W_\ell(\tilde a)]-\E[W_\ell(a)]| \geq \sigma t^2$,
which completes the proof in \textit{Case 1B}.

\textit{Case 2: $c^{(1)}_1 = 0$ and $c^{(2)}_1\neq 0$.}
Define $u\in \CC^\infty(\T^2,G)$ as follows.
Let $X\in\mfg$ such that $c^{(2)}_1 X \neq 0$.
Consider smooth $Z\colon [-\frac12,\frac12]\to [0,1]$ such that $Z(0)=Z(y)=0$ for all $|y|>\frac14$
and $\dot Z(0)=1$.
We then take $u$ such that $\d_1 u=0$ and, for $y\in[0,\frac14]$,
$
u(0,y) = e^{Z(y)X}$ and $u(0,1-y) = e^{Z(-y)X}$.
%Note that this is equivalent to
%\begin{equ}
%\d_2 g^{(1)}(0,y) = \dot Z(y) X g^{(1)}(0,y)\;, \quad \d_2 g^{(1)}(0,1-y) = -Z(y) X g^{(1)} (0,y)
%\end{equ}
%(We identify $\T^2\simeq [0,1)^2$ as usual.)
We further set $u(0,y)= 1$ for $y\in[\frac14,\frac34]$.
With these definitions, $u\colon \T^2 \to G$ is smooth
and
$
h_2(x,0)\eqdef (\d_2 u) u^{-1}(x,0) = X
$ for all $x\in [0,1]$.
In particular, $\int_0^1 c^{(2)}_1 h_2(x,0) \mrd x = c^{(2)}_1 X$.
We then define $\tilde u$ exactly as above but with
$
\tilde u(0,y) = e^{4Z(y)X}$ and $\tilde u(0,1-y) = e^{4Z(-y)X}$,
so that $\int_0^1 c^{(2)}_1 \tilde h_2(x,0)\mrd x = 4c^{(2)}_1 X$.
The conclusion now follows exactly as in \textit{Case 1} upon subdividing into the cases $\Trace c^{(2)}_1X\neq 0$ and $\Trace c^{(2)}_1X = 0$.
\end{proof}

\begin{remark}\label{rem:Abelian_vs_simply_conn}
For \textit{Case 1} in the above proof,
one should contrast  two  situations: $G$ is Abelian, and $G$ is simply connected.
If $G$ is Abelian, we are always in \textit{Case 1A} and
the existence of $\zeta$ relies crucially on the non-contractibility of $\ell$ (for $G=\U(1)$ and $\R^2$, which is simply connected, Theorem \ref{thm:A_tilde_A} fails~\cite[Rem.~1.16]{Chevyrev22YM}).
If $G$ is simply connected, so $\mfg$ is semi-simple,
we are always in \textit{Case 1B}
and it is not crucial that $\ell$ is non-contractible since almost the same argument implies the existence of $u,\tilde u$ with the desired properties for any smooth simple loop $\ell$;
for $G$ simply connected, we therefore expect our argument to apply when $\T^2$ is replaced by an arbitrary compact manifold.
(In \textit{Case 2}, the nature of $G$ is not important.)
\end{remark}

\subsection{Proof of Theorem~\ref{thm:discrete_dynamics}}\label{subsec:proof_C=bar_C}

\begin{definition}\label{def:varsigma_N}
For $N\geq1$ and $\alpha\in(\frac12,1]$, define the map $\varsigma_N\colon \Omega_{\gr\alpha}\to\log Q_N \subset \mfq_N$ by $(\varsigma_N a)(b)=\log \hol(a,b)$ for all $b\in\obonds_N$.
\end{definition}

\begin{lemma}\label{lem:piecewise_hol}
Let $\alpha\in(\frac12,1]$ and $a\in \Omega_{\gr\alpha}$.
Then
\begin{equ}
|\varsigma_N a-\pi_N a|_{N;\gr\alpha}  \lesssim 2^{N(1-2\alpha)}|a|_{\gr\alpha}^2\;,\quad
|\varsigma_N a-\pi_N a|_{N;\alpha;\rho}  \lesssim 2^{N(1-3\alpha/2)}|a|_{\gr\alpha}^2\;.
\end{equ}
\end{lemma}

\begin{proof}
For $p\in [1,2)$ and $\gamma\in \CC^{\var p}([0,1],\mfg)$ with $\gamma(0)=0$,
we obtain from \eqref{eq:V_gamma_expansion} that
$|E^\gamma(1)-\id - \gamma(1)| \lesssim |\gamma|_{\var p}^2+|\gamma|_{\var p}^3$.
Since $e^X = \id + X + O(X^2)$ and $\log(G)$ is bounded,
we have
$|\log E^\gamma - \gamma(1)| \lesssim |\gamma|_{\var p}^2$.

Taking $\gamma=b_a$, so that $\hol(a,b)= E^{\gamma}$, and using
$
|\gamma|_{\var{(1/\alpha)}} \leq |\gamma|_{\Hol\alpha} \leq 2^{-N\alpha}|a|_{\gr\alpha}
$,
we obtain $|(\varsigma_N a)(b)-(\pi_N a)(b)|\lesssim 2^{-2N\alpha}|a|_{\gr\alpha}^2$.
Consequently, %for a line $\ell\in\olines_N$,
\begin{equ}
|(\varsigma_N a - \pi_N a)(\ell)| = \Big|\sum_{b\in \ell} (\varsigma_N a)(b)-(\pi_N a)(b)\Big|\leq 2^{N}|\ell|2^{-2N\alpha}|a|_{\gr\alpha}^2 \;,
%\leq 2^{N(1-2\alpha)}|\ell|^\alpha|a|_{\gr\alpha}^2
\quad
\forall \ell\in\olines_N
\;.
\end{equ}
Likewise, if $\bar\ell\neq\ell$ are parallel, then $|(\varsigma_N a - \pi_N a)(\ell)-(\varsigma_N a - \pi_N a)(\bar \ell)|$ is bounded by a multiple of
\begin{equs}
%\Big|\sum_{b\in \ell}\sum_{\bar b\in\bar\ell} a^{(N)}(b)-a^{(N)}(\bar b)-(\pi_N a)(b)+(\pi_N a)(\bar b)\Big|
2^{N}|\ell|2^{-2N\alpha}|a|_{\gr\alpha}^2
\leq d(\ell,\bar\ell)^{\alpha/2}
2^{N\alpha/2}|\ell|^{\alpha/2}2^{N(1-2\alpha)}|a|_{\gr\alpha}^2\;,
%=\rho(\ell,\bar \ell)^\alpha2^{N(1-3\alpha/2)}|a|_{\gr\alpha}^2\;.
\end{equs}
from which the conclusion follows.
\end{proof}

\begin{proof}[of Theorem~\ref{thm:discrete_dynamics}]
We first show that $C=\bar C$ for $C$ appearing in Theorem~\ref{thm:some_C}
for any sequence $\e=2^{-N}\downarrow0$ as chosen in Assumption~\ref{as:subsequence}.
Take discrete approximations $a^{(N)}\eqdef \varsigma_N a \in \log Q_N$. 
Then, by Lemma~\ref{lem:piecewise_hol}, $a^{(N)}$ and $a$ satisfy the assumptions of Theorem~\ref{thm:some_C} and therefore
there exists $C\in L(\mfg^2,\mfg^2)$ such that, for a random time $T^*>0$ with inverse moments of all orders,
\begin{equ}\label{eq:approx_discrete_A}
\lim_{N\to\infty}\|A^{(N)} - \pi_N A\|_{\CC^{T^*}_{\eta-\alpha/2}(\Omega_{N;\alpha})} =0
\quad \text{in probability,}
\end{equ}
where $A=\SYM(C,a)$ and $A^{(N)} = 2^N\log U^{(N)}$.

Consider henceforth any $K>0$, lattice loop $\ell$,
and $b\sim a$ with $|a|_{\CC^1},|b|_{\CC^1}<K$.
By Theorem~\ref{thm:A_tilde_A}, since $C,\bar C\in L_G(\mfg^2,\mfg^2)$, to prove that $C = \bar C$, it suffices to show
\begin{equ}\label{eq:equal_condition}
|\E W_\ell \SYM_t(C,a) - \E W_\ell \SYM_t(C,b)| = o(t^2)
\end{equ}
for all $t>0$ sufficiently small, where the proportionality constant can depend on $K,\ell$ but not on $a,b$.
(In fact, it suffices to consider $\ell$ from Theorem~\ref{thm:A_tilde_A} and $a=0$ and $b=0^g$ with $|g|_{\CC^3}<K$, but the more general case is no harder.)

The idea is to use the approximation~\eqref{eq:approx_discrete_A} together with the gauge-covariance of the discrete dynamic to show~\eqref{eq:equal_condition}.
Recall the process $\check U$ from~\eqref{eq:discrete_hat_U}.
Recall that, if $\check U_0=U_0$, then $U_t=\check U_t$ for all $t<\varpi$, where $\varpi$ is the exit time~\eqref{eq:varpi_def}.

Define $b^{(N)}=\varsigma_N b$.
Recalling the identity for $(x,y)\in\bonds$
\begin{equ}
g(x)\hol(a,(x,y))g(y)^{-1}=\hol(a^g,(x,y))\;,
\end{equ}
we obtain $b^{(N)}\sim a^{(N)}$, i.e. $\exp b^{(N)} \sim \exp a^{(N)}\in Q_N$ (this is where we use the definition $a^{(N)}=\varsigma_N a$ 
and would not be true if $\varsigma_N$ is replaced by $\pi_N$).
Denoting $\check U^u$ the solution to~\eqref{eq:discrete_hat_U} with $U_0=u$ and
$
\check A^{(N)} \eqdef \log \check U^{\exp a^{(N)}}$
and $\check B^{(N)} \eqdef \log \check U^{\exp b^{(N)}}$,
it follows from Proposition~\ref{prop:gauge_covar}\ref{pt:DeTurck_to_DeTurck}
that $[\check B^{(N)}_t] \eqlaw [\check A^{(N)}_t]$ for all $t\geq 0$.
In particular, for $N$ sufficiently large such that $W_\ell$ is well-defined on $\Omega_{N}$,
\begin{equ}\label{eq:tilde_A_tilde_B_equal}
\E W_\ell(\check A^{(N)}_t) = \E W_\ell(\check B^{(N)}_t)\;.
\end{equ}
On the other hand, for any $M>0$,
\begin{equ}[eq:A_lim_A_N]
|\E W_\ell(A(t)) - \E W_\ell(A^{(N)}(t))|
\lesssim \E \bone_{t<T^*}\big|W_\ell(A(t)) - W_\ell(A^{(N)}(t))\big| 
+ o(t^M)\;.
%\\
%&\eqdef \Theta(t,N) + \P[t> T^*] 
%= \Theta(t,N)+ o(t^{M}) \;,
\end{equ}
where we used $\P[t> T^*] = o(t^M)$ due to $\E [(T^*)^{-M}]<\infty$
and Markov's inequality.
Furthermore,~\eqref{eq:approx_discrete_A} implies that, for every fixed $t>0$,
\begin{equ}\label{eq:Theta_to_zero}
\lim_{N\to\infty}\E \bone_{t<T^*}\big|W_\ell(A(t)) - W_\ell(A^{(N)}(t))\big|  = 0\;.
\end{equ}
The same statement holds for $B$ and $B^{(N)}$.

Finally, since $A^{(N)}_t=\check A^{(N)}_t$ for $t<\varpi$
and since $\varpi$, by the same reason as for $T^*$, has inverse moments of all orders bounded uniformly in $N$, the same argument implies that, for any $M>0$,
\begin{equ}\label{eq:A_tilde_A_bound}
\sup_{N \geq 1}\big|\E W_\ell(\check A^{(N)}(t)) - \E W_\ell(A^{(N)}(t))\big| = o(t^M)\;,
\end{equ}
and the same for $B$ and $B^{(N)}$.
Combining~\eqref{eq:tilde_A_tilde_B_equal},~\eqref{eq:A_lim_A_N},~\eqref{eq:Theta_to_zero}, and~\eqref{eq:A_tilde_A_bound}, we conclude that
$
|\E W_\ell(A(t)) - \E W_\ell(B(t))| = o(t^{M})
$
for any $M>0$, which in particular implies~\eqref{eq:equal_condition}.
Therefore $C=\bar C$ as claimed.

The proof of Theorem~\ref{thm:some_C} (see in particular~\eqref{eq:T_N_conv})
now reveals that, for every $K\geq 1$ and $\delta>0$, there exists $N_0>0$ such that
\begin{equ}
\sup_{N>N_0} \P[\|A^{(N)} - \pi_N A\|_{\CC([0,\tau_{K}],\Omega_{N;\alpha})}  > \delta] < \delta\;,
\end{equ}
where $\tau_{K} \eqdef K\wedge \inf\{t>0\,:\, |A(t)|_\alpha > K\}$.
The proof readily follows.
\end{proof}

\section{Moment estimates on the gauge-fixed YM measure}
\label{sec:moment_estimates}

In this section, we prove the following moment bounds on the gauge-fixed YM measure on $\T^2$
that will be used in the proof of Theorem~\ref{thm:invar_measure}.
%Namely, we show that there is a gauge-fixed representation $A$ of the YM measure such that $|A|_\alpha$ has moments of all orders bounded uniformly in the lattice spacing.

\begin{theorem}\label{thm:YM_gauge_fixed}
Suppose that $G$ is simply connected and that
Assumption~\ref{assum:P_N} on $S_N$ holds.
Define the probability measure $\mu_N$ on $Q_N \simeq G^{\obonds_N}$ by~\eqref{eq:mu_N}.
Let $U$ be a random variable with law $\mu_N$.
Then, on the same probability space, there exists a $\mfG_N$-valued random variable $g$ such that for all $\beta\in[1,\infty)$, $\alpha\in (0,1)$, and $N\geq 1$,
one has 
$
\E|\log U^g|^\beta_\alpha \leq C
$,
where $C$ depends only on $\alpha,\beta,G$ and the parameters in Assumption~\ref{assum:P_N}.  
\end{theorem}
We give the proof of Theorem~\ref{thm:YM_gauge_fixed} at the end of this section.
This result quantifies and extends the main result of~\cite{Chevyrev19YM}, where only tightness of $|\log U^g|_\alpha$ is shown.
Standard actions listed in Section~\ref{sec:approx-YM} (Manton, Wilson, Villain) satisfy Assumption~\ref{assum:P_N},
see Example~\ref{ex:verify_S_N_assump}.
The result is new even for the Villain and Wilson actions and the proof relies on rough path estimates from Appendix~\ref{app:Gaussian_tails}.

The assumption that $G$ is simply connected is used in only one step in this section, which is Lemma~\ref{lem:axial}.
This assumption, however, is crucial for Theorem~\ref{thm:YM_gauge_fixed} and the result is not true without it.

\subsection{Lattice gauge-fixing and rough Uhlenbeck compactness}

In this subsection, we improve the analysis of the deterministic gauge-fixing procedure from~\cite[Sec.~4]{Chevyrev19YM}.\footnote{While the core idea is taken
from~\cite{Chevyrev19YM}, we prefer to make our presentation essentially self-contained with the aim to (i) facilitate reading, and (ii) point out and correct several errors in~\cite{Chevyrev19YM}.}
See also~\cite[Sec.~6]{ChandraChevyrev22}, where this procedure is carried out
in the much simpler Abelian case and with stronger estimates.

\begin{definition}
An $N$-\textit{rectangle} is a triple $r=(x,m2^{-N}e_1,n2^{-N}e_2) \in\Lambda_N \times \R^2\times\R^2$ where $1\leq m,n\leq 2^N$ and $m\wedge n=1$.
We denote by $|r|=mn2^{-2N}$ the area of $r$.
We let $\rect_N$ denote the set of $N$-rectangles.
\end{definition}
Because we are in dimension $d=2$, for a plaquette $p\in\plaq_N$, the set $p\cap \Lambda_{N-1}$ is a singleton $\{z\}$. We call $z$ the \textit{origin} of $p$. 

Every $r=(x,m2^{-N},n2^{-N})\in\rect_N$ can be seen as a rectangular subset of $\T^2$ with $x$ in the lower-left corner.
We also treat $r$ as a subset of $\plaq_N$ containing $k\eqdef mn$ plaquettes $p_1,\ldots, p_k$ of $\Lambda_N$.
We order these plaquettes in the following way.
If $k=1$, there is nothing to order.
If $k>1$, then the boundary of $r$ contains two lines $\ell,\bar\ell$ consisting of 2 or more bonds.
Precisely one of these lines, say $\ell$, is contained in the `grid'  
\begin{equ}\label{eq:grid_def}
\Grid_{N-1} = \{x+te_i2^{-N+1}\,:\, x\in\Lambda_{N-1}\,,\,t\in[0,1]\,,\,i=1,2\} \subset \T^2\;.
\end{equ}
We call $r$ \textit{lower} or \textit{upper} if $\ell$ is horizontal (so $m>1$ and $n=1$) and is below or above $\bar \ell$ respectively.
Similarly we call $r$ \textit{leftward} or \textit{rightward} if $\ell$ is vertical (so $n>1$ and $m=1$) and is to the left or right of $\bar \ell$ respectively.
See Figure~\ref{fig:rectangles}.

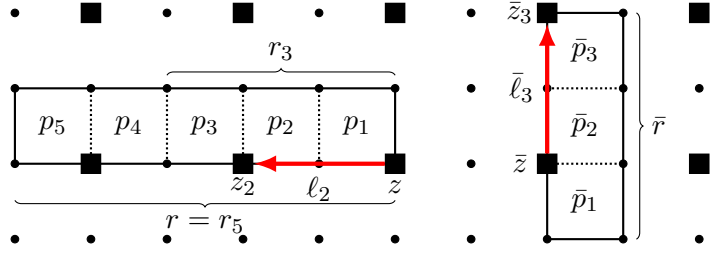
\begin{figure}[t]
\centering
\begin{tikzpicture}[scale = 2.0]

%skeleton of r
\draw[thick] (0.0,0.5)--++(2.5,0)--++(0,0.5)--++(-2.5,0)--++(0,-0.5);

%skeleton of \bar r
\draw[thick] (3.5,0.0)--++(0.5,0)--++(0,1.5)--++(-0.5,0)--++(0,-1.5);

%dots and plaquettes for r
\foreach \x in {0,...,3}{
	\draw[thick,densely dotted] (0.5+\x/2,0.5)--++(0,0.5);
}
\foreach \x in {1,...,5}{
	\node (p\x) at (2.25 + 0.5 - \x/2,3/4) {$p_{\x}$};
}

%dots and plaquettes for \bar r
\foreach \x in {0,...,1}{
	\draw[thick,densely dotted] (3.5,0.5+\x/2)--++(0.5,0.0);
}
\foreach \x in {1,...,3}{
	\node (p\x) at (3.75,1/4 - 0.5 + \x/2) {$\bar p_{\x}$};
}

%tags for r
\draw[decorate, decoration={brace}, yshift=0.5ex]  (1.0,1) -- node[above=0.4ex] {$r_3$} (2.5,1);
\draw[decorate, decoration={brace,mirror}, yshift=-1.4ex]  (0.0,0.5) -- node[below=0.4ex] {$r=r_5$}  (2.5,0.5);

%tags for \bar r
\draw[decorate, decoration={brace,mirror}, xshift= 0.5ex]  (4,0.0) -- node[right=0.4ex] {$\bar r$}  (4.0,1.5);

%lattice
\foreach \x in {0,...,4}{
  \foreach \y in {0,...,1}{
	\node [draw, fill, name=square\x\y] at (\x+0.5,\y+0.5) {};
	\node [draw, dot, name=c\x\y] at (\x,\y) {};
	\node [draw, dot, name=c\x\y] at (\x+0.5,\y) {};
	\node [draw, dot, name=c\x\y] at (\x,\y+0.5) {};
%    \fill[black] (\x,\y) circle (0.02);
  }
}

%line for r
\draw[ultra thick, red, -{Latex[length=3mm]}] (square20) -- (square10)
node[draw=none, black, midway, below=0.5]{$\ell_2$};
%node[draw=none, midway, left=0.5]{$b_5$};

%line for \bar r
\draw[ultra thick, red, -{Latex[length=3mm]}] (square30) -- (square31)
node[draw=none, black, midway, left=0.5]{$\bar \ell_3$};;

%origins for r
\node (barz) at (2.5,0.5-0.15) {$z$};
\node (barz3) at (1.5,0.5-0.15) {$z_2$};

%origins for \bar r
\node (barz) at (3.5-0.18,0.5) {$\bar z$};
\node (barz3) at (3.5-0.18,1.5) {$\bar z_3$};
\end{tikzpicture}
\caption{
Example of lower rectangle $r$ with ordered plaquettes $p_1,\ldots, p_5$ and leftward rectangle $\bar r$ with ordered plaquettes $\bar p_1,\bar p_2,\bar p_3$.
Squares are points in $\Lambda_{N-1}$, dots are points in $\Lambda_N\setminus\Lambda_{N-1}$.
Thick red arrow on the left denotes $\ell_2=\ell_3$,
and thick red arrow on the right denotes $\bar \ell_3$.
We have $\ell_1 = \protect\emptyset$ and $\bar\ell_1=\bar\ell_2 = \protect\emptyset$.
The origin of $r$ (and of $p_1$) is $z$ and the origin of $p_2,p_3$ is $z_2=z_3$.
The origin of $\bar r$ (and of $\bar p_1,\bar p_2$) is $\bar z$ and the origin of $\bar p_3$ is $\bar z_3$.
}
\label{fig:rectangles}
\end{figure}

If $r$ is lower (resp. upper) we order $p_1,\ldots, p_k$ from the right to left (resp. left to right).
If $r$ is leftward (resp. rightward) we order $p_1,\ldots, p_k$ from the bottom to top (resp. top to bottom).
We call $p_1,\ldots, p_k$ the \textit{ordered plaquettes} of $r$.
The origin of $r\in \rect_N$ is defined as the origin of $p_1$.  

Throughout this section, for $U\in Q_N$ and $r=(x,m2^{-N}e_1,n2^{-N}e_2)$, we define the holonomy of $U$ around $r$ by
$U(\d r) = U(b_1)\ldots U(b_j)$, where $b_1,\ldots,b_j\in\bonds_N$ are the bonds contained in the boundary of $r$ (treated in the obvious way as a subset of $\T^2$)
oriented counter-clockwise and starting at the origin of $r$.
In particular, plaquettes are oriented counter-clockwise.

(In the case that $m=2^N$ and $n=1$, the boundary of $r$ consists of horizontal bonds and a single vertical bond, $(x,x+ 2^{-N}e_2)$.
If $z$, the origin of $r$, is on this vertical bond then there are two ways to traverse the boundary counter-clockwise starting from $z$.
In this case we view the boundary of $r$ as a closed simple curve in $\R^2$,
in which case there is only one way to traverse the boundary counter-clockwise starting from $z$, and we choose this way.
A similar consideration applies in the case that $m=1$ and $n=2^N$.)

\begin{definition}
For $U\in Q_N$, $r\in\rect_N$ with ordered plaquettes $p_1,\ldots,p_k$, and $1\leq i\leq k$, 
let $r_i\in\rect_N$ be the unique $N$-rectangle whose ordered plaquettes are $p_1,\ldots, p_i$, see Figure~\ref{fig:rectangles}.
The \textit{anti-development} of $U$ along $r$ is the $\mfg$-valued sequence $X=\{X_i\}_{i=0}^k$ with $X_0=0$ and $X_i-X_{i-1} = \log (U(\d r_{i-1})^{-1}U(\d r_i))$ where $U(\partial r_0)\eqdef \id$.   
\end{definition}
As the name suggests, the development into $G$ of the anti-development $X$ is $U(\d r_i)=e^{X_1-X_0}\ldots e^{X_i-X_{i-1}}$.
To explain our choice of ordering $p_1,\ldots, p_k$,
observe that $U(\d r_i) = x_1\ldots x_i$ where $x_i = U(\ell_i) U(\d p_i) U(\ell_i)^{-1}$,
where $\ell_i=(z,z_i-z)\in\lines_N$ is the (oriented) line connecting $z\in\Lambda_{N-1}$, the origin of $r$, to $z_i\in\Lambda_{N-1}$, the origin of $p_i$ (see again Figure~\ref{fig:rectangles}),
and $U(\ell_i)=\prod_{b\in\ell_i}U(b)$ is the ordered product of oriented bonds $b\in\bonds_N$ as they appear in $\ell_i$.
In particular, $U(\d r_{i-1})^{-1} U(\d r_i) = x_i$, and thus   
\begin{equ}\label{eq:anti_dev_identity}
X_i - X_{i-1} = \Ad_{U(\ell_i)}(\log U(\d p_i))\;.
\end{equ}
Note that $\ell_i$ for all $i=1,\ldots, k$ is contained in the grid $\Grid_{N-1}$, which will be important in the proof of Lemma~\ref{lem:q-var_bound} below.
The identity~\eqref{eq:anti_dev_identity} shows that the increments
$X_{i}-X_{i-1}$ are $\log$ of \textit{lassos}~\cite[Sec.~2.3]{Levy03}
-- see also~\cite{Gross85}
where this concept is introduced in the continuum.

Much like $\{\Omega_N\}_{N\geq 1}$, $\{Q_N\}_{N\geq 1}$ forms a projective system
with projections $\pi_{N,M}\colon Q_N\to Q_M$
for $M\leq N$, where $(\pi_{N,M} U) (x,x\pm e_i2^{-M})$ is given by
$\prod_{k=0}^{2^{N-M}-1}  U(x\pm ke_i 2^{-N}, x\pm 
(k+1)e_i2^{-N})$.
%\begin{multline*}
%(\pi_{N,M} U) (x,x\pm e_i2^{-M}) = U(x,x\pm e_i2^{-N})U(x\pm e_i2^{-N},x\pm 2e_i2^{-N})
%\\
%\cdots U(x\pm (2^{N-M}-1)e_i 2^{-N}, x\pm 2^{N-M}e_i2^{-N})\;.
%\end{multline*}
We keep $\pi_{N,M}$ implicit and treat every $U\in Q_N$ as defining an element of $Q_M$.
In particular, for every $M\leq N$, $U\in Q_N$, and $r\in\rect_M$, the holonomy $U(\d r)$ and the anti-development of $U$ around $r$ are well-defined.

We measure `non-flatness' of $U\in Q_N$ through the gauge-invariant quantities
\begin{equ}{}
[U]_{\alpha;q,N} \eqdef \max_{0 \leq M \leq N} \max_{r\in \rect_M} |r|^{-\alpha/2} |X|_{\var q}\;,
\end{equ}
where $q \in [1,\infty)$ and $X$ is the anti-development of $U$ along $r\in\rect_M$,
and
\begin{equs}
\fancynorm{U}_{\alpha;N} &\eqdef \max_{0 \leq M \leq N} \fancynorm{U}_{\alpha;=M}\;,
\\
\fancynorm{U}_{\alpha;=M} &\eqdef \max_{j,k} |r_{M,j,k}|^{-\alpha/2}|\log U(\d r_{M,j,k})|\;,
\end{equs}
where the final maximum is over all $0\leq j < 2^{M}$ and $1\leq k\leq 2^M$ and where $r_{M,j,k}\in\rect_M$ is the rectangle $r_{M,j,k}=((0,j2^{-M}),k2^{-M}e_1,2^{-M}e_2))$.
Recall that $|X|_{\var q}$ for a sequence $X=\{X_i\}_{i=0}^k$
%(taking values in a normed space)
is defined by~\eqref{eq:p-var}
where we treat $X$ as a piecewise constant function on $[a,b)\eqdef [0,k+1)$ equal to $X_i$ on $[i,i+1)$.
%\begin{equ}
%|X|_{\var q} = \max_s \sum_{j=1}^{|s|}\Big( |X_{s_{j}}-X_{s_{j-1}}|^q \Big)^{1/q} \;,
%\end{equ}
%where the maximum is over all subsequences $s=(s_0,\ldots, s_n)$ of $(0,\ldots, k)$ and $|s|\eqdef n$
%(this is the discrete analogue of~\eqref{eq:p-var}).

One should think of $[U]_{\alpha;q,N}$ and $\fancynorm{U}_{\alpha;N}$ as substituting the YM energy in classical Uhlenbeck compactness~\cite{Uhlenbeck82,Wehrheim04}.
Indeed, the main result of this subsection is the following rough version of Uhlenbeck compactness that quantifies~\cite[Thm.~4.5]{Chevyrev19YM}.
Recall that every $U\in Q_N$ defines an element $\log U \in\Omega_N$, see~\eqref{eq:A_def_1_form} and the discussion that follows.

\begin{theorem}\label{thm:gauge_fix}
Suppose that $G$ is simply connected.
Consider $\alpha\in (\frac12,1)$, $q\in [1,\frac1{1-\alpha})$, and $\bar\alpha\in (0,\alpha)$.
There exists $C>0$, depending only on $\alpha,\bar\alpha,q,G$, with the following property.
For all $N\geq 1$ and $U\in Q_N$, there exists
$g\in \mfG_N$ such that, denoting $A=\log U^g \in\Omega_N$,  
\begin{equ}[eq:A_gr_bound]
|A|_{N;\gr{\bar\alpha}} \leq Z\eqdef C(1+\fancynorm{U}_{\alpha;N}^4) e^{C(\log(1+[U]_{\alpha;q,N}))^2}\;,
\end{equ}
and, if $\alpha\in (\frac23,1)$, then also $|A|_{N;\bar\alpha;\rho} \leq Z$. 
\end{theorem}
\begin{remark}
The reader may wonder why $[U]_{\alpha;q,N}$ involving $|X|_{\var q}$ is a natural quantity to consider.
It appears due to Young's estimate for sums/integrals, see~\eqref{eq:Young_bound}, which is used in~\eqref{eq:Delta2_bound} in the proof of Theorem~\ref{thm:gauge_fix}.
\end{remark}
\begin{remark}
Although we do not use it later,
Theorem~\ref{thm:gauge_fix} respects different scales in the following sense:
if we consider $M\leq N$, then $g$ can be chosen such that $g\restr_{\Lambda_M}$
is the transformation corresponding to $\pi_{N,M}U\in Q_M$.
\end{remark}
We give the proof of Theorem~\ref{thm:gauge_fix} at the end of this subsection.
\begin{lemma}[Axial gauge]\label{lem:axial}
Suppose that $G$ is simply connected.
There exists $C>0$ with the following property.
For all $M\geq 1$, and $U\in Q_M$, there exists $g\in\mfG_M$ such that, for all $\alpha\geq 0$,
\begin{equ}[eq:axial_gauge_bound]
\max_{b\in\bonds_M}|\log (U^g(b))| \leq C\fancynorm{U}_{\alpha;=M}2^{-M\alpha/2} + C2^{-M}\;.
\end{equ}
\end{lemma}

\begin{proof}
Throughout the proof, we let $C$ be a constant that depends on $G$ and which may change from line to line.
For $0\leq j\leq 2^{M}-1$, consider the points $y_j=(0,j2^{-M})\in \Lambda_M$ and define the holonomy $V=U(y_0,y_1)\ldots U(y_{2^M-1},y_0)$.
We define a gauge transformation $\bar g\in\mfG_M$ as follows.
We set $\bar g(y_0)=1$ and, for $1\leq j\leq 2^M-1$, we let $\bar g(y_j)$ be the unique group element such that
\begin{equ}\label{eq:g_y_j_def}
U^{\bar g}(y_j,y_{j+1}) \eqdef \bar g(y_j)U(y_j,y_{j+1}) \bar g(y_{j+1})^{-1} = \exp(2^{-M}\log V) \eqdef v
\end{equ} 
(such a choice for $\bar g(y_1),\ldots, \bar g(y_{2^M-1})\in G$ exists and is unique).
For all remaining $x\in\Lambda_M$, we set $\bar g(x)=1$.

For $1\leq k\leq 2^M-1$ and $0\leq j\leq 2^M-1$, define $y_{j,k} = (k2^{-M},j2^{-M}) \in\Lambda_M$ and the holonomies
$
u_j = U^{\bar g}(y_j,y_{j,1})U^{\bar g}(y_{j,1},y_{j,2})\cdots U^{\bar g}(y_{j,2^{M}-1}, y_j)$.
Recall that, by quantitative homotopy theory (see~\cite[Lem.~4.14]{Chevyrev19YM} or~\cite[Thm.~B]{CDMW18}),
there exist curves $\gamma_j \in C([0,1],G)$ such that $\gamma_j(0) = 1$, $\gamma_j(1)=u_j$, and
\begin{equ}\label{eq:gamma_vert_bound}
\max_j \sup_{t\in[0,1]}|\log (\gamma_j(t)^{-1}\gamma_{j+1}(t))| \leq C(\delta +2^{-M})
\end{equ}
and
\begin{equ}\label{eq:gamma_horiz_bound}
\max_{j} \sup_{\substack{s,t\in [0,1]\\ s\neq t}} |t-s|^{-1}|\log (\gamma_j(s)^{-1}\gamma_{j}(t))| \leq C(2^M\delta+1)\;,
\end{equ}
where $\delta = \max_j |\log (u_ju_{j+1}^{-1})|$.
(This is where we use that $G$ is simply connected.)
We now define $g\in\mfG_M$ by $g\equiv \bar g$ on $\{y_0,\ldots, y_{2^M-1}\}$
and by the identities
\begin{equ}\label{eq:g_final_def}
U^g(y_{j,k},y_{j,k+1}) = \gamma_j(k2^{-M})^{-1}\gamma_j((k+1)2^{-M})\;,
\end{equ}
which uniquely determine $g$. See Figure~\ref{fig:paths}.

\begin{figure}[ht]
\centering
\begin{tikzpicture}
%dots in lattice
\foreach \x in {0,...,8}{
  \foreach \y in {0,...,8}{
    \fill[black] (\x/2,\y/2) circle (0.05);
  }
}

%verticle line
\draw[thick] (0,0)--++(0,4);

%horizontal lines
\foreach \y in {0,...,8}{
\draw[thick] (0,\y/2)--++(4,0);
}

%square at 0
\foreach \x in {0,...,1}{
\foreach \y in {0,...,1}{
\node[name=zero\x\y, mark size=2.0pt] at (4*\x,4*\y) {\pgfuseplotmark{square*}};
}}
\node[name=zero] at (0,-1/4) {$y_0$};

%arrows
\draw[-{Latex[length=2mm]}] (0-1/4,0)--++(0,4)
node[draw=none, midway, left=0.2]{$V$};
\draw[very thick, red, -{Latex[length=2mm]}] (0,0) --++ (4,0)
node[draw=none, midway, below=0.0]{{\footnotesize $u_0$}};
\draw[very thick, darkergreen, -{Latex[length=2mm]}] (0,1/2)--++(4,0)
node[draw=none, midway, below=0.0]{{\footnotesize $u_1$}};
\draw[very thick, darkblue, -{Latex[length=2mm]}] (0,1)--++(2,0)
node[draw=none, midway, below=-0.5]{{\footnotesize $\gamma_2(1/2)$}};
\draw[very thick, purple, -{Latex[length=2mm]}] (0,4-1/2)--++(3,0)
node[draw=none, midway, below=-0.5]{{\footnotesize $\gamma_7(3/4)$}};
\end{tikzpicture}
\quad
\begin{tikzpicture}
%G
\draw[smooth cycle, tension=0.4, fill=white] plot coordinates{(1,2.5) (-1.5,0.5) (2,-1.5) (4,1)} node at (-0.8,1.9) {$G$};

%id
\node [draw,dot,name=id,label={180:$\id$}] at (0,0) {};

%points
\node [draw,dot,name=u0,label={0:$u_0$}] at (2.3,1) {};
\node [draw,dot,name=u1,label={45:$u_1$}] at (2,1.4) {};
\node [draw,dot,name=u2,label={90:$u_2$}] at (1.4,1.5) {};
\node [draw,dot,name=u7,label={-45:$u_7$}] at (2.2,0.4) {};

%curves
\path[->,red,thick] (id) edge [bend right] (u0);
\path[->,darkergreen,thick] (id) edge [bend right] (u1);

\begin{scope}[thick,decoration={
    markings,
    mark=at position 0.5 with {\arrow{>}}}
    ] 
	\draw[postaction={decorate},->,darkblue,dotted,bend right] (id) to[bend right=15] node[draw=none, midway, left=0.0]{{\footnotesize $\gamma_2$}} (u2);
	\draw[postaction={decorate},->,purple,dotted,bend right] (id) to[bend right=40] node[draw=none, midway, below=0.0]{{\footnotesize $\gamma_7$}} (u7);
%    \path[postaction={decorate},->,darkblue,dotted] (id) edge [bend right] node[left] {{\footnotesize $\gamma_2$}} (u2);
\end{scope}

\end{tikzpicture}
\caption{Illustration of case $M=3$ in proof of Lemma~\ref{lem:axial}. The holonomy of $U^g$ along blue and purple lines in left picture are $\gamma_2(1/2)$ and $\gamma_7(3/4)$ respectively.}
\label{fig:paths}
\end{figure}
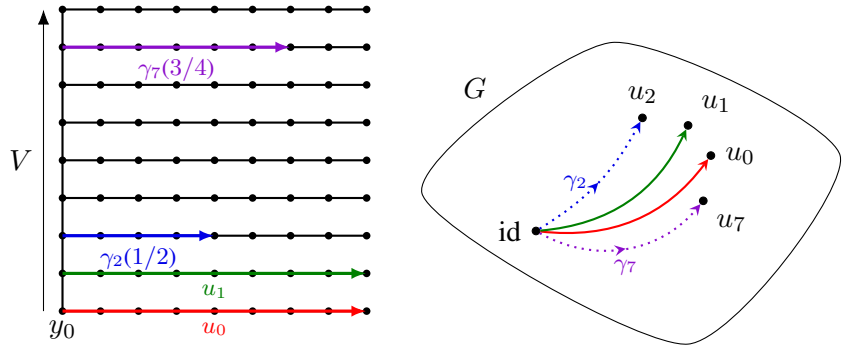

It follows from the definition of $g(y_j)=\bar g(y_j)$, namely from~\eqref{eq:g_y_j_def}, that
\begin{equ}\label{eq:y_j_bound}
|\log U^{g}(y_j,y_{j+1})| = |\log v| \leq C2^{-M}\;.
\end{equ}
%\ilya{We don't know that $\log v = 2^{-M}\log V$ for small $M$, but once $M$ is sufficiently large (depending only on $G$), then we do have this simply because $\log$ is a diffeomorphism around the identity. For $M$ small for which $\log v = 2^{-M}\log V$ might not hold, the bound is still true for some large $C$.}
Furthermore, for all $x,y\in G$, we have
$
|\log(xy)| \leq  C(|\log x| + |\log y|)$
and $|\log (y xy^{-1})| = |\log x|$.
Since $u_ju_{j+1}^{-1} = (u_j v u_{j+1}^{-1} v^{-1})v(u_{j+1}v^{-1}u_{j+1}^{-1})$, where we observe that $u_j v u_{j+1}^{-1} v^{-1}$ is a conjugate of $U^g(\partial r_{M,j,2^{M}})$,  by~\eqref{eq:y_j_bound}
we obtain for all $\alpha\geq 0$ 
%\ilya{This is due to the identity $|\log (y xy^{-1})| = |\log x|$ where $x=v^{-1}$ and $y=u_{j+1}$, so $|\log(u_{j+1}v^{-1}u_{j+1}^{-1})| \lesssim 2^{-M}$
%and then $|\log(vu_{j+1}v^{-1}u_{j+1}^{-1})| \lesssim |\log(v)| + |\log(u_{j+1}v^{-1}u_{j+1}^{-1})| \lesssim 2^{-M}$.
%Should we add more detail in the argument?}
\begin{equ}\label{eq:u_j_bound}
\delta = \max_j |\log (u_ju_{j+1}^{-1})| \leq C2^{-M\alpha/2}\fancynorm{U}_{\alpha;=M} + C2^{-M}\;.
\end{equ}
We now obtain from~\eqref{eq:gamma_vert_bound} and a similar decomposition that each `vertical' bond $
|\log U^g(y_{j,k},y_{j+1,k})|$ is bounded by the right-hand side of \eqref{eq:axial_gauge_bound}.
Finally, directly by the definition~\eqref{eq:g_final_def} and by~\eqref{eq:gamma_horiz_bound} and~\eqref{eq:u_j_bound}, we obtain the same bound for
 `horizontal' bonds
$|\log U^g(y_{j,k},y_{j,k+1})|$.
\end{proof}

\subsubsection{Binary Landau gauge}

The binary Landau gauge $U^g$ for $U\in\mfG_N$ starting from scale $M\leq N$ is defined in~\cite[Sec.~4.1]{Chevyrev19YM} (our $M,N$ are denoted by $N_0,N_1$ therein).
We now recall the definition of this gauge but using a formulation closer to that of~\cite[Sec.~6]{ChandraChevyrev22}.

Fix $1\leq M\leq N$.
We define $g$ inductively beginning with $g\equiv 1$ on $\Lambda_M$.
Suppose now $g$ is defined on $\Lambda_{n-1}$ for some $M<n\leq N$.
To extend the definition of $g$ to $\Lambda_n$, consider first $\Lambda^1_n\subset \Lambda_n$ consisting of all points in $\Lambda_{n-1}$ together with midpoints of bonds in $\obonds_{n-1}$.
For $x\in\Lambda^1_n$ which is the midpoint of $(y,z)\in\obonds_{n-1}$, we set $g(x)\in G$ as the unique element such that
\begin{equ}[e:def-g-midpoint]
U^g(y,x)=U^g(x,z) = \exp\Big(\frac12 \log U^g(y,z)\Big)\;.
\end{equ}
Now consider a point $x\in\Lambda_n\setminus\Lambda_n^1$.
Consider bonds $b_1,\ldots, b_8,a_1,\ldots,a_4\in\bonds_n$, points $y_1,\ldots, y_4\in\Lambda^1_{n}$, and plaquettes $p_1,\ldots,p_4\in\plaq_n$ as in Figure~\ref{fig:bonds_plaq_def} (recall that plaquettes are oriented counter-clockwise with origin in $\Lambda_{n-1}$).
\begin{figure}[t]
\centering
\begin{tikzpicture}[scale = 1.7]
\foreach \x in {0,2}{
  \foreach \y in {0,2}{
    \node [draw, fill, name=c\x\y] at (\x,\y) {};
  }
}

\node [draw, dot, name=d10,label=below:$y_4$] at (1,0) {};
\node [draw, dot, name=d12,label=above:$y_2$] at (1,2) {};
\node [draw, dot, name=d01,label=left:$y_3$] at (0,1) {};
\node [draw, dot, name=d21,label=right:$y_1$] at (2,1) {};

\draw[thick, -{Latex[length=3mm]}] (c00) -- (d10) node[draw=none, midway, below=0.5]{$b_6$};
\draw[thick, -{Latex[length=3mm]}] (d10) -- (c20)
node[draw=none, midway, below=0.5]{$b_7$};
\draw[thick, -{Latex[length=3mm]}] (c20) -- (d21)
node[draw=none, midway, right=0.5]{$b_8$};
\draw[thick, -{Latex[length=3mm]}] (d21) -- (c22)
node[draw=none, midway, right=0.5]{$b_1$};
\draw[thick, -{Latex[length=3mm]}] (c22) -- (d12)
node[draw=none, midway, above=0.5]{$b_2$};
\draw[thick, -{Latex[length=3mm]}] (d12) -- (c02)
node[draw=none, midway, above=0.5]{$b_3$};
\draw[thick, -{Latex[length=3mm]}] (c02) -- (d01)
node[draw=none, midway, left=0.5]{$b_4$};
\draw[thick, -{Latex[length=3mm]}] (d01) -- (c00)
node[draw=none, midway, left=0.5]{$b_5$};

\draw[thick, -{Latex[length=3mm]}] (1,1)--++(1,0)
node[draw=none, midway, above=0.5]{$a_1$};
\draw[thick, -{Latex[length=3mm]}] (1,1)--++(0,1)
node[draw=none, midway, right=0.5]{$a_2$};
\draw[thick, -{Latex[length=3mm]}] (1,1)--++(-1,0)
node[draw=none, midway, above=0.5]{$a_3$};
\draw[thick, -{Latex[length=3mm]}] (1,1)--++(0,-1)
node[draw=none, midway, right=0.5]{$a_4$};

\node (p1) at (1.5,1.5) {$p_1$};
\node (p2) at (0.5,1.5) {$p_2$};
\node (p3) at (0.5,0.5) {$p_3$};
\node (p4) at (1.5,0.5) {$p_4$};

\node [draw,dot,name=x,label={-135:$x$}] at (1,1) {};

\end{tikzpicture}
\caption{
Squares are points of $\Lambda_{n-1}$, small dots are points of $\Lambda_n$.
The points $y_1,\ldots, y_4$ are in $\Lambda^1_n$.
The oriented bonds $b_1,\ldots, b_8$ and $a_1,\ldots, a_4$ are in $\bonds_n$.
}
\label{fig:bonds_plaq_def}
\end{figure}
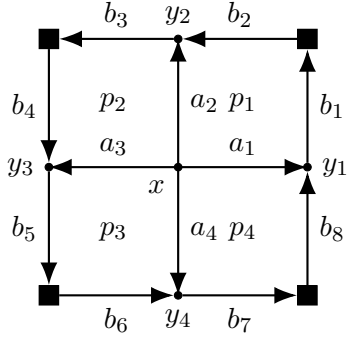
Then there exists a unique choice for $g(x)\in G$ such that, for all $i=1,\ldots, 4$,
\begin{equs}[eq:U^g_a_i_def]
\log U^g(a_i) &= \frac{\log U^g(b_{2i-3}) - \log U^g(b_{2i})}{2}
+ \frac38 (\log U^g(\d p_i) - \log U^g(\d p_{i-1}))
\\
&\qquad\qquad+ \frac18(\log U^g(\d p_{i+1}) - \log U^g(\d p_{i+2})) + \mcE_i
\end{equs}
where $\mcE_1 = 0$. Indexes for $b_j$ are modulo $8$ and those for $p_j$ are modulo $4$.
The condition $\mcE_1=0$ and the fact that $\log$ is injective determine $g(x)$ uniquely;
once $g(x)$ is determined, 
$\log U^g(a_i)$ for $i=2,3,4$ are fixed which then uniquely determine $\mcE_i$ for $i=2,3,4$.
Defining
\begin{equ}\label{eq:delta_def}
\delta = \sum_{i=1}^4 |\log U(\d p_i)| + \sum_{i=1}^8 |\log U^g(b_i)|\;,
\end{equ}
we claim that
\begin{equ}[eq:mcE_def]
\mcE_i = O(\delta^2)\;.
\end{equ}
Below we first prove \eqref{eq:mcE_def}, and then give some motivation for \eqref{eq:U^g_a_i_def}.

To prove \eqref{eq:mcE_def}, we note
\begin{equ}[eq:b_ix_j]
U^g(b_1)\cdots U^g(b_8) = x_1x_2x_3x_4
\qquad \mbox{where}\quad
 x_i = u_i U^g(\d p_i)u_i^{-1}
\end{equ}
 and $u_i$ is a product of elements of the form $U^g(\d p_j)^{-1}$ (with $j<i$) and $U^g(b_k)$ (with $k<2i$),
namely $u_1 = U^g(b_1)$ and
$u_2 = u_1 U^g(\d p_1)^{-1}U^g(b_2)U^g(b_3)$,
$u_3=u_2U^g(\d p_2)^{-1}U^g(b_4)U^g(b_5)$,
$u_4=u_3U^g(\d p_3)^{-1}U^g(b_6)U^g(b_7)$.
To see \eqref{eq:b_ix_j}, remark that the definition of $u_i$ implies directly
\[
x_1 x_2 x_3 x_4
= U^g(b_1)U^g(b_2)U^g(b_3) U^g(b_4)U^g(b_5) U^g(b_6)U^g(b_7) U^{g}(\d p_4) u_4^{-1}
\]
and then that $U^{g}(\d p_4) u_4^{-1} = U^{g}(b_8)$;
the latter identity follows immediately from the more general identity
$u_j = U^g(a_1)^{-1}U^g(a_j)U^g(b_{2j-1})$ that one can prove by induction.

The BCH formula implies that, for all $X_1,\ldots,X_n\in\mfg$,
\begin{equ}[eq:BCH]
\Big|\log (e^{X_1}\cdots e^{X_n}) - \sum_{i=1}^n X_i\Big| \lesssim \sum_{i=1}^n|X_i|^2\;,
\end{equ}
where the proportionality constant depends only on $n,G$.
It follows that
\begin{equ}\label{eq:log_sums}
\sum_{i=1}^4 \log U^g(\d p_i) = \sum_{i=1}^8 \log U^g(b_i) + O(\delta^2)\;.
\end{equ}
%\haoText{Actually, if we take $u_1 = U^g(b_1)$ and
%$u_2 = u_1  U^g(b_2)U^g(b_3)$,
%$u_3=u_2 U^g(b_4)U^g(b_5)$,
%$u_4=u_3 U^g(b_6)U^g(b_7)$, wouldn't everything also work? Why do we take the factors $U^g(\d p_1)^{-1}$, $U^g(\d p_2)^{-1}$, $U^g(\d p_3)^{-1}$ in $u_2,u_3,u_4$?}
%\ilyaText{I don't think this choice of $u_i$ will work, at least I don't see why it would.
%The choice of $u_i$ above works because when we substitute their expressions we get
%\[
%x_1 = u_1 U^g(\d p_1) u_1^{-1}
%\]
%\[
%x_1 x_2 = (U^g(b_1) U^g(\d p_i) u_1^{-1}) (u_1 U^g(\d p_1)^{-1}U^g(b_2)U^g(b_3)) U^{g}(\d p_2) u_2^{-1}
%\]
%\[
%= U^g(b_1)U^g(b_2)U^g(b_3) U^{g}(\d p_2) u_2^{-1}
%\]
%\[
%(x_1 x_2) x_3 = (U^g(b_1)U^g(b_2)U^g(b_3) U^{g}(\d p_2) u_2^{-1})(u_2U^g(\d p_2)^{-1}U^g(b_4)U^g(b_5) U^g(\d p_3) u_3^{-1})
%\]
%\[
%= U^g(b_1)U^g(b_2)U^g(b_3) U^g(b_4)U^g(b_5) U^{g}(\d p_3) u_3^{-1}
%\]
%and
%\[
%x_1 x_2 x_3 x_4
%= U^g(b_1)U^g(b_2)U^g(b_3) U^g(b_4)U^g(b_5) U^g(b_6)U^g(b_7) U^{g}(\d p_4) u_4^{-1}
%\]
%and one can check that $U^{g}(\d p_4) u_4^{-1} = U^{g}(b_8)$.
%}

On the other hand, for $i=1,2,3,4$ (indexes for $a_i$ are modulo $4$)
\begin{equ}\label{eq:bonds_prod}
U^g(a_{i+1}) = U^g(a_{i})U^g(b_{2i-1})U^g(\d p_i)^{-1}U^g(b_{2i})\;.
\end{equ}
%One can readily verify that $\mcE_i = O(\delta^2)$
%by taking $\log$ of both sides of~\eqref{eq:bonds_prod} and
%using the gauge conditions
%$U^g(b_{2i})=U^g(b_{2i+1})$
%and~\eqref{eq:U^g_a_i_def}
%together with~\eqref{eq:log_sums} and the definition $\CE_1=0$
%(see also~\cite[Lem.~4.8]{Chevyrev19YM} and~\cite[Sec.~6.1]{ChandraChevyrev22}).
One can readily verify that $\mcE_i=O(\delta^2)$ from the above identities.
Indeed, using shorthands $B_j = \log U^g(b_j)$, $P_i = \log U^g(\d p_i)$,
and taking $i=1$ in~\eqref{eq:bonds_prod}, we have by~\eqref{eq:U^g_a_i_def} and \eqref{eq:BCH}
\begin{equs}[eq:mcE_verification]{}
&\log U^g(a_2) = \frac{B_{1}-B_4}{2}
+\frac38(P_2-P_1)
+\frac18(P_3-P_4) + \mcE_2
\\
&= \frac{B_{7}-B_2}{2} +\frac38(P_1-P_4)+\frac18(P_2-P_3) + \mcE_1
+B_1 - P_1 + B_2
+O(\delta^2)\;.
\end{equs}

Since $B_2=B_3,B_4=B_5,B_6=B_7,B_8=B_1$ by~\eqref{e:def-g-midpoint}, it follows that
\begin{equ}
\mcE_2=\frac14(B_1+\cdots+B_8) - \frac14(P_1+\cdots+P_4) + \mcE_1 + O(\delta^2)\;.
\end{equ}
Then~\eqref{eq:log_sums} implies that the first two terms on the right-hand side cancel
up to order $\delta^2$, and since $\mcE_1=0=O(\delta^2)$, we obtain $\mcE_2=O(\delta^2)$.
The claims $\mcE_{3}=O(\delta^2)$ and $\mcE_4=O(\delta^2)$ follow analogously by shifting indexes in~\eqref{eq:mcE_verification}.

This choice of $g$ in \eqref{eq:U^g_a_i_def} is motivated by the fact that $
\sum_{i=1}^4 \log (U^g(a_i)) = O(\delta^2)$,
which corresponds to an approximation of the Coulomb/Landau gauge $\sum_{i=1}^2\d_iA_i=0$ in the continuum.
In fact, a naive choice with only $\frac12 (B_{2i-3}-B_{2i})$ on the right-hand side of 
\eqref{eq:U^g_a_i_def} would make this discrete divergence zero, but since one has only one degree of freedom $g(x)$ this choice would lead to
$\mcE_i$ for $i=2,3,4$ of order $O(\delta)$.
The plaquette terms are added in  \eqref{eq:U^g_a_i_def}
to cancel it, resulting in $\mcE_i$  of order $O(\delta^2)$.
This cancellation is possible thanks to  
\eqref{eq:log_sums} (which obviously holds in the abelian case without  $O(\delta^2)$).
See also \cite[Remark~4.6]{Chevyrev19YM}
for a Poisson equation heuristic.

With $U^g$ defined, we now proceed to improve the analysis of this gauge in comparison to~\cite{Chevyrev19YM}.
We denote throughout the section $A=\log U^g \in\Omega_N$.

\begin{remark}
By construction, $\log (U^g(b)) = A(b)$ for all $b\in\bonds_n$.  To see this, note that, by definition,
$A(b) = \sum_{\bar b\in\obonds_N, \bar b\in b} A(\bar b)$ while $U^g(b) = \prod_{\bar b\in\obonds_N, \bar b\in b} U^g(\bar b)$ where $\bar b\in b$ means that $\bar b$ is a sub-bond of $b$ in the obvious sense,
and then observe that $A(\bar b)$ are all equal due to~\eqref{e:def-g-midpoint}, from which we obtain $\log (U^g(b)) = A(b)$.
(The equality $\log (U(b)) = (\log U)(b)$ does not hold for arbitrary $U\in Q_N$.)
\end{remark}
The following corresponds to (and corrects)~\cite[Lem.~4.10]{Chevyrev19YM}.

\begin{lemma}[Bonds bound]\label{lem:bonds_bound}
There exist constants $c_1,c_2>0$, depending only on $G$, with the following property.
Let $\alpha\in (0,1)$ and
suppose that
\begin{equ}[eq:initial_bound]
\plaqnorm{U}_{\alpha;M,N} 2^{-M\alpha} \leq c_1\;,\quad
\text{and} \quad
\max_{b\in\obonds_M} |A(b)| \leq c_2\;,
\end{equ}
where
$
\plaqnorm{U}_{\alpha;M,N} \eqdef \max_{M\leq n \leq N} \max_{p\in\plaq_n} |p|^{-\alpha/2}|\log U(\d p)|$.
Then
\begin{equ}
\max_{b\in \obonds_n} |A(b)| \leq C2^{-n\alpha}(\plaqnorm{U}_{\alpha;M,N}+2^{\alpha M} \max_{b\in\obonds_{M}} |A(b)|)\;,
\end{equ}
for all $M\leq n\leq N$,
where $C$ depends only on $\alpha$ and $G$.
\end{lemma}

\begin{proof}
Fix any $\eta \in (0,\frac14)$ such that $\theta \eqdef (\eta+1/2)2^{\alpha}\in(2^{\alpha-1},1)$.
We first prove by induction that, if $c_1\ll c_2$ are both sufficiently small, then for all $ n\in\{M,\ldots, N\}$
\begin{equ}\label{eq:first_induction}
\delta_n \eqdef \max_{b\in\obonds_n} |A(b)| \leq c_2\;.
\end{equ}
The case $n=M$ holds by assumption~\eqref{eq:initial_bound}.
Suppose~\eqref{eq:first_induction} holds for $n\geq M$.
We can assume $\plaqnorm{U}_{\alpha;M,N}2^{-M\alpha}\leq c_1<1$
and thus obtain from~\eqref{eq:U^g_a_i_def}-\eqref{eq:mcE_def} 
\begin{equ}
\delta_{n+1} \leq \delta_{n}/2 + \plaqnorm{U}_{\alpha;M,N} 2^{-(n+1)\alpha} + C_1 \delta_{n}^2\;,
\end{equ}
where $C_1$ depends only on $G$
and where we used that $|p|=2^{-2(n+1)}$ for $p\in\plaq_{n+1}$.
Fixing now $c_2 < \eta/C_1$, we have
\begin{equ}\label{eq:second_induc}
\delta_{n+1} \leq (\eta+1/2) \delta_{n} + \plaqnorm{U}_{\alpha;M,N} 2^{-(n+1)\alpha}
\leq \frac34 c_2 + c_1\;,
\end{equ}
where we used the induction hypothesis in the
first inequality to bound $C_1\delta_{n}^2 \leq \eta\delta_{n}$ and in the final inequality to bound $\delta_n\leq c_2$.
Taking $c_1 < \frac14 c_2$, we conclude that~\eqref{eq:first_induction} holds with $n$ replaced by $n+1$ with this choice of $c_1,c_2$.
We now obtain from the first bound in~\eqref{eq:second_induc} and another induction that, for all $n\in \{M,\ldots,N\}$,
\begin{equs}
\delta_n
&\leq (\eta+1/2)^{n-M}\delta_M + \plaqnorm{U}_{\alpha;M,N}  2^{-n\alpha} \sum_{k=0}^{n-M-1} \theta^k
\\
&\lesssim 2^{-\alpha (n-M)} \delta_M + \plaqnorm{U}_{\alpha;M,N} 2^{-n\alpha}\;,
\end{equs}
where we used
$(\eta+1/2) < 2^{-\alpha}$ and $\theta<1$ in the second inequality.
\end{proof}
The following lemma is essentially~\cite[Lem.~4.11]{Chevyrev19YM}.\footnote{The convention
in~\cite{Chevyrev19YM} for ordering plaquettes differs from ours, which is a typo therein.}
\begin{lemma}\label{lem:q-var_bound}
For any $\beta\in (\frac12,1]$ and $q\in [1,\frac1{1-\beta})$,
there exists $C>0$ such that for all $r\in\rect_n$ with ordered plaquettes $p_1,\ldots, p_k$
\begin{equ}
\Big| \sum_{i=1}^k \log U^g(\d p_i)\Big|
\leq C(1+(k2^{-n})^\beta|A|_{n-1;\gr\beta})|X|_{\var q}\;,
\end{equ}
where $X$ is the anti-development of $U$ along $r$. 
\end{lemma}

\begin{proof}
Due to~\eqref{eq:anti_dev_identity}, we have $
\sum_{i=1}^k \log U^g(\d p_i) = \sum_{i=1}^k Y_i(X^g_i-X^g_{i-1})$
where $Y_i = \Ad_{U^g(\ell_i)^{-1}}$
and $X^g$ is the anti-development of $U^g$ along $r$.
Since $\beta+q^{-1}>1$, Young's estimate for sums~\eqref{eq:Young_int} implies
\begin{equ}\label{eq:Young_bound}
\Big|\sum\nolimits_{i=1}^k Y_i(X^g_i-X^g_{i-1})\Big| \lesssim (|Y_1|+|Y|_{\var{(1/\beta)}})|X^g|_{\var q}\;.
\end{equ}
Note that $|X^g|_{\var q}=|X|_{\var q}$. Furthermore $\{Y_i\}_{i=1}^k$ is the development of $\{-A(\ell_i)\}_{i=1}^k$ into $\Aut(\mfg)$ (via the adjoint representation) with $Y_1 = \id$.
Hence
$
|Y|_{\var{(1/\beta)}} \lesssim |A(\ell_{\act})|_{\var{(1/\beta)}}
\lesssim |k2^{-n}|^\beta|A|_{n-1;\gr\beta}
$,
where the first bound follows from Young's estimate for ODEs and that $Y$ takes values in a compact set,
and the second bound is a discrete analogue of~\eqref{eq:ell_omega_Hol}.
\end{proof}

\begin{theorem}\label{thm:Landau}
Let $\alpha\in(\frac12,1)$, $q \in [1,\frac{1}{1-\alpha})$, and  $\bar\alpha \in (0,\alpha)$.
Define
\begin{equ}{}
[U]_{\alpha;q,M,N} \eqdef \max_{M \leq n \leq N} \max_{r\in \rect_n} |r|^{-\alpha/2}|X|_{\var q}\;,
\end{equ}
where $X$ is the anti-development of $U$ along $r\in\rect_n$.
Suppose that~\eqref{eq:initial_bound} holds.

Then there exists $C \geq 0$, depending only on $\alpha,\bar\alpha,q,G$, such that% for all parallel $\ell,\bar\ell \in \olines_N$
\begin{equ}\label{eq:Aell_bound}
|A|_{N;\gr{\bar\alpha}} \leq P
\; \eqdef C\max\{1,\delta_M 2^{M}, \delta_M^2 2^{2\alpha M}\} e^{C (\log (1+[U]_{\alpha;q,M,N}))^2}
\end{equ}
where $\delta_M \eqdef \max_{b\in\obonds_M}|A(b)|$
%\begin{equ}
%P\eqdef C\max\{1,\delta_M 2^{M}, \delta_M^2 2^{2\alpha M}\} e^{C (\log (1+[U]_{\alpha;q,M,N}))^2}\;,\qquad \delta_M \eqdef \max_{b\in\obonds_M}|A(b)|\;,
%\end{equ}
and, if $\alpha \in [\frac23,1)$, then for all parallel $\ell,\bar\ell \in \olines_N$
\begin{equ}\label{eq:Aell_barell_bound}
|A(\ell) -  A(\bar\ell)| \leq (2^{M(1+\bar\alpha/2)+1}\delta_M + P) |\ell|^{\alpha/2} d(\ell,\bar\ell)^{\bar\alpha/2}\;.
\end{equ}
% and again $C \geq 0$ depends only on $\alpha,\bar\alpha,q,G$.
\end{theorem}

\begin{proof}
It suffices to consider $\bar\alpha \in (\frac12 \vee (1-q^{-1}),\alpha)$.
To prove~\eqref{eq:Aell_bound}, we proceed by induction on $n \geq M$.
For the base case $n=M$,
we have $|A(\ell)| \leq \delta_M 2^{M}|\ell| \leq P_{M}|\ell| \leq P_{M}|\ell|^{\bar\alpha}$,
where we take
\[
P_M \eqdef  \max\{1,\delta_M 2^{M}, 2^{\alpha M}\delta_M^2,[U]_{\alpha;q,M,N}^2\}
\;.
\]
For the inductive step,
assume that $|A|_{n-1;\gr{\bar\alpha}} \leq P_{n-1}$ for $P_{n-1} \geq P_M$.
Let $\ell \in \olines_n$ and consider two cases.

\textit{Case A1:} $\ell$ (as a subset of $\T^2$) is contained in the grid $\Grid_{n-1}$ defined by~\eqref{eq:grid_def}.
Consider the contraction $\underline{\ell}$ and extension $\bar \ell$ of $\ell$ which are respectively the longest and shortest elements of $\olines_{n-1}$ such that $\underline{\ell}\subset\ell \subset \bar \ell$.
Observe that $|\ell|=\frac12(|\underline{\ell}|+|\bar\ell|)$ and, by definition of the gauge $g$, 
in particular \eqref{e:def-g-midpoint}, 
one has $A(\ell)=\frac12(A(\underline{\ell})+A(\bar\ell))$.
Therefore, by concavity of $x\mapsto x^{\bar\alpha}$
and the inductive hypothesis\footnote{This step corrects an error in the corresponding step in the proof of~\cite[Thm.~4.12]{Chevyrev19YM}.}
$$
|A(\ell)| \leq \frac12 P_{n-1}(|\underline{\ell}|^{\bar\alpha} + |\bar\ell|^{\bar\alpha}) \leq P_{n-1}|\ell|^{\bar\alpha}\;,
$$
which proves the inductive step in the case $\ell \subset \Grid_{n-1}$.
Note that the same constant $P_{n-1}$ appears, which will be used in the next case.

\textit{Case A2:}  $\ell$ is not contained in $\Grid_{n-1}$.
Then by the definition of $g$ in~\eqref{eq:U^g_a_i_def}, we have
\begin{equ}\label{eq:ell_average}
A(\ell) = \frac{A(\ell_1) + A(\ell_2)}{2} + \Delta_1 + \Delta_2\;,
\end{equ}
where $\ell_1,\ell_2 \in\olines_n$ are the closest lines parallel to $\ell$ that are contained in $\Grid_{n-1}$.

To see \eqref{eq:ell_average}, recalling \eqref{eq:A_def_1_form},
$A(\ell)$ is a sum of terms of the form $\log U^g(a_i)$ on the left-hand side 
of \eqref{eq:U^g_a_i_def}.
The term $\frac12(A(\ell_1) + A(\ell_2))$ arises from summing 
the terms $\frac12 (\log U^g(b_{2i-3}) - \log U^g(b_{2i}))$ on the right-hand side 
of \eqref{eq:U^g_a_i_def}.

Moreover, $\Delta_1$ in \eqref{eq:ell_average} accounts for the terms $\log U^g(\d p_j)$ in~\eqref{eq:U^g_a_i_def}
and satisfies for proportionality constants depending on $G,\bar\alpha,q$
\begin{equs}[eq:Delta1_bound]
|\Delta_1|
&\leq \Big| \sum_{p} \log U^g(\partial p) \Big| + 8\plaqnorm{U}_{\alpha;M,N} 2^{-n\alpha}
\\
&\lesssim (1+ |\ell|^{\bar \alpha}P_{n-1}) [U]_{\alpha;q,M,N}|\ell|^{\alpha/2} 2^{-n\alpha/2} + 2^{-n\alpha}[U]_{\alpha;q,M,N}
%\\
%&\lesssim |\ell|^{\bar\alpha}[U]_{\alpha;q,M,N}( 2^{-n(\alpha-\bar\alpha)} + P_{n-1}2^{-n\alpha/2})
\\
&\lesssim P_{n-1}|\ell|^{\bar\alpha}[U]_{\alpha;q,M,N}2^{-n(\alpha-\bar\alpha)}\;.
\end{equs}
The sum is taken over all plaquettes $p\in\plaq_n$ which share a side with $\underline\ell$, the longest line contained in $\ell$ whose closest parallel lines in $\Grid_{n-1}$ are in $\olines_{n-1}$;
the term $8\plaqnorm{U}_{\alpha;M,N} 2^{-n\alpha}$ accounts for potential overhangs of $\ell$ arising in the case $\ell\neq\underline\ell$.
% (e.g. like the plaquettes $p_5$ are $\bar p_1$ in Figure~\ref{fig:rectangles}).
In the second inequality we used Lemma~\ref{lem:q-var_bound} (with $\beta=\bar\alpha\in (\frac12,1)$ therein) and
\begin{equ}[eq:plaqnorm_bracket]
\plaqnorm{U}_{\alpha;M,N}\leq [U]_{\alpha;q,M,N}\;,
\end{equ}
which follows from the fact that $|X|_{\var q}=|\log U(\d p)|$ where $X$ is the anti-development of $U$ along a plaquette $p$
(interpreted as a rectangle).
In the final inequality in~\eqref{eq:Delta1_bound} we used $P_{n-1}\geq 1$, $|\ell|\geq 2^{-n}$, $\bar\alpha>\alpha/2$, and $2^{-n\alpha/2} \leq 2^{-n(\alpha-\bar\alpha)}$.

The term $\Delta_2$ accounts for the terms $\mcE_i$ in~\eqref{eq:U^g_a_i_def} and satisfy $\mcE_i=O(\delta^2)$ with $\delta$ as in~\eqref{eq:delta_def}, hence
for proportionality constants depending only on $G,\alpha$
\begin{equs}[eq:Delta2_bound]
|\Delta_2| &\lesssim |\ell| 2^{n} ([U]_{\alpha;q,M,N}2^{-n\alpha} + 2^{\alpha(M-n)}\delta_M )^2
%\\
%&\lesssim |\ell| 2^{-n(2\alpha-1)}([U]_{\alpha;q,M,N}^2+2^{2\alpha M}\delta_M)
%\\
%&\lesssim |\ell|2^{-n(2\alpha-1)}([U]_{\alpha;q,M,N}+2^{\alpha M}\delta_M)^2
\\
&\lesssim |\ell| 2^{n(1-2\alpha)} P_{n-1} \leq |\ell|2^{-n(\alpha-\bar \alpha)} P_{n-1} \;.
\end{equs}
In the first inequality we used that, by~\eqref{eq:plaqnorm_bracket},
\begin{equ}
|\log U^g(\d p_i)| \leq 2^{-n\alpha}\plaqnorm{U}_{\alpha;M,N} \leq 2^{-n\alpha}[U]_{\alpha;q,M,N}\;,
\end{equ}
together with Lemma~\ref{lem:bonds_bound} and the fact that $\ell$ contains $|\ell|2^n$ bonds in $\obonds_n$.
In the second inequality we used that $P_{n-1} \gtrsim [U]_{\alpha;q,M,N}^2+2^{2\alpha M}\delta_M^2$.
%In the final inequality we used that $2^{n(1-2\alpha)} \leq 2^{-n(\alpha-\bar \alpha)}$.

%In the final inequality we simply used that $|\ell|\leq 1$ and that $2\alpha-1>\alpha-\bar\alpha$.

Using~\eqref{eq:Delta1_bound} and~\eqref{eq:Delta2_bound} together with~\eqref{eq:ell_average} and Case 1, we obtain
\begin{equ}
|A(\ell)|_{n;\gr{\bar\alpha}} \leq P_{n-1} (1+C 2^{-n(\alpha-\bar\alpha)}[U]_{\alpha;q,M,N}) \eqdef P_n\;,
%+2^{-n(2\alpha-1)}([U]_{\alpha;q,M,N}+2^{\alpha M}\delta_M)^2\;.
\end{equ}
where $C$ depends only on $\alpha,\bar\alpha,q,G$.

This completes the inductive step and shows that
\begin{equ}
\sup_n P_n \leq P_M \prod_{n=1}^\infty (1+C 2^{-n(\alpha-\bar\alpha)}[U]_{\alpha;q,M,N}) \leq P_M C e^{C (\log (1+[U]_{\alpha;q,M,N}))^2}\;.
\end{equ}
Absorbing the factor $[U]_{\alpha;q,M,N}^2$ in $P_M$ into the exponential, we obtain~\eqref{eq:Aell_bound}.

We now prove~\eqref{eq:Aell_barell_bound}. Proceeding again by induction,
note that for the base case $n=M$ with distinct and parallel $\ell,\bar\ell \in \olines_M$
\begin{equs}
|A(\ell)-A(\bar \ell)|
\leq 2^{M+1}|\ell|\delta_M
&\leq 2^{M(1+\bar\alpha/2)+1}\delta_M|\ell|^{\alpha/2}d(\ell,\bar\ell)^{\bar\alpha/2}
\\
&\eqdef S_M|\ell|^{\alpha/2}d(\ell,\bar\ell)^{\bar\alpha/2} \;,
\end{equs}
where we used in the first bound the triangle inequality and that $|A(\ell)|\leq 2^M|\ell|\delta_M$, and in the second bound that
$d(\ell,\bar \ell)\geq 2^{-M}$ and $|\ell|\leq|\ell|^{\alpha/2}$. This proves \eqref{eq:Aell_barell_bound} for scale $n=M$.

Suppose now that~\eqref{eq:Aell_barell_bound} holds at scale $n-1\geq M$ with proportionality constant $S_{n-1} \geq S_M$.
Consider distinct and parallel $\ell,\bar\ell \in \olines_n$. We consider again two cases.

\textit{Case B1:} both $\ell$ and $\bar\ell$ are contained in $\Grid_{n-1}$.
Using concavity of $x\mapsto x^{\alpha/2}$, we obtain as in Case A1 of the proof of~\eqref{eq:Aell_bound}
\begin{equ}\label{eq:case_1_ell_bar_ell}
|A(\ell)-A(\bar\ell)| \leq S_{n-1}|\ell|^{\alpha/2}d(\ell,\bar\ell)^{\bar\alpha/2}\;.
\end{equ}

\textit{Case B2:} $\ell$ is not contained in $\Grid_{n-1}$.
% (the remaining case that $\bar\ell$ is not in $\Grid_{n-1}$ follows by symmetry).
Then we know $A(\ell)$ and $A(\bar \ell)$ admit the expressions~\eqref{eq:ell_average} with the same bounds on $\Delta_i$ and $\bar\Delta_i$ (if $\bar\ell$ is in $\Grid_{n-1}$, then $\bar\ell_i=\bar\ell$ and $\bar\Delta_i=0$).
Furthermore $\ell_i\parallel \bar\ell_i$ for $i=1,2$ with either $d(\ell_i,\bar\ell_i)=d(\ell,\bar\ell)$ if $\bar\ell$ is not $\Grid_{n-1}$, or
$
d(\ell_1,\bar\ell) = d(\ell,\bar\ell) - 2^{-N}$ and $d(\ell_2,\bar\ell) \leq d(\ell,\bar\ell) + 2^{-N}$,
if $\bar\ell$ is in $\Grid_{n-1}$.
By~\eqref{eq:case_1_ell_bar_ell} and the concavity of $x\mapsto x^{\bar\alpha/2}$,
\begin{equ}
|A(\ell)-A(\bar\ell)|
\leq S_{n-1} |\ell|^{\alpha/2}d(\ell,\bar\ell)^{\bar\alpha/2} 
+ \sum\nolimits_{i\in \{1,2\}}\{ \Delta_i + \bar \Delta_i\}\;.
\end{equ}
Since $|\ell|\wedge d(\ell,\bar\ell) \geq 2^{-n}$, the second bound in~\eqref{eq:Delta1_bound} implies
\begin{equ}
\Delta_1 + \bar\Delta_1 \lesssim P_{n-1}[U]_{\alpha;q,M,N}|\ell|^{\alpha/2}d(\ell,\bar\ell)^{\bar\alpha/2} 2^{-n(\alpha-\bar\alpha)/2}\;.
\end{equ}
The second bound in~\eqref{eq:Delta2_bound} implies that
$
\Delta_2 + \bar\Delta_2 \lesssim P_{n-1}|\ell| d(\ell,\bar\ell)^{\bar\alpha/2}2^{-n(\alpha-\bar\alpha)/2}$,
where we used $2^{n(1-2\alpha)} \leq d(\ell,\bar\ell)^{\bar\alpha/2}2^{-n(\alpha-\bar\alpha)/2}$ since $\alpha\geq\frac23$ and $d(\ell,\bar\ell) \geq 2^{-n}$.
It follows that
\begin{equ}
|A(\ell)-A(\bar\ell)| \leq (S_{n-1} + C P_{n-1}[U]_{\alpha;q,M,N}2^{-n(\alpha-\bar\alpha)/2}) |\ell|^{\alpha/2}d(\ell,\bar\ell)^{\bar\alpha/2}\;.
\end{equ}
This completes the inductive step and shows that
\begin{equ}
\sup_n S_n \leq S_M + C (\sup_nP_n) [U]_{\alpha;q,M,N} \leq S_M + P\;,
\end{equ}
which completes the proof of~\eqref{eq:Aell_barell_bound}.
\end{proof}

\begin{proof}[of Theorem~\ref{thm:gauge_fix}]
Let $C$ be the constant from Lemma~\ref{lem:axial} and let $c_1,c_2$ be the constants from Lemma~\ref{lem:bonds_bound}.
We can suppose $N$ is sufficiently large such that $C2^{-N} < c_2/2$.
Suppose first that $C2^{-N\alpha/2}\fancynorm{U}_{\alpha;N} + C2^{-N}\geq c_2$ or $[U]_{\alpha;q,N} 2^{-N\alpha} \geq c_1$.
We take $g\equiv 1$. Since $
\max_{b\in\obonds_N} |A(b)| \lesssim 1$, we have
\begin{equ}
|A|_{\gr{\bar\alpha}} \leq |A|_{\gr1} \lesssim 2^{N} \lesssim \fancynorm{U}_{\alpha;N}^{2/\alpha} + [U]_{\alpha;q,N}^{1/\alpha} \lesssim Z \;.
\end{equ}
Furthermore, for distinct and parallel $\ell,\bar\ell\in\olines_N$ we have
$
|A(\ell)-A(\bar\ell)| \leq |A(\ell)| + |A(\bar\ell)| \leq 2|\ell|2^N$,
and thus if $\alpha\in [\frac23,1)$, then
\begin{equ}
|A(\ell)-A(\bar\ell)| \lesssim |\ell| d(\ell,\bar\ell)^{\bar\alpha/2}2^{N(1+\bar\alpha/2)} \Rightarrow |A|_{\bar\alpha;\rho}\lesssim \fancynorm{U}_{\alpha;N}^{4} + [U]_{\alpha;q,N}^2 \lesssim Z\;,
\end{equ}
where in the first bound we used $d(\ell,\bar\ell) \geq 2^{-N}$ and in the second bound we used $\alpha\geq \frac23$ and $\bar\alpha<\alpha$ which implies
\begin{equ}\label{eq:2N_bound}
2^{N(1+\bar\alpha/2)} \leq 2^{2N\alpha} \lesssim \fancynorm{U}_{\alpha;N}^4 + [U]_{\alpha;q,N}^2\;.
\end{equ}

Now suppose $C2^{-N\alpha/2}\fancynorm{U}_{\alpha;N} + C2^{-N}< c_2$ and $[U]_{\alpha;q,N} 2^{-N\alpha} < c_1$.
%Note that we also  $[U]_{\alpha;q,N} \geq \plaqnorm{U}_{\alpha;M,N}$ for all $M\leq N$, we also have $\plaqnorm{U}_{\alpha;0,N} 2^{-N\alpha} < c_1$.
Let $1\leq M \leq N$ be the smallest integer such that
\begin{equ}
C2^{-M\alpha/2}\fancynorm{U}_{\alpha;N} + C2^{-M} < c_2\;,
\qquad \mbox{and} \qquad
[U]_{\alpha;q,N}2^{-M\alpha} < c_1\;.
\end{equ}
Since $[U]_{\alpha;q,N} \geq \plaqnorm{U}_{\alpha;M,N}$ by~\eqref{eq:plaqnorm_bracket}, we thus satisfy the first condition in~\eqref{eq:initial_bound}.
Applying the axial gauge bound (Lemma~\ref{lem:axial}) at scale $M$, we in addition obtain
$
\max_{b\in\obonds_M} |\log U^g(b)| < c_2
$,
so that the second condition in~\eqref{eq:initial_bound} is also satisfied with $A=\log U^g$.
Our choice of $M$ further implies that
\begin{equ}\label{eq:2M_bound}
2^M \lesssim 1+\fancynorm{U}_{\alpha;N}^{2/\alpha}+ [U]_{\alpha;q,N}^{1/\alpha}\;.
\end{equ}
We then apply the binary Landau gauge starting at scale $M$, which, by~\eqref{eq:Aell_bound}, yields~\eqref{eq:A_gr_bound}.  
If furthermore $\alpha\in [\frac23,1)$, then
%, just like in~\eqref{eq:2N_bound},
$\bar\alpha <\alpha$ and~\eqref{eq:2M_bound} imply
$
2^{M(1+\bar\alpha/2)} \leq 2^{2M\alpha}\lesssim 1+\fancynorm{U}_{\alpha;N}^4 + [U]_{\alpha;q,N}^{2}
$,
and thus $|A|_{N;\bar\alpha;\rho} \leq Z$ by~\eqref{eq:Aell_barell_bound}.
\end{proof}

\subsection{Probabilistic bounds}

To use Theorem~\ref{thm:gauge_fix} in the proof of Theorem~\ref{thm:YM_gauge_fixed}, we require bounding the gauge-invariant quantities $\fancynorm{U}_{\alpha;N}$ and $[U]_{\alpha;q,N}$.
This is where we need to make a significant improvement over the method in~\cite{Chevyrev19YM} because the rough path BDG inequality employed therein yields at best moment bounds on $\fancynorm{U}_{\alpha;N}+[U]_{\alpha;q,N}$
which do not translate to moment bounds on $|A|_{N;\alpha}$ due to the appearance of $(\log (1+[U]_{\alpha;q,N}))^2$ inside the exponential in~\eqref{eq:A_gr_bound}.

Our strategy is to substitute the BDG inequality by Gaussian estimates on the rough path norm of random walks. We derive these estimates in Appendix~\ref{app:Gaussian_tails} by a
moment comparison method with the Brownian rough path.
To apply the results of Appendix~\ref{app:Gaussian_tails}, we make the following assumption on the action $S_N$. 

\begin{assumption}\label{assum:P_N}
Consider the probability measure $\P_N(\mrd x) \propto e^{-S_N(x)}\mrd x$ on $G$.
\begin{enumerate}[label=(\roman*)]
\item\label{pt:density_bound} There exist $C_l,C_u > 0$ such that for all $x\in G$ and $N\geq 2$
\begin{equ}\label{eq:density_bound}
C_l^{-1} \leq \frac{\mrd \P^{\star 2^{2N-3}}_N}{\mrd x}(x) \leq C_u\;,
\end{equ}
where $\P_N^{\star k}$ is the $k$-fold convolution of $\P_N$.

\item\label{pt:Gauss_tail} There exists $\bar C>0$ such that for all $N\geq 2$ and $\beta\in[1,\infty)$
\begin{equ}[eq:Gauss_tail]
\Big(
\int_G |x,\id|_G^\beta \P_N(\mrd x)
\Big)^{1/\beta}
\leq \bar C 2^{-N} \sqrt \beta\;.
\end{equ}
\end{enumerate}
\end{assumption}

\begin{remark}
Assumption~\ref{assum:P_N}\ref{pt:density_bound} is identical to~\cite[Eq.~(5.1)]{Chevyrev19YM} and
%means that, after $2^{2N-3}$ steps, the $G$-valued random walk with increments distributed by $\P_N$ has a density with respect to the Haar measure which is uniformly bounded above and below.
%This condition
is used to compare, in law, the holonomies $U(\d r)$ to a $G$-valued random walk (see~\cite[Prop.~1.8]{Levy20} that relates
conditioned random walks to the YM measure on closed surfaces).
If $\T^2$ is replaced by the square $[0,1]^2$ then this condition can be dropped.
\end{remark}

\begin{remark}
Assumption~\ref{assum:P_N}\ref{pt:Gauss_tail} is new
and implies Gaussian tails on $\P_N$ comparable to the normal distribution $N(0,2^{-2N})$.
This substitutes the moment condition~\cite[Eq.~(5.2)]{Chevyrev19YM}
and will be used to prove Gaussian tails for $\fancynorm{U}_{\alpha;N}$ and $[U]_{\alpha;q,N}$, which is (more than) sufficient to prove Theorem~\ref{thm:YM_gauge_fixed}
using Theorem~\ref{thm:gauge_fix}.
%In light of Theorem~\ref{thm:gauge_fix}, however, Theorem~\ref{thm:YM_gauge_fixed} holds
%once $(1+\fancynorm{U}_{\alpha;N}^4)e^{(\log(1+[U]_{\alpha;q,N}))^2}$ has moments of all orders bounded uniformly in $N$,
%so we believe this assumption can be significantly relaxed.
\end{remark}
The following proposition gives a simple sufficient (but not necessary) condition on $S_N$ to satisfy Assumption~\ref{assum:P_N}.
We will see in Example~\ref{ex:verify_S_N_assump} that the Manton, Wilson, and Villain actions from Section~\ref{sec:approx-YM} all satisfy Assumption~\ref{assum:P_N}.

\begin{proposition}\label{prop:criterion_S_N}
Suppose there exist $r,\theta,\Theta>0$ such that, for all $\eps=2^{-N}$, $N \gg 1$
\begin{equ}\label{eq:S_N_upper_bound}
 \forall x\in B_{\e r}\eqdef \{x\in G\,:\, |x,\id|<\e r\}\;,\quad S_N(x) \leq \Theta\e^{-2}|x,\id|^2_G\;,
\end{equ}
and
\begin{equ}\label{eq:S_N_lower_bound}
\forall x\in G\setminus B_{\e r}\;,\quad S_N(x) \geq \theta\e^{-2}|x,\id|^2_G\;.
\end{equ}
Then Assumption~\ref{assum:P_N}\ref{pt:Gauss_tail} holds.
If furthermore~\eqref{eq:S_N_lower_bound} holds for all $x\in G$, then Assumption~\ref{assum:P_N}\ref{pt:density_bound} holds.
\end{proposition}

\begin{proof}
We can find a connected domain $F\subset\mfg$ on which $\exp\colon F\to G$ is a bijection
and such that $|e^X,\id| \leq C |X|_\mfg$ for all $X\in F$
(e.g. $F\eqdef \{X\in\mfg\,:\, \sigma(X)\subset(-i\pi,i\pi]\}$, where $\sigma(X)$ is the spectrum of $X$).
The left-hand side of~\eqref{eq:Gauss_tail} is then bounded above by
\begin{equ}\label{eq:Gaussian_integral}
C\Big(\int_{F} Z^{-1}e^{-S_N(e^X) } |X|^\beta_\mfg J(X)\mrd X\Big)^{1/\beta}\;,
\end{equ}
where $
Z \eqdef \int_G e^{-S_N(x)}\mrd x = \int_F e^{-S_N(e^X)} J(X)\mrd X$
and $J$ is the Jacobian of $\exp\colon F \to G$.
Clearly $J$ is bounded above on $F$ and is bounded from below by $\delta >0$ on a small ball in $F$ (e.g. by~\eqref{e:d-exp}).
The bounds~\eqref{eq:S_N_upper_bound}-\eqref{eq:S_N_lower_bound}
therefore imply $Z\asymp \e^{D}$ where $D=\dim(\mfg)$.
Splitting the integral in~\eqref{eq:Gaussian_integral} over $F$ into $B_{\e r}$ and $G\setminus B_{\e r}$ and applying~\eqref{eq:S_N_lower_bound}, we obtain that
Assumption~\ref{assum:P_N}\ref{pt:Gauss_tail} holds.

Suppose now that~\eqref{eq:S_N_lower_bound} holds for all $x\in G$.
We will apply Proposition~\ref{prop:density_bound} with $\Omega = B_{\eps h}$ for sufficiently small $h>0$.
Denote
$
p = Z^{-1}e^{-S_N} = \mrd \P_N/\mrd x
$.
The bounds $Z\asymp \eps^{D}$ and~\eqref{eq:S_N_upper_bound}-\eqref{eq:S_N_lower_bound} (with~\eqref{eq:S_N_lower_bound} holding for all $x\in G$) imply that
$p^{\star 2}(e)=|p|_{L^2}^2 \asymp \eps^{-D}$,
where the equality follows from the symmetry of $p$.
Furthermore, for all $x \in B_{\eps h}^2$,
\begin{equ}
p^{\star 2}(x) = \int_G p(xy)p(y)\mrd y \geq \int_{B_{\eps h}} p(xy)p(y)\mrd y \gtrsim \int_{B_{\eps h}}\eps^{-2D}\mrd y \gtrsim \eps^{-D}
\end{equ}
for sufficiently small $h>0$ (independent of $\eps$).
We therefore satisfy the assumptions of Proposition~\ref{prop:density_bound} (with $m=2N$ and $n=3$) from which we see that Assumption~\ref{assum:P_N}\ref{pt:density_bound} holds.
\end{proof}

\begin{example}\label{ex:verify_S_N_assump}
By Proposition~\ref{prop:criterion_S_N},
the Manton, Wilson, and Villain actions all satisfy Assumption~\ref{assum:P_N}.
Indeed, these actions satisfy Assumption~\ref{assump:R}, which implies 
\begin{equ}
\forall x\in\mathring V\;,\quad S_N(x) = \e^{-2}|x,\id|_G^2 + O(|x,\id|_G^2 +\e^{-2}|x,\id|_G^3)\;.
\end{equ}
This implies both the desired upper bound~\eqref{eq:S_N_upper_bound} and lower bound~\eqref{eq:S_N_lower_bound} for $x$ in an $\e$-independent neighbourhood of $\id$.
The lower bound outside this neighbourhood follows (a) for the Manton action
trivially, (b) for the Villain action from the Varadhan formula~\eqref{eq:Varadhan_formula} and the upper bound $e^{t\Delta}(\id) \lesssim t^{-\dim(\mfg)/2}$,
and (c) for the Wilson action from
\begin{equ}
S_N(e^X) = -\e^{-2}\sum_{k=1}^\infty \Trace(X^{2k})(2k)! = \e^{-2}\Trace(1-\cosh(X)) \geq -\theta \e^{-2} \Trace(X^2)
\end{equ}
for $\theta>0$ small and all $X\in\mfg$ with spectrum contained in $(i\pi,i\pi]$.
\end{example}
\begin{remark}
The fact that the Villain action satisfies Assumption~\ref{assum:P_N}\ref{pt:density_bound} is obvious,
% because the heat kernel on $G$ at every time $t>0$ has a smooth density.
while this fact for the Wilson action
%satisfies Assumption~\ref{assum:P_N}\ref{pt:density_bound}
alternatively follows from~\cite[Thm.~8.8]{Driver89},
which moreover proves convergence of $p^{\star 2^{2N-3}}\to e^{2^{-3}\Delta}$.
\end{remark}
We now come to the main result of this subsection. Denote $\mcR_N \eqdef \sqcup_{1\leq n\leq N} \rect_n$.
\begin{theorem}\label{thm:improved_prob}
Suppose we are in the setting of Theorem~\ref{thm:YM_gauge_fixed}.
Then for any $q>2$ there exists $\lambda>0$, depending only on $G,q$,
such that for all $\alpha<1$ and $\beta\geq 10/(1-\alpha)$
\begin{equ}
\E\Big[
\Big|\sup_{r\in\mcR_N}\frac{|\log U(\partial r)| + |X|_{\var q}}{|r|^{\alpha/2}}\Big|^\beta
\Big]
\leq \lambda^\beta C_l C_u \bar C^\beta \beta^{\beta/2}\;,
\end{equ}
where $X$ is the anti-development of $U$ along $r$.
\end{theorem}
The proof of Theorem \ref{thm:improved_prob} is similar to that of \cite[Thm.~5.3]{Chevyrev19YM}
and uses an axial-type gauge to reduce the holonomies of the discrete YM measure to pinned $G$-valued random walks.
For the proof of Theorem~\ref{thm:improved_prob}, we require the following two lemmas,
the first of which provides the necessary bounds to control these walks
(it is an analogue of~\cite[Lem.~5.4]{Chevyrev19YM} under the new condition~\eqref{eq:Gauss_tail}, the main difference being that $\lambda$ below is \textit{uniform} in $\beta$).
\begin{lemma}\label{lem:P^star_M_bound}
Suppose~\eqref{eq:Gauss_tail} holds.
Then, for all $q>2$ there exists $\lambda>0$ depending only on $G,q$ such that
for all $\beta \geq 1$ and integers $k,M\geq 1$
\begin{equ}
\Big|
\int_{G^k} (|X|_{\var q}^\beta + |\log (y_1\cdots y_k)|^\beta) \P_N^{\star M}(\mrd y_1)\cdots
\P_N^{\star M}(\mrd y_k)
\Big|^{1/\beta}
%\\
%&\qquad
\leq
\lambda \bar C 2^{-N}\sqrt{kM\beta}\;,
\end{equ}
where $X_j=\sum_{i=1}^j \log y_i$ for $0\leq j \leq k$.
\end{lemma}

\begin{proof}
The proof is essentially identical to that of~\cite[Lem.~5.4]{Chevyrev19YM} upon using
Proposition~\ref{prop:RW_Gauss_tail}
in place of the enhanced (rough path) BDG inequality.
\end{proof}

\begin{lemma}\label{lem:Kolmogorov}
Let $\alpha<1$, $\beta\geq 10/(1-\alpha)$,
and $Z$ be an $\R_+^{\mcR_N}$-valued random variable.
Then
\begin{equ}\label{eq:Kolmogorov}
\E\Big[\sup_{r\in\mcR_N} |r|^{-\alpha\beta/2}Z(r)^\beta\Big]
\leq
\sup_{r\in\mcR_N} |r|^{-\beta/2}\E[Z(r)^\beta]\;.
\end{equ}
\end{lemma}

\begin{proof}
For $A,B$ denoting the left-hand and right-hand side of~\eqref{eq:Kolmogorov} respectively,
\begin{equ}
A
\leq \sum_{n\geq 1}\sum_{r\in \rect_n} B |r|^{\beta/2-\alpha\beta/2}
\leq
2B\sum_{n \geq 1} \sum_{x\in\Lambda_n}\sum_{k=1}^{2^n} (k2^{-2n})^{\beta(1-\alpha)/2}
%\\
%&\leq
%2B\sum_{n \geq 1} 2^{n(3-\beta(1-\alpha)/2)}
\leq B\;,
\end{equ}
where we used $|\Lambda_n|= 2^{2n}$ and $3-\beta(1-\alpha)/2 \leq -2$ in the final bound.
\end{proof}

\begin{proof}[of Theorem~\ref{thm:improved_prob}]
The same argument as in the proof of~\cite[Thm.~5.3]{Chevyrev19YM}, which in particular uses \eqref{eq:density_bound}, and
now using Lemma~\ref{lem:P^star_M_bound}
instead of~\cite[Lem.~5.4]{Chevyrev19YM},
implies that for all $\beta\geq 1$ and
$r\in\mcR_N$,
\begin{equ}
\E[|\log U(\partial r)|^\beta + |X|_{\var q}^\beta]
\leq \lambda^\beta C_lC_u \bar C^\beta \beta^{\beta/2} |r|^{\beta/2}\;,
\end{equ}
where $X$ is the anti-development of $U$ along $r$ and
where $\lambda$ depends only on $G,q$.
The conclusion follows from Lemma~\ref{lem:Kolmogorov}.
\end{proof}

\begin{proof}[of Theorem~\ref{thm:YM_gauge_fixed}]
For every $\alpha\in(0,1)$ and $q>2$, Theorem~\ref{thm:improved_prob} implies that
$\fancynorm{U}_{\alpha;N}$ and $[U]_{\alpha;q,N}$ have Gaussian tails uniform in $N$.
Hence every moment of $(1+\fancynorm{U}_{\alpha;N}^4)e^{(\log(1+[U]_{\alpha;q,N}))^2}$
is bounded uniformly in $N$,
and the conclusion follows from Theorem~\ref{thm:gauge_fix}.
\end{proof}

\section{Invariant measure for the 2D SYM}
\label{sec:invar_measure}

In this section, we prove Theorem~\ref{thm:invar_measure}
and its corollaries.
We first state a basic ergodicity property.
Denote $\overline\mfO_\alpha\eqdef \mfO_\alpha\sqcup\{\skull\}$, equipped with a metric for which $\skull$ is an isolated point.
(Recalling $\hat\mfO_\alpha$ from Section~\ref{subsec:main_results},
we have $\overline\mfO_\alpha=\hat\mfO_\alpha$ as sets but $\overline\mfO_\alpha$ carries a stronger topology.)

\begin{proposition}\label{prop:X_strong_Feller}
The Markov process $X$ in Theorem~\ref{thm:invar_measure}
is strong Feller on $\overline\mfO_\alpha$.
\end{proposition}

\begin{proof}
Let $A = \SYM(\bar C,\cdot)$, which is a Markov process on $\bar\Omega^1_\alpha\eqdef \Omega^1_\alpha \sqcup \{\skull\}$.
The main result of~\cite{HM16} (using similar arguments as Sec.~5.1 therein)
implies that $A$ is strong Feller, which is almost the result we want; the only subtlety is that $X$ is not simply $[A]$ but is given by  restarting $A$ at random times with gauge equivalent initial conditions and then projecting to the quotient space.
We thus show that the strong Feller property persists under this transformation.

Let $\mcP(\overline\mfO_\alpha)$ denote the space of probability measures on $\overline\mfO_\alpha$
and $P^X_t \colon \overline\mfO_\alpha \to \mcP(\overline\mfO_\alpha)$ the transition probabilities of $X$ for $t\geq 0$.
We prove the `ultra Feller' property that $P^X_t(x)$, for every $t>0$, is continuous in $x$ when $\mcP(\overline\mfO_\alpha)$ is equipped with the total variation distance $|\cdot|_{\TV}$.
(This seemingly stronger property is equivalent to the strong Feller property in our setting, see~\cite{Hairer09_Markov}).

Let $P^A_t \colon \bar\Omega^1_\alpha \to \mcP(\bar\Omega^1_\alpha)$ denote the transition probabilities of $A$.
By~\cite{HM16}, $P^A_t$ is continuous for the total variation topology for $t>0$.
Furthermore, since the blow-up time of $A$ has inverse moments of all orders, it readily follows that, for any $M>0$, $y\in x\in\mfO_\alpha$, and $t\geq 0$
\begin{equ}\label{eq:x_y_diff}
|P^X_t(x) - \pi_* P^A_t(y)|_{\TV} \leq C t^M\;,
\end{equ}
where $C>0$ depends only on $M$ and $|y|_\alpha$.

Consider $x_n,x\in\overline\mfO_\alpha$ with $x_n\to x$.
By definition of the metric on $\mfO_\alpha$ (see~\cite[Sec.~3.6]{CCHS_2D}),
we can find $y_n,y\in\bar\Omega^1_\alpha$ with $y_n\in x_n$ and $y\in x$ and $y_n \to y$ in $\bar\Omega^1_\alpha$.
We can in particular suppose~\eqref{eq:x_y_diff} holds for $x,y$ and for $x_n,y_n$ and $M=1$.
Consider $\delta>0$ and any $t>0$ such that $Ct < \delta/3$.
Then
\begin{equs}
|P_t^X(x)-P^X_t(x_n)|_\TV
&\leq 
|P_t^X(x)-\pi_*P^A_t(y)|_\TV
+
|P_t^A(y)-P^A_t(y_n)|_\TV
\\
&\qquad +|\pi_* P_t^A(y_n)-P^X_t(x_n)|_\TV\;,
\end{equs}
which is smaller than $\delta$ for all $n$ sufficiently large.
Since $|P_t^X(x)-P^X_t(x_n)|_\TV$ is non-increasing in $t$,
the conclusion follows.
\end{proof}

\subsection{Simply connected case}

For $N\geq 1$ and $A\in \Omega_N$, we denote
$
M_\alpha(A) \eqdef \inf_{B\sim A} |B|_{N;\alpha}
$.
We make the same definition for $A\in\Omega_\alpha^1$ with $|B|_{N;\alpha}$ replaced by $|B|_{\alpha}$.

\begin{lemma}\label{lem:unif_bound_M}
Suppose that $G$ is simply connected.
Suppose further that $S_N$ satisfies Assumptions~\ref{assump:R} and~\ref{assum:P_N}
and is twice differentiable on $G$.
Consider the $\Omega_N$-valued Markov process $\check A^{(N)} = \log \check U^{(N)}$ where $\check U$ is the diffusion~\eqref{eq:discrete_hat_U} with $\check S_N\eqdef S_N$ therein
and with initial distribution $\mu_N$ as in~\eqref{eq:mu_N}.
Then for every $p \geq 1$
\begin{equ}
\sup_{N \geq 1} \E\Big[ \Big|\sup_{t \in [0,1]}M_\alpha(\check A^{(N)}_t) \Big|^p \Big] <\infty \;.
\end{equ}
\end{lemma}

\begin{proof}
Consider an integer $N \geq 1$ and introduce the shorthands $A\equiv \check A^{(N)}$ and $|\cdot|_\alpha\equiv |\cdot|_{N;\alpha}$.
%We identify here $\conf$ with a subset of $\Omega^{1,(N)}$ via the $\log$ map, so we can treat $A^{(N)}_t$ as an $\conf$-valued Markov process.
Consider a measurable map $\Phi \colon \Omega_N/{\sim} \to \Omega_N$ such that $\Phi([B])\in [B]$ and $|\Phi([B])|_\alpha \leq 1+M_\alpha(B)$.
Consider further another integer $K \geq 2$
and define a $\log(Q_N)$-valued process $B$ by
$\exp(B_{j/K}) \eqdef \exp \Phi([A_{j/K}])$ for all $j\in\{0,\ldots, K-1\}$, 
and then driven by the same (suitably rotated) noise over $[j/K,(j+1)/K)$ (see Proposition~\ref{prop:gauge_covar}\ref{pt:DeTurck_to_DeTurck})
so that $B_t \sim A_t$ for all $t \in [0,1]$ and that $B\restr_{[j/K,(j+1)/K)}$ is equal in law to $A\restr_{[0,1/K)}$ started from $B_{j/K}$.

For an interval $I\subset\R$ and $X\in \CC(I,\Omega_{N;\alpha})$, denote $|X|_{\alpha;I} = \sup_{u\in I} |X_u|_\alpha$.
Then for $h>0$
\begin{equs}
\P\Big[ \sup_{t \in [0,1]}M_\alpha(A_t) \geq h \Big] &\leq \sum_{j=0}^{K-1} \P[|B|_{\alpha;[j/K,(j+1)/K)} \geq h]
\\
&= K\P[|B|_{\alpha;[0,1/K)} \geq h]\;,
\end{equs}
where we used first that $B_t\sim A_t$ for all $t\in[0,1]$
and then that $e^{B_{j/K}} =e^{\Phi([A_{j/K}])}\eqlaw e^{\Phi([A_{0}])}=e^{B_0}$ by invariance of $\pi_*\mu_N$ for $[e^A]$ (Proposition~\ref{prop:gauge_covar}).
Note that $\P[|B|_{\alpha;[0,1/K)} \geq h]$ is bounded above, for any $L \geq 0$, by
\begin{equ}
\P\big[ |B|_{\alpha;[0,1/K)} \geq h \; \Big| \; |B_0|_{\alpha} \vee \$Z\$_{\gamma;O}^{(\e)} \leq L \big] + \P\big[|B_0|_{\alpha} \vee \$Z\$_{\gamma;O}^{(\e)} > L\big]\;,
\end{equ}
where $Z$ is the (renormalised discrete) model built from the noise, $O=[-1,2]\times\T^2$, and $\gamma$ is sufficiently large.
For the second term, by Markov's inequality, moment bounds on $\$Z\$_{\gamma;O}^{(\e)}$ (Proposition~\ref{prop:Ze_moments}), and Theorem~\ref{thm:YM_gauge_fixed} (which requires $G$ to be simply connected), for every $p \geq 1$,
\begin{equ}
\P\big[|B_0|_{\alpha} \vee \$Z\$_{\gamma;O}^{(\e)} > L\big] \leq C_p L^{-p}\;.
\end{equ}
For the first term,
there exists $q \geq 1$ such that, if $K = L^{q}$, $L\geq 2$,
and $|B_0|_{\alpha} \vee \$Z\$_{\gamma;O}^{(\e)} \leq L$,
then
$
|B|_{\alpha;[0,1/K)} \lesssim L
$
(cf.~\eqref{eq:tau_inverse_bound}-\eqref{eq:A_h_bound} and remark that the discrete heat flow is a contraction on $\Omega_{N;\alpha}$ which implies there is no blow-up at time $t=0$).
In particular, the first term vanishes if we take $h \gtrsim L$,
and thus
$
\P\Big[ \sup_{t \in [0,1]}M_\alpha(A_t) \geq L \Big] \lesssim L^{q-p}
$
for a proportionality constant depending only on $p$.
Since $p\geq 1$ is arbitrary, the conclusion follows.
\end{proof}
Recall the map $\varsigma_N\colon \Omega_{\gr\alpha} \to \log Q_N$ from Definition~\ref{def:varsigma_N}.
\begin{lemma}\label{lem:M_lower_semi_cont}
Let $a\in\Omega^1_\alpha$.
Then
$
M_\alpha(a) = \lim_{N\to\infty} M_\alpha(\varsigma_N a)
$.
\end{lemma}

\begin{proof}
We first show that
\begin{equ}\label{eq:M_alpha_upper}
M_\alpha(a) \leq \liminf_{N\to\infty} M_\alpha(\varsigma_N a)\;.
\end{equ}
For every $N\geq 1$, let $\Omega_N\ni b^{(N)} \sim \varsigma_N a$ such that $|b^{(N)}|_{N;\alpha} < M_\alpha(\varsigma_N a)+1/N$.
By~\cite[Thm.~3.26]{Chevyrev19YM} applied to the (deterministic) random variables $b^{(N)}$, there exists $b\in\Omega_\alpha^1$ such that
\begin{equ}
|b|_\alpha \leq \liminf_{N\to\infty} |b^{(N)}|_{N;\alpha} = \liminf_{N\to\infty} M_\alpha(\varsigma_N a)
\end{equ}
and, for every $M\geq 1$, $\pi_{N_k,M}b^{(N_k)} \to \pi_M b$ in $\Omega_M$ for a subsequence $N_k\to\infty$.
It follows from the above uniform bounds and from Young ODE estimates that, for every axis parallel line $\ell$ supported on $\cup_{N\geq 1}\obonds_N$, $\hol(b^{(N_k)},\ell) \to \hol(b,\ell)$.

On the other hand, 
there exists $g^{(N)}\in G$ such that, for all lattice loops based at $0$ (see Definition~\ref{def:lattice_loop}),
\begin{equ}
\Ad_{g^{(N)}} \hol(b^{(N)},\ell) = \hol(\varsigma_N a,\ell)=\hol(a,\ell)\;,
\end{equ}
where we used $b^{(N)}\sim \varsigma_N a$ in the first equality and the definition of $\varsigma_N a$ in the second equality.
Passing to another subsequence so that $g^{(N)}\to g$, we see that
$
\Ad_g\hol(b,\ell) = \hol(a,\ell)
$,
from which it follows that $b\sim a$ by~\cite[Prop.~3.35]{CCHS_2D} (or, more precisely, by its obvious extension to piecewise axis parallel loops).
This shows~\eqref{eq:M_alpha_upper}.

Conversely, let $\delta>0$ and $b\sim a$ such that $|b|_\alpha < M_\alpha(a)+\delta$.
Then $\varsigma_N b \sim \varsigma_N a$, so $|\varsigma_N b|_{N;\alpha}\geq M_\alpha(\varsigma_N a)$.
Furthermore, by Lemma~\ref{lem:piecewise_hol}, $|\varsigma_N b - \pi_N b|_{N;\alpha} \to 0$ and $|\pi_N b|_{N;\alpha} \nearrow |b|_\alpha$.
Therefore
\begin{equ}
M_\alpha(a)+\delta>|b|_\alpha = \lim_{N\to\infty}|\varsigma_N b|_{N;\alpha} \geq \limsup_{N\to\infty} M_\alpha(\varsigma_N a)\;.
\end{equ}
Since $\delta>0$ is arbitrary, the conclusion follows.
\end{proof}

\begin{proof}[of Theorem~\ref{thm:invar_measure} for simply connected $G$]
Let $S_N$ be the Villain action of Section~\ref{subsubsec:example_actions} and $\mu_N$ defined by~\eqref{eq:mu_N}.
Then $\mu_N$ is the discrete YM measure in the sense of L{\'e}vy~\cite{Levy03}.
By~\cite[Thm.~1.1]{Chevyrev19YM} (which requires $G$ to be simply connected),
there exists an $\Omega^1_\alpha$-valued random variable $a$ such that $\mbox{Law}(\pi e^{\varsigma_N a}) = \pi_*\mu_N$, where $\pi\colon Q_N \to Q_N/{\sim}$ is the projection map. 
Taking $\mu$ as the law of $[a]$ implies the existence of a probability measure $\mu$ on 
$\mfO_\alpha$ that induces the YM holonomies for all lattice loops (recall  Definition~\ref{def:lattice_loop}).
Recall that there exists a measurable selection $\Upsilon\colon \mfO_\alpha\to\Omega^1_\alpha$ such that $|\Upsilon(x)|<1+\inf_{A\in x}|A|_\alpha$~\cite[Lem.~7.40]{CCHS_2D}.
By composing with $\Upsilon$,
we can and will assume that $|a|_\alpha < 1+ M_\alpha(a)$.

Recalling the metric space $\hat\Omega^1_\alpha = \Omega^1_\alpha\sqcup\{\skull\}$ from Section~\ref{subsec:main_results},
consider the Markov process $A\in D(\R_+,\hat\Omega_\alpha^1)$ with initial condition $a$ and the sequence of stopping times $\sigma_0=0 < \sigma_1 \leq \sigma_2\leq\ldots$
that satisfy\footnote{We use a construction of stopping times closer to \cite[Sec.~7.3]{CCHS_3D} rather than~\cite[Sec.~7.4]{CCHS_2D}
since the latter has an error in the argument that $T^*\eqdef \lim_{j\to\infty}\sigma_j$ is the blow up time of $M_\alpha(A)$.
This gap comes from the fact that
$t\mapsto e^{t\Delta}a\in\Omega^1_\alpha$, while continuous,
can converge arbitrarily slowly as $t\downarrow0$ for generic $a\in\Omega^1_\alpha$.
Luckily the argument in \cite[Sec.~7.3]{CCHS_3D}, which uses crucially the factor $4$ in \eqref{eq:new_sigma_def}, does not suffer from this issue and readily corrects this error.}
\begin{equ}[eq:new_sigma_def]
\sigma_{i} = \inf\{t>\sigma_{i-1} \,:\, |A(t)|_\alpha\geq 1+4M_\alpha(A(t))\}
\end{equ}
and $A(\sigma_i)=\Upsilon([A(\sigma_{j}-)])$, where we write
$f(t-) \eqdef \lim_{s\uparrow t} f(s)$, and there exists $g_i\in\mfG^{0,\alpha}$ such that
$A(\sigma_i)\sim A(\sigma_i-)^{g_i}$
(provided $A(\sigma_{i-1})\neq \skull$).
For $t\in [\sigma_{i-1},\sigma_{i})$, we define $A(t)$ as $\SYM_t(\bar C,A(\sigma_{i-1}))$ started at time $\sigma_{i-1}$.
It holds that $[A]$ is equal in law to $X$ with initial condition $[a]$.

For $h\geq 0$, we define further the stopping times
\begin{equ}
\sigma_j^h = \sigma_j\wedge 1 \wedge \inf\{t\geq 0\,:\, M_\alpha(A(t)) \geq h\}\;,
\end{equ}
which are stopping times.
For every $N,j\geq 1$, consider the dynamics $A^{(N)}\in D([0,\sigma^h_j],\mfq_N)$ defined inductively as follows.
On the interval $[0,\sigma^h_{1})$, we let $A^{(N)}$ be the process from Theorem~\ref{thm:discrete_dynamics} with initial condition $A^{(N)}(0)=\varsigma_N a$, the above choice of $\check S_N \eqdef S_N$,
and with the minor difference that $A^{(N)} \eqdef \log \check U^{(N)}$
in the notation of Section~\ref{sec:approx-YM}.
(By Remark~\ref{rem:local_condition}, Theorem~\ref{thm:discrete_dynamics} holds with $U$ replaced  by $\check U$.)
Then, for every $1\leq i < j$, define
\begin{equ}
A^{(N)}(\sigma^h_i) = A^{(N)}(\sigma^h_i-)^{g_i^{(N)}}
\end{equ}
where $g^{(N)}_i=g_i\restr_{\Lambda_N}$ and $g_i\in \mfG^{0,\alpha}$ is such that $A(\sigma^h_i) = A(\sigma^h_i-)^{g_i}$.
On the interval $[\sigma^h_i,\sigma^h_{i+1})$, we again define $A^{(N)} = \log \check U^{(N)}$
(driven by the same noise and with initial condition $A^{(N)}(\sigma^h_i)$).

By Lemma~\ref{lem:piecewise_hol}, $\lim_{N\to\infty}|A^{(N)}(0)-\pi_N A(0)|_{N;\alpha} = 0$ in probability, and thus by Theorem~\ref{thm:discrete_dynamics}
\begin{equ}
|A^{(N)}(\sigma^h_1-) - \varsigma_N A(\sigma^h_1-)|_{N;\alpha} \to 0 \quad \text{in probability as } N\to\infty\;.
\end{equ}
By~\cite[Eq.~(3.24)]{CCHS_2D},
\begin{equ}
|g_1|_{\Hol\alpha} \lesssim |A(\sigma^h_1-)|_{\alpha} + |A(\sigma^h_1-)^{g_1}|_{\alpha}\lesssim |A(\sigma^h_1-)|_{\alpha}\;.
\end{equ}
Furthermore, $
\{\varsigma_N A(\sigma^h_1-)\}^{g_1^{(N)}} = \varsigma_N A(\sigma^h_1)$.
Therefore, by the elementary extension to the lattice of~\cite[Thm.~3.27]{CCHS_2D}, which shows local uniform continuity of the gauge group action,
it follows that
\begin{equs}
|A^{(N)}(\sigma^h_1) - \varsigma_N A(\sigma^h_1)|_{N;\alpha} &= |A^{(N)}(\sigma^h_1-)^{g_1^{(N)}} - \{\varsigma_N A(\sigma^h_1-)\}^{g_1^{(N)}}|_{N;\alpha} \to 0
\end{equs}
in probability as $N\to\infty$.
Hence $a^{(N)}_1\eqdef A^{(N)}(\sigma^h_1)$ and $a_1 \eqdef A(\sigma^h_1)$
satisfy the assumptions of Theorem~\ref{thm:discrete_dynamics}.

By applying Theorem~\ref{thm:discrete_dynamics} and the above argument repeatedly $j$ times,
it follows that, almost surely along a subsequence,
\begin{equ}[eq:interval_conv]
\lim_{N\to\infty}\max_{1\leq i \leq j}
\sup_{t\in [\sigma^h_{i-1},\sigma^h_i]} (t-\sigma^h_{i-1})^{1/2}|A^{(N)}(t) - \varsigma_N A(t)|_{N;\alpha} = 0\;.
\end{equ}
It follows that
\begin{equs}[eq:sup_t_M_bound]
\sup_{t\in [0,\sigma^h_j]} M_\alpha(A(t))
&= \sup_{t\in [0,\sigma^h_j]} \lim_{N\to\infty} M_\alpha(\varsigma_N A(t))
=
\sup_{t\in [0,\sigma^h_j]} \lim_{N\to\infty} M_\alpha(A^{(N)}(t))
\\
&\leq \liminf_{N\to\infty} \sup_{t\in [0,\sigma^h_j]} M_\alpha(A^{(N)}(t))\;,
\end{equs}
where we used  Lemma~\ref{lem:M_lower_semi_cont} in the first equality and~\eqref{eq:interval_conv} in the second equality together with the implication
$|x_N - y_N|_{N;\alpha}\to 0\Rightarrow M_\alpha(x_N)-M_\alpha(y_N) \to 0$ for all $x_N,y_N\in\Omega_N$ with $\sup_N|x_N|_{N;\alpha} < \infty$
(cf.~\cite[Lem.~7.39]{CCHS_2D}).

Since $\sigma^h_i$ are stopping times
and since $A^{(N)}(\sigma^h_i) \sim A^{(N)}(\sigma^h_i-)$,
the process $A^{(N)}$ on $[0,\sigma^h_j]$ is pathwise gauge equivalent to another Markov process of the type in Theorem~\ref{thm:discrete_dynamics}
but driven by a rotated noise.
In particular, by Lemma~\ref{lem:unif_bound_M}, which is applicable to the Villain action by Example~\ref{ex:verify_S_N_assump}, 
$\sup_{t\in [0,\sigma^h_j]} M_\alpha(A^{(N)}(t))$ has every moment bounded uniformly in $N$, $h$, and $j$.
It thus follows from~\eqref{eq:sup_t_M_bound} and Fatou's lemma that, for every $p\geq 1$,
\begin{equ}
\sup_{h\geq 0}\sup_{j\geq0}\E\Big[\Big|\sup_{t\in [0,\sigma_j^h]} M_\alpha(A(t))\Big|^p\Big] < \infty\;.
\end{equ}
Since the quantity in the above expectation is increasing in both $h$ and $j$,
we can exchange $\sup_{h\geq 0}\sup_{j\geq0}$ for $\lim_{h\to\infty}\lim_{j\to\infty}$
and pass the limits inside the expectation.
Therefore, since
$\lim_{h\to\infty}\lim_{j\to\infty}\sigma^h_j < 1$ only on the event that $\sup_{t\in [0,1]} M_\alpha(A(t)) = \infty$,
we see that, a.s.
$
\lim_{h\to\infty} \lim_{j\to\infty}\sigma^h_j = 1
$,
and thus
\begin{equ}\label{eq:M_alpha_A_bound}
\E\Big[\Big|\sup_{t\in [0,1]} M_\alpha(A(t))\Big|^p\Big] < \infty\;.
\end{equ}
With the a priori bound~\eqref{eq:M_alpha_A_bound}, it follows from similar considerations as above that, for every $t\in [0,1]$,
$
\lim_{N\to\infty} |A^{(N)}(t) - \pi_N A(t)|_{N;\alpha} = 0
$ almost surely along a subsequence,
where $A^{(N)}$ is defined as above with $h=\infty$.
Since $[A^{(N)}(t)]\eqlaw [\varsigma_N a]$  for every $t\in[0,1]$ by Proposition~\ref{prop:gauge_covar},
it follows that $[A(t)]\eqlaw [a]$.
This shows invariance of $\mu$ for the Markov process $X$ on $\mfO_\alpha$.
The fact that $X$ is reversible with respect to $\mu$ follows from reversibility of $\check U^{(N)}$ with respect to the discrete YM measure $\mu_N$.

Finally, for any $C\in L(\mfg^2,\mfg^2)$ and $a\in\Omega^1_\alpha$, it follows from the main result of~\cite{HS22} that $\SYM(C,a)$ has full support in $\{Y\in (\Omega^1_\alpha)^\sol \,:\, Y_0=a\}$.
Consequently, $X^x$, for any initial condition $x\in\mfO_\alpha$, has full support in $\{Y\in (\mfO_\alpha)^\sol \,:\, Y_0=x\}$.
Since $\mu$ is invariant for $X$, $\mu$ has full support in $\mfO_\alpha$.
The fact that $\mu$ is the unique invariant probability measure of $X$ now follows from the strong Feller property (Proposition~\ref{prop:X_strong_Feller}) and the same argument as in~\cite[Cor.~3.9]{HM16}.
\end{proof}

\subsection{Product case}

In this subsection, we prove Theorem~\ref{thm:invar_measure} for $G$ of the form $G=\T^n\times L$ where $L$ is simply connected, which we call the `product case'.

We start by clarifying the role of the Lie group $G$ in the construction of the state space in~\cite{CCHS_2D}.
Consider connected Lie groups $G,H$ with the same Lie algebra $\mfg$.
Since the Banach space $\Omega_\alpha^1$ of distributional $\mfg$-valued $1$-forms  from~\cite{Chevyrev19YM} or \cite{CCHS_2D}
depends only on $\mfg$,
these Banach spaces will be the same for $G$ and for $H$.

In the setting of~\cite{CCHS_2D}, the choice of $G$ and $H$ becomes important in the definition of gauge equivalence $\sim_G$ and $\sim_H$. That is, in the smooth setting, $A\sim_G B$ if and only if there exists $g\in \CC^\infty(\T^2,G)$ such that $A^g\eqdef \Ad_g A - \mrd g g^{-1} = B$.
Here and below we write the underlying group in the subscript whenever relevant.

\begin{proof}[Theorem~\ref{thm:invar_measure} in the product case]
For any Lie groups $J,K$ with Lie algebras $\mfj,\mfk$, let us write $A=(A_J,A_K)$ for the corresponding decomposition of $\Omega^1_\alpha(\T^2,\mfj\oplus\mfk)$.
Remark that $A\sim_{J\times K} B$ if and only if $A_J\sim_J B_J$ and $A_K\sim_K B_K$, and thus
there is a canonical bijection $\mfO_{J\times K} \to \mfO_J\times \mfO_K$
given by $\mfO_{J\times K}\ni x \mapsto ([A_J],[A_K])\in \mfO_J\times \mfO_K$
for any $A\in x$.
In particular, there is a bijection $\mfO_G \simeq \mfO_{\T^n}\times \mfO_L$ (which one can check is a homeomorphism, but we do not use this).

By the simply connected case, there exists a YM measure $\mu_L$ on $\mfO_L$.
Furthermore, there exists a YM measure $\mu_{\T^n}$ which corresponds to an $\R^n$-valued white noise on $\T^2$ conditioned to have total mass $0$ (see~\cite[Thm.~1]{Levy06}),
together with the two i.i.d. $\T^n$-valued random variables distributed as the Haar measure
which encode the holonomy of two arbitrary generators of the fundamental group of $\T^2$. 
It is easy to see that $0\in L(\R^n,\R^n)$ is the unique gauge-covariant constant for $\SYM_{\T^n}$, which is simply the additive SHE (see~\cite[Sec.~I.E]{Chevyrev22YM}), and $\SYM_{\T^n}(0,\cdot)$ has invariant measure $\mu_{\T^n}$.

Finally, writing $\SYM_{\T^n\times L}(C,\cdot) = (A_{\T^n},A_L)$,
the components of $A_{\T^n}$ and $A_L$
evolve independently for any $C\in L(\mfg,\mfg)$.
It follows that the Markov process $X$ on $\mfO_{G}$ associated to $\SYM_{\T^n\times L}(0\oplus \bar C,\cdot)$
has invariant measure $\mu \eqdef \mu_{\T^n}\times \mu_L$, where we identify $\mfO_G = \mfO_{\T^n\times L}\simeq \mfO_{\T^n}\times \mfO_{L}$ as before.
The remaining statements in the theorem are easily verified as in the simply connected case.
\end{proof}

\subsection{General case}

In this subsection, we prove Theorem~\ref{thm:invar_measure} for any connected compact Lie group.
Suppose that $G,H$ are connected (not necessarily compact) and that $H$ is a covering group of $G$ with projection $p\colon H\to G$.
%(The assumption that $G,H$ are compact is not crucial here.)
By applying $p$ pointwise, we obtain a map
\begin{equ}
\tilde p\colon \mfG_H \eqdef \CC^{0,\alpha}(\T^2,H) \to \mfG_G \eqdef \CC^{0,\alpha}(\T^2,G)\;,
\end{equ}
where $\CC^{0,\alpha}$ is the closure of smooth functions in $\CC^\alpha$.
\begin{proposition}\label{prop:mfGs}
\begin{enumerate}[label=(\roman*)]
\item\label{pt:normal_subgroup} $\tilde p \mfG_H$ is a normal subgroup of $\mfG_G$.

%\item $\Ker \tilde p$ is given by the constant maps $h\colon \T^2\to H$ with values in $\Ker p$.

\item\label{pt:finite_index} If $|\Ker p|<\infty$, then $|\mfG_G/\tilde p \mfG_H | < \infty$.
%\footnote{If $G,H$ were not compact, one would require $|\Ker p|<\infty$ for this.}
\end{enumerate}
\end{proposition}

\begin{proof}
\ref{pt:normal_subgroup} Consider $g\in\mfG_G$, $h\in\mfG_H$,
and define $k\eqdef g (\tilde p h) g^{-1} \in \mfG_G$.
It remains to show that there exists $\tilde k \in \mfG_H$ such that $\tilde p\tilde k = k$.
Fix $x \in \T^2$.
Then for any $y\in\T^2$ and $\gamma \in \CC^\infty([0,1],\T^2)$ with $\gamma(0)=x$ 
and $\gamma(1)=y$, the lift of $k\circ \gamma$ to $H$
is given by $f (h\circ \gamma) f^{-1}$ where $f\in \CC^\alpha([0,1],H)$ is the lift of $g\circ \gamma$ with arbitrary $f(0) \in p^{-1}g(x)$.
For any other $\tilde\gamma$ joining $x$ and $y$,
the resulting $\tilde f(1)$ differs from $f(1)$ by a group element in $\Ker p$.
Since $\Ker p$ is contained in the centre of $H$,
the map $\tilde k\in\mfG_H$ given by $\tilde k(y) = f(1) h(y) f(1)^{-1}$ is well-defined 
and satisfies $\tilde p\tilde k = k$.

\ref{pt:finite_index} Let $\Gamma = \{\gamma_1,\gamma_2\}$ be a set of loops that generates $\pi_1(\T^2)\simeq \Z^2$.
By the same considerations as in the proof of~\ref{pt:normal_subgroup}, if $f,g \in \mfG_G$ satisfy
\begin{equ}\label{eq:equal_lifts}
\tilde f_i(1)^{-1} \tilde f_i(0) = \tilde g_i(1)^{-1}\tilde g_i(0)\;,
\end{equ}
where $\tilde f_i \colon [0,1]\to H$ is the lift of $f\circ \gamma_i$
and likewise for $g$,
then $f^{-1}g \in \tilde p \mfG_H$, i.e. $f,g$ are in the same coset of $\tilde p \mfG_H$.
Conversely, if $f^{-1}g \in \tilde p \mfG_H$, then~\eqref{eq:equal_lifts} holds.
This proves that there is a bijection between $\mfG_G/\tilde p \mfG_H$ and $\Hom(\pi_1(\T^2), \Ker p)$, which is finite since $\Ker p$ is finite.
\end{proof}
By considering the projection $\tilde p h\in\mfG_G$ for any $h\in \mfG_H$, it is easy to see that $A\sim_H B \Rightarrow A\sim_G B$.
Conversely, unless $p$ is a bijection,
there exist $A\sim_G B$ such that $A\not\sim_H B$
because the image $\tilde p [\CC^\infty(\T^2,H)]$ is strictly smaller than $\CC^\infty(\T^2,G)$.
(E.g. let $A_2 \equiv 0$ and $A_1(x,y)=-\mrd g(x) g(x)^{-1}$ where $g\in \CC^\infty(\T^1, G)$ with $g(0)=g(1)=\id_G$ with lift $h\in \CC^\infty([0,1], H)$ such that $h(0)=\id\neq h(1)$.)
%Then $A\sim_G 0$ since $0^g = A$ where we extend $g\in \CC^\infty (\T^2, G)$ with $\d_2 g=0$, while $A\not\sim_H 0$ because $\hol_H(A,\ell)=h(1)\neq \id_H$ for $\ell(t)=(t,0)$.)
In conclusion, there is a projection
$
\hat p\colon \mfO_H  \twoheadrightarrow \mfO_G$, $\hat p\colon [A]_H \mapsto [A]_G$,
which is well-defined since $A\sim_H B \Rightarrow A\sim_G B$
(and is injective if and only if $p\colon H\to G$ is injective.)

\begin{remark}\label{rem:fibration}
$\mfG_G/\tilde p\mfG_H$ is a group by Proposition~\ref{prop:mfGs}\ref{pt:normal_subgroup}
and acts transitively and freely on $\hat p^{-1} x$ for any $x\in\mfO_G$.
Hence $\hat p^{-1} x$ is in bijection with
$\mfG_G/\tilde p\mfG_H$, which
by the proof of Proposition~\ref{prop:mfGs}\ref{pt:finite_index},
is in bijection with $\Hom(\pi_1(\T^2),\Ker p)$.
\end{remark}

\begin{proof}[of Theorem~\ref{thm:invar_measure}
for connected compact $G$]
By the structure Theorem~\ref{thm:structure_compact_groups}, there exists $n\geq 0$ and a simply connected semi-simple compact Lie group $L$ such that
$G=(\T^n\times L)/Z$ where $Z$ is a finite subgroup of the centre of $\T^n\times L$.
We denote $H\simeq\T^n \times L$, so that $p \colon H\to H/Z\simeq G$ is a finite covering.

Define the pushforward probability measure $\mu\eqdef \hat p_*\mu_H$ on $\mfO_G$, where 
%$\hat p\colon\mfO_H\to\mfO_G$ is the projection from~\eqref{eq:H_to_G} and
$\mu_H$ is the YM measure  on $\mfO_{H}$ from the product case.
We claim that $\mu$ induces the YM holonomies for the trivial principal $G$-bundle over $\T^2$
and that $\mu$ is invariant for the Markov process $X$ in the theorem statement.
The first of these claims follows from~\cite[Sec.~2.2]{Levy06}.

For the second claim, let $A_H$ be the process as in the simply connected case
associated to $H$ with initial condition $a\in\Omega^1_\alpha$,
so that $X_H\eqdef [A]_H$ is the corresponding $\mfO_H$-valued Markov process.

We now claim that $\hat p X_H$ is equal in law to $X$ in the sense that $\hat p X^x_H(t) \eqlaw X^{\hat p x}(t)$ for any $x\in \mfO_H$.
Indeed, let $A_G$ be the process as in the simply connected case
associated to $G$ with initial condition $b\in\Omega^1_\alpha$, where $b\sim_G a$.
It suffices prove that $\hat p [A_H]_H \eqlaw [A_G]_G$.
To see this equality in law, we can proceed in the same way as the proof of~\cite[Thm.~2.13(ii)]{CCHS_2D}.
More precisely, we can couple $A_H$ and $A_G$ in such a way that $A_H \sim_G A_G$ (this follows from the fact that transformations in $\mfG_H$ preserve the relation $\sim_G$).
The final remark is that $A_H$ blows up in $\hat\Omega^1_\alpha$ if and only if $\hat p [A_H]_H$ blows up in $\mfO_G$, which follows from Remark~\ref{rem:fibration}
since $Z$ is finite and thus $|\hat p^{-1}x| = |\mfG_G/\tilde p \mfG_H| < \infty$ for any $x\in\mfO_G$.
This proves the claim and thus the fact that 
$\mu$ is invariant for $X$ 
because $X_H$ has invariant measure $\mu_H$ due to the product case.
All the remaining statements in the theorem are again easily verified as in the simply connected case.
\end{proof}

\subsection{Proofs of Corollaries~\ref{cor:decomposition} and~\ref{cor:universality}}
\label{subsec:proofs_cor}

\begin{proof}[of Corollary~\ref{cor:decomposition}]
Let $a$ be an $\Omega^1_\alpha$-valued random variable
distributed by $\Upsilon_*\mu$ for $\Upsilon$ as in the proof of Theorem \ref{thm:invar_measure} (for simply connected $G$).
By definition, $|a|_\alpha < 1+\inf_{b\sim a}|b|_\alpha$ almost surely.
Let $A\in D(\R_+,\Omega^1_\alpha)$ be a process defined also as in the proof of Theorem~\ref{thm:invar_measure}
with initial condition $a$
but where we now take the sequence of stopping times
$\sigma_0=0<\sigma_1<\sigma_2<\ldots$
defined by
\begin{equ}[eq:sigma_def_2]
\sigma_{j+1} \leq \inf\{t>\sigma_{j} \,:\, (2+|A(\sigma_{j}-)|_\alpha + \$Z\$_{\gamma;[-1,t]\times\T^2})^{-q} \leq t-\sigma_{j}\}\;,
\end{equ}
where $q \geq 1$ is fixed,
$Z$ is the BPHZ model built from the white noise $\xi$,
and $f(t-)=\lim_{s\uparrow t}f(s)$.
%One has $A(\sigma_j)=\Upsilon(A(\sigma_j-))$ and so $|A(\sigma_j)|_\alpha < 1+\inf_{B\sim A(\sigma_j-)}|B|_\alpha$.
Since $t\mapsto \$Z\$_{\gamma;[-1,t]\times\T^2}$ is increasing and continuous in $t$,
\begin{equ}\label{eq:sigma_gap_bound}
\sigma_{j+1}-\sigma_{j} \leq (2+|A(\sigma_{j}-)|_\alpha + \$Z\$_{\gamma;[-1,\sigma_{j+1}]\times\T^2})^{-q}\;.
\end{equ}
The process $[\sigma_j,\sigma_{j+1}) \ni t\mapsto A(t)$ is given by $\SYM_t(\bar C,A(\sigma_j))$ started at time $\sigma_j$,
which is guaranteed to exist by \eqref{eq:sigma_gap_bound} and $|A(\sigma_j)|_\alpha < 1+\inf_{B\sim A(\sigma_j-)}|B|_\alpha$
and bounds similar to \eqref{eq:tau_inverse_bound}-\eqref{eq:A_h_bound} (cf. \cite[Prop.~A.4]{CCHS_2D}).
While $A$ is not Markov,
its law is still generative (in the language of \cite[Def.~2.11]{CCHS_2D});
indeed, this follows immediately from $\lim_{j\to\infty}\sigma_j=\infty$ a.s., which
itself follows from the fact that
we can couple $A$ to the process in proof of
Theorem~\ref{thm:invar_measure} (for simply connected $G$)
so that they are pathwise gauge equivalent for all times
and thus $\sup_{t\in[0,T]}M_\alpha(A(t)) < \infty$ a.s.  for any $T>0$ (see in particular~\eqref{eq:M_alpha_A_bound}).

Let $\Phi$ solve $\d_t \Phi = \Delta\Phi+\xi$ with $\Phi(0)$ a GFF (Definition~\ref{def:GFF}).
Then $\Phi(1) = \Psi + f$ where $\Psi$ is also a GFF and $f$ is a $\mfg^2$-valued Gaussian random variable.
%Recall that $\Phi$ is H\"older-continuous in time with values in $\Omega^1_\alpha$.

Let $J\geq 1$ be the largest integer such that $\sigma_J < 1$
%($\sigma_J$ is \textit{not} a stopping time).
and $\Theta\colon [\sigma_{J-1},\sigma_{J+1}]\to \Omega^1_\alpha$ denote
$\SYM(\bar C,A(\sigma_{J-1}))$ started at time $\sigma_{J-1}$.
Observe that, if $q$ is taken sufficiently large, then $\Theta$ does not blow up on $[\sigma_{J-1},\sigma_{J+1}]$ (more precisely,
due to the presence of $A(\sigma_{j}-)$ in \eqref{eq:sigma_def_2}, the entire interval is covered by two Picard iterations in a modelled distribution space $\cD^{1+2\kappa,\alpha-1}_{-\kappa}$, e.g. see \cite[Proof of Thm.~2.4]{CCHS_2D}).
We can then decompose $\Theta(t) = \Phi(t) + B(t)$, where $B$ is the
reconstruction over the interval $[\sigma_{J-1},\sigma_{J+1}]$ of a suitable remainder in $\cD^{1+2\kappa,\alpha-1}_{0}$
with initial condition $A(\sigma_{J-1}) - \Phi(\sigma_{J-1})$
(%see, e.g. the proof of Lemma~\ref{lem:3_term} or of~\cite[Thm.~2.4]{CCHS_2D} --
note that $[\sigma_j,\sigma_{j+1}) \ni t\mapsto A(t)$ solves SYM driven by the \textit{same} model $Z$ for every $j$).

Observe that $|A(\sigma_{J-1})|_\alpha$, $|\Phi(\sigma_{J-1})|_\alpha$, and $(1-\sigma_{J-1})^{-1}<(\sigma_{J}-\sigma_{J-1})^{-1}$ have moments of all orders.
It readily follows that $\E |B(1)|^p_{\CC^{1-\kappa}} < \infty$ for all $p,\kappa>0$.
The conclusion follows since $\mbox{Law}([\Theta(1)]) = \mbox{Law}([A(1)]) = \mu$.
\end{proof}

\begin{remark}
The reason we consider $\SYM(\bar C,\cdot)$ started at time $\sigma_{J-1}$ instead of $\sigma_{J}$ in the above proof is that we do not know if $(1-\sigma_{J})^{-1}$ has moments of all orders,
so cannot use the smoothing effect of the heat flow just on the interval $[\sigma_J,1]$.
\end{remark}

\begin{proof}[of Corollary~\ref{cor:universality}]
\textit{Step 1: tightness.}
By a straightforward adaptation of Lemma~\ref{lem:axial}, we see that Theorem~\ref{thm:YM_gauge_fixed} holds for $\mu_{N,\T^n}$
(it is crucial here that $\mu_{N,\T^n}$ is the YM measure for the \textit{trivial} principal $\T^n$-bundle).
Consequently, Theorem~\ref{thm:YM_gauge_fixed} also holds for $\mu_{N,H} = \mu_{N,\T^n}\times\mu_{N,L}$ where we recall that $H=\T^n\times L$ and $L$ is simply connected.
The projection $p\colon H\to G$ is Lipschitz, and thus  Theorem~\ref{thm:YM_gauge_fixed} also holds for $\mu_{N} = p_* \mu_{N,H}$.
It follows that, defining $a^{(N)}=\log U^g$ as the corresponding $\Omega_N$-valued random variable,
the family $\{|a^{(N)}|_{N;\alpha}\}_{N\geq 1}$ has all moments bounded uniformly in $N$ and is in particular tight.

By~\cite[Thm.~3.26]{Chevyrev19YM}, every sequence $N_k$ has a subsequence $N_{k_m}$ such that $
\pi_M a^{(N_{k_m})} \to \pi_M a \;\;\text{in law as $\Omega_M$-valued random variables for every } M\geq 1$,
where $a$ is an $\Omega^1_\alpha$-valued random variable (that may depend on $N_{k_m}$) such that $\E|a|_{\alpha}^p<\infty$ for all $p\geq 1$.
Here $\pi_{M} \colon \Omega_N\to\Omega_M$ for $N\geq M$ and $\pi_M\colon\Omega^1_\alpha\to \Omega_M$ are shorthands for the projections from Section~\ref{sec:norms}.
By stability of Young ODEs, $\lim_{M\to\infty} f_\ell(\pi_M b) = f_\ell(b)$
uniformly over $b$ in any ball in $\Omega_{\gr\alpha}$, and thus
$
\lim_{M\to\infty} \E |f_\ell(\pi_M a)-f_\ell(a)| = 0
$.
Likewise, by tightness of $\{|a^{(N)}|_{N;\gr\alpha}\}_{N\geq 1}$,
\begin{equ}
\lim_{M\to\infty} \sup_{N\geq M}  \E |f_\ell(\pi_M a^{(N)}) - f_\ell(a^{(N)})| = 0\;.
\end{equ}
Finally, by continuity of $f_\ell\colon\Omega_M \to \R$ for every $M\geq 1$ sufficiently large, we obtain
$
\lim_{m\to \infty} \E f_\ell (a^{(N_{k_m})}) = \E f_\ell (a)$.

\textit{Step 2: identification of the limit.}
Suppose first that $G$ is simply connected and that $S_N$ is twice differentiable on $G$.
Then exactly the same proof as that of Theorem~\ref{thm:invar_measure} (for simply connected $G$),
with the initial condition of $A^{(N)}$ therein taken as $a^{(N)}$ above instead of $\varsigma_N a$,
shows that the law of $[a]$ is invariant for the Markov process $X$ from Theorem~\ref{thm:invar_measure}.
By uniqueness of the invariant measure of $X$, we conclude that $\mbox{Law}([a]) = \mu$, so the corollary statement holds.

To handle the case of simply connected $G$ but general $S_N$ (not necessarily twice differentiable everywhere),
we define a smooth approximation $S_N^t$ by $S_N^t(\id)=0$ and $e^{-S_N^t} \propto e^{-S_N} \star e^{t\Delta}$ for $t>0$,
and let $\mu_N^t$ denote the corresponding measure as in~\eqref{eq:mu_N}.
Note that $S_N^t\colon G\to \R$ is now smooth and, for every $N$, we can find $t_N>0$ sufficiently small
such that Assumptions~\ref{assump:R} and~\ref{assum:P_N} still hold for the sequence of actions $\{ S_N^{t_N}\}_{N\geq 1}$.
We therefore conclude that the corollary statement holds with $\mu_N$ replaced by $\mu^{t_N}_N$.
Finally, $\lim_{t\to 0}\mu^t_N = \mu_N$ weakly,
and therefore, by lowering $t_N$ if necessary,
we can ensure that
$
\lim_{N\to\infty} \mu_N(f_\ell) = \lim_{N\to\infty} \mu_N^{t_N}(f_\ell)
$,
concluding the proof for general $S_N$ and simply connected $G$.

The proof in the general case $G=(\T^n\times L)/Z$ follows in the same way upon remarking that Lemma~\ref{lem:unif_bound_M} (which is used in the proof of Theorem~\ref{thm:invar_measure})
holds for this $G$ and $\mu_N$ since,
as we argued in \textit{Step 1}, Theorem~\ref{thm:YM_gauge_fixed} does. 
\end{proof}

\appendix

\section{Gaussian tails of rough path norms for random walks}
\label{app:Gaussian_tails}

In this appendix we show Gaussian tails for the H{\"older} norm 
of the rough path lift of a symmetric random walk whose increments are independent and have Gaussian tails.
This replaces the use of the rough path BDG inequality~\cite{CF19} used in~\cite{Chevyrev19YM}.
The method is based on an elementary moment comparison with piecewise linear interpolation of Brownian motion.

Consider a random walk $X\colon\{0,\ldots, K\} \to \R^d$
such that, for all $j\in[K]$,
the increments $\delta_j \eqdef X_{j}-X_{j-1}$ are independent
and satisfy $\delta_j \eqlaw -\delta_j$.
Suppose further that there exists $\sigma > 0$ such that, for all $j\in[K]$
and $\beta\geq 1$
\begin{equ}\label{eq:sigma_bound}
\E[|\delta_j|^\beta]^{1/\beta} \leq \sigma \sqrt \beta \;.
\end{equ}
Denote $t_j=j/K$ and define $Y\colon[0,1]\to \R^d$
by $Y(t_j) = X_j$ and affine over $[t_j,t_{j+1}]$.
Let $\mbY\colon[0,1]\to G^2(\R^d)$ denote the canonical rough path lift of $Y$,
i.e. $\mbY_t = (Y_t, \int_0^t Y_s\otimes \mrd Y_s)$,
where $G^2(\R^d) \subset \R^d\oplus (\R^d)^{\otimes 2}$
is the step-2 free nilpotent Lie group equipped with its geodesic distance (see~\cite[Sec.~7]{FV10}).
Let $|\mbY|_{\Hol\gamma}$ denote the corresponding H\"older `semi-norm' as in~\eqref{eq:Hol_def}.

\begin{proposition}\label{prop:RW_Gauss_tail}
For $\gamma<\frac12$,
there exists $\lambda=\lambda(\gamma,d)>0$
such that for all $\beta \in [1,\infty)$
we have
$
\E[|\mbY|_{\Hol\gamma}^{\beta}]^{1/\beta} \leq \lambda \sqrt K \sigma\sqrt{\beta}$.
\end{proposition}
The following lemma is easily proved by expanding the exponential function as a power series and applying Stirling's formula.
%\footnote{This lemma quantifies part of~\cite[Lem.~A.17]{FV10} and clarifies a typo therein.}
%
\begin{lemma}\label{lem:Gauss_equiv}
There exists $M>0$ with the following property.
Let $Z$ be a real random variable.
Then for all $\eta>0$ such that
$L\eqdef \E[e^{\eta Z^2}] <\infty$, we have
\begin{equ}{}
\E[|Z|^q]^{1/q} \leq M \eta^{-1/2}L^{1/q}\sqrt q
\end{equ}
for all $q\in[1,\infty)$.
Conversely, if $C>0$ is such that $\E[|Z|^q]^{1/q} \leq C\sqrt q$ for all $q\in[1,\infty)$, then
$\E[e^{\eta Z^2}] \leq 2$ for all $\eta< C^{-2}/M$.
\end{lemma}
In the next lemma, let $(E,\rho)$ be a complete separable metric space.
For $f\colon [0,1]\to E$ and $\zeta\colon (0,1] \to (0,\infty)$, denote
$
\|f\|_{\Hol\zeta} = \sup_{0\leq s < t \leq 1} \zeta(t-s)^{-1}\rho(f_s,f_t)
$.
The proof of the next lemma follows from that of~\cite[Thm.~A.19]{FV10}.

\begin{lemma}\label{lem:Gauss_condition}
There exists $c>0$ with the following property.
Let $\{Z_t\}_{t\in[0,1]}$ be a continuous stochastic taking values in $E$ such that, for some $\eta>0$,
\begin{equ}\label{eq:Gauss_condition}
L\eqdef \sup_{0\leq s<t\leq 1} \E
\big[\exp
\big(
\eta \rho(Z_s,Z_t)^2|t-s|^{-1}
\big)
\big] < \infty \;.
\end{equ}
Define $\zeta(h)\eqdef\int_0^h u^{-1/2}\sqrt{\log(1+1/u^2)}\mrd u$. Then
\begin{equ}\label{eq:Gauss_Holder}
\E
\big[\exp
\big(c\eta \|Z\|_{\Hol\zeta}^2
\big)
\big]
\leq L\vee 4\;.
\end{equ}
\end{lemma}

\begin{proof}[of Proposition~\ref{prop:RW_Gauss_tail}]
Throughout the proof, we let $C$ denote a constant depending only on $d$ which may change from line to line.
Writing in components $\delta_j=(\delta_j^1,\ldots, \delta_j^d)$
and multiplying $\sigma$ by a scalar, it suffices to consider the case $\sigma=K^{-1/2}$ and where we replace~\eqref{eq:sigma_bound} by
\begin{equ}\label{eq:sigma_bound_double_factorial}
\forall a\in[d]\;,\quad \forall \, \textnormal{ even } m \geq 1\;,\quad \E[(\delta^a_j)^m] \leq K^{-m/2} (m-1)!!\;,
\end{equ}
where $n!!$ is the double factorial.
Let $B\colon[0,1]\to \R^d$ denote the piecewise linear interpolation of standard Brownian motion over the points $t_0,\ldots,t_K$.
Note that
\begin{equ}
\E[(B^a(t_i)-B^a(t_{i-1}))^m]=K^{-m/2}(m-1)!!
\end{equ}
for all $a\in[d]$ and even $m\geq 1$.

For $a,b\in[d]$ and $m,n\geq 0$,
due to $\delta_j\eqlaw-\delta_j$, we have
$\E[(\delta_j^a)^m(\delta_j^b)^n] = 0$ if $m+n$ is odd.
If $m+n$ is even then, by~\eqref{eq:sigma_bound_double_factorial} and H{\"o}lder's inequality with exponents $\frac{n+m}{m}$ and $\frac{n+m}{n}$,
$
|\E[(\delta^a_j)^m(\delta^b_j)^n]| \leq K^{-(m+n)/2}(m+n-1)!!
$.
Hence, for all $a\in[d]$ and even $m \geq 1$,
expanding $(Y_s^a-Y_t^a)^m$ into monomials in $\{\delta_j^a\}_j$,
we obtain
\begin{equ}
\E[(Y_s^a-Y_t^a)^m]^{1/m}
\leq \E[(B_s^a-B_t^a)^m]^{1/m}
\leq C \sqrt m |t-s|^{1/2}\;.
\end{equ}
Moreover,  writing $\int_s^t (Y_u^a-Y^a_s) \mrd Y_u^b$ as a sum of degree $2$-monomials in $\{\delta^a_i\}_i$ and $\{\delta^b_j\}_j$,
we see that for all even $m \geq 1$
\begin{equs}
\E
\Big[
\Big(
\int_s^t (Y^a_u-Y^a_s) \mrd Y_u^b
\Big)^m
\Big]^{1/m}
&\leq
\E
\Big[
\Big(
\int_s^t (B_u^a-B_s^a) \mrd B_u^a
\Big)^m
\Big]^{1/m}
\\
&=\frac12 \E[(B^a_t-B^a_s)^{2m}]^{1/m}
\leq C^2m|t-s|\;.
\end{equs}
In conclusion,
$
\sup_{0 \leq s < t \leq 1} |t-s|^{-1/2}\E[\rho(\mbY_s,\mbY_t)^m]^{1/m} \leq C \sqrt m
$,
where $\rho$ is the geodesic distance on $G^2(\R^d)$.
Hence, Lemma~\ref{lem:Gauss_equiv} implies that~\eqref{eq:Gauss_condition} in Lemma~\ref{lem:Gauss_condition} holds for $Z=\mbY$ with $L=2$ and $\eta\asymp C^{-2}$,
and in turn~\eqref{eq:Gauss_Holder}
implies the desired bound
since $\zeta(h) \geq h^\gamma$ for all $h \leq h_0(\gamma)$.
\end{proof}

\section{Compact Lie groups}
\label{app:reps}

We recall a structure theorem for compact Lie groups~\cite[Thm.~V.8.1]{BtD85}.

\begin{theorem}\label{thm:structure_compact_groups}
Suppose $G$ is a connected compact Lie group.
Then $G\simeq F/Z$ where $F=K\times H$, $K$ is a torus, $H$ is a simply connected semi-simple compact Lie group, and $Z$ is a finite subgroup of the centre of $F$.
\end{theorem}
The following lemma, based on sub-Riemannian geometry, is used in the proof of Proposition~\ref{prop:A_tilde_A}.
\begin{lemma}\label{lem:curve_selection}
Consider non-zero $J \in L(\mfg,\R)$.
Then there exists a smooth curve $\zeta\colon [0,1] \to \mfg$
such that
\begin{enumerate}[label=(\roman*)]
\item\label{pt:const_deriv} $\dot\zeta=0$ on $[0,\frac14]$ and $[\frac34,1]$,
\item\label{pt:non-zero}
$\zeta(0)=0$ and $J\zeta(1)\neq 0$,
\item\label{pt:targets}
$L^\zeta(1)=\id$ where $L^\zeta\in\CC^\infty([0,1],G)$ solves
$
\mrd L^\zeta = (\mrd \zeta) L^\zeta$, $L^\zeta(0)=\id
$,
\item \label{pt:4targets} $\tilde L^\zeta(1)=\id$ where $\tilde L^\zeta\in\CC^\infty([0,1],G)$ solves
$
\mrd \tilde L^\zeta = 4(\mrd \zeta) \tilde L^\zeta$, $\tilde L^\zeta(0)=\id
$.
\end{enumerate}
\end{lemma}

\begin{proof}
By pre-composing with a suitable smooth $\psi\colon[0,1]\to[0,1]$,
it suffices to find $\zeta$ satisfying~\ref{pt:non-zero}\ref{pt:targets}\ref{pt:4targets}.
Let $F = K\times H$ and $Z\subset F$ be as in Theorem~\ref{thm:structure_compact_groups} so that $G=F/Z$.
We write $\mfg=\mfk \oplus \mfh$ for the corresponding decomposition, where $\mfk\eqdef T_{\id} K$ is Abelian and $\mfh\eqdef T_{\id}H$ is semi-simple.

Suppose first that $J\mfk \neq 0$.
Since $K$ is a torus, there clearly exists $X \in\mfk$ such that $JX\neq 0$ and $\exp_K(X) = \exp_K(4X) = \id$.
So we are done by setting $\zeta(t) = tX$ since then
$
L^{\zeta}(t) = \exp_G(tX) = p\exp_K(tX)
$
and $\tilde L^{\zeta}(t) = \exp_G(4tX) = p\exp_K(4tX)
$,
where $p\colon F\to G$ denotes the natural projection.

Suppose now $J\mfh \neq 0$.
Consider the product manifold $H\times H\times \mfh$ equipped with a distribution $\Delta$ such that $\Delta_{(g,h,X)}\subset T_{(g,h,X)}( H\times H\times \mfh)$
is spanned by all vectors of the form $(R^g_*Y, 4R^h_* Y, Y)$ for $Y\in\mfh$, where $R^g \colon H\to H, x\mapsto xg$.
Note that the horizontal lift of $\zeta \in \CC^\infty([0,1],\mfh)$ is the curve
\begin{equ}[eq:horizontal_lift_zeta]
[0,1]\ni x\mapsto (L^\zeta(x),\tilde L^\zeta(x),\zeta(x))
\end{equ}
(with $L^\zeta,\tilde L^\zeta$ understood as taking values in $H$).
We claim that $\Delta$ is bracket-generating (i.e. satisfies H\"ormander's condition).
Indeed, because $\mfh$ is semi-simple, every element $Y\in \mfh$ can be expressed as $Y=[A,B]$ for some $A,B\in\mfh$.
Therefore, denoting by $\Delta^{(1)}\eqdef \Gamma(\Delta)$ the space of horizontal vector fields, and $\Delta^{(n+1)} \eqdef [\Delta^{(n)},\Delta^{(1)}]$,
it follows that
\begin{equ}
\{(X+Y+Z,4X+16Y+64Z,X)\,:\, X,Y,Z\in\mfh\}
\end{equ}
is contained in $\Delta^{(3)}_{(\id,\id,0)}\subset T_{(\id,\id,0)} (H\times H\times\mfh)$
and a similar statement holds for every tangent space $T_{(g,h,A)} (H\times H\times\mfh)$.
Therefore $\Delta^{(3)}_x = T_x(H\times H\times \mfh)$ which proves the claim that $\Delta$ is bracket-generating.

By the Chow--Rashevskii theorem (see, e.g.~\cite[Sec.~1.2.B]{Gromov96_CC}), we can join $(\id,\id,0)$ and any $(g,h,Y)$ with a smooth horizontal curve, and thus, by~\eqref{eq:horizontal_lift_zeta}, there exists $\zeta\in\CC^\infty([0,1],\mfh)$
such that $\zeta(0)=0$ and $(L^\zeta(1) ,\tilde L^\zeta(1),\zeta(1))=(g,h,Y)$.
Taking $g=h=\id$ and any $Y\in\mfh$ such that $JY\neq 0$, the conclusion follows.
\end{proof}

\section{Transition functions of random walks on groups}
\label{app:RWs}

Let $G$ be a unimodular locally compact group equipped with a Haar measure, which we denote by $|A| = \int_A \mrd x$ for (Borel) measurable $A\subset G$.
Suppose $p \colon G\to [0,\infty)$ is a symmetric transition kernel, i.e. $\int p(x)\mrd x = 1$ and $p(x)=p(x^{-1})$.
We let $e$ denote the identity element of $G$.
We are interested in this appendix in deriving lower and upper bounds on the convolution power $p^{(m)} \eqdef p^{\star m}$.
These results,
inspired by~\cite{Hebisch_Saloff_Coste_93},
are used to verify Assumption~\ref{assum:P_N}\ref{pt:density_bound} (see Proposition~\ref{prop:density_bound}).

The two main results below are Theorems~\ref{thm:weak_Harnack} and~\ref{thm:pe_bound}.
Theorem~\ref{thm:weak_Harnack}
is a Harnack-type inequality and is almost identical to~\cite[Thm.~4.2]{Hebisch_Saloff_Coste_93}.
Theorem~\ref{thm:pe_bound} is a bound on $p^{(2^m)}(e)$, which can be seen as a generalisation and quantitative version of~\cite[Thm.~4.1(i)]{Hebisch_Saloff_Coste_93}
-- our proof here is different and is based on induction that does not require 
a global doubling assumption used in the proof of~\cite[Thm.~4.1(i)]{Hebisch_Saloff_Coste_93}
(we only assume a \textit{local} doubling property, see~\eqref{eq:doubling} --
one recovers the global doubling property by setting $a=\infty$ therein).
%\footnote{As far as we can tell, the proof of~\cite[Thm~4.1(i)]{Hebisch_Saloff_Coste_93} uses crucially a \textit{global} doubling assumption.}

We fix throughout a measurable set $\Omega\subset G$ such that $e\in \Omega$ and that is symmetric, i.e. $x^{-1}\in\Omega$ for all $x\in\Omega$.
Define $\rho(x) = \inf\{n\geq 1\,:\,x\in \Omega^n\}$.

\begin{theorem}\label{thm:weak_Harnack}
Suppose $c_0\eqdef \inf_{x\in\Omega^2}p^{(2)}(x) >0$ and let $B=c_0|\Omega|$.
Then for all $n,m\geq 1$, we have
$
|p^{(2n+m)}(x)-p^{(2n+m)}(e)| \leq 2 (m B)^{-1/2}\rho(x) p^{(2n)}(e)
$.
\end{theorem}

\begin{proof}
The proof is identical to that of \cite[Thm.~4.2]{Hebisch_Saloff_Coste_93};
one merely needs to chase through the constants and use that \cite[Lem.~3.2]{Hebisch_Saloff_Coste_93}
can be improved (with the same proof) to the statement that, for $K$ be a symmetric Markov operator and all integers $m,n\geq 1$,
$
|(\id -K^{2n})^{1/2}K^{m}f|_{L^2}^2 \leq n(2m)^{-1}|f|^2_{L^2}
$.
\end{proof}

\begin{theorem}\label{thm:pe_bound}
Suppose there exist $a\in [0,\infty]$, $C,D\geq 0$ such that, for all $n \geq 1$,
\begin{equ}\label{eq:doubling}
|\Omega^n| \geq C\min\{a, |\Omega|n^D\}\;.
\end{equ}
With the notation and assumptions of Theorem \ref{thm:weak_Harnack}, one has for all $m\geq1$
\begin{equ}\label{eq:induction_p_e}
p^{(2^m)}(e) \lesssim \max\{2^{-Dm/2}B^{-D/2}c_1,
C^{-1}a^{-1},
C^{-1}2^{-Dm/2} |\Omega|^{-1}B^{-D/2}\}\;,
\end{equ}
where 
$c_1=p^{(2)}(e)$ and
the proportionality constant depends only on $D$.
\end{theorem}

\begin{proof}
Suppose $m \geq 1$ is such that
$
\sigma_m \eqdef \floor{(2^{m}B)^{1/2}2^{-D/2}/20} = 0
$.
Then $p^{(2^m)}(e) \leq c_1\lesssim 2^{-Dm/2}B^{-D/2}c_1$ and thus~\eqref{eq:induction_p_e} holds for a proportionality constant depending only on $D$.
We proceed by induction on $m$.
The base case, i.e. the smallest $m\geq 1$ such that $\sigma_m\geq 1$, follows as before.
Suppose now~\eqref{eq:induction_p_e} holds for some $m\geq 1$ with $\sigma_m\geq 1$.
We now have two cases. The first is that $p^{(2^{m+1})}(e) \leq 2^{-D/2}p^{(2^m)}(e)$, in which case~\eqref{eq:induction_p_e}
holds with $m$ replaced by $m+1$ for the same proportionality constant by the inductive hypothesis.
The second case is that $p^{(2^{m+1})}(e) > 2^{-D/2}p^{(2^m)}(e)$.
In this case, for all
\begin{equ}\label{eq:V_m}
x\in V_m \eqdef \{y\in G\,:\, 2( 2^{m}B)^{-1/2}\rho(y) \leq 2^{-D/2}/10\}\;,
\end{equ}
it follows from Theorem~\ref{thm:weak_Harnack} that
\begin{equ}\label{eq:p(x)_lower_bound}
p^{(2^{m+1})}(x) \geq p^{(2^{m+1})}(e)/2\;.
\end{equ}
Remark that $V_m = \Omega^{\sigma_m}$, and  since $\sigma_m\geq 1$,
it follows from~\eqref{eq:doubling} that
\begin{equ}\label{eq:V_m_volume_bound}
|V_m| \gtrsim C \min \{a, |\Omega| (2^m B)^{D/2}\}
\end{equ}
for a proportionality constant depending only on $D$.
Since
$
\int_{V_m}p^{(2^{m+1})}(x)\mrd x \leq \int_G p^{(2^{m+1})}(x)\mrd x=1
$,
we obtain from~\eqref{eq:p(x)_lower_bound} and~\eqref{eq:V_m_volume_bound}
\begin{equ}
p^{(2^{m+1})}(e)\leq 2|V_m|^{-1}\lesssim C^{-1} \max\{a^{-1}, |\Omega|^{-1} (2^m B)^{-D/2}\}\;.
\end{equ}
This completes the inductive step in the second case (remark that we did not use the inductive hypothesis for this case).
\end{proof}
The following consequence of Theorems~\ref{thm:weak_Harnack} and~\ref{thm:pe_bound} is much more specific but gives a simple way to verify Assumption~\ref{assum:P_N}\ref{pt:density_bound} (see Proposition~\ref{prop:criterion_S_N}).
One should think of $\eps$ below as parametrising a family of kernels $p$.
\begin{proposition}\label{prop:density_bound}
Suppose that
\begin{itemize}
\item  $G$ is a compact connected Lie group of dimension $D$,
\item $\Omega$ is a ball centred at $e$ of radius $\eps\ll1$ for a geodesic distance on $G$, and
\item $p^{(2)}(e) \asymp \inf_{x\in \Omega^2} p^{(2)}(x) \asymp \eps^{-D}$ uniformly in $\eps$.
\end{itemize}
Then for every integer $n\geq 0$, there exist $\e_n,K_n>0$, independent of $\eps$,
such that $K^{-1}_n < p^{(2^{m-n})}(x) < K_n$ for all $x\in G$ and $m> n$
provided that $2^m\geq \e^{-2}>\e^{-2}_n$.
\end{proposition}

\begin{proof}
We have $|\Omega|\asymp \eps^D$
and~\eqref{eq:doubling} holds with $a=1$ and $C\asymp 1$.
Let $c_1=p^{(2)}(e)$ and $c_0 = \inf_{x\in \Omega^2} p^{(2)}(x)$.
Then $B\eqdef c_0|\Omega| \asymp 1$
and the right-hand side of~\eqref{eq:induction_p_e} is bounded above by a multiple of $\max\{1, 2^{-Dm/2}\eps^{-D}\}$ for all $m\geq 1$.
By Theorem~\ref{thm:pe_bound}, we obtain the upper bound $p^{(2^{m-n})} \lesssim 2^{nD/2}$ for $m>n$ such that $2^m\geq \e^{-2}$, where the proportionality constant depends only on $D$.

For the lower bound, fix $2^{m}\geq \eps^{-2}$.
From the upper bound $p^{(2^{m-n})}(e) \lesssim 2^{nD/2}$
and the obvious lower bound $p^{(2k)}(e) \gtrsim 1$,
it follows that for every integer $j\geq j(D)\geq 1$,
there exists $n\in \{j,\ldots 2j\}$ such that
$
p^{(2^{m-n})}(e) \geq 2^{-D}p^{(2^{m-n-1})}(e)
$.
For such $n$, it follows from Theorem~\ref{thm:weak_Harnack} that
$p^{(2^{m-n})}(x) \geq \frac12 p^{(2^{m-n})}(e) \gtrsim 1$ for all
\begin{equ}
x\in U_n \eqdef \{y\in G\,:\, 2( 2^{m-n-1}B)^{-1/2}\rho(y) \leq 2^{-D}/10\} = \Omega^{\varsigma_n}\;,
\end{equ}
where $\varsigma_n = \floor{(2^{m-n-1}B)^{1/2}2^{-D}/20}$.
Remark that $\varsigma_n \gtrsim \e^{-1}2^{-n/2}$ provided that $\e \ll 2^{n/2}$, which we assume henceforth.
Since $\Omega$ is a geodesic ball centred at $e$ of radius $\e$, $U_n$ contains a geodesic ball centred at $e$ of radius $\delta 2^{-n/2}$ for $\delta>0$ independent of $n$,
and thus $U_n^{\floor{k2^{n/2}}}=G$ for some integer $k>0$ independent of $n$.
It now readily follows from a chaining argument (see, e.g.~\cite[p.~689]{Hebisch_Saloff_Coste_93})
that for some $k>0$ independent of $n$,
there exist $K_{n}>0$ such that $K_{n}^{-1} < p^{(k2^{n/2}2^{m-n})}$.
Since $k2^{n/2}2^{-n} \to 0$ as $n\to\infty$, the conclusion follows by taking $n$ large.
\end{proof}

\section{Symbolic index}

We collect in this appendix commonly used symbols of the article, together
with their meaning and, if relevant, the page where they first occur.

 \begin{center}
\renewcommand{\arraystretch}{1.1}
\begin{longtable}{lll}
\toprule
Symbol & Meaning & Page\\
\midrule
\endfirsthead
\toprule
Symbol & Meaning & Page\\
\midrule
\endhead
\bottomrule
\endfoot
\bottomrule
\endlastfoot
%$|\cdot|_{\Hol\alpha}$  & $\alpha$-H\"older `semi-norm'  &  \pageref{eq:Hol_def} \\
%$|\cdot|_{\var p}$  & $p$-variation semi-norm  &  \pageref{page:variation_norm} \\
$*_{(i)}$, $*$  & (Semi)discrete convolution on $\R\times \obonds_i$  &  \pageref{page:convolution-i} \\
$\bonds$, $\obonds$ &  Oriented bonds, positively oriented bonds in $\Lambda$ & \pageref{bonds_page_ref}, \pageref{obonds_page_ref} \\
$\cD_{\<IXi>,\e}^{\gamma,\eta}$ &  Modelled distributions with special form & \pageref{eq:A_expansion} \\
$\CE_{\times}$  & Neighbouring horizontal and vertical bonds  &  \pageref{e:def-CEtimes} \\
$\cE$ & Multiplication by $\e$ on modelled distributions & \pageref{eq:B_expand}\\
$\hat F$ & $\hat F = F\sqcup\{\skull\}$ for a metric space $F$ & \pageref{page:F_hat}\\
$G$ & Compact connected Lie group $G\subset \U(n)$ & \pageref{page:G}\\
$\mfg$ & Lie algebra of $G$, $\mfg\subset\mfu(n)$ & \pageref{page:mfg}\\
% $\mfG^\alpha$  & Gauge transformations $\mfG^\alpha = \CC^{\alpha}(\T^2,G)$ &  \pageref{mfG_page_ref} \\
 $\mfG^{0,\alpha}$  & Closure of smooth functions in $\CC^{\alpha}(\T^2,G)$ &  \pageref{mfG_0_page_ref} \\
$K^{i;\e}$  & Truncated heat kernels on $\obonds_i$ (similar for $P^{i;\e}$ etc.)  & \pageref{page:K12} \\
$\log$  & log map with good properties  & \pageref{page:log} \\
$L_G$  & Space of operators commuting with $\Ad_G$ & \pageref{eq:L_G} \\
$\Lambda$ & Lattice $\Lambda\subset \T^d$ with spacing $\eps = 2^{-N}$ & \pageref{Lamdba_page_ref}\\
$\Omega_N$ & Functionals on line segments in $\lines_N$  & \pageref{page:Omega_N} \\
$\Omega_{N;\alpha}$ &   $\Omega_N$ equipped with  $|\cdot|_{N;\alpha} \eqdef |\cdot|_{N;\gr\alpha}+|\cdot|_{N;\alpha;\rho}$  &
\pageref{page:Omega_Nalpha}\\
$\Omega$  & Projective limit of   $\Omega_N$  & \pageref{page:Omega_alpha} \\
$\Omega_{\gr\alpha}$  & Subspace of $\Omega$ with $ |A|_{\gr\alpha}<\infty$  & \pageref{page:Omega_alpha} \\
$\Omega_{\alpha}$ & Subspace of $\Omega$ with $ |A|_{\gr\alpha}+|A|_{\alpha;\rho}<\infty$ & \pageref{page:Omega_alpha} \\
$\Omega^1_{\alpha}$ & Closure of smooth functions in $\Omega_\alpha$ & \pageref{page:Omega_1_alpha} \\
%$\Omega\mcE$ & $\mfg$-valued $1$-forms with components in $\mcE$ & \pageref{page:Omega_mcE} \\
$\mfO_\alpha$ & Space of gauge orbits $\Omega^1_\alpha/\mfG^{0,\alpha}$ & \pageref{page:mfO} \\
$\plaq$ & Set of plaquettes &  \pageref{def:plaquette}\\
$\CP^\e,\CP$ & Convolution with (discrete) heat kernel &  \pageref{CP_page_ref}, \pageref{eq:A_general_FPP}\\
$\pi$  & Quotient map $\mcH\to\mcH/{\sim}$,  $\mcH\in \{Q_N,\Omega_N,\Omega^1_\alpha\}$ &
	\pageref{page:pi1}, \pageref{page:pi2}, \pageref{page:pi3}\\
$\pi_{N,M}$  & Projections $\Omega_{N} \twoheadrightarrow \Omega_{M}$ for $M\leq N$ &
	\pageref{page:pi_NM}\\
$\SYM_t(C,a)$  & Solution to SYM at time $t$ with i.c. $a$ and mass $C$ & \pageref{page:SYM_t}\\
$\varsigma_N$  &  Restricting elements of $ \Omega_{\gr\alpha}$ to bonds by $\log \hol$ &
	\pageref{def:varsigma_N} \\
%$\mfT^{\YM}_-$ & Negative trees of continuum YM-type & \pageref{page:mfT_YM_-}\\
%$\mfT^{\rem}_-$ & Negative trees arising from remainders & \pageref{page:mfT_rem_-}\\
$\mathring{V}$, $V$  &  $\mathring V$ is neighbourhood of $0$ in $\mfg$, $V=\mathring{V}^{\obonds}\subset\mfq$ & \pageref{ring_V_page_ref}, \pageref{V_page_ref}  \\
%$V$  &  Extension of $\mathring{V}$ to $\mfq$, i.e. $V=\mathring{V}^{\obonds}\subset\mfq$ &  \pageref{V_page_ref} \\
%$V_{\be}$, $V_{\CD}$ &  Domains of $\CS_{\be}$ and $\CD$   &	\pageref{eq:Diff_def} \\
$\mathring{W}$, $W$  &  $\mathring W$ is neighbourhood of $\id$ in $G$, $W=\mathring{W}^{\obonds}\subset Q$  & \pageref{ring_W_page_ref}, \pageref{W_page_ref}  \\
%$W$  &  Extension of $\mathring{W}$ to $Q$, i.e. & \pageref{W_page_ref} 
\end{longtable}
 \end{center}

\endappendix

\bibliographystyle{./Martin}
\bibliography{./refs}

\end{document}